\documentclass[reqno]{memo-l}

\usepackage[T1]{fontenc}
\usepackage[utf8]{inputenc}
\usepackage{lmodern}
\usepackage{microtype}
\usepackage{constants}
\usepackage{upgreek}
\usepackage{array}
\usepackage{esint}
\usepackage{mathtools}
\usepackage{amssymb}
\usepackage{graphicx}
\usepackage{graphbox} 
\usepackage{tikz}
\usepackage{enumerate}

\newenvironment{notationitemize}
{%
  \begin{itemize}
    \setlength{\itemindent}{-2.3em}
    \setlength{\leftmargin}{1em} 
    \setlength{\itemsep}{0.1em}  
  }
  {%
  \end{itemize}
}

\usepackage[
pdftitle={Generic topological screening and approximation of Sobolev maps},%
pdfauthor={Pierre Bousquet, Augusto Ponce and Jean Van Schaftingen},final, unicode=true]{hyperref}
\usepackage[abbrev,backrefs]{amsrefs}

\usepackage{prelim2e}

\usepackage{todonotes}
{}
{}
{}
{}
{}

\counterwithin{figure}{chapter} 

\newtheorem{theorem}{Theorem}[chapter]
\newtheorem{lemma}[theorem]{Lemma}
\newtheorem{proposition}[theorem]{Proposition}

\newtheorem{corollary}[theorem]{Corollary}
\newtheorem{definition}[theorem]{Definition}

\newtheorem{question}{Problem}

\newtheorem*{Claim}{Claim}

\theoremstyle{definition}
\newtheorem{example}[theorem]{Example}
\AtEndEnvironment{example}{\quad\null\hfill\qedsymbol}%

\theoremstyle{remark}
\newtheorem{remark}[theorem]{Remark}
\AtEndEnvironment{remark}{\quad\null\hfill\qedsymbol}%

\numberwithin{section}{chapter}
\numberwithin{equation}{chapter}

\newtheorem{PartThm}{Proof of Theorem~\ref{thm_density_manifold_open} --- Part}

\newtheoremstyle{addendumstyle}{\topsep}{\topsep}{\itshape}{}{\bfseries}{.}{.5em plus 1pt minus 1pt}{#1 #2 to #3}
\theoremstyle{addendumstyle}

\newcommand{\abs}[1]{\mathopen\lvert#1\mathclose\rvert}
\newcommand{\bigabs}[1]{\bigl\lvert#1\bigr\rvert}

\newcommand{\biggabs}[1]{\biggl\lvert#1\biggr\rvert}

\newcommand{\norm}[1]{\mathopen\lVert#1\mathclose\rVert}
\newcommand{\bignorm}[1]{\mathopen\big\lVert#1\mathclose\big\rVert}

\newcommand{\dualprod}[2]{\langle #1, #2 \rangle}

\newcommand{\compose}{\circ}
\newcommand{\seminorm}[1]{\mathopen[#1\mathclose]}
\newcommand{\floor}[1]{\lfloor#1\rfloor}
\newcommand{\N}{{\mathbb N}}
\newcommand{\R}{{\mathbb R}}
\newcommand{\Z}{{\mathbb Z}}

\newcommand{\F}{\mathbb{F}}
\newcommand{\Ball}{\mathbb{B}}
\newcommand{\cBall}{\overline{\Ball}{}}
\newcommand{\Sphere}{{\mathbb S}}
\newcommand{\Disk}{{\mathbb D}}
\newcommand{\Sobolev}{{\mathrm W}}
\newcommand{\Hilbert}{{\mathrm H}}
\newcommand{\ClassR}{{\mathrm R}}
\newcommand{\Lebesgue}{{\mathrm L}}
\newcommand{\Smooth}{{\mathrm C}}
\newcommand{\UContinuous}{{\mathrm{BUC}}}
\newcommand{\BUC}{{\mathrm{BUC}}} 
\newcommand{\Simplex}{{\mathbb T}}
\newcommand{\Fuglede}{{\mathrm{Fug}}}

\newcommand{\st}{\mathpunct{:}}

\usepackage{mathrsfs}

\newcommand{\cE}{\mathscr{E}}

\newcommand{\cK}{\mathscr{K}}
\newcommand{\cL}{\mathscr{L}}

\newcommand{\cS}{\mathscr{S}}
\newcommand{\cT}{\mathscr{T}}
\newcommand{\cU}{\mathscr{U}}

\newcommand{\cZ}{\mathscr{Z}}
\newcommand{\cH}{\mathcal{H}}
\newcommand{\manfM}{\mathcal{M}}
\newcommand{\manfN}{\mathcal{N}}
\newcommand{\manfV}{\mathcal{V}}
\newcommand{\manfA}{\mathcal{A}}
\newcommand{\manfE}{\mathcal{E}}
\newcommand{\manfS}{\mathcal{S}}

\newcommand{\VMO}{\mathrm{VMO}}

\newcommand{\SO}{\mathrm{SO}}
\newcommand{\SU}{\mathrm{SU}}
\newcommand{\RP}{\mathbb{R}\mathrm{P}}
\newcommand{\CP}{\mathbb{C}\mathrm{P}}
\newcommand{\CC}{\mathbb{C}}
\newcommand{\loc}{_\mathrm{loc}}
\newcommand{\weak}{_\mathrm{w}}

\DeclareMathOperator{\supp}{supp}

\DeclareMathOperator{\jac}{Jac} 

\DeclareMathOperator{\Diam}{diam}
\DeclareMathOperator{\Trace}{Tr}

\DeclareMathOperator{\Ker}{ker}
\DeclareMathOperator{\cte}{c}

\newcommand{\Id}{\mathrm{Id}}

\DeclareMathOperator{\Int}{int}
\DeclareMathOperator{\hur}{\mathfrak{hur}}
\DeclareMathOperator{\Hur}{Hur}
\DeclareMathOperator{\Hopf}{Hopf}
\DeclareMathOperator{\Deg}{deg}
\DeclareMathOperator{\hopf}{\mathfrak{hopf}}
\DeclareMathOperator{\Mod}{Mod}
\DeclareMathOperator{\Det}{det}
\newcommand{\Jacobian}[2]{\mathscr{J}_{#1}#2} 
\newcommand{\Lip}{\mathrm{Lip}}
\newcommand{\dif}{\,\mathrm{d}}
\newcommand{\dext}{{}\mathrm{d}}
\newcommand{\dco}{\updelta}
\newcommand{\Matrix}[1]{[#1]}
\newcommand{\Tangent}[1]{\mathrm{T}_{#1}}

\newcommand{\Forms}{\Lambda}
\newcommand{\Maximal}{\mathscr{M}}

\usepackage{color}

\makeindex

\begin{document}

\frontmatter

\title{Generic topological screening and approximation of Sobolev~maps}



\author[P. Bousquet]{Pierre Bousquet}

\address{
Pierre Bousquet\hfill\break\indent
Universit{\'e} de Toulouse \hfill\break\indent
Institut de Math\'ematiques de Toulouse, UMR CNRS 5219\hfill\break\indent
Universit\'e Paul Sabatier Toulouse 3\hfill\break\indent 
118 Route de Narbonne\hfill\break\indent
31062 Toulouse Cedex 9\hfill\break\indent
France}




\author[A. C. Ponce]{Augusto C. Ponce}

\address{
Augusto C. Ponce\hfill\break\indent
Universit{\'e} catholique de Louvain\hfill\break\indent
Institut de Recherche en Math{\'e}matique et Physique\hfill\break\indent
Chemin du cyclotron 2, bte L7.01.02\hfill\break\indent
1348 Louvain-la-Neuve\hfill\break\indent
Belgium}




\author[J. Van Schaftingen]{Jean Van Schaftingen}

\address{
Jean Van Schaftingen\hfill\break\indent
Universit{\'e} catholique de Louvain\hfill\break\indent
Institut de Recherche en Math{\'e}matique et Physique\hfill\break\indent
Chemin du cyclotron 2, bte L7.01.02\hfill\break\indent
1348 Louvain-la-Neuve\hfill\break\indent
Belgium}







\begin{abstract}
This manuscript develops a framework for the strong approximation of Sobolev maps with values in compact manifolds, emphasizing the interplay between local and global topological properties. Building on topological concepts adapted to \(\VMO\) maps, such as homotopy and the degree of continuous maps, it introduces and analyzes extendability properties, focusing on the notions of \(\ell\)-extendability and its generalization, \((\ell, e)\)-extendability. 

We rely on Fuglede maps, providing a robust setting for handling compositions with Sobolev maps. Several constructions --- including opening, thickening, adaptive smoothing, and shrinking --- are carefully integrated into a unified approach that combines homotopical techniques with precise quantitative estimates.

Our main results establish that a Sobolev map \(u \in \Sobolev^{k, p}\) defined on a compact manifold of dimension \(m > kp\) can be approximated by smooth maps if and only if \(u\) is \((\floor{kp}, e)\)-extendable with \(e = m\). When \(e < m\), the approximation can still be carried out using maps that are smooth except on structured singular sets of rank \(m - e - 1\).
\end{abstract}

\maketitle

\tableofcontents


\mainmatter

\cleardoublepage
\chapter{Introduction}
\label{chapterIntroduction}

Given \(1\leq p<\infty\) and  \(m, \nu \in \N_*\), the vector-valued Sobolev space \(\Sobolev^{1, p} (\Ball^m; \R^\nu)\), defined on the open unit ball \(\Ball^m\subset\R^m\) centered at \(0\), is a powerful framework for problems in calculus of variations and partial differential equations. 
When they involve constraints on the values of the admissible maps, manifold-valued Sobolev spaces come into play. 
More precisely, given a compact smooth manifold \(\manfN\) of dimension \(n\) isometrically embedded in  \(\R^\nu\), we define the class of Sobolev maps from \(\Ball^m\) with values into \(\manfN\) as 
\begin{equation}\label{def-Sob-1}
\Sobolev^{1,p}(\Ball^m; \manfN) = \Bigl\{ u \in \Sobolev^{1,p}(\Ball^m; \R^\nu) : u(x) \in \manfN\ \text{for every \(x \in \Ball^{m}\)} \Bigr\}.
\end{equation}
The role of these nonlinear Sobolev  spaces is prominent in the mathematical analysis of \emph{harmonic maps} in Riemannian geometry \cites{HardtLin1987, LinWang2008, Moser2005, Schoen-Uhlenbeck,Jost2011,Helein2008,EellsSampson1964,Brezis2003,EellsLemaire1995,EellsFuglede2001, Simon1996, GiaquintaMucci2006}. 
These  spaces are also the natural setting for the study of \emph{ordered media} in condensed matter physics: \emph{superconductors} \cites{HoffmannTang2001,Mermin1979}, 
\emph{superfluids} \cite{Mermin1979}, \emph{ferromagnetism} \cites{HubertSchaefer1998,Mermin1979} and  \emph{liquid crystals}  \cites{Mermin1979,BallZarnescu2011,Brezis1991,Mucci2012,Coronetal1991}. 
They are  involved in other areas of physics as well, such as Cosserat materials in \emph{elasticity} \cite{Neff2004} or \emph{gauge theories}   \cite{Lieb1993}. 
Finally, they arise in the analysis of \emph{framefields} (or cross-fields) in \emph{computer graphics} and in \emph{meshing algorithms} for numerical computation of solutions to partial differential equations \cites{HuangTongWeiBao2011,Li-Liu-Yang-Yu-Wang-Guo,Bernard_Remacle_Kowalski_Geuzaine}.

When energy integrals or partial differential equations involve higher-order derivatives, the natural functional framework becomes the  Sobolev space \(\Sobolev^{k,p}(\Ball^m;\manfN)\) with \(k\in \N_*\), which is defined in analogy with \eqref{def-Sob-1} as a subset of \(\Sobolev^{k,p}(\Ball^m;\R^\nu)\) instead of \(\Sobolev^{1,p}(\Ball^m;\R^\nu)\). 
In this regard, we mention biharmonic maps \cites{ChangWangYang1999, Moser2008, Scheven2008, Struwe2008, Urakawa2011, Montaldo-Oniciuc-2006}, minima of higher order energies such as polyharmonic maps with values into manifolds \cites{EellsSampson1966, AngelsbergPumberger2009, GastelScheven2009, GoldsteinStrzeleckiZatorska2009, GongLammWang2012, LammWang2009, Branding-Montaldo-Oniciuc-Ratto-2020, He-Jiang-Lin, Gastel-2016}, and also biharmonic heat flows \cites{Lamm-2004, Gastel-2006} or biharmonic wave maps \cites{Herr-Lamm-Schmid-Schnaubelt-2020, Herr-Lamm-Schmid-Schnaubelt-2020-1}.

Many natural questions related to functional analysis or topology arise in the setting of manifold-valued Sobolev spaces:
\begin{description}
    \item[\emph{Approximation problem}] 
    Decide whether a map in \(\Sobolev^{k,p}(\Ball^m;\manfN)\)
    may be strongly or weakly approximated by smooth maps with values in \(\manfN\)\,;
    \item[\emph{Trace problem}] 
    Characterize the maps \(u|_{\partial \Ball^m}\) that are obtained as the traces of maps \(u\in \Sobolev^{k,p}(\Ball^m;\manfN)\)\,;
    \item[\emph{Lifting problem}] 
    Given a covering \(\pi \colon \manfE \to \manfN\) defined on a Riemannian manifold \(\manfE\), decide whether a given map \(u\in \Sobolev^{k,p}(\Ball^m; \manfN)\) may be written as \(u = \pi\compose v\) for some \(v\in \Sobolev^{k,p}(\Ball^m;\manfE)\)\,;
    \item[\emph{Homotopy problem}] 
    Determine whether two maps in \(\Sobolev^{k,p}(\Ball^m;\manfN)\) are connected by a path.
\end{description}
Although we have mentioned these problems for maps defined in the ball \(\Ball^{m}\), they make sense and are interesting for maps defined more generally on a Riemannian manifold \(\manfM\).

We focus in this monograph on the first question related to the  approximation by smooth maps. 
In classical (linear) functional analysis, this is a standard step as one typically deduces properties of Sobolev functions that are inherited from smooth functions using a density argument.
As for real-valued functions, vector-valued smooth maps are strongly dense in \(\Sobolev^{k,p}(\Ball^m; \R^\nu)\) with respect to the distance 
\begin{equation}
\label{eq-Introduction-distance}
d_{\Sobolev^{k, p}}(u, v)
\vcentcolon = \norm{u - v}_{\Lebesgue^{p}(\Ball^{m})} + \sum_{j = 1}^{k}{\norm{D^{j}u - D^{j}v}_{\Lebesgue^{p}(\Ball^{m})}}.
\end{equation}
In particular any element of \(\Sobolev^{k, p}(\Ball^m; \manfN)\) can be approximated by mappings in \(\Smooth^\infty(\cBall^m; \R^\nu)\).
Such mappings however need not take their values in \(\manfN\),
which leads to the following first fundamental question: 

\begin{question}\label{question-1}
Is \(\Smooth^\infty(\cBall^m; \manfN)\) strongly dense in \(\Sobolev^{k, p}(\Ball^m; \manfN)\)?
\end{question}

This question has not just an intrinsic interest, but also plays a role in the other problems above. 
For instance, the answers to the trace and the homotopy problems rely in some cases on the possibility to approximate a Sobolev map by smooth ones. 
In addition, the tools themselves to answer the approximation problem have also been exploited in the other ones.

Unlike the classical case, the answer to Problem~\ref{question-1} may be negative due to topological constraints coming from \(\manfN\):

\begin{example}
	\label{exampleIntroduction1}
	The map \(u  \colon  \Ball^{m} \to \Sphere^{m - 1}\) defined for \(x \ne 0\) by \(u(x) = x/\abs{x}\) belongs to \(\Sobolev^{k, p}(\Ball^m; \Sphere^{m - 1})\) but cannot be strongly approximated by a sequence of maps in \(\Smooth^\infty(\cBall^m; \manfN)\) for \(m-1 < kp < m\).{}
	Indeed, assume by contradiction that there exists a sequence \((u_{j})_{j \in \N}\) in this set that converges to \(u\) in \(\Sobolev^{k, p}\).{}
	Then, by a Fubini-type argument, for almost every \(0 < r < 1\), there exists a subsequence of restrictions \(u_{j_{i}}|_{\partial B_{r}^{m}}\) to the sphere \(\partial B_{r}^{m}\) of radius \(r\) that converges to \(u|_{\partial B_{r}^{m}}\) in \(\Sobolev^{k, p}\).
	We then have a contradiction since the maps \(u_{j_{i}}|_{\partial B_{r}^{m}}\) have a continuous extension to \(\overline{B}{}_{r}^{m}\) with values into \(\Sphere^{m - 1}\) and converge uniformly to \(u|_{\partial B_{r}^{m}}\), while this map does not have such an extension.
\end{example}

Let us denote by \(\Hilbert^{k,p}(\Ball^m; \manfN)\) the closure of \(\Smooth^{\infty}(\cBall^m; \manfN)\) in \(\Sobolev^{k, p}(\Ball^m; \manfN)\) with respect to the distance \eqref{eq-Introduction-distance}.
We wish to know whether \(\Hilbert^{k,p}(\Ball^m; \manfN)\) coincides with \(\Sobolev^{k, p}(\Ball^m; \manfN)\) for suitable choices of \(k\) and \(p\).
The first affirmative case concerns the range \(kp \ge m\) and goes back to Schoen and Uhlenbeck~\cite{Schoen-Uhlenbeck}*{Section~4, Proposition}:

\begin{theorem}
	\label{theoremSchoen-Uhlenbeck}
	If \(kp \ge m\), then \(\Hilbert^{k,p}(\Ball^m; \manfN) = \Sobolev^{k,p}(\Ball^m; \manfN)\).
\end{theorem}

Here is a sketch of the argument:
Given \(u \in \Sobolev^{k, p}(\Ball^m; \manfN)\), consider the convolution \(\varphi_\epsilon \ast u\) with a smooth kernel \(\varphi_\epsilon (x) = \varphi(x/\epsilon) / \epsilon^m\). 
When the range of \(\varphi_\epsilon \ast u\) lies in a small tubular neighborhood of \(\manfN\), one may project \(\varphi_\epsilon \ast u\) pointwise into \(\manfN\). 
That is always the case for \(\epsilon\) sufficiently small as long as \(kp \ge m\). 
Indeed, for \(kp > m\), the space \(\Sobolev^{k, p}(\Ball^m; \manfN)\) is continuously imbedded in \(\Smooth^0(\cBall^m; \manfN)\), therefore \(\varphi_\epsilon \ast u\) converges uniformly to \(u\).
In particular, the distance of \(\varphi_{\epsilon} \ast u\) to \(\manfN\), namely \(d(\varphi_\epsilon \ast u, \manfN)\), converges uniformly to \(0\).
For \(kp = m\), one has the weaker property that \(\Sobolev^{k, p}\) imbeds into the space of functions of vanishing mean oscillation \(\VMO\).
As a result, \(\varphi_{\epsilon} \ast u\) need not converge uniformly to \(u\), but one still has that \(d(\varphi_\epsilon \ast u, \manfN)\) converges uniformly to \(0\), see \cite{BrezisNirenberg1995} or Lemma~\ref{lemmaVMOUniformConvergence} below, which entitles one to conclude as before.

When \(kp < m\), Example~\ref{exampleIntroduction1} illustrates that an obstruction to the approximation by smooth maps may arise from the topology of the target manifold \(\manfN\), namely \(\Sphere^{m - 1}\) in the above example.
In fact, it is possible to identify a topological assumption in \(\manfN\) that prevents such an obstruction:

\begin{theorem}
	\label{theoremIntroductionBethuel} Assume that \(kp<m\).
	Then, \(\Hilbert^{k,p}(\Ball^m; \manfN) = \Sobolev^{k, p}(\Ball^{m}; \manfN)\)  if and only if \(\pi_{\floor{kp}}(\manfN) \simeq \{0\}\), where \(\floor{kp}\) denotes the integral part of \(kp\).
\end{theorem}

The condition \(\pi_{\floor{kp}}(\manfN) \simeq \{0\}\) of triviality of the \(\floor{kp}\)th homotopy group \(\pi_{\floor{kp}}(\manfN)\) means that any continuous map \(f\) from the \(\floor{kp}\) dimensional sphere \(\Sphere^{\floor{kp}}\) into \(\manfN\) is homotopic to a constant in \(\Smooth^{0}(\Sphere^{\floor{kp}}; \manfN)\) or equivalently that \(f\) has a continuous extension \(\overline{f}  \colon  \cBall^{\floor{kp}} \to \manfN\).{}
The above theorem is due to Bethuel \cite{Bethuel} when \(k = 1\), see also \cite{Hajlasz} for an alternative approach when \(\manfN\) is \(\floor{p}\)-connected, namely \(\pi_j(\manfN)\simeq \{0\}\) for every \(j \in \{0, \dots, \floor{p}\}\). The higher order case \(k \ge 2\) was later established by the authors in \cite{BPVS_MO}.

The reader might be intrigued by the role of the integer \(\floor{kp}\) in the previous theorem. 
This can be clarified by sketching a proof of the converse implication of Theorem~\ref{theoremIntroductionBethuel}, following the approach of Example~\ref{exampleIntroduction1} when \(kp\) is noninteger:

\begin{example}
	\label{exampleIntroduction1bis}
Take any \(f \in \Smooth^\infty(\Sphere^{\floor{kp}}; \manfN)\). 
The map \(u  \colon  \Ball^m \to \manfN\) defined for \(x = (x', x'') \in \R^{\floor{kp} + 1} \times \R^{m - \floor{kp} - 1} \) and \( x'\ne 0 \) by
\begin{equation}
\label{exampleIntroductionClasseR}
u(x) = f\Bigl(\frac{x'}{\abs{x'}}\Bigr)
\end{equation}
belongs to \(\Sobolev^{k, p}(\Ball^m; \manfN)\) since \(kp < \floor{kp} + 1\). 
If there exists a sequence \((u_j)_{j \in \N}\) in \(\Smooth^\infty(\cBall^m; \manfN)\) strongly converging to \(u\) in \(\Sobolev^{k, p}\), then some subsequence \((u_{j_{i}})_{i \in \N}\) converges uniformly in \(\partial B_{r}^{\floor{kp} + 1} \times \{a''\}\) for almost every \(r\) and \(a''\). 
This implies that there exists a sequence of smooth maps on \(\cBall^{\floor{kp} + 1}\) with values in \( \manfN\) that converge uniformly to \(f\) on \(\Sphere^{\floor{kp}}\) and then one deduces that \(f\) is homotopic to a constant in \(\Smooth^{0}(\Sphere^{\floor{kp}}; \manfN)\). 
\end{example}

When \(\pi_{\floor{kp}}(\manfN) \not\simeq \{0\}\), one has
\(
\Hilbert^{k,p}(\Ball^m; \manfN) \subsetneqq \Sobolev^{k, p}(\Ball^{m}; \manfN).
\)
The strong approximation problem may then be reformulated as

\begin{question}\label{question-2}
Identify all maps in \(\Hilbert^{k,p}(\Ball^m; \manfN)\).
\end{question}

This task is not straightforward, as \(\Hilbert^{k,p}(\Ball^m; \manfN)\) contains elements that present genuine discontinuities, that is, those that cannot be removed by modifying the maps in a negligible set.

\begin{example}
	\label{exampleIntroduction2}
	Let \(u  \colon  \Ball^{m} \to \manfN\) be defined for \(x \ne 0\) by 
	\[{}
	u(x) = f\Bigl(\frac{1}{\abs{x}^{\alpha}}\Bigr),
	\]
	where \(f  \colon  [0, \infty) \to \manfN\) is smooth and \(\alpha > 0\).{}
	For \((\alpha + k)p < m\), we have \(u \in \Sobolev^{k, p}(\Ball^{m}; \manfN)\).{}
	Regardless of \(\pi_{\floor{kp}}(\manfN)\), the map \(u\) can be approximated by the sequence \((u_j)_{j \in \N}\) given for every  \( x \in  \cBall^{m} \) by
	\[{}
	u_{j}(x) = f(\varphi_{j}(x)),
	\]
	where \(\varphi_{j}  \colon  \cBall^{m} \to [0, \infty)\) is a smooth approximation of \(1/|x|^{\alpha}\) in \(\Sobolev^{k, p}(\Ball^{m}; \R)\).
\end{example}

This example shows the existence of discontinuities that do not forestall the approximability by smooth maps. 
We tend to see those singularities as purely \emph{analytical}, in contrast with those of Example~\ref{exampleIntroduction1} which yield \emph{topological} obstructions to the approximation.
The map \eqref{exampleIntroductionClasseR} in Example~\ref{exampleIntroduction1bis} exhibits an analytical singularity when \(f\) is nonconstant, while the singularity becomes topological when \(f\) is not homotopic to a constant in \(\Smooth^{0}(\Sphere^{\floor{kp}}; \manfN)\). 
To distinguish these two types of singularity, it is possible to formulate suitable criteria which are based on the notion of genericity that we pursue in this work.

So far, to detect a topological singularity it was enough to consider restrictions on spheres enclosing the point of discontinuity.
The following example based on a \emph{dipole} construction shows that such an approach does not generally work:

\begin{example}
	Given distinct points \(a, b \in \Ball^m\), let \(u \in \Sobolev^{k, p}(\Ball^{m}; \Sphere^{m - 1})\) with \(kp < m\) be a smooth map on \(\Ball^m \setminus \{a, b\}\) such that \(u|_{\partial B_{r}^{m}(a)}\) has degree \(1\) and \(u|_{\partial B_{r}^{m}(b)}\) has degree \(-1\) for every \(0 < r < |a - b|\).
	In particular, \(u|_{\partial{B_{r}^{m}}}\) has degree zero for any ball \(B_r^m\) that contains both \(a\) and \(b\), but proceeding as in Example~\ref{exampleIntroduction1} in a neighborhood of \(a\), one verifies that \(u\) cannot be approximated by smooth maps for \(m-1 < kp < m\).{} 
 
	An explicit example of such a map with \(k = 1\) is given for \(x = (x', x'')  \in \R \times \R^{m - 1}\) by
    \[
	u(x) = f\Bigl( \frac{x''}{\rho - |x'|} \Bigr)
	\]
	where \(0 < \rho < 1\) and \(f  \colon  \R^{m-1} \to \Sphere^{m-1}\) is a suitable smooth function that is constant on \(\R^{m - 1} \setminus \Ball^{m - 1}\).{}
    In this case, \(a = (-\rho, 0'')\) and \(b = (\rho, 0'')\).{}
    This map is constant outside the set of points \(x = (x', x'')\) such that \(|x'| + |x''| < \rho\) and,	by a scaling argument, one has \(\norm{D u}_{\Lebesgue^{p}(\Ball^{m})} \le C\rho^{\frac{m}{p} - 1}\).
    As the right-hand side of the estimate converges to zero when \(\rho \to 0\), it is possible to glue an infinite number of dipole-type maps of this form.
	The resulting map \(U\) has infinitely many topological singularities.
	To detect all of them, one then needs to restrict \(U\) to an infinite number of spheres. 
\end{example}

The previous example shows the necessity of considering enough restrictions to identify the topological obstructions to approximation.
However, not all restrictions may be suitable to rule out approximability by smooth maps:

\begin{example}
	For every \(1 \le p < 2\), there exists \(u \in \Hilbert^{1, p}(\Ball^{2}; \Sphere^{1})\) such that \(u(x) = x/|x|\) for every \(x \in \partial B_{1/2}^{2}\) and also in the sense of traces.
	Indeed, the argument function \(x \in \partial B_{1/2}^{2} \mapsto \arg{(x/|x|)} \in \R\) belongs to the trace space \(\Sobolev^{1-1/p, p}\) when \(p > 1\) or \(\Lebesgue^1\) when \(p = 1\).
    It is therefore the trace of \(\Sobolev^{1, p}\)~\emph{real-valued} functions \(v_{1}\) in \(B^{2}_{1/2}\) and  \(v_{2}\) in \(\Ball^{2} \setminus \overline{B}{}^{2}_{1/2}\)\,.{}
	Then, 
	\[{}
	v \vcentcolon={}
	\begin{cases}
		v_{1} & \text{in \(B^{2}_{1/2}\)\,,}\\
		v_{2} & \text{in \(\Ball^{2} \setminus \overline{B}{}^{2}_{1/2}\)\,,}
	\end{cases}
	\]
	belongs to \(\Sobolev^{1, p}(\Ball^{2}; \R)\) and it suffices to take \(u = (\cos{v}, \sin{v})\).{}
	Note that \(u\) belongs to \(\Sobolev^{1, p}(\Ball^{2}; \Sphere^{1})\), and even \(\Hilbert^{1, p}(\Ball^{2}; \Sphere^{1})\), since \(v\) can be classically approximated by smooth real-valued functions.
	However, the trace of \(u\) in \(\partial B_{1/2}^{2}\) has topological degree \(1\) and is not homotopic to a constant map in \(\Smooth^{0}(\partial B_{1/2}^{2}; \Sphere^{1})\). 
\end{example}

\begin{example}[Variant of the previous example]
For every \(1 \le p < 2\) and every \(d \in \mathbb{Z}\), there exists \(u \in \Hilbert^{1, p}(\Ball^{2}; \Sphere^{1})\) such that \(u \in \Smooth^\infty (\cBall^2 \setminus \{a\}; \Sphere^1)\) and,  for every \(x \in \partial B_{1/2}^{2} \setminus \{a\}\), 
\[
u(x) = f((x - a)/|x - a|),
\] 
where \(a \vcentcolon= (0, -1/2)\) and \(f\in \Smooth^{\infty}(\Sphere^1;\Sphere^1)\) has topological degree \(d\).
Indeed, we start with a function \(\psi \in \Smooth^{\infty}(\Sphere^1; \R)\) such that, for every \(0 \le \theta \le \pi\),
\[
\psi(\cos{\theta}, \sin{\theta}) = 2\theta d.
\]
We then define \(u = (\cos{\varphi}, \sin{\varphi})\)
where, for \(x \in \Ball^{2} \setminus \{a\}\),
\[
\varphi (x) = \psi ((x - a)/\abs{x - a}).
\]
Since \(1 \le p < 2\), we have \(\varphi \in \Sobolev^{1, p}(\Ball^{2}; \R)\) and the classical Sobolev approximation of \(\varphi\) implies that \(u \in \Hilbert^{1, p}(\Ball^{2}; \Sphere^{1})\).
On the other hand, the restriction of \(u\) to  \(\partial B_{1/2}^{2} \setminus \{a\}\) can be extended as a continuous map of degree \(d\) on \(\partial B_{1/2}^{2} \)\,.
\end{example}

In the previous examples, the maps are continuous except at finitely or countably many points, which implies that their restrictions to \emph{generic} spheres remain continuous. 
This allows us to utilize classical topological tools, such as the existence of continuous extensions (equivalent to the existence of a homotopy with a constant map), to identify an obstruction to smooth approximation.
When \(kp < m\), an arbitrary \(\Sobolev^{k, p}\) map \(u\) may exhibit genuine discontinuities. 
However, given a \(\floor{kp}\)-dimensional surface \(\manfS\), a Morrey-type property ensures that for almost every choice of \(\xi\) the restrictions of \(u\) to the translated surfaces \(\manfS + \xi\) belong to \(\Sobolev^{k, p}\).
When \(kp \notin \N\), this space embeds continuously into \(\Smooth^{0}\), hence it still holds that the restrictions of \(u\) to \emph{generic} \(\floor{kp}\)-dimensional surfaces coincide almost everywhere with a continuous map, provided that the term ``generic'' is appropriately interpreted. 

The examples so far deal with genericity involving spheres, although we could also have relied on different geometric objects like cubes or simplices in the domain.
From this observation, it seems enough to have a notion of genericity based on almost every translation and scaling of these objects.
While this approach has yielded valuable insight, it is less clear though how one can transfer some information gathered from one class of generic objects to another.
For instance, when a map is homotopic to a constant on generic spheres, how can one make sure that the same holds on generic cubes or simplices?
Such an \emph{ad hoc} approach is also strongly dependent on the Euclidean geometry of \(\R^{m}\), and would require some further adaptation to manifolds.{}

We therefore aim for a different strategy based on a novel perspective --- \emph{genericity by composition}.
Our approach revolves around the ingenious utilization of summable functions, which enables us to formulate a generic composition with Lipschitz maps defined on various geometric objects.
The heart of the matter relies on the following result that we prove in greater generality in Section~\ref{sectionFugledeSobolev}:

\begin{proposition}
		\label{propositionIntroductionFugledeGeneric}
	Let \(\manfV\) be a compact Riemannian manifold or an open subset of \(\R^{m}\).
    Given \(v \in \Sobolev^{1, p}(\manfV)\), there exists a summable function \(w  \colon  \manfV \to [0, +\infty]\) such that, for every \(\ell \in \N_{*}\) and every Lipschitz map \(\gamma  \colon  \Sphere^{\ell} \to \manfV\) with 
	\begin{equation}
		\label{eqIntroductionFugledeGeneric}
	\int_{\Sphere^{\ell}} w \compose \gamma \dif\cH^{\ell} < \infty\text{,}
	\end{equation}
	we have \(v \compose \gamma \in \Sobolev^{1, p}(\Sphere^{\ell})\) and
	\[{}
	D(v \compose \gamma)(x) = Dv(\gamma(x)) [D\gamma(x)]
	\quad \text{for \(\cH^{\ell}\)-almost every \(x \in \Sphere^{\ell}\).}
	\]
\end{proposition}

The summable function \(w\) acts as a \emph{detector}, allowing us to screen various families of Lipschitz maps \(\gamma\) and identify many for which the composition \(v \compose \gamma\) is a Sobolev function.
We call any \(\gamma\) verifying property \eqref{eqIntroductionFugledeGeneric} a \emph{Fuglede map} associated to \(w\).
Our approach, exemplified by Proposition~\ref{propositionIntroductionFugledeGeneric}, is reminiscent of the concept of a property being satisfied for almost every measure, introduced by Fuglede~\cite{Fuglede} through his notion of modulus of a family of measures.
However, our method diverges from that of his, as we explain in more detail in Section~\ref{section_Fuglede}.

This Morrey-type property associated to Fuglede maps encapsulates the essence of generic composition, shedding light on the behavior of Sobolev functions under compositions with Lipschitz maps. 
In this formalism, when \(\manfV = \R^{m}\), the restriction of a function \(v\) to the sphere \(\partial B_{r}^{m}(a)\) can be naturally identified with the composition \(v \compose \gamma\), where \(\gamma  \colon  x \in \Sphere^{m - 1} \mapsto a + rx \in \partial B_{r}^{m}(a)\).
From Proposition~\ref{propositionIntroductionFugledeGeneric}, we have that, for every \(r > 0\) and every \(a \in \R^{m}\) with
\begin{equation}
\label{eqIntroductionFugledeGenericExample}
\int_{\Sphere^{m-1}} w(a + rx) \dif\cH ^{m - 1}(x) = \int_{\Sphere^{m- 1}} w \compose \gamma \dif\cH ^{m-1} < \infty,
\end{equation}
the function \(x \in \Sphere^{m - 1} \mapsto v(a + rx)\) belongs to \(\Sobolev^{1, p}(\Sphere^{m - 1})\) or, equivalently, the restriction \(v|_{\partial B_{r}^{m}(a)}\) belongs to \(\Sobolev^{1, p}(\partial B_{r}^{m}(a))\).
That the genericity condition \eqref{eqIntroductionFugledeGeneric} covers genericity under restrictions can now be seen as a consequence of Tonelli's theorem.
Indeed, condition \eqref{eqIntroductionFugledeGenericExample} is verified for a given \(r > 0\) and almost every \(a \in \R^{m}\) as follows
\[{}
\begin{split}
\int_{\R^{m}} \biggl(  \int_{\Sphere^{m-1}} w(a + rx) \dif\cH ^{m - 1}(x) \biggr) \dif a
& = 
\int_{\Sphere^{m-1}} \biggl(  \int_{\R^{m}} w(a + rx) \dif a \biggr) \dif\cH ^{m - 1}(x)\\
& = \cH ^{m - 1}(\Sphere^{m-1}) \int_{\R^{m}} w
< \infty.
\end{split}
\]

Inspired by Morrey-type characterizations, we gain a deeper understanding of the intrinsic properties of Sobolev functions, even when defined on varying geometric domains, such as cubes, simplices in \(\R^m\), or geodesic balls in a manifold \(\manfM\). 
The framework of generic composition, as formulated in \eqref{eqIntroductionFugledeGeneric}, offers a promising alternative to the traditional notion of generic restriction. 
This approach is particularly suitable for the study of Sobolev maps and for identifying elements in \(\Hilbert^{k, p}(\Ball^{m}; \manfN)\). 
Unlike generic restrictions, where the images of simplices or spheres under diffeomorphisms may not remain simplices or spheres, generic compositions retain a greater degree of flexibility, particularly in relation to compositions with Lipschitz maps. 
This flexibility improves the applicability of topological criteria in the analysis of Sobolev maps.

A counterpart of Proposition~\ref{propositionIntroductionFugledeGeneric} also holds for \(\Sobolev^{k, p}\)~maps under stronger regularity of \(\gamma\), but we do not pursue this direction for the moment, see Proposition~\ref{propositionFugledeWkp}.
Instead, we wish to enjoy the flexibility of composition with Lipschitz maps without losing the topological properties that are related to the \(\Sobolev^{k, p}\)~scale for \(\manfN\)~valued maps.{}
To this aim, we rely, since \(\manfN\) is bounded, on the Gagliardo-Nirenberg interpolation \(\Sobolev^{k, p} \cap \Lebesgue^{\infty} \subset \Sobolev^{1, kp}\) that emphasizes the role of the quantity \(kp\),  justifies the \(\Sobolev^{1, kp}\)~nature of maps in \(\Sobolev^{k, p}(\Ball^{m}; \manfN)\), and yields the imbedding 
\[{}
\Hilbert^{k, p}(\Ball^{m}; \manfN) \subset \Hilbert^{1, kp}(\Ball^{m}; \manfN).
\]
Proposition~\ref{propositionIntroductionFugledeGeneric} applied to \(u \in \Sobolev^{k, p}(\Ball^{m}; \manfN)\) with \(q = kp\) and \(\ell = \floor{kp}\) thus implies that \(u \compose \gamma \in \Sobolev^{1, kp}(\Sphere^{\floor{kp}}; \manfN)\) for every Lipschitz map \(\gamma  \colon  \Sphere^{\floor{kp}} \to \Ball^{m}\) satisfying \eqref{eqIntroductionFugledeGeneric}.
We then deduce that \(u \compose \gamma\) is \(\VMO\), and is even equal almost everywhere to a continuous map when \(kp\) is not an integer.

From the pioneering work of Brezis and Nirenberg~\cite{BrezisNirenberg1995}, one can extend to the \(\VMO\) setting many topological properties from the classical \(\Smooth^{0}\)~setting.{}
For example, a \(\VMO\)~map \(u \compose \gamma\) is homotopic to a constant in \(\VMO(\Sphere^{\floor{kp}}; \manfN)\) whenever there exists a continuous path \(H  \colon  [0, 1] \to \VMO(\Sphere^{\floor{kp}}; \manfN)\) such that \(H(0) = u \compose \gamma\) and \(H(1)\) is a constant map.
This notion agrees with the conventional approach in topology: If \(u \compose \gamma\) is continuous and homotopic to a constant in the \(\VMO\)~sense, then \(u \compose \gamma\) is also homotopic to a constant in \(\Smooth^{0}(\Sphere^{\floor{kp}}; \manfN)\).
By the ability of working with \(\VMO\)~maps, we are well equipped to capture topological properties of \(u \compose \gamma\) for generic maps \(\gamma  \colon  \Sphere^{\floor{kp}} \to \Ball^m\) regardless of whether \(kp\) is an integer or not.

Based on condition \eqref{eqIntroductionFugledeGeneric}, we may now introduce the fundamental concept of generic extension:
We say that \(u \in \Sobolev^{k, p}(\Ball^{m}; \manfN)\) is \emph{\(\floor{kp}\)-extendable} whenever, for any generic Lipschitz map \(\gamma  \colon  \Sphere^{\floor{kp}} \to \Ball^{m}\), the map \(u \compose \gamma\) belongs to \(\VMO(\Sphere^{\floor{kp}}; \manfN)\) and is homotopic to a constant in \(\VMO\).{}
Here, generic means that there exists a summable function \(w  \colon  \Ball^{m} \to [0, +\infty]\) such that the assertion above is verified for every \(\gamma\) for which \eqref{eqIntroductionFugledeGeneric} holds.
This setting provides a screening toolbox where \(w\) detects enough maps that can be used to decide whether \(u\) belongs or not to \(\Hilbert^{k, p}(\Ball^m; \manfN)\).
Indeed, one of the main results of this work is

\begin{theorem}
\label{theoremhkpGreatBall}
Let \(kp < m\) and \(u \in \Sobolev^{k,p}(\Ball^m; \manfN)\). 
Then, \(u \in \Hilbert^{k, p}(\Ball^{m}; \manfN)\) if and only if \(u\) is \(\floor{kp}\)-extendable. 
\end{theorem}

It is natural to expect that the constructions of an approximating sequence can be suitably adapted to inherit topological properties satisfied by the initial map \(u\).
The tools introduced in our work, particularly the concept of Fuglede maps, offer a well-adapted framework for encapsulating the genericity of topological properties. 
In this respect, Theorem~\ref{theoremhkpGreatBall} includes as a special case Theorem~\ref{theoremIntroductionBethuel} since all maps \(u \compose \gamma\) are necessarily homotopic to a constant in \(\VMO\) for generic Lipschitz maps \(\gamma  \colon  \Sphere^{\floor{kp}} \to \Ball^{m}\) when \(\pi_{\floor{kp}}(\manfN) \simeq \{0\}\).
The possibility of screening for topological obstructions to the approximation of a \emph{single} given map \(u \in \Sobolev^{k, p}(\Ball^m; \manfN)\) without imposing a further assumption on the topology of \(\manfN\) is one of the main novelties of our present work.

\begin{example}
	Let \(u  \colon  \Ball^{m} \to \manfN\) be defined for \(x \ne 0\) by 
	\[{}
	u(x) = f\Bigl( \frac{x}{|x|} \Bigr),
	\]
	where \(f \in \Smooth^{\infty}(\Sphere^{m - 1}; \manfN)\) is homotopic to a constant in \(\Smooth^{0}(\Sphere^{m-1}; \manfN)\).
	Then, \(u \in \Sobolev^{k, p}(\Ball^{m}; \manfN)\) with \(kp < m\) and, as we show in Lemma~\ref{propositionExtensionHomogeneity}, \(u\) is \(\floor{kp}\)-extendable.{}
	To verify the extendability of \(u\) it suffices to take the summable function \(w(x) = 1/|x|^{\floor{kp}}\).{}
	This choice of \(w\) implies that, for every Lipschitz map \(\gamma  \colon  \Sphere^{\floor{kp}} \to \Ball^{m}\) satisfying the genericity condition \eqref{eqIntroductionFugledeGeneric}, we have \(0 \not\in \gamma( \Sphere^{\floor{kp}})\). 
	Hence, \(u \compose \gamma\) is smooth and, by the assumption on \(f\), is homotopic to a constant in \(\Smooth^{0}(\Sphere^{\floor{kp}}; \manfN)\).{}
	Since \(u\) is \(\floor{kp}\)-extendable, by Theorem~\ref{theoremhkpGreatBall} we deduce that \(u \in \Hilbert^{k, p}(\Ball^{m}; \manfN)\), regardless of \(\pi_{\floor{kp}}(\manfN)\).
\end{example}

As a consequence of Theorem~\ref{theoremhkpGreatBall}, the problem of determining whether a Sobolev map \(u\) can be approximated by smooth maps reduces to verifying its \(\floor{kp}\)-extendability. It is therefore useful to identify conditions that ensure this property while minimizing the number of candidates required to test genericity. Using the notion of generic restrictions on simplices, in Section~\ref{sectionRestrictionsGeneric} we establish the following criterion:

\begin{theorem}\label{propositionSimplexLipschitz-NewBall}
Let \(kp < m\) and \(u \in \Sobolev^{k, p}(\Ball^{m}; \manfN)\).{}
Then, \(u\) is \(\floor{kp}\)-extendable if and only if, for every simplex \(\Simplex^{\floor{kp} + 1} \subset \Ball^{m}\) and for almost every \(\xi \in \R^{m}\) with \(\abs{\xi} < d(\Simplex^{\floor{kp} + 1}, \partial \Ball^{m})\), the map \(u|_{\xi + \partial\Simplex^{\floor{kp} + 1}}\) is homotopic to a constant in \(\VMO(\xi + \partial\Simplex^{\floor{kp} + 1}; \manfN)\).
\end{theorem}

This condition considerably simplifies the verification of \(\floor{kp}\)-extendability: Instead of testing compositions with Lipschitz maps, one only needs to examine restrictions to generic simplices. We further develop this approach in Chapter~\ref{chapterExtendability}.

In some cases, the extendability can also be checked through cohomological computations:

\begin{theorem}\label{propositionJacobianSphereIntro}
Let \(n \le kp < n + 1 \le m\) and \(u \in \Sobolev^{k, p}(\Ball^{m}; \Sphere^n)\).{}
Then, \(u\) is \(\floor{kp}\)-extendable if and only if
\begin{equation*}
\dext (u^{\#} \omega_{\Sphere^n}) = 0
\quad \text{in the sense of currents in \(\Ball^m\),}
\end{equation*}
where \(\omega_{\Sphere^n}\) is a volume form on \(\Sphere^n\) such that \(\int_{\Sphere^n} \omega_{\Sphere^n} = 1\).
\end{theorem}

Since \(n \le kp < n + 1\), we have in this case \(\floor{kp} = n\).
The above result is due to Bethuel~\cite{BethuelH1} for \(k=1\) and \(p=n\), and to Demengel~\cite{Demengel} for \(k=1\), \(1\leq p<2\), and \(n=1\). 
More generally, for \(k=1\) and \(n\leq p<n+1\), a proof is outlined in~\cite{BethuelCoronDemengelHelein}, see also~\cite{Ponce-VanSchaftingenW1p}.
From an algebraic topology perspective, Theorem~\ref{propositionJacobianSphereIntro}, which we establish in Section~\ref{sectionHurewiczGeneric}, is made possible by identifying homotopy classes in \(\pi_{n} (\Sphere^n)\) through the topological degree. 
A more general framework based on the Hurewicz homomorphism is developed in Chapter~\ref{chapter-cohomological-obstructions}.

The notion of \(\floor{kp}\)-extendability can also be defined for maps \(u\in \Sobolev^{k,p}(\manfM;\manfN)\), that is, when the domain is a compact Riemannian manifold \(\manfM\) without boundary instead of the Euclidean ball \(\Ball^{m}\).{}
The difference here is that the generic Lipschitz maps \(\gamma  \colon  \Sphere^{\floor{kp}} \to \manfM\) involved in that condition must have a continuous extension to \(\cBall^{\floor{kp} + 1}\).
This choice is made to avoid interference with the topology of \(\manfM\), although this is insufficient to ensure a satisfactory generalization of Theorem~\ref{theoremhkpGreatBall} due to a global obstruction involving both \(\manfM\) and \(\manfN\).
This issue was first discovered by Hang and Lin~\cite{Hang-Lin}, see Example~\ref{exampleExtensionCP} below: When \(2 \le kp < 3\), every map in \(\Sobolev^{k, p}(\manfM; \RP^3)\) is \(2\)-extendable since \(\pi_{2}(\RP^3) \simeq \{0\}\), nevertheless
\[
\Hilbert^{k, p}(\Ball^4; \RP^3) = \Sobolev^{k, p}(\Ball^4; \RP^3)
\quad \text{and} \quad
\Hilbert^{k, p}(\RP^4; \RP^3) \subsetneqq \Sobolev^{k, p}(\RP^4; \RP^3).
\]

The counterpart of Theorem~\ref{theoremhkpGreatBall} for the manifold \(\manfM\) thus requires a more elaborate version of the extension property that involves simplicial complexes.
Given \(e \in \N\), we recall that a simplicial complex \(\cK^{e}\) is a finite collection of simplices of dimension \(e\) and their faces of dimension \(r \in \{0, \ldots, e-1\}\) in a Euclidean space such that the intersection of two elements in \(\cK^{e}\) still belongs to \(\cK^{e}\), see Definition~\ref{definitionComplex}. 
We then define for \(\ell \in \{0, \ldots, e\}\) the polytope \(K^{\ell}\) as the set of points of all simplices of dimension \(\ell\) in \(\cK^{e}\).{}
Given \(e\geq \floor{kp}+1\), we say that a map \(u \in \Sobolev^{k, p}(\manfM; \manfN)\) is \emph{\((\floor{kp}, e)\)-extendable} whenever, for each simplicial complex \(\cK^{e}\) and each generic Lipschitz map \(\gamma  \colon  K^{e} \to \manfM\), the map \(u \compose \gamma|_{K^{\floor{kp}}}\) belongs to \(\VMO(K^{\floor{kp}}; \manfN)\) and is homotopic in \(\VMO\) to the restriction \(F|_{K^{\floor{kp}}}\) of some map \(F \in \Smooth^{0}(K^{e}; \manfN)\).{}

We show in Chapter~\ref{chapterDecouple} that \(u\) is \((\floor{kp}, m)\)-extendable if and only if \(u\) is \(\floor{kp}\)-extendable and satisfies the additional global topological condition that, for a fixed triangulation \(\cT^{m}\) of \(\manfM\), the restrictions of \(u\) to generic perturbations of its \(\floor{kp}\)-dimensional skeleton \(T^{\floor{kp}}\) are homotopic in \(\VMO\) to continuous maps defined on \(\manfM\), see Theorem~\ref{propositionDecouple}.{}
	Such a condition is satisfied by every \(u \in \Sobolev^{k, p}(\manfM; \manfN)\) for example when there exists \(i \in \{0, \ldots, \floor{kp} - 1\}\) such that
	\[{}
	\pi_0 (\manfM) \simeq \ldots \simeq \pi_i (\manfM) \simeq \{0\}
	\quad \text{and} \quad \pi_{i + 1} (\manfN) \simeq \ldots \simeq \pi_{\floor{kp}} (\manfN) \simeq \{0\},
	\] 
see Corollary~\ref{corollaryHangLin} and also~\citelist{\cite{Hang-Lin}*{Corollary 5.3}\cite{Hajlasz}} in the specific case \(k=1\).

Using the notion of \((\floor{kp}, m)\)-extendability we have a counterpart of Theorem~\ref{theoremhkpGreatBall} for a manifold \(\manfM\)\,:

\begin{theorem}
\label{theoremhkpGreat}
Let \(kp < m\) and \(u \in \Sobolev^{k,p}(\manfM; \manfN)\). 
Then, \(u \in \Hilbert^{k, p}(\manfM; \manfN)\) if and only if \(u\) is \((\floor{kp}, m)\)-extendable. 
\end{theorem}

In particular, only the integer part of the product \(kp\) is relevant for the approximation of smooth maps and we deduce the following, see Corollary~\ref{corollary_h1kp_manifold}:

\begin{corollary}
\label{corollaryHkpH1kp}
Let \(kp < m\) and \(u\in \Sobolev^{k,p}(\manfM;\manfN)\). 
Then, \(u\in \Hilbert^{k,p}(\manfM;\manfN)\) if and only if \(u\in \Hilbert^{1, \floor{kp}}(\manfM;\manfN)\).
\end{corollary} 

The proof of the direct implication is straightforward and follows from the inclusion \(\Hilbert^{k,p}(\manfM;\manfN) \subset\Hilbert^{1, kp}(\manfM;\manfN)\), which in turn is a simple consequence of the Gagliardo-Nirenberg interpolation inequality. 
The converse is far more subtle. 
Indeed, for every \(u\in (\Sobolev^{k,p} \cap \Hilbert^{1, \floor{kp}})(\manfM;\manfN)\), there exists a sequence of smooth maps \((u_j)_{j\in \N}\) that converges to \(u\) in \(\Sobolev^{1, \floor{kp}}(\manfM;\manfN)\) but there is no obvious way to modify \(u_j\) and improve the \(\Sobolev^{1, \floor{kp}}\)~convergence to \(\Sobolev^{k,p}\)~convergence. 
As a simple consequence of Corollary~\ref{corollaryHkpH1kp}, let us consider the case \(k=1\) and two exponents \(p\leq q\) which have the same integer parts. Then, a map \(u\in \Sobolev^{1,q}(\manfM;\manfN)\) can be approximated by smooth maps, provided that such an approximation holds for the weaker \(\Sobolev^{1,p}\)~norm. 

In the absence of topological assumptions on \(u\), it is possible to approximate it by maps that are smooth except for a small singular set.
For this purpose, we define for \(i\in \{0, \dots, m-1\}\) the set \(\ClassR_i(\manfM; \manfN)\) containing all maps \(u  \colon  \manfM \to \manfN\) that are smooth on \(\manfM\setminus T^{i}\), where \(T^{i}\) is a \emph{structured singular set of rank \(i\)} in the sense of Definition~\ref{defnStructuredSingularSet}, such that
\[
\abs{D^ju(x)}\leq \frac{C}{(d(x, T^{i}))^j}
\]
for some constant \(C\geq 0\) depending on \(u\) and \(j\).{}
One verifies that
\[{}
\ClassR_i(\manfM; \manfN) \subset \Sobolev^{k, p}(\manfM; \manfN)
\quad \text{for \(i < m - kp\)}
\]
and \(i = m - \floor{kp} - 1\) is the largest integer for which such an inclusion holds.
We then prove the following

\begin{theorem}
	\label{theoremIntroductionClassR}
	Let \(kp < m\).{}
	Then, \(\ClassR_{m - \floor{kp} - 1}(\manfM; \manfN)\) is dense in \(\Sobolev^{k, p}(\manfM; \manfN)\).
\end{theorem}

Theorem~\ref{theoremIntroductionClassR} was known for \(k = 1\) \cites{Bethuel,Hang-Lin} and also for \(k \ge 2\) when \(\manfM\) is replaced by \(\Ball^{m}\) \cite{BPVS_MO}. 
When \(k=1\), it is easy to deduce an approximation result for Sobolev maps defined on a manifold \(\manfM\) from the analogous result when the maps are defined in an open subset \(\Omega\) of \(\R^m\). 
Indeed, this essentially follows from the existence of a triangulation for \(\manfM\) together with the invariance of \(\Sobolev^{1,p}(\Omega)\) through composition with biLipschitz homeomorphisms.

When \(k\geq 2\), the rigidity of the space \(\Sobolev^{k,p}\) prevents one from using such compositions with merely Lipschitz maps and it is necessary to introduce a new approach that relies on the \((\floor{kp}, e)\)-extendability property for any \(e \in \{\floor{kp} + 1, \dots, m\}\). 
The counterpart of Theorem~\ref{theoremIntroductionClassR} under this extendability property is given by the following statement:

\begin{theorem}\label{thm_density_manifold_open_introduction}
Let \(kp < m\) and \(e \in \{\floor{kp} + 1, \dots, m\}\). 
Then, \(u \in \Sobolev^{k, p}(\manfM; \manfN)\) is \((\floor{kp}, e)\)-extendable if and only if 
there exists a sequence \((u_j)_{j\in \N}\) in \(\ClassR_{m - e - 1}(\manfM; \manfN)\) such that
\[{}
u_{j} \to u 
\quad \text{in \(\Sobolev^{k, p}(\manfM; \manfN)\).}
\]
\end{theorem} 

We adopt the convention \(\ClassR_{-1} = \Smooth^{\infty}\) to seamlessly include the case \(e = m\), which corresponds to Theorem~\ref{theoremhkpGreat}. Theorems~\ref{theoremhkpGreat}, \ref{theoremIntroductionClassR}, and~\ref{thm_density_manifold_open_introduction} are all included within Theorem~\ref{thm_density_manifold_open}, which includes the density result for bounded Lipschitz open sets \(\Omega \subset \R^m\) in place of compact manifolds \(\manfM\).

The proof of Theorem~\ref{thm_density_manifold_open} illustrates the versatility of \((\floor{kp}, e)\)-extendability. It is first established for domains \(\Omega \subset \R^m\), where the approximation techniques are more directly implemented. To extend the result to compact manifolds \(\manfM\), we take advantage of the fact that \(\manfM\) retracts smoothly onto itself within its ambient space \(\R^{\varkappa}\). This is achieved using the smooth projection \(\widetilde{\Pi}\) from a tubular neighborhood \(\manfM + B_\iota^\varkappa\) onto \(\manfM\).
Given a Sobolev map \(u \in \Sobolev^{k,p}(\manfM; \manfN)\), we define \(v \vcentcolon= u \circ \widetilde{\Pi}\) on \(\manfM + B_\iota^\varkappa\). Since \(v\) inherits the same extendability properties as \(u\), a suitable approximation for \(u\) can then be obtained by restricting an approximation of \(v\) to \(\manfM\).

In the specific case \(e=\floor{kp}+1\), the
\((\floor{kp}, e)\)-extendability
reduces to the \(\floor{kp}\)-extendability. 
In view of Theorem~\ref{theoremhkpGreatBall}, the latter is equivalent to the approximability by smooth maps when \(\manfM\) is replaced by the ball \(\Ball^m\). 
Interestingly, even when the domain is a manifold \(\manfM\), the \(\floor{kp}\)-extendability has a local character, in the sense that it is satisfied if and only if it holds on every open set of a finite covering \((U_i)_{i \in I}\) of \(\manfM\), see Lemma~\ref{proposition_Local_Charts}. 
This implies that a map \(u\in \Sobolev^{k,p}(\manfM;\manfN)\) satisfies the \(\floor{kp}\)-extendability if and only if, for every \(i\in I\), the restriction \(u|_{U_i}\) can be approximated by a sequence in \(\ClassR_{m-\floor{kp}-2}(U_i; \manfN)\).

The general Theorem~\ref{thm_density_manifold_open} also has some consequences on the weak density problem that we summarize in Corollaries~\ref{corollaryWeak1} and~\ref{corollaryWeak2}.
To begin with, let us introduce the set \(\Hilbert\weak^{k,p}(\Ball^m;\manfN)\) of all maps \(u\in \Sobolev^{k,p}(\Ball^m;\manfN)\) for which there is a sequence \((u_j)_{j \in \N}\) in \(\Smooth^{\infty}(\cBall^m;\manfN)\) that converges strongly to \(u\) in \(\Lebesgue^p(\Ball^m;\manfN)\) and is bounded in \(\Sobolev^{k,p}(\Ball^m; \manfN)\). 
A counterpart of Problem~\ref{question-2} to this setting is the following:

\begin{question}\label{question-3}
Identify all maps in \(\Hilbert\weak^{k,p}(\Ball^m; \manfN)\).
\end{question}

When \(kp\not\in \N\) or when \(kp\in \N\) and \(\pi_{kp}(\manfN)\simeq \{0\}\), we show  in Chapter~\ref{chapter-approximation-Sobolev-manifolds} that
\[
\Hilbert^{k,p}\weak(\Ball^m;\manfN) = \Hilbert^{k,p}(\Ball^m;\manfN),
\]
thus generalizing the case \(k = 1\) originally due to Bethuel~\cite{Bethuel}.
When \(kp\in \N\) and \(\pi_{kp}(\manfN) \not\simeq \{0\}\), by the obstruction to the strong approximation problem we have
\[
\Hilbert^{k,p}(\Ball^m;\manfN) \neq \Sobolev^{k, p}(\Ball^m; \manfN).
\]
Nevertheless, it is known that
\begin{equation*}
\Hilbert^{k,p}\weak(\Ball^m;\manfN) = \Sobolev^{k,p}(\Ball^m;\manfN),
\end{equation*}
provided that \(\manfN\) satisfies 
\[
\pi_0(\manfN) \simeq\ldots \simeq \pi_{kp-1}(\manfN)\simeq \{0\},
\]
see \cites{Hang-W11, Hajlasz, BPVS:2013, DetailleVanSchaftingen}.
Without further assumptions on the manifold \(\manfN\), the general case where \(kp\) is an integer is not fully understood.
Remarkable recent works by Detaille and Van~Schaftingen~\cite{DetailleVanSchaftingen} and by Bethuel~\cite{Bethuel-2020} show that weak density fails for \(k = 1\) and \(p = 2, 3\).
For example, when \(p = 3\) one has 
\[
\Hilbert\weak^{1,3}(\Ball^4;\Sphere^2)\not= \Sobolev^{1, 3}(\Ball^4;\Sphere^2),
\]
The counterexamples suggest that the weak sequential density of smooth functions does not only involve the topology of \(\manfN\), but also depends on purely analytical phenomena related to the energy necessary to create a topological singularity.

Beyond the consequences of Theorem~\ref{thm_density_manifold_open} itself, we believe that the tools used in its proof have suitable flexibility to be exploited in other settings. This is particularly the case for the following techniques: \emph{opening}, \emph{adaptative smoothing} and \emph{thickening} for which we provide a detailed and self-contained description in Chapter~\ref{chapter-approximation-Sobolev-manifolds}.  
In the framework of \emph{fractional} Sobolev spaces,  nonlocal extensions of these tools have been obtained in \cite{Detaille} to establish the analogue of Theorem~\ref{theoremIntroductionBethuel} in \(\Sobolev^{s,p}\) with \(s>0\). 
The density of bounded maps in \(\Sobolev^{1,p}(\Ball^m;\manfN)\) when \(\manfN\) is a complete, not necessarily compact, manifold that has been investigated in \cite{BPVS:2017} also exploits the opening and adaptative smoothing techniques.
The thickening operation is involved in the regularity theory of polyharmonic maps, see \cite{Gastel-2016}. 

\medskip

Let us briefly outline the structure of the manuscript leading to the proof of the main approximation theorems, focusing on the interplay between local and global topological properties:

\subsection*{Chapter~\ref{chapterFuglede}. Genericity and Fuglede Maps}
We introduce the concept of genericity using Fuglede maps and establish fundamental results on the stability of Sobolev functions under generic compositions, as in Proposition~\ref{propositionIntroductionFugledeGeneric}.
We also revisit the opening technique introduced by Brezis and Li~\cite{Brezis-Li}, further developed in \cite{BPVS_MO}, and adapt it to our new setting. We then further exploit it in Chapter~\ref{chapter-approximation-Sobolev-manifolds} to prove the approximation of Sobolev maps.

\subsection*{Chapter~\ref{chapter-3-Detectors}. VMO-Detectors}
Given the critical imbedding of \(\Sobolev^{1,p}\) into \(\VMO\) when the dimension of the domain is \(\floor{p}\), we introduce \(\VMO\)-detectors associated with a given Sobolev map \(u\). These allow us to identify Fuglede maps \(\gamma\) for which the composition \(u \compose \gamma\) remains in \(\VMO\). We are thus led to enlarge the functional framework to \(\VMO^\ell\) maps, which enjoy good stability properties under composition with Fuglede maps and provide the natural setting for our main results on the extendability of maps.

\subsection*{Chapter~\ref{section_VMO}. Topology for \boldmath\(\VMO\) Maps}
Building on the work of Brezis and Nirenberg~\cite{BrezisNirenberg1995}, we extend classical topological concepts to \(\VMO\) maps. We introduce and analyze homotopy in the \(\VMO\) framework, showing its stability under \(\VMO\) convergence. We investigate in particular the notion of homotopy in the \(\VMO\) sense, and its stability with respect to \(\VMO\) convergence. This chapter plays a crucial role in linking analytical properties with topological invariants.

\subsection*{Chapter~\ref{chapterGenericEllExtension}. \boldmath\(\ell\)-Extendability}
This chapter explores the notion of extendability using Fuglede maps on \(\Sphere^\ell\) or the boundary \(\partial\Simplex^{\ell + 1}\) of a simplex. We examine its implications for Sobolev and \(\VMO^\ell\) maps. 
A key result states that if \(u\) is \(\ell\)-extendable and two Fuglede maps \(\gamma_0\) and \(\gamma_1\) are homotopic, then \(u \circ \gamma_0\) and \(u \circ \gamma_1\) remain homotopic. This is reminiscent of the \(\ell\)-homotopy type introduced by White~\cites{White, White-1986}. Our approach relies on transversal perturbations of the identity, inspired by the work of Hang and Lin~\cite{Hang-Lin}.

\subsection*{Chapter~\ref{chapterExtensionGeneral}. \boldmath$(\ell, e)$-Extendability}
Expanding upon the framework of \(\ell\)-extendability, we introduce the more general notion of \((\ell, e)\)-extendability, defined via Fuglede maps on simplicial complexes. This concept extends \(\ell\)-extendability when \(e \geq \ell + 1\).
We establish that every map in \(\Hilbert^{k, p}\) is \((\ell, e)\)-extendable with \(\ell = \floor{kp}\) and \(e = m\), highlighting the significance of this property for approximation results.

\subsection*{Chapter~\ref{chapter-approximation-Sobolev-manifolds}. Approximation of Sobolev Maps}
In this final chapter, we adapt and refine the approximation techniques from our previous work \cite{BPVS_MO} to the new setting of \((\ell, e)\)-extendability. 
The classical density result for \(\Sobolev^{k, p}(\Ball^m; \manfN)\) when \(\pi_{\floor{kp}}(\manfN) \simeq \{0\}\) is no longer applicable, requiring new tools.
We revisit and extend the opening, thickening, and shrinking methods to accommodate cases where topological obstructions prevent a smooth approximation.
This allows us to establish Theorem~\ref{thm_density_manifold_open}, first for open sets and then for compact manifolds, a case previously out of reach in~\cite{BPVS_MO}.

\subsection*{Additional Insights and Applications}
Beyond the core chapters, more insights are developed in Chapters~\ref{chapterExtendability}, \ref{chapter-cohomological-obstructions}, and~\ref{chapterDecouple}:
\begin{description}
    \item[Chapter~\ref{chapterExtendability}] introduces simplified criteria for verifying \(\ell\)-extendability, using Fubini-type properties based on almost everywhere restrictions of a detector.
    \item[Chapter~\ref{chapter-cohomological-obstructions}] explores the cohomological obstructions to extendability, establishing links with topological invariants such as the distributional Jacobian.
    \item[Chapter~\ref{chapterDecouple}] investigates the distinction between the local nature of \(\ell\)-ex\-tend\-a\-bil\-i\-ty and the global behavior of \((\ell, e)\)-extendability, clarifying their interaction.
\end{description}

\bigskip

\subsubsection*{Notational Conventions.}
We systematically use the letter \(\manfV\) to denote a Riemannian manifold of dimension \(m\).  
By a \emph{Riemannian manifold}, we mean either a Lipschitz open subset \(\Omega \subset \R^m\) or a compact Riemannian manifold \(\manfM\) without boundary.  
In the latter case, we consider the standard measure associated with the Riemannian metric to define the Lebesgue space \(\Lebesgue^p\), but for simplicity, we omit explicit mention of this measure in integrals.  
A more comprehensive list of notational conventions is provided in the \textbf{Notation} chapter at the end of the monograph.

\cleardoublepage

\cleardoublepage
\chapter{Fuglede maps}
\label{chapterFuglede}

The restriction of a \(\Sobolev^{1,p}\) function \(u \colon \R^m \to \R\) to generic \(\ell\)-dimensional planes exhibits remarkable regularity properties: For almost every \(\ell\)-dimensional plane \(P + \xi\), with \(\xi \in \R^m\), the restriction \(u|_{P + \xi}\) is \(\Sobolev^{1,p}\) with respect to the \(\ell\)-dimensional Hausdorff measure \(\cH^{\ell}\). 
Here, genericity refers to the property that holds for almost every translation \(\xi\), and can be formulated more robustly in terms of composition: For almost every \(\xi \in \R^m\), the map \(\gamma \colon x \in P \mapsto x + \xi \in \R^m\) satisfies \(u \circ \gamma \in \Sobolev^{1,p}(P)\). 
By expressing this regularity through compositions, we fix the domain \(P\) and focus on various choices of \(\gamma\) within a given functional class.

In this chapter, we generalize this composition framework to generic Lipschitz maps \(\gamma \colon \manfA \to \R^m\), where \(\manfA\) is an \(\ell\)-dimensional Riemannian manifold. 
To describe genericity in this setting, we introduce a summable function \(w \colon \R^m \to [0, +\infty]\), which serves as a screening mechanism to identify suitable maps \(\gamma\).
The fundamental message here is that it is always possible to find a good choice of \(w\) which implies that \emph{the composition \(u \circ \gamma\) is \(\Sobolev^{1,p}\) whenever \(w \circ \gamma\) is summable.}\/
This framework is particularly powerful as, once \(u\) is given, the same function \(w\) applies universally to detect admissible maps \(\gamma\), regardless of the choice of \(\manfA\) or its dimension \(\ell\).

This formulation of genericity might seem overly permissive at first, but it is more structured than it appears. For instance, if a sequence of functions \((u_j)_{j \in \N}\) converges to \(u\) sufficiently quickly in \(\Sobolev^{1,p}\), then there exists a summable function \(w\) such that the compositions \(u_j \circ \gamma\) converge to \(u \circ \gamma\) in \(\Sobolev^{1,p}\) whenever \(w \compose \gamma\) is summable.

To develop the tools necessary for these results, we begin by examining the stability properties of compositions in the simpler setting of Lebesgue spaces \(\Lebesgue^p\). Given \(u \in \Lebesgue^p\), we study the behavior of \(u \circ \gamma\) for generic Borel maps \(\gamma \colon X \to \R^m\) defined on a metric measure space \(X\). 
Two types of stability results are established: One for sequences \((u_j)_{j \in \N}\) in \(\Lebesgue^p(X)\) converging to \(u\), and another for sequences of Borel maps \((\gamma_j)_{j \in \N}\) converging pointwise to \(\gamma\).

In the Sobolev framework, stability under \((\gamma_j)_{j \in \N}\) is more delicate. EquiLipschitz continuity of the sequence \((\gamma_j)_{j \in \N}\) is insufficient to guarantee that \((u \circ \gamma_j)_{j \in \N}\) converges generically to \(u \circ \gamma\) in \(\Sobolev^{1,p}\). However, we establish that uniform convergence occurs when the domain of the maps \(\gamma_j\) is one-dimensional.

\section{Composition with Lebesgue functions}

A standard property from measure theory asserts that every convergent sequence in \(\Lebesgue^{p}\) with \(1 \le p < \infty\) has a subsequence that converges almost everywhere.
We show how such a property can be seen as a particular case involving generic composition with Borel measurable maps \(\gamma\).{}
A fundamental tool that we explore in this work is the existence of a summable function \(w\) that is capable of detecting maps \(\gamma\) that preserve \(\Lebesgue^{p}\) convergence under composition.
The summability of \(w\) is essential as it ensures the genericity of such maps.

\begin{proposition}
\label{lemmaModulusLebesguesequence}
Given a measurable function \(u  \colon  \manfV \to \R\), let \((u_j)_{j \in \N}\) be a sequence of measurable functions in \(\Lebesgue^{p}(\manfV)\) such that
\[{}
u_{j} \to u \quad \text{in \(\Lebesgue^{p}(\manfV)\).}
\]
Then, there exist a summable function \(w  \colon  \manfV \to [0, +\infty]\) and a subsequence \((u_{j_i})_{i\in \N}\)
such that, for every metric measure space \((X, d, \mu)\) and every measurable map \(\gamma  \colon  X \to \manfV\) such that \(w \compose \gamma \in \Lebesgue^{1}(X)\), we have that \(  u \compose \gamma \) and \((u_{j_i} \compose \gamma)_{i\in \N}\) are contained in \(\Lebesgue^{p}(X)\) and satisfy
\[{}
u_{j_i} \compose \gamma \to u \compose \gamma{}
\quad \text{in \(\Lebesgue^{p}(X)\).}
\]
\end{proposition}

By a \emph{metric measure space}  \((X, d, \mu)\) we mean that \(X\) is a metric space equipped with a distance \(d\) and a Borel measure \(\mu\).
We systematically consider Borel measurability, which has the advantage that the composition of two functions with such a property is still Borel measurable.
As we shall be dealing with composition of functions in various spaces, composition is always meant in the classical pointwise sense.
All functions are assumed to be defined pointwise and we discard the use of equivalence classes, including in the setting of Lebesgue or Sobolev spaces.

\begin{proof}[Proof of Proposition~\ref{lemmaModulusLebesguesequence}]
Let \((\kappa_{i})_{i \in \N}\) be a sequence of positive numbers diverging to infinity such that \(\kappa_{i} \ge 1\) for every \(i \in \N\).
We take a subsequence \((u_{j_{i}})_{i \in \N}\) such that  
\begin{equation}
\label{eqFugledeEstimateChoiceSubsequence}
\sum_{i = 0}^{\infty}{\kappa_{i}\norm{u_{j_{i}} - u}_{\Lebesgue^{p}(\manfV)}}
< \infty
\end{equation}
and define the function \(w  \colon  \manfV \to [0, +\infty]\) by 
\begin{equation*}
w = \biggl(\abs{u} + \sum_{i = 0}^{\infty}{\kappa_{i} \abs{u_{j_{i}} - u}} \biggr)^{p}.
\end{equation*}
By the triangle inequality and Fatou's lemma, \(w\) is summable in \(\manfV\) and
\begin{equation}
~ \label{eqFugledeEstimateChoiceSubsequence-bis}
\biggl( \int_{\manfV}{w} \biggr)^{\frac{1}{p}}
= \norm{w^{1/p}}_{\Lebesgue^{p}(\manfV)}
\le \norm{u}_{\Lebesgue^{p}(\manfV)} + \sum_{i = 0}^{\infty}{\kappa_{i} \norm{u_{j_{i}} - u}_{\Lebesgue^{p}(\manfV)}}.
\end{equation}
We also have \( \abs{u}^p \le w\) and, since \(\kappa_{i} \ge 1\), 
\begin{equation*}
\abs{u_{j_{i}}}^{p}
\le \bigl( \abs{u} + \kappa_{i}{\abs{u_{j_{i}} - u}}  \bigr)^{p}{}
\le w.
\end{equation*}
Thus, for every measurable map \(\gamma  \colon  X \to \manfV\) such that \(w \compose \gamma\) is summable in \(X\), \(u \compose \gamma\) and \(u_{j_i} \compose \gamma\) belong to \(\Lebesgue^{p}(X)\) and satisfy
\[{}
\int_{X} \abs{u_{j_{i}} \compose \gamma - u \compose \gamma}^{p} \dif\mu
\le \frac{1}{\kappa_{i}^{p}} \int_{X} w \compose \gamma \dif\mu.
\]
Since \( (\kappa_{i})_{i \in \N} \) diverges to infinity, as \(i \to \infty\) we then get
\[{}
u_{j_i} \compose \gamma \to u \compose \gamma{}
\quad \text{in \(\Lebesgue^{p}(X)\)}.
\qedhere
\]
\end{proof}

Observe that if the entire sequence \((u_{j})_{j \in \N}\) converges sufficiently fast, in the sense that it satisfies \eqref{eqFugledeEstimateChoiceSubsequence}, then the conclusion of Proposition~\ref{lemmaModulusLebesguesequence} holds for \((u_{j})_{j \in \N}\) itself.
Using the notation of Proposition~\ref{lemmaModulusLebesguesequence}, we now deduce as a particular case the classical pointwise convergence of \(\Lebesgue^{p}\) sequences:

\begin{example}
Let \( w \) and \( (u_{j_{i}})_{i \in \N} \) be given by Proposition~\ref{lemmaModulusLebesguesequence}.
For any \(a \in \manfV\) such that \(w(a) < +\infty\), the constant map \(\gamma  \colon  \{0\} \to \manfV\), \(\gamma(0) = a\), is such that \(\int w \compose \gamma \dif\delta_{0} = w(a) < +\infty\).{}
Hence,
\[{}
u_{j_{i}}(a) \to u(a)
\quad \text{whenever \(w(a) < +\infty\).}
\]
By summability of \(w\), this is the case almost everywhere in \(\manfV\).
\end{example}

Every \(\Lebesgue^{p}\) function in \(\R^{m}\) enjoys \(\Lebesgue^{p}\) stability with respect to translations in the sense that if \(\tau_{h}u(x) = u(x + h)\) is the translation of \(u\) with respect to \(h \in \R^m\), then \(\tau_{h_j}u \to u\) in \(\Lebesgue^{p}(\R^{m})\) for every \(u \in \Lebesgue^{p}(\R^{m})\) and every sequence \((h_j)_{j \in \N}\) in \(\R^m\) that converges to \(0\).{}
We now generalize this property in the spirit of Proposition~\ref{lemmaModulusLebesguesequence} in terms of composition with a sequence of Borel-measurable maps.
As before, we rely on a summable function \(w\) to select sequences that are admissible:

\begin{proposition}
	\label{propositionFugledeApproximationMaps}
	Given \(u \in \Lebesgue^{p}(\manfV)\), there exists a summable function \(w  \colon  \manfV \to [0, +\infty]\) with 
	\(w \ge \abs{u}^{p}\) in \(\manfV\) 
	such that, for every metric measure space \((X, d, \mu)\) with finite measure and every sequence of measurable maps \(\gamma_{j}  \colon  X \to \manfV\) with
	\[{}
	\gamma_{j} \to \gamma{}
	\quad \text{pointwise in \(X\),}
	\]
	such that \((w \compose \gamma_{j})_{j \in \N}\) is bounded in \(\Lebesgue^{1}(X)\) and \(\gamma  \colon  X \to \manfV\) satisfies \(w \compose \gamma \in \Lebesgue^{1}(X)\),
	we have
	\[{}
	u \compose \gamma_{j} \to u \compose \gamma \quad \text{in \(\Lebesgue^{p}(X)\).}
	\]
\end{proposition}

\resetconstant
\begin{proof}
	Since \(\manfV\) is locally compact, we may take a sequence \((\kappa_{i})_{i \in \N}\) of positive numbers that diverges to infinity and a sequence \((f_{i})_{i \in \N}\) in \(\Smooth^{0}_{c}(\manfV)\) that converges to \(u\) in \(\Lebesgue^{p}(\manfV)\) with 
	\[{}
	\sum_{i = 0}^{\infty}{\kappa_{i} \norm{u - f_{i}}_{L^{p}(\manfV)}} < \infty.
	\]
	By the triangle inequality and Fatou's lemma, the function \(w  \colon  \manfV \to [0, +\infty]\) defined by
	\[{}
	w = \biggl( \abs{u} + \sum_{i = 0}^{\infty}{\kappa_{i} {\abs{u - f_{i}}}} \biggr)^{p}
	\]
	is summable in \(\manfV\); cf.\@ estimate \eqref{eqFugledeEstimateChoiceSubsequence-bis} above.{}
	If \((\gamma_{j})_{j \in \N}\) is a sequence of maps as in the statement, then for every \(i, j \in \N\) we have
	\begin{multline*}
	\norm{u \compose \gamma_{j} - u \compose \gamma}_{\Lebesgue^{p}(X)}
	\le \norm{u \compose \gamma_{j} - f_{i} \compose \gamma_{j}}_{\Lebesgue^{p}(X)}
	+ \norm{f_{i} \compose \gamma_{j} - f_{i} \compose \gamma}_{\Lebesgue^{p}(X)}\\
	+ \norm{f_{i} \compose \gamma - u \compose \gamma}_{\Lebesgue^{p}(X)}.
	\end{multline*}
	Since \((\kappa_{i} \abs{u - f_{i}})^{p} \le w\) and \((w \compose \gamma_{j})_{j \in \N}\) is bounded in \(\Lebesgue^{1}(X)\),
	\[{}
	{\norm{u \compose \gamma_{j} - f_{i} \compose \gamma_{j}}_{\Lebesgue^{p}(X)}}
	\le \frac{1}{\kappa_{i}} \biggl( \int_{X}{w \compose \gamma_{j} \dif\mu} \biggr)^{\frac{1}{p}}
	\le \frac{\C}{\kappa_{i}}.
	\]
	A similar estimate holds with \(\gamma_{j}\) replaced by \(\gamma\), since \(w \compose \gamma \in \Lebesgue^{1}(X)\).{}
	Thus,
	\[{}
	\norm{u \compose \gamma_{j} - u \compose \gamma}_{\Lebesgue^{p}(X)}	
	\le 
	\frac{\Cl{cte-170}}{\kappa_{i}} + \norm{f_{i} \compose \gamma_{j} - f_{i} \compose \gamma}_{\Lebesgue^{p}(X)}.
	\]
	For \(i\) fixed, we let \(j \to \infty\).{}
	Since \( f_{i} \) is bounded and \( X \) has finite measure, the sequence \( (f_{i} \compose \gamma_{j})_{j \in \N} \) is dominated by a constant \( \Lebesgue^{p} \) function.
    As \( f_{i} \) is also continuous and \( (\gamma_{j})_{j \in \N} \) converges pointwise to \( \gamma \), it then follows from Lebesgue's dominated convergence theorem that
    \[
    \lim_{j \to \infty}{\norm{f_{i} \compose \gamma_{j} - f_{i} \compose \gamma}_{\Lebesgue^{p}(X)}}
    = 0.
    \]
    We then get
	\[{}
	\limsup_{j \to \infty}{	\norm{u \compose \gamma_{j} - u \compose \gamma}_{\Lebesgue^{p}(X)}}
	\le \frac{\Cr{cte-170}}{\kappa_{i}}
	\]
	and the conclusion follows as \(i \to \infty\) since \( (\kappa_{i})_{i \in \N} \) diverges to infinity.
\end{proof}

From Proposition~\ref{propositionFugledeApproximationMaps} one deduces a continuity property for \( \Lebesgue^{p} \) functions:

\begin{example}
\label{exampleFugledePointwiseConvergence}
    Given \(u \in \Lebesgue^{p}(\manfV)\), let \(w\) be the summable function provided by Proposition~\ref{propositionFugledeApproximationMaps}.
    Then, for every \(a \in \manfV\) such that \(w(a) < +\infty\) and every sequence \( (a_j)_{j \in \N} \) in \(\manfV\) such that \( a_j \to a \) for which \( (w(a_j))_{j \in \N} \) is bounded, we have 
    \[
    u(a_j) \to u(a).
    \]
    Indeed, it suffices to apply the proposition with \( X = \{0\} \) equipped with the Dirac mass \( \delta_0 \) and take the sequence of constant maps \(\gamma_j = a_j\).
\end{example}

We also deduce a local continuity of the translation operator in the setting of \(\Lebesgue^{p}\) spaces:

\begin{example}
For any summable function \(w  \colon  \R^{m} \to [0, +\infty]\) and any sequence \((h_{j})_{j \in \N}\) in \(\R^m\) that converges to \(0\), the sequence of maps \(\gamma_{j}  \colon  B_r^{m} \to \R^{m}\) defined by \(\gamma_{j}(x) = x + h_{j}\) converges to the identity map on the open ball \(B^{m}_r \subset \R^m\) and is such that \((w \compose \gamma_{j})_{j \in \N}\) is bounded in \(\Lebesgue^{1}(B_r^{m})\).{}
For every \(u \in \Lebesgue^{p}(\R^{m})\), it then follows from Proposition~\ref{propositionFugledeApproximationMaps} applied with \( \manfV = \R^m  \) and \( X = B^{m}_r \) that \(\tau_{h_{j}}u = u \compose \gamma_{j} \to u\) in \(\Lebesgue^{p}(B_r^{m})\).
\end{example}

The next example illustrates how the summable function \(w\) is sensitive to changes of \(u\) on a negligible set:

\begin{example}
	Given \(f \in \Smooth_{c}^{0}(\R^{m})\) and a negligible Borel set \(A \subset \R^{m}\), let
	\[{}
	u(x)
	=
	\begin{cases}
		f(x)	& \text{if \(x \not\in A\),}\\
		0	& \text{if \(x \in A\).}
	\end{cases}
	\]
	The conclusion of Proposition~\ref{propositionFugledeApproximationMaps} is satisfied in this case with
	\[{}
	w(x)
	=
	\begin{cases}
		(f(x))^{p}	& \text{if \(x \not\in A\),}\\
		+\infty	& \text{if \(x \in A\).}
	\end{cases}
	\]
	Indeed, take a metric measure space \(X\) with finite measure and a sequence of measurable maps  \(\gamma_{j}  \colon  X \to \manfV\) that converges pointwise to \(\gamma\) and such that \(w\compose \gamma\) and each \(w_j\compose \gamma\) belong to \(\Lebesgue^1(X)\).
    We first observe that
    \begin{equation}
    \label{eqFuglede-219}
    u \compose \gamma_{j} \to u \compose \gamma
    \quad \text{almost everywhere in \(X\).}
    \end{equation}
	To this end, by continuity of \(f\), for every \(x \in X \setminus \bigl( \gamma^{-1}(A) \cup \bigcup_{j \in \N}\gamma_{j}^{-1}(A) \bigr)\), we have
	\[{}
	u \compose \gamma_{j}(x)
	= f(\gamma_{j}(x))
	\to f(\gamma(x)){}
	= u \compose \gamma(x).
	\] 
    Next, considering the choice of \(w\) and the fact that \(w \compose \gamma \in \Lebesgue^{1}(X)\), the set \(\gamma^{-1}(A)\) is negligible in \(X\).
	Similarly, since \((w \compose \gamma_{j})_{j \in \N}\) is contained in \(\Lebesgue^{1}(X)\), each set \(\gamma_{j}^{-1}(A)\) is also negligible in \(X\). 
    Hence,  \eqref{eqFuglede-219} holds.
    Using that \(u\) is bounded and \(X\) has finite measure, we then deduce that
    \[
    u \compose \gamma_{j} \to u \compose \gamma 
    \quad \text{in \(\Lebesgue^1(X)\)}
    \] 
    from Lebesgue's dominated convergence theorem.
\end{example}

\section{Composition with Sobolev functions}
\label{sectionFugledeSobolev}

By a classical Morrey-type restriction property for Sobolev functions \(u \in \Sobolev^{1, p}(\R^{m})\), one has that for almost every \(x'' \in \R^{m - \ell}\),
\begin{equation}
\label{eqFuglede-249}
u(\cdot, x'') \in \Sobolev^{1, p}(\R^{\ell})
\end{equation}
and for almost every \(r > 0\),
\begin{equation}
\label{eqFuglede-254}
u|_{\partial B_{r}^{m}}
\in \Sobolev^{1, p}(\partial B_{r}^{m}).
\end{equation}
In another direction, White~\cite{White}*{Lemma~3.1} has shown that, for every Lipschitz map \(h \colon \Omega \to \R^m\) on a bounded open set and for almost every \(\xi \in \R^m\), 
\begin{equation}
u \compose (h + \xi) \in \Sobolev^{1, p}(\Omega)
\end{equation}
and the chain rule \(D(u \compose (h + \xi)) = Du(h + \xi)[Dh]\) is satisfied.
We rely on the formalism of Proposition~\ref{lemmaModulusLebesguesequence} to show that these properties can be detected through composition with Lipschitz maps \(\gamma\) using the same summable function \(w\).{}
To this end, we introduce the class of Fuglede maps associated to \(w\)\,:

\begin{definition}
\label{defnFuglede}
    Given a summable function \(w  \colon  \manfV \to [0, +\infty]\) and a metric measure space \((X, d, \mu)\),  we call a \emph{Fuglede map} any  Lipschitz map \(\gamma  \colon  X \to \manfV\) such that \(w \compose \gamma\) is summable in \(X\). 
	We denote by \(\Fuglede_{w}(X; \manfV)\) the class of Fuglede maps on \(X\) with respect to \(w\).{}
\end{definition}

We acknowledge the foundational work of Fuglede~\cite{Fuglede}, who introduced the notion of modulus to study properties that hold for almost every measure. His ideas, which we discuss in Section~\ref{section_Fuglede} below, inspired the concept of Fuglede maps.
In this context, we introduce a genericity property for Sobolev functions, which generalizes the well-known Morrey-type property regarding the restrictions of Sobolev functions to lower-dimensional sets.
This can be seen as a rephrasing of \cite{White}*{Lemma~3.1} in terms of Fuglede maps defined on a Riemannian manifold, which, as we recall, is either a Lipschitz open subset of a Euclidean space or a compact Riemannian manifold without boundary.

\begin{proposition}
\label{corollaryCompositionSobolevFuglede} 
Given \(u \in \Sobolev^{1, p}(\manfV)\), there exists a summable function \(w  \colon  \manfV \to [0, +\infty]\) such that, for every Riemannian manifold \(\manfA\)
and every \(\gamma \in \Fuglede_{w}(\manfA; \manfV)\), we have 
	\[{}	
	u \compose \gamma \in \Sobolev^{1, p}(\manfA)
	\]
	and
	\[{}
	D(u \compose \gamma)(x){}
	= Du(\gamma(x))[D\gamma(x)]
	\quad \text{for almost every \(x \in \manfA\).}
	\]
When \(\manfA\) is a Lipschitz open subset of a Euclidean space and the continuous extension of \(\gamma\) to \(\overline{\manfA}\) is such that \(\gamma|_{\partial \manfA} \in \Fuglede_{w}(\partial \manfA; \manfV)\), the trace of \(u \compose \gamma\) satisfies
	\[{}
	\Trace{(u \compose \gamma)}(x)
	= u(\gamma(x))
	\quad \text{for almost every \(x \in \partial \manfA\).}
	\]	
\end{proposition}

The Sobolev space \(\Sobolev^{1,p}(\manfV)\) is defined in the following classical way when \(\manfV = \manfM\) is a compact Riemannian manifold without boundary: 
Consider a finite covering of \(\manfM\) by charts \((U_i,\varphi_i)_{i\in I}\)\,, where each \(\varphi_i\) is a smooth diffeomorphism from an open neighborhood of \(\overline U_i\) in \(\manfM\) onto an open set of \(\R^m\). 
We then say that a map \(u\in \Lebesgue^p(\manfM)\) belongs to \(\Sobolev^{1,p}(\manfM)\) whenever \(u\compose \varphi_{i}^{-1}\in \Sobolev^{1,p}(\varphi_i(U_i))\) for every \(i\in I\).  
This definition does not depend on the choice of the covering. 

Fix \(u\in \Sobolev^{1,p}(\manfM)\) and a chart \((U_i, \varphi_i)\) as above. 
Since \(u\compose \varphi_{i}^{-1}\in \Sobolev^{1,p}(\varphi_i(U_i))\) and \(\varphi_i\) is a smooth diffeomorphism from \(U_i\) onto \(\varphi_i(U_i)\), one can define for every \(x\in U_i\)\,,
\[
Du(x)=D(u\circ \varphi_{i}^{-1}) (\varphi_i(x))\compose D\varphi_i(x).
\]
Then, \(Du(x)\) is a linear form on the tangent space \(\Tangent{x}\manfM\).  
If \((V,\psi)\) is another chart such that \(U_i \cap V\not=\emptyset\), then writing \(u\circ \varphi_{i}^{-1}= (u\circ \psi^{-1})\compose (\psi\compose \varphi_{i}^{-1})\) and using the chain rule, one concludes that the linear form \(D(u\circ \psi^{-1}) (\psi(x))\compose D\psi(x)\) coincides with \(Du(x)\) for almost every \(x\in U_i\cap V\). 
This entitles one to define \(Du(x)\) intrinsically, namely without any reference to a specific choice of charts, up to a negligible set. 
However, remember that according to our convention, all functions are defined pointwise. 
We are thus led to extend arbitrarily \(Du\) as a function defined everywhere on \(\manfM\) with values into \(\Tangent{}^*\manfM\). 
Moreover, the function \(x\mapsto \abs{Du(x)}\) belongs to \(\Lebesgue^p(\manfM)\), where the notation \(\abs{\cdot}\) refers to the norm on the linear space \((\Tangent{x}\manfM)^*\)  inherited from the metric on \(\manfV\).

\resetconstant
\begin{proof}[Proof of Proposition~\ref{corollaryCompositionSobolevFuglede}]
Let us first assume that \(\manfV=\Omega\) is a Lipschitz open set.
We take a sequence \((u_{j})_{j \in \N}\) in \((\Sobolev^{1, p} \cap \Smooth^{\infty})(\Omega)\) that converges to \(u\) in \(\Sobolev^{1, p}(\Omega)\).
By application of Proposition~\ref{lemmaModulusLebesguesequence} to \((u_{j})_{j \in \N}\) and the components of \((Du_{j})_{j \in \N}\)\,, we obtain a subsequence \((u_{j_{i}})_{i \in \N}\) and a summable function \(w  \colon  \Omega \to [0, +\infty]\)
such that, if \(\manfA\) is a Riemannian manifold and \(\gamma \in \Fuglede_{w}(\manfA ; \Omega)\), then 
\begin{equation*}
u_{j_i} \compose \gamma \to u \compose \gamma \quad \text{and} \quad 
D u_{j_i} \compose \gamma \to D u \compose \gamma \quad \text{in \(\Lebesgue^p (\manfA)\).}
\end{equation*}
Since every map \(u_{j}\) is smooth and \(\gamma\) is Lipschitz, we have \(u_{j_{i}} \compose \gamma \in \Sobolev^{1, p}(\manfA)\) and
\[{}
D(u_{j_{i}} \compose \gamma){}
= (Du_{j_{i}} \compose \gamma) [D\gamma{}]
\quad \text{almost everywhere in \(\manfA\),}
\]
that is \(D(u_{j_{i}} \compose \gamma)(x){} = Du_{j_{i}} (\gamma(x)) [D\gamma(x)]\) for almost every \(x \in \manfA\).
Thus,
\[{}
\abs{D(u_{j_{i}} \compose \gamma) - (Du \compose \gamma)[D\gamma]}
\le \norm{D\gamma}_{\Lebesgue^{\infty}(\manfA)} \abs{Du_{j_{i}} \compose \gamma - Du \compose \gamma},
\]
whence 
\begin{equation*}
D(u_{j_i} \compose \gamma) 
\to  (Du \compose \gamma) [D\gamma ]
\quad \text{in \(\Lebesgue^p (\manfA)\).}
\end{equation*}
As a consequence of the stability property in Sobolev spaces, see e.g.~\cite{Brezis}*{Chapter~9, Remark~4} for domains in a Euclidean space, we thus have \(u \compose \gamma \in \Sobolev^{1, p}(\manfA)\) and 
\[{}
D(u \compose \gamma) = (Du \compose \gamma) [ D\gamma{}]
\quad \text{almost everywhere in \(\manfA\).}
\]

We next assume that \(\manfV=\manfM\) is a compact manifold without boundary. We introduce a finite covering of \(\manfM\) by charts \((U_i,\varphi_i)_{i\in I}\), where each \(\varphi_i\) is a smooth biLipschitz diffeomorphism and \(\varphi_i(U_i)\) is a Lipschitz open subset of \(\R^m\). 
Applying the case of open domains to \(u\compose \varphi_{i}^{-1}\in \Sobolev^{1,p}(\varphi_i(U_i))\), we get a summable function \(w_i \colon \varphi_i(U_i)\to [0,+\infty]\) such that, for every
Riemannian manifold \(\manfA\) and every \(\gamma\in \Fuglede_{w_i}(\manfA;\varphi_i(U_i))\), we have 
\begin{equation}
    \label{eqFuglede-358}
    (u\compose \varphi_{i}^{-1}) \compose \gamma \in \Sobolev^{1,p}(\manfA)
\end{equation} 
and 
\begin{equation}
    \label{eqFuglede-363}
D\bigl((u\compose\varphi_{i}^{-1}) \compose \gamma \bigr)
= \bigl(D(u\compose \varphi_{i}^{-1})\compose \gamma \bigr)[D\gamma]
\quad \text{almost everywhere in \(\manfA\).}
\end{equation}
For every \(x\in \manfM\) we then define
\[
\widetilde{w}_{i}(x)=
\begin{cases}
w_i\compose \varphi_i(x) & \textrm{if } x\in U_i,\\
0 & \textrm{if } x\in \manfM\setminus U_i,
\end{cases}
\]
and
\[
w(x)=\sum_{i\in I}{\widetilde{w}_i(x)}.
\]
Since there are finitely many terms in the sum above, \(w\) is summable on \(\manfM\).

Let \(i \in I\).
Since \(0 \le w_i \le w\) in \(U_i\) and \((\varphi_i \compose \gamma)(\gamma^{-1}(U_i)) \subset U_i\)\,,
for every Riemannian manifold \(\manfA\) and every \(\gamma\in \Fuglede_w(\manfA;\manfM)\) we have 
\[
\varphi_i\compose \gamma|_{\gamma^{-1}(U_i)} \in \Fuglede_{w_i}(\gamma^{-1}(U_i);\varphi_i(U_i)). 
\]
It follows that \eqref{eqFuglede-358} and \eqref{eqFuglede-363} that
\[
(u\compose \varphi_{i}^{-1})\compose (\varphi_i\compose \gamma)|_{\gamma^{-1}(U_i)} \in \Sobolev^{1,p}(\gamma^{-1}(U_i))
\]
and, for almost every \(x\in \gamma^{-1}(U_i)\),
\begin{align*}
D\bigl((u\compose \varphi_{i}^{-1})\compose (\varphi_i\compose \gamma)\bigr)(x)
&= \bigl(D(u\compose \varphi_{i}^{-1}) \compose (\varphi_i\compose \gamma(x)) \bigr)[ D(\varphi_i\compose \gamma)(x)]\\
&=D(u\compose \varphi_{i}^{-1})(\varphi_i\compose \gamma(x))\compose D\varphi_i(\gamma(x))\compose D\gamma(x).
\end{align*}
The left-hand side is simply \(D(u\compose \gamma)(x)\), while in the right-hand side we can rely on the definition of \(Du\) to get
\[
D(u\compose \gamma)(x) 
=Du(\gamma(x))[D\gamma(x)].
\]
Since \((U_i)_{i\in I}\) is an open covering of \(\manfM\), this completes the proof of the first part of the proposition.

We now assume that \(\manfV\) is either a compact manifold or a Lipschitz open set, and \(\manfA\) is a Lipschitz open subset of a Euclidean space.
We further suppose that the continuous extension of \(\gamma\) to \(\overline{\manfA}\) is such that \(\gamma|_{\partial \manfA} \in \Fuglede_{w}(\partial \manfA; \manfV)\).{}
In this case, by Proposition~\ref{lemmaModulusLebesguesequence} we also have
\begin{equation*}
u_{j_i} \compose \gamma|_{\partial \manfA} \to u \compose \gamma|_{\partial \manfA} 
\quad \text{in \(\Lebesgue^p (\partial \manfA)\).}
\end{equation*}
By continuity of the trace operator in \(\Lebesgue^p (\partial \manfA)\), we conclude that \(\Trace{(u \compose \gamma)} = u \compose \gamma|_{\partial \manfA}\) almost everywhere on \(\partial \manfA\).
\end{proof}

Using our framework, we now explain the genericity of \(x'' \in \R^{m - \ell}\) and \(r > 0\) that verify \eqref{eqFuglede-249} and \eqref{eqFuglede-254}, respectively, and how to detect them using the same summable function:

\begin{example}
	Given \(u \in \Sobolev^{1, p}(\R^{m})\), let \(w\) be the summable function given by Proposition~\ref{corollaryCompositionSobolevFuglede}.
By Tonelli's theorem,
\[{}
\int_{\R^{m - \ell}}\biggl( \int_{\R^{\ell}} w(x', x'') \dif x' \biggr) \dif x''
= \int_{\R^{m}} w
< \infty.
\]
Thus, for almost every \(x'' \in \R^{m - \ell}\),
\begin{equation}
\label{eqFuglede-45}
\int_{\R^{\ell}} w(x', x'') \dif x' < \infty.
\end{equation}
For any \(x'' \in \R^{m - \ell}\) that satisfies \eqref{eqFuglede-45}, one has \eqref{eqFuglede-249}
since the map \(\gamma  \colon  \R^{\ell} \to \R^{m}\) defined by \(\gamma(x') = (x', x'')\) belongs to \(\Fuglede_{w}(\R^{\ell}; \R^{m})\).{}
Similarly, by the integration formula in polar coordinates,
\[{}
\int_{0}^{\infty} \biggl( \int_{\partial B_{r}^{m}} w \dif\sigma \biggr) \dif r
= \int_{\R^{m}} w 
< \infty.
\]
Thus, for almost every \(r > 0\),
\begin{equation}
\label{eqFuglede-61}
\int_{\partial B_{r}^{m}} w \dif\sigma < \infty{}
\end{equation}
For any \(r > 0\) that verifies \eqref{eqFuglede-61}, the inclusion map of \( \partial B_r^m \) into \(\R^m\) belongs to \(\Fuglede_{w}(\partial B_r^m; \R^{m})\), and then \eqref{eqFuglede-254} holds.
\end{example}

Concerning the generic behavior on sequences of Sobolev functions, the following holds:

\begin{proposition}
\label{lemmaModulusSobolevsequenceHighk}
Let \((u_j)_{j \in \N}\) be a sequence of functions in \(\Sobolev^{1, p}(\manfV)\) such that
\[{}
u_{j} \to u
\quad \text{in \(\Sobolev^{1, p}(\manfV)\).}
\]
Then, there exist a summable function \(w  \colon  \manfV \to [0, +\infty]\) and a subsequence \((u_{j_i})_{i \in \N}\) such that, for every Riemannian manifold \(\manfA\) and every \(\gamma \in \Fuglede_{w}(\manfA; \manfV)\), the sequence \((u_{j_{i}} \compose \gamma)_{i \in \N}\) is contained in \(\Sobolev^{1, p}(\manfA)\) and 
\[{}
u_{j_{i}} \compose \gamma \to u \compose \gamma \quad \text{in \(\Sobolev^{1, p}(\manfA)\).}
\]
\end{proposition}

\begin{proof}
We only consider the case where \(\manfV=\Omega\) is a Lipschitz open subset in \(\R^m\); when \(\manfV\) is a compact manifold, one may proceed as in the proof of Proposition~\ref{corollaryCompositionSobolevFuglede}.
Applying Proposition~\ref{lemmaModulusLebesguesequence} to the sequences \((u_{j})_{j \in \N}\) and \((Du_{j})_{j \in \N}\), there exists an increasing sequence of positive integers \((j_{i})_{i \in \N}\) and a summable function \(\overline{w}  \colon  \Omega \to [0, +\infty]\) such that, if \(\gamma \in \Fuglede_{\bar{w}}(\manfA ; \Omega)\), then
\begin{equation}
\label{eqDensityNew-818}
u_{j_i} \compose \gamma \to u \compose \gamma \quad \text{and} \quad 
D u_{j_i} \compose \gamma \to D u \compose \gamma \quad \text{in \(\Lebesgue^p (\manfA)\).}
\end{equation}
By Proposition~\ref{corollaryCompositionSobolevFuglede}, for each \(i \in \N\) there exists a summable function \({w}_{i}  \colon  \Omega \to [0, +\infty]\) such that if \(\gamma \in \Fuglede_{w_{i}}(\manfA ; \Omega)\), then \(u_{j_{i}} \compose \gamma \in \Sobolev^{1, p}(\manfA)\) and
\begin{equation}
\label{eqDensityNew-829}
D(u_{j_{i}} \compose \gamma) 
= (Du_{j_{i}} \compose \gamma) [D\gamma{}]
\quad \text{almost everywhere in \(\manfA\).}
\end{equation}
Since the functions \((w_{i})_{i \in \N}\) are summable, there exists a 
sequence of positive numbers \((\lambda_i)_{i \in \N}\) such that 
\[
\sum_{i = 0}^{\infty}{\lambda_i \int_{\manfA}{{w}_{i}}} < + \infty.
\]
We then define \(w  \colon  \Omega \to [0, +\infty]\) by 
\[
w = \overline{w} + \sum_{i = 0}^{\infty}{\lambda_{i} w_{i}}\,.
\]
Observe that \(\gamma \in \Fuglede_{w}(\manfA; \Omega)\) is a Fuglede map with respect to \(\overline{w}\) and to each \({w}_{i}\). 
One thus has \(u_{j_{i}} \compose \gamma \in \Sobolev^{1, p}(\manfA)\) for every \(i \in \N\).{}
Moreover, by \eqref{eqDensityNew-818} and \eqref{eqDensityNew-829},
\[
D(u_{j_{i}} \compose \gamma) 
\to (Du \compose \gamma) [D\gamma{}]
\quad \text{in \(\Lebesgue^p (\manfA)\).}
\]
By the stability property of Sobolev functions, we deduce that \(u \compose \gamma \in \Sobolev^{1, p}(\manfA)\) and 
\[{}
u_{j_{i}} \compose \gamma \to u \compose \gamma{}
\quad \text{in \(\Sobolev^{1, p}(\manfA)\).}
\qedhere
\]
\end{proof}

We now consider genericity of Sobolev functions in the fractional setting that serves as another example where composition with Fuglede maps can be successfully implemented.
We begin by recalling the definition of \(\Sobolev^{s, p} (\manfV)\)\,:

\begin{definition}
\label{definitionSobolevFractional}
Given \(0 < s < 1\) and \(1 \le p < \infty\), the fractional Sobolev space \(\Sobolev^{s, p} (\manfV)\) is the set of all measurable functions \(u  \colon  \manfV \to \R\) such that \(u \in \Lebesgue^{p}(\manfV)\) and
\[
\seminorm{u}_{\Sobolev^{s, p}(\manfV)} 
\vcentcolon= \biggl( \int_{\manfV} \int_{\manfV} \frac{\abs{u (y) - u (z)}^p}{d(y, z)^{sp + m}} \dif z \dif y \biggr)^{\frac{1}{p}} < \infty.
\]
\end{definition}

We focus ourselves on the composition with Lipschitz maps \(\gamma \colon \manfA \to \manfV\) defined on a Riemannian manifold \(\manfA\) of dimension \( \ell \in \N_{*} \) having uniform measure \(\mu\), that is, such that there exists a constant \(C > 0\) satisfying
\begin{equation}
\label{eqFuglede-411}
\mu(B_{\rho}^{\ell}(a)) 
\le C \rho^{\ell}
\quad \text{for every \(a \in \manfA\) and \(\rho > 0\).}
\end{equation}
In contrast with \(\Sobolev^{1, p}\) spaces, we can choose a summable function \(w\) that encodes \(\Sobolev^{s, p}\)~genericity with an explicit expression in terms of \(u\)\,:

\begin{proposition}
\label{propositionModulusFractionalSobolev}
Given \(u \in \Sobolev^{s, p}(\manfV)\) with \(0 < s < 1\), let \(w  \colon  \manfV \to [0, +\infty]\) be the summable function defined for every \(x \in \manfV\) by
\begin{equation*}
w(x) = \abs{u(x)}^p + \int_{\manfV} \frac{\abs{u (x) - u (y)}^p}{d(x, y)^{sp + m}} \dif y.
\end{equation*}
If there exists \(c > 0\) such that, for every \(\xi \in \manfV\) and every \(0 < \rho \le \Diam{\manfV}\),
\(\cH^{m}(B_{\rho}^{m}(\xi)){}
\ge c \rho^{m}\),
then, for every Riemannian manifold \(\manfA\) of dimension \(\ell\) having uniform measure and every \(\gamma \in \Fuglede_{w}(\manfA; \manfV)\), it follows that 
\[{}
u \compose \gamma \in \Sobolev^{s, p}(\manfA)
\] 
and
\[{}
\seminorm{u \compose \gamma}_{\Sobolev^{s, p}(\manfA)}
\le C \abs{\gamma}_{\Lip}^{s} \norm{w \compose \gamma}_{\Lebesgue^{1}(\manfA)}^{1/p}.
\]
\end{proposition}

We denote the best Lipschitz constant of a Lipschitz map \(\gamma  \colon  \manfA \to \manfV\) by
\[{}
\abs{\gamma}_{\Lip}
= \sup_{\substack{x, y \in \manfA\\ x \ne y}}\frac{d(\gamma(x), \gamma(y))}{d(x, y)},
\]
where in the numerator we use the distance in \(\manfV\).
To prove Proposition~\ref{propositionModulusFractionalSobolev} we need the following estimate:

\begin{lemma}\label{lemma-Hedberg}
If \(\manfA\) is a Riemannian manifold of dimension \( \ell \) having uniform measure \(\mu\), then there exists \(C' > 0\) such that, for every \(\rho>0\) and every \(a\in \manfA\),
\[
\int_{\manfA\setminus B^{\ell}_\rho(a)}\frac{1}{d(a, t)^{sp + \ell}} \dif\mu(t){}
\leq \frac{C'}{\rho^{sp}}.
\]
\end{lemma}
\resetconstant
\begin{proof}[Proof of Lemma~\ref{lemma-Hedberg}]
Decomposing \(\manfA\setminus B^{\ell}_\rho(a)\) as the union of dyadic annuli,
\[
\manfA\setminus B^{\ell}_\rho(a) 
= \bigcup\limits_{j = 0}^{\infty}{\{t \in \manfA : 2^j \rho\leq d(a, t) < 2^{j+1}\rho\}},
\]
one gets
\[
\int_{\manfA\setminus B^{\ell}_\rho(a)}\frac{1}{d(a, t)^{sp+\ell}} \dif\mu(t){}
\leq
\sum_{j = 0}^{\infty}{\frac{1}{(2^{j}\rho)^{sp+\ell}}\mu\bigl(\{ t\in \manfA : 2^j \rho \leq d(a, t)< 2^{j+1}\rho\}\bigr)}.
\]
Since \(\mu\) is monotone and uniform,
\[{}
\mu\bigl(\{t\in \manfA : 2^j \rho \leq d(a,t) < 2^{j+1} \rho\}\bigr){}
\leq \mu\bigl(B^{\ell}_{2^{j+1}\rho}(a)\bigr){}
\leq \C (2^{j}\rho)^{\ell}
,
\]
which implies the desired inequality.
\end{proof}

\resetconstant
\begin{proof}[Proof of Proposition~\ref{propositionModulusFractionalSobolev}]
Let \(\gamma  \colon  \manfA \to \manfV\) be a Lipschitz map.
We claim that, for every measurable set \(E \subset \manfA\), we have
\begin{equation}
		\label{eqSobolevFractionalEstimateSubset}
  \int_{E} \int_{E} \frac{\abs{u \compose \gamma (r) - u \compose \gamma (t)}^p}{d(r, t)^{s p + \ell}} \dif\mu(t) \dif\mu(r)
  \le \C \abs{\gamma}_{\Lip}^{s p} \int_{E} w \compose \gamma \dif\mu,
\end{equation}
where \(w\) is the summable function given by the statement and \(\mu\) is the uniform measure of \(\manfA\).
Taking in particular \(E = \manfA\),  for every \(\gamma \in \Fuglede_{w}(\manfA; \manfV)\) we deduce that \(\seminorm{u \compose \gamma}_{\Sobolev^{s, p}(\manfA)} < \infty\).{}
Since \( \abs{u}^p \le w \), we then have \(u \compose \gamma \in \Sobolev^{s, p}(\manfA)\).

It suffices to prove \eqref{eqSobolevFractionalEstimateSubset} when \(\gamma\) is not constant.
We proceed as follows:
For every \(r, t \in \manfA\) and every \(y \in \manfV\), one has 
\begin{equation*}
\abs{u \compose \gamma (r) - u \compose \gamma (t)}^p
\le \Cl{cte-414} \bigl(\abs{u \compose \gamma (r) - u(y)}^p +  \abs{u \compose \gamma (t) - u(y)}^p \bigr).
\end{equation*}
We temporarily assume that \(\rho \vcentcolon= d(\gamma(r), \gamma(t)) > 0\).
Computing the average integral of both sides with respect to \(y\) over \(B^{m}_{\rho}(\xi)\) for some \(\xi \in \manfV\), we get
\[
\abs{u \compose \gamma (r) - u \compose \gamma (t)}^p
\le \Cr{cte-414} \biggl(\fint_{B_{\rho}^{m}(\xi)} \abs{u \compose \gamma (r) - u(y)}^{p} \dif y
+ \fint_{B_{\rho}^{m}(\xi)} \abs{u \compose \gamma (t) - u(y)}^{p} \dif y \biggl).
\]
By assumption, the measure of the ball \(B^{m}_{\rho}(\xi)\) in \( \manfV \) is bounded from below by \(c\rho^{m}\). 
Choosing \( \xi = \gamma(r) \), by the choice of \(\rho\) we have
\[{}
B^{m}_{\rho}(\xi) \subset B^{m}_{2\rho}(\gamma(r)) \cap B^{m}_{2\rho}(\gamma(t)).
\]
Hence,
\begin{multline*}
\abs{u \compose \gamma (r) - u \compose \gamma (t)}^p\\
\le \Cl{cte-442} \biggl(\int_{B_{2\rho}^{m}(\gamma(r))} \frac{\abs{u \compose \gamma (r) - u(y)}^{p}}{d(\gamma(r), y)^{m}} \dif y + \int_{B_{2\rho}^{m}(\gamma(t))} \frac{\abs{u \compose \gamma (t) - u(y)}^{p}}{d(\gamma(t), y)^{m}} \dif y \biggl).{}
\end{multline*}
Such an estimate also holds when \(\rho = 0\) since the left-hand side vanishes.
Dividing both sides by \(d(r, t)^{sp + \ell}\) and integrating with respect to \(r\) and \(t\) over a measurable set \(E \subset \manfA\), 
one gets
\begin{multline*}
  \int_{E} \int_{E} \frac{\abs{u \compose \gamma (r) - u \compose \gamma (t)}^p}{d(r, t)^{s p + \ell}} \dif\mu(t) \dif\mu(r) \\
  \le 2\Cr{cte-442} \int_{E} \int_{E} \int_{B_{2d (\gamma (r), \gamma (t))}^{m}(\gamma (r))}  \frac{1}{d(r, t)^{s p + \ell}}\frac{\abs{u \compose \gamma (r) - u(y)}^p}{d (\gamma(r), y)^m} \dif y \dif\mu(t) \dif\mu(r).
\end{multline*}
We observe that, for every \(r, t \in \manfA\),
\[
d (\gamma (r), \gamma (t)) 
\le \abs{\gamma}_{\Lip} \, d(r, t),
\]
and then
\begin{multline*}
  \int_{E} \int_{E} \frac{\abs{u \compose \gamma (r) - u \compose \gamma (t)}^p}{d(r, t)^{s p + \ell}} \dif\mu(t) \dif\mu(r) \\
  \le 2\Cr{cte-442} \int_{E} \int_{E} \int_{B_{2\abs{\gamma}_{\Lip} d(r, t)}^{m}(\gamma (r))}  \frac{1}{d(r, t)^{s p + \ell}}\frac{\abs{u \compose \gamma (r) - u(y)}^p}{d (\gamma(r), y)^m} \dif y \dif\mu(t) \dif\mu(r).
\end{multline*}
Using Tonelli's theorem, we interchange the order of integration between the variables \(y\) and \(t\) to get
\begin{multline*}
  \int_{E} \int_{E} \frac{\abs{u \compose \gamma (r) - u \compose \gamma (t)}^p}{d(r, t)^{s p + \ell}} \dif\mu(t) \dif\mu(r) \\
  \le 2\Cr{cte-442} \int_{E} \int_{\manfV} \biggl( \int_{\manfA \setminus B_{\frac{d(\gamma(r), y)}{2\abs{\gamma}_{\Lip}}}^{\ell}(r)}  \frac{1}{d(r, t)^{s p + \ell}}  \dif\mu(t)  \biggr) \frac{\abs{u \compose \gamma (r) - u(y)}^p}{d (\gamma(r), y)^m} \dif y  \dif\mu(r).
\end{multline*}
By Lemma~\ref{lemma-Hedberg}, we can estimate the integral with respect to \(t\) as follows
\[
\int_{\manfA \setminus B_{\frac{d(\gamma(r), y)}{2\abs{\gamma}_{\Lip}}}^{\ell}(r)}  \frac{1}{d(r, t)^{s p + \ell}}  \dif\mu(t) 
\le \C \frac{\abs{\gamma}_{\Lip}^{s p}}{d (\gamma(r), y)^{s p}}.
\]
We then get
\begin{multline*}
  \int_{E} \int_{E} \frac{\abs{u \compose \gamma (r) - u \compose \gamma (t)}^p}{d(r, t)^{s p + \ell}} \dif\mu(t) \dif\mu(r)\\
   \le \Cl{cte-479} \abs{\gamma}_{\Lip}^{s p} \int_{E} \int_{\manfV} \frac{\abs{u \compose \gamma (r) - u(y)}^p}{d (\gamma(r), y)^{sp + m}} \dif y  \dif\mu(r),
\end{multline*}
which gives \eqref{eqSobolevFractionalEstimateSubset} and completes the proof.
\end{proof}

\section{Opening}
\label{section_opening}

The composition of Sobolev functions with Fuglede maps sheds new light on the opening technique introduced by Brezis and Li~\cite{Brezis-Li} and  pursued in \cite{BPVS_MO}*{Section~2}. The main aim of this section is to explore these connections to illustrate the potential applications of generic composition with Fuglede maps. 
In order to simplify the presentation, we begin by focusing our attention on the case where the domain of the Sobolev functions and the Fuglede maps is the Euclidean space \(\R^m\). 
In this setting, Proposition~\ref{corollaryCompositionSobolevFuglede} asserts that for every \(u\in \Sobolev^{1,p}(\R^m)\), there exists a summable function \(w \colon \R^m\to [0,+\infty]\) such that, for every \(\gamma\in \Fuglede_w(\R^m;\R^m)\), we have \(u\circ \gamma \in \Sobolev^{1,p}(\R^m)\) and
\[
D(u\circ \gamma)(x)
= Du(\gamma(x))[D\gamma(x)]
\quad \text{for almost every \(x\in \R^m\).}
\]

This statement does not provide a \(\Sobolev^{1,p}\) estimate of the composition \(u\circ \gamma\). 
In this section, we investigate an additional quantitative perspective by considering specific functions \(\gamma\) for which such an estimate is available in terms of the derivatives of \(u\) and \(\gamma\). 
We also pursue this technique for higher-order Sobolev spaces \(\Sobolev^{k,p}(\R^m)\)
since such a generalization is required in Chapter~\ref{chapter-approximation-Sobolev-manifolds} and does not involve substantial new difficulties. 

Figure~\ref{figureJEMSOpeningPoint} illustrates the opening technique in a model situation involving open cubes \( Q_{r}^{m} = (-r, r)^{m} \) with \( r > 0 \).
More precisely, given a function \(u \in \Sobolev^{k, p}(\R^{m})\) and \(\delta > 0\), we wish to construct a smooth map \(\Phi^{\mathrm{op}}  \colon  \R^{m} \to \R^{m}\) such that
\begin{enumerate}[\((\mathrm{op}_1)\)]
	\item 
	\label{itemJEMS-11}
	\(u \circ \Phi^{\mathrm{op}}\) is constant in \(Q_{\delta}^{m}\)\,,{}
	\item{}
	\label{itemJEMS-14}
	 \(u \circ \Phi^{\mathrm{op}} = u\) in \(\R^{m} \setminus Q^{m}_{4\delta}\)\,,{}
	\item{}
	\label{itemJEMS-17}
	 \(u \circ \Phi^{\mathrm{op}} \in \Sobolev^{k, p}(\R^{m})\) and 
	\[{}
	\norm{u \circ \Phi^{\mathrm{op}}}_{\Sobolev^{k, p}(\R^{m})} \le C \norm{u}_{\Sobolev^{k, p}(\R^{m})},
	\]
	for some constant \(C > 0\) depending on \(m\) and \(\delta\).
\end{enumerate}

\begin{figure}
\centering
\includegraphics{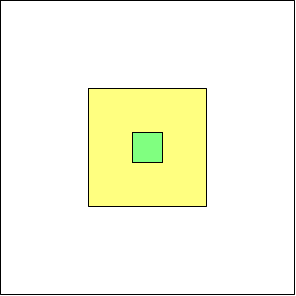}
\caption{\(u \compose \Phi^{\mathrm{op}}\) is constant in the green region}
\label{figureJEMSOpeningPoint}
\end{figure}

The construction of \( \Phi^{\mathrm{op}} \) is based on the following elementary inequality~\cite{BPVS_MO}*{Lemma~2.5}: 
Given a smooth map \(\zeta  \colon  \R^{m} \to \R^{m}\) and \(\Phi_{z}  \colon  \R^{m} \to \R^{m}\), with \(z \in \R^m\), defined by
\[{}
\Phi_{z}(x) = \zeta(x + z) - z,
\]
then for every summable function \(w  \colon  \R^{m} \to [0, +\infty]\) we have
\begin{equation}
	\label{eqJEMS-20}
\fint_{Q_{\delta}^{m}} \biggl( \int_{Q_{5\delta}^{m}} w \circ \Phi_{z} \biggr) \dif z
\le C' \int_{\R^{m}} w,
\end{equation}
where \(C' > 0\) is a constant depending on \(m\).

The proof of \eqref{eqJEMS-20} is based on Tonelli's theorem and an affine change of variables:{}
\[{}
\begin{split}
\int_{Q_{\delta}^{m}}{\biggl( \int_{Q_{5\delta}^{m}}{w(\zeta(x + z) - z) \dif x} \biggr) \dif z} 
& = \int_{Q_{\delta}^{m}}{\biggl( \int_{Q_{5\delta}^{m} + z}{w(\zeta(y) - z) \dif y} \biggr) \dif z}\\
& \le \int_{Q_{6\delta}^{m}}{\biggl( \int_{\R^{m}}{w(\zeta(y) - z) \dif z} \biggr) \dif y}\\
& = \int_{Q_{6\delta}^{m}}{\biggl(\int_{\R^{m}}{w(a) \dif a} \biggr) \dif y}
= |Q_{6\delta}^{m}| \int_{\R^{m}}{w}.
\end{split}
\]
Although this simple argument relies strongly on the Euclidean space setting, this approach fits more generally in the context of transversal perturbation of the identity that we present in Chapter~\ref{chapterGenericEllExtension}, see Example~\ref{exampleTransversalPertubationEuclidean} and Proposition~\ref{lemmaTransversalFamily}.

Starting with \(\zeta\) so that \(\zeta = 0\) in \(Q_{2\delta}^{m}\) and \(\zeta = \Id\) in \(\R^{m} \setminus Q_{3\delta}^{m}\), we then have for any \(z \in Q_{\delta}^{m}\) that \(\Phi_{z}\) is constant in \(Q_{\delta}^{m}\) and \(\Phi_{z} = \Id\) in \(\R^{m} \setminus Q_{4\delta}^{m}\).{}
Thus, for any function \(u \in \Sobolev^{k, p}(\R^{m})\), properties \((\mathrm{op}_{\ref{itemJEMS-11}})\) and \((\mathrm{op}_{\ref{itemJEMS-14}})\) are satisfied.
The issue is to choose a suitable point \(z\), now depending on \(u\), so that \((\mathrm{op}_{\ref{itemJEMS-17}})\) also holds.

Assuming that \(u\) is smooth and taking \(w = \abs{u}^{p} + \sum\limits_{i = 1}^{k}\abs{D^{i}u}^{p}\) in \eqref{eqJEMS-20}, one gets
\[{}
\fint_{Q_{\delta}^{m}}{\norm{u \circ \Phi_{z}}_{\Sobolev^{k, p}(Q_{5\delta}^{m})}^{p}} \dif z
\le C'' \norm{u}_{\Sobolev^{k, p}(\R^{m})}^{p},
\]
and then there exists \(a \in Q_{\delta}^{m}\) such that 
\[{}
\norm{u \circ \Phi_{a}}_{\Sobolev^{k, p}(Q_{5\delta}^{m})}^{p}
\le C'' \norm{u}_{\Sobolev^{k, p}(\R^{m})}^{p}.
\]
Such an averaging procedure is reminiscent of the work of Federer and Fleming~\cite{Federer-Fleming} and was adapted to Sobolev functions by Hardt, Kinderlehrer and Lin~\cite{Hardt-Kinderlehrer-Lin}.
For a general function \(u \in \Sobolev^{k, p}(\R^m)\), the argument above can be rigorously justified using an approximation of \(u\) by smooth functions, see \cite{BPVS_MO}*{Lemma 2.4} and also \cite{Hang-Lin-III}*{Section~7}.
As we explain in \cite{BPVS_MO}, where we rely on this approach via approximation, the notation \(u \compose \Phi_z\) that is suggested by the smooth case can be misleading: It is not meant to be a true composition between \(u\) and \(\Phi_z\)\,, but rather the limit of a sequence of smooth maps of the form \((u_{j} \compose \Phi_z)_{j \in \N}\).{}

Using Fuglede maps as we do here, one gets a neater approach where \(u \compose \Phi^{\mathrm{op}}\) is a true composition of functions.
As we now show, one has the following higher-order analogue of Proposition~\ref{corollaryCompositionSobolevFuglede}: 
\begin{proposition}
  \label{propositionFugledeWkp}
  Given \(u \in \Sobolev^{k, p}(\manfV)\), there exists a summable function \(w  \colon  \manfV \to [0, +\infty]\) such that \(\int_{\manfV}{w} \le 1\) and, for every \(\ell\in \N_*\), every Riemannian manifold \(\manfA\), and every smooth map \(\Phi \in \Fuglede_{w}(\manfA; \manfV)\) with bounded derivatives, we have \(u \compose \Phi \in \Sobolev^{k, p}(\manfA)\). Moreover, when \(\manfA\) and \(\manfV\) are open subsets of Euclidean spaces, for every \(t \in \{1, \dots, k\}\), 
  \begin{equation}\label{eq713}
  \norm{D^{t}(u \compose \Phi)}_{\Lebesgue^{p}(\manfA)}
  \le C''' \, \norm{w \compose \Phi}_{\Lebesgue^1(\manfA)} ^{1/p} \sum_{i=1}^{t}{B_{i} \norm{D^{i}u}_{\Lebesgue^{p}(\manfV)}}\text{,}
  \end{equation}
 where \(B_i \vcentcolon= \sum\limits_{r_{1} + \dots + r_{i} = t}{\norm{D^{r_{1}}\Phi}_{\Lebesgue^{\infty}(\manfA)} \cdots \norm{D^{r_{i}}\Phi}_{\Lebesgue^{\infty}(\manfA)}}\) and \(C''' > 0\) is a constant depending on \(k\), \(p\), \(m\) and \(\manfA\).
\end{proposition}

When \(\manfV = \manfM\) is a compact Riemannian manifold, the higher order Sobolev spaces \(\Sobolev^{k,p}(\manfM)\) are defined through a finite covering of \(\manfM\) by local charts, exactly as in the first-order case when \(k=1\). In particular, a sequence \((u_j)_{j\in \N}\) converges to some \(u\) in \(\Sobolev^{k,p}(\manfM)\) whenever for every chart \((U,\varphi)\) such that \(U\) is an open subset of \(\manfM\) and \(\varphi\) is a diffeomorphism from an open neighborhood of \(\overline{U}\) onto an open subset of  \(\R^m\), we have
\[
u_j\compose \varphi^{-1} \to u\compose \varphi^{-1} \quad \textrm{in  } \Sobolev^{k,p}(\varphi(U)).
\]
An analogue of the estimate~\eqref{eq713} also holds when \(\manfA\) or \(\manfV\) are Riemannian manifolds, but requires a suitable interpretation of \(D^t(u\circ \Phi)\) and \(D^tu\) that we do not pursue in this work.

We recall that in Proposition~\ref{lemmaModulusLebesguesequence} one has \( u \compose \gamma \in \Lebesgue^{p}(X) \) whenever \( w \compose \gamma \in \Lebesgue^{1}(X) \).
The conclusion is still valid if one replaces \( w \) by \( \lambda w \), with \( \lambda > 0\), or by any other summable function \( \widetilde{w} \ge w \).
The latter choice reduces in practice the family of maps \( \gamma \) that can be used, but shows the great flexibility at our disposal in the actual choice of summable function.
We now exploit this issue to have a robust way of getting an estimate for \( u \compose \gamma \) in terms of \( w \compose \gamma \) without having to keep track the way \( w \) is obtained.
To obtain the estimate in Proposition~\ref{propositionFugledeWkp}, we rely on the following

\begin{lemma}
\label{remarkLemmaFuglede}
Given \(u \in \Lebesgue^p(\manfV)\) and a summable function \(w  \colon  \manfV \to [0, \infty]\), there exist \(\lambda > 0\) and a summable function \(\widetilde w \ge \lambda w\) with \(\int_{\manfV} \widetilde w \le 2\)
such that, for every metric measure space \((X, d, \mu)\) and every measurable map \(\gamma  \colon  X \to \manfV\) with \(\widetilde w \compose \gamma \in \Lebesgue^{1}(X)\), we have 
\[
\norm{u \compose \gamma}_{\Lebesgue^p(X)}
\le \norm{\widetilde w \compose \gamma}_{\Lebesgue^1(X)}^{1/p} \norm{u}_{\Lebesgue^p(\manfV)}.
\]
\end{lemma}

\begin{proof}[Proof of Lemma~\ref{remarkLemmaFuglede}]
    We first assume that \(\int_{\manfV}|u|^p > 0\).
    In this case, let \(w_1 = \alpha w + |u|^p\) in \(\manfV\) where \(\alpha > 0\) is such that 
    \[
    \alpha \int_{\manfV} w \le \int_{\manfV} |u|^p.
    \]
    Thus,
    \begin{equation}
        \label{eqFuglede-728}
    \int_{\manfV} w_1 \le 2 \int_{\manfV} |u|^p.
    \end{equation}
    Since \(|u|^p \le w_1\), for every measurable map \(\gamma  \colon  X \to \manfV\) we have \(\abs{u \compose \gamma}^{p} \le w_1 \compose \gamma\) and then, by integration,
    \begin{equation}
        \label{eqFuglede-733}
    \int_{X}{\abs{u \compose \gamma}^{p} \dif\mu}
    \le
    \int_{X}{w_1 \compose \gamma \dif\mu}.
    \end{equation}
    Combining \eqref{eqFuglede-728} and \eqref{eqFuglede-733}, we then have
    \[
    \int_{X}{\abs{u \compose \gamma}^{p} \dif\mu}
    \le
     \frac{2 \int_{X}{w_1 \compose \gamma \dif\mu}}{\int_{\manfV} w_1}\int_{\manfV} |u|^p.
    \]
    The conclusion then follows with \(\widetilde w \vcentcolon= 2w_1/\int_{\manfV} w_1\) and \(\lambda \vcentcolon= 2\alpha/\int_{\manfV} w_1\).

    Now assuming that \(\int_{\manfV} |u|^p = 0\), the set \(\{u \ne 0\}\) is negligible in \( \manfV \).
    We take \(\alpha > 0\) such that \(\alpha \norm{w}_{\Lebesgue^1(\manfV)} \le 2\) and define the summable function \(\widetilde w = \alpha w + \infty \chi_{\{u \ne 0\}}\) in \(\manfV\).
    In particular, \(\int_{\manfV} \widetilde w \le 2\).
    For every measurable map \(\gamma  \colon  X \to \manfV\) with \(\widetilde w \compose \gamma \in \Lebesgue^1(X)\), we have \(\gamma(x) \in \{u = 0\}\) for \(\mu\)-almost every \(x \in X\) and then \(u \compose \gamma = 0\) \(\mu\)-almost everywhere in \(X\).
    Hence, 
    \[
    \int_{X}
		{\abs{u \compose \gamma}^{p} \dif\mu}
    = 0.
    \]
    As a result, the estimate in the statement also holds in this case.
\end{proof}

\resetconstant
\begin{proof}[Proof of Proposition~\ref{propositionFugledeWkp}]
We first assume that \(\manfV=\Omega\) is a Lipschitz open set.
  Let \((u_{j})_{j \in \N}\) be a sequence in \((\Sobolev^{k, p} \cap \Smooth^{\infty})(\Omega)\) that converges to \(u\) in \(\Sobolev^{k, p}(\Omega)\).
  Applying inductively Proposition~\ref{lemmaModulusLebesguesequence} to the sequences \((D^{t}u_{j})_{j \in \N}\) for each \(t \in \{0, \dots, k\}\), one finds summable functions \(w_{t}  \colon  \Omega \to [0, +\infty]\) and an increasing sequence of positive integers \(({j_{i}})_{i \in \N}\) such that, for every \(\Phi \in \Fuglede_{w_{t}}(\manfA; \Omega)\), the sequence \((D^{t} u_{j_i} \compose \Phi)_{i \in \N}\) converges to \(D^{t} u \compose \Phi\) in \(\Lebesgue^{p}(\manfA)\). 
  By Lemma~\ref{remarkLemmaFuglede} applied to \(w_t\)  and \(\abs{D^t u}\) for each \(t \in \{0, \dots, k\}\), there exist summable functions  \(\widetilde{w}_{t}  \colon  \Omega \to [0, +\infty]\) such that \(\widetilde{w}_t\ge \lambda_t w_t \) for some \(\lambda_t>0\), \(\int_{\Omega}{\widetilde{w}_{t}} \leq 2\)
    and moreover, every  \(\Phi \in \Fuglede_{\tilde{w}_{t}}(\manfA ; \Omega)\) verifies
\begin{equation}
\label{eqJEMS-274}
\int_{\manfA}{\abs{D^{t}u \compose \Phi}^{p}}
\le \int_{\manfA}{\widetilde{w}_{t} \compose \Phi}
\int_{\Omega}{\abs{D^{t}u}^{p}}.
\end{equation}

Take
\[{}
w = \frac{1}{2(k+1)} \sum_{t=0}^{k}\widetilde{w}_{t}.
\]
Then, \(\int_{\Omega}w \leq 1\).
Since \(\widetilde{w}_{t} \le 2(k+1)  w\), we have 
\[{}
\int_{\manfA}{\widetilde{w}_{t} \compose \Phi} 
\le 2(k+1) \int_{\manfA}{w \compose \Phi},
\]
which by \eqref{eqJEMS-274} implies that
\begin{equation}
\label{eq-901}
\int_{\manfA}{\abs{D^{t}u \compose \Phi}^{p}}
\le 2(k+1) \biggl(\int_{\manfA}{w \compose \Phi}\biggr) \int_{\Omega}{\abs{D^{t}u}^{p}}.
\end{equation}

Using the Fa\'a di Bruno composition formula for \(D^{t}(u_{j_{i}} \compose \Phi)\) and proceeding as in the proof of Proposition~\ref{corollaryCompositionSobolevFuglede}, one shows that if in addition \(\Phi\) is smooth and has bounded derivatives in \(\manfA\), then \(u \compose \Phi \in \Sobolev^{k, p}(\manfA)\).{}
When \(\manfA\) is an open subset of a Euclidean space, the composition formula implies that:
\[{}
\resetconstant
\abs{D^{t}(u \compose \Phi)}^{p}
\le \Cl{cteJEMS-303} \sum_{i = 1}^{t}{B_{i}^{p} \, \abs{D^{i}u \compose \Phi}^{p}},
\]
where the \(B_{i}\)s are the constants of the statement.
Integrating the above estimate and applying \eqref{eq-901}, 
\[{}
\begin{aligned}
\int_{\manfA}{\abs{D^{t}(u \compose \Phi)}^{p}}
& \le \Cr{cteJEMS-303} \sum_{i = 1}^{t}{B_{i}^{p} \int_{\manfA}{\abs{D^{i}u \compose \Phi}^{p}}}\\
& \le 2(k+1) \Cr{cteJEMS-303}  \biggl(\int_{\manfA}{w \compose \Phi}\biggr) \sum_{i = 1}^{t}{ B_{i}^{p} \int_{\Omega}{\abs{D^{i}u}^{p}}},
\end{aligned}
\]
which completes the proof of the proposition when \(\manfV=\Omega\) is a Lispchitz open set. When \(\manfV=\manfM\) is a compact manifold without boundary, we can repeat the above arguments for each chart of an open covering of \(\manfM\), exactly as in the proof of Proposition~\ref{corollaryCompositionSobolevFuglede}. 
\end{proof}

We also mention the following analogue of Proposition~\ref{lemmaModulusSobolevsequenceHighk} in the setting of \(\Sobolev^{k,p}\) spaces:
\begin{proposition}
\label{lemmaModulusSobolevsequenceHighkp}
Let \((u_j)_{j \in \N}\) be a sequence of functions in \(\Sobolev^{k, p}(\manfV)\) such that
\[{}
u_{j} \to u
\quad \text{in \(\Sobolev^{k, p}(\manfV)\).}
\]
Then, there exist a summable function \(w  \colon  \manfV \to [0, +\infty]\) and a subsequence \((u_{j_i})_{i \in \N}\) such that, for every Riemannian manifold \(\manfA\) and every \(\gamma \in \Fuglede_{w}(\manfA; \manfV)\), the sequence \((u_{j_{i}} \compose \gamma)_{i \in \N}\) is contained in \(\Sobolev^{k, p}(\manfA)\) and 
\[{}
u_{j_{i}} \compose \gamma \to u \compose \gamma \quad \text{in \(\Sobolev^{k, p}(\manfA)\).}
\]
\end{proposition}
The proof closely follows that of Proposition~\ref{lemmaModulusSobolevsequenceHighk}. The only modification needed is to replace the chain rule for first-order derivatives with the Faà di Bruno formula. The details of the argument are omitted.

\bigskip
We now return to the problem introduced at the beginning of this section:
Applying Proposition~\ref{propositionFugledeWkp} with \(\manfA=\manfV = \R^{m}\) for a given \(u\in \Sobolev^{k,p}(\R^m)\), one gets a summable function \(w\), which is then used in the estimate \eqref{eqJEMS-20}. 
In view of the latter, we may take \(a \in Q_{\delta}^{m}\) such that 
\[
\int_{Q_{5\delta}^{m}} w \circ \Phi_{a}
\le C' \int_{\R^{m}} w.
\]
Since by the choice of \(\zeta\) we have \(\Phi_{a} = \Id\) in \(\R^{m} \setminus Q_{5\delta}^{m}\)\,, we deduce that \(\Phi_{a}\) verifies properties~\((\mathrm{op}_{\ref{itemJEMS-11}})\)--\((\mathrm{op}_{\ref{itemJEMS-17}})\). 

Let us mention two additional features that we have not exploited so far but are important in the application of the opening technique.
Firstly, relying on Proposition~\ref{propositionFugledeWkp} with \(\manfA = \manfV = Q_{5\delta}^{m}\) we get an estimate that does not involve \(u\) on the entire space \(\R^{m}\).{}
Secondly, it is possible to obtain an estimate for each derivative \(D^{j}u\) and then better track the dependence on the factor \(\delta\).{}
Combining those ideas in our case, we finally get a map \(\Phi^{\mathrm{op}}\) with the desired properties and such that
  \[{}
  \norm{D^{j}(u \compose \Phi^{\mathrm{op}})}_{\Lebesgue^{p}(Q_{5\delta}^{m})}
  \le C'''' \sum_{i=1}^{j}{\delta^{i - j}\norm{D^{i}u}_{\Lebesgue^{p}(Q_{5\delta}^{m})}}
  \quad \text{for every \(j \in \{1, \dots, k\}\),}
  \]
for some constant \(C'''' > 0\) depending on \(m\) and \(p\).
In particular, when \(j = 1\) such an estimate is invariant under \(\delta\).
Observe that the estimate in \(\R^{m}\) can be deduced from this one since \(\Phi^{\mathrm{op}} = \Id\) in \(\R^{m} \setminus Q_{5\delta}^{m}\)\,.

\section{Convergence in generic 1-dimensional complexes}
\label{sectionSimplex}

Sobolev \(\Sobolev^{1, p}\)~functions on an interval are equal almost everywhere to a conti\-nuous function.
In this section, we consider a generic version of this property with respect to the skeleton of a \(1\)-dimensional simplicial complex.
We begin by recalling the notions of a simplex and a simplicial complex in a Euclidean space \(\R^{a}\) with  \(a\in \N_*\).{}

\begin{definition}
\label{definitionSimplex}
Given \(\ell\in \N\) and points \(x_0, \dots, x_\ell \in \R^{a}\) that are not contained in any affine subspace of dimension less than \(\ell\), the \emph{simplex} \(\Simplex^\ell\) of dimension \(\ell\) with vertices \(x_0, \dots, x_\ell\) is the convex set
\[
\Simplex^\ell{}
\vcentcolon= \biggl\{\sum_{i=0}^\ell \mu_i x_i : \mu_i\geq 0 \ \text{and}\  \sum_{i=0}^{\ell}\mu_i=1\biggr\}.
\] 
\end{definition}

We then introduce a \emph{face} \(\Sigma^{r}\) of \(\Simplex^\ell\) of dimension \(r\in \{0, \dots, \ell\}\) as the convex hull of \(r+1\) vertices of \(\Simplex^\ell\). 
In particular, \(\Simplex^\ell\) is the unique \(\ell\)-dimensional face of itself.

\begin{definition}
\label{definitionComplex}
A \emph{simplicial complex} \(\cK^\ell\) of dimension \(\ell\) is a finite set of simplices with the following properties:
	\begin{enumerate}[$(i)$]
\item each simplex of \(\cK^\ell\) is a face of an \(\ell\)-dimensional simplex in \(\cK^{\ell}\), 
\item each face of a simplex in \(\cK^{\ell}\) belongs to \(\cK^{\ell}\), 
\item if \(\Sigma_{1}^{r_1}, \Sigma_{2}^{r_2} \in \cK^{\ell}\), then \(\Sigma_{1}^{r_1} \cap \Sigma_{2}^{r_2}\) is either empty or a face of both \(\Sigma_{1}^{r_1}\) and \(\Sigma_{2}^{r_2}\).
	\end{enumerate}
\end{definition}

An \(r\)-dimensional subcomplex \(\cL^{r}\) of \(\cK^{\ell}\) is a subset of \(\cK^{\ell}\) which is  a simplicial complex of dimension \(r\) in the above sense.
In particular,  the set of all simplices of dimension less than or equal to \(r\) in \(\cK^{\ell}\) is a subcomplex, which we denote by \(\cK^{r}\).

The notation \(K^\ell\) refers to the union of the simplices in \(\cK^\ell\). 
Since the \emph{polytope} \(K^\ell\) is a subset of \(\R^a\), we may equip \(K^{\ell}\) with the distance induced by the Euclidean norm \(\abs{\cdot}\) in \(\R^a\). 
For every \(x\in K^{\ell}\) and \(\rho>0\), we then denote by \(B^{\ell}_\rho(x)\) the open ball of center \(x\) and radius \(\rho\) in \(K^\ell\), which is the intersection of the corresponding Euclidean ball with \( K^{\ell} \).
To each \(K^{r}\) with \(r \in \{0, \ldots, \ell\}\), we associate the Hausdorff measure \(\cH^r\). 

A Sobolev function can be simply defined in the \(1\)-dimension polytope \(K^{1}\) by imposing a continuity condition to handle compatibility at the vertices of \(K^{1}\).{}
More precisely,

\begin{definition}
	Let \(1 \le p < \infty\).
	The Sobolev space \(\Sobolev^{1, p} (K^1)\) on a simplicial complex \(\cK^{1}\) is the set of measurable functions \(v  \colon  K^{1} \to \R\) such that
	\begin{enumerate}[$(i)$]
	\item for every \(1\)-dimensional face \(\Sigma^{1} \in \cK^1\), \(v \vert_{\Sigma^{1}} \in \Sobolev^{1, p} (\Sigma^{1})\),{}
	\item there exists a continuous function \(\widehat v  \colon  K^{1} \to \R\), which we call the continuous representative of \(v\), with \(\widehat v = v\) \(\cH^{1}\)-almost everywhere in \(K^{1}\).
	\end{enumerate} 
\end{definition}

We prove a genericity property of Sobolev functions on \(1\)-dimensional polytopes:{}

\begin{proposition}
\label{corollaryCompositionSobolevFugledeDim1} 
Given \(u \in \Sobolev^{1, 1}(\manfV)\), there exists a summable function \(w  \colon  \manfV \to [0, +\infty]\) with
\begin{equation}
\label{eqFuglede-652}
\norm{w}_{\Lebesgue^{1}(\manfV)}
\le 2 \norm{u}_{\Sobolev^{1, 1}(\manfV)}
\end{equation}
such that, for every simplicial complex \(\cK^{1}\) and every \(\gamma \in \Fuglede_{w}(K^{1}; \manfV)\), we have \(u \compose \gamma \in \Sobolev^{1, 1}(K^{1})\) and its continuous representative \( \widehat{u \circ \gamma} \) verifies
\begin{equation}
\label{eqFuglede-937}
\norm{\widehat{u \circ \gamma}}_{\Smooth^{0}(K^{1})}
\le C (1 + \abs{\gamma}_{\Lip}) \norm{w \circ \gamma}_{\Lebesgue^{1}(K^{1})} \text{,}
\end{equation}
where \(C > 0\) depends on \(K^{1}\).
\end{proposition}

We begin with a standard \(1\)-dimensional estimate for Lipschitz functions based on the Fundamental theorem of Calculus:

\begin{lemma}
	\label{lemmaPoincareWirtinger-New}
	Let \(\cK^{1}\) be a simplicial complex such that \(K^1\) is connected.
	There exists \(\kappa > 0\) such that, for every \(x \in K^{1}\), every \(\rho > 0\), and every Lipschitz function \(v  \colon  K^{1} \to \R\), 
\[{}
\abs{v(y) - v(z)}
\le 2 \norm{Dv}_{\Lebesgue^{1}(B_{\kappa\rho}^{1}(x))}
\quad \text{for every \(y, z \in B_{\rho}^{1}(x)\)}.
\]
\end{lemma}

\begin{proof}[Proof of Lemma~\ref{lemmaPoincareWirtinger-New}]
We first observe that by connectedness of \(K^{1}\) there exists \(\kappa>0\) such that, for every \(\rho>0\), \(x\in K^1\) and \(y\in B^{1}_\rho(x)\), there exists a Lipschitz path \(\sigma \colon   [0, 1]\to K^{1}\cap B_{\kappa\rho}^{1}(x)\) such that \(\sigma(0)=x\) and \(\sigma(1) = y\).
Indeed, take
\[
\delta = \min{\Bigl\{ d(\Sigma_1^1, \Sigma_2^1) :  \Sigma_1^1, \Sigma_2^1 \in \cK^1 \ \text{and}\ \Sigma_1^1 \cap \Sigma_2^1 = \emptyset\Bigr\}}.
\]
If all elements in \(\cK^1\) intersect each other, then simply take any \(\delta > 0\).

Let \(0 < \rho \le \delta\). 
If \(x\in K^1\) and \(y\in B^{1}_\rho(x)\) belong to the same face \(\Sigma^1 \in \cK^1\), then it suffices to take \(\sigma(t) = (1-t)x + ty \in B^{1}_\rho(x)\).
Otherwise, since \(\rho \le \delta\), by the definition of \(\delta\) we must have \(x \in \Sigma_1^1\) and \(y \in \Sigma_2^1\) with \(\Sigma_1^1 \cap \Sigma_2^1\) non-empty, and then  \(\Sigma_1^1 \cap \Sigma_2^1 = \{\xi\}\) for some vertex \(\xi \in K^0\). 
By finiteness of \(\cK^1\) and an elementary trigonometric argument, there exists \(0 < \theta \le 1\), which can be taken independent of \(\Sigma_1^1\) and \(\Sigma_2^1\), such that, for \(i, j \in \{1, 2\}\) with \(i \ne j\),
\[
\theta \abs{z - \xi} \le d(z, \Sigma_j^1)
\quad \text{for every \(z \in \Sigma_{i}^1\)}.
\]
We thus have, with \(j = 1\) and \(i = 2\),
\[
\theta \abs{y - \xi} 
\le d(y, \Sigma_1^1)
\le \abs{y - x}
< \rho.
\]
Similarly, we also have \(\theta \abs{x - \xi} < \rho\). 
It then follows that both segments \([y, \xi]\) and \([x, \xi]\) are contained in the ball \(B_{2\rho/\theta}(x)\).
Hence, one finds a Lipschitz path \(\sigma\) with \(\kappa = 2/\theta\).
Finally, for \(\rho > \delta\), by connectedness of the polytope \(K^1\) there exists a Lipschitz path \(\sigma\) from \(x\) to \(y\). 
Its image is necessarily contained in \(B_{\kappa\rho}^{1}(x)\) with \(\kappa=\Diam{K^1}/\delta\).{}
As a result, it thus suffices to choose
\[{}
\kappa = \max{\bigl\{ 1, 2/\theta, \Diam{K^1}/\delta  \bigr\}}.{}
\]

Given \(\rho>0\), \(x\in K^1\) and \(y \in B^{1}_\rho(x)\), take a Lipschitz path \(\sigma\) as above. 
We can further require  that \(\sigma\) is either constant or injective on \([0, 1]\).
Then,
\begin{align*}
|v(y)-v(x)|
&= |v\circ \sigma (1)-v\circ \sigma(0)|\\
&\leq \int_{0}^{1}|Dv(\sigma(t))||\sigma'(t)| \dif t{}
=\int_{\sigma([0,1])} |Dv(s)| \dif\cH^1(s),
\end{align*}
where the last equality follows from the area formula. 
Since \(\sigma([0, 1]) \subset  B_{\kappa\rho}^{1}(x)\), one gets
\[
|v(y)-v(x)| \leq \int_{ B_{\kappa\rho}^{1}(x)} |Dv(s)| \dif\cH^1(s).
\]
The conclusion then follows from the triangle inequality.
\end{proof}

We now proceed by approximation as in the proof of Proposition~\ref{corollaryCompositionSobolevFuglede}:

\begin{proof}[Proof of Proposition~\ref{corollaryCompositionSobolevFugledeDim1}]
We first assume that \(\norm{u}_{\Sobolev^{1, 1}(\manfV)} > 0\).
We take a sequence \((u_{j})_{j \in \N}\) in \((\Sobolev^{1, 1} \cap \Smooth^{\infty})(\manfV)\) that converges to \(u\) in \(\Sobolev^{1, 1}(\manfV)\).
From Proposition~\ref{lemmaModulusLebesguesequence} applied recursively with \( (u_{j})_{j \in \N} \) and \( (Du_{j})_{j \in \N} \)\,, there exist a summable function \(w  \colon  \manfV \to [0, +\infty]\) and a subsequence \((u_{j_{i}})_{i \in \N}\) in \((\Sobolev^{1, 1} \cap \Smooth^{\infty})(\manfV)\) converging to \(u\) in \(\Sobolev^{1, 1}(\manfV)\) and such that, if \(\gamma \in \Fuglede_{w}(K^{1}; \manfV)\), then \(u \compose \gamma\) and \(Du \compose \gamma\) belong to \(\Lebesgue^{1}(K^{1})\), and
\begin{equation}
\label{eqFuglede-989}
u_{j_i} \compose \gamma \to u \compose \gamma \quad \text{and} \quad 
D u_{j_i} \compose \gamma \to D u \compose \gamma \quad \text{in \(\Lebesgue^1 (K^1)\).}
\end{equation}
Since \(\gamma\) is Lipschitz-continuous, we also deduce from the stability property of Sobolev functions that \(u \compose \gamma\vert_{\Sigma^{1}} \in \Sobolev^{1, 1}(\Sigma^{1})\) for every \(\Sigma^{1} \in \cK^{1}\).

In the spirit of Lemma~\ref{remarkLemmaFuglede}, we now replace \( w \) by a summable function \( \widetilde{w} \colon \manfV \to [0, +\infty] \) with more suitable properties, but the choice we make here is slightly different.
In our case, it suffices to take
\begin{equation}
\label{eqFuglede-709}
\widetilde{w} 
= \alpha w + \abs{u} + \abs{Du},
\end{equation}
where \( \alpha > 0 \) is such that \( \alpha \norm{w}_{\Lebesgue^1(\manfV)} \le \norm{u}_{\Sobolev^{1, 1}(\manfV)} \), which is possible since \(\norm{u}_{\Sobolev^{1, 1}(\manfV)} > 0\).
In particular, estimate \eqref{eqFuglede-652} holds with \( w \) replaced by \( \widetilde{w} \).

We may henceforth assume that \(K^1\) is connected.
Indeed, when that is not the case, it suffices to apply the conclusion of the statement to each connected component of \(K^1\).
Note that if estimate \eqref{eqFuglede-937} is satisfied on each component then it also holds on the entire polytope \(K^1\).

Take \( \gamma \in \Fuglede_{\tilde w}(K^{1}; \manfV) \).
Since \( \widetilde{w} \ge \alpha w \), we have \( \gamma \in \Fuglede_{w}(K^{1}; \manfV) \) and then \eqref{eqFuglede-989} is satisfied.
As we are assuming that \(K^1\) is connected, we may apply Lemma~\ref{lemmaPoincareWirtinger-New} with \(v \vcentcolon= u_{j_{i}} \compose \gamma\).
Using that 
\[
\abs{Dv} \le \abs{Du_{j_{i}} \compose \gamma} \, \abs{\gamma}_{\Lip}
\quad \text{\(\cH^{1}\)-almost everywhere in \(K^{1}\),}
\]
for each \( x \in K^{1} \) and \( \rho > 0 \) we have
\begin{equation}
\label{eqFuglede-563}
\abs{u_{j_{i}} \compose \gamma(y) - u_{j_{i}} \compose \gamma(z)}
\le 2\abs{\gamma}_{\Lip} \norm{Du_{j_{i}} \compose \gamma}_{\Lebesgue^{1}(B_{\alpha\rho}^{1}(x))}
\quad \text{for every \(y, z \in B_{\rho}^{1}(x)\)}.
\end{equation}
As the sequence \((D u_{j_i} \compose \gamma)_{i \in \N}\) converges in \(\Lebesgue^{1}(K^{1})\), it is equi-integrable.
We then have from \eqref{eqFuglede-563} that \((u_{j_{i}} \compose \gamma)_{i \in \N}\) is equicontinuous in \(K^{1}\).
Since \(u_{j_{i}} \compose \gamma \to u \compose \gamma\) in \(\Lebesgue^{1}(K^{1})\), we conclude that there exists \(f \in \Smooth^{0}(K^{1})\) such that
\[{}
u_{j_{i}} \compose \gamma{} \to f
\quad \text{in \(\Smooth^{0}(K^{1})\),}
\]
and \(f = u \compose \gamma\) \(\cH^{1}\)-almost everywhere in \(K^{1}\).{}
Hence, \(f = \widehat{u \compose \gamma}\) and, for each \(x \in K^{1}\) and \(\rho > 0\),
\begin{equation}
\label{eqFuglede-670}
\abs{\widehat{u \compose \gamma}(y) - \widehat{u \compose \gamma}(z)}
\le 2\abs{\gamma}_{\Lip} \norm{Du \compose \gamma}_{\Lebesgue^{1}(B_{\kappa\rho}^{1}(x))}
\quad \text{for every \(y, z \in B_{\rho}^{1}(x)\).}
\end{equation}
In particular, \(\widehat{u \compose \gamma}\) is uniformly continuous.
Taking \(\rho = \Diam{K^{1}}\) and using the triangle inequality, we get
\[{}
\abs{\widehat{u \compose \gamma}(y)}
\le \abs{\widehat{u \compose \gamma}(z)}
+ 2\abs{\gamma}_{\Lip} \norm{Du \compose \gamma}_{\Lebesgue^{1}(K^{1})}.
\]
To obtain the uniform estimate for the continuous representative, it suffices to integrate with respect to \(z \in K^{1}\) and apply the pointwise estimate \( \widetilde w \ge \abs{u} + \abs{Du} \).
We thus have the conclusion when \(\norm{u}_{\Sobolev^{1, 1}(\manfV)} > 0\) by taking the summable function \( \widetilde{w} \).

If we have \(\norm{u}_{\Sobolev^{1, 1}(\manfV)} = 0\), then \(u = 0\) almost everywhere in \(\manfV\) and we take \(w  \colon  \manfV \to [0, +\infty]\) defined by \(w = \infty \chi_{\{u \ne 0\}}\).
In particular, \eqref{eqFuglede-652} holds since both sides equal zero.
Note that for every simplicial complex \(\cK^1\) and every \(\gamma \in \Fuglede_w(K^1; \manfV)\), we have \(u \compose \gamma = 0\) almost everywhere in \(K^1\) and then 
\begin{equation}
\label{eqFuglede-1073}
    \widehat{u \compose \gamma} = 0
    \quad \text{in \(K^1\)}
\end{equation} 
since \(K^1\) is the finite union of \(1\)-dimensional simplices.
Thus, \eqref{eqFuglede-937} is also satisfied in this case. 
\end{proof}

We next investigate uniform convergence on generic \(1\)-dimensional polytopes for \(\Sobolev^{1, 1}\)~functions.{}
We begin with a convergent sequence in \(\Sobolev^{1, 1}\):{}

\begin{proposition}
	\label{lemmaFugledeSobolevDetectorDim1}
	Let \(u \in \Sobolev^{1, 1}(\manfV)\) and let \((u_{j})_{j \in \N}\) be a sequence in \(\Sobolev^{1, 1}(\manfV)\) such that
	\[{}
	u_{j} \to u
	\quad \text{in \(\Sobolev^{1, 1}(\manfV)\).}
	\]
	Then, there exist a summable function \(w  \colon  \manfV \to [0, +\infty]\) and a subsequence \((u_{j_{i}})_{i \in \N}\) such that, for every simplicial complex \(\cK^{1}\) and every \(\gamma \in \Fuglede_{w}(K^{1}; \manfV)\), the functions
	\(u \compose \gamma\) and all \(u_{j_{i}} \compose \gamma\) belong to \(\Sobolev^{1, 1}(K^{1})\), and
	\[{}
	\widehat{u_{j_{i}} \compose \gamma} \to \widehat{u \compose \gamma}
	\quad \text{in \(\Smooth^{0}(K^{1})\).}
	\]
\end{proposition}

\resetconstant
\begin{proof}
	Let \(\widetilde{w}_{1}\) be the summable function for \(u\) given by Proposition~\ref{corollaryCompositionSobolevFugledeDim1} and, for each \(j \in \N\), let \(w_{j}\) be the summable function associated to each \(u_{j} - u\) given by the same proposition.
	By estimate \eqref{eqFuglede-652} and the \(\Sobolev^{1, 1}\) convergence of \((u_{j})_{j \in \N}\), we have \(w_{j} \to 0\) in \(\Lebesgue^{1}(\manfV)\).{}
	We then apply Proposition~\ref{lemmaModulusLebesguesequence} to \((w_{j})_{j \in \N}\) and get a subsequence \((w_{j_{i}})_{i \in \N}\) and a summable function \(\widetilde{w}_{2}\) satisfying its conclusion.
	
	Let \(w = \widetilde{w}_{1} + \widetilde{w}_{2}\).{}
	For any simplicial complex \(\cK^{1}\) and \(\gamma \in \Fuglede_{w}(K^{1}; \manfV)\), 
	we have \(u \compose \gamma \in \Sobolev^{1, 1}(K^{1})\).{}
	We also have, for each \(i \in \N\), \(w_{j_{i}} \compose \gamma \in \Lebesgue^{1}(K^{1})\) and then \((u_{j_{i}} - u) \compose \gamma \in \Sobolev^{1, 1}(K^{1})\).{}
	Hence, \(u_{j_{i}} \compose \gamma \in \Sobolev^{1, 1}(K^{1})\).{}
	The continuous representative of \((u_{j_{i}} - u) \compose \gamma\) is thus given by \(\widehat{u_{j_{i}} \circ \gamma} - \widehat{u \circ \gamma}\) and, by the estimate in Proposition~\ref{corollaryCompositionSobolevFugledeDim1}, for every \(i \in \N\) we have
	\[{}
	\norm{\widehat{u_{j_{i}} \circ \gamma} - \widehat{u \circ \gamma}}_{\Smooth^{0}(K^{1})}
	\le C (1 + \abs{\gamma}_{\Lip}) \norm{w_{j_{i}} \circ \gamma}_{\Lebesgue^{1}(K^{1})}.
	\]
	Since \(w_{j_{i}} \compose \gamma \to 0\) in \(\Lebesgue^{1}(K^{1})\), the conclusion follows. 
\end{proof}

We next consider uniform convergence by composition with generic sequences of Fuglede maps.
We begin by introducing a notion of convergence in \(\Fuglede_{w}\) classes that we shall often refer to.

\begin{definition}
	\label{definitionFugledeConvergence}
    Let \(w  \colon  \manfV \to [0, +\infty]\) be a summable function and let \((\gamma_{j})_{j \in \N}\) be a sequence in \(\Fuglede_{w}(X; \manfV)\), where \(X\) is a metric measure space.
    Given \(\gamma \in \Fuglede_{w}(X; \manfV)\), we say that \((\gamma_{j})_{j \in \N}\) converges to \(\gamma\) in \(\Fuglede_{w}(X; \manfV)\), which we denote by
	\[{}
	\gamma_{j} \to \gamma \quad \text{in \(\Fuglede_{w}(X; \manfV)\),}
	\] 
	whenever  \((\gamma_{j})_{j \in \N}\) converges uniformly to \(\gamma\) in \(X\) and there exist constants \(C_{1}, C_{2} > 0\) such that, for every \(j \in \N\),
\[{}
\abs{\gamma_{j}}_{\Lip}
\le C_{1}
\quad \text{and} \quad
\norm{w \compose \gamma_{j}}_{\Lebesgue^{1}(X)}
\le C_{2}.
\]
\end{definition}

Using this notation, we have the following

\begin{proposition}
	\label{propositionFugledeApproximationMapsDim1}
	If \(u \in \Sobolev^{1, 1}(\manfV)\) and \(w\) is the summable function given by the proof of Proposition~\ref{corollaryCompositionSobolevFugledeDim1}, then there exists a summable function \(\widetilde w  \colon  \manfV \to [0, +\infty]\) with \(\widetilde w \ge w\) in \(\manfV\)
	such that, for every simplicial complex \(\cK^{1}\) and every sequence \((\gamma_{j})_{j \in \N}\) in \(\Fuglede_{\tilde w}(K^{1}; \manfV)\) with
	\[{}
	\gamma_{j} \to \gamma \quad \text{in \(\Fuglede_{\tilde w}(K^{1}; \manfV)\),}
	\]
	we have
	\[{}
	\widehat{u \compose \gamma_{j}} \to \widehat{u \compose \gamma} \quad \text{in \(\Smooth^{0}(K^{1})\).}
	\]
\end{proposition}

\begin{proof}
	Applying Proposition~\ref{propositionFugledeApproximationMaps} to \(u\) and \(Du\), one finds a summable function \(w_{1}\) such that, for every simplicial complex \(\cK^{1}\) and every sequence \((\gamma_{j})_{j \in \N}\) with \(\gamma_{j} \to \gamma\) in \(\Fuglede_{w_{1}}(K^{1}; \manfV)\),{}
	we have
	\begin{equation}
		\label{eqFuglede-837}
	u \compose \gamma_{j} \to u \compose \gamma 
	\quad \text{and} \quad 
	Du \compose \gamma_{j} \to Du \compose \gamma 
	\quad \text{in \(\Lebesgue^{1}(K^{1})\).}
	\end{equation}
	Let \(\widetilde w = w + w_{1}\) and assume that \(\gamma_{j} \in \Fuglede_{\tilde w}(K^{1}; \manfV)\).
    If \( \norm{u}_{\Sobolev^{1, 1}(\manfV)} > 0 \),
	then by \eqref{eqFuglede-670} for every \(x \in K^{1}\) and every \( \rho > 0 \) we have
	\[{}
	\abs{\widehat{u \circ \gamma_{j}}(y) - \widehat{u \circ \gamma_{j}}(z)}
	\le 2 \abs{\gamma_{j}}_{\Lip} \norm{Du \circ \gamma_{j}}_{\Lebesgue^{1}(B_{\kappa\rho}^{1}(x))}
	\quad \text{for every \(y, z \in B_{\rho}^{1}(x)\),}
	\]
 where \(\kappa\) is the parameter introduced in Lemma~\ref{lemmaPoincareWirtinger-New}.
	If \(\gamma_{j} \to \gamma\) in \(\Fuglede_{\tilde w}(K^{1}; \manfV)\), then \(\abs{\gamma_{j}}_{\Lip} \le C_{1}\) for every \(j \in \N\) and \eqref{eqFuglede-837} holds, so that \((Du \circ \gamma_{j})_{j \in \N}\) is equi-integrable.
	We deduce that \((\widehat{u \circ \gamma_{j}})_{j \in \N}\) is equicontinuous in \(K^{1}\) and then, by \(\Lebesgue^{1}\)~convergence of \((u_{j} \compose \gamma)_{j \in \N}\) to \(u \compose \gamma\), we get
	\[{}
	\widehat{u \compose \gamma_{j}} \to \widehat{u \compose \gamma}
	\quad \text{in \(\Smooth^{0}(K^{1})\).}
	\]
    When  \( \norm{u}_{\Sobolev^{1, 1}(\manfV)} = 0 \), for every \( j \in \N \) we have by \eqref{eqFuglede-1073} that \( \widehat{u \compose \gamma_{j}} = \widehat{u \compose \gamma} = 0 \) in \( K^{1} \) and the conclusion is trivially true in this case.
\end{proof}

Proposition~\ref{lemmaFugledeSobolevDetectorDim1} also has a valid counterpart for the generic \(\Sobolev^{1, 1}\) convergence involving a subsequence \( (u_{j_{i}}) \) in the spirit of Proposition~\ref{lemmaModulusSobolevsequenceHighk} where the manifold \(\manfA\) is replaced by the simplicial complex \(\cK^1\). 
However, under the assumptions of Proposition~\ref{propositionFugledeApproximationMapsDim1}, one should not expect generic convergence in \(\Sobolev^{1, 1}\) as the following example illustrates:

\begin{example}
	Let \(u \in \Smooth_{c}^{\infty}(\R^{2})\) and \( 1 \le p < \infty\).
	For every summable function \(w  \colon  \R^{2} \to \R\) and for almost every \(a \in \R^{2}\), there exists a sequence \((\gamma_{j})_{j \in \N}\) in \(\Fuglede_{w}([0, 1]; \R^{2})\) such that \(\gamma_{j} \to a\) in \(\Fuglede_{w}([0, 1]; \R^{2})\) and
	\begin{equation}
		\label{eqFuglede-856}
	\norm{(u \compose \gamma_{j})'}_{\Lebesgue^{p}(0, 1)} \to C_{p} |\nabla u(a)|,
	\end{equation}
	for some constant \(C_{p} > 0\).
	Hence, if \(\nabla u(a) \ne 0\), then
	\[{}
	u \compose \gamma_{j} \not\to u(a)
	\quad \text{in \(\Sobolev^{1, p}(0, 1)\).}
	\]	
	Indeed, take a sequence of positive numbers \((\epsilon_{j})_{j \in \N_{*}}\) converging to zero to be chosen later on.
	Given \(a \in \R^{2}\), let \(\gamma_{a, j}  \colon  [0, 1] \to \R^{2}\) be defined by
	\[{}
	\gamma_{a, j}(t)
	= a + \epsilon_{j} \Bigl( \cos{\frac{t}{\epsilon_{j}}}, \sin{\frac{t}{\epsilon_{j}}} \Bigr).{}
	\]
	Then, \((\gamma_{a, j})_{j \in \N_{*}}\) is equiLipschitz and converges uniformly in \([0, 1]\) to the constant map \(a\) as \(j \to \infty\).{}
	We claim that, for almost every \(a \in \R^{2}\), there exist a subsequence \((\gamma_{a, j_{i}})_{i \in \N}\) and \(C > 0\) such that
	\[{}
	\int_{0}^{1} w \compose \gamma_{a, j_{i}}
	\le C
	\quad \text{for every \(i \in \N\).}
	\]
	Indeed, for every \(j \in \N\), by Tonelli's theorem and a change of variables,
	\[{}
 \begin{split}
	\int_{\R^{2}} \biggl( \int_{0}^{1} w \compose \gamma_{a, j}  \biggr) \dif a
	& = \int_{0}^{1} \biggl( \int_{\R^{2}} w\Bigl( a + \epsilon_{j} \Bigl( \cos{\frac{t}{\epsilon_{j}}}, \sin{\frac{t}{\epsilon_{j}}} \Bigr)  \Bigr) \dif a  \biggr) \dif t\\
	& = \int_{0}^{1} \biggl( \int_{\R^{2}} w(b) \dif b \biggr) \dif t
	= \int_{\R^{2}} w.
 \end{split}
	\]
	It thus follows from Fatou's lemma that
	\[{}
	\int_{\R^{2}} \biggl(\liminf_{j \to \infty} \int_{0}^{1} w \compose \gamma_{a, j}  \biggr) \dif a
	\le \int_{\R^{2}} w < \infty.
	\]
	For almost every \(a \in \R^{2}\), the integrand is thus finite and we can extract a subsequence \((\gamma_{a, j_{i}})_{i \in \N}\) such that \(\gamma_{a, j_{i}} \to a\) in \(\Fuglede_{w}([0, 1]; \R^{2})\).{}
	Observe that, for every \(j \in \N\) and every \(t \in [0, 1]\),
	\[{}
	\begin{split}
	(u \compose \gamma_{a, j})'(t)
	= \nabla u(\gamma_{a, j}(t)) \cdot \gamma'_{a, j}(t)
	& = \nabla u(a) \cdot \Bigl( - \sin{\frac{t}{\epsilon_{j}}}, \cos{\frac{t}{\epsilon_{j}}} \Bigr) + o(1)\\
	& = -|\nabla u(a)| \sin{\Bigl(\frac{t}{\epsilon_{j}} - \theta \Bigr)} + o(1),
	\end{split}
	\]
	for some \(0 \le \theta < 2\pi\) depending on \(\nabla u(a)/|\nabla u(a)|\) when \(\nabla u(a) \ne 0\).{}
	Thus,
	\[{}
	\norm{(u \compose \gamma_{a, j})'}_{\Lebesgue^{p}(0, 1)}
	= \abs{\nabla u(a)} \biggl(\int_{0}^{1} \Bigl|\sin{\Bigl(\frac{t}{\epsilon_{j}} - \theta \Bigr)} \Bigr|^{p} \dif t\biggr)^{\frac{1}{p}} + o(1). 
	\]
	Choosing \(\epsilon_{j} = 1/2\pi j\), we then get by a change of variables and \(\pi\) periodicity of \(\abs{\sin}\),{}
	\[{}
 \begin{split}
	\int_{0}^{1} \Bigl|\sin{\Bigl(\frac{t}{\epsilon_{j}} - \theta \Bigr)} \Bigr|^{p} \dif t
	& = \frac{1}{2\pi j} \int_{-\theta}^{2\pi j - \theta} |\sin{s}|^{p} \dif s\\
	  & = \frac{1}{2\pi j} \int_{0}^{2\pi j} |\sin{s}|^{p} \dif s
	= \frac{1}{\pi} \int_{0}^{\pi} |\sin{s}|^{p} \dif s.
 \end{split}
	\]
	Taking 
	\[{}
	C_{p} = \biggl(\frac{1}{\pi} \int_{0}^{\pi} |\sin{s}|^{p} \dif s\biggr)^{\frac{1}{p}},
	\]	
    one then obtains \eqref{eqFuglede-856}.
\end{example}

\section{Connection to Fuglede's work}
\label{section_Fuglede}

Let us compare the notion of genericity introduced in this chapter with Fuglede's original work~\cite{Fuglede}, where he introduced the concept of \emph{modulus} to quantify the size of a family \(\Gamma\) of Borel measures on \(\manfV\). The modulus of \(\Gamma\) is defined as
\[
\Mod{\Gamma} = \inf{\left\{ \int_{\manfV} f : f \ge 0, \ f \ \text{is measurable in} \ \manfV \ \text{and} \ \int_{\manfV} f \, \dif\nu \ge 1 \ \text{for all} \ \nu \in \Gamma \right\}}.
\]
A property \(\mathcal{P}\) is then said to hold for almost every measure if it fails only for a family of measures with modulus zero.
For example, by~\cite{Fuglede}*{Theorem~3}, if \(f_j \to f\) in \(\Lebesgue^{1}(\manfV)\), then there exists a subsequence \((f_{j_i})_{i \in \N}\) such that, for almost every Borel measure \(\nu\) on \(\manfV\),
\[
\lim_{i \to \infty}{\int_{\manfV} |f_{j_i} - f| \dif\nu} = 0.
\]

This robust notion of genericity can be reformulated without explicitly mentioning the modulus. According to~\cite{Fuglede}*{Theorem~2}, a property \(\mathcal{P}\) holds for almost every measure if and only if there exists a summable function \(w \colon \manfV \to [0, +\infty]\) such that \(\mathcal{P}\) holds for any measure \(\nu\) on \(\manfV\) satisfying
\begin{equation}
\label{eqFuglede-1340}
\int_{\manfV} w \dif\nu < \infty.
\end{equation}

We may now apply this formulation to pushforward measures.
We recall that, given a continuous map \(\gamma \colon X \to \manfV\) on a metric measure space \((X, d, \mu)\), the pushforward measure \(\nu \vcentcolon= \gamma_* \mu\) is defined for any Borel set \(U \subset \manfV\) by
\[
\gamma_* \mu (U) = \mu (\gamma^{-1} (U)).
\] 
Since
\[
\int_{\manfV} w \dif\nu = \int_{X} w \circ \gamma \dif\mu
\quad \text{and} \quad 
\int_{\manfV} |f_{j_i} - f| \dif\nu
= \int_{X} |f_{j_i} \circ \gamma - f \circ \gamma| \dif\mu,
\]
it follows that Proposition~\ref{lemmaModulusLebesguesequence} above can be deduced from Fuglede's theory applied to these pushforward measures.

Fuglede uses this notion of genericity to characterize limits of irrotational vector fields in Lebesgue spaces and to describe Sobolev functions through integration of \(\Lebesgue^p\) differential forms.
These problems are handled using measures \(\mathcal{H}^k \lfloor_{\gamma(\manfA)}\) given by biLipschitz homeomorphisms \(\gamma \colon \manfA \to \gamma(\manfA) \subset \manfV\) defined on an open subset \(\manfA\) of \(\R^k\) with \(k \in \{1, \ldots, m-1\}\).
These measures \(\mathcal{H}^k \lfloor_{\gamma(\manfA)}\) are defined for any Borel subset \(U \subset \manfV\) by 
\[
(\mathcal{H}^k \lfloor_{\gamma(\manfA)})(U) = \mathcal{H}^k (\gamma(\manfA) \cap U)
= \int_{\gamma^{-1}(U)} \Jacobian{k}{\gamma},
\]
where \(\mathcal{H}^k\) is the \(k\)-dimensional Hausdorff measure on \(\manfV\) and \(\Jacobian{k}{\gamma}\) is the \(k\)-dimensional Jacobian of \(\gamma\),
\begin{equation}
\label{eqFugledeJacobian}
\Jacobian{k}{\gamma} (\xi) = \bigl[\det{\bigl(D \gamma (\xi)^* \compose D \gamma (\xi) \bigr)}\bigr]^{\frac{1}{2}}.    
\end{equation}
The genericity condition of a property \(\mathcal{P}\) for the measures \(\mathcal{H}^k \lfloor_{\gamma(\manfA)}\) in terms of \eqref{eqFuglede-1340} means that there exists a summable function \(w \colon \manfV \to [0, +\infty]\) such that  \(\mathcal{P}\) holds for any biLipschitz homeomorphism \(\gamma \colon \manfA \to \gamma(\manfA) \subset \manfV\) satisfying
\begin{equation}
\label{eqFuglede-1308}
\int_{\manfA} w \circ \gamma \, \Jacobian{k}{\gamma}
= \int_{\gamma(\manfA)} w \dif\mathcal{H}^k \lfloor_{\gamma(\manfA)}{} 
 < \infty.
\end{equation}

Here, our viewpoints diverge as we allow maps \(\gamma\) that are not necessarily biLipschitz or even injective, and we replace the Jacobian-weighted integral \eqref{eqFuglede-1308} by the condition
\begin{equation}
\label{eqFuglede-1384}
\int_{\manfA} w \circ \gamma < \infty.
\end{equation}
These differences from Fuglede's work are crucial for us.
More specifically, non-injective Lipschitz transformations naturally arise in homotopical constructions, even in the basic case of a homotopy to a constant map. The generic opening of Sobolev functions, introduced in Section~\ref{section_opening} and applied later in Chapter~\ref{chapter-approximation-Sobolev-manifolds} to the approximation of Sobolev mappings, also relies on compositions of maps that are never injective.
By expanding the class of admissible maps \(\gamma\), we require a stronger integral condition given by \eqref{eqFuglede-1384}.
Our approach to generic properties on lower-dimensional sets, introduced in Section~\ref{sectionSimplex} and pursued in the next chapter, thus follows a different path from Fuglede's original one.

\cleardoublepage

\chapter{VMO detectors}
\label{chapter-3-Detectors}

Given an integer \(p > 1\), \(\Sobolev^{1,p}\) functions defined on a \(p\)-dimensional domain generally lack continuity. However, they exhibit a vanishing mean oscillation (\(\VMO\)) property, which, as we explain in the next chapter, serves as a robust substitute for continuity in capturing the topological features of Sobolev maps with values in a manifold.

In this chapter, we investigate the stability properties of Sobolev functions under generic composition within the \(\VMO\) framework. To this end, we introduce the notion of an \(\ell\)-detector \(w\) for a given measurable function \(u\). An \(\ell\)-detector is a summable function \(w\) such that \(u \circ \gamma \in \VMO(K^\ell)\) whenever \(K^\ell\) is a polytope of dimension \(\ell\) and \(\gamma\) belongs to \(\Fuglede_w(K^\ell)\). This concept provides a precise mechanism for identifying compositions that preserve the generic \(\VMO\) nature of \(u\) at the \(\ell\)-dimensional level.

We establish the existence of \(\ell\)-detectors for Sobolev functions \(u\) in \(\Sobolev^{1,p}\) with \(p \geq \ell\) or in \(\Sobolev^{s,p}\) with \(0 < s < 1\) and \(sp \geq \ell\), regardless of whether \(p\) is integer or not. These results demonstrate the versatility of \(\VMO\) as a framework for generalizing continuity-like properties across different Sobolev scales.

To unify these findings, we introduce the space \(\VMO^\ell\) as a common framework encompassing these various Sobolev spaces concerning generic \(\VMO\) stability at dimension \(\ell\). Functions in \(\VMO^\ell\) enjoy a fundamental stability property: If \(\gamma_j \to \gamma\) in \(\Fuglede_w(K^\ell)\), then \(u \circ \gamma_j \to u \circ \gamma\) in \(\VMO(K^\ell)\). This extends the one-dimensional stability result involving \(\Smooth^0\) convergence from the previous chapter to the higher-dimensional setting.

\section{VMO functions}
\label{sectionVMOFunctions}

Continuity provides a natural framework for classical topological problems, such as homotopy theory and continuous extensions.  
However, many topological constructions can be extended to a slightly broader setting where functions need not be continuous. 
In this context, the space \(\VMO\) of vanishing mean oscillations offers a good replacement for continuity.
In fact, we develop in the next chapter a \(\VMO\)~theory for the topological degree based on the work of Brezis and Nirenberg~\cite{BrezisNirenberg1995}.
In this section, we focus on the definition and examples of \(\VMO\)~functions.

We begin by recalling the notion of a \(\VMO\) function defined in a metric measure space \( X \) equipped with a Borel measure \( \mu \)\,:

\begin{definition}
We say that a function \(v \in \Lebesgue^{1}(X)\) has \emph{vanishing mean oscillation} whenever, for each \( \epsilon > 0 \), there exists \( \delta > 0 \) such that, for every \( 0 < \rho \le \delta \) and \( x \in X \), we have
\[
\int_{B_{\rho}(x)}{\int_{B_{\rho}(x)}{\abs{v(y) - v(z)}  \dif\mu(z) \dif\mu(y)}}
\le \epsilon \, \mu(B_{\rho}(x))^2 \text{,}
\]
where \(B_\rho(x)\) is the open ball of center \(x \in X\) and radius \(\rho > 0\). 
\end{definition}

Assuming that \(  \mu \) is nondegenerate, that is, \( \mu(B_{\rho}(x)) > 0 \) for every \( x \in X \) and \( \rho > 0 \), we may define the \emph{mean oscillation} of \(v\) at the scale \(\rho > 0\) as
\begin{equation}
\label{eqDetector-33}
\seminorm{v}_{\rho} = \sup_{x \in X}{\fint_{B_{\rho}(x)}{\fint_{B_{\rho}(x)}{\abs{v(y) - v(z)}  \dif\mu(z) \dif\mu(y)}}}.
\end{equation}
Then, \( v \) has vanishing mean oscillation if and only if 
\begin{equation*}
\label{conditionVMO}
\lim_{\rho \to 0}{\seminorm{v}_{\rho}} 
= 0.
\end{equation*}
Although we rely on the double average of \(v\) to define the mean oscillation, it is also customary to use instead the equivalent quantity
\[
\fint_{B_{\rho}(x)}{\biggabs{v(y) - \fint_{B_{\rho}(x)} v}  \dif\mu(y)}.
\]
Their equivalence is easily established as follows.
On the one hand, by monotonicity of the integral,
\begin{equation}
\label{eqDetector-48}
\begin{split}
\fint_{B_{\rho}(x)}{\biggabs{v(y) - \fint_{B_{\rho}(x)} v}  \dif\mu(y)}
& = \fint_{B_{\rho}(x)}{\biggabs{ \fint_{B_{\rho}(x)} [v(y) - v(z)] \dif\mu(z)}  \dif\mu(y)} \\
& \le \fint_{B_{\rho}(x)}{\fint_{B_{\rho}(x)}{\abs{v(y) - v(z)}  \dif\mu(z) \dif\mu(y)}}.
\end{split}
\end{equation}
On the other hand, by the triangle inequality, for every \(y, z \in B_{\rho}(x)\) we have
\[
\abs{v(y) - v(z)}
\le \biggabs{v(y) - \fint_{B_{\rho}(x)} v} + \biggabs{v(z) - \fint_{B_{\rho}(x)} v}
\]
and then, taking the average integrals with respect to \(y\) and \(z\), 
\begin{equation}
\label{eqDetector-61}
\fint_{B_{\rho}(x)}{\fint_{B_{\rho}(x)}{\abs{v(y) - v(z)}  \dif\mu(z) \dif\mu(y)}}
\le 2 \fint_{B_{\rho}(x)}{\biggabs{v(y) - \fint_{B_{\rho}(x)} v}  \dif\mu(y)}.
\end{equation}
The combination of \eqref{eqDetector-48} and \eqref{eqDetector-61} gives the equivalence we claimed.

\bigskip

Denote by \(\UContinuous(X)\) the vector space of bounded uniformly continuous functions on \(X\) equipped with the sup norm.
When \(X\) is compact, \(\UContinuous(X)\) coincides with \(\Smooth^{0}(X)\).
We also denote by \(\VMO(X)\) the vector space of all functions \(v \in \Lebesgue^{1}(X)\) that have vanishing mean oscillation and we associate to it the norm
\begin{equation}
\label{eqDetector-VMO}
\norm{v}_{\VMO(X)}
\vcentcolon= \norm{v}_{\Lebesgue^{1}(X)} + \seminorm{v}_{\VMO(X)}\,,
\end{equation}
where
\begin{equation}
\label{eqDetectorSeminorm}
\seminorm{v}_{\VMO(X)} \vcentcolon= \sup_{\rho > 0}{\seminorm{v}_{\rho}}\,.
\end{equation}
When \(X\) has finite measure the continuous imbedding occurs
\[{}
\UContinuous(X) \subset \VMO(X),
\]
with strict inclusion.
For example, a function \(v \colon  (-1/2, 1/2) \to \R\) defined for \(x \ne 0\) by \(v(x) = \log{\abs{\log{\abs{x}}}}\) belongs to \(\VMO(-1/2, 1/2)\).{}

We shall assume that \( X \) is a locally compact and separable metric space and that the measure \(\mu\) satisfies
\begin{itemize}
\item  the \emph{doubling property}: 
There exists \(C>0\) such that
\begin{equation*}
\mu(B_{2\rho}(x)) \leq C \mu(B_\rho(x)){}
\quad
\text{for every \(x \in X\) and \(\rho > 0\),}
\end{equation*}
\item the \emph{metric continuity}:
For every \(\epsilon > 0\), there exists \(\delta > 0\) such that
\begin{equation*}
\mu(B_{\rho_{1}}(x_{1}) \triangle B_{\rho_{2}}(x_{2})) \le \epsilon{}
\quad \text{whenever \(d(x_{1}, x_{2}) + \abs{\rho_{1} - \rho_{2}} \le \delta\),}
\end{equation*}
where \(A \triangle B \vcentcolon= (A \setminus B) \cup (B \setminus A)\) is the symetric difference between two sets \(A\) and \(B\),{}
\item the \emph{uniform nondegeneracy}: 
For every \(\rho > 0\), there exists \(\eta > 0\) such that
\begin{equation}\label{eq-uniform-nondegeneracy}
\mu(B_{\rho}(x)) \ge \eta{}
\quad \text{for every \(x \in X\).}
\end{equation}
\end{itemize}

We prove in Section~\ref{sectionBUC} that under these assumptions \(\UContinuous(X)\) is dense in \( \VMO(X) \).
Observe that, as a consequence of the metric continuity with \(x_{1} = x_{2} = x \) one has 
\[
\mu(\partial B_{\rho}(x)) = 0
\quad
\text{for every \(x \in X\) and \(\rho > 0\).}
\]
When the measure \(\mu\) is clear from the context, we shall omit it.
Some examples of spaces \(X\) that we consider later on satisfying these assumptions are: \(X = \Omega\) is an open subset of \(\R^{m}\) with Lipschitz boundary, \(X = \manfM\) is a compact Riemannian manifold and \(X = K^{\ell}\) is the polytope of a simplicial complex \(\cK^{\ell}\), see Definition~\ref{definitionComplex}.
Let us clarify why the latter is included in this setting:

\begin{example}
Let \(\cK^{\ell}\) be a simplicial complex and consider the polytope \(K^{\ell} \subset \R^a\) equipped with the Hausdorff measure \(\mu \vcentcolon= \cH^\ell\lfloor_{K^\ell}\).
If \(\Sigma^{\ell}\) is a simplex in \(\cK^{\ell}\), then there exist constants \(c, C>0\)  such that, for every \(x\in \Sigma^{\ell}\) and \(0 < \rho \le \Diam{\Sigma^{\ell}}\),
\[
c \rho^{\ell} 
\leq \cH^{\ell}(B^\ell_{\rho}(x)\cap \Sigma^{\ell})
\leq C \rho^{\ell}.
\]
Since \(\cK^{\ell}\) is finite, we can choose \(c, C>0\) so that both inequalities hold true for every \(\Sigma^{\ell}\in \cK^{\ell}\). 
In particular, the above estimate implies the doubling property for \(\mu\). 
In a similar vein, the metric continuity of \(\cH^\ell\) on every \(\ell\)-dimensional affine plane in \(\R^a\) implies the metric continuity of \(\mu\) on every \(\Sigma^\ell \in \cK^{\ell}\), and finally on the whole \(K^\ell\). 
\end{example}

\begin{example}
	\label{exampleDetectorsVMOSobolev}
	Given a bounded convex open set \(\Omega \subset \R^{m}\), then every \(v \in \Sobolev^{1, m}(\Omega)\) has vanishing mean oscillation.
	Indeed, by the Poincaré-Wirtinger inequality, one has for every \(x \in \Omega\) and every \(\rho > 0\),
	\[{}
	\fint_{B_{\rho}^{m}(x) \cap \Omega}\fint_{B_{\rho}^{m}(x) \cap \Omega}{\abs{v(y) - v(z)} \dif z \dif y}
	\le C \norm{Dv}_{\Lebesgue^{m}(B_{\rho}^{m}(x) \cap \Omega)}.
	\]
	By summability of \(\abs{Dv}^{m}\) and Lebesgue's dominated convergence theorem, the right-hand side converges uniformly to zero with respect to \(x\) as \(\rho \to 0\).{}
\end{example}

\begin{example}
	\label{exampleVMOSobolevFractional}
	Given \(0 < s < 1\) and a compact Riemannian manifold \(\manfM\) of dimension \(m\), each \(v \in \Sobolev^{s, {m}/{s}}(\manfM)\) has vanishing mean oscillation.
Indeed, for every \(\xi \in \manfM\) and \(0 < \rho \le \Diam{\manfM}\), we have
\(\cH^{m}(B_{\rho}^{m}(\xi)) \ge c \rho^{m}\).
	Using Hölder's inequality one finds that, for every \(x \in \manfM\) and every \(0 < \rho \le \Diam{\manfM}\),
	\begin{multline*}
	\fint_{B_{\rho}^{m}(x)}\fint_{B_{\rho}^{m}(x)}{\abs{v(y) - v(z)} \dif\cH^{m}(z)\dif\cH^{m}(y)}\\
	\le C \biggl(\int_{B_{\rho}^{m}(x)} \int_{B_{\rho}^{m}(x)} \frac{\abs{v (y) - v(z)}^m}{d(y, z)^{2m}} \dif\cH^{m}(z) \dif\cH^{m}(y)\biggr)^{\frac{1}{m}},
	\end{multline*}
	where the right-hand side converges uniformly to zero with respect to \(x\) as \(\rho \to 0\) by Lebesgue's dominated convergence theorem.{}
\end{example}

\section{\texorpdfstring{$\ell$}{l}-detectors}

One deduces from Proposition~\ref{corollaryCompositionSobolevFugledeDim1} that given a \(\Sobolev^{1, 1}\) function \(u\), then for generic maps \(\gamma\) defined on a \(1\)-dimensional polytope \(K^{1}\), the continuous representative  \(\widehat{u \compose \gamma}\) belongs to \(\Smooth^{0}(K^{1})\) and in particular \(u \compose \gamma\) itself belongs to \(\VMO(K^{1})\).{}
In Example~\ref{exampleDetectorsVMOSobolev} we saw that \(\VMO\) is the right substitute to continuity for \(\Sobolev^{1, m}\) maps in dimension \(m\).{}
We now pursue \(\VMO\) genericity for Sobolev functions on \(\ell\)-dimensional polytopes for any given \(\ell \in \N\).
We begin with the concept of detector to help us identify Fuglede maps that yield the \(\VMO\) property under composition, where \(\manfV\) is a Lipschitz open subset of \(\R^m\) or a compact Riemannian manifold of dimension \(m\) without boundary. 

\begin{definition}
	Given \(\ell \in \N\) and a measurable function \(u \colon  \manfV \to \R\), we say that a summable function \(w \colon  \manfV \to [0, +\infty]\) is a \emph{\((\VMO, \ell)\)-detector} for \(u\) whenever,	for each simplicial complex \(\cK^{\ell}\) of dimension \(\ell\) and each \(\gamma \in \Fuglede_{w}(K^{\ell}; \manfV)\),
	\[{}
	u \compose \gamma{}
	\in \VMO(K^{\ell}).
	\]
\end{definition}

We simply call \(w\) an \emph{\(\ell\)-detector} for \(u\).{}
Note that a detector is strongly associated to the function it refers to.
That said, when \(u\) is clear from the context, we shall simply say that \(w\) is an \(\ell\)-detector.
We do not explicitly mention the \(\VMO\) nature of \(u\) with respect to composition as we shall  systematically be dealing with this property.
It might be interesting in other contexts to keep track of other spaces like \(\Lebesgue^{p}\) (Proposition~\ref{lemmaModulusLebesguesequence}), \(\Sobolev^{1, p}\) (Proposition~\ref{corollaryCompositionSobolevFuglede}) or \(\Sobolev^{s, p}\) for \(0 < s < 1\) (Proposition~\ref{propositionModulusFractionalSobolev}) which we do not further pursue in this work.  

\begin{proposition}
	\label{propositionFugledeDetector}
	Let \(u \colon  \manfV \to \R\) be a measurable function.
	Every \(\ell\)-detector for \(u\) is also an \(r\)-detector with \(r \in \{0, \dots, \ell - 1\}\).
\end{proposition}

To go from a polytope of dimension \(r\) to another one of dimension \(\ell > r\), we merely add a dummy variable \(t \in [0, 1]^{\ell - r}\).
To this end, we rely on an elementary extension property for \(\VMO\) functions:

\begin{lemma}
	\label{lemmaVMOellExtension}
	Given a simplicial complex \(\cK^{r}\) and a function \(v \colon  K^{r} \to \R\), let \(\widetilde v \colon  K^{r} \times [0, 1]^{\ell - r} \to \R\) be defined by
	\[{}
	\widetilde{v}(x, t)
	= v(x).
	\]
	If\/ \(\widetilde{v} \in \VMO(K^{r} \times [0, 1]^{\ell - r})\), then \(v \in \VMO(K^{r})\) and 
	\[{}
	\norm{v}_{\VMO(K^{r})}
	\le C \norm{\widetilde{v}}_{\VMO(K^{r} \times [0, 1]^{\ell - r})}.
	\]
\end{lemma}

Note that \(E^{\ell} \vcentcolon= K^{r} \times [0, 1]^{\ell - r}\) can be seen as the polytope of a simplicial complex \(\cE^\ell\) such that \(\cK^r \subset \cE^r\), see \cite{Munkres}*{Chapter~2, Lemma~7.8}.
Assuming that \(K^{r} \subset \R^{a}\), we have 
\begin{equation}
\label{eqDetectors-168}
E^{\ell}= K^{r} \times [0, 1]^{\ell - r}\subset \R^{a + \ell - r}.
\end{equation}
On \(E^{\ell}\), the product measure \(\mu \vcentcolon= \cH^{r} \otimes \cH^{\ell - r}\) coincides with the Hausdorff measure \(\cH^{\ell}\) associated to the distance inherited from \( \R^{a + \ell - r} \).
To simplify the computation that follows, we work with an equivalent product metric on \(E^{\ell}\), namely
\begin{equation*}
\widetilde d((y, s), (z, t))
\vcentcolon= \max{\bigl\{ d(y, z), \abs{s - t} \bigr\}}
\quad \text{for every \((y, s), (z, t) \in E^{\ell}\).}
\end{equation*}
If we denote by \(\widetilde{B}_\rho(x, 0')\) the open ball in \(E^{\ell}\) with respect to \(\widetilde d\) of center \((x, 0') \in E^{\ell}\) and radius \(\rho\), then
	\begin{equation}
		\label{eqDetectorBallProduct}
	\widetilde{B}_{\rho}(x, 0'){}
	= B_{\rho}^{r}(x) \times \bigl(B_{\rho}^{\ell - r} \cap [0, 1]^{\ell - r}\bigr),
	\end{equation}
	where \(B_{\rho}^{r}(x)\) and \(B_{\rho}^{\ell - r}= B_{\rho}^{\ell - r}(0')\) are open balls in \(K^{r}\) and \(\R^{\ell - r}\), respectively.
	It then follows from \eqref{eqDetectorBallProduct} that, for every measurable function \(\varphi \colon  K^{r} \times [0, 1]^{\ell - r} \to [0, +\infty]\) and every \(x \in K^{r}\),
	\begin{equation}
		\label{eqDetectorEstimateVMOellExtension}
	\int_{\widetilde B_{\rho}(x, 0')}{\varphi \dif\mu}
	=
	\int_{B_{\rho}^{r}(x)}\biggl( \int_{B_{\rho}^{\ell - r} \cap [0, 1]^{\ell - r}} \varphi \dif\cH^{\ell - r}  \biggr) \dif\cH^{r}.
	\end{equation}

\resetconstant
\begin{proof}[Proof of Lemma~\ref{lemmaVMOellExtension}]
	By Tonelli's theorem,
	\[{}
	\norm{\widetilde v}_{\Lebesgue^{1}(E^{\ell})}
	= \norm{v}_{\Lebesgue^{1}(K^{r})},
	\]
	where \(E^{\ell}\) is given by \eqref{eqDetectors-168}.
	We now prove that, for every \(0 < \rho \le 1\),
	\begin{equation}
	\label{eqExtension-104}
	\seminorm{\widetilde{v}}_{\rho}
	\ge c \seminorm{v}_{\rho} \,.
	\end{equation}
	To this end, we apply \eqref{eqDetectorEstimateVMOellExtension} twice to \(\varphi\) of the form \(\abs{\widetilde{v} - \widetilde{v}(\xi)}\) with fixed \(\xi \in E^{\ell}\) and use the independence of \(\widetilde{v}(x, t) = v(x)\) with respect to \(t \in [0, 1]^{\ell - r}\).{}
	We get
	\begin{multline*}
	\int_{\widetilde B_{\rho}(x, 0')}\int_{\widetilde B_{\rho}(x, 0')}{\abs{\widetilde{v}(\eta) - \widetilde{v}(\zeta)} \dif\mu(\eta)}\dif\mu(\zeta)\\
	= \left(\cH^{\ell-r}\left(B^{\ell-r}_\rho\cap [0,1]^{\ell-r}\right)\right)^{2}\int_{B_{\rho}^{r}(x)}\int_{B_{\rho}^{r}(x)}\abs{v(y) - v(z)} \dif\cH^{r}(y)\dif\cH^{r}(z).
	\end{multline*}	
	When \(\rho \le 1\), one has 
	\[{}
	\cH^{\ell-r}\left(B^{\ell-r}_\rho\cap [0,1]^{\ell-r}\right) = \C \rho^{\ell - r}
	\]
	and \eqref{eqExtension-104} then follows for every \(0 < \rho \le 1\).
	Estimate \eqref{eqExtension-104} also holds for \(1 < \rho \le \Diam{K^{r}}\), with possibly a different constant.
	Taking the supremum with respect to \(\rho\), one gets the estimate involving the \(\VMO\)~norms.
\end{proof}

\begin{proof}[Proof of Proposition~\ref{propositionFugledeDetector}]
	Let \(w\) be an \(\ell\)-detector for \(u\).
	For every simplicial complex \(\cK^{r}\), we equip the set \(E^{\ell} = K^{r} \times [0, 1]^{\ell - r}\) with an \(\ell\)-dimensional simplicial structure. 
	For \(\gamma \in \Fuglede_{w}(K^{r}; \manfV)\), consider the map \(\widetilde{\gamma} \colon  E^{\ell} \to \manfV\) given by
	\begin{equation}
	\label{eqExtension-113}
	\widetilde{\gamma}(z, t) = \gamma(z).
	\end{equation}
	As \(\widetilde{\gamma} \in \Fuglede_{w}(E^{\ell}; \manfV)\), we have \(u \compose \widetilde\gamma \in \VMO(E^{\ell})\).{}
	Hence, by Lemma~\ref{lemmaVMOellExtension}, \(u \compose \gamma \in \VMO(K^{r})\),{}
	which implies that \(w\) is also an \(r\)-detector for \(u\).
\end{proof}

We now focus on the main result of this section concerning the existence of \(\ell\)-detectors for Sobolev functions:

\begin{proposition}
	\label{propositionSobolevVMO>1}
	If \(u \in \Sobolev^{1, p}(\manfV)\) with \(p \ge \ell\) and \(p \ne 1\), then \(u\) has an \(\ell\)-detector \(w \colon  \manfV \to [0, +\infty]\) such that
	\[{}
	\norm{w}_{\Lebesgue^{1}(\manfV)}
	\le C \norm{u}_{\Sobolev^{1, p}(\manfV)}^{p}
	\]
	and, for every simplicial complex \(\cK^{\ell}\) and every \(\gamma \in \Fuglede_{w}(K^{\ell}; \manfV)\), 
	\begin{equation}
	\label{eqDetectors-280}
	\norm{u \compose \gamma}_{\VMO(K^{\ell})}
	\le C' (1 + \abs{\gamma}_{\Lip}) \norm{w \compose \gamma}_{\Lebesgue^{1}(K^{\ell})}^{1/p} \text{,}
	\end{equation}
	where the constants \(C, C'> 0\) depend on \(p\), \(m\), \(\manfV\) and \(C' > 0\) also depend on \(K^{\ell}\).
\end{proposition}

To handle Proposition~\ref{propositionSobolevVMO>1}, we rely on the following classical estimate of an integral average on balls involving the value of the precise representative.
We recall that \(v(z)\) is the \emph{precise representative} of a function \(v\) at \(z\) whenever
\begin{equation*}
\lim_{\rho \to 0}{\fint_{B_{\rho}^m(z)}{\abs{v - v(z)}}}
= 0.
\end{equation*}
More precisely, we have

\begin{lemma}
	\label{lemmaEstimateMaximalFunction}
	Let \(\rho > 0\) and \(z \in \R^{m}\).
    If \(v \in \Sobolev^{1, 1}(B_{\rho}^{m}(z))\) and \(v(z)\) is the precise representative of \(v\) at \(z\), then
	\begin{equation}\label{eq-Poincare-exact}
	\fint_{B_{\rho}^{m}(z)}{\abs{v - v(z)}}
	\le C'' \rho \, \sup_{0<s<\rho}{\fint_{B^{m}_s(z)}\abs{Dv}}.
	\end{equation}
\end{lemma}

\resetconstant
\begin{proof}[Proof of Lemma~\ref{lemmaEstimateMaximalFunction}]
We may assume without loss of generality that \(z = 0\).
We first assume that \( v \) is smooth.
By \cite{Evans_Gariepy}*{Lemma~4.1},
	\[{}
	\fint_{B^{m}_\rho}\abs{v - v(0)}
	\le \C \int_{B^{m}_\rho} \frac{\abs{Dv}(y)}{\abs{y}^{m - 1}} \dif y.
	\]
When \(m = 1\), this inequality can be deduced directly from the Fundamental theorem of Calculus and already gives \eqref{eq-Poincare-exact} when \( v \) is smooth.
We therefore proceed assuming that \( m \ge 2 \).
 By Cavalieri's principle,
	\[{}
	\int_{B^{m}_\rho} \frac{\abs{Dv}(y)}{\abs{y}^{m - 1}} \dif y
	= \int_{0}^{\infty} \biggl( \int_{\{y \in B^{m}_\rho \colon  1/\abs{y}^{m - 1} > t\}}{\abs{Dv}}  \biggr) \dif t.
	\]
	Since \( m > 1 \), the change of variable \(t = 1/s^{m - 1}\) then gives
	\[{}
	\int_{B^{m}_\rho} \frac{\abs{Dv}(y)}{\abs{y}^{m - 1}} \dif y
	= (m - 1) \int_{0}^{\infty} \biggl( \int_{B^{m}_{\min{\{s, \rho\}}}} \abs{Dv} \biggr) \frac{\dif s}{s^{m}}
	\le \C \rho \, \sup_{0<s<\rho}\fint_{B^{m}_s}\abs{Dv}.
	\]
	This implies the desired inequality \eqref{eq-Poincare-exact} when \(v\) is smooth.
	
	We next justify that \eqref{eq-Poincare-exact} holds for \(v \in \Sobolev^{1, 1}(B^{m}_\rho)\) when \( m \ge 1 \) and \(v(0)\) is the precise representative of \(v\) at \(0\). 
	To this end, let \(\varphi\) be a \emph{mollifier}, that is,
\[
\varphi \in \Smooth_c^\infty(B^{m}_{1}),  
\quad
\varphi \ge 0
\text{ in \(B^{m}_{1}\)}
\quad \text{and} \quad
\int_{B^{m}_{1}} \varphi = 1.
\]
We then define for every \(\epsilon>0\) the function 
\[
\varphi_{\epsilon} \colon  x\in \R^m \longmapsto \frac{1}{\epsilon^m}\varphi\left(\frac{x}{\epsilon}\right).
\]	
Given \(0 < \rho_1<\rho\), we apply \eqref{eq-Poincare-exact} on \(B^{m}_{\rho_{1}}\) to the smooth function \(\varphi_{\epsilon}*v\), for every \(0<\epsilon<\rho-\rho_1\). 
This yields
\begin{equation}\label{eq-613}
\begin{split}
\fint_{B^{m}_{\rho_1}}\abs{(\varphi_{\epsilon} *v)-(\varphi_{\epsilon}*v)(0)}
&\leq C \rho_1 \sup_{0<s<\rho_1}\fint_{B^{m}_s}\abs{\varphi_{\epsilon}*Dv}\\
&\leq C \rho_1 \sup_{0<s<\rho_1}\fint_{B^{m}_s}\varphi_{\epsilon}*\abs{Dv}.
\end{split}
\end{equation}
We observe that, for \( s > 0 \) and \(s + \epsilon < \rho \),
\[ 
\fint_{B^{m}_s}\varphi_{\epsilon}*\abs{Dv}\leq \Cl{cte-630}\fint_{B^{m}_{s + \epsilon}}\abs{Dv}\leq \Cr{cte-630}\sup_{0< t <\rho}{\fint_{B^{m}_{t}}\abs{Dv}}.
\]
By \eqref{eq-613}, this implies that
\[
\fint_{B^{m}_{\rho_{1}}}\abs{(\varphi_{\epsilon} *v)-(\varphi_{\epsilon}*v)(0)}
\leq \Cl{cte-6222}	\rho_1\sup_{0<t<\rho}{\fint_{B^{m}_{t}}\abs{Dv}}.
\]
Since \(v(0)\) is the precise representative of \(v\) at \(0\), we have \(\varphi_{\epsilon}*v(0)\to v(0)\) when \(\epsilon\to 0\). 
This yields
\[
\fint_{B^{m}_{\rho_{1}}}\abs{v-v(0)}
\leq \Cr{cte-6222} \rho_1	\sup_{0<s<\rho}\fint_{B^{m}_s}\abs{Dv}.
\]
Letting \(\rho_1\to \rho\) gives the desired result.
\end{proof}

We now state the counterpart of Lemma~\ref{lemmaEstimateMaximalFunction} on a compact manifold in terms of the maximal function of \(\abs{Dv}\).
We recall that the \emph{maximal function} \(\Maximal f \colon  \manfV \to [0, +\infty]\) of a measurable function \(f \colon  \manfV \to [0, +\infty]\) is defined for every \(z \in \manfV\) as
\begin{equation}
\label{eqMaximal}
\Maximal f(z)
= \sup_{\rho > 0}{\fint_{B_{\rho}^{m}(z)}{f}}.
\end{equation}
We then have

\begin{lemma}
	\label{lemmaEstimateMaximalFunctionManifold}
	Let \( \manfM \) be a compact manifold of dimension \(m\) without boundary.
    If \(v \in \Sobolev^{1, 1}(\manfM)\) and \(v(z)\) is the precise representative of \(v\) at \(z \in \manfM\), then
	\[{}
	\fint_{B_{\rho}^{m}(z)}{\abs{v - v(z)}}
	\le C''' \rho \, \Maximal \abs{Dv}(z)
	\quad \text{for every \(\rho>0\).}
	\]
\end{lemma}

\begin{proof}[Proof of Lemma~\ref{lemmaEstimateMaximalFunctionManifold}]
By using the exponential map, estimate \eqref{eq-Poincare-exact} can be transferred to functions \(v \in \Sobolev^{1, 1}(\manfM)\) provided that \(\rho<\kappa\), where \(\kappa\) is the injectivity radius of \(\manfM\). 
More precisely,
\begin{equation}\label{eq-646}
\fint_{B^{m}_\rho(z)}\abs{v-v(z)}\leq \Cl{cte-647} \rho \sup_{0<s<\rho}\fint_{B^{m}_s(z)}\abs{Dv}\leq \Cr{cte-647}\rho \,\Maximal \abs{Dv}(z),
\end{equation}
for every \(z\in \manfM\) such that \(v(z)\) is the precise representative of \(v\) at \(z\).

We proceed to establish that \eqref{eq-646} remains true for \(\rho\geq \kappa\), with possibly a different constant.  
For every \(\rho>0\), we have
\begin{align*}
\fint_{B^{m}_\rho(z)}\abs{v-v(z)}
&\leq 
\fint_{B^{m}_\rho(z)}\biggl|v-\fint_{B^{m}_{\kappa/2}(z)}v\biggr| +  \biggl|\fint_{B^{m}_{\kappa/2}(z)}v-v(z)\biggr|\\
&\leq \fint_{B^{m}_\rho(z)}\fint_{B^{m}_{\kappa/2}(z)}\abs{v(x)-v(y)} \dif y \dif x
  + \fint_{B^{m}_{\kappa/2}(z)}\abs{v-v(z)}.
\end{align*}
By \eqref{eq-646}, the last term is not larger than \(\Cr{cte-647}\frac{\kappa}{2}\Maximal \abs{Dv}(z)\). 
Assuming that \(\rho\geq \kappa\), we get
\begin{equation}\label{eq667}
\fint_{B^{m}_\rho(z)}\abs{v-v(z)}\leq \C\left(\int_{\manfM}\int_{\manfM}\abs{v(y)-v(z)}\dif y \dif z + \Maximal \abs{Dv}(z)\right).
\end{equation}
Since \(\manfM\) is compact and connected, by the standard Poincaré inequality in \(\manfM\), see e.g.\@ \cite{Hebey}*{Lemma~3.8},
\[
\int_{\manfM}\left|v-\fint_{\manfM}v\right| \leq \Cl{cte-675}\int_{\manfM}\abs{Dv}.
\]
Since \(\manfM\) has finite measure, by \eqref{eqDetector-61} we then have 
\[
	\int_{\manfM}\int_{\manfM}\abs{v(x)-v(y)}\dif y\dif x  
	\le \C \int_{\manfM}\left|v-\fint_{\manfM}v\right|
    \le \C \int_{\manfM}\abs{Dv}
	\le \C \Maximal \abs{Dv}(z).
\]
In conjunction with \eqref{eq667}, this gives the desired result for every \(\rho\geq \kappa\). 
\end{proof}

\resetconstant
\begin{proof}[Proof of Proposition~\ref{propositionSobolevVMO>1}]
We first assume that \( \manfV \) is \( \R^{m} \) or a compact manifold without boundary \(\manfM\). 
Let \(\gamma \colon  K^{\ell} \to \manfV\) be a Lipschitz map.
We claim that if \(u \compose \gamma(r)\) is the precise representative of \(u\) at \(\gamma(r)\) for almost every \(r \in K^{\ell}\), then for every \(\rho>0\),
\begin{equation}
\label{eqGeneric-477}
\seminorm{u \compose \gamma}_{\rho}
\le \Cr{cteGeneric-340}\abs{\gamma}_{\Lip} \, \rho^{\frac{p - \ell}{p}} \sup_{x \in K^{\ell}}\bignorm{(\Maximal \abs{Du})^{p} \compose \gamma}_{\Lebesgue^{1}(B_{\rho}^{\ell}(x))}^{1/p}.
\end{equation}

Let \(\alpha = \abs{\gamma}_{\Lip}\) that we may assume, without loss of generality, to be positive since \eqref{eqGeneric-477} is trivially true when \(\gamma\) is constant.{}
For every \(r, t \in K^{\ell}\) and every \(y \in \manfV\), by the triangle inequality one has 
\begin{equation*}
\abs{u \compose \gamma (r) - u \compose \gamma (t)}
\le \abs{u(y) - u \compose \gamma (r) } +  \abs{u(y) - u \compose \gamma (t)}.
\end{equation*}	
Taking the average integral of both sides with respect to \(y\) over \(B^{m}_{\alpha \rho}(\gamma(x))\) with \(x \in K^{\ell}\), we obtain
\[
\abs{u \compose \gamma (r) - u \compose \gamma (t)}
\le \fint_{B_{\alpha\rho}^{m}(\gamma(x))} \abs{u(y) - u \compose \gamma (r)} \dif y
+ \fint_{B_{\alpha\rho}^{m}(\gamma(x))} \abs{u(y) - u \compose \gamma (t)} \dif y.
\]	
Taking now the averages with respect to \(r\) and \(t\) over the ball \(B_{\rho}^{\ell}(x)\), we then get
\begin{multline}
    \label{eq693}
\fint_{B_{\rho}^{\ell}(x)}\fint_{B_{\rho}^{\ell}(x)}
\abs{u \compose \gamma (r) - u \compose \gamma (t)} \dif\cH^{\ell}(t)\dif\cH^{\ell}(r)\\
\le 2 \fint_{B_{\rho}^{\ell}(x)} \biggl(\fint_{B_{\alpha\rho}^{m}(\gamma(x))} \abs{ u(y) - u \compose \gamma (r)} \dif y \biggr) \dif\cH^{\ell}(r).
\end{multline}
Observe that for \(r \in B_{\rho}^{\ell}(x)\), we have
\[{}
d(\gamma(r), \gamma(x))
\le \alpha d(r, x)
\le \alpha \rho
\]
and then
\[{}
B_{\alpha\rho}^{m}(\gamma(x)) 
\subset B_{2 \alpha \rho}^{m}(\gamma(r)). 
\]
Therefore, denoting the left-hand side of \eqref{eq693} by \(I_\rho\), we get 
\[
I_{\rho}
\le \C \fint_{B_{\rho}^{\ell}(x)} \biggl(\fint_{B_{2 \alpha \rho}^{m}(\gamma(r))} \abs{u(y) - u (\gamma (r))} \dif y \biggr) \dif\cH^{\ell}(r).
\]
We now assume that \(u (\gamma(r))\) is the precise representative of \(u\) at \(\gamma(r)\) for \(\cH^{\ell}\)-almost every \(r \in K^{\ell}\).{}
Applying Lemma~\ref{lemmaEstimateMaximalFunction} when \(\manfV = \R^{m}\) or Lemma~\ref{lemmaEstimateMaximalFunctionManifold} when \(\manfV = \manfM\),  we have in both cases that
\[
\fint_{B_{2 \alpha \rho}^{m}(\gamma(r))} \abs{u(y) - u(\gamma (r))} \dif y
\le \C\alpha\rho \Maximal \abs{Du}(\gamma(r))
\]
and then
\[{}
 I_{\rho}
\le \C \alpha  \rho \fint_{B_{\rho}^{\ell}(x)} \Maximal \abs{Du} \compose \gamma \dif\cH^{\ell}.
\]
Using Hölder's inequality, we get
\[{}
I_{\rho}
\le \Cl{cteGeneric-340} \alpha \rho^{\frac{p - \ell}{p}} \biggl( \int_{B_{\rho}^{\ell}(x)} (\Maximal \abs{Du})^{p} \compose \gamma \dif\cH^{\ell} \biggr)^{\frac{1}{p}}.
\]
Since \(\alpha = \abs{\gamma}_{\Lip}\), estimate \eqref{eqGeneric-477} follows.

We recall that from the Hardy-Littlewood maximal inequality we have \(\Maximal \abs{Du} \in \Lebesgue^{p}(\manfV)\) for \(1 < p < \infty\).{}
We then take \(w \colon  \manfV \to [0, +\infty]\) defined by
\[{}
w(z) = 
\begin{cases}
	\abs{u}^{p}(z) + (\Maximal \abs{Du})^{p}(z)	& \text{if \(u(z)\) is the precise representative of \(u\) at \(z\),}\\
	+\infty	& \text{otherwise.}
\end{cases}
\]
By the Lebesgue differentiation theorem, \(u(z)\) is the precise representative of \(u\) at almost every point \(z\) and then 
\begin{equation}
	\label{eqDetectors-480}
	w = \abs{u}^{p} + (\Maximal \abs{Du})^{p}
	\quad \text{almost everywhere in \(\manfV\).}
\end{equation}
In particular, \(w\) is summable in \(\manfV\) and
\[{}
\norm{w}_{\Lebesgue^{1}(\manfV)}
= \norm{u}_{\Lebesgue^{p}(\manfV)}^{p} + \norm{\Maximal \abs{Du}}_{\Lebesgue^{p}(\manfV)}^{p}
\le \C \norm{u}_{\Sobolev^{1, p}(\manfV)}^{p}.
\]
Observe that if \(\gamma \in \Fuglede_{w}(K^{\ell}; \manfV)\), then \(w \compose \gamma(r) < +\infty\) for \(\cH^{\ell}\)-almost every \(r \in K^{\ell}\).{}
Hence, estimate \eqref{eqGeneric-477} holds.
By Hölder's inequality and \eqref{eqDetectors-480}, we thus have
\[{}
\norm{u \compose \gamma}_{\Lebesgue^{1}(K^{\ell})}
\le \cH^{\ell}(K^{\ell})^{\frac{p - 1}{p}}\norm{u \compose \gamma}_{\Lebesgue^{p}(K^{\ell})}
\le \cH^{\ell}(K^{\ell})^{\frac{p - 1}{p}}\norm{w \compose \gamma}_{\Lebesgue^{1}(K^{\ell})}^{1/p}
\]
and
\begin{equation}
\label{eqGeneric-517}
\seminorm{u \compose \gamma}_{\rho}
\le \Cr{cteGeneric-340} \abs{\gamma}_{\Lip} \, \rho^{\frac{p - \ell}{p}} \sup_{x \in K^{\ell}}\norm{w \compose \gamma}_{\Lebesgue^{1}(B_{\rho}^{\ell}(x))}^{1/p}
\quad \text{for every \(\rho > 0\).}
\end{equation}
For \(p \ge \ell\), the right-hand side of \eqref{eqGeneric-517} converges to zero as \(\rho \to 0\) and then \(u \compose \gamma \in \VMO(K^{\ell})\). 
Hence, \(w\) is an \(\ell\)-detector for \(u\).{}
Moreover, taking the supremum over \(\rho \leq \Diam{K^{\ell}}\) in \eqref{eqGeneric-517} we get
\[{}
\seminorm{u \compose \gamma}_{\VMO(K^{\ell})}
\le \C \abs{\gamma}_{\Lip} \, \norm{w \compose \gamma}_{\Lebesgue^{1}(K^{\ell})}^{1/p},
\]
which gives the conclusion when \( \manfV \) is \( \R^{m} \) or a compact manifold \(\manfM\). 

We now assume that \( \manfV = \Omega \) is a Lipschitz open subset of \( \R^m \).
Given \(u \in \Sobolev^{1, p}(\Omega)\), we take an extension \(\overline{u} \in \Sobolev^{1, p}(\R^m)\) such that
\[
\norm{\overline{u}}_{\Sobolev^{1, p}(\R^m)}
\le \Cl{cteDetector-571} \norm{u}_{\Sobolev^{1, p}(\Omega)},
\]
for some constant \( \Cr{cteDetector-571} > 0 \) depending on \(m\), \(p\) and \(\Omega\).
We may apply the previous case to \(\overline{u}\) on \(\R^m\) and deduce the existence of an \(\ell\)-detector \(\overline{w} \colon \R^{m} \to [0, +\infty]\) for \(\overline{u}\).
Then, \(w \vcentcolon= \overline{w}|_{\Omega}\) satisfies
\[
\norm{w}_{\Lebesgue^{1}(\Omega)}
\le \norm{\overline{w}}_{\Lebesgue^{1}(\R^{m})}
\le \C \norm{u}_{\Sobolev^{1, p}(\R^{m})}^{p}
\le \C \norm{u}_{\Sobolev^{1, p}(\Omega)}^{p}.
\]
Moreover, for every simplicial complex \( \cK^{\ell} \) and every map \(\gamma \in \Fuglede_{w}(K^{\ell}; \Omega)\) we have \(\gamma \in \Fuglede_{\overline{w}}(K^{\ell}; \R^{m})\) and then \(u \compose \gamma = \overline{u} \compose \gamma \in \VMO(K^{\ell})\).
Hence, \(w\) is an \(\ell\)-detector for \(u\) and \eqref{eqDetectors-280} is satisfied.
It is worth mentioning that this argument also implies \eqref{eqGeneric-517}.
The conclusion thus follows for a Lipschitz open subset \( \Omega \) of \(\R^{m}\), which completes the proof of the proposition.
\end{proof}

\begin{remark}
	\label{remarkDetectorsCasep=1}
Although the proof of Proposition~\ref{propositionSobolevVMO>1} leaves aside the case \(p = 1\) and \(\ell \in \{0, 1\}\), the conclusion is still true for any \(u \in \Sobolev^{1, 1}(\manfV)\).
Indeed, when \(\ell=0\) one simply takes \(w = \abs{u}\) as \(0\)-detector for \(u\). 
That \eqref{eqDetectors-280} is satisfied merely follows from the equivalence between the norms in \(\Lebesgue^{1}(K^{0})\) and \(\VMO(K^{0})\).{}

When \(\ell = 1\), the summable function \( w \) given by the proof of Proposition~\ref{corollaryCompositionSobolevFugledeDim1} is a \(1\)-detector for \(u\).{}
Indeed, take \(\gamma \in \Fuglede_{w}(K^{1}; \manfV)\).{}
Since \(u \compose \gamma = \widehat{u \compose \gamma}\) almost everywhere in \(K^{1}\), for every \(\rho > 0\) we then have
\[{}
\seminorm{u \compose \gamma}_{\rho}
= \seminorm{\widehat{u \compose \gamma}}_{\rho}
\le \sup_{x \in K^{1}}{ \sup_{r, t \in B_{\rho}^{1}(x)}{{\abs{\widehat{u \compose \gamma}(r) - \widehat{u \compose \gamma}(t)}}}}.
\]
When \( \norm{u}_{\Sobolev^{1, 1}(\manfV)} > 0 \), by \eqref{eqFuglede-670} and the choice of \(w \ge \abs{u} + \abs{Du}\), see \eqref{eqFuglede-709}, we have
\[{}
\seminorm{u \compose \gamma}_{\VMO(K^{1})}
\le 2\abs{\gamma}_{\Lip} \norm{Du \compose \gamma}_{\Lebesgue^{1}(K^{1})}
\le 2\abs{\gamma}_{\Lip} \norm{w \compose \gamma}_{\Lebesgue^{1}(K^{1})},
\]
which then implies \eqref{eqDetectors-280} for \(p = 1\).
When \( \norm{u}_{\Sobolev^{1, 1}(\manfV)} = 0 \), it follows from \eqref{eqFuglede-1073} that \( \norm{u \compose \gamma}_{\VMO(K^{1})} = 0 \).
Hence, in both cases we have the desired conclusion.
\end{remark}

When \(p > \ell\), one deduces from estimate \eqref{eqGeneric-517} that, for every \(\gamma\in \Fuglede_{w}(K^\ell;\manfV)\), the function \(u\compose \gamma\) belongs to a Campanato space in \(K^\ell\).
As in the Euclidean setting \citelist{\cite{Giusti}*{Theorem~2.9} \cite{Giaquinta-Martinazzi}*{Theorem~5.5}}, it follows that 
there exists a Hölder-continuous function \(g \colon  K^{\ell} \to \R\) with exponent \(1 - \ell/p\) such that 
\[{}
u \compose \gamma = g
\quad \text{\(\cH^{\ell}\)-almost everywhere in \(K^{\ell}\).}
\]

A natural alternative approach to prove Proposition~\ref{propositionSobolevVMO>1} is to rely on a counterpart of the generic Sobolev property (Proposition~\ref{corollaryCompositionSobolevFuglede}) that replaces Riemannian manifolds \(\manfA\) by simplicial complexes \(\cK^{\ell}\), with a suitable adaptation of the notion of \(\Sobolev^{1, p}\) on a polytope.
The next step would be to prove a counterpart of Example~\ref{exampleDetectorsVMOSobolev} concerning the imbedding of \(\Sobolev^{1, \ell}\) into \(\VMO\) on \(\ell\)-dimensional domains.
However, as observed by White~\cite{White}, the latter step fails on polytopes \(K^{\ell}\) that are not sufficiently regular:

\begin{figure}
\centering
\hfill\includegraphics{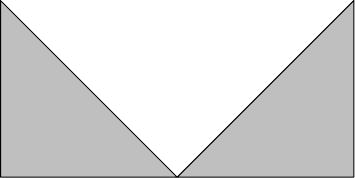}
\hfill{}
\caption{The polytope of Example~\ref{exampleUglyComplex}}
\label{figureUglyComplex}
\end{figure}

\begin{example}
\label{exampleUglyComplex}
	Let \(K^{2}\) be the polytope in \(\R^{2}\) formed by the union of two (solid) triangles: \(\Sigma_{1}^{2}\) with vertices \((0, 0)\), \((1, 0)\), \((1, 1)\), and \(\Sigma_{2}^{2}\) with vertices \((0, 0)\), \((-1, 0)\), \((-1, 1)\), see Figure~\ref{figureUglyComplex}.{}
	Then, \(u \colon  K^{2} \to \R\) defined for \(x = (x_{1}, x_{2})\) by
	\[{}
	u(x)
	= 
	\begin{cases}
		1	& \text{if \(x_{1} \ge 0\),}\\
		-1	& \text{if \(x_{1} < 0\),}
	\end{cases}	
	\]
	belongs to \(\Sobolev^{1, p}(\Int{K^{2}})\) for every \(p \ge 1\).{}
	However, for any \(\rho > 0\),{}
	\[{}
	\fint_{B_{\rho}^{2} \cap K^{2}}	\fint_{B_{\rho}^{2} \cap K^{2}} |u(y) - u(z)| \dif z \dif y 
	= 4 \biggl(\frac{\cH^{2}(B_{\rho}^{2} \cap \Sigma_{1}^{2})}{\cH^{2}(B_{\rho}^{2} \cap K^{2})}\biggr)	^{2}
	= 1.
	\]
	Hence, \(u \not\in \VMO(K^{2})\).
\end{example}

\section{Generic stability in VMO}
\label{sectionDetectorsStability}

We now study the generic \(\VMO\) convergence on \(\ell\)-dimensional polytopes.
We begin in the context of Sobolev functions, which motivates the introduction of the larger class of \(\VMO^{\ell}\) functions in the next section.

\begin{proposition}
	\label{propositionSobolevApproximationFuglede}
	If \(u \in \Sobolev^{1, p}(\manfV)\) with \(p \ge \ell\), then \(u\) has an \(\ell\)-detector \(\widetilde{w}\) such that, for every simplicial complex \(\cK^{\ell}\) and every map \(\gamma\) and sequence \((\gamma_{j})_{j \in \N}\) in \(\Fuglede_{\tilde{w}}(K^{\ell}; \manfV)\) with
	\begin{equation}
    \label{eqDetector-665}
    \gamma_{j} \to \gamma \quad \text{in \(\Fuglede_{\tilde{w}}(K^{\ell}; \manfV)\),}
	\end{equation}
	we have
	\begin{equation}
    \label{eqDetector-670}
	u \compose \gamma_{j} \to u \compose \gamma 
	\quad \text{in \(\VMO(K^{\ell})\).}
	\end{equation}
\end{proposition}

To handle the case \(\ell \ge 2\), we rely on the concept of equiVMO sequence that yields a characterization of convergence in \(\VMO\):

\begin{definition}
	A sequence \((v_{j})_{j \in \N}\) in \(\VMO(X)\) is \emph{equiVMO} whenever, for each \(\epsilon > 0\), there exists \(\delta > 0\) such that
	\begin{equation}
		\label{eqEquiVMO}
	\seminorm{v_{j}}_{\rho}
	\le \epsilon{}
	\quad \text{for every \(0 < \rho \le \delta\) and \(j \in \N\).}
	\end{equation}	
\end{definition}

	Then, see \cite{BrezisNirenberg1995}*{Lemma~A.16},

\begin{lemma}
	\label{lemmaEquiVMO}
	Let \((v_{j})_{j \in \N}\) be a sequence in \(\VMO(X)\).{}
	Given \(v \in \VMO(X)\), we have
	\[{}
	v_{j} \to v 
	\quad \text{in \(\VMO(X)\)}
	\]
	if and only if
	\[{}
	v_{j} \to v
	\quad \text{in \(\Lebesgue^{1}(X)\)} \quad \text{and} \quad 
	(v_{j})_{j \in \N}
	\ \text{is equiVMO.}
	\]
\end{lemma}

\begin{proof}[Proof of Lemma~\ref{lemmaEquiVMO}]
	``\(\Longrightarrow\)''.{}
	It suffices to prove that \((v_{j})_{j \in \N}\) is equiVMO.
	For every \(\rho > 0\),{}
	\[{}
	\seminorm{v_{j}}_{\rho}
	\le \seminorm{v_{j} - v}_{\VMO(X)} + \seminorm{v}_{\rho}\,.
	\]
	Thus, by convergence of \((v_{j})_{j \in \N}\) in \(\VMO(X)\), for every \(\delta > 0\) we have
	\[{}
	\limsup_{j \to \infty}{\Bigl(\sup_{\rho \le \delta}{\seminorm{v_{j}}_{\rho}} \Bigr)}
	\le \sup_{\rho \le \delta}{\seminorm{v}_{\rho}}\,.
	\]
	Since \(v \in \VMO(X)\), given \(\epsilon > 0\) there exists \(\delta_{1} > 0\) such that
	\[{}
	\sup_{\rho \le \delta_{1}}{\seminorm{v}_{\rho}}
	\le \frac{\epsilon}{2}.
	\]
	We then take \(J \in \N\) such that
	\[{}
	\sup_{\rho \le \delta_{1}}{\seminorm{v_{j}}_{\rho}}
	\le \sup_{\rho \le \delta_{1}}{\seminorm{v}_{\rho}} + \frac{\epsilon}{2}
	\le \epsilon
	\quad \text{for every \(j \ge J\).}
	\]
	Since each function \(v_{j}\) belongs to \(\VMO(X)\), there exists \(\delta_{2} > 0\) such that 
	\[{}
	\sup_{\rho \le \delta_{2}}{\seminorm{v_{j}}_{\rho}}
	\le \epsilon 
	\quad \text{for every \(j \le J-1\).}
	\]
	Inequality \eqref{eqEquiVMO} thus holds with \(\delta = \min{\{\delta_{1}, \delta_{2}\}}\).{}
	
	``\(\Longleftarrow\)''.{}
We claim that, for every \(\delta > 0\),
\begin{equation}
	\label{eqBMOLipschitz}
	\seminorm{v_{j} - v}_{\VMO(X)}
	\le C_{\delta} \norm{v_{j} - v}_{\Lebesgue^{1}(X)} + \sup_{\rho \le \delta}{\seminorm{v}_{\rho}} + \sup_{\rho \le \delta}{\seminorm{v_{j}}_{\rho}}\,,
\end{equation}	
for some constant \(C_{\delta} > 0\) that depends on \(\delta > 0\).
To this end, given \(x \in X\), \(j \in \N\) and \(\rho > 0\), let 
	\[{}
I_{j, \rho}(x)
={}
\fint_{B_{\rho}(x)}\fint_{B_{\rho}(x)}
\bigabs{v_{j}(y) - v(y) - v_{j}(z) + v(z)} \dif\mu(z) \dif\mu(y).
\]
Then, 
\[
\seminorm{v_{j} - v}_{\VMO(X)} = \sup_{x\in X, \rho>0}I_{j,\rho}(x).
\]

Take \(\delta > 0\).
For \(\rho \ge \delta\), by the triangle inequality and monotonicity of \(\mu\) we have
\[{}
I_{j, \rho}(x)
\le 2 \fint_{B_{\rho}(x)}{\abs{v_{j} - v} \dif\mu}
\le \frac{2}{\mu(B_{\delta}(x))} \norm{v_{j} - v}_{L^{1}(X)}.
\]
By uniform nondegeneracy of \(\mu\), see~\eqref{eq-uniform-nondegeneracy}, we then get
\[{}
I_{j, \rho}(x)
\le C_{\delta} \norm{v_{j} - v}_{L^{1}(X)}.
\]
For \(\rho \le \delta\) we estimate \(I_{j, \rho}(x)\) as
\[{}
I_{j, \rho}(x)
\le \seminorm{v_{j}}_{\rho} + \seminorm{v}_{\rho}\,.
\]
Inequality \eqref{eqBMOLipschitz} thus holds.

We now rely on \eqref{eqBMOLipschitz} to show that \(v_j\to v\) in \(\VMO(X)\). 
Let \(\epsilon > 0\). Using the equiVMO property of \((v_{j})_{j \in \N}\)\,, we take \(\delta_{1} > 0\) such that
	\[{}
	\sup_{\rho \le \delta_{1}}{\seminorm{v_{j}}_{\rho}}
	\le \epsilon
	\quad \text{for every \(j \in \N\).}
	\]
	Thus, for every \(0 < \delta \le \delta_{1}\) and \(j \in \N\),{} by \eqref{eqBMOLipschitz} we get
	\[	
	\seminorm{v_{j} - v}_{\VMO(X)}
	\le C_{\delta} \norm{v_{j} - v}_{\Lebesgue^{1}(X)} + \sup_{\rho \le \delta}{\seminorm{v}_{\rho}} + \epsilon.
	\]
	As \(j \to \infty\) in this estimate,
\[{}
\limsup_{j \to \infty}{\seminorm{v_{j} - v}_{\VMO(X)}}
\le \sup_{\rho \le \delta}{\seminorm{v}_{\rho}}  + \epsilon.
\]
The conclusion thus follows as \(\delta \to 0\) since \(v \in \VMO(X)\) and  \(\epsilon\) is arbitrary.
\end{proof}

\resetconstant
\begin{proof}[Proof of Proposition~\ref{propositionSobolevApproximationFuglede}]	
    The case \( \ell = 0 \) concerns the generic pointwise convergence for \( \Lebesgue^{1} \) functions and can be deduced from Example~\ref{exampleFugledePointwiseConvergence}.
	Concerning the next case \(\ell = 1\), we take the summable function \( \widetilde{w} \) given by Proposition~\ref{propositionFugledeApproximationMapsDim1}.
    From Remark~\ref{remarkDetectorsCasep=1}, \(\widetilde{w}\) is a \( 1 \)-detector.
    By Proposition~\ref{propositionFugledeApproximationMapsDim1}, for any Fugelde maps satisfying \eqref{eqDetector-665} one has
    \[
	\widehat{u \compose \gamma_{j}} \to \widehat{u \compose \gamma} \quad \text{in \(\Smooth^{0}(K^{1})\).}
	\]
    In view of the continuous imbedding \( \Smooth^{0}(K^{1}) \subset \VMO(K^{1}) \), the convergence \eqref{eqDetector-670} is satisfied.
 
	We now assume that \(p \ge \ell \ge 2\).
	Given the \(\ell\)-detector \(w\) provided by the proof of   Proposition~\ref{propositionSobolevVMO>1},
	let \(\widetilde{w}\) be a summable function that satisfies the conclusion of Proposition~\ref{propositionFugledeApproximationMaps} applied to \(u \in \Lebesgue^{p}(\manfV)\) and to \(w \in \Lebesgue^{1}(\manfV)\).
	Let us assume that \(\gamma\) and \((\gamma_{j})_{j \in \N}\) satisfy \eqref{eqDetector-665}.
	From Proposition~\ref{propositionFugledeApproximationMaps},
	\[{}
	u \compose \gamma_{j} \to u \compose \gamma{}
	\quad \text{in \(\Lebesgue^{p}(K^{\ell})\)}.
	\]
    As \( \widetilde{w} \ge w \), we have that \(\widetilde{w}\) is also an \(\ell\)-detector for \(u\) and then \(u \compose \gamma_{j} \in \VMO(K^{\ell})\) for every \(j \in \N\).{}
    Since the sequence \((\gamma_{j})_{j \in \N}\) is equiLipschitz, we have by \eqref{eqGeneric-517} that
	\begin{equation}
		\label{eqSobolevVMOInteger}
	\seminorm{u \compose \gamma_{j}}_{\rho}
	\le \C \rho^{\frac{p - \ell}{p}} \sup_{x \in K^{\ell}}{\norm{w \compose \gamma_{j}}_{\Lebesgue^{1}(B_{\rho}^{\ell}(x))}^{1/p}  }
	\quad \text{for every \(\rho > 0\) and \(j \in \N\).}
	\end{equation}
	By Proposition~\ref{propositionFugledeApproximationMaps} we also have the convergence of \((w \compose \gamma_{j})_{j \in \N}\) in \(\Lebesgue^{1}(X)\) and, in particular, this sequence is equi-integrable.
	Thus, for every \(\epsilon > 0\), there exists \(\delta > 0\) such that
	\[{}
	\sup_{x \in K^{\ell}}\norm{w \compose \gamma_{j}}_{\Lebesgue^{1}(B_{\rho}^{\ell}(x))}
	\le \epsilon^{p}
	\quad \text{for every \(0 < \rho \le \delta\) and \(j \in \N\).}
	\]
    Hence, \((u \compose \gamma_{j})_{j \in \N}\) is equiVMO and it then follows from Lemma~\ref{lemmaEquiVMO} that \eqref{eqDetector-670} holds.
\end{proof}

\section{\texorpdfstring{$\VMO^\ell$}{VMOl} functions}

Motivated by Proposition~\ref{propositionSobolevApproximationFuglede}, we introduce the class of \(\VMO^{\ell}\) functions that are stable under \(\VMO\) convergence on generic polytopes of dimension \(\ell\).{}
This notion will serve as a common roof for introducing later the notion of \(\ell\)-extendability for maps between manifolds.

\begin{definition}
\label{definitionVMOell}
Given \(\ell \in \N\), we say that a measurable function \(u \colon  \manfV \to \R\) is \(\VMO^{\ell}\) whenever there exists an \(\ell\)-detector \(w \colon  \manfV \to [0, +\infty]\) such that, for every simplicial complex \(\cK^{\ell}\) and every map \(\gamma\) and sequence \((\gamma_{j})_{j \in \N}\) in \(\Fuglede_{w}(K^{\ell}; \manfV)\) with
	\[{}
	\gamma_{j} \to \gamma{}
	\quad \text{in \(\Fuglede_{w}(K^{\ell}; \manfV)\),}
	\]
	we have
	\[{}
	u \compose \gamma_{j} \to u \compose \gamma{}
	\quad \text{in \(\VMO(K^{\ell})\).}
	\]
\end{definition}

We denote by \(\VMO^{\ell}(\manfV)\) the class of \(\VMO^{\ell}\) functions defined in \(\manfV\).
Note that, for every \(\ell \in \N\),{}
\[{}
\Smooth^{0}(\manfV) \subset \VMO^{\ell}(\manfV).
\]
For \(\ell = 0\), it follows from Proposition~\ref{propositionFugledeApproximationMaps}, see Example~\ref{exampleFugledePointwiseConvergence},
\[{}
\Lebesgue^{1}(\manfV) \subset \VMO^{0}(\manfV).
\]
By Proposition~\ref{propositionSobolevApproximationFuglede}, we also have the inclusions
\begin{equation}
\label{eqDetector-775}
\Sobolev^{1, p}(\manfV) \subset \VMO^{\ell} (\manfV)
\quad \text{for every \(p \ge \ell\).}
\end{equation}
In particular, when \(\manfV = \Omega\) is a Lipschitz open set and \(k \in \N_*\), by the Gagliardo-Nirenberg interpolation inequality we have \((\Sobolev^{k, p} \cap \Lebesgue^\infty)(\Omega) \subset \Sobolev^{1, kp}(\Omega)\), and we deduce from \eqref{eqDetector-775} that
\begin{equation}
\label{eqDetector-882}
(\Sobolev^{k, p} \cap \Lebesgue^\infty)(\Omega) \subset \VMO^{\ell} (\Omega)
\quad \text{for every \(kp \ge \ell\).}
\end{equation}

We now present another example, which is not necessarily included in the previous ones.
We first introduce the notion of a structured singular set.

\begin{definition}
    \label{defnStructuredSingularSet}
    Given \( i \in \{0, \ldots, m - 1\} \), we say that \(T^i \subset \manfV\) is a \emph{structured singular set of rank \(i\)}, whenever
    \begin{description}
        \item[if \(\manfV = \Omega\) is an open subset of \(\R^m\)] \(T^i\) is contained in finite unions of \(i\)-dimensional affine spaces of \(\R^{m}\) that are orthogonal to \(m - i\) coordinate axes;
        \item[if \(\manfV = \manfM\) is a compact manifold] \(T^{i}\) is contained in a finite union of \(i\)-dimensional compact submanifolds \((S_j)_{j\in J}\) of \(\manfM\) without boundary that intersect cleanly, that is, for every \(S_{j_1}, \dots, S_{j_k}\), their intersection is either empty or a submanifold whose tangent space at any point \(x\) is given by \(\Tangent{x} S_{j_1} \cap \ldots \cap \Tangent{x} S_{j_k}\).
    \end{description}
\end{definition}

\begin{proposition}
\label{proposition_Extension_Property_R_0-New}
Assume that \( \manfV \) has finite measure.
Let \(\ell \in \{1, \ldots, m-1\}\) and let \(T^{\ell^{*}}\) be a structured singular set of rank \(\ell^{*} \vcentcolon= m - \ell - 1\).
Then, every measurable function \(u \colon \manfV \to \R\) that is continuous in \(\manfV\setminus T^{\ell^{*}}\) belongs to \(u \in \VMO^{\ell}(\manfV)\).
\end{proposition}

\begin{proof}
Given \(\ell \le \alpha < \ell + 1\), let \(w \colon  \manfV \to [0, +\infty]\) be defined by
\[{}
w (x)=
\begin{cases}
1/d(x, T^{\ell^{*}})^{\alpha} &\text{if \(x \in \manfV \setminus T^{\ell^{*}}\)},\\
  +\infty & \text{if \(x \in T^{\ell^{*}}\)}.
\end{cases}
\]
If \( \manfV = \manfM\) is a compact manifold, then \(T^{\ell^{*}}\) is contained in a finite union of \(\ell^{*}\)-dimensional compact submanifolds \((S_j)_{j\in J}\) of \(\manfM\) that intersect cleanly.
If \( \manfV = \Omega\) is an open subset of \(\R^m\), then \(T^{\ell^{*}}\) is contained in a finite union of sets \((S_j)_{j\in J}\) that are \(\ell^{*}\)-dimensional affine spaces of \(\R^{m}\) orthogonal to \(m - \ell^{*} = \ell + 1\) coordinate axes.
In both cases, one has
\[
\frac{1}{d(x,T^{\ell^{*}})^\alpha}
\leq \sum_{j\in J}\frac{1}{d(x,S_j)^\alpha}\qquad \text{for every \(x\in \manfV\setminus T^{\ell^{*}}\).}
\]
Since \(\manfV\) has finite measure and \(\alpha < \ell + 1\), this implies that  \(w\) is summable in \(\manfV\). 
We claim that every \(\gamma \in \Fuglede_{w}(K^{\ell}; \manfV)\) satisfies
\begin{equation}
	\label{eqDetectors-886}
\gamma (K^\ell) \cap T^{\ell^{*}} = \emptyset.
\end{equation}
To this end, assume by contradiction that \(\gamma (K^\ell) \cap T^{\ell^{*}} \neq \emptyset\) and take \(x \in K^\ell\) such that \(\gamma (x) \in T^{\ell^{*}}\).{}
By our choice of \(w\), for every \(y \in K^\ell\) we have
\[
w (\gamma (y)) 
\ge \frac{1}{d (\gamma (x), \gamma (y))^{\alpha}}
\ge \frac{1}{\abs{\gamma}_{\Lip}^{\alpha} d (x, y)^{\alpha}}
\]
and thus
\begin{equation}
\label{eqEstimateT}
\int_{K^\ell} w \compose \gamma \dif\cH^{\ell}
\ge \frac{1}{\abs{\gamma}_{\Lip}^{\alpha}}
\int_{K^\ell} \frac{1}{d (x, y)^{\alpha}} \dif\cH^{\ell}(y).
\end{equation}
Since \(\alpha \ge \ell\), the right-hand side of \eqref{eqEstimateT} is infinite and, as a consequence, \(w \compose \gamma\) is not summable in \(K^{\ell}\).{} 
This is a contradiction and \eqref{eqDetectors-886} follows.
	By continuity of \(u\) in \(\manfV\setminus T^{\ell^{*}}\), it then follows that \(u \compose \gamma\) is continuous and, in particular, \(\VMO\).{}
	Hence, \(w\) is an \(\ell\)-detector for \(u\). 
	
	For the stability of \(u\), observe that if \(\cK^{\ell}\) is a simplicial complex and \((\gamma_{j})_{j \in \N}\) is a sequence in \(\Fuglede_{w}(K^{\ell}; \manfV)\) such that
	\[{}
	\gamma_{j} \to \gamma \quad \text{in \(\Fuglede_{\tilde{w}}(K^{\ell}; \manfV)\),}
	\]
	then by uniform convergence of \((\gamma_{j})_{j \in \N}\) and compactness of \(K^{\ell}\), it follows from \eqref{eqDetectors-886} applied to \(\gamma\) and each \(\gamma_{j}\) that there exists \(\epsilon > 0\) such that \(d(\gamma_{j}(K^{\ell}),  T^{\ell^{*}}) \ge \epsilon\) for every \(j \in \N\).{}
	Hence, by continuity of \(u\) on \(\manfV\setminus T^{\ell^{*}}\), we have
	\[{}
	u \compose \gamma_{j} \to u \compose \gamma \quad \text{in \(\Smooth^{0}(K^{\ell})\),}
	\]
	and in particular in \(\VMO(K^{\ell}; \R)\).
\end{proof}

Another family of \(\VMO^{\ell}\)~functions arises in the setting of fractional Sobolev spaces \(\Sobolev^{s, p}\)\,: 

\begin{proposition}
	\label{propositionVMOSobolevFractional}
	If \(0 < s < 1\) and \(sp \ge \ell\), then \(\Sobolev^{s, p}(\manfV) \subset \VMO^{\ell}(\manfV)\).	
    More precisely, for every \(u \in \Sobolev^{s, p}(\manfV)\), the summable function \(w  \colon  \manfV \to [0, +\infty]\) defined for every \(x \in \manfV\) by
\begin{equation*}
w(x) = \abs{u(x)}^p + \int_{\manfV} \frac{\abs{u (x) - u (y)}^p}{d(x, y)^{sp + m}} \dif y
\end{equation*}
is an \(\ell\)-detector for \(u\) such that, for every simplicial complex \(\cK^{\ell}\) and every \(\gamma \in \Fuglede_{w}(K^{\ell}; \manfV)\), 
	\begin{equation*}
	\norm{u \compose \gamma}_{\VMO(K^{\ell})}
	\le C (1 + \abs{\gamma}_{\Lip}^s) \norm{\widetilde{w} \compose \gamma}_{\Lebesgue^{1}(K^{\ell})}^{1/p} \text{,}
	\end{equation*}
	for a constant \(C > 0\) depending on \(p\), \(m\), \(\manfV\) and \(K^{\ell}\).
    Moreover, there exists a summable function \( \widetilde{w} \ge w \) in \(\manfV\) such that, whenever \(\gamma\) and \((\gamma_{j})_{j \in \N}\) belong to \(\Fuglede_{w}(K^{\ell}; \manfV)\) and satisfy \(\gamma_{j} \to \gamma\) in \(\Fuglede_{w}(K^{\ell}; \manfV)\), we have
	\[{}
	u \compose \gamma_{j} \to u \compose \gamma{}
	\quad \text{in \(\VMO(K^{\ell})\).}
	\]
\end{proposition}

\resetconstant
\begin{proof}
	Let \(\gamma \in \Fuglede_{w}(K^{\ell}; \manfV)\). 
    Since \(\abs{u}^p \le w\) and \( K^{\ell} \) has finite measure, by Hölder's inequality we have
    \[
    \norm{u \compose \gamma}_{\Lebesgue^{1}(K^{\ell})}
    \le \Cl{cteDetector-980} \norm{u \compose \gamma}_{\Lebesgue^{p}(K^{\ell})}
    \le \Cr{cteDetector-980} \norm{w \compose \gamma}_{\Lebesgue^{1}(K^{\ell})}^{1/p}.
    \]
    We next recall the following straightforward form of the Poincaré-Wirtinger inequality: 	
	\begin{multline*}
	\fint_{B_{\rho}^{\ell}(x)}\fint_{B_{\rho}^{\ell}(x)}{\abs{u \compose \gamma(r) - u \compose \gamma(t)}^{p} \dif\cH^{\ell}(t)\dif\cH^{\ell}(r)}\\
	\le \C \rho^{sp - \ell} \int_{B_{\rho}^{\ell}(x)}\int_{B_{\rho}^{\ell}(x)}{ \frac{\abs{u \compose \gamma (r) - u \compose \gamma (t)}^p}{d(s, t)^{s p + \ell}} \dif\cH^{\ell}(t) \dif\cH^{\ell}(r)},
	\end{multline*}
	for every \(\rho>0\) and every \(x \in K^{\ell}\).{}
	We apply the counterpart of \eqref{eqSobolevFractionalEstimateSubset} for measurable sets \(E \subset K^{\ell}\); this is possible because \eqref{eqFuglede-411} holds on \(K^{\ell}\) with \(\mu = \cH^{\ell}\lfloor_{K^{\ell}}\).
	It thus follows from Hölder's inequality and \eqref{eqSobolevFractionalEstimateSubset} with \(E = B_{\rho}^{\ell}(x)\) that
	\begin{multline*}
	\fint_{B_{\rho}^{\ell}(x)}\fint_{B_{\rho}^{\ell}(x)}{\abs{u \compose \gamma(r) - u \compose \gamma(t)} \dif\cH^{\ell}(t)\dif\cH^{\ell}(r)}\\
	\le \Cl{cteDetector-995} \abs{\gamma}_{\Lip}^{s} \, \rho^{\frac{sp - \ell}{p}} \norm{w \compose \gamma}_{\Lebesgue^{1}(B_{\rho}^{\ell}(x))}^{1/p}.
	\end{multline*}
	We deduce that
	\begin{equation}
		\label{eqGeneric-799}
	\seminorm{u \compose \gamma}_{\rho}
	\le \Cr{cteDetector-995} \abs{\gamma}_{\Lip}^{s} \, \rho^{\frac{sp - \ell}{p}} \sup_{x \in K^{\ell}}{\norm{w \compose \gamma}_{\Lebesgue^{1}(B_{\rho}^{\ell}(x))}^{1/p}}
	\quad \text{for every \( \rho >0\)}.
	\end{equation}
	When \(sp \ge \ell\) and \(\gamma \in \Fuglede_{w}(K^{\ell}; \manfV)\), the right-hand side converges to zero as \(\rho \to 0\).{}
	Hence, \(u \compose \gamma \in \VMO(K^{\ell})\).
    Moreover, taking the supremum of the left-hand side with respect to \(0 < \rho \le \Diam{K^{\ell}}\), we get
    \[
   	\seminorm{u \compose \gamma}_{\VMO(K^{\ell})}
	\le \C \abs{\gamma}_{\Lip}^s \norm{w \compose \gamma}_{\Lebesgue^{1}(K^{\ell})}^{1/p}.
    \]
	
    Let \(w_1\) and \(w_2\) be two summable functions obtained from Proposition~\ref{propositionFugledeApproximationMaps} applied to \( u \in \Lebesgue^{p}(\manfV) \) and to \( w \in \Lebesgue^{1}(\manfV) \). 
    We then take \( \widetilde{w} = w_1 + w_2 \).
    Given any equiLipschitz sequence  \((\gamma_{j})_{j \in \N}\) in \(\Fuglede_{\tilde{w}}(K^{\ell}; \manfV)\), by estimate \eqref{eqGeneric-799} we have
	\begin{equation}
	\label{eqDetectors-1004}
	\seminorm{u \compose \gamma_{j}}_{\rho}
	\le \C \, \rho^{\frac{sp - \ell}{p}} \sup_{x \in K^{\ell}}{\norm{w \compose \gamma_{j}}_{\Lebesgue^{1}(B_{\rho}^{\ell}(x))}^{1/p}}
	\quad \text{for every \(\rho > 0\) and \(j \in \N\).}
	\end{equation}
	One then proceeds as in the proof of Proposition~\ref{propositionSobolevApproximationFuglede}, replacing \eqref{eqSobolevVMOInteger} by \eqref{eqDetectors-1004}.	
\end{proof}

We now prove an inclusion among \(\VMO^{\ell}\) spaces with respect to different dimensions that follows from the use of a dummy variable in the spirit of Lemma~\ref{lemmaVMOellExtension}:

\begin{proposition}
\label{propositionVMOellInclusion}
For every  \(r \in \{0, \dots, \ell - 1\}\),{}
\[{}
\VMO^{\ell}(\manfV) \subset \VMO^{r}(\manfV).
\]
\end{proposition}

\begin{proof}
	Given \(u \in \VMO^{\ell}(\manfV)\), let \(w\) be a summable function given by Definition~\ref{definitionVMOell}.
	By Proposition~\ref{propositionFugledeDetector}, \(w\) is an \(r\)-detector for \(u\).
	Next, given a map \(\gamma\) and a sequence \((\gamma_{j})_{j \in \N}\) in  \(\Fuglede_{w}(K^{r}; \manfV)\) such that 
	\[{}
	\gamma_{j} \to \gamma{}
	\quad \text{in \(\Fuglede_{w}(K^{r}; \manfV)\),}
	\]
	then the corresponding sequence \((\widetilde\gamma_{j})_{j \in \N}\) defined as in \eqref{eqExtension-113} is contained in \(\Fuglede_{w}(E^{\ell}; \manfV)\) with \(E^{\ell} \vcentcolon= K^{r} \times [0, 1]^{\ell - r}\) and satisfies
	\[{}
	\widetilde\gamma_{j} \to \widetilde\gamma{}
	\quad \text{in \(\Fuglede_{w}(E^{\ell}; \manfV)\).}
	\]
	By the choice of \(w\), we thus have
	\[{}
	u \compose \widetilde\gamma_{j} \to u \compose \widetilde\gamma{}
	\quad \text{in \(\VMO(E^{\ell})\).}
	\]
	It then follows from Lemma~\ref{lemmaVMOellExtension} that
	\[{}
	u \compose \gamma_{j} \to u \compose \gamma{}
	\quad \text{in \(\VMO(K^{r})\).}
	\qedhere
	\]
\end{proof}

There is a natural notion of convergence in the setting of \(\VMO^\ell\) spaces:

\begin{definition}
\label{def_conv_VMOl}
Let \((u_j)_{j \in \N}\) be a sequence in \(\VMO^\ell(\manfV)\).
We say that \((u_j)_{j \in \N}\) converges to \(u\) in \(\VMO^\ell(\manfV)\), which we denote by
\[
u_j \to u \quad \text{in \(\VMO^\ell(\manfV)\),}
\]
whenever \(u \in \VMO^\ell(\manfV)\) and there exists an \(\ell\)-detector \(w\) for \(u\) and all \(u_j\) such that, for every simplicial complex \(\cK^\ell\) and \(\gamma \in \Fuglede_w(K^\ell; \manfV)\),
\[
u_j \compose \gamma \to u \compose \gamma
\quad \text{in \(\VMO(K^\ell)\).}
\]
\end{definition}

A connection between \(\VMO^\ell\) convergence and \(\Sobolev^{s, p}\) convergence for \(sp \ge \ell\) is given by the following:

\begin{proposition}
	\label{lemmaFugledeSobolevDetector}
	Let \(0 < s \le 1\) and \(1 \le p < \infty\).
	If \((u_{j})_{j \in \N}\) is a sequence in \(\Sobolev^{s, p}(\manfV)\) with \(sp \ge \ell\) such that
	\[{}
	u_{j} \to u
	\quad \text{in \(\Sobolev^{s, p}(\manfV)\),}
	\]
	then there exists a subsequence \((u_{j_i})_{i \in \N}\) such that
	\[{}
	u_{j_i} \to u
	\quad \text{in \(\VMO^\ell(\manfV)\).}
	\]
\end{proposition}
\resetconstant
\begin{proof}
	Let \(\widetilde{w}_{1}\) be an \(\ell\)-detector for \(u\) given by Proposition~\ref{propositionSobolevVMO>1} for \(s = 1\) and \(p \ne 1\), by Remark~\ref{remarkDetectorsCasep=1} for \(s = p = 1\) and by Proposition~\ref{propositionVMOSobolevFractional} for \(0 < s < 1\).{}
	Applying the aforementioned results, for each \(j \in \N\) we also get an \(\ell\)-detector \(w_{j}\) for \(u_{j} - u\) such that
	\begin{equation}
		\label{eq-Tools-Convolution-401}
	\norm{w_{j}}_{\Lebesgue^{1}(\manfV)}
	\le \C \norm{u_{j} - u}_{\Sobolev^{s, p}(\manfV)}^{p}
	\end{equation}
	and, for every simplicial complex \(\cK^{\ell}\) and every \(\gamma \in \Fuglede_{w_{j}}(K^{\ell}; \manfV)\),{}
	\begin{equation}
		\label{eq-Tools-Convolution-407}
	\norm{u_{j} \compose \gamma - u \compose \gamma}_{\VMO(K^{\ell})}
	\le \Cl{eq-Toolos-Convolution-409}(1 + \abs{\gamma}_{\Lip}^s) \norm{w_{j} \compose \gamma}_{\Lebesgue^{1}(K^{\ell})}^{1/p}.
	\end{equation}
	By \eqref{eq-Tools-Convolution-401} and \(\Sobolev^{s, p}\) convergence of \((u_{j})_{j \in \N}\), we have 
	\[{}
	w_{j} \to 0
	\quad \text{in \(\Lebesgue^{1}(\manfV)\).}
	\]
	We now take a summable function \(\widetilde w_{2} \colon  \manfV \to [0, +\infty]\) and a subsequence \((w_{j_{ i}})_{i \in \N}\) given by Proposition~\ref{lemmaModulusLebesguesequence} applied to \((w_{j})_{j \in \N}\).{}
	Observe that \(\widetilde w_{2}\) is an \(\ell\)-detector for each \(u_{j_{i}} - u\).{}
	Indeed, if \(\gamma \in \Fuglede_{\tilde w_{2}}(K^{\ell}; \manfV)\), then \(\gamma \in \Fuglede_{w_{j_{i}}}(K^{\ell}; \manfV)\) for every \(i \in \N\), which justifies our claim.
		
	Let \(w = \widetilde w_{1} + \widetilde w_{2}\).{}
	Then, \(w\) is an \(\ell\)-detector for \(u\) and for each \(u_{j_{i}} - u\), whence also for \(u_{j_{i}}\).
	Moreover, for every simplicial complex \(\cK^{\ell}\) and every \(\gamma \in \Fuglede_{w}(K^{\ell}; \manfV)\), as we have 
	\[{}
	w_{j_{i}} \compose \gamma \to 0
	\quad \text{in \(\Lebesgue^{1}(K^{\ell})\),}
	\]
	it follows from \eqref{eq-Tools-Convolution-407} with \(j = j_{i}\) that \(u_{j_{i}} \compose \gamma \to u \compose \gamma\) in \(\VMO(K^{\ell})\).
\end{proof}

Although we have introduced \(\ell\)-detectors and \(\VMO^{\ell}\) functions using Fuglede maps defined on polytopes, it is a simple observation that these concepts are equally suitable for composition with Fuglede maps on spheres.
In fact,

\begin{proposition}
	\label{propositionExtensionVMOSphereDetector}
	If a measurable function \(u \colon \manfV \to \R\) has an \(\ell\)-detector  \(w \colon \manfV \to [0, +\infty]\), then for every \(\gamma \in \Fuglede_{w}(\Sphere^{\ell}; \manfV)\) we have \(u \compose \gamma \in \VMO(\Sphere^{\ell})\).{}
\end{proposition}

\begin{proposition}
	\label{propositionExtensionVMOSphere}
	Let \(u \in \VMO^{\ell}(\manfV)\) and let \(w\) be an \(\ell\)-detector given by Definition~\ref{definitionVMOell}.
    Then, for every map \(\gamma\) and every sequence \((\gamma_{j})_{j \in \N}\) in \(\Fuglede_{w}(\Sphere^{\ell}; \manfV)\) such that
	\(\gamma_{j} \to \gamma\) in \(\Fuglede_{w}(\Sphere^{\ell}; \manfV)\), we have
	\[{}
	u \compose \gamma_{j} \to u \compose \gamma{}
	\quad \text{in \(\VMO(\Sphere^{\ell})\).}
	\]
\end{proposition}

\begin{proposition}
	\label{propositionExtensionVMOSphereConvergence}
	Let \(u \in \VMO^{\ell}(\manfV)\).
    If \( (u_{j})_{j \in \N} \) is a sequence in \( \VMO^{\ell}(\manfV) \) such that \(u_{j} \to u\)
	in \(\VMO^\ell(\manfV)\), then an \(\ell\)-detector for \(u\) and all \(u_{j}\) given by Definition~\ref{def_conv_VMOl} is such that, for every \(\gamma \in \Fuglede_{w}(\Sphere^{\ell}; \manfV)\),
	\[{}
	u_{j} \compose \gamma \to u \compose \gamma{}
	\quad \text{in \(\VMO(\Sphere^{\ell})\).}
	\]
\end{proposition}

We rely on the following property of composition of \(\VMO\) functions with a biLipschitz homeomorphism:

\begin{lemma}
	\label{lemmaDetectorsVMObiLipschitz}
	Let \(\Simplex^{\ell + 1}\) be a simplex and let \(v \colon  \partial\Simplex^{\ell + 1} \to \R\) be a measurable function.
	If \(\Phi \colon  \partial\Simplex^{\ell + 1} \to \Sphere^{\ell}\) is a biLipschitz homeomorphism, then \(v \in \VMO(\partial\Simplex^{\ell + 1})\) if and only if \(v \compose \Phi^{-1} \in \VMO(\Sphere^{\ell})\).{}
	Moreover,
	\[{}
	\frac{1}{C} \norm{v}_{\VMO(\partial\Simplex^{\ell + 1})}
	\le \norm{v \compose \Phi^{-1}}_{\VMO(\Sphere^{\ell})}
	\le C \norm{v}_{\VMO(\partial\Simplex^{\ell + 1})}.
	\]
\end{lemma}

\resetconstant
\begin{proof}[Proof of Lemma~\ref{lemmaDetectorsVMObiLipschitz}]
	To prove the direct implication, it suffices to show that	 
	\begin{equation}
		\label{eqDetectors-982}
	\norm{v \compose \Phi^{-1}}_{\VMO(\Sphere^{\ell})}
	\le \C \norm{v}_{\VMO(\partial\Simplex^{\ell + 1})}.
	\end{equation}
	The \(\Lebesgue^{1}\) estimate immediately follows from a change of variable.
	We now prove that, for every \(\rho > 0\), the mean oscillations of \(v \compose \Phi^{-1}\) and \(v\) satisfy
	\begin{equation}
	\label{eqDetectors-975}
	\seminorm{v \compose \Phi^{-1}}_{\rho}
	\le \Cl{cte-977} \seminorm{v}_{\alpha\rho},
	\end{equation}
	where \(\alpha \vcentcolon= \abs{\Phi}_{\Lip} + \abs{\Phi^{-1}}_{\Lip}\).
	To this end, take \(x \in \partial\Simplex^{\ell + 1}\) and \(\xi = \Phi(x) \in \Sphere^{\ell}\).{}
	By Lipschitz continuity of \(\Phi^{-1}\),{}
	\[{}
	\Phi^{-1}(B_{\rho}^{\ell}(\xi)) \subset B_{\alpha\rho}^{\ell}(x).
	\]
	Thus, since the Jacobian \(\Jacobian{\ell}{\Phi}\) is bounded, by a change of variables and the inclusion above,
	\begin{multline}
	\label{eqDetectors-986}
	\int_{B_{\rho}^{\ell}(\xi)} \int_{B_{\rho}^{\ell}(\xi)} |v \compose \Phi^{-1}(\eta) - v \compose \Phi^{-1}(\zeta)|{}
	\dif\cH^{\ell}(\zeta) \dif\cH^{\ell}(\eta)\\
	\le \C \int_{B_{\alpha\rho}^{\ell}(x)} \int_{B_{\alpha\rho}^{\ell}(x)} |v(y) - v(z)| \dif\cH^{\ell}(z) \dif\cH^{\ell}(y).
	\end{multline}
	Next, by Lipschitz continuity of \(\Phi\),{}
	\begin{equation}\label{eq1236}
	\Phi(B_{\alpha\rho}^{\ell}(x)) \subset B_{\alpha^{2}\rho}^{\ell}(\xi).
	\end{equation}
	Since the Jacobian \(\Jacobian{\ell}{\Phi}\) is bounded from below by a positive constant, it follows from the area formula that 
    \[
    \cH^{\ell}(B_{\alpha\rho}^{\ell}(x))\le \C \cH^\ell(\Phi(B_{\alpha\rho}^{\ell}(x))).
    \]
    Together with the inclusion ~\eqref{eq1236} and the fact that the Hausdorff measure \(\cH^{\ell}\) satisfies the doubling property on \(\Sphere^{\ell}\), one gets
	\begin{equation}
	\label{eqDetectors-997}
	\cH^{\ell}(B_{\alpha\rho}^{\ell}(x))
	\le \C \cH^{\ell}(B_{\rho}^{\ell}(\xi)).
	\end{equation}
	Combining \eqref{eqDetectors-986} and \eqref{eqDetectors-997}, we obtain \eqref{eqDetectors-975}.  
	Taking the supremum with respect to \(\rho > 0\), we then get
	\[{}
	\seminorm{v \compose \Phi^{-1}}_{\VMO(\Sphere^{\ell})}
	\le \Cr{cte-977} \seminorm{v}_{\VMO(\partial\Simplex^{\ell + 1})},
	\]
	which completes the proof of \eqref{eqDetectors-982} and the direct implication.
	The reverse implication is proved along the same lines.
\end{proof}

\begin{proof}[Proofs of Propositions~\ref{propositionExtensionVMOSphereDetector}, \ref{propositionExtensionVMOSphere} and \ref{propositionExtensionVMOSphereConvergence}]
	Take a simplex \(\Simplex^{\ell + 1}\) and let \(\Phi \colon  \partial\Simplex^{\ell + 1} \to \Sphere^{\ell}\) be a biLipschitz homeomorphism.
	If \(u \colon \manfV \to \R \) has an \(\ell\)-detector \(w\) and \(\gamma \in \Fuglede_{w}(\Sphere^{\ell}; \manfV)\), then \(\gamma \compose \Phi\) is Lipschitz-continuous and, by a change of variable,
    \[
    \int_{\partial\Simplex^{\ell + 1}} w \compose \gamma \compose \Phi 
    \le C \int_{\Sphere^{\ell}} w \compose \gamma
    < \infty.
    \]
    Hence, \(\gamma \compose \Phi \in \Fuglede_{w}(\partial\Simplex^{\ell + 1}; \manfV)\).{}
	Since \(w\) is an \(\ell\)-detector, we have \(u \compose \gamma \compose \Phi \in \VMO(\partial\Simplex^{\ell + 1})\) and then, by composition with \(\Phi^{-1}\) and applying Lemma~\ref{lemmaDetectorsVMObiLipschitz}, we get 
    \[
    u \compose \gamma \in \VMO(\Sphere^{\ell}),
    \]
    which proves Proposition~\ref{propositionExtensionVMOSphereDetector}.{}
	
	Next, if \(\gamma_{j} \to \gamma\) in \(\Fuglede_{w}(\Sphere^{\ell}; \manfV)\), then \(\gamma_{j} \compose \Phi \to \gamma \compose \Phi\) in \(\Fuglede_{w}(\partial\Simplex^{\ell + 1}; \manfV)\).
	Thus, assuming that \(u\) is \(\VMO^{\ell}\) and \(w\) is an \(\ell\)-detector given by Definition~\ref{definitionVMOell},
	\[{}
	u \compose \gamma_{j} \compose \Phi \to u \compose \gamma \compose \Phi
	\quad \text{in \(\VMO(\partial\Simplex^{\ell + 1})\),}
	\]
	and the conclusion of Proposition~\ref{propositionExtensionVMOSphere} follows by composition with \(\Phi^{-1}\) using Lemma~\ref{lemmaDetectorsVMObiLipschitz}.

    Finally, let \( (u_{j})_{j \in \N} \) be a sequence in \( \VMO^{\ell}(\manfV) \) such that \(u_{j} \to u\) in \(\VMO^\ell(\manfV)\) and let \(w\) be an \(\ell\)-detector given by Definition~\ref{def_conv_VMOl}.
    For every \(\gamma \in \Fuglede_{w}(\Sphere^{\ell}; \manfV)\), we have \(\gamma \compose \Phi  \in \Fuglede_{w}(\partial\Simplex^{\ell + 1}; \manfV)\) and then, by the choice of \(w\),
	\[{}
	u_{j} \compose \gamma \compose \Phi \to u \compose \gamma \compose \Phi{}
	\quad \text{in \(\VMO(\partial\Simplex^{\ell + 1})\).}
	\]
    Hence, by composition with \(\Phi^{-1}\) and applying Lemma~\ref{lemmaDetectorsVMObiLipschitz},
	\[{}
	u_{j} \compose \gamma \to u \compose \gamma{}
	\quad \text{in \(\VMO(\Sphere^{\ell})\),}
	\]
    which settles Proposition~\ref{propositionExtensionVMOSphereConvergence}.
\end{proof}

\cleardoublepage
\chapter{Topology and VMO}
\label{section_VMO}

Brezis and Nirenberg~\cite{BrezisNirenberg1995} developed a comprehensive \(\VMO\) theory that extends classical topological notions, including the degree of continuous functions between manifolds, to the broader \(\VMO\) framework. This chapter builds on their foundational work based on the concept of \(\VMO\) homotopy and explores its applications to maps defined on metric measure spaces \(X\), such as Riemannian manifolds or polytopes. 

Our self-contained approach extends the setting of~\cite{BrezisNirenberg1995}, which deals mainly with maps defined on manifolds, to more general metric measure spaces \(X\). Furthermore, we revisit the proof of the density of continuous functions in \(\VMO(X; \manfN)\) by adapting it to this more general setting.
We highlight key ideas from the original proof that are essential for our purposes.

Given a compact manifold \(\manfN\), we define two \(\VMO\) maps \(u, v \colon X \to \manfN\) as homotopic in \(\VMO\) whenever they can be connected by a continuous path in \(\VMO(X; \manfN)\). In other words, \(u\) and \(v\) belong to the same path-connected component of \(\VMO(X; \manfN)\). We show that these path-connected components coincide with the connected components of \(\VMO(X; \manfN)\), strengthening the topological consistency of the \(\VMO\) framework.

Furthermore, we establish that each connected component of \(\VMO(X; \manfN)\) contains exactly one connected component of \(\Smooth^{0}(X; \manfN)\). This result underscores the compatibility of \(\VMO\) homotopy with classical topological concepts, providing a natural bridge between the \(\Smooth^{0}\) and \(\VMO\) settings.

\section{Density of BUC}
\label{sectionBUC}
We denote by \(\UContinuous(X)\) the space of bounded uniformly continuous functions on \(X\) equipped with the sup norm.
When \(X\) is compact, \(\UContinuous(X) = \Smooth^{0}(X)\).{}
We assume throughout the chapter that \(X\) is a locally compact and separable metric space and has a finite measure \(\mu\) that satisfies the doubling property, metric continuity and uniform nondegeneracy conditions introduced in Section~\ref{sectionVMOFunctions}.
In particular, we have the inclusion 
\[{}
\UContinuous(X) \subset \VMO(X).
\]
The space \(\VMO(X)\) can generally be characterized as the completion of \(\UContinuous(X)\) with respect to the \(\VMO\)~norm.{}
Indeed, the classical argument for manifolds \cite{Sarason_1975}*{Theorem~1} is also valid for polytopes, see also \cite{Hadwin_Yousefi_2008}*{Theorem~3.4} for a statement with a general measure.

\begin{proposition}
	\label{propositionDensityVMO}
	The space \(\UContinuous(X)\) is dense in \(\VMO(X)\).{}
	More precisely, for every \(v \in \VMO(X)\) and  every \(s > 0\), the averaged function \(v_{s} \colon  X \to \R\) defined for every \(x \in X\) by
\begin{equation}
\label{eqVMOConvolution}
v_{s}(x) = \fint_{B_{s}(x)}{v \dif\mu}
\end{equation} 
belongs to \(\UContinuous(X)\) and we have
\[{}
\lim_{s \to 0}{\norm{v_{s} - v}_{\VMO(X)}}
= 0.
\]
\end{proposition}

The strategy of the proof is standard.
We present the details here to clarify the role played by the various assumptions satisfied by \(\mu\).{}
We begin with some properties of \(v_{s}\):

\begin{lemma}
	\label{lemmaVMOLebesque}
	If \(v \in \Lebesgue^{1}(X)\), then, for every \(s > 0\), 
	\[{}
  \norm{v_{s}}_{\Lebesgue^{1}(X)} \le C \norm{v}_{\Lebesgue^{1}(X)}.
	\]
\end{lemma}

\resetconstant
\begin{proof}[Proof of Lemma~\ref{lemmaVMOLebesque}]
	By Tonelli's theorem,
	\[{}
	\begin{split}
	\norm{v_{s}}_{\Lebesgue^{1}(X)}
	& \le \int_{X} \frac{1}{\mu(B_{s}(x))} \int_{B_{s}(x)}{\abs{v(y)} \dif\mu(y)} \dif\mu(x){}\\
	& = \int_{X} \biggl(\int_{B_{s}(y)}{\frac{1}{\mu(B_{s}(x))}} \dif\mu(x) \biggr) \abs{v(y)} \dif\mu(y).
	\end{split}
	\]
	For every \(x \in B_{s}(y)\), we have \(B_{s}(y) \subset B_{2s}(x)\).{}
	Thus, by monotonicity and the doubling property of \(\mu\),
	\[{}
	\mu(B_{s}(y)) 
	\le \mu(B_{2s}(x)) 
	\le \Cl{cteVMO-133} \mu(B_{s}(x)). 
	\]
	Hence,
	\[{}
	\norm{v_{s}}_{\Lebesgue^{1}(X)}
	\le \Cr{cteVMO-133} \int_{X} \abs{v(y)} \dif\mu(y).{}
	\qedhere
	\]
\end{proof}
	
We now give a separate argument that \(v_{s}\) is bounded and uniformly continuous.
	
\begin{lemma}
	\label{lemmaVMOConvolutionCompactness}
	Let \(s > 0\).{}
	If \(v \in \Lebesgue^{1}(X)\), then \(v_{s} \in \UContinuous(X)\) and
	\[{}
	\norm{v_{s}}_{\Smooth^{0}(X)} \le C' \norm{v}_{\Lebesgue^{1}(X)}.
	\]	
\end{lemma}

\begin{proof}[Proof of Lemma~\ref{lemmaVMOConvolutionCompactness}]
	Given \(v \in \Lebesgue^{1}(X)\), for every \(x \in X\) we have
	\[{}
	\abs{v_{s}(x)} \le \frac{1}{\mu(B_{s}(x))} \norm{v}_{\Lebesgue^{1}(X)}.
	\]
	By the uniform nondegeneracy of \(\mu\), there exists \(\eta > 0\) independent of \(x\) such that \(\mu(B_{s}(x)) \ge \eta\).{}
	Hence,
	\begin{equation*}
		\sup_{x \in X}{\abs{v_{s}(x)}} \le \frac{1}{\eta} \norm{v}_{\Lebesgue^{1}(X)}.
	\end{equation*}

	To prove that \(v_{s}\) is uniformly continuous, we rely on the metric continuity of \(\mu\).{}
	For \(y, z \in X\), we first use the triangle inequality to estimate 
	\begin{multline*}
	\abs{v_{s}(y) - v_{s}(z)}\\
	\le 
	\frac{1}{\mu(B_{s}(y))} \biggabs{\int_{B_{s}(y)}{v \dif\mu} - \int_{B_{s}(z)}{v \dif\mu}}
	+ \frac{\bigabs{\mu(B_{s}(y)) - \mu(B_{s}(z))}}{\mu(B_{s}(y))\mu(B_{s}(z))}
	\int_{B_{s}(z)}{\abs{v} \dif\mu}.
	\end{multline*}
	By the uniform nondegeneracy of \(\mu\) and additivity of the integral, we thus have
	\begin{equation*}
	\abs{v_{s}(y) - v_{s}(z)}
	\le \frac{1}{\eta} \int_{B_{s}(y) \triangle B_{s}(z)}{\abs{v} \dif\mu} + \frac{1}{\eta^{2}} \mu(B_{s}(y) \triangle B_{s}(z)) \int_{X}\abs{v} \dif\mu.
	\end{equation*}
	The uniform continuity of \(v_{s}\) thus follows from the absolute continuity of the integral of a summable function and the metric continuity of \(\mu\).{}
\end{proof}

We now estimate the mean oscillations of \(v_{s}\) in terms of those of \(v\):

\begin{lemma}
  \label{lemmaVMOUpperBound}
  Let \(s > 0\).{}
  If \(v \in \Lebesgue^{1}(X)\), then{}
\begin{equation}
	\label{eq-lemmaVMOUpperBound1}
  \seminorm{v_{s}}_{\rho} \le C'' \max{\bigl\{\seminorm{v}_{2\rho}, \seminorm{v}_{2s} \bigr\}}
  \quad \text{for every \(\rho, s > 0\).}
  \end{equation}
\end{lemma}

\resetconstant
\begin{proof}[Proof of Lemma~\ref{lemmaVMOUpperBound}]
For every \(y, z \in X\),
\[{}
\abs{v_{s} (y) - v_{s} (z)}
\le  \fint_{B_s (y)} \fint_{B_s (z)} \abs{v (\xi) - v (\zeta)} \dif\mu(\zeta) \dif\mu(\xi)
\]
and then, by Tonelli's theorem,
\begin{multline*}
\fint_{B_{\rho} (x)} \fint_{B_{\rho} (x)} \abs{v_{s} (y) - v_{s} (z)}\dif\mu(z) \dif\mu(y)\\
\le \frac{1}{\mu(B_\rho(x))^2} \int_{B_{\rho + s} (x)} \int_{B_{\rho + s} (x)} \abs{v (\xi) - v (\zeta)} 
\\
\times
\int_{B_{\rho}(x) \cap B_{s}(\xi)}{\frac{\dif\mu(y)}{\mu(B_{s}(y))}}
\int_{B_{\rho}(x) \cap B_{s}(\zeta)}{\frac{\dif\mu(z)}{\mu(B_{s}(z))}}
 \dif\mu(\zeta) \dif\mu(\xi).{}
\end{multline*}
For every \(y \in B_{s}(\xi)\), we have \(B_{s}(\xi) \subset B_{2s}(y)\).{}
Hence, by monotonicity of \(\mu\) and the doubling property,
\[
\mu(B_{s}(\xi)) 
\leq \mu(B_{2s}(y))
\leq \Cl{cteVMO-157} \mu(B_s(y)).{}
\]
We thus have
\[{}
\int_{B_{\rho}(x) \cap B_{s}(\xi)}{\frac{\dif\mu(y)}{\mu(B_{s}(y))}}
\le \Cr{cteVMO-157} \, \frac{\mu(B_\rho (x) \cap B_{s}(\xi))}{\mu(B_s (\xi))}
\]
and a similarly estimate holds for the integral with respect to \(z\).{}
Hence,
\begin{multline*}
\fint_{B_{\rho} (x)} \fint_{B_{\rho} (x)} \abs{v_{s} (y) - v_{s} (z)}\dif\mu(y) \dif\mu(z)\\
\le \Cr{cteVMO-157}^{2} \int_{B_{\rho + s} (x)} \int_{B_{\rho  + s} (x)} \abs{v (\xi) - v (\zeta)} 
\, 
\frac{\mu(B_\rho (x) \cap B_{s}(\xi))}{\mu(B_{\rho}(x))\mu(B_s (\xi))}
\,{}
\frac{\mu(B_\rho (x) \cap B_{s}(\zeta))}{\mu(B_{\rho}(x))\mu(B_s (\zeta))}
\dif\mu(\zeta) \dif\mu(\xi).{}
\end{multline*}
When \(s \le \rho\), the monotonicity of \(\mu\) gives
\[{}
\frac{\mu(B_\rho (x) \cap B_{s}(\xi))}{\mu(B_{\rho}(x))\mu(B_s (\xi))}
\,{}
\frac{\mu(B_\rho (x) \cap B_{s}(\zeta))}{\mu(B_{\rho}(x))\mu(B_s (\zeta))}
\le \frac{1}{\mu(B_{\rho}(x))} \, \frac{1}{\mu(B_{\rho}(x))}
= \frac{1}{\mu(B_{\rho}(x))^{2}}
\]
and we deduce \eqref{eq-lemmaVMOUpperBound1} from the doubling property.
When \(\rho < s\), we use instead that
\[{}
\frac{\mu(B_\rho (x) \cap B_{s}(\xi))}{\mu(B_{\rho}(x))\mu(B_s (\xi))}
\,{}
\frac{\mu(B_\rho (x) \cap B_{s}(\zeta))}{\mu(B_{\rho}(x))\mu(B_s (\zeta))}
\le \frac{1}{\mu(B_s (\xi))\mu(B_s (\zeta))}.
\]
Observe that for \(\xi \in B_{\rho + s} (x)\), we have \(B_{2s}(x) \subset B_{4s}(\xi)\) and then, by monotonicity of \(\mu\) and the doubling property,
\(\mu(B_{2s}(x)) \le \Cr{cteVMO-157}^{2} \mu(B_{s}(\xi))\).{}
A similar estimate holds for \(\mu(B_{s}(\zeta))\) and the lemma then follows.
\end{proof}

\begin{proof}[Proof of Proposition~\ref{propositionDensityVMO}]
	By Lemma~\ref{lemmaVMOConvolutionCompactness}, each \(v_{s}\) is bounded and uniformly continuous.
	Since \(X\) is locally compact and separable, the set \(\Smooth_{c}^{0}(X)\) is dense in \(\Lebesgue^{1}(X)\) and thus a standard argument based on Lemma~\ref{lemmaVMOLebesque} implies the convergence of \(v_{s}\) to \(v\) in \(\Lebesgue^{1}(X)\) as \(s \to 0\).{}
	To conclude, by Lemma~\ref{lemmaEquiVMO} it suffices to prove that, for any sequence \((s_{j})_{j \in \N}\)  of positive numbers that converges to zero, \((v_{s_{j}})_{j \in \N}\) is equiVMO.
	Since \(v\) is a \(\VMO\)~function, given \(\epsilon > 0\)  there exists \(\delta_{1} > 0\) such that
	\[{}
	\sup_{\tilde{\rho} \le 2\delta_{1}}{\seminorm{v}_{\tilde{\rho}}}
	\le \frac{\epsilon}{C''},
	\]
	where \(C'' > 0\) is the constant in Lemma~\ref{lemmaVMOUpperBound}.
	By this lemma, for \(s_{j} \le \delta_{1}\) and \(\rho \le \delta_{1}\), we then have
	\[{}
	\seminorm{v_{s_{j}}}_{\rho}
	\le C'' \max{\bigl\{\seminorm{v}_{2\rho}, \seminorm{v}_{2s_j} \bigr\}}
	\le C'' \sup_{\tilde{\rho} \le 2\delta_{1}}{\seminorm{v}_{\tilde{\rho}}}
	\le \epsilon.
	\]		
	We then take \(J \in \N\) such that, for \(j \ge J\), \(s_{j} \le \delta_{1}\).{}
	By Lemma~\ref{lemmaVMOLebesque}, each \(v_{s_{j}}\) is uniformly continuous.
	Hence, there exists \(\delta_{2} > 0\) such that
	\[{}
	\seminorm{v_{s_{j}}}_{\rho} \le \epsilon{}
	\quad \text{for every \(0 < \rho \le \delta_{2}\) and \(j \le J - 1\).}
	\]
	The sequence \((v_{s_{j}})_{j \in \N}\) thus satisfies \eqref{eqEquiVMO} with \(\delta \vcentcolon= \min{\{\delta_{1}, \delta_{2}\}}\).
\end{proof}

\section{Homotopy in VMO}
\label{sectionVMOHomotopy}

One defines the space \(\VMO(X; \R^{\nu})\) of vector-valued maps \(v \colon  X \to \R^{\nu}\) with vanishing mean oscillation by requiring that each of the components of \(v\) belongs to \(\VMO(X)\). All the results obtained in the previous section for real-valued \(\VMO\) functions remain true in the vector-valued case.
{}
Given a compact manifold \(\manfN\) imbedded in \(\R^{\nu}\), we then consider
\[{}
\VMO(X; \manfN)
\vcentcolon= \bigl\{ v \in \VMO(X; \R^{\nu}) : v(x) \in \manfN\ \text{for every \(x \in X\)}  \bigr\}.
\]
From the topological point of view, \(\VMO\) maps inherit several properties that are satisfied by continuous maps.
For example, one can continuously extend the notion of topological degree in \(\Smooth^{0}(\Sphere^{\ell}; \Sphere^{\ell})\) to the larger space \(\VMO(\Sphere^{\ell}; \Sphere^{\ell})\), see \cite{BrezisNirenberg1995}.{}

\begin{definition}
\label{defnVMOHomotopy}
Two maps \(v_{0}, v_{1} \in \VMO(X; \manfN)\) are homotopic in \(\VMO(X; \manfN)\) whenever there exists a continuous path \(H\colon  [0, 1] \to \VMO(X; \manfN)\) with \(H(0) = v_{0}\) and \(H(1) = v_{1}\).{}
We write in this case 
\[{}
v_{0} \sim v_{1} 
\quad \text{in \(\VMO(X; \manfN)\)}.{}
\]
\end{definition}

Two maps in \(\VMO(X; \manfN)\) are thus homotopic provided that they belong to the same path-connected component with respect to the \(\VMO\) topology.
We prove that every path-connected component of \(\VMO(X; \manfN)\) has an element in \(\UContinuous(X; \manfN)\).

\begin{proposition}
	\label{propositionHomotopyVMOContinuousMap}
	For every \(v \in \VMO(X; \manfN)\), there exists \(f \in \UContinuous(X; \manfN)\) such that
	\[{}
	v \sim f 
	\quad \text{in \(\VMO(X; \manfN)\)}.{}
	\]
\end{proposition}

The main tool in the proof of Proposition \ref{propositionHomotopyVMOContinuousMap} is the family \(v_s\) introduced in Section \ref{sectionBUC}. 
However, as \(v_{s}\) need not take its values in \(\manfN\), we first justify why \(v_{s}\) can be projected to \(\manfN\) for small parameters \(s\).
As a result, we construct a continuous path \(H \colon  [0, 1] \to \VMO(X; \manfN)\) between \(v\) and \(f\) which satisfies 
\[{}
H(t) \in \UContinuous(X; \manfN)
\quad \text{for every \(0 < t \le 1\).}
\] 
One deduces in particular that 

\begin{corollary}
\label{corollaryVMODensity}
The subset \(\UContinuous(X; \manfN)\) is dense in \(\VMO(X; \manfN)\).{}
\end{corollary}

We begin with the standard continuity of composition in \(\VMO\) with Lipschitz functions:

\begin{lemma}
\label{lemma_VMO_average_homotopy}
   Let \(\Phi \colon  \R^{\nu} \to \R^{\nu}\) be a Lipschitz function.
   For every \(v \in \VMO(X; \R^{\nu})\), the map \(s \in {[}0, +\infty{)} \mapsto \Phi \circ v_{s}\) is continuous in \(\VMO(X; \R^{\nu})\), where \(v_{s}\) is defined by \eqref{eqVMOConvolution} for \(s > 0\) and \(v_{0} \vcentcolon= v\).
\end{lemma}

\begin{proof}[Proof of Lemma~\ref{lemma_VMO_average_homotopy}]
By Proposition~\ref{propositionDensityVMO}, the map \(s \in {}[0, \infty) \mapsto v_{s} \in \VMO(X; \R^{\nu})\) is continuous at \(a = 0\).
To obtain the continuity at \(a > 0\), we take  \(f \in \UContinuous(X; \R^{\nu})\).
Applying the triangle inequality and Lemmas~\ref{lemmaVMOLebesque} and~\ref{lemmaVMOUpperBound} componentwise, for every \(s, a > 0\) we have
\[{}
\begin{split}
\norm{v_{s} - v_{a}}_{\VMO(X)}
& \le \norm{v_{s} - f_{s}}_{\VMO(X)}
+ \norm{f_{s} - f_{a}}_{\VMO(X)}
+ \norm{f_{a} - v_{a}}_{\VMO(X)}\\
& \le C \norm{v - f}_{\VMO(X)}
+ \norm{f_{s} - f_{a}}_{\VMO(X)}\,.
\end{split}
\]
Thus, by uniform continuity of \(f\),{}
\[{}
\limsup_{s \to a}{\norm{v_{s} - v_{a}}_{\VMO(X)}}
\le C \norm{v - f}_{\VMO(X)}\,.
\]
Taking the infimum in the right-hand side with respect to \(f\), we deduce from Proposition~\ref{propositionDensityVMO} that the limsup vanishes.{}

Now, given any \(a\geq 0\), we take a Lipschitz map \(\Phi\) and a sequence of positive numbers \((s_{j})_{j \in \N}\) that converges to \(a\).{}
	The convergence 
	\begin{equation}
	\label{eq-VMO297}
	\Phi \compose v_{s_{j}} \to \Phi \compose v_{a}
	\quad \text{in \(\Lebesgue^{1}(X; \R^{\nu})\)}
	\end{equation}
	follows from the Lipschitz continuity of \(\Phi\).
	We next observe that, for every \(\rho > 0\) and \(j \in \N\),
	\begin{equation}
		\label{eq-VMO345}
	\seminorm{\Phi \compose v_{s_{j}}}_{\rho}
	\le \abs{\Phi}_{\Lip} \seminorm{v_{s_{j}}}_{\rho}\, .
	\end{equation}
	Since \(v_{s_{j}} \to v_{a}\) in \(\VMO(X; \R^{\nu})\), Lemma~\ref{lemmaEquiVMO} implies that the sequence \((v_{s_{j}})_{j \in \N}\) is equiVMO.
	By estimate~\eqref{eq-VMO345}, we deduce that \((\Phi \compose v_{s_{j}})_{j \in \N}\) is also equiVMO.
	Using \eqref{eq-VMO297}, the reverse implication of Lemma~\ref{lemmaEquiVMO} gives 
	\[{}
	\Phi \compose v_{s_{j}} \to \Phi \compose v_{a}
	\quad \text{in \(\VMO(X; \R^{\nu})\).}
	\qedhere
	\]
\end{proof}

A map \(v \in \VMO(X; \manfN)\) cannot be uniformly approximated by the family \((v_{s})_{s > 0} \) unless \(v\) itself is continuous.
We show however that the image of \(v_{s}\) is uniformly close to the target manifold \(\manfN\) due to the \(\VMO\) property of \(v\).
This is a simple but fundamental observation that lies at the heart of the proof of Theorem~\ref{theoremSchoen-Uhlenbeck}:

\begin{lemma}
	\label{lemmaVMOUniformConvergence}
	If \(v \in \Lebesgue^{1}(X; \manfN)\), then
	\[{}
	\sup_{x \in X}{d(v_{s}(x), \manfN)} \le \seminorm{v}_{s}
	\quad \text{for every \(s > 0\).}
	\]
	Hence, for every \(v \in \VMO(X; \manfN)\) we have
	\[{}
	\lim_{s \to 0}{\Bigl(\,\sup_{x \in X}{d(v_{s}(x), \manfN) \Bigr)}} = 0.
	\]
\end{lemma}

\begin{proof}[Proof of Lemma~\ref{lemmaVMOUniformConvergence}]
As \(v\) takes its values in the manifold \(\manfN\), for every \(x, z \in X\) we have
\[{}
d(v_{s}(x), \manfN)
\le \abs{v_{s}(x) - v(z)}
\le \fint_{B_{s}(x)}{\abs{v(y) - v(z)}} \dif\mu(y).
\]
Hence, averaging both sides with respect to \(z\) over the ball \(B_s (x)\), we get
\begin{equation*}
d(v_{s}(x), \manfN)
\le \fint_{B_{s}(x)}{\fint_{B_{s}(x)}{\abs{v(y) - v(z)}  \dif\mu(y) \dif\mu(z)}} 
\le \seminorm{v}_{s}\,.
\end{equation*}
The quantity in the right-hand side converges uniformly to zero with respect to \(x\) as \(s \to 0\) when \(v \in \VMO(X; \manfN)\).
\end{proof}

\begin{proof}[Proof of Proposition~\ref{propositionHomotopyVMOContinuousMap}]
	Since \(\manfN\) is compact, the nearest point projection \(\Pi\)  to \(\manfN\) is well defined and smooth in a tubular neighborhood \(\manfN + B_{\iota}^{\nu}\), with \(\iota > 0\).
    Since \(v\) has vanishing mean oscillation, we can take \(\delta > 0\) such that \(\sup\limits_{\rho \le \delta}{\seminorm{v}_{\rho}} < \iota\).{}
	By Lemma~\ref{lemmaVMOUniformConvergence}, for every \(0 < s \le \delta\) the set \(v_{s}(X)\) is contained in \(\manfN + B_{\iota}^{\nu}\).
	By Lemma~\ref{lemma_VMO_average_homotopy}, the map
	\begin{equation*}
	H \colon  t \in [0, 1] \longmapsto \Pi \compose v_{\delta t}
	\in \VMO(X; \manfN)
	\end{equation*}
	is a continuous path between \(v\) and \(f \vcentcolon= \Pi \compose v_{\delta}\) in \(\VMO(X; \manfN)\).{}
\end{proof}

We now prove the following properties taken from \cite{BrezisNirenberg1995}*{Lemmas~A.19 and~A.20}, where they are stated for functions defined in manifolds.
The first implies that the path-connected components of \(\VMO(X; \manfN)\) are open and closed.
As a consequence, they coincide with the connected components of \(\VMO(X; \manfN)\).

\begin{proposition}
\label{propositionHomotopyVMOLimit}
  Let \(v \in \VMO(X; \manfN)\).{}
  If \((v_{j})_{j \in \N}\) is a sequence in \(\VMO(X; \manfN)\) with
  \[{}
  v_{j} \to v
  \quad \text{in \(\VMO(X; \R^{\nu})\),}
  \]
  then there exists \(J \in \N\) such that, for every \(j \ge J\),{}
  \[{}
  v_{j} \sim v
  \quad \text{in \(\VMO(X; \manfN)\).}
  \]
\end{proposition}

\begin{proof}
  By Lemma~\ref{lemmaEquiVMO}, the sequence \((v_{j})_{j \in \N}\) is equiVMO.
  Thus, there exists \(\delta > 0\) such that, for every \(j \in \N\), 
  \( \sup\limits_{\rho \le \delta}{\seminorm{v_{j}}_{\rho}} < \iota\).
  We deduce from Lemmas~\ref{lemma_VMO_average_homotopy} and~\ref{lemmaVMOUniformConvergence} that the map 	
  \[{}
	t \in [0, 1] \longmapsto \Pi \compose v_{j, \delta t}
	\in \VMO(X; \manfN)
	\]
	is a continuous path between \(v_{j}\) and \(\Pi \compose v_{j, \delta}\)\,.{}
	Then, for every \(j \in \N\),
  \begin{equation}
  \label{eqVMO-553}
  v_{j} 
  \sim \Pi \compose v_{j, \delta}
  \quad \text{in \(\VMO(X; \manfN)\).}
  \end{equation}
  We may also assume that \(\sup\limits_{\rho \le \delta}{\seminorm{v}_{\rho}} < \iota\).{}
  Similarly, we have in this case
  \begin{equation*}
   v \sim \Pi \compose v_{\delta} 
  \quad \text{in \(\VMO(X; \manfN)\).}
  \end{equation*}
  Since \(v_{j} \to v\) in \(L^{1}(X)\), the equicontinuous sequence \((v_{j, \delta})_{j \in \N}\) converges uniformly to \(v_{\delta}\) on \(X\), as a consequence of Lemma~\ref{lemmaVMOConvolutionCompactness}.
  We may thus take \(J \in \N\) such that, for every \(j \ge J\), there exists a map \(H \in \UContinuous(X \times [0, 1]; \manfN)\) such that \(H(\cdot, 0) = \Pi \compose v_{j, \delta}\) and \(H(\cdot, 1) = \Pi \compose v_{\delta}\).{}
  In particular,
  \begin{equation}
  \label{eqVMO-568}
  \Pi \compose v_{j, \delta} \sim \Pi \compose v_{\delta}
  \quad \text{in \(\VMO(X; \manfN)\).}
  \end{equation}
  It thus follows from \eqref{eqVMO-553}--\eqref{eqVMO-568} and the transitivity of the homotopy relation that, for \(j \ge J\),
  \[
  v_{j} \sim v
  \quad \text{in \(\VMO(X; \manfN)\).}
  \qedhere
  \]
\end{proof}

We conclude this chapter by proving that each path-connected component of \(\VMO(X; \manfN)\) contains exactly one path-connected component of \(\UContinuous(X)\).

\begin{proposition}
\label{propositionHomotopyVMOtoC}
If \(v_{0}, v_{1} \in \UContinuous(X; \manfN)\) are such that
\[{}
v_{0} \sim v_{1} 
\quad \text{in \(\VMO(X; \manfN)\),}
\] 
then 
\[{}
v_{0} \sim v_{1} 
\quad \text{in \(\BUC(X; \manfN)\).}
\]
\end{proposition}

\begin{proof}
  By assumption, there exists a continuous path \(H \colon  [0, 1] \to \VMO(X; \manfN)\) such that \(H(0) = v_{0}\) and \(H(1) = v_{1}\).{}
  By Lemma~\ref{lemmaEquiVMO}, the continuity of \(H\) implies that, for every convergent sequence \((t_{n})_{n \in \N}\) in \([0, 1]\), the sequence \((H(t_{n}))_{n \in \N}\) is equiVMO. 
  It thus follows from a compactnesss argument over \([0 ,1]\) that 
\[{}
\lim_{\rho \to 0}{\seminorm{H(t)}_{\rho}} = 0
\]
uniformly with respect to \(t \in [0, 1]\), see also \cite{BrezisNirenberg1995}*{Lemma~A.16}.
The proposition relies on this property, as we can now take \(\delta > 0\) such that
\begin{equation}
\label{eqHomotopyEquiVMOPath}
\sup_{\rho \le \delta}{\seminorm{H(t)}_{\rho}}
< \iota{}
\quad \text{for every \(t \in [0, 1]\).}
\end{equation}
Since the map
\[{}
t \in [0, 1] \longmapsto H(t) \in \Lebesgue^{1}(X; \R^{\nu})
\]
is continuous, by Lemmas~\ref{lemmaVMOConvolutionCompactness} and \ref{lemmaVMOUniformConvergence} and by Lipschitz-continuity of \(\Pi\), the map
\begin{equation}
\label{eqHomotopyConvolution}
t \in [0, 1] \longmapsto \Pi \compose (H(t)_{\delta})
\in \BUC(X; \manfN)
\end{equation}
is a well defined continuous path.
Hence,
\[{}
\Pi \compose v_{0, \delta} \sim \Pi \compose v_{1, \delta}
\quad \text{in \(\BUC(X; \manfN)\).}
\]
Since \(v_{i} \in \UContinuous(X; \manfN)\) for \(i \in \{0, 1\}\) and \(\mu\) is metric continuous, the map 
\[
t\in [0,1] \longmapsto v_{i, \delta t} \in \BUC(X; \R^{\nu})
\]
is a continuous path, with the convention that \(v_{i,0}=v_i\). 
Moreover, by \eqref{eqHomotopyEquiVMOPath}, we have in particular that \(\sup\limits_{\rho \le \delta}{\seminorm{v_{i}}_{\rho}} < \iota\). 
By Lemma~\ref{lemmaVMOUniformConvergence} and by Lipschitz-continuity of \(\Pi\), it follows that the map 
  \[{}
  t \in [0, 1] \longmapsto \Pi \compose v_{i, \delta t} 
  \in \BUC(X; \manfN)
  \]
  is also a continuous path.
  Hence,
  \[{}
  v_{i} \sim \Pi \compose v_{i, \delta}
  \quad \text{in \(\BUC(X; \manfN)\).}
  \]
  Summarizing, we have
  \[{}
  v_{0} 
  \sim \Pi \compose v_{0, \delta}
  \sim \Pi \compose v_{1, \delta}
  \sim v_{1}
  \quad \text{in \(\BUC(X; \manfN)\)}.
  \]
  We thus have the conclusion by transitivity of the homotopy relation.
\end{proof}

\cleardoublepage
\chapter{Generic \texorpdfstring{$\ell$}{l}-extension}
\label{chapterGenericEllExtension}

Given \(\ell \in \N_*\), consider a map \(u \colon \R^{\ell + 1} \to \Sphere^{\ell}\) defined for \(x \ne 0\) by \(u(x)=x/\abs{x}\), which is continuous except at the origin. The topological nature of this singularity can be quantified by the degree of the restriction \(u|_{\Sphere^\ell} \colon \Sphere^\ell \to \Sphere^\ell\). More generally, for \(a \in \R^{\ell + 1}\) and \(\gamma_a \colon t \in \Sphere^\ell \mapsto t + a \in \R^{\ell + 1}\), we observe that
\[
\deg{(u \circ \gamma_a)} =
\begin{cases}
1 & \text{if } \abs{a} < 1, \\
0 & \text{if } \abs{a} > 1.
\end{cases}
\] 
Therefore, the connected component of \(\Smooth^0(\Sphere^\ell; \Sphere^\ell)\) to which \(u \circ \gamma_a\) belongs depends on the choice of the Lipschitz map \(\gamma_a\). 
This shows that, for an arbitrary map \(u \in \Sobolev^{1, \ell}(\manfV; \manfN)\), it is not always possible to associate a unique homotopy class in \(\pi_{\ell}(\manfN)\) by considering compositions \(u \circ \gamma\) with generic Lipschitz maps \(\gamma \colon \Sphere^{\ell} \to \manfV\).

Such a correspondence, however, becomes possible when \(u\) is \emph{\(\ell\)-extendable}, a notion introduced in this chapter in the broad setting of \(\VMO^\ell\) maps. When \(\pi_{\ell}(\manfN) \simeq \{0\}\), every map in \(\VMO^{\ell}(\manfV; \manfN)\) is \(\ell\)-extendable. Even without this topological property, a prominent example of an \(\ell\)-extendable map is given by \(u \colon \Ball^{\ell+1} \to \manfN\), defined for \(x \neq 0\) by \(u(x) = f(x / \abs{x})\), where \(f \in \Smooth^{\infty}(\cBall^{\ell+1}; \manfN)\).

The \(\ell\)-extendability of a map \(u\) implies the following topological stability property:
For generic Lipschitz maps \(\gamma_0\) and \(\gamma_1\) that are homotopic in \(\Smooth^0(\Sphere^\ell; \manfV)\), the maps \(u \circ \gamma_0\) and \(u \circ \gamma_1\) are homotopic in \(\VMO(\Sphere^\ell; \manfN)\). This result relies on \emph{Fuglede approximations}: 
Given a summable map \(w \colon \manfV \to [0, +\infty]\), every equiLipschitz sequence \((\gamma_j)_{j \in \N}\) converging to a generic map \(\gamma\) can be slightly modified to ensure convergence in \(\Fuglede_w\).

This chapter lays the groundwork for understanding \(\ell\)-extendability, a concept that will later be generalized to \((\ell, e)\)-extendability in Chapter~\ref{chapterExtensionGeneral}. 
By focusing here on the specific case of \(\ell\)-extendability, we establish key properties and results that will guide the broader discussion of extendability properties on simplicial complexes introduced later in this work.

\section{Extension and homotopy invariance}

The \(\ell\)-extension property for \(\VMO^{\ell}\) maps is motivated by the close interaction between the existence of continuous extensions and homotopies in classical topology theory.
For example, a map \(f \in \Smooth^{0}(\Sphere^{\ell}; \manfN)\) has an extension \(\Bar f \in \Smooth^{0}(\cBall^{\ell + 1}; \manfN)\) if and only if \(f\) is homotopic to a constant in \(\Smooth^{0}(\Sphere^{\ell}; \manfN)\).{}
Indeed, given a continuous extension \(\Bar f\), one may take the homotopy \(H \colon  \Sphere^{\ell} \times [0, 1] \to \manfN\) defined by \(H(x, t) = \Bar{f}((1-t)x)\) between \(f\) and \(c \vcentcolon= \Bar f(0)\).
Conversely, given a homotopy \(H\) with \(H(\cdot, 0) = f\) and \(H(\cdot, 1) =\vcentcolon c\in \manfN\), one can take 
\[{}
\Bar f(x) 
=
\begin{cases}
	H\bigl(x/|x|, 1 - \abs{x}\bigr)
	& \quad \text{if \(x \ne 0\),}\\
	c
	& \quad \text{if \(x = 0\).}
\end{cases}
\]

When \(f\) merely belongs to \(\VMO(\Sphere^{\ell}; \manfN)\), it is unclear how to proceed, since there is no analogue in the \(\VMO\) setting for the notion of \emph{continuous extension}. 
Alternatively, there is no suitable definition of traces for \(\VMO\) maps defined on \(\Ball^{\ell+1}\).   
We saw nevertheless in Chapter~\ref{section_VMO} that Brezis and Nirenberg~\cite{BrezisNirenberg1995} developed a robust homotopy theory for \(\VMO\) maps that coincides with the classical one.
This is the approach that we shall pursue as we consider \(\VMO^{\ell}\)~maps that are generically homotopic to a constant:

\begin{definition}
\label{definitionExtensionVMO-1}
Let \(\ell \in \N\).  
A map \(u \in \VMO^{\ell}(\manfV ; \manfN)\) is \emph{\(\ell\)-extendable} whenever it has an \(\ell\)-detector \(w \colon  \manfV \to [0,+\infty]\) such that, for every Lipschitz map \(\gamma \colon  \cBall^{\ell+1} \to  \manfV\) with 
\(\gamma|_{\Sphere^{\ell}} \in \Fuglede_{w}(\Sphere^\ell; \manfV)\),
the restricted map \(u \compose \gamma |_{\Sphere^{\ell}}\) is homotopic to a constant in \(\VMO(\Sphere^{\ell}; \manfN)\).
\end{definition}

Observe that \(\gamma\) itself being defined on \(\cBall^{\ell+1}\), its restriction \(\gamma|_{\Sphere^{\ell}}\) is homotopic to a constant. 
The \(\ell\)-extendability is related to the topology of \(\manfN\) at the level of its homotopy group \(\pi_{\ell}(\manfN)\)\,:

\begin{proposition}
	\label{corollaryExtensionHomotopyGroupEll}
	If \(\pi_{\ell}(\manfN) \simeq \{0\}\), then every map \(u \in \VMO^{\ell}(\manfV; \manfN)\) is \(\ell\)-extendable.
\end{proposition}

\begin{proof}
	Take any \(\ell\)-detector \(w\) for \(u\).{}
	If \(\gamma \in \Fuglede_{w}(\Sphere^{\ell}; \manfV)\), then \(u \compose \gamma \in \VMO(\Sphere^{\ell}; \manfN)\).{}
	By Proposition~\ref{propositionHomotopyVMOContinuousMap}, there exists \(f \in \Smooth^{0}(\Sphere^{\ell}; \manfN)\) such that
	\[{}
	u \compose \gamma \sim f 
	\quad \text{in \(\VMO(\Sphere^{\ell}; \manfN)\)}.
	\]
	Since \(\pi_{\ell}(\manfN) \simeq \{0\}\), the map \(f\) is homotopic to a constant in \(\Smooth^{0}(\Sphere^{\ell}; \manfN)\) and thus also in \(\VMO(\Sphere^{\ell}; \manfN)\). 
	By transitivity of the homotopy relation, \(u\compose \gamma\) is then homotopic to a constant in \(\VMO(\Sphere^{\ell}; \manfN)\).
\end{proof}

In the absence of the topological assumption \(\pi_{\ell}(\manfN) \simeq \{0\}\) one finds examples of maps that are not \(\ell\)-extendable, in the spirit of the work of Schoen and Uhlenbeck~\cite{Schoen-Uhlenbeck}:
 
\begin{proposition}
	\label{propositionExtensionHomotopyGroupEll}
	If \( \ell \le m-1 \) and \(\pi_{\ell}(\manfN) \not\simeq \{0\}\), then there exists a map \(u\) which is not \(\ell\)-extendable that belongs to \(\Sobolev^{k, p}(\Ball^{m}; \manfN)\) for any \(k \in \N_{*}\) and \(1 \le p < \infty\) with \(\floor{kp} \le \ell\).
\end{proposition}

To prove Proposition~\ref{propositionExtensionHomotopyGroupEll}, we begin with the observation that a naive way of extending a map \(f \in \Smooth^{0}(\Sphere^{\ell}; \manfN)\) is to take \(u \colon  \Ball^{\ell + 1} \to \manfN\) defined for \(x \ne 0\) by
\begin{equation}
	\label{eqLExtension-113}
	u(x) = f\Bigl( \frac{x}{\abs{x}} \Bigr){}.
\end{equation}
	By Proposition~\ref{proposition_Extension_Property_R_0-New} applied to each component of \(u\) with \(m = \ell + 1\), \(\manfV = \Ball^{\ell + 1}\) and \(T^{\ell^{*}} = \{0\}\), we have \(u \in \VMO^{\ell}(\Ball^{\ell + 1}; \manfN)\).{}
	However, there is no suitable choice for \(u(0)\) that makes \(u\) continuous unless \(f\) itself is constant.
	The following equivalence relates \(f\) and \(u\) in terms of extensions:

\begin{lemma}
	\label{propositionExtensionHomogeneity}
	Let \(\ell \in \N\) and \(f \in \Smooth^{0}(\Sphere^{\ell}; \manfN)\).{}
	Then, \(u \colon  \Ball^{\ell + 1} \to \manfN\) satisfying \eqref{eqLExtension-113} is \(\ell\)-extendable if and only if \(f\) has a continuous extension to \(\cBall^{\ell + 1}\).
\end{lemma}

\resetconstant
\begin{proof}[Proof of Lemma~\ref{propositionExtensionHomogeneity}]
	``\(\Longleftarrow\)''. 	
	From the proof of Proposition~\ref{proposition_Extension_Property_R_0-New}, an \(\ell\)-detector for \(u\) is given by the function \(w \colon  \Ball^{\ell + 1} \to [0, +\infty]\) defined by 
	\[
	w(x) = 
 \begin{cases}
  {1}/{\abs{x}^{\ell}} & \textrm{if } x\not=0,\\
+\infty & \textrm{if } x=0,
\end{cases}	
 \]
	and, for every \(\gamma \in \Fuglede_{w}(\Sphere^{\ell}; \Ball^{\ell + 1})\), we have 
	\begin{equation*}
		0 \not\in \gamma(\Sphere^{\ell}).{}
	\end{equation*}
    This implies that
    \begin{equation}
        \label{eqLExtension-131}
        u \compose \gamma 
	    = f \compose ({\gamma}/{\abs{\gamma}}).
    \end{equation}
	We now show that Definition~\ref{definitionExtensionVMO-1} is satisfied.
	To this end, take a Lipschitz map \(\gamma \colon  \cBall^{\ell+1} \to \Ball^{\ell + 1}\) with \(\gamma|_{\Sphere^{\ell}} \in \Fuglede_{w}(\Sphere^{\ell}; \Ball^{\ell + 1})\).{}
	Since \(f\) has a continuous extension to \(\cBall^{\ell+1}\), the map \(f \compose (\gamma/\abs{\gamma})|_{\Sphere^{\ell}}\) is homotopic to a constant in \(\Smooth^{0}(\Sphere^{\ell}; \manfN)\).{}
    We deduce from \eqref{eqLExtension-131} that \(u \compose \gamma|_{\Sphere^{\ell}}\) is homotopic to a constant in \(\Smooth^{0}(\Sphere^{\ell}; \manfN)\) and then also in \(\VMO(\Sphere^{\ell}; \manfN)\).{}
	
	``\(\Longrightarrow\)''.{}
	Take an \(\ell\)-detector \(w\) for \(u\) given by Definition~\ref{definitionExtensionVMO-1}.
	By Tonelli's theorem, for almost every \(0 < r < 1\), the restriction \(w|_{\partial B_{r}^{\ell + 1}}\) is summable in \(\partial B_{r}^{\ell + 1}\).
	Taking such an \(r\) and \(\gamma \colon  \cBall^{\ell+1} \to \Ball^{\ell + 1}\) defined by \(\gamma(x) = rx\), we have \(\gamma|_{\Sphere^{\ell}} \in \Fuglede_{w}(\Sphere^{\ell}; \Ball^{\ell + 1})\) and then \(f = u \compose \gamma|_{\Sphere^{\ell}}\) is homotopic to a constant in \(\VMO(\Sphere^{\ell}; \manfN)\).{}
	Since \(f\) is continuous, by Proposition~\ref{propositionHomotopyVMOtoC} we deduce that \(f\) is homotopic to a constant in \(\Smooth^{0}(\Sphere^{\ell}; \manfN)\).
    Hence, \(f\) has a continuous extension to \(\cBall^{\ell+1}\).
\end{proof}

\begin{proof}[Proof of Proposition~\ref{propositionExtensionHomotopyGroupEll}]
	Since \(\pi_{\ell}(\manfN) \not\simeq \{0\}\), there exists \(f \in \Smooth^{\infty}(\Sphere^{\ell}; \manfN)\) that cannot be continuously extended to \(\cBall^{\ell+1}\).{}
	Take \(\widetilde u \colon  \R^{\ell + 1} \times \R^{m - \ell - 1} \to \manfN\) defined for \(x = (x', x'') \in \R^{\ell + 1} \times \R^{m - \ell - 1}\) and \(x'\neq 0'\) by
	\[{}
	\widetilde u(x)
	= f\Bigl( \frac{x'}{\abs{x'}}  \Bigr).
	\]
	For every \(j \in \N_*\),{}
	\[{}
	\abs{D^{j}\widetilde u(x)}
	\le \frac{C_{j}}{\abs{x'}^{j}}.
	\]
	This implies that \(\widetilde u \in \Sobolev\loc^{k, p}(\R^{m}; \manfN)\) for \(kp < \ell + 1\).{}
	
	Let us assume by contradiction that \(\widetilde u|_{\Ball^{m}}\) is \(\ell\)-extendable.{}
	That is also the case for \(\widetilde u|_{B_{r}^{m}}\) with \(r \vcentcolon= \sqrt{m - \ell}\) by homogeneity of \(\widetilde u\), and then also for \(v \vcentcolon= \widetilde u|_{\Ball^{\ell + 1} \times (0, 1)^{m - \ell - 1}}\) since \(\Ball^{\ell + 1} \times (0, 1)^{m - \ell - 1} \subset B_{r}^{m}\).{}
	Let \(w\) be an \(\ell\)-detector that verifies the \(\ell\)-extendability of \(v\).{}
	By Tonelli's theorem, there exists \(a'' \in (0, 1)^{m - \ell - 1}\) with \(w(\cdot, a'')\) summable in \(\Ball^{\ell + 1}\).{}
	Let us prove that the map
	\begin{equation}
	\label{eqExtension-239}
	x' \in \Ball^{\ell+1} \longmapsto v(x', a'') \in \manfN
	\end{equation}
	is \(\ell\)-extendable using \(w(\cdot, a'')\) as \(\ell\)-detector.	
	Indeed, for any Lipschitz map \(\gamma \colon  \cBall^{\ell+1} \to \Ball^{\ell + 1}\) with \(\gamma|_{\Sphere^{\ell}} \in \Fuglede_{w(\cdot, a'')}(\Sphere^{\ell}; \Ball^{\ell + 1})\), we have 
	\[{}
	(\gamma|_{\Sphere^{\ell}}, a'') \in \Fuglede_{w}\bigl(\Sphere^{\ell}; \Ball^{\ell + 1} \times (0, 1)^{m - \ell - 1}\bigr)
	\]
	and then \(v \compose (\gamma|_{\Sphere^{\ell}}, a'')\) is homotopic to a constant in \(\VMO(\Sphere^{\ell}; \manfN)\).{}
	Therefore, the map \eqref{eqExtension-239} is \(\ell\)-extendable.
	Note however that \(v(x', a'') = f(x'/\abs{x'})\) for \(x' \ne 0'\).{}
	We thus have a contradiction with Lemma~\ref{propositionExtensionHomogeneity}.
	Therefore, \(\widetilde u|_{\Ball^{m}}\) is not \(\ell\)-extendable
\end{proof}

A large class of \(\ell\)-extendable maps is provided by a subset of \(\Sobolev^{k, p}(\manfV;\manfN)\):

\begin{definition}
	\label{definitionHkp}
	Let \(k \in \N_*\) and \(1 \le p < \infty\).{}
	We denote by \(\Hilbert^{k, p}(\manfV; \manfN)\) the set of maps \(u \in \Sobolev^{k, p}(\manfV; \manfN)\) for which there exists a sequence \((u_{j})_{j \in \N}\) in \((\Smooth^{\infty}\cap \Sobolev^{k, p})(\manfV; \manfN)\) such that 
	\[{}
	u_{j} \to u \quad \text{in \(\Sobolev^{k, p}(\manfV; \R^\nu)\).}
	\]
\end{definition}

When \(\manfV = \Omega\) is a bounded open set in \(\R^m\) with Lipschitz boundary, the set \(\Hilbert^{k, p}(\Omega; \manfN)\) coincides with the closure of \(\Smooth^{\infty}(\overline{\Omega}; \manfN)\) for the \(\Sobolev^{k,p}\)~distance, see Remark~\ref{rk-equivalence-of-definition}. 

\begin{proposition}
    \label{propositionDensityExtendable}
    If \(u \in \Hilbert^{k, p}(\manfV; \manfN)\) with \(kp \ge \ell\), then \(u\) is \(\ell\)-extendable.
\end{proposition}

\begin{proof}
    Take a sequence \((u_{j})_{j \in \N}\) in \((\Smooth^{\infty} \cap \Sobolev^{k, p})(\manfV; \manfN)\) that converges to \(u\) in \(\Sobolev^{k, p}(\manfV; \R^{\nu})\).
    Since \(\manfN\) is compact, by the Gagliardo-Nirenberg interpolation inequality, \(u_j \to u\) in \(\Sobolev^{1, kp}(\manfV; \manfN)\).
    By Propositions~\ref{lemmaFugledeSobolevDetector} and~\ref{propositionExtensionVMOSphereConvergence}, there exist a subsequence \((u_{j_{i}})_{i \in \N}\) and an \(\ell\)-detector \(w\) for \(u\) and all \(u_{j_{i}}\) such that, for every \(\gamma \in \Fuglede_{w}(\Sphere^{\ell}; \manfV)\),
    \[
    u_{j_{i}} \compose \gamma \to u \compose \gamma 
    \quad \text{in \(\VMO(\Sphere^{\ell}; \R^{\nu})\).}
    \]
    Then, by Proposition~\ref{propositionHomotopyVMOLimit} there exists \(J \in \N\) such that, for every \(j \ge J\),{}
  \[{}
  u_{j_{i}} \compose \gamma \sim u \compose \gamma
  \quad \text{in \(\VMO(\Sphere^{\ell}; \manfN)\).}
  \]
    Assuming that \(\gamma\) has a Lipschitz continuous extension  \(\Bar{\gamma} \colon \cBall^{\ell + 1} \to \manfV\), then, for every \(i\), by continuity of \(u_{j_{i}}\) the map \(u_{j_{i}} \compose \gamma\) is homotopic to a constant in \(\Smooth^{0}(\Sphere^{\ell}; \manfN)\), whence also in \(\VMO(\Sphere^{\ell}; \manfN)\).
    We deduce by transitivity of the homotopy relation that \(u \compose \gamma\) is homotopic to a constant in \(\VMO(\Sphere^{\ell}; \manfN)\).
\end{proof}

One of the main results proved by White in \cite{White-1986} states that every map \(u\in \Hilbert^{1,p}(\manfV;\manfN)\) has a well defined \(\floor{p}\)-homotopy type, in the sense that restrictions of \(u\) to generic \(\floor{p}\)-dimensional skeletons of \(\manfV\) are all homotopic.
In a subsequent work~\cite{White}, White showed that every \(u\in \Sobolev^{1,p}(\manfV;\manfN)\) has a well defined \((\floor{p}-1)\)-homotopy type.
We unify and extend these two properties in the framework of \(\ell\)-extendable functions in \(\VMO^{\ell}(\manfV; \manfN)\):

\begin{theorem}
\label{proposition_Reference_Homotopy-Sphere}
If \(u \in \VMO^{\ell}(\manfV; \manfN)\) is \(\ell\)-extendable, then there exists an \(\ell\)-detector \(\widetilde{w}\) such that, for every \(\gamma_0, \gamma_1 \in \Fuglede_{\tilde{w}}(\Sphere^\ell ; \manfV)\) with 
\[{}
\gamma_{0} \sim \gamma_{1}
\quad \text{in \(\Smooth^{0}(\Sphere^{\ell}; \manfV)\),} 
\] 
we have
\[
u \compose \gamma_0  \sim u \compose \gamma_1{}
\quad \text{in \(\VMO (\Sphere^\ell; \manfN)\).}
\]
\end{theorem}

A generalization of the above statement where the sphere \(\Sphere^\ell\) is replaced by a general simplicial complex is presented in Section~\ref{sectionLExtensionIndependence}, see Theorem~\ref{proposition_Reference_Homotopy} and also Corollary~\ref{corollatyWhiteHomotopyType}. 
The combination of Proposition~\ref{propositionDensityExtendable} and Corollary~\ref{corollatyWhiteHomotopyType} allows us to recover White's property~\cite{White-1986} concerning the existence of a \(\floor{kp}\)-homotopy type for maps in \(\Hilbert^{k,p}(\manfV;\manfN)\). 
We show in the next section, see Proposition~\ref{propositionLExtensionSobolev}, that any map \(u\in \Sobolev^{k,p}(\manfV;\manfN)\) is \((\floor{kp}-1)\)-extendable, from which one may deduce by the same argument that \(u\) has a \((\floor{kp} - 1)\)-homotopy type.

A consequence of Theorem~\ref{proposition_Reference_Homotopy-Sphere} is that the function
\begin{equation}
\label{eqExtension-284}
\gamma \in \Fuglede_{w}(\Sphere^{\ell}; \manfV)
\longmapsto u \compose \gamma \in \VMO(\Sphere^{\ell}; \manfN)
\end{equation}
may be quotiented in terms of homotopy equivalence classes on both sides to yield a well defined function among homotopy classes,
\begin{equation}
\label{eqExtension-290}
[\gamma] \in [\Sphere^\ell, \manfV]
\longmapsto [u \compose \gamma] \in [\Sphere^\ell, \manfN].
\end{equation}
In some cases, for example, if \(\pi_1 (\manfV)\) and \(\pi_{1} (\manfN)\)
are trivial or if, more generally, their actions on \(\pi_{\ell} (\manfV)\) and \(\pi_{\ell} (\manfN)\) are trivial, then \eqref{eqExtension-290} becomes a mapping from \(\pi_{\ell} (\manfV)\) into \(\pi_{\ell} (\manfN)\).
Map \eqref{eqExtension-290} intervenes, for example, in the homotopy classification of Ginzburg-Landau minimizers that has been implemented by Rubinstein and Sternberg~\cite{Rubinstein-Sternberg} in connection with the Sobolev space \(\Sobolev^{1, 2}(\Ball^3; \Sphere^1)\). 
Note that, by Proposition~\ref{propositionLExtensionSobolev} below, every map in this space is \(1\)-extendable.

Let us now explain how \eqref{eqExtension-290} is defined from \eqref{eqExtension-284}.
Firstly, the fact that \(\Smooth^0(\Sphere^\ell; \manfV)\) is not contained in the domain of \eqref{eqExtension-284} should be addressed.
It suffices however to observe that every element \(\alpha \in [\Sphere^\ell, \manfV]\) can be written as
\begin{equation}
    \label{eqExtension-305}
    \alpha = [\gamma]
    \quad \text{with \(\gamma \in \Fuglede_w(\Sphere^\ell; \manfV)\).}
\end{equation}
To justify \eqref{eqExtension-305}, note that by an approximation argument one can always write \(\alpha\) in this form using a Lipschitz continuous representative \(\gamma \colon  \Sphere^\ell \to \manfV\), although \(\gamma\) itself need not be a Fuglede map.
However, with the help of a transversal approximation of the identity that we implement in Section~\ref{sectionExtensionApproximation}, 
one can find a sequence \((\gamma_j)_{j \in \N}\) in \(\Fuglede_w(\Sphere^\ell; \manfV)\) that converges uniformly to \(\gamma\)\,, see assertion~\eqref{eqExtension-568}.
In particular, \(\gamma_j \sim \gamma\) in \(\Smooth^0(\Sphere^\ell; \manfV)\) for every \(j\) sufficiently large.
Thus, replacing \(\gamma\) by \(\gamma_j\) if necessary, we have \eqref{eqExtension-305}.

Next, the meaning of \([u \compose \gamma]\) in \eqref{eqExtension-290} is the following: The path-connected component of \(\VMO(\Sphere^{\ell}; \manfN)\) that contains \(u \compose \gamma\) also contains some \(f \in \Smooth^{0}(\Sphere^{\ell}; \manfN)\), see Proposition~\ref{propositionHomotopyVMOContinuousMap}, and any other \(\widetilde f \in \Smooth^{0}(\Sphere^{\ell}; \manfN)\) in this component is homotopic to \(f\) in \(\Smooth^{0}(\Sphere^{\ell}; \manfN)\), see Proposition~\ref{propositionHomotopyVMOtoC}.{}
We then set
\[{}
[u \compose \gamma] = [f],
\]
as the equivalence class that contains \(f\).

We postpone the proof of Theorem~\ref{proposition_Reference_Homotopy-Sphere} to Section~\ref{sectionLExtensionIndependence}.
The main tool is the approximation of Fuglede maps from Section~\ref{sectionExtensionApproximation}.
Observe that the converse of Theorem~\ref{proposition_Reference_Homotopy-Sphere} readily follows from the definition of \(\ell\)-extendability.
In fact, if \(u \in \VMO^{\ell}(\manfV ; \manfN)\) is not \(\ell\)-extendable, then for every \(\ell\)-detector \(w\) there exists a Lipschitz map \(\gamma \colon  \cBall^{\ell+1} \to  \manfV\) with 
\(\gamma|_{\Sphere^{\ell}} \in \Fuglede_{w}(\Sphere^{\ell}; \manfV)\) so that \(u \compose \gamma|_{\Sphere^{\ell}}\) is not homotopic to a constant in \(\VMO (\Sphere^{\ell}; \manfN)\).
We thus have
\[{}
\gamma|_{\Sphere^{\ell}} \sim \cte
\quad \text{in \(\Smooth^{0}(\Sphere^{\ell}; \manfV)\),} 
\] 
where the constant map \(\cte\) can be adjusted so that \(\cte \in \Fuglede_w(\Sphere^{\ell}; \manfV)\). 
However,
\[
u \compose \gamma|_{\Sphere^{\ell}}  \not\sim u \compose \cte
\quad \text{in \(\VMO (\Sphere^{\ell}; \manfN)\).}
\]

\section{Approximation of Fuglede maps}
\label{sectionExtensionApproximation}

An important example of \(\ell\)-extendability in the setting of Sobolev maps regardless of the topology of \(\manfN\) is the following:

\begin{proposition}
	\label{propositionLExtensionSobolev}
	If \(u \in \Sobolev^{k, p}(\manfV; \manfN)\) with \(kp \ge \ell + 1\), then \(u\) is \(\ell\)-extendable.
\end{proposition}

The proof relies on the approximation of Fuglede maps that we develop in this section.
The first step is to establish a robust setting where generic perturbations of arbitrary Lipschitz maps \(\gamma \colon  X \to \manfV\) yield legitimate Fuglede maps. 
When \(\manfV\) is a bounded open subset of the Euclidean space, such perturbations can be naturally achieved through the translation of \(\gamma\)\,:

\begin{example}
    Let \(\manfV = \Omega\) be an open subset of \(\R^m\), let \(w \colon  \Omega \to [0, +\infty]\) be a summable function and let \(\gamma \colon  X \to \Omega\) be a Lipschitz continuous map defined on a metric measure space \((X, d, \mu)\) with finite measure.
    If \(\delta > 0\) is such that \(d(\gamma(X), \partial\Omega) \ge \delta > 0\),
    then, for almost every \(\xi \in B_\delta^m\),
    \[
    \gamma + \xi \in \Fuglede_w(X; \Omega).
    \]
    Indeed, by Tonelli's theorem and a change of variables we have
    \[
    \begin{split}
    \int_{B_\delta^m} \biggl(  \int_X w(\gamma(x) + \xi) \dif\mu(x) \biggr) \dif \xi
    & =  \int_X \biggl(  \int_{B_\delta^m} w(\gamma(x) + \xi) \dif\xi \biggr) \dif\mu(x)\\
    & = \int_X \biggl(  \int_{\gamma(x) + B_\delta^m} w(\zeta) \dif\zeta \biggr) \dif\mu(x).
    \end{split}
    \]
    Thus,
    \[
    \int_{B_\delta^m} \biggl(  \int_X w(\gamma(x) + \xi) \dif\mu(x) \biggr) \dif \xi
    \le \int_X \biggl(  \int_\Omega w(\zeta) \dif\zeta \biggr) \dif\mu(x)
    = \mu(X) \int_\Omega w < \infty,
    \]
    which implies that, for almost every \(\xi \in B_\delta^m\), 
    \[
    \int_X w(\gamma(x) + \xi) \dif\mu(x) < \infty,
    \]
    whence \(\gamma + \xi\) is a Fuglede map.
\end{example}

The notion of transversal perturbation of the identity gathers essential properties of translations in a way that is adapted to general manifolds.

\begin{definition}
\label{definitionTransversalFamily}
Given an open set \(U \subset \manfV \times \R^{q}\) with \(q \ge m\) and \(\manfV \times \{0\} \subset U\), we say that a smooth map \(\tau \colon  U \to \manfV\) is a \emph{transversal perturbation of the identity} whenever
\begin{enumerate}[(a)]
\item \label{propertyTransversalLipschitz} 
for every \(z \in \manfV\), \(\tau(z, 0) = z\), and there exists \(\alpha>0\) such that
\[{}
d (\tau (z, \xi), z) \le \alpha\abs{\xi}
\quad{}
\text{for every \((z, \xi) \in U\),}
\]
\item \label{propertyTranversality} 
there exists \(\lambda > 0\) such that the \(m\)-dimensional Jacobian of \(\tau^z \vcentcolon= \tau (z, \cdot)\) satisfies 
\[{}
\Jacobian{m}{\tau^z}(\xi) \ge \lambda{}
\quad \text{for every \((z, \xi)\in U\),}
\]
\item \label{propertyBoundInverseMeasure} 
there exists \(\Lambda > 0\) such that, for each \(\delta>0\), 
\[{}
\cH^{q - m} \bigl(B^{q}_\delta \cap (\tau^z)^{-1} (\{y\})\bigr) \le \Lambda \delta^{q - m}
\quad \text{for every \(y, z \in \manfV\)}.
\]
\end{enumerate}
\end{definition}

Denoting the differential map of \(\tau^z\) at \(\xi\) by \(D\tau^z(\xi)\colon \R^q\to T_{\tau(z, \xi)}\manfV\), where \(T_{\tau(z, \xi)}\manfV\) is the tangent plane to \(\manfV\) at \(\tau(z, \xi)\), the \(m\)-dimensional Jacobian of \(\tau^z\) at \(\xi\) can be computed as
\begin{equation}
\label{eqLExtensionJacobian}
\Jacobian{m}{\tau^z} (\xi) = \bigl[\det{\bigl(D \tau^z (\xi) \compose D \tau^z (\xi)^* \bigr)}\bigr]^{\frac{1}{2}}.
\end{equation}
Definition~\ref{definitionTransversalFamily} is reminiscent of a notion by Hang and Lin~\cite{Hang-Lin} of parametrized families of mappings from a subdomain of a compact manifold with values in \(\manfV\), see \cite{Hang-Lin}*{p.~68, \((H_1)\)--\((H_3)\)}. 
The variable \(\xi\) in \(\tau(z, \xi)\) can be seen as a perturbation parameter, and the map
\[{}
\tau_{\xi} \vcentcolon= \tau(\cdot, \xi){}
\]
is a perturbation of the identity in \(\manfV\).

One of the goals of this section is to establish the following generic approximation property:

\begin{proposition}
\label{propositionFugledeApproximation}
Let \(\tau \colon  U \to \manfV\) be a transversal perturbation of the identity.
Given a summable function \(w \colon  \manfV \to [0, +\infty]\), there exists a summable function 
\(\widetilde{w} \ge w\) in \(\manfV\) with the following property:
For every metric measure space \((X, d, \mu)\) with finite measure, every \(\gamma \in \Fuglede_{\tilde{w}}(X ; \manfV)\) such that \(B_\delta^m(\gamma(X)) \times B_{\delta}^{q} \subset U\)
for some \(\delta > 0\), and every equiLipschitz sequence of maps \(\gamma_{j} \colon  X \to \manfV\) that converges uniformly to \(\gamma\) in \(X\),
there exist a subsequence \((\gamma_{j_{i}})_{i \in \N}\) and a sequence \((\xi_{i})_{i \in \N}\) in \(B_{\delta}^{q}\) that converges to \(0\)  such that 
\(\tau_{\xi_{i}} \compose \gamma_{j_{i}} \in \Fuglede_{w}(X ; \manfV)\) for every \(i \in \N\) and
\[{}
\tau_{\xi_i} \compose \gamma_{j_{i}}
\to \gamma \quad \text{in \(\Fuglede_{w}(X ; \manfV)\).}
\]
More generally, there exists a sequence of positive numbers \((\epsilon_{i})_{i \in \N}\) that converges to zero such that
\begin{equation}
\label{eqTransversalPerturbationLimit}
\lim_{i \to \infty}{\fint_{B_{\epsilon_{i}}^{q}}{\biggl( \int_{X}{\abs{w \compose \tau_{\xi} \compose \gamma_{j_{i}} - w \compose \gamma} \dif\mu} \biggr)} \dif \xi}
= 0.
\end{equation}
\end{proposition}

In the statement above, we denote by \(B_\delta^m(A)\) the neighborhood of radius \(\delta > 0\) of a subset \(A \subset \manfV\).
We first present ways to construct transversal perturbations of the identity in various cases:

\begin{example} 
\label{exampleTransversalPertubationEuclidean}
If \(\manfV = \Omega\) is an open subset of the Euclidean space \(\R^m\), we take
\[{}
\tau(z, \xi) = z + \xi
\] 
defined in the open set 
\(
U = \bigl\{(z, \xi) \in \Omega \times \R^m \st z + \xi \in \Omega\bigr\}
\).
Then, \(\tau\) is smooth and, for every \(z \in \Omega\), we have \(\Jacobian{m}{\tau^z}  = 1\).{}
Since 
\((\tau^{z})^{-1} (\{ y\}) = \{y - z\}\) for every \(y \in \Omega\), we also have \(\cH^{0}((\tau^{z})^{-1} (\{ y\})) = 1\).
\end{example}

\begin{example}
If \(\manfV=\manfM\) is a compact Riemannian manifold isometrically embedded into \(\R^{\varkappa}\), then there exists a smooth retraction \(\pi \colon  \overline{V} \to \manfM\), where \(V \subset \R^{\varkappa}\) is open, that is, \(\pi|_{\manfM}=\Id_{\manfM}\) and \(D\pi(x)\colon \R^\varkappa\to T_{\pi(x)}\manfM\) is surjective for every \(x\in \overline{V}\).  
The map
\[{}
\tau (z, \xi) = \pi (z + \xi)
\]
defined in the open set
\(
U = \{ (z, \xi) \in \manfM \times \R^{\varkappa} \st z + \xi \in V\}
\)
is a transversal perturbation of the identity, see \citelist{\cite{Hang-Lin}*{Remark 4.1}\cite{White-1986}*{\S 3, proof of Lemma}}.
\end{example}

\begin{example}
	If the manifold \(\manfM\) is a Lie group  --- for example \(\manfM = \Sphere^1 \simeq \SO (2)\) or \(\manfM = \Sphere^3 \simeq \SU (2)\) --- then we can take  \(\tau\) defined by the group action 
	\[{}
	\tau (z, \xi) = \xi z
	\] 
	in \(U = \manfM \times V\), where \(V \subset \R^m\) is, up to a diffeomorphism, a small neighborhood of the origin in \(\manfM\).
	When \(\manfM\) is a symmetric space like \(\Sphere^m\) or \(\RP^m\), there exists a Lie group \(G\) of isometries that acts transversally on \(\manfM\) and we can take in this case \(\tau (z, \xi)\) to be the action of the element \(\xi \in G\) on \(z\). 
\end{example}

Before tackling Proposition~\ref{propositionFugledeApproximation}, we begin with a property of transversal perturbations of the identity:

\begin{proposition}
\label{lemmaTransversalFamily}
Let \(\tau \colon  U \to \manfV\) be a transversal perturbation of the identity, let \((X, d, \mu)\) be a metric measure space  with finite measure and let \(w \colon  \manfV \to [0, +\infty]\) be a summable function.
If \(\gamma \colon  X \to \manfV\) is a Lipschitz map and \(\delta > 0\) is such that \(\gamma(X) \times B_{\delta}^{q} \subset U\), then
\[
\int_{{B}^q_\delta} \biggl( \int_{X} w \compose \tau_\xi \compose \gamma \dif\mu \biggr) \dif \xi
\le \frac{\Lambda}{\lambda} \delta^{q - m} \mu(X) \int_{\manfV} w.
\]
\end{proposition}

Since \(\mu(X) < \infty\), it follows from the above statement that  
\begin{equation}
    \label{eqExtension-568}
    \tau_\xi \compose \gamma \in \Fuglede_{w}(X; \manfV)
\quad \text{for almost every \(\xi \in {B}^q_\delta\)}
\end{equation}
and one can choose \(\widetilde{\xi} \in {B}^q_\delta\) such that
\[
\int_{X} w \compose \tau_{\tilde\xi} \compose \gamma \dif\mu
\le C \int_{\manfV} w,
\]
for some constant \(C > 0\) depending on \(\Lambda/\lambda\), \(\delta\), \(q\), \(m\) and \(\mu(X)\).

To prove Proposition~\ref{lemmaTransversalFamily}, we need the following fundamental property of transversal perturbations. 
In the case of translations in the Euclidean space, it is a straightforward consequence of a linear change of variables in the integral:
\[
  \int_{{B}^m_\delta} w (z + \xi) \dif \xi 
  = \int_{B^m_\delta (z)} w (y) \dif y.
\]

\begin{lemma}
\label{lemmaFundamentalTransversalFamily}
Let \(\tau \colon  U \to \manfV\) be a transversal perturbation of the identity and \(w \colon  \manfV \to [0, +\infty]\) be a summable function.
For every \(z \in \manfV\) and every \(\delta > 0\) such that \(\{z\}  \times {B}_\delta^q \subset U\), we have
\[
\int_{{B}^q_\delta} w \compose \tau^z 
\le \frac{\Lambda}{\lambda} \delta^{q - m} \int_{{B}_{\alpha\delta}^{m} (z)} w \text{,}
\]
where \(\alpha>0\) is the constant given by \eqref{propertyTransversalLipschitz} in Definition~\ref{definitionTransversalFamily} and \({B}_{\alpha\delta}^{m} (z) \subset \manfV\).
\end{lemma}

Since \(\cH^{m}({B}_{\alpha\delta}^{m} (z))\leq C' \delta^m\) for every \(\delta>0\), we can restate the estimate above in terms of average integrals as
\[{}
\fint_{B_{\delta}^{q}}{w \compose \tau^{z}}
\le C'' \fint_{{B}_{\alpha\delta}^{m} (z)} w,
\]
with a constant \(C'' > 0\) independent of \(\delta\).

\begin{proof}
[Proof of Lemma~\ref{lemmaFundamentalTransversalFamily}]
By the co-area formula, see \cite{Federer}*{Theorem~3.2.22},
\begin{equation}\label{eq846}
\int_{{B}^q_\delta}  (w \compose \tau^{z})\, \Jacobian{m}{\tau^z}
= \int_{\manfV} w (y) \cH^{q - m} \bigl({B}^q_\delta \cap (\tau^{z})^{-1} (\{y\})\bigr) \dif y.
\end{equation}
On the one hand, by \eqref{propertyTranversality} in Definition~\ref{definitionTransversalFamily}, we have 
\begin{equation}\label{eq844}
\lambda \int_{{B}^q_\delta} w \compose \tau^z
\le  \int_{{B}^q_\delta}  (w \compose \tau^{z})\, \Jacobian{m}{\tau^z}.
\end{equation}
On the other hand, we observe that if \({B}_\delta^q \cap (\tau^{z})^{-1} (\{y\}) \ne \emptyset\), then there exists \(\xi \in {B}^q_\delta\) such that 
\(y = \tau (z, \xi)\) and thus, by \eqref{propertyTransversalLipschitz} in Definition~\ref{definitionTransversalFamily}, we deduce that 
\[{}
d (z, y) = d (z, \tau (z, \xi)) \le  \alpha\abs{\xi} <  \alpha\delta.{}
\]
Hence, the integrand in the right-hand side of \eqref{eq846} vanishes in \(\manfV \setminus B_{\alpha\delta}^{m}(z)\).
In view of \eqref{propertyBoundInverseMeasure} in Definition~\ref{definitionTransversalFamily}, we then have
\begin{equation}
\label{eq417}
\int_{\manfV} w (y) \cH^{q - m} \bigl({B}^q_\delta \cap (\tau^{z})^{-1} (\{y\})\bigr) \dif y
\le \Lambda \delta^{q - m} \int_{B_{ \alpha\delta}^{m} (z)} w.
\end{equation}
The estimate then follows from \eqref{eq846}, \eqref{eq844} and \eqref{eq417}.
\end{proof}

\begin{proof}[Proof of Proposition~\ref{lemmaTransversalFamily}]
	Let \(w\) and \(\gamma\) be as in the statement.
By Tonelli's theorem, we have 
\[
\int_{{B}^q_\delta} \biggl( \int_{X} w \compose \tau_\xi \compose \gamma \dif\mu \biggr) \dif \xi
=  \int_{X} \biggl( \int_{{B}^q_\delta} w \compose \tau^{\gamma (x)}(\xi) \dif\xi \biggr) \dif \mu(x).
\]
Applying Lemma~\ref{lemmaFundamentalTransversalFamily} to estimate the integral between parentheses in the right-hand side, we get
\[
\int_{{B}^q_\delta} \biggl( \int_{X} w \compose \tau_\xi \compose \gamma \dif\mu \biggr) \dif \xi
\le \frac{\Lambda}{\lambda} \delta^{q - m}
\int_{X} \biggl(\int_{B_{\alpha\delta}^{m} (\gamma (x))} w(\zeta) \dif\zeta \biggr) \dif \mu(x).
\]
The conclusion follows since \(B_{\alpha\delta}^{m}(\gamma (x)) \subset \manfV\) and the latter set is independent of \(x\).
\end{proof}

We now turn to the proof of Proposition~\ref{propositionFugledeApproximation}, where we use the fact that, for any summable function \(v \colon  \manfV \to [0, +\infty]\), one has
\[{}
\lim_{s \to 0}
\int_{\manfV} \biggl(\fint_{B_{s}^{m}(z)} \abs{v - v(z)} \biggr) \dif z = 0.
\]
This is a consequence of the fact that \(\Smooth^{0}_{c}(\manfV)\) is dense in \(\Lebesgue^{1}(\manfV)\) and the measure  on \(\manfV\) satisfies the doubling property.

\resetconstant
\begin{proof}[Proof of Proposition~\ref{propositionFugledeApproximation}]
Let \((\gamma_{j})_{j \in \N}\) be an equiLipschitz sequence of maps that converges uniformly to a Lipschitz map \(\gamma \colon  X \to \manfV\) such that 
\begin{equation}
\label{eqFugledeApproximation-790}
B_{\delta}^{m}(\gamma(X)) \times B_{\delta}^{q} \subset U.
\end{equation}
Take a sequence of positive numbers \((s_{j})_{j \in \N}\) in \((0, \delta)\) that converges to \(0\).{}
Passing to a subsequence of \((\gamma_{j})_{j \in \N}\) if necessary, we may assume that, for every \(j \in \N\),
\begin{equation}
	\label{eqLipschitzClose}
\sup_{x \in X}{d(\gamma_{j}(x), \gamma(x))}
\le s_{j}.
\end{equation}
Since \(s_{j} < \delta\) and \eqref{eqFugledeApproximation-790} holds, for every \(j \in \N\) we have 
\[{}
\gamma_{j}(X) \times B^{q}_{s_{j}} \subset U.
\]

We next recall that
\[{}
\tau_{\xi} \circ \gamma_{j}(x) 
= \tau(\gamma_{j}(x), \xi){}
= \tau^{\gamma_{j}(x)}(\xi).
\]
For every \(x \in X\), we then have by Lemma~\ref{lemmaFundamentalTransversalFamily} applied to the summable function \(\abs{w - w \circ \gamma(x)}\), 
\begin{equation}
\label{eqFugledeApproximation-815}
\begin{split}
\fint_{{B}^{q}_{s_{j}}} \abs{w \circ \tau_{\xi} \circ \gamma_{j}(x) - w \circ \gamma (x)} \dif\xi
& = \fint_{{B}^{q}_{s_{j}}} \abs{w \circ \tau^{\gamma_{j}(x)} - w \circ \gamma (x)}\\
& \le \C
\fint_{B_{\alpha s_{j}}^{m}(\gamma_{j}(x))} \abs{w - w \circ \gamma (x)}.
\end{split}
\end{equation}
By \eqref{eqLipschitzClose}, for every \(x \in X\) we have
\begin{equation}
\label{eqFugledeApproximation-752}
B_{\alpha s_{j}}^{m}(\gamma_{j}(x)) \subset B_{(1+\alpha)s_{j}}^{m}(\gamma(x)).
\end{equation}
Applying again \eqref{eqLipschitzClose} and the doubling property on \(\manfV\),{}
\begin{equation}
\label{eqFugledeApproximation-757}
\abs{B^{m}_{(1 + \alpha)s_{j}}(\gamma(x))}
\le \abs{B^{m}_{(2 + \alpha)s_{j}}(\gamma_{j}(x))}
\le \C \abs{B^{m}_{\alpha s_{j}}(\gamma_{j}(x))}.
\end{equation}
Combining \eqref{eqFugledeApproximation-815}, \eqref{eqFugledeApproximation-752} and \eqref{eqFugledeApproximation-757}, we get
\[
\begin{split}
\fint_{{B}^{q}_{s_{j}}} \abs{w \circ \tau_{\xi} \circ \gamma_{j}(x) - w \circ \gamma (x)} \dif\xi
 &\le \Cl{cstxdpry} \fint_{B_{(1+\alpha)s_{j}}^{m} (\gamma(x))} \abs{w - w \circ \gamma (x)}\\
&= \Cr{cstxdpry} \, V_{j} \circ \gamma(x),
\end{split}
\]
where the constant \(\Cr{cstxdpry} > 0\) is independent of \(j\) and
\[{}
V_{j}(z) \vcentcolon= \fint_{B_{(1+\alpha)s_{j}}^{m}(z)} \abs{w - w(z)}.
\] 
We then integrate both sides of the estimate with respect to \(x \in X\).
By Tonelli's theorem,
\begin{equation}
\label{eqGenericApproximationFubini}
\fint_{{B}^q_{s_{j}}} \biggl( \int_{X} \abs{w \compose \tau_{\xi} \circ \gamma_{j} - w \compose \gamma} \dif\mu \biggr) \dif \xi
\le  \Cr{cstxdpry} \int_{X} V_{j} \circ \gamma \dif\mu.
\end{equation}

We now observe that 
\[{}
V_{j} \to 0 
\quad \text{in \(\Lebesgue^{1}(\manfV)\).}
\]
Then, by Proposition~\ref{lemmaModulusLebesguesequence} applied to the sequence \((V_{j})_{j \in \N}\)\,, there exist a summable function \(w_{1} \colon  \manfV \to [0, + \infty]\) and an increasing sequence of positive integers \((j_{i})_{i \in \N}\) such that, for every \(\gamma \in \Fuglede_{w_{1}}(X; \manfV)\),
\begin{equation}
\label{eqGenericApproximationConvergence}
V_{j_{i}} \circ \gamma \to 0
\quad \text{in \(\Lebesgue^{1}(X)\).}
\end{equation}
Let \(\widetilde{w} = w + w_{1}\).{}
By \eqref{eqGenericApproximationFubini} and \eqref{eqGenericApproximationConvergence}, we have for any \(\gamma \in \Fuglede_{\tilde{w}}(X; \manfV)\),
\[{}
\lim_{i \to \infty}{\fint_{{B}^q_{s_{j_{i}}}} \biggl( \int_{X} \abs{w \compose \tau_{\xi} \circ \gamma_{j_{i}} - w \compose \gamma} \dif\mu \biggr) \dif \xi{}}
= 0.
\]
In particular, taking \(\xi_{i} \in {B}^q_{s_{j_{i}}}\) such that
\[{}
\int_{X} \abs{w \compose \tau_{\xi_{i}} \circ \gamma_{j_{i}} - w \compose \gamma} \dif\mu
\le \fint_{{B}^q_{s_{j_{i}}}} \biggl( \int_{X} \abs{w \compose \tau_{\xi} \circ \gamma_{j_{i}} - w \compose \gamma}  \dif\mu \biggr) \dif \xi,
\]
we have
\[{}
\int_{X} \abs{w \compose \tau_{\xi_i} \compose \gamma_{j_{i}}} \dif\mu
\le \C{}
\quad \text{for every \(i \in \N\).}
\]
Since \((\gamma_{j_{i}})_{i \in \N}\) is equiLipschitz, \((\tau_{\xi_i} \compose \gamma_{j_{i}})_{i \in \N}\) is also equiLipschitz. 
As \((\tau_{\xi_i} \compose \gamma_{j_{i}})_{i \in \N}\) converges uniformly to \(\gamma\), we deduce that
\[{}
\tau_{\xi_i} \compose \gamma_{j_{i}} \to \gamma{}
\quad \text{in \(\Fuglede_{w}(X; \manfV)\).}
\qedhere
\]
\end{proof}

We now turn to the proof of Proposition~\ref{propositionLExtensionSobolev}.
We need the following property of traces of Sobolev maps that is reminiscent of the classical extension property for continuous maps:

\begin{lemma}
	\label{lemmaTraceVMO}
	If \(v \in \Sobolev^{1, p}(\Ball^{\ell + 1}; \manfN)\) with \(p \ge \ell + 1\), then its trace \(\Trace{v}\) belongs to \(\VMO(\Sphere^{\ell}; \manfN)\) and is homotopic to a constant in \(\VMO(\Sphere^{\ell}; \manfN)\).
\end{lemma}

\begin{proof}[Proof of Lemma~\ref{lemmaTraceVMO}]
	Since \( p \ge \ell + 1\), by Theorem~\ref{theoremSchoen-Uhlenbeck} there exists a sequence \((v_{j})_{j \in \N}\) in \(\Smooth^{\infty}(\cBall^{\ell+1}; \manfN)\) that converges to \(v\) in \(\Sobolev^{1, p}(\Ball^{\ell + 1}; \manfN)\).{}
	Then, by continuity of the trace operator, 
	\[{}
	\Trace{v_{j}} \to \Trace{v}
	\quad \text{in \(\Sobolev^{1 - \frac{1}{p}, p}(\Sphere^{\ell}; \manfN)\).}
	\]
	Since \((1 - 1/p)p = p - 1 \ge \ell\), we have by Example~\ref{exampleVMOSobolevFractional} that \(\Trace{v}\) and  \((\Trace{v_j})_{j \in \N}\) are contained in \(\VMO(\Sphere^{\ell}; \manfN)\) and
	\[{}
	\Trace{v_{j}} \to \Trace{v}
	\quad \text{in \(\VMO(\Sphere^{\ell}; \manfN)\).}
	\]
	By Proposition~\ref{propositionHomotopyVMOLimit}, for \(j\) large enough we then have
	\[{}
	\Trace{v_{j}} \sim \Trace{v}
	\quad \text{in \(\VMO(\Sphere^{\ell}; \manfN)\).}
	\]
	Since \(v_{j}\) is smooth in \(\cBall^{\ell+1}\), \(\Trace{v_{j}} = v_{j}|_{\Sphere^{\ell}}\) and we conclude that \(\Trace{v_{j}}\) is homotopic to a constant in \(\VMO(\Sphere^{\ell}; \manfN)\).{}
	The conclusion then follows from the transitivity of the homotopy relation.
\end{proof}

\begin{proof}[Proof of Proposition~\ref{propositionLExtensionSobolev}]
	Since \(\manfN\) is compact, by the Gagliardo-Nirenberg interpolation inequality we have \(u \in \Sobolev^{1, kp}(\manfV; \manfN)\).
	Let \(w\) be a summable function satisfying the conclusions of Propositions~\ref{corollaryCompositionSobolevFuglede} and~\ref{propositionSobolevApproximationFuglede} with order of integrability \(kp\).{}
	Let \(\tau \colon U  \to \manfV\) be a transversal perturbation of the identity, where  \(U\) is an open subset of \(\manfV \times \R^{q}\) with \(q \ge m\) and \(\manfV \times \{0\} \subset U\). 
    Let \(\widetilde{w} \ge w\) be a summable function provided by Proposition~\ref{propositionFugledeApproximation}.
	Given a Lipschitz map \(\gamma \colon  \cBall^{\ell+1} \to \manfV\), there exists \(\delta>0\) such that \(B^{m}_{\delta}\bigl(\gamma(\cBall^{\ell+1})\bigr) \times B^{q}_{\delta}\subset U\).
	By Proposition~\ref{lemmaTransversalFamily}, for almost every \(\xi \in B_{\delta}^{q}\) we have
	\begin{equation}
		\label{eqLExtension-1026}
		\tau_{\xi} \compose \gamma \in \Fuglede_{w}(\cBall^{\ell+1}; \manfV)
		\quad \text{and} \quad
		\tau_{\xi} \compose \gamma|_{\Sphere^{\ell}} \in \Fuglede_{w}(\Sphere^{\ell}; \manfV),
	\end{equation}
	which by Proposition~\ref{corollaryCompositionSobolevFuglede} imply that \(u \compose \tau_{\xi} \compose \gamma \in \Sobolev^{1, kp}(\Ball^{\ell + 1}; \manfN)\) and 
	\begin{equation}
		\label{eqLExtension-1033}
	\Trace{(u \compose \tau_{\xi} \compose \gamma )} = u \compose \tau_{\xi} \compose \gamma|_{\Sphere^{\ell}}
	\quad \text{almost everywhere in \(\Sphere^{\ell}\).}
	\end{equation}
	Assuming in addition that \(\gamma\) satisfies 
	\[{}
	\gamma|_{\Sphere^{\ell}} \in \Fuglede_{\tilde{w}}(\Sphere^{\ell}; \manfV),
	\]
	it follows from Proposition~\ref{propositionFugledeApproximation} that we may choose a sequence \((\xi_{i})_{i \in \N}\) in \(B_{\delta}^{q}\) with \(\xi_{i} \to 0\) such that each \(\xi_{i}\) satisfies \eqref{eqLExtension-1026} and
	\[{}
	\tau_{\xi_{i}} \compose \gamma|_{\Sphere^{\ell}} \to \gamma|_{\Sphere^{\ell}}
	\quad \text{in \(\Fuglede_{w}(\Sphere^{\ell}; \manfV)\).}
	\]
	By the choice of \(w\), it follows from Proposition~\ref{propositionSobolevApproximationFuglede} and Proposition~\ref{propositionExtensionVMOSphere} that
	\[{}
	u \compose \tau_{\xi_{i}} \compose \gamma|_{\Sphere^{\ell}} \to u \compose \gamma|_{\Sphere^{\ell}}
	\quad \text{in \(\VMO(\Sphere^{\ell}; \manfN)\).}
	\]
	By Proposition~\ref{propositionHomotopyVMOLimit}, for \(i\) large enough we then have
	\[{}
	u \compose \tau_{\xi_{i}} \compose \gamma|_{\Sphere^{\ell}} \sim u \compose \gamma|_{\Sphere^{\ell}}
	\quad \text{in \(\VMO(\Sphere^{\ell}; \manfN)\).}
	\]	
	Since \(kp \ge \ell + 1\), we may combine \eqref{eqLExtension-1033}, applied with \(\xi = \xi_{i}\), and Lemma~\ref{lemmaTraceVMO}, to deduce that \(u \compose \tau_{\xi_{i}} \compose \gamma|_{\Sphere^{\ell}}\) is homotopic to a constant in \(\VMO(\Sphere^{\ell}; \manfN)\).{}
	The conclusion then follows from the transitivity of the homotopy relation.
\end{proof}

The conclusion of Proposition~\ref{propositionLExtensionSobolev} is still true when \(u\in \VMO^{\ell+1}(\manfV; \manfN)\) as we shall see in a forthcoming chapter, see Corollary~\ref{propositionVMOlExtension}.

\section{\texorpdfstring{\(\ell\)}{l}-extension on simplicial complexes}

We have chosen to introduce the concept of \(\ell\)-extendability for Fuglede maps on the sphere \(\Sphere^\ell\) as it yields a comfortable setting to have explicit examples.
An alternative but equivalent approach that we develop in this section is to replace \(\Sphere^\ell\) by the boundary of a simplex \(\Simplex^{\ell + 1}\).
We then show that such an equivalence can be further developed using simplicial complexes.
We later return to this viewpoint in Chapter~\ref{chapterExtensionGeneral} to pursue the more general notion of \((\ell, e)\)-extendability.

We first prove that the definition of \(\ell\)-extendability could have been stated in the setting of simplices:

\begin{proposition}
	\label{propositionExtendabilitySimplex}
	Let \(\Simplex^{\ell + 1}\) be a simplex.
	A map \(u \in \VMO^{\ell}(\manfV; \manfN)\) is \(\ell\)-extendable if and only if there exists an \(\ell\)-detector \(w\) such that, for every Lipschitz map \(\gamma \colon  \Simplex^{\ell + 1} \to  \manfV\) with 
	\(\gamma|_{\partial\Simplex^{\ell + 1}} \in \Fuglede_{w}(\Simplex^{\ell + 1}; \manfV)\), the restricted map \(u \compose \gamma |_{\partial\Simplex^{\ell + 1}}\) is homotopic to a constant in \(\VMO(\partial\Simplex^{\ell + 1}; \manfN)\).
\end{proposition}

\begin{proof}
We assume that \(u\) is \(\ell\)-extendable and take an \(\ell\)-detector  \(w\) given by Definition~\ref{definitionExtensionVMO-1}. 
We take a biLipschitz homeomorphism \(\Psi \colon  \Simplex^{\ell + 1} \to \cBall^{\ell+1}\).
For every Lipschitz map \(\gamma \colon  \Simplex^{\ell + 1} \to \manfV\) with \(\gamma\vert_{\partial\Simplex^{\ell + 1}} \in \Fuglede_{w}(\partial\Simplex^{\ell + 1}; \manfV)\), the map \(\gamma \compose \Psi^{-1} \colon  \cBall^{\ell+1} \to \manfV\) is also Lipschitz continuous and, by a change of variables, \(\gamma \compose \Psi^{-1}\vert_{\Sphere^{\ell}} \in \Fuglede_{w}(\Sphere^{\ell}; \manfV)\).{}
By \(\ell\)-extendability of \(u\), the map \(u \compose \gamma \compose \Psi^{-1}|_{\Sphere^{\ell}}\) is homotopic to a constant in \(\VMO(\Sphere^{\ell}; \manfN)\).
	By composition with \(\Psi\vert_{\partial\Simplex^{\ell + 1}}\), it then follows from Lemma~\ref{lemmaDetectorsVMObiLipschitz} that \(u \compose \gamma \vert_{\partial\Simplex^{\ell + 1}}\) is homotopic to a constant in \(\VMO (\partial\Simplex^{\ell + 1}; \manfN)\).
	The proof of the reverse implication is similar.
\end{proof}

We now show a global aspect of \(\ell\)-extendability, where instead of being defined on the sphere or the boundary of a simplex, the Fuglede maps are taken on general polytopes.

\begin{proposition}
	\label{lemma_l_l1_property}
A map \(u \in \VMO^{\ell}(\manfV; \manfN)\) is \(\ell\)-extendable if and only if \(u\) has an \(\ell\)-detector \(w\) with the following property:
For every simplicial complex \(\cK^{\ell + 1}\) and every Lipschitz map \(\gamma \colon  K^{\ell + 1} \to  \manfV\) with 
\(\gamma|_{K^{\ell}} \in \Fuglede_{w}(K^\ell; \manfV)\), there exists \(F \in \Smooth^{0}(K^{\ell + 1}; \manfN)\) such that
\[{}
u \compose \gamma |_{K^{\ell}} \sim F|_{K^{\ell}}
\quad \text{in \(\VMO(K^{\ell}; \manfN)\).}
\]
\end{proposition}

We first need a straightforward observation concerning the restriction of a \(\VMO\) function in a polytope \(K^\ell\) to a subpolytope  \(L^\ell \subset K^\ell\) with the same dimension \(\ell\), which implies in particular that when two maps \(u, v \in \VMO(K^\ell;\manfN)\) are \(\VMO\)-homotopic, their restrictions \(u|_{L^\ell}\) and \(v|_{L^{\ell}}\) are also \(\VMO\)-homotopic.

\begin{lemma}
	\label{lemmaVMORestriction}
	Let \(\cK^{\ell}\) be a simplicial complex and let \(\cL^{\ell}\) be a subcomplex of \(\cK^{\ell}\). 
	If \(v \in \VMO(K^{\ell})\), then \(v\vert_{L^{\ell}} \in \VMO(L^{\ell})\) and the restriction map
	\[{}
	v \in \VMO(K^{\ell}) \longmapsto v\vert_{L^{\ell}} \in \VMO(L^{\ell})
	\]
	is continuous.
\end{lemma}

\begin{proof}[Proof of Lemma~\ref{lemmaVMORestriction}]
	There exists a constant \(C > 0\) such that, for every \(x \in L^{\ell}\) and every \(\rho > 0\),{}
	\[{}
	\cH^{\ell}(B_{\rho}^{\ell}(x))
	\le C \cH^{\ell}(B_{\rho}^{\ell}(x) \cap L^{\ell}).
	\]
	Thus, \(\seminorm{v|_{L^{\ell}}}_{\rho}	\le C^{2} \seminorm{v}_{\rho}\)	and the conclusion readily follows.
\end{proof}

We now prove that \(\ell\)-extendability at a local level propagates to the entire simplicial complex:

\begin{lemma}
\label{lemmaExtensionHomotopyConstant}
	Let \(\cK^{\ell + 1}\) be a simplicial complex and \(v \in \VMO(K^{\ell}; \manfN)\).{}
	If \(v|_{\partial\Sigma^{\ell + 1}}\) is homotopic to a constant in \(\VMO(\partial\Sigma^{\ell + 1}; \manfN)\) for every \(\Sigma^{\ell + 1} \in \cK^{\ell + 1}\), then there exists \(F \in \Smooth^{0}(K^{\ell + 1}; \manfN)\) such that
\[{}
v|_{K^{\ell}} \sim F|_{K^{\ell}}
\quad \text{in \(\VMO(K^{\ell}; \manfN)\).}
\]
\end{lemma}

Observe that, by Lemma~\ref{lemmaVMORestriction}, \(v|_{\partial\Sigma^{\ell + 1}} \in \VMO(\partial\Sigma^{\ell + 1}; \manfN)\).

\begin{proof}[Proof of Lemma~\ref{lemmaExtensionHomotopyConstant}]
Since \(v \in \VMO(K^{\ell}; \manfN)\), by Proposition~\ref{propositionHomotopyVMOContinuousMap} there exists a map \(f \in \Smooth^{0}(K^{\ell};\manfN)\) such that 
\begin{equation*}
v\vert_{K^{\ell}}
\sim f
\quad \text{in \(\VMO(K^{\ell};\manfN)\).}
\end{equation*}
We prove that \(f\) has a continuous extension to \(K^{\ell + 1}\).
To this end, we take a simplex \(\Sigma^{\ell+1} \in \cK^{\ell+1}\).{}
By assumption, \(v\vert_{\partial\Sigma^{\ell + 1}}\) is homotopic to a constant in \(\VMO(\partial\Simplex^{\ell + 1}; \manfN)\). 
Thus, by transitivity of the homotopy relation, \(f\vert_{\partial \Sigma^{\ell+1}}\) is homotopic to a constant in \(\VMO(\partial \Sigma^{\ell+1}; \manfN)\), and then, by  Proposition~\ref{propositionHomotopyVMOtoC}, also in \(\Smooth^{0}(\partial \Sigma^{\ell+1}; \manfN)\).
Hence, \(f\vert_{\partial \Sigma^{\ell+1}}\) has a continuous extension to \(\Sigma^{\ell+1}\) as a map with values in \(\manfN\). 
Since this is true for every \(\Sigma^{\ell+1} \in \cK^{\ell+1}\), 
we deduce that \(f\) has a continuous extension to the entire simplicial complex \(K^{\ell+1}\). 
\end{proof}

\begin{proof}[Proof of Proposition~\ref{lemma_l_l1_property}]
The reverse implication is a special case of Proposition~\ref{propositionExtendabilitySimplex} by taking the simplicial complex \(\cK^{\ell + 1}\) associated to a simplex \(\Simplex^{\ell + 1}\).
For the direct implication, let \(w\) be the \(\ell\)-detector from   Definition~\ref{definitionExtensionVMO-1}, let \(\cK^{\ell+1}\) be a simplicial complex and let \(\gamma \colon  K^{\ell+1} \to \manfV\) be a Lipschitz map such that \(\gamma\vert_{K^{\ell}} \in \Fuglede_{w}(K^{\ell}; \manfV)\). 
In particular, for each simplex \(\Sigma^{\ell+1} \in \cK^{\ell+1}\), as \(\partial\Sigma^{\ell + 1} \subset K^{\ell}\) we have \(\gamma\vert_{\partial \Sigma^{\ell+1}} \in \Fuglede_{w}(\partial \Sigma^{\ell+1}; \manfV)\).{}
We obtain that 
\(
u \compose \gamma\vert_{\partial\Sigma^{\ell + 1}}\) is homotopic to a constant in \(\VMO(\partial\Simplex^{\ell + 1}; \manfN)\). 
The conclusion then follows from Lemma~\ref{lemmaExtensionHomotopyConstant}.
\end{proof}

\section{Independence of homotopy classes}
\label{sectionLExtensionIndependence}

The goal of this section is to prove the following generalization of Theorem~\ref{proposition_Reference_Homotopy-Sphere} where, instead of considering homotopies over \(\Sphere^\ell\), we consider homotopies over \(\ell\)-dimensional polytopes:

\begin{theorem}
\label{proposition_Reference_Homotopy}
If \(u \in \VMO^{\ell}(\manfV; \manfN)\) is \(\ell\)-extendable, then there exists an \(\ell\)-detector \(\widetilde{w}\) such that, for every simplicial complex \(\cK^{\ell}\) and every Lipschitz maps \(\gamma_0, \gamma_1 \in \Fuglede_{\tilde{w}}(K^\ell ; \manfV)\) with 
\[{}
\gamma_{0} \sim \gamma_{1}
\quad \text{in \(\Smooth^{0}(K^{\ell}; \manfV)\),} 
\] 
we have
\[
u \compose \gamma_0  \sim u \compose \gamma_1{}
\quad \text{in \(\VMO (K^\ell; \manfN)\).}
\]
\end{theorem}

The counterpart of Theorem~\ref{proposition_Reference_Homotopy-Sphere} for Fuglede maps on \(\Sphere^{\ell}\) is obtained by first restricting the conclusion of Theorem~\ref{proposition_Reference_Homotopy} to Fuglede maps defined on the boundary of a simplex \(\Simplex^{\ell + 1}\) and then arguing by composition with a biLipschitz homeomorphism as in the proof of Proposition~\ref{propositionExtendabilitySimplex}. 

Let us begin with a classical observation that the existence of a continuous homotopy between two Lipschitz maps yields a Lipschitz homotopy between those maps.
	
\begin{lemma}
	\label{lemmaHomotopyLipschitz}
	Let \(X\) be a compact metric space. 
	If \(\gamma_{0}, \gamma_{1} \colon  X \to \manfV\) are Lipschitz continuous and
	\[{}
	\gamma_{0} \sim \gamma_{1}
	\quad \text{in \(\Smooth^{0}(X; \manfV)\),}
	\]
	then there exists a Lipschitz homotopy \(H \colon  X \times [0, 1] \to \manfV\) such that \(H(\cdot, 0) = \gamma_{0}\) and \(H(\cdot, 1) = \gamma_{1}\).
\end{lemma}	

\begin{proof}[Proof of Lemma~\ref{lemmaHomotopyLipschitz}]
	Let \(F \colon  X \times [0, 1] \to \manfV\) be a continuous homotopy with 
	\begin{equation}
	\label{eqGeneric-1263}
	F(\cdot, t) = \gamma_{0} 
	\quad \text{when \(t \le 1/3\)}
	\quad \text{and} \quad 
	F(\cdot, t) = \gamma_{1}
	\quad \text{when \(t \ge 2/3\).}
	\end{equation}
	We may assume that \(\manfV\) is isometrically imbedded in a Euclidean space \(\R^{\varkappa}\).{}
	For each \(j \in \N\), we then define \(H_{j} \colon  X \times [0, 1] \to \R^{\varkappa}\) for each component \(\kappa \in \{1, \dots, \varkappa\}\) using Young's transform:
	\[{}
	H_{j, \kappa}(x, t)
	\vcentcolon= \inf_{(y, s) \in X \times [0, 1]}{\Bigl\{  F_{\kappa}(y, s) + j \bigl(d(x, y) + \abs{t - s}\bigr)\Bigr\}}.
	\] 
	Then, each component \(H_{j, \kappa}\) is Lipschitz continuous with Lipschitz constant less than or equal to \(j\) and \(H_{j} \to F\) uniformly in \(X \times [0, 1]\).{}
	Moreover, by \eqref{eqGeneric-1263} we have
	\[{}
	H_{j}(\cdot, 0) = \gamma_{0} 
	\quad \text{and} \quad 
	H_{j}(\cdot, 1) = \gamma_{1} 	
	\]
	for every \(j \ge \max{\bigl\{\abs{\gamma_{0}}_{\Lip}, \abs{\gamma_{1}}_{\Lip}, 6 \norm{F}_{\Smooth^{0}}\bigr\}}\).
	Thus, it suffices to take \(H = \Pi \compose H_{j}\) for \(j\) sufficiently large, where \(\Pi\) denotes the nearest point projection to \(\manfV\).
\end{proof}

An important ingredient in the proof of Theorem~\ref{proposition_Reference_Homotopy-Sphere} is the possibility of perturbing a given Fuglede map to obtain another one that is Fuglede with respect to a larger skeleton.

\begin{lemma}
	\label{propositionGenericStability}
	Let \(u \in \VMO^{\ell}(\manfV; \manfN)\), let \(w\) be an \(\ell\)-detector for \(u\) and let \( \tau \colon U \to \manfV \) be a transversal perturbation of the identity.{}
	Then, there exists an \(\ell\)-detector \(\widetilde{w} \ge w\) in \(\manfV\) with the following property:
	For every \( \epsilon > 0 \), every simplicial complex \(\cK^{e}\), every subcomplex \(\cL^{\ell}\) and every Lipschitz map \(\gamma \colon  K^{e} \to \manfV\) with \(\gamma\vert_{L^{\ell}} \in \Fuglede_{\tilde{w}}(L^{\ell}; \manfV)\), there exists \( \xi \in B_{\epsilon}^{q} \) such that \( \gamma(K^{e}) \times \{\xi\} \subset U \), 
	\[{}
	\tau_{\xi} \compose \gamma\vert_{K^{i}} \in \Fuglede_{w}(K^{i}; \manfV)
	\quad \text{for every \(i \in \{0, \dots, e\}\)}
	\]
	and
	\[
	u \compose \tau_{\xi} \compose \gamma\vert_{L^{\ell}} {}
	\sim u \compose \gamma\vert_{L^{\ell}} {}
	\quad \text{in \(\VMO(L^{\ell}; \manfN)\).}
	\]
\end{lemma}

\resetconstant
\begin{proof}[Proof of Lemma~\ref{propositionGenericStability}]
	Let \(w_{1}\) be an \(\ell\)-detector for \(u\) from Definition~\ref{definitionVMOell}.
	We then take the \(\ell\)-detector \(\widetilde{w} \ge w + w_{1}\) given by Proposition~\ref{propositionFugledeApproximation}.
	Let  \(\gamma \colon  K^{e} \to \manfV\) be any Lipschitz map as in the statement.{}
	Since \(\gamma(K^{e}) \times \{0\} \subset U\), by compactness of \(\gamma(K^{e})\) there exists \(0<\delta <\varepsilon\) such that \( B_{\delta}^{m}(\gamma(K^{e})) \times B_{\delta}^{q} \subset U\).	
	Proposition~\ref{lemmaTransversalFamily} applied with \(X = K^{0}, \ldots, K^{e}\) gives
	\begin{equation}
	\label{eqExtension-532}
	\int_{{B}^q_\delta} \biggl( \sum_{i = 0}^{e} \int_{K^{i}} w \compose \tau_\xi \compose \gamma \dif\cH^{i} \biggr) \dif \xi
	\le \C \int_{\manfV} w.
	\end{equation}
	Since \(\gamma\vert_{L^{\ell}} \in \Fuglede_{\tilde{w}}(L^{\ell}; \manfV)\), we may apply Proposition~\ref{propositionFugledeApproximation} with the constant sequence \(\gamma_{j} = \gamma\) and \(X = L^{\ell}\) to find a sequence of positive numbers \((\epsilon_{j})_{j \in \N}\) in \((0, \epsilon)\) that converges to zero and such that
	\begin{equation}
	\label{eqExtension-540}
	\lim_{j \to \infty}{\fint_{B_{\epsilon_{j}}^{q}}{\biggl( \int_{L^{\ell}}{\abs{w \compose \tau_{\xi} \compose \gamma - w \compose \gamma}} \dif\cH^{\ell} \biggr)} \dif \xi}
	= 0.
	\end{equation}
	Combining \eqref{eqExtension-532} and \eqref{eqExtension-540}, one finds a sequence \((\xi_{j})_{j \in \N}\) in \(B_{\delta}^{q}\) that converges to \(0\) in \(\R^{q}\) such that
	\[{}
	\sum_{i = 0}^{e} \int_{K^{i}} w \compose \tau_{\xi_{j}} \compose \gamma \dif\cH^{i}
	< \infty 
	\quad \text{for every \(j \in \N\)}
	\]
	and
	\[{}
	\lim_{j \to \infty}{\int_{L^{\ell}}{\abs{w \compose \tau_{\xi_{j}} \compose \gamma - w \compose \gamma} \dif\cH^{\ell}}}
	= 0.
	\]
	Hence, the sequence \((\tau_{\xi_{j}} \compose \gamma)_{j \in \N}\) satisfies
	\[{}
	\tau_{\xi_{j}} \compose \gamma\vert_{K^{i}} \in \Fuglede_{w}(K^{i}; \manfV)
	\quad \text{for every \(i \in \{0, \dots, e\}\) and \(j \in \N\)}
	\]	
	and
	\[{}
	\tau_{\xi_{j}} \compose \gamma\vert_{L^{\ell}} \to \gamma\vert_{L^{\ell}}
	\quad \text{in \(\Fuglede_{w}(L^{\ell}; \manfV)\).}
	\]
	Since \(u\) is \(\VMO^{\ell}\) with \(\ell\)-detector \(w\), 
	we then have
	\[{}
	u \compose \tau_{\xi_{j}} \compose \gamma\vert_{L^{\ell}} \to u \compose \gamma\vert_{L^{\ell}}
	\quad \text{in \(\VMO(L^{\ell}; \manfN)\).}
	\]
	Hence, by Proposition~\ref{propositionHomotopyVMOLimit} 
	there exists \(J \in \N\) such that, for every \(j \ge J\),{}
	\[{}
	u \compose \tau_{\xi_{j}} \compose \gamma|_{L^{\ell}} \sim u \compose \gamma|_{L^{\ell}}
	\quad \text{in \(\VMO(L^{\ell}; \manfN)\).}
	\]
	It thus suffices to take \(\xi = \xi_{j}\) for any \(j \ge J\).
\end{proof}

\begin{proof}[Proof of Theorem~\ref{proposition_Reference_Homotopy}]
	Let \(w_{1}\) be an \(\ell\)-detector for \(u\) given by Proposition~\ref{lemma_l_l1_property} and then let \(\widetilde w \ge w_{1}\) be an \(\ell\)-detector for \(u\) given by Lemma~\ref{propositionGenericStability}.
	Given Lipschitz maps \(\gamma_0\) and \(\gamma_1\) as in the statement, by Lemma~\ref{lemmaHomotopyLipschitz} there exists a Lipschitz homotopy  \(H \colon  K^{\ell} \times [0, 1] \to \manfV\) such that \(H(\cdot, 0) = \gamma_{0}\) and \(H(\cdot, 1) = \gamma_{1}\).
	Take a simplicial complex \(\cE^{\ell + 1}\) and a subcomplex \(\cL^{\ell}\) with
	\[{}
	E^{\ell + 1}
	= K^{\ell} \times [0, 1]
	\quad \text{and} \quad 
	L^{\ell}
	= K^{\ell} \times \{0, 1\}.
	\]
	Since \(\gamma_{0}\) and \(\gamma_{1}\) are Fuglede maps with respect to \(\widetilde{w}\), we have \(H|_{L^{\ell}} \in \Fuglede_{\tilde w}(L^{\ell}; \manfV)\).{}
	Thus, by Lemma~\ref{propositionGenericStability} with \(e \vcentcolon= \ell + 1\) there exists a Lipschitz map \(\widetilde{H} \colon  E^{\ell + 1} \to \manfV\) of the form \( \widetilde{H} = \tau_{\xi} \compose H \) such that 
	\begin{equation}
		\label{eqVMOell+1}
	\widetilde{H}|_{E^{\ell}}
	\in \Fuglede_{w_{1}}(E^{\ell}; \manfV)
	\end{equation}
	and
	\begin{equation}
		\label{eqVMOell+1Homotopy}
	u \compose \widetilde{H}|_{L^{\ell}}
	\sim 
	u \compose H|_{L^{\ell}}
	\quad
	\text{in \(\VMO(L^{\ell}; \manfN)\).}
	\end{equation}	
	By \eqref{eqVMOell+1} and Proposition~\ref{lemma_l_l1_property}, there exists \(F \in \Smooth^{0}(E^{\ell + 1}; \manfN)\) such that
	\begin{equation}
	\label{eqLExtension-1034}
	u \compose \widetilde{H}|_{E^{\ell}}
	\sim F|_{E^{\ell}}
	\quad \text{in \(\VMO(E^{\ell}; \manfN)\).}
	\end{equation}
	By Lemma~\ref{lemmaVMORestriction}, we may restrict \eqref{eqLExtension-1034} to \(L^{\ell}\).{}
	Then, by \eqref{eqVMOell+1Homotopy} and transitivity of the homotopy relation, we get
	\[{}
	u \compose H|_{L^{\ell}}
	\sim F|_{L^{\ell}}
	\quad
	\text{in \(\VMO(L^{\ell}; \manfN)\).}
	\]
	By further restriction to \(K^{\ell} \times \{0\}\) and \(K^{\ell} \times \{1\}\), the latter can be restated as
	\[{}
	u \compose \gamma_{0}
	\sim 
	F|_{K^{\ell} \times \{0\}}
	\quad \text{and} \quad{}
	u \compose \gamma_{1}
	\sim 
	F|_{K^{\ell} \times \{1\}}
	\quad
	\text{in \(\VMO(K^{\ell}; \manfN)\).}
	\]	
	Since \(F \in \Smooth^{0}(K^{\ell} \times [0, 1]; \manfN)\), we also have
	\[{}
	F|_{K^{\ell} \times \{0\}}
	\sim 
	F|_{K^{\ell} \times \{1\}}
	\quad
	\text{in \(\VMO(K^{\ell}; \manfN)\).}
	\]	
	The conclusion follows from transitivity of the homotopy relation.
\end{proof}

It is possible to reformulate the conclusion of Theorem~\ref{proposition_Reference_Homotopy} based on the perturbation of Lipschitz maps as in White~\cite{White-1986}, without explicit reference to Fuglede maps:

\begin{corollary}
    \label{corollatyWhiteHomotopyType}
Let \(\tau \colon U \to \manfV\) be a transversal perturbation of the identity.
If \(u \in \VMO^{\ell}(\manfV; \manfN)\) is \(\ell\)-extendable, then \(u\) has an \emph{\(\ell\)-homotopy type} in the sense that, for every simplicial complex \(\cK^{\ell}\), every Lipschitz maps \(\gamma_0, \gamma_1 \colon K^\ell \to \manfV\) such that
\[
\gamma_0 \sim \gamma_1 \quad \text{in \(\Smooth^0(K^\ell; \manfV)\),}
\]
and for almost every \(\xi_0, \xi_1 \in \R^q\) in a neighborhood of \(0\), we have
\[
u \compose \tau_{\xi_0} \compose \gamma_0 \sim u \compose \tau_{\xi_1} \compose \gamma_1
\quad \text{in \(\VMO(K^\ell; \manfN)\).}
\]
\end{corollary}

\begin{proof}
    Let \(\widetilde{w}\) be the \(\ell\)-detector given by Theorem~\ref{proposition_Reference_Homotopy}.
    Take \(\delta > 0\) such that \(\gamma_i(K^\ell) \times B_\delta^q \subset U\) for \(i \in \{0, 1\}\).
    By \eqref{eqExtension-568}, for almost every \(\xi \in B_\delta^q\) we have \(\tau_\xi \compose \gamma_i \in \Fuglede_{\tilde w}(K^\ell; \manfV)\).
    Take \(\xi_i \in B_\delta^q\) satisfying this property.
    Note that 
    \[
     \tau_{\xi_i} \compose \gamma_i \sim \gamma_i
     \quad \text{in \(\Smooth^0(K^\ell; \manfV)\).}
    \]
    Then, since \(\gamma_0\) and \(\gamma_1\) are homotopic, by transitivity of the homotopy relation we get 
    \[
     \tau_{\xi_0} \compose \gamma_0 \sim \tau_{\xi_1} \compose \gamma_1
     \quad \text{in \(\Smooth^0(K^\ell; \manfV)\).}
    \]
    It then follows from Theorem~\ref{proposition_Reference_Homotopy} that
\[
u \compose \tau_{\xi_0} \compose \gamma_0 \sim u \compose \tau_{\xi_1} \compose \gamma_1
\quad \text{in \(\VMO(K^\ell; \manfN)\).}
\qedhere
\]
\end{proof}

We pursue this approach in Chapter~\ref{chapterExtendability} to identify \(\ell\)-extendable maps, based on almost every translation of Lipschitz maps in the case where \(\manfV = \Omega\) is an open subset of \(\R^m\).
We note that the converse of Corollary~\ref{corollatyWhiteHomotopyType} holds and follows along the lines of the proof of Proposition~\ref{proposition-Fuglede-translations} in the next chapter.
In Section~\ref{sectionHEP}, using Corollary~\ref{corollatyWhiteHomotopyType}, we revisit the notion introduced by White of \(\ell\)-homotopy between two Sobolev maps.
This concept has been used in the minimization of the energy functional
\[
E_{1, p} \colon u \in \Sobolev^{1, p}(\manfM;\manfN) \longmapsto \int_{\manfM} |Du|^p\dif x. 
\]
for \(p \ge \ell\) within equivalent classes defined by \(\ell\)-homotopy to find critical points of \(E_{1, p}\) in \(\Sobolev^{1, p}(\manfM;\manfN) \).

\cleardoublepage
\chapter{Characterizations of \texorpdfstring{$\ell$}{l}-extendability}
\label{chapterExtendability}

As discussed in the previous chapter, every map \(u \in \VMO^{\ell}(\manfV; \manfN)\) is \(\ell\)-extendable when \(\pi_{\ell}(\manfN) \simeq \{0\}\). In cases where this topological condition on \(\manfN\) does not hold, determining whether a given map \(u\) is generically \(\ell\)-extendable becomes significantly more challenging. Specifically, one must verify that \(u \circ \gamma\) is homotopic to a constant in \(\VMO(\Sphere^{\ell}; \manfN)\) for each Fuglede map \(\gamma \colon \Sphere^{\ell} \to \manfV\) associated with some \(\ell\)-detector \(w\).

To address this difficulty, this chapter provides equivalent formulations of \(\ell\)-extendability intended to simplify the verification process. When \(\manfV = \Omega\) is an open subset of \(\R^{m}\), we prove that \(\ell\)-extendability can be reduced to checking restrictions of \(u\) on the boundaries of simplices or embedded disks, with the notion of genericity adapted to each specific setting.

For Riemannian manifolds \(\manfV = \manfM\), we show that \(\ell\)-extendability for maps on the entire manifold \(\manfM\) can be deduced from \(\ell\)-extendability on local charts. 
This result highlights that verifying \(\ell\)-extendability locally is sufficient to establish the property for the entire manifold.

\section{Generic restrictions}
\label{sectionRestrictionsGeneric}

We begin by introducing the notions of disk and circle in \(\manfV\) using smooth imbeddings of the Euclidean ball and sphere in \(\R^{\ell + 1}\), respectively.

\begin{definition}
	Given any smooth imbedding \(\Phi\) from a neighborhood of \(\cBall^{\ell+1}\) into \(\manfV\), we define the set \(\Disk^{\ell + 1} = \Phi(\cBall^{\ell+1})\) as an \((\ell + 1)\)-dimensional \emph{disk} of \(\manfV\) and the set \(\partial \Disk^{\ell + 1} = \Phi(\Sphere^{\ell})\) as an \(\ell\)-dimensional \emph{circle}.
	We equip \(\partial \Disk^{\ell + 1}\) with the Hausdorff measure \(\cH^{\ell}\) induced by the distance on \(\manfV\).
\end{definition}

It is a simple observation in the spirit of Proposition~\ref{propositionExtensionVMOSphereDetector} that \(\ell\)-detectors are also suitable to identify \(\ell\)-dimensional circles where a given \(\VMO^{\ell}\) function is \(\VMO\).

\begin{proposition}
	Let \(u \in \VMO^{\ell}(\manfV)\).{}
	If \(w\) is an \(\ell\)-detector, then for any disk \(\Disk^{\ell + 1} \subset \manfV\) such that \(w|_{\partial\Disk^{\ell + 1}}\) is summable in \(\partial\Disk^{\ell + 1}\), we have
	\(u|_{\partial\Disk^{\ell + 1}} \in \VMO(\partial\Disk^{\ell + 1})\).
\end{proposition}

\begin{proof}
	Let \(\Phi \colon  \cBall^{\ell+1} \to \Disk^{\ell + 1}\) be a diffeomorphism.
	If \(w|_{\partial\Disk^{\ell + 1}}\) is summable in \(\partial\Disk^{\ell + 1}\), then by a smooth change of variables we have
	\(w \compose \Phi|_{\Sphere^{\ell}}\) summable in \(\Sphere^{\ell}\).{}
	Thus, \(\Phi|_{\Sphere^{\ell}} \in \Fuglede_{w}(\Sphere^{\ell}; \manfV)\) and then \(u \compose \Phi|_{\Sphere^{\ell}} \in \VMO(\Sphere^{\ell})\), see Proposition~\ref{propositionExtensionVMOSphereDetector}.
	By another smooth change of variables, we get \(u|_{\partial\Disk^{\ell + 1}} \in \VMO(\partial\Disk^{\ell + 1})\).
\end{proof}

We then have the following characterization of \(\ell\)-extendability based on generic restriction to \(\ell\)-dimensional circles:

\begin{proposition}
	\label{propositionCharacterizationDisks}
	Let \(u \in \VMO^{\ell}(\manfV; \manfN)\).{}
	Then, \(u\) is \(\ell\)-extendable if and only if there exists an \(\ell\)-detector \(\widetilde w\) such that, for every disk\/ \(\Disk^{\ell + 1} \subset \manfV\) for which \(\widetilde{w}|_{\partial\Disk^{\ell + 1}}\) is summable in \(\partial\Disk^{\ell + 1}\), the map \(u|_{\partial\Disk^{\ell + 1}}\) is homotopic to a constant in \(\VMO(\partial\Disk^{\ell + 1}; \manfN)\).
\end{proposition}

The proof of Proposition~\ref{propositionCharacterizationDisks} is presented in Section~\ref{sectionCharacterizationsDisks}.
Using local charts, we may restrict our attention to \(\VMO^{\ell}\) maps defined on an open set \(\Omega \subset \R^{m}\).{}
In this case, it is more convenient to manipulate restrictions to the boundary of generic simplices.

\begin{proposition}
	\label{propositionCharacterizationExtensionSimplex}
	Let \(\Omega \subset \R^{m}\) be an open subset and let \(u \in \VMO^{\ell}(\Omega; \manfN)\).{}
	Then, \(u\) is \(\ell\)-extendable if and only if there exists an \(\ell\)-detector \(w\) such that, for every simplex \(\Simplex^{\ell + 1} \subset \Omega\) for which \(w|_{\partial\Simplex^{\ell + 1}}\) is summable in \(\partial\Simplex^{\ell + 1}\), the map \(u|_{\partial\Simplex^{\ell + 1}}\) is homotopic to a constant in \(\VMO(\partial\Simplex^{\ell + 1}; \manfN)\).
\end{proposition}

The proof of Proposition~\ref{propositionCharacterizationExtensionSimplex} is postponed to Section~\ref{sectionCharacterizationsSimplices}.
Up to this point, we have relied on the notion of genericity based on detectors. 
An alternative approach in the spirit of Theorem~\ref{propositionSimplexLipschitz-NewBall} uses almost every translation of simplices in \(\R^{m}\). 
We justify that these two approaches yield the same result in the following sense:

\begin{proposition}
\label{proposition-Fuglede-translations}
	Let \(u \in \VMO^{\ell}(\Omega; \manfN)\) and let \(\tau_{\xi} \colon  \R^{m} \to \R^{m}\) be the translation by \(\xi \in \R^{m}\) defined by \(\tau_{\xi}(x) = x + \xi\).{}
	Then, the following properties are equivalent:
	\begin{enumerate}[\((i)\)]
		\item{}
		\label{itemCharacterization-96}
		 for every simplex \(\Simplex^{\ell + 1} \subset \Omega\) and for almost every \(\xi \in \R^{m}\) such that \(\abs{\xi} < d(\Simplex^{\ell + 1}, \partial\Omega)\), the map \(u \compose \tau_{\xi}|_{\partial\Simplex^{\ell + 1}}\) is homotopic to a constant in \(\VMO(\partial\Simplex^{\ell + 1}; \manfN)\),{}
		\item there exists an \(\ell\)-detector \(\widetilde w\) such that, for every simplex \(\Simplex^{\ell + 1} \subset \Omega\) for which \(\widetilde{w}|_{\partial\Simplex^{\ell + 1}}\) is summable in \(\partial\Simplex^{\ell + 1}\), the map \(u|_{\partial\Simplex^{\ell + 1}}\) is homotopic to a constant in \(\VMO(\partial\Simplex^{\ell + 1}; \manfN)\).
	\end{enumerate}
\end{proposition}

\resetconstant
\begin{proof}
	``\((ii) \Rightarrow (i)\)''.{}
	Given a simplex \(\Simplex^{\ell + 1} \subset \Omega\) and \(0 < \delta < d(\Simplex^{\ell + 1}, \partial\Omega)\), by Proposition~\ref{lemmaTransversalFamily} we have
	\[
	\int_{{B}^m_\delta} \biggl( \int_{\partial\Simplex^{\ell + 1}} \widetilde w \compose \tau_\xi \dif\cH^{\ell} \biggr) \dif \xi
	\le \C \int_{\Omega} \widetilde w.
	\]
	In particular, for almost every \(\xi \in B_{\delta}^{m}\), we see that \(\widetilde w \compose \tau_\xi|_{\partial\Simplex^{\ell + 1}}\) is summable in \(\partial\Simplex^{\ell + 1}\) and then, by assumption, \(u \compose \tau_{\xi}|_{\partial\Simplex^{\ell + 1}}\) is homotopic to a constant in \(\VMO(\partial\Simplex^{\ell + 1}; \manfN)\).{}
	
	``\((i) \Rightarrow (ii)\)''.{}
	Let \(w\) be an \(\ell\)-detector given by Definition~\ref{definitionVMOell} of \(\VMO^{\ell}\) and let \(\widetilde w \ge w\) be an \(\ell\)-detector given by Proposition~\ref{propositionFugledeApproximation}.
	Given a simplex \(\Simplex^{\ell + 1} \subset \Omega\) for which \(\widetilde w|_{\partial\Simplex^{\ell + 1}}\) is summable in \(\partial\Simplex^{\ell + 1}\), we apply Proposition~\ref{propositionFugledeApproximation} to the identity map \(\gamma \vcentcolon= \Id\) and the constant sequence \(\gamma_{j} \vcentcolon= \Id\).{}
	Then, there exists a sequence of positive numbers \((\epsilon_{i})_{i \in \N}\) converging to zero such that
	\[{}
	\lim_{i \to \infty}{\fint_{B_{\epsilon_{i}}^{m}}{\biggl( \int_{\partial\Simplex^{\ell + 1}}{\abs{w \compose \tau_{\xi} - w} \dif\cH^{\ell}} \biggr)} \dif \xi}
	= 0.
	\]
	We may thus take a sequence \(\xi_{i} \to 0\) in \(\R^{m}\) such that \((\tau_{\xi_{i}}|_{\partial\Simplex^{\ell + 1}})_{i \in \N}\) is contained in \(\Fuglede_{w}(\partial\Simplex^{\ell + 1}; \Omega)\), 
	\begin{equation}
    \label{eqCharacterization-95}
    \tau_{\xi_{i}}|_{\partial\Simplex^{\ell + 1}} \to \Id|_{\partial\Simplex^{\ell + 1}}
	\quad \text{in \(\Fuglede_{w}(\partial\Simplex^{\ell + 1}; \Omega)\)}
	\end{equation}
	and, by Assumption~\((\ref{itemCharacterization-96})\), \(u \compose \tau_{\xi_{i}}|_{\partial\Simplex^{\ell + 1}}\) is homotopic to a constant in \(\VMO(\partial\Simplex^{\ell + 1}; \manfN)\) for every \(i \in \N\).
	By \eqref{eqCharacterization-95} and the generic \(\VMO\) stability of \(u\),{}
	\[{}
	u \compose \tau_{\xi_{i}}|_{\partial\Simplex^{\ell + 1}} \to u|_{\partial\Simplex^{\ell + 1}}
	\quad \text{in \(\VMO(\partial\Simplex^{\ell + 1}; \manfN)\).}
	\]
	Then, by Proposition~\ref{propositionHomotopyVMOLimit}, there exists \(J \in \N\) such that, for every \(i \ge J\),
	\[{}
	u \compose \tau_{\xi_{i}}|_{\partial\Simplex^{\ell + 1}} \sim u|_{\partial\Simplex^{\ell + 1}}
	\quad \text{in \(\VMO(\partial\Simplex^{\ell + 1}; \manfN)\).}
	\]
	Since \(u \compose \tau_{\xi_{i}}|_{\partial\Simplex^{\ell + 1}}\) is homotopic to a constant in \(\VMO(\partial\Simplex^{\ell + 1}; \manfN)\), we then conclude by transitivity of the homotopy relation.
\end{proof}

\begin{proof}[Proof of Theorem~\ref{propositionSimplexLipschitz-NewBall}]
    Let \(u \in \Sobolev^{k, p}(\Ball^{m}; \manfN)\).
    Since \(\manfN\) is compact, \(u\) is bounded.
    It then follows from the inclusion \eqref{eqDetector-882} applied to each component of \(u\) that \(u \in \VMO^{\ell}(\Ball^m; \manfN)\) with \(\ell = \floor{kp}\).
    Theorem~\ref{propositionSimplexLipschitz-NewBall} then follows by combining Proposition~\ref{propositionCharacterizationExtensionSimplex} and~\ref{proposition-Fuglede-translations}, where we take \(\Omega = \Ball^m\).
\end{proof}

\section{Testing on simplices}
\label{sectionCharacterizationsSimplices}

In this section, we prove Proposition~\ref{propositionCharacterizationExtensionSimplex} by approximating Lipschitz maps defined on a simplex \(\Simplex^{\ell + 1}\) with maps that are affine and injective on each element of a simplicial complex decomposition \(\cK^{\ell+1}\) of \(\Simplex^{\ell + 1}\). Such maps are uniquely determined by linearity once their values at the vertices of \(\cK^{\ell+1}\) are specified. 
More precisely, a map \(\gamma \colon \Simplex^{\ell + 1} \to \R^m\) is \emph{simplicial} whenever its restriction \(\gamma|_{\Sigma^{\ell+1}}\) to any simplex \(\Sigma^{\ell + 1}\in \cK^{\ell+1}\)  is affine, that is, denoting by \(v_0, \ldots, v_{\ell + 1}\) the vertices of \(\Sigma^{\ell + 1}\),  one has
\[
\gamma(t_0 v_0 + \cdots + t_{\ell + 1} v_{\ell + 1}) = 
t_0 \gamma(v_0) + \cdots + t_{\ell + 1} \gamma(v_{\ell + 1}),
\]
for all nonnegative numbers such that \(t_0 + \cdots + t_{\ell + 1} = 1\).

\begin{proposition}
\label{lemmaApproximationSimplexLipschitz}
Let \(\ell + 1 \le m\) and let \(\Simplex^{\ell + 1}\) be a simplex.
Given a summable function \(w \colon  \Omega \to [0, +\infty]\), there exists a summable function \(\widetilde w \ge w\) in \(\Omega\) with the following property:
For every Lipschitz map \(\gamma \colon  \Simplex^{\ell + 1} \to \Omega\) with \(\gamma|_{\partial\Simplex^{\ell + 1}} \in \Fuglede_{\tilde w}(\partial\Simplex^{\ell + 1}; \Omega)\), there exist a sequence of simplicial complexes \((\cK_{j}^{\ell+1})_{j \in \N}\) with \(K_{j}^{\ell + 1} = \Simplex^{\ell + 1}\) and a sequence of  Lipschitz maps \(\gamma_{j} \colon  \Simplex^{\ell + 1} \to \Omega\) with \(\gamma_{j}|_{K_{j}^{\ell}} \in \Fuglede_{w}(K_{j}^{\ell}; \Omega)\) such that
\begin{enumerate}[\((i)\)]
	\item
	\label{itemCharacterization-140}
	 \(\gamma_{j}|_{\Sigma^{\ell + 1}}\) is affine and injective for every \(\Sigma^{\ell + 1} \in \cK_{j}^{\ell + 1}\), 
	\item{}
	\label{itemCharacterization-143}
	 \(\gamma_{j}|_{\partial\Simplex^{\ell + 1}}  \to \gamma|_{\partial\Simplex^{\ell + 1}} \) in \(\Fuglede_{w}(\partial\Simplex^{\ell + 1}; \Omega)\).
	\end{enumerate} 
\end{proposition}

The proof of Proposition~\ref{lemmaApproximationSimplexLipschitz} relies on the existence of subdivisions of \(\Simplex^{\ell + 1}\) with arbitrarily small diameters. To be more specific, we need to introduce the following definitions:
Given a simplex \(\Sigma^{\ell + 1} \subset \R^m\), the \emph{thickness} of \(\Sigma^{\ell + 1}\) is defined as the ratio
\[
\tau
= \frac{r}{\Diam \Sigma^{\ell + 1}},
\] 
where \(r > 0\) is the \emph{radius} of \(\Sigma^{\ell + 1}\) given by the distance from the barycenter of \(\Sigma^{\ell + 1}\) to \(\partial \Sigma^{\ell + 1}\). 
An estimate of the thickness of a simplex \(\Sigma^{\ell + 1}\) entitles one to control the Lipschitz constant \(\abs{\gamma}_{\Lip}\) of an affine map \(\gamma\colon \Sigma^{\ell + 1}\to \R^m\) in terms of its values at the vertices of \(\Sigma^{\ell + 1}\)\,: 

\begin{lemma}
\label{lemmaApproximationSimplexLipschitz-0}
Given a simplex \(\Sigma^{\ell + 1}\) with thickness \(\tau > 0\), the Lipschitz constant of any affine map \(\gamma\colon \Sigma^{\ell + 1}\to \R^m\) satisfies 
\begin{equation}
\label{eqCharacterization-421}
\abs{\gamma}_{\Lip} 
\le \frac{C}{\tau^{\ell + 1}}
    \max{ \biggl\{ \frac{\abs{\gamma(v_{1}) - \gamma(v_{2})}}{\abs{v_{1} - v_{2}}} : v_1 \ne v_2\ \text{are vertices of \(\Sigma^{\ell + 1}\)} \biggr\}}\text{,}
\end{equation}
for some constant \(C>0\) depending only on \(\ell\). 
\end{lemma}
\resetconstant
\begin{proof}[Proof of Lemma~\ref{lemmaApproximationSimplexLipschitz-0}]
Without loss of generality, one can assume that \(\Sigma^{\ell + 1} \subset \R^{\ell + 1}\). 
Let \(v_0, \ldots, v_{\ell + 1}\) be the vertices of \(\Sigma^{\ell + 1}\).  
     Take two points \(x, y \in \Sigma^{\ell + 1}\), which we write \(x = \sum\limits_{i = 0}^{\ell + 1}{t_i v_i}\) and \(y = \sum\limits_{i = 0}^{\ell + 1}{s_i v_i}\), with nonnegative coefficients whose sums equal \(1\).
    As \(t_0 = 1 - \sum\limits_{i = 1}^{\ell + 1}{t_i}\), we may then rewrite \(x\) as 
    \[
    x = \sum\limits_{i = 1}^{\ell + 1}{t_i (v_i - v_0)} + v_0.
    \]
    Since \(\gamma\) is affine, we also have
    \[
    \gamma(x) = \sum\limits_{i = 1}^{\ell + 1}{t_i \bigl(\gamma(v_i) - \gamma(v_0)\bigr)} + \gamma(v_0).
    \]
    We also write \(y\) and \(\gamma(y)\) accordingly.
    By the triangle inequality, we then have
    \[
    \abs{\gamma(x) - \gamma(y)}
    = \biggabs{\sum\limits_{i = 1}^{\ell + 1}{(t_i - s_i) (\gamma(v_i) - \gamma(v_0))}}
    \le \sum\limits_{i = 1}^{\ell + 1}{\abs{t_i - s_i} \abs{\gamma(v_i) - \gamma(v_0)}}.
    \]
    Denoting by \(A\) the maximum in \eqref{eqCharacterization-421}, we then have
    \begin{equation}
    \label{eqCharacterization-460}
    \abs{\gamma(x) - \gamma(y)}
    \le A \sum\limits_{i = 1}^{\ell + 1}{\abs{t_i - s_i} \abs{v_i - v_0}}.
    \end{equation}
    Let \((e_i)_{i \in \{1, \ldots, \ell+1\}}\) be the canonical basis of \(\R^{\ell + 1}\).
    By the Cauchy-Schwarz inequality,
    \begin{equation}
    \label{eqCharacterization-467}
    \sum\limits_{i = 1}^{\ell + 1}{\abs{t_i - s_i} \bigabs{v_i - v_0}}
    \le \sqrt{\ell + 1} \: \biggabs{\sum_{i=1}^{\ell + 1} (t_i - s_i) \abs{v_i - v_0} e_i}
    \end{equation}
We now take the linear transformation \(L \colon \R^{\ell + 1} \to \R^{\ell + 1}\) such that, for every \(i \in \{1, \ldots, \ell + 1\}\),
\[
L(e_i) = \frac{v_i-v_0}{\abs{v_i-v_0}}.
\]
Since \((v_i-v_0)_{i \in \{1, \ldots, \ell + 1\}}\) is a basis of \(\R^{\ell + 1}\), \(L\) is invertible and we have \(\abs{v_i-v_0} e_i = L^{-1}(v_i - v_0)\).
It then follows from \eqref{eqCharacterization-460}, \eqref{eqCharacterization-467} and the linearity of \(L^{-1}\) that
\begin{equation}
    \label{eqCharacterization-465}
    \abs{\gamma(x) - \gamma(y)}
    \le A\sqrt{\ell + 1} \: \abs{L^{-1}(x - y)}
    \le A\sqrt{\ell + 1} \:  \norm{L^{-1}}\abs{x - y}.
\end{equation}

We now estimate \(\norm{L^{-1}}\).
To this end, denote by \(\Matrix{L}\) and \(\Matrix{L^{-1}}\) the \((\ell + 1) \times (\ell + 1)\) matrices associated to \(L\) and \(L^{-1}\) in the canonical basis of \(\R^{\ell + 1}\), respectively.
We recall that 
\begin{equation}
    \label{eqCharacterization-488}
    \Matrix{L^{-1}} = \frac{1}{\det{\Matrix{L}}} D,
\end{equation}
where \(D\) is the cofactor matrix of \(\Matrix{L}\).
Let \(r > 0\) be the radius of \(\Sigma^{\ell + 1}\).
Since \(\Sigma^{\ell + 1}\) contains the ball of radius \(r\) centered at the barycenter, we deduce that its volume is at least \({\Cl{volball}} r^{\ell + 1}\) and then
\begin{equation}
    \label{eqCharacterization-495}
\begin{split}
{\Cr{volball}} r^{\ell + 1}
\leq \det{(v_1-v_0, \dots, v_{\ell + 1} - v_0)} 
& = \det{\Matrix{L}} \, \prod_{i=1}^{\ell + 1}\abs{v_j-v_0}\\
& \leq \det{\Matrix{L}} \, (\Diam{\Sigma^{\ell + 1}})^{\ell + 1}.
\end{split}
\end{equation}
Since each column of \(\Matrix{L}\) is obtained from the coordinates of a unit vector in the canonical basis,  
we have \(\norm{D}\leq \Cl{norm479}\).
Hence, by \eqref{eqCharacterization-488} and \eqref{eqCharacterization-495}, we get
\[
\norm{{L^{-1}}}
= \norm{\Matrix{L^{-1}}}
\le \Cl{cte-489} \frac{(\Diam{\Sigma^{\ell + 1}})^{\ell + 1}}{r^{\ell + 1}}
= \frac{\Cr{cte-489}}{\tau^{\ell + 1}},
\]
which combined with \eqref{eqCharacterization-465} gives the estimate in the statement.
\end{proof}

The core of the proof of Proposition~\ref{lemmaApproximationSimplexLipschitz} is contained in the following lemma presenting the approximation of a Lipschitz map by a sequence of simplicial maps. 
The proof is standard and relies on the decomposition of \(\Simplex^{\ell+1}\) by simplices having arbitrarily small diameters but thicknesses bounded from below by a positive constant, see \cite{Munkres}*{Lemma~9.4}.

\begin{lemma}
\label{lemmaApproximationSimplexLipschitz-1}
Let \(\Simplex^{\ell + 1}\) be a simplex.
For every Lipschitz map \(\gamma \colon  \Simplex^{\ell + 1} \to \R^m\), there exist a sequence of simplicial complexes \((\cK_{j}^{\ell+1})_{j \in \N}\) with \(K_{j}^{\ell + 1} = \Simplex^{\ell + 1}\) and an equiLipschitz sequence of simplicial maps \(\gamma_{j} \colon  \Simplex^{\ell + 1} \to \R^m\)  such that
\begin{enumerate}[\((i)\)]
	\item
	 \(\gamma_{j}|_{\Sigma^{\ell + 1}}\) is affine for every \(\Sigma^{\ell + 1} \in \cK_{j}^{\ell + 1}\), 
	\item{}
	 \(\gamma_{j}|_{\Simplex^{\ell + 1}}  \to \gamma|_{\Simplex^{\ell + 1}} \) in \(\Smooth^0(\Simplex^{\ell + 1}; \R^m)\).
	\end{enumerate} 
\end{lemma}

\resetconstant

\begin{proof}[Proof of Lemma~\ref{lemmaApproximationSimplexLipschitz-1}]
By \cite{Munkres}*{Lemma~9.4}, there exist \(\eta > 0\) and a sequence of simplicial complexes \((\cK_{j}^{\ell+1})_{j \in \N}\) with \(K_{j}^{\ell + 1} = \Simplex^{\ell + 1}\) such that, for each simplex \(\Sigma^{\ell+1}\in \cK_j^{\ell+1}\), we have \(\Diam{\Sigma^{\ell+1}} \leq 1/2^{j}\) and the thickness of \(\Sigma^{\ell+1}\) is bounded from below by \(\eta\). 
Let \(\gamma_j\) be the secant map associated to \(\gamma\) on this subdivision, that is, \(\gamma_j|_{\Sigma^{\ell+1}}\) is affine for every \(\Sigma^{\ell+1}\in \cK_{j}^{\ell+1}\) and \(\gamma_j(v)=\gamma(v)\) for every \(v\in \cK_{j}^{0}\). 

Since \(\Simplex^{\ell+1}\) is convex, one has 
\[
\abs{\gamma_j}_{\Lip(\Simplex^{\ell + 1})}
\le \max_{\Sigma^{\ell+1}\in \cK_{j}^{\ell+1}}\abs{\gamma_j}_{\Lip(\Sigma^{\ell+1})}.
\]
Since \(\gamma\) is Lipschitz continuous and \(\gamma_j\) coincides with \(\gamma\) on the vertices of each simplex \(\Sigma^{\ell + 1} \in \cK_{j}^{\ell+1}\), by Lemma~\ref{lemmaApproximationSimplexLipschitz-0} we have
\begin{equation}\label{eq154}
\abs{\gamma_j}_{\Lip(\Sigma^{\ell+1})}
\leq \frac{\C}{\eta^{\ell+1}}\abs{\gamma}_{\Lip(\Simplex^{\ell + 1})}.
\end{equation}
This shows that the sequence \((\gamma_j)_{j\in \N}\) is equiLipschitz. 
To obtain the uniform convergence in \(\Simplex^{\ell + 1}\), it suffices to observe that if \(x \in \Sigma^{\ell + 1} \in \cK_{j}^{\ell+1}\) and \(v\) is a vertex of \(\Sigma^{\ell + 1}\), we have
\[
\abs{\gamma_j(x) - \gamma(x)}
\le \abs{\gamma_j(x) - \gamma_j(v)}
+ \abs{\gamma(v) - \gamma(x)}
\le \Cl{cteCharacterization-548} \abs{x - v}
\le \Cr{cteCharacterization-548} \Diam{\Sigma^{\ell + 1}}
\le  \frac{\Cr{cteCharacterization-548}}{2^{j}}.
\]
The conclusion follows.
\end{proof}

\resetconstant
\begin{lemma}
\label{lemmaApproximationSimplexLipschitz-2}
If \(\ell + 1 \le m\), then in the conclusion of Lemma~\ref{lemmaApproximationSimplexLipschitz-1} we can further require that, for every \(j\in \N\) and every \(\Sigma^{\ell+1}\in \cK_{j}^{\ell+1}\), the restriction \(\gamma_j|_{\Sigma^{\ell+1}}\) is injective.
\end{lemma}
\begin{proof}[Proof of Lemma~\ref{lemmaApproximationSimplexLipschitz-2}]
We modify the secant sequence \(\gamma_j\) given by Lemma~\ref{lemmaApproximationSimplexLipschitz-1} as follows. 
Fix \(j\in \N\) and let \(J \in \N\) be the cardinality of \(\cK_{j}^{0}\). 
We enumerate the vertices of \(\cK_{j}^0\) denoting them by \(v_1, \ldots, v_J\).
Let \(e_{1}, \ldots, e_J \in \R^m\) be such that  any  \(\ell+1\) vectors \(e_{i_1}, \dots, e_{i_{\ell+1}}\) are linearly independent and the affine space generated by \(e_{i_1}, \dots, e_{i_{\ell+1}}\) does not contain any other vector \(e_{i_0}\) with \(i_{0}\not\in \{i_1, \dots, i_{\ell+1}\}\). Such a family exists because \(\ell+1\leq m\). 
In particular, the family \((e_{i_1}-e_{i_0}, \dots, e_{i_{\ell+1}}-e_{i_0})\) is free. 
We now take \(\varepsilon_{\ell+2}, \dots, \varepsilon_{m}\in \R^m\) such that \((e_{i_1}-e_{i_0}, \dots, e_{i_{\ell+1}}-e_{i_0}, \varepsilon_{\ell+2}, \dots, \varepsilon_{m})\) is a basis  of \(\R^m\). It follows that the function 
\begin{multline*}
t \in \R \longmapsto \Det{\big(\gamma(v_{i_1})-\gamma(v_{i_0})+t(e_{i_1}-e_{i_0}), \dots,}
\\  
{\gamma(v_{i_{\ell+1}})-\gamma(v_{i_0})+t(e_{i_{\ell+1}}-e_{i_0}), \varepsilon_{\ell+2}, \dots, \varepsilon_{m}\big)}
\end{multline*}
is a polynomial of degree \(\ell+1\). 
Hence, there exists \(0 < t_j < 1/2^{j}\) that is not a root of such a polynomial for any choice of integers \(i_0 < \ldots < i_{\ell+1}\) in \(\{1, \dots, J\}\). 
We further require that
\[
t_j \max_{i\not= i'}{\frac{\abs{e_{i}-e_{i'}}}{\abs{v_i-v_{i'}}}}
\leq 1.
\]
We then define the simplicial map \(\widetilde{\gamma}_j \colon \Simplex^{\ell + 1} \to \R^m\) which is affine on each \(\Sigma^{\ell+1}\in \cK_{j}^{\ell+1}\) and such that, for every \(i \in \{i, \ldots, J\}\),
\[
\widetilde{\gamma}_j(v_i)=\gamma(v) + t_j e_{i}.
\]
Note that
\begin{multline*}
    \max_{i\not= i'}{ \frac{\abs{\widetilde{\gamma}_j(v_{i}) - \widetilde{\gamma}_j(v_{i'})}}{\abs{v_{i} - v_{i'}}}}
    \leq \abs{\gamma}_{\Lip(\Simplex^{\ell + 1})} + t_j \max_{i\not= i'}{\frac{\abs{e_{i}-e_{i'}}}{\abs{v_i-v_{i'}}}}
    \leq \abs{\gamma}_{\Lip(\Simplex^{\ell + 1})} + 1.
\end{multline*}
As for \eqref{eq154}, one gets
\[
\abs{\widetilde{\gamma}_j}_{\Lip(\Simplex^{\ell + 1})} 
\leq \frac{\C}{\eta^{\ell+1}} (\abs{\gamma}_{\Lip(\Simplex^{\ell + 1})} + 1),
\]
which implies again that \((\widetilde{\gamma}_j)_{j\in \N}\) is equiLipschitz. Moreover, for every \(j\in \N\) and every simplex \(\Sigma^{\ell+1}\in \cK^{\ell+1}_{j}\) with vertices \(v_{i_0}, \dots, v_{i_{\ell+1}}\), the vectors  
\[
\widetilde{\gamma}_j(v_{i_s})-\widetilde{\gamma}_j(v_{i_0})=\gamma(v_{i_s})-\gamma(v_{i_0})+t_j(e_{v_{i_s}}-e_{v_{i_0}})
\]
with \(1\leq s \leq \ell+1\) are, by construction, linearly independent. This proves that \(\widetilde{\gamma_j}|_{\Sigma^{\ell+1}}\) is injective.
We thus have the conclusion taking the sequence \( (\widetilde\gamma_j)_{j \in \N} \). 
\end{proof}

The completion of the proof of Proposition~\ref{lemmaApproximationSimplexLipschitz} now relies on an argument involving a transversal perturbation of the identity:

\begin{proof}[Proof of Proposition~\ref{lemmaApproximationSimplexLipschitz}]
Let \(\tau \colon U \to \Omega\) be a transversal perturbation of the identity, given by translation as in Example~\ref{exampleTransversalPertubationEuclidean}.
We then take a summable function \(\widetilde w \ge w\) as provided by Proposition~\ref{propositionFugledeApproximation}.
Let \(\gamma \colon  \Simplex^{\ell + 1} \to \Omega\) be a Lipschitz map such that \(\gamma|_{\partial\Simplex^{\ell + 1}} \in \Fuglede_{\tilde w}(\partial\Simplex^{\ell + 1}; \Omega)\).

Let \((\cK^{\ell+1}_j)_{j\in \N}\) be a sequence of simplicial complexes as in Lemma~\ref{lemmaApproximationSimplexLipschitz-1} and let \((\gamma_j)_{j\in \N}\) be the corresponding sequence of simplicial maps converging to \(\gamma\) in \(\Smooth^{0}(\Simplex^{\ell+1};\R^m)\). 
In view of Lemma~\ref{lemmaApproximationSimplexLipschitz-2}, we can require that each \(\gamma_j|_{\Sigma^{\ell+1}}\) is injective for every \(j\in \N\) and \(\Sigma^{\ell+1}\in \cK^{\ell+1}_j\).

Since \(\gamma(\Simplex^{\ell+1})\) is a compact subset of \(\Omega\), there exists \(\delta > 0\) such that \((\gamma(\Simplex^{\ell + 1}) + B_{\delta}^m) \times B_{\delta}^m \subset U\). By uniform convergence, we can  assume that this inclusion also holds for \(\gamma_{j}\) instead of \(\gamma\). By Proposition~\ref{propositionFugledeApproximation},  there exist a subsequence \((\gamma_{j_{i}})_{i \in \N}\) and a sequence \((\xi_{i})_{i \in \N}\) that converges to \(0\) in \(\R^{m}\) such that 
\(\tau_{\xi_{i}} \compose \gamma_{j_{i}}|_{\partial\Simplex^{\ell + 1}} \in \Fuglede_{w}(\partial\Simplex^{\ell + 1} ; \Omega)\) for every \(i \in \N\) and
\[{}
\tau_{\xi_i} \compose \gamma_{j_{i}}|_{\partial\Simplex^{\ell + 1}}
\to \gamma|_{\partial\Simplex^{\ell + 1}} \quad \text{in \(\Fuglede_{w}(\partial\Simplex^{\ell + 1} ; \Omega)\).}
\]
In view of \eqref{eqTransversalPerturbationLimit}, we may choose each \(\xi_i\) from a subset of positive measure in \(B_{\epsilon_i}^m\).
Thus, by applying property~\eqref{eqExtension-568} with \(X = K_{j_{i}}^{\ell}\), we can further assume that \(\xi_i\) also satisfies \(\tau_{\xi_{i}} \compose \gamma_{j_{i}}|_{K_{j_{i}}^{\ell}} \in \Fuglede_{w}(K_{j_{i}}^{\ell} ; \Omega)\).
We therefore have the conclusion with the sequence \((\tau_{\xi_i} \compose \gamma_{j_{i}})_{i \in \N}\)\,.
\end{proof}

\begin{proof}[Proof of Proposition~\ref{propositionCharacterizationExtensionSimplex}]
For the direct implication, let \(\Simplex^{\ell + 1} \subset \Omega\). 
The conclusion follows from Proposition~\ref{propositionExtendabilitySimplex} by taking as \(\gamma\colon \Simplex^{\ell+1}\to \Omega\) the inclusion map.
To prove the converse, let us take the \(\ell\)-detectors \( w_{0} \) given by the assumption
and \( w_{1} \) from Definition~\ref{definitionVMOell} of \( \VMO^{\ell} \).
Let \( w = w_{0} + w_{1} \) and \(\widetilde w \ge w\) be an \(\ell\)-detector given by Proposition~\ref{lemmaApproximationSimplexLipschitz}.
Let \(\Simplex^{\ell+1}\) be a simplex. Given a Lipschitz map \(\gamma \colon  \Simplex^{\ell + 1} \to \Omega\) with \(\gamma|_{\partial\Simplex^{\ell + 1}} \in \Fuglede_{\tilde w}(\partial\Simplex^{\ell + 1}; \Omega)\), we take a sequence of simplicial complexes \((\cK_{j}^{\ell + 1})_{j \in \N}\) and a sequence of Lipschitz maps \((\gamma_{j})_{j \in \N}\) that satisfy the conclusion of Proposition~\ref{lemmaApproximationSimplexLipschitz}.
	
	Fix \(j \in \N\) and denote the \((\ell + 1)\)-dimensional simplices of \(\cK_{j}^{\ell + 1}\) by \((\Sigma_{i})_{i \in I}\).{}
	Since each \(\gamma_{j}(\Sigma_{i})\) is a simplex and \(\gamma_{j}|_{\Sigma_{i}}\) is an affine diffeomorphism, we have \(\partial \gamma_{j}(\Sigma_{i}) = \gamma_{j}(\partial\Sigma_{i})\).{}
	As a result, the property \(\gamma_{j}|_{\partial\Sigma_{i}} \in \Fuglede_{w}(\partial\Sigma_{i}; \Omega)\) shows that \(w|_{\gamma_{j}(\partial\Sigma_{i})}\) is summable in \(\gamma_{j}(\partial\Sigma_{i})\). 
	Since \( w \ge w_{0} \), the map \(u|_{\gamma_{j}(\partial\Sigma_{i})}\) is homotopic to a constant in \(\VMO(\gamma_{j}(\partial\Sigma_{i}); \manfN)\), which then implies that \(u \compose \gamma_{j}|_{\partial\Sigma_{i}}\) is homotopic to a constant in \(\VMO(\partial\Sigma_{i}; \manfN)\).{}
	Hence, by Lemmas~\ref{lemmaVMORestriction} and~\ref{lemmaExtensionHomotopyConstant}, there exists \(F_{j} \in \Smooth^{0}(\Simplex^{\ell + 1}; \manfN)\) such that
	\[{}
	u \compose \gamma_{j}|_{\partial\Simplex^{\ell + 1}} 
	\sim F_{j}|_{\partial\Simplex^{\ell + 1}}
	\quad \text{in \(\VMO(\partial\Simplex^{\ell + 1}; \manfN)\).}
	\]
	Since the map \(F_{j}|_{\partial\Simplex^{\ell + 1}}\) is homotopic to a constant in \(\Smooth^{0}(\partial\Simplex^{\ell + 1}; \manfN)\), whence also in \(\VMO(\partial\Simplex^{\ell + 1}; \manfN)\), we deduce by transitivity of the homotopy relation that \(u \compose \gamma_{j}|_{\partial\Simplex^{\ell + 1}} \) is homotopic to a constant in \(\VMO(\partial\Simplex^{\ell + 1}; \manfN)\).
	As \( w \ge w_{1} \) and
	\[{}
	\gamma_{j}|_{\partial\Simplex^{\ell + 1}}  \to \gamma|_{\partial\Simplex^{\ell + 1}} 
	\quad \text{in \(\Fuglede_{w}(\partial\Simplex^{\ell + 1}; \Omega)\),}
	\]
	by generic VMO stability of \(u\) we have
	\[{}
	u \compose \gamma_{j}|_{\partial\Simplex^{\ell + 1}}  \to u \compose \gamma|_{\partial\Simplex^{\ell + 1}} 
	\quad \text{in \(\VMO(\partial\Simplex^{\ell + 1}; \manfN)\).}
	\]
	By Proposition~\ref{propositionHomotopyVMOLimit}, there exists \(J \in \N\) such that, for every \(j \ge J\),
	\[{}
	u \compose \gamma_{j}|_{\partial\Simplex^{\ell + 1}}  \sim u \compose \gamma|_{\partial\Simplex^{\ell + 1}} 
	\quad \text{in \(\VMO(\partial\Simplex^{\ell + 1}; \manfN)\).}
	\]
	Therefore, by transitivity of the homotopy relation, \(u \compose \gamma|_{\partial\Simplex^{\ell + 1}}\) is homotopic to a constant in \(\VMO(\partial\Simplex^{\ell + 1}; \manfN)\).	
\end{proof}

\section{Testing on disks}
\label{sectionCharacterizationsDisks}

Before proving Proposition~\ref{propositionCharacterizationDisks}, we show that if a \(\VMO^{\ell}\) map is locally \(\ell\)-extendable then it is \(\ell\)-extendable.{}
More precisely,

\begin{lemma}
\label{proposition_Local_Charts}
Let \(u \in \VMO^{\ell}(\manfV; \manfN)\) and let \((V_i)_{i \in I}\) be an open covering of \(\manfV\).{}
If \(u|_{V_{i}}\) is \(\ell\)-extendable for every \(i \in I\), then \(u\) is \(\ell\)-extendable.
\end{lemma}

\begin{proof}
[Proof of Lemma~\ref{proposition_Local_Charts}]
Without loss of generality, we may assume that \(I\) is countable. 
For each \(i \in I\), let \(w_{i} \colon  V_{i} \to [0, +\infty]\) be an \(\ell\)-detector given by the \(\ell\)-extendability of \(u|_{V_{i}}\).{}
We consider that \(w_{i}\) is defined on \(\manfV\) by assigning the value \(0\) in \(\manfV \setminus V_{i}\).{}
Take a family of positive numbers \((\lambda_{i})_{i \in I}\) such that \(\sum\limits_{i \in I}{\lambda_{i} w_{i}}\) is summable in \(\manfV\) and then take an \(\ell\)-detector \(w\) for \(u\) with \(w \ge \sum\limits_{i \in I}{\lambda_{i} w_{i}}\).
Finally, let  \( \tau \colon U \to \manfV \) be a transversal perturbation of the identity and let \(\widetilde w \ge w\) be an \(\ell\)-detector for \(u\) given by Lemma~\ref{propositionGenericStability}.

Given a simplex \(\Simplex^{\ell + 1}\), let \(\gamma \colon  \Simplex^{\ell + 1} \to \manfV\) be a Lipschitz map such that \(\gamma|_{\partial\Simplex^{\ell + 1}} \in \Fuglede_{\tilde w}(\partial\Simplex^{\ell + 1}; \manfV)\).{}
We prove that \(u \compose \gamma|_{\partial\Simplex^{\ell + 1}}\) is homotopic to a constant map in \(\VMO(\partial\Simplex^{\ell + 1}; \manfN)\), which by Proposition~\ref{propositionExtendabilitySimplex} implies that \(u\) is \(\ell\)-extendable.

To this end, we proceed with barycentric subdivisions of \(\Simplex^{\ell + 1}\) to get a simplicial complex \(\cK^{\ell+1}\) such that \(K^{\ell + 1} = \Simplex^{\ell + 1}\) and, for every \(\Sigma^{\ell + 1} \in \cK^{\ell + 1}\), there exists \(i = i(\Sigma^{\ell + 1}) \in I\) with
\begin{equation*}
	\gamma(\Sigma^{\ell + 1}) \subset V_{i}.
\end{equation*}
Since \( \cK^{\ell + 1} \) is finite and \( \Sigma^{\ell + 1} \) is compact, there exists \( \epsilon > 0 \) independent of \( \Sigma^{\ell + 1} \) such that, for every \( \xi \in B_{\epsilon}^q \),
\begin{equation}
\label{eqCharacterizationsInclusion}
\tau_{\xi} \compose \gamma (\Sigma^{\ell + 1}) \subset V_{i}\,.
\end{equation}

Let \(\cL^{\ell}\) be a subskeleton of \(\cK^{\ell}\) such that \(L^{\ell} = \partial\Simplex^{\ell + 1}\).{}
By assumption on \(\gamma\), we have \(\gamma|_{L^{\ell}} \in \Fuglede_{\tilde w}(L^{\ell}; \manfV)\).
Applying Lemma~\ref{propositionGenericStability}, we take \( \xi \in B_{\epsilon}^q \) such that the Lipschitz map \(\widetilde\gamma \vcentcolon= \tau_{\xi} \compose \gamma \) satisfies
\begin{equation}
	\label{eqCharacterizations-552}
\widetilde\gamma|_{K^{\ell}}
\in \Fuglede_{w}(K^{\ell}; \manfV)
\end{equation}
and, as \(L^{\ell} = \partial\Simplex^{\ell + 1}\),
\begin{equation}
	\label{eqCharacterizations-558}
u \compose \widetilde\gamma|_{\partial\Simplex^{\ell + 1}}
\sim u \compose \gamma|_{\partial\Simplex^{\ell + 1}}
\in \VMO(\partial\Simplex^{\ell + 1}; \manfV).
\end{equation}
Since \(w\) is an \(\ell\)-detector and \eqref{eqCharacterizations-552} holds, we see that \(u \compose \widetilde\gamma|_{K^{\ell}} \in \VMO(K^{\ell}; \manfN)\). 
Then, by Proposition~\ref{propositionHomotopyVMOContinuousMap}, there exists \(f \in \Smooth^{0}(K^{\ell}; \manfN)\) such that
\begin{equation}
	\label{eqCharacterizations-569}
u \compose \widetilde\gamma|_{K^{\ell}}
\sim f
\quad \text{in \(\VMO(K^{\ell}; \manfN)\).}
\end{equation}
By restriction of \eqref{eqCharacterizations-569} to \(\partial\Simplex^{\ell + 1}\), Lemma~\ref{lemmaVMORestriction} implies
\[{}
u \compose \widetilde\gamma|_{\partial\Simplex^{\ell + 1}}
\sim f|_{\partial\Simplex^{\ell + 1}}
\quad \text{in \(\VMO(\partial\Simplex^{\ell + 1}; \manfN)\)}
\]
and then, by \eqref{eqCharacterizations-558} and transitivity of the homotopy relation,
\begin{equation}
\label{eqCharacterizations-269}
u \compose \gamma|_{\partial\Simplex^{\ell + 1}}
\sim f|_{\partial\Simplex^{\ell + 1}}
\quad \text{in \(\VMO(\partial\Simplex^{\ell + 1}; \manfN)\)}.
\end{equation}

Let us show that \(f\) has an extension \(F \in \Smooth^{0}(\Simplex^{\ell + 1}; \manfN)\).
For each \(\Sigma^{\ell + 1} \in \cK^{\ell + 1}\), we have that \( \widetilde{\gamma} = \tau_{\xi} \compose \gamma \)  satisfies \eqref{eqCharacterizationsInclusion} with \( i = i(\Sigma^{\ell + 1}) \in I \).
Since \(w \ge \lambda_{i}w_{i}\)\,, by \eqref{eqCharacterizations-552} we have \(\widetilde\gamma|_{\partial\Sigma^{\ell+1}}
\in \Fuglede_{w_{i}}(\partial\Sigma^{\ell+1}; V_{i})\).
Thus, by \(\ell\)-extendability of \(u|_{V_{i}}\), we deduce that \(u \compose \widetilde\gamma\) is homotopic to a constant  in \(\VMO(\partial\Sigma^{\ell + 1}; \manfN)\).{}
By restriction of \eqref{eqCharacterizations-569} to \(\partial\Sigma^{\ell + 1}\) and transitivity of the homotopy relation, we deduce that \(f|_{\partial\Sigma^{\ell + 1}}\) is homotopic to a constant in \(\VMO(\partial\Sigma^{\ell + 1}; \manfN)\).{}
Since both functions are continuous, they are also homotopic in \(\Smooth^{0}(\partial\Sigma^{\ell + 1}; \manfN)\).{}
Thus, \(f|_{\partial\Sigma^{\ell + 1}}\) has a continuous extension to \(\Sigma^{\ell + 1}\).{}
Combining these extensions, we deduce that \(f\) has a continuous extension \(F\) to \(\Simplex^{\ell + 1}\), and then it is homotopic to a constant in \(\Smooth^{0}(\partial\Simplex^{\ell + 1}; \manfN)\).{}
We conclude from \eqref{eqCharacterizations-269} and Proposition~\ref{propositionExtendabilitySimplex} that \(u\) is \(\ell\)-extendable.
\end{proof}

\begin{proof}[Proof of Proposition~\ref{propositionCharacterizationDisks}]
	``\(\Longrightarrow\)''. 
	Let \(w\) be an \(\ell\)-detector given by the definition of \(\ell\)-extendability of \(u\).{}
	Given a disk \(\Disk^{\ell + 1} \subset \manfV\) for which \(w|_{\partial\Disk^{\ell + 1}}\) is summable in \(\partial\Disk^{\ell + 1}\), let \(\Phi \colon  \cBall^{\ell+1} \to \Disk^{\ell + 1}\) be a diffeomorphism.{}
	Then, by a change of variables, \(w \compose \Phi|_{\Sphere^{\ell}}\) is summable in \(\Sphere^{\ell}\).{}
	Thus, \(\Phi|_{\Sphere^{\ell}} \in \Fuglede_{w}(\Sphere^{\ell}; \manfV)\).{}
	Hence, \(u \compose \Phi|_{\Sphere^{\ell}}\) is homotopic to a constant in \(\VMO(\Sphere^{\ell}; \manfN)\) and then \(u|_{\partial\Disk^{\ell + 1}}\) is homotopic to a constant in \(\VMO(\partial\Disk^{\ell + 1}; \manfN)\).

	``\(\Longleftarrow\)''. 
	We first assume that \(\manfV = \Omega\) is an open subset of \(\R^{m}\).{}
	Let \(w\) be an \(\ell\)-detector from the assumption and let \(w_{1}\) be an \(\ell\)-detector from Definition~\ref{definitionVMOell} of \(\VMO^{\ell}\).
	Given a transversal perturbation of the identity \( \tau \colon U \to \manfV\), we then select an \(\ell\)-detector \(\widetilde{w} \geq w + w_{1}\), as ensured by Proposition~\ref{propositionFugledeApproximation}.

	Now consider a simplex \(\Simplex^{\ell + 1} \subset \Omega\) for which \(\widetilde w|_{\partial\Simplex^{\ell + 1}}\) is summable in \(\partial\Simplex^{\ell + 1}\).
    Take a biLipschitz homeomorphism \(\Psi \colon  \cBall^{\ell+1} \to \Simplex^{\ell + 1}\) and an equiLipschitz sequence of smooth imbeddings \(\Psi_{j} \colon  \cBall^{\ell+1} \to \Omega\) such that \(\Psi_{j} \to \Psi\) uniformly in \(\cBall^{\ell+1}\).{}
	By Proposition~\ref{propositionFugledeApproximation}, we can extract a subsequence \((\Psi_{j_{i}})_{i \in \N}\) and take a sequence \((\xi_{i})_{i \in \N}\) in \(\R^{m}\) that converges to \(0\) such that \((\tau_{\xi_{i}} \compose \Psi_{j_{i}}|_{\Sphere^{\ell}})_{i \in \N}\) is contained in \(\Fuglede_{w + w_{1}}(\Sphere^{\ell}; \Omega)\) and
	\begin{equation}
		\label{eqExtension-816}
	\tau_{\xi_{i}} \compose \Psi_{j_{i}}|_{\Sphere^{\ell}}
	\to \Psi|_{\Sphere^{\ell}}
	\quad \text{in \(\Fuglede_{w + w_{1}}(\Sphere^{\ell}; \Omega)\).}
	\end{equation}

	Denote \(\Disk_{i}^{\ell + 1} = \tau_{\xi_{i}} \compose \Psi_{j_{i}}(\cBall^{\ell+1})\).
	Since \(\tau_{\xi_{i}} \compose \Psi_{j_{i}}|_{\Sphere^{\ell}} \in \Fuglede_{w}(\Sphere^{\ell}; \Omega)\), by a change of variable we have that \(w|_{\partial\Disk_{i}^{\ell + 1}}\) is summable in \(\partial\Disk_{i}^{\ell + 1}\).{}
	Hence, by assumption, \(u|_{\partial\Disk_{i}^{\ell + 1}}\) is homotopic to a constant in \(\VMO(\partial\Disk_{i}^{\ell + 1}; \manfN)\).{}
	Since \(\tau_{\xi_{i}} \compose\Psi_{j_{i}}|_{\Sphere^{\ell}}\) is a diffeomorphism between \(\Sphere^{\ell}\) and \(\partial\Disk_{i}^{\ell + 1}\), the map \(u \compose \tau_{\xi_{i}} \compose \Psi_{j_{i}}|_{\Sphere^{\ell}}\)  is homotopic to a constant in \(\VMO(\Sphere^{\ell}; \manfN)\).{}
	By \eqref{eqExtension-816} and generic VMO stability of \(u\), we have
	\[{}
	u \compose \tau_{\xi_{i}} \compose \Psi_{j_{i}}|_{\Sphere^{\ell}}
	\to u \compose \Psi|_{\Sphere^{\ell}}{}
	\quad \text{in \(\VMO(\Sphere^{\ell}; \manfN)\).}
	\]
	By Proposition~\ref{propositionHomotopyVMOLimit}, for \(i \in \N\) sufficiently large,
	\[{}
	u \compose \tau_{\xi_{i}} \compose \Psi_{j_{i}}|_{\Sphere^{\ell}}
	\sim u \compose \Psi|_{\Sphere^{\ell}}{}
	\quad \text{in \(\VMO(\Sphere^{\ell}; \manfN)\).}
	\]
	Thus, by transitivity of the homotopy relation, we have that \(u \compose \Psi|_{\Sphere^{\ell}}\)  is homotopic to a constant in \(\VMO(\Sphere^{\ell}; \manfN)\). 
	Since \(\Psi|_{\Sphere^{\ell}}\) is a biLipschitz homeomorphism between \(\Sphere^{\ell}\) and \(\partial\Simplex^{\ell + 1}\), we conclude that \(u|_{\partial\Simplex^{\ell + 1}}\)  is homotopic to a constant in \(\VMO(\partial\Simplex^{\ell + 1}; \manfN)\).{}
	It thus follows from Proposition~\ref{propositionCharacterizationExtensionSimplex} that \(u\) is \(\ell\)-extendable when \(\manfV = \Omega\).
	
	When \(\manfV = \manfM\) is a compact manifold we take a covering \((V_{i})_{i \in I}\) of \(\manfM\) such that, for each \(i \in I\), there exists an imbedding \(h_{i} \colon  W_{i} \to \R^{m}\) defined on an open set \(W_{i} \Supset V_{i}\)\,.{}
	The map \(u \compose h_{i}^{-1}\) satisfies the assumption of the proposition on the open set \(h_{i}(V_{i})\). 
	Thus, by the first part of the proof concerning open subsets of \(\R^{m}\), the map \(u \compose h_{i}^{-1}\) is \(\ell\)-extendable in \(h_{i}(V_{i})\), which implies that \(u|_{V_{i}}\) is \(\ell\)-extendable in \(V_{i}\).{}
	We deduce from Lemma~\ref{proposition_Local_Charts} that \(u\) is \(\ell\)-extendable in \(\manfM\).
\end{proof}

A natural question is to determine whether the topological disks can be replaced by standard Euclidean disks in Proposition~\ref{propositionCharacterizationDisks}.
Thanks to the cohomological tools of Chapter~\ref{chapter-cohomological-obstructions}, we show in Corollary~\ref{proposition-spheres-jacobian} that this is possible when homotopy classes can be identified cohomologically.

\cleardoublepage
\chapter{Cohomological obstructions}
\label{chapter-cohomological-obstructions}

Identifying homotopy classes for maps from \(\Sphere^\ell\) to \(\manfN\) and computing the homotopy group \(\pi_{\ell} (\manfN)\) are challenging tasks in general. 
However, in many cases, cohomological tools can provide powerful insights into these questions. 
In this chapter, we build on the classical connection between homotopy and homology to develop criteria for \(\ell\)-extendability, using analytical tools such as the distributional Jacobian and its generalization, the Hurewicz current.

We begin by introducing the Hurewicz degree in the classical setting and extending it to the \(\VMO\) framework, as enabled by the results of Chapter~\ref{section_VMO}. 
The Hurewicz current is then introduced to detect and quantify topological singularities of Sobolev maps. By studying its properties, we establish conditions under which a map is \(\ell\)-extendable when its Hurewicz current vanishes.

To achieve this, we rely on a generic condition for deciding whether a differential form is closed and a generic integral formulation of the Hurewicz degree for Sobolev maps. In the spirit of the results in the previous chapter, we also show that when the Hurewicz degree identifies null-homotopic maps from \(\Sphere^{\ell}\) to \(\manfN\), the property of \(\ell\)-extendability can be verified using restrictions to Euclidean disks parallel to the coordinate axes, rather than simplices.

Finally, we address cases where the Hurewicz current alone is insufficient and consider instead Hopf invariants that provide an alternative framework for analyzing topological singularities and extendability.

\section{Hurewicz degree}
\label{section_Hurewicz}
Given an integer \(\ell\in \N_*\), the topological degree of a map \(u\in \Smooth^{\infty}(\Sphere^\ell;\Sphere^\ell)\) is defined by
\begin{equation}\label{eq-def-degree}
\Deg{u} =\int_{\Sphere^\ell}u^\sharp \omega_{\Sphere^\ell},
\end{equation}
where \(\omega_{\Sphere^\ell}\) is any smooth volume form on \(\Sphere^\ell\) such that \(\int_{\Sphere^\ell}\omega_{\Sphere^\ell}=1\) and \(u^\sharp \omega_{\Sphere^\ell}\) denotes the pullback of this form by \(u\).
The topological degree takes its values in \(\Z\) and, by the Hopf theorem, characterizes the homotopy classes in \(\Smooth^{\infty}(\Sphere^\ell;\Sphere^\ell)\), in the sense that for every \(u, \widetilde{u}\in \Smooth^{\infty}(\Sphere^\ell;\Sphere^\ell)\) one has
\begin{equation}\label{equivalence-degree}
u\sim \widetilde{u}\ \text{in  } \Smooth^0(\Sphere^\ell;\Sphere^\ell) \quad \Longleftrightarrow \quad \Deg{u} = \Deg{\widetilde{u}}\,.
\end{equation}
More generally, the topological degree can be defined for continuous maps between compact oriented manifolds of the same dimension. 
However, the equivalence \eqref{equivalence-degree} need not be true in such a general setting.

In this section,  we extend the explicit definition \eqref{eq-def-degree} of topological degree to cases where the domain and the target manifold do not need to have the same dimension. 
More specifically, given an integer \(\ell\in \N_*\) and denoting by \(\Smooth^{\infty}(\manfN; \Forms^\ell)\) the class of smooth \(\ell\)-differential forms in \(\manfN\), we introduce the following topological invariant:

\begin{definition}
	\label{definitionHurewiczDegree}
	Let \(\varpi \in \Smooth^{\infty}(\manfN; \Forms^\ell)\) be a closed differential form on \(\manfN\).
	For every \(f \in \Smooth^{\infty}(\Sphere^{\ell}; \manfN)\), we define the \emph{Hurewicz degree} of \(f\) relative to \(\varpi\) as
	\[{}
	\hur_{\varpi}(f)
	= \int_{\Sphere^{\ell}} f^{\sharp}\varpi.
	\]
\end{definition}

	That \(\hur_{\varpi}\) is a topological invariant is a consequence of the fact that if \(f_{0}, f_{1} \in \Smooth^{\infty}(\Sphere^{\ell}; \manfN)\) are homotopic in \(\Smooth^{0}(\Sphere^{\ell}; \manfN)\), then
	\begin{equation}
		\label{eqJacobian-53}
	\hur_{\varpi}(f_{0}) = \hur_{\varpi}(f_{1}).
	\end{equation}
	Indeed, taking a smooth homotopy \(H \colon  [0, 1] \times \Sphere^{\ell} \to \manfN\) with \(H(0, \cdot) = f_{0}\) and \(H(1, \cdot) = f_{1}\), by Stokes's theorem we have
	\[{}
	\hur_{\varpi}(f_{1}) - \hur_{\varpi}(f_{0})
  =\int_{\Sphere^\ell} f_{1}^\sharp \varpi
  - \int_{\Sphere^\ell} f_{0}^\sharp \varpi
  = \int_{[0, 1] \times \Sphere^\ell} \dext (H^\sharp \varpi).
	\]
	Since \(\dext (H^\sharp \varpi) = H^{\sharp}(\dext\varpi)\) and \(\varpi\) is closed, the right-hand side vanishes and \eqref{eqJacobian-53} is satisfied.

The Hurewicz degree is only interesting for closed forms \( \varpi \) that are not exact:

\begin{example}
\label{Examplehurewicz-exact}
If \(\varpi  \in \Smooth^{\infty}(\manfN; \Forms^\ell)\) is exact, then \(\varpi = \dext \chi\) with \(\chi\in\Smooth^{\infty}(\manfN; \Forms^{\ell - 1})\) and
\[
\hur_{\varpi}(f)
	= \int_{\Sphere^{\ell}} f^{\sharp}(\dext \chi)
	= \int_{\Sphere^{\ell}} \dext(f^{\sharp}\chi) = 0,
\]
according to Stokes's theorem.
\end{example}

There can also be closed differential forms that are not exact for which the Hurewicz degree is trivial:

\begin{example}
Let \(k \in \{1, \dotsc, \ell - 1\}\) and let \(\varpi_1 \in \Smooth^\infty (\manfN; \Forms^k)\) and \(\varpi_2 \in \Smooth^\infty (\manfN; \Forms^{\ell - k})\) be  two closed forms.
Note that \(\varpi \vcentcolon= \varpi_1 \wedge \varpi_2\) is closed but need not be exact.
Nevertheless, for every \(f \in \Smooth^{\infty}(\Sphere^\ell;\manfN)\) we have
\begin{equation}
\label{eqHurewiczTrivial}
\hur_{\varpi}(f)
= 0.
\end{equation}
In fact, note that \(\dext(f^\sharp \varpi_1) =f^\sharp (\dext \varpi_1) = 0\).
Since \( k < \ell \), the de Rham cohomology \(H^k_{\mathrm{dR}} (\Sphere^\ell)\) is trivial, whence there exists \(\eta \in  \Smooth^\infty (\Sphere^{\ell}; \Forms^{k - 1})\) such that \(f^\sharp \varpi_1 = \dext \eta\). 
Therefore,
\[
\begin{split}
\hur_{\varpi}(f)
=
\int_{\Sphere^\ell} f^{\sharp} \varpi_1 \wedge f^{\sharp} \varpi_2
&=\int_{\Sphere^\ell} \dext \eta \wedge f^{\sharp} \varpi_2\\
&= \int_{\Sphere^\ell} \dext (\eta \wedge f^{\sharp} \varpi_2)
- (-1)^{k-1}\int_{\Sphere^\ell} \eta \wedge \dext (f^{\sharp} \varpi_2).
\end{split}
\]
Applying Stokes's theorem and the fact that we also have \(\dext(f^\sharp \varpi_2) = 0\), we then deduce \eqref{eqHurewiczTrivial}.
\end{example}

\begin{remark}
Since the map \(\varpi \mapsto \hur_{\varpi} (f)\) is linear in \(\varpi\) and vanishes on exact forms, Definition~\ref{definitionHurewiczDegree} and a quotient procedure provide a linear map \(\hur{(f)}\) from the de Rham cohomology group \(H^\ell_{\mathrm{dR}} (\manfN)\) into \(\R\).
By homotopy invariance, \(\hur\) is well defined on homotopy classes of maps from \(\Sphere^\ell\) to \(\manfN\) and induces a homomorphism from \(\pi_\ell (\manfN)\) to the linear maps from \(H^\ell_{\mathrm{dR}} (\manfN)\) to \(\R\).
\end{remark}

In the spirit of Chapter~\ref{section_VMO}, we proceed to show that the invariant \(\hur_{\varpi}\) can be continuously extended to maps in \(\VMO(\Sphere^{\ell}; \manfN)\):

\begin{proposition}
	\label{propositionHurewiczVMO}
	Given a closed differential form \(\varpi \in \Smooth^{\infty}(\manfN; \Forms^{\ell})\), the Hurewicz degree extends as a continuous function \(\hur_{\varpi} \colon  \VMO(\Sphere^{\ell}; \manfN) \to \R\) such that,
 whenever \(v_{0}, v_{1} \in \VMO(\Sphere^{\ell}; \manfN)\) verify \(v_{0} \sim v_{1}\) in \(\VMO(\Sphere^{\ell}; \manfN)\), we have
\[
 \hur_{\varpi}(v_{0}) = \hur_{\varpi}(v_{1}) \text{,}
\]
and, for every \(v \in \Sobolev^{1, \ell}(\Sphere^{\ell}; \manfN)\),
\begin{equation}
\label{eqHurewicz-123}
\hur_{\varpi}(v)
	= \int_{\Sphere^{\ell}} v^{\sharp}\varpi.
\end{equation}
\end{proposition}

\begin{proof}
By Corollary~\ref{corollaryVMODensity}, the set \( \Smooth^{0}(\Sphere^\ell; \manfN) \) is dense in \(\VMO(\Sphere^{\ell}; \manfN)\) and, analogously, so is \( \Smooth^{\infty}(\Sphere^\ell; \manfN) \).
Given \(v \in \VMO(\Sphere^{\ell}; \manfN)\) and any sequence \((u_j)_{j\in \N}\) in \( \Smooth^{\infty}(\Sphere^\ell; \manfN) \) which converges to \(v\) in \(\VMO(\Sphere^{\ell};\manfN)\), Proposition~\ref{propositionHomotopyVMOLimit} implies that there exists \(J \in \N\) such that, for every \(j \ge J\),
	\[
	u_{j} \sim v
	\quad \text{in \(\VMO(\Sphere^{\ell}; \manfN)\).}
	\]
	In particular, taking \(j_{0}, j_{1} \ge J\), we have by transitivity of homotopies that 
	\[{}
	u_{j_{0}} \sim u_{j_{1}}
	\quad \text{in \(\VMO(\Sphere^{\ell}; \manfN)\).}
	\]
	Since both \(u_{j_{0}}\) and \(u_{j_{1}}\) are smooth, we then have by Proposition~\ref{propositionHomotopyVMOtoC} that
	\[{}
	u_{j_{0}} \sim u_{j_{1}}
	\quad \text{in \(\Smooth^{0}(\Sphere^{\ell}; \manfN)\).}
	\]
	In this case, 
	\[{}
	\hur_{\varpi}(u_{j_{0}}) = \hur_{\varpi}(u_{j_{1}}).
	\]
    We then define \(\hur_{\varpi}(v)\) to be this common number.
    A standard argument based on the combination of two sequences that converge to \( v \) implies that this definition does not depend on the choice of the sequence that converges to \( v \).
	
	To obtain the invariance with respect to \(\VMO\) homotopy, we take \(v_{0}, v_{1} \in \VMO(\Sphere^{\ell}; \manfN)\) that are homotopic in \(\VMO(\Sphere^{\ell}; \manfN)\).{}
	Let \((u_{0, j})_{j \in \N}\) and \((u_{1, j})_{j \in \N}\) be sequences in \(\Smooth^{\infty}(\Sphere^{\ell}; \manfN)\) that converge to \(v_{0}\) and \(v_{1}\) in \(\VMO(\Sphere^{\ell}; \manfN)\), respectively.
	Then, for \(j\) sufficiently large, by Proposition~\ref{propositionHomotopyVMOLimit}, transitivity of the homotopy relation and Proposition~\ref{propositionHomotopyVMOtoC}, we have
	\[{}
	u_{0, j} \sim u_{1, j}
	\quad \text{in   \(\Smooth^{0}(\Sphere^{\ell}; \manfN)\).}
	\]
By definition of \(\hur_{\varpi}\) and its invariance for homotopic smooth maps, this implies that
	\[{}
	\hur_{\varpi}(v_{0})
	= \hur_{\varpi}(u_{0, j})
	= \hur_{\varpi}(u_{1, j})
	= \hur_{\varpi}(v_{1}).
	\]

	To prove the integral formula for every \(v \in \Sobolev^{1, \ell}(\Sphere^{\ell}; \manfN)\), observe that the pullback \(v^{\sharp}\varpi(x)\) at a point \(x \in \Sphere^{\ell}\) depends continuously on \(v(x)\) and \(Dv(x)\) and satisfies the estimate
	\[{}
	\abs{v^{\sharp}\varpi(x)}
	\le C \abs{Dv(x)}^{\ell} \norm{\varpi}_{\Lebesgue^{\infty}}.
	\]
	Hence, the function
    \begin{equation}
    \label{eqJacobian-197}
    v \in \Sobolev^{1, \ell}(\Sphere^{\ell}; \manfN)
    \longmapsto \int_{\Sphere^{\ell}} v^{\sharp}\varpi \in \R
    \end{equation}
    is well defined and continuous. 
	By the counterpart of Theorem~\ref{theoremSchoen-Uhlenbeck} for maps defined on \( \Sphere^{\ell} \), see \cite{Schoen-Uhlenbeck}, \(\Smooth^{\infty}(\Sphere^{\ell}; \manfN)\) is dense in \(\Sobolev^{1, \ell}(\Sphere^{\ell}; \manfN)\).
    Since \eqref{eqJacobian-197} coincides with \( \hur_{\varpi} \) on a dense set, the space \(\Sobolev^{1, \ell}(\Sphere^{\ell}; \manfN)\) imbeds continuously in \(\VMO(\Sphere^{\ell}; \manfN)\), and \( \hur_{\varpi} \) is continuous with respect to the \(\VMO\) topology, we then have that \eqref{eqHurewicz-123} is valid for every \(v\in \Sobolev^{1, \ell}(\Sphere^{\ell}; \manfN)\).
\end{proof}

Since \(\hur_{\varpi}\) vanishes on constant maps,
we get as an immediate corollary of Proposition~\ref{propositionHurewiczVMO},

\begin{corollary}	
\label{corollaryHurewiczHomotopicConstant}
	Given a closed differential form \(\varpi \in \Smooth^{\infty}(\manfN; \Forms^{\ell})\),
	if \(v \in \VMO(\Sphere^{\ell}; \manfN)\)
	is homotopic to a constant in \(\VMO(\Sphere^{\ell}; \manfN)\),
	then
	\begin{equation*}
        \hur_{\varpi}(v) = 0.
	\end{equation*}
\end{corollary}

The converse of Corollary~\ref{corollaryHurewiczHomotopicConstant}  holds for a class of manifolds \( \manfN \) for which the Hurewicz degree identifies null-homotopic maps:

\begin{definition}
\label{definition_Hurewicz_identification}
Given \(\ell \in \N_{*}\), we say that the \emph{Hurewicz degree identifies null-homotopic maps from \(\Sphere^\ell\) to \(\manfN\)} whenever any smooth map \(f \in \Smooth^\infty (\Sphere^\ell; \manfN)\) that satisfies \(\hur_{\varpi}(f) = 0\) for all closed differential forms \(\varpi \in \Smooth^{\infty}(\manfN; \Forms^{\ell})\) is homotopic to a constant map in \(\Smooth^{0}(\Sphere^{\ell}; \manfN)\).
\end{definition}

\begin{example}
     \label{exampleHopfDegree}
    Let \(\manfN = \Sphere^{n} \) and \( \ell = n \).
    By the Hopf theorem, a map \( f \in \Smooth^{\infty}(\Sphere^{n}; \Sphere^{n}) \) is homotopic to a constant in \(\Smooth^{0}(\Sphere^{n}; \Sphere^{n})\) if and only if \( \Deg{f} = 0 \). 
    Since \(\Deg = \hur_{\omega_{\Sphere^{n}}}\), we deduce that the Hurewicz degree identifies null-homotopic maps from \( \Sphere^{n} \) to \( \Sphere^{n} \).
\end{example}

\begin{example}
    \label{exampleHopfFibration}
    Let \(\manfN = \Sphere^{2} \) and \( \ell = 3 \).
    The Hopf fibration \(f \colon  \Sphere^{3} \to \Sphere^{2}\) cannot be continuously extended as a map from \(\Ball^{4}\) to \(\Sphere^{2}\), whence it is not homotopic to a constant in \( \Smooth^{0}(\Sphere^{3}; \Sphere^{2})\). 
    On the other hand, if \(\varpi \in \Smooth^{\infty}(\Sphere^{2}; \Forms^{3})\), then necessarily \(\varpi = 0\) and we trivially have \(\hur_{\varpi} (f) = 0\).
    Hence, the Hurewicz degree does not identify null-homotopic maps from \( \Sphere^{3} \) to \( \Sphere^{2} \). 
\end{example}

\begin{example}
\label{example-Chap7-Hurewicz-product}
Let \(\manfN_1\) and \(\manfN_2\) be two compact manifolds without boundary whose Hurewicz degrees identify null-homotopic maps from \(\Sphere^\ell\) to \(\manfN_i\) with \(i \in \{1, 2\}\). 
Then, the Hurewicz degree identifies null-homotopic maps from \(\Sphere^\ell\) to \(\manfN_1 \times \manfN_2\).

Indeed, let \(f\in \Smooth^\infty(\Sphere^\ell ; \manfN_1\times \manfN_2)\) be such that \(\hur_{\varpi}(f)=0\) for every closed differential form \(\varpi\in \Smooth^{\infty}(\manfN_1\times \manfN_2 ; \Lambda^\ell)\). 
Consider the projections 
\[
p_i \colon (x_1,x_2) \in \manfN_1\times \manfN_2 \longmapsto x_i \in \manfN_i\,.
\]
For every closed differential form \(\varpi_i \in \Smooth^{\infty}(\manfN_i; \Forms^\ell)\), the pull-back \(p_{i}^\sharp \varpi_i\) is closed in \(\manfN_1 \times \manfN_2\). 
By assumption on the Hurewicz degree of \(f\), we then have
\[
\int_{\Sphere^\ell} (p_i \compose f)^\sharp\varpi_i
=\int_{\Sphere^\ell} f^\sharp (p_{i}^\sharp \varpi_i) 
= 0.
\]
Since the Hurewicz degree identifies null-homotopic maps from \(\Sphere^\ell\) to \(\manfN_i\)\,, we deduce that \(p_i \compose f\) is homotopic to a constant map in \(\Smooth^{0}(\Sphere^\ell; \manfN_i)\). 
It then follows that \(f\) itself is homotopic to a constant map in \(\Smooth^{0}(\Sphere^\ell;\manfN_1\times \manfN_2)\).
\end{example}

Recall that, by Example~\ref{Examplehurewicz-exact}, the condition \(\hur_{\varpi}(f) = 0\) is automatically satisfied when \(\varpi\) is exact. 
As a result, to decide whether the Hurewicz degree identifies null-homotopic maps from \(\Sphere^\ell\) to \(\manfN\), it is only necessary to test this condition for the finitely many closed forms \(\varpi\) representing the elements of a given basis in the cohomology group \(H^{\ell}_{\mathrm{dR}}(\manfN)\), see \cite{BottTu1982}*{Proposition~5.3.1}.

The Hurewicz theorem \cite{Hatcher_2002}*{Theorem~4.32} provides a sufficient condition on the homotopy groups of \(\manfN\) under which the Hurewicz degree identifies null-homotopic maps:

\begin{proposition}
\label{lemma-homotopy-homology}
Let \( \manfN \) be an orientable compact manifold without boundary.
If \( \pi_0 (\manfN)\simeq \ldots \simeq \pi_{\ell - 1} (\manfN) \simeq \{0\} \)
and \( \pi_{\ell} (\manfN) \) is abelian and torsion-free, then the Hurewicz degree identifies null-homotopic maps from \(\Sphere^\ell\) to \(\manfN\).
\end{proposition}

We refer the reader to \cite{DetailleMironescuXiao}*{Appendix~A} for a proof of Proposition~\ref{lemma-homotopy-homology}.

\begin{remark}
\label{remark_torsion}
When \(\pi_\ell(\manfN)\) is not torsion-free, there exists  \(f\in \Smooth^{\infty}(\Sphere^{\ell};\manfN)\) which has finite order \(q\) in \(\pi_{\ell}(\manfN)\). 
Then, according to the Hurewicz theorem, the current
\[
f_\sharp [\Sphere^\ell]\colon \varpi \in \Smooth^{\infty}(\manfN;\Forms^{\ell}) \longmapsto \int_{\Sphere^\ell} f^\sharp \varpi
\]
defines an element of order \(q\) in the integer singular homology group \(H_\ell(\manfN;\Z)\), which is thus trivial when considered as an element of the real singular homology group.
Indeed, for every closed differential form \(\varpi\in\Smooth^{\infty}(\manfN;\Lambda^\ell)\),
\[
\int_{\Sphere^{\ell}}f^\sharp \varpi 
= \langle f_{\sharp}[\Sphere^\ell], \varpi \rangle 
= \langle q f_{\sharp}[\Sphere^\ell], \varpi/q \rangle 
= 0.
\]
In this situation, there is a loss of topological information for maps \(v\in \Sobolev^{1,\ell}(\manfV;\manfN)\) in terms of \(\hur_{\varpi}\).
A similar phenomenon occurs when \(\pi_\ell (\manfN)\) is not abelian, in which case one necessarily has \(\ell = 1\).
It then suffices to take \(f\in \Smooth^{\infty}(\Sphere^{1};\manfN)\) whose homotopy class is a non-trivial element of the commutator subgroup of \(\pi_1 (\manfN)\).
\end{remark}

From the previous remark we see that the assumption on \( \pi_{\ell}(\manfN) \) is necessary in Proposition~\ref{lemma-homotopy-homology}.
In contrast, the conclusion may still hold when the condition \( \pi_0 (\manfN) \simeq \ldots \simeq \pi_{\ell - 1} (\manfN) \simeq \{0\} \) on the lower homotopy groups fails:

\begin{remark}
For every \(\ell \in \N_*\)\,, let \( \manfN = \Sphere^\ell \times \Sphere^1\).
Then, \(\pi_{1}(\Sphere^\ell \times \Sphere^1) \not\simeq \{0\}\), but the Hurewicz degree identifies null-homotopic maps from \(\Sphere^\ell\) to \(\Sphere^\ell \times \Sphere^1\).
Indeed, by Proposition~\ref{lemma-homotopy-homology} we see that the Hurewicz degree identifies null-homotopic maps from \(\Sphere^\ell\) into itself and also from \(\Sphere^\ell\) into \(\Sphere^1\) when \(\ell=1\).
The Hurewicz degree also identifies null-homotopic maps  from \(\Sphere^\ell\) into \(\Sphere^1\) when \(\ell\geq 2\) due to the triviality of \(\pi_\ell(\Sphere^1)\).
In all cases, we then conclude that the Hurewicz degree identifies null-homotopic maps from \(\Sphere^\ell\) into \(\Sphere^\ell \times \Sphere^1\) as a consequence of Example~\ref{example-Chap7-Hurewicz-product}.
\end{remark}

As a consequence of Proposition \ref{propositionHurewiczVMO}, we have

\begin{corollary}
\label{corollaryHurewiczHomotopicConstantConverse}
If the Hurewicz degree identifies null-homotopic maps from \( \Sphere^{\ell} \) to \( \manfN \) and if \(v \in \VMO(\Sphere^{\ell}; \manfN)\) is such that \(\hur_{\varpi}(v) = 0\) for every closed differential form \(\varpi \in \Smooth^{\infty}(\manfN; \Forms^{\ell})\),  then \(v\) is homotopic to a constant in \(\VMO(\Sphere^{\ell}; \manfN)\).
\end{corollary}

\begin{proof}
    By the density of \( \Smooth^{\infty}(\Sphere^{\ell}; \manfN) \) in \( \VMO(\Sphere^{\ell}; \manfN) \) and Proposition~\ref{propositionHomotopyVMOLimit}, there exists \( u \in \Smooth^{\infty}(\Sphere^{\ell}; \manfN) \) such that 
    \begin{equation}
        \label{eqHurewiczHomotopicConstant}
        u \sim v
        \quad \text{in \(\VMO(\Sphere^{\ell}; \manfN)\).}
    \end{equation}
    Then, see Proposition~\ref{propositionHurewiczVMO}, for every closed differential form \(\varpi \in \Smooth^{\infty}(\manfN; \Forms^{\ell})\),
    \[
    \hur_{\varpi}(u) = \hur_{\varpi}(v) = 0.
    \]
    Since the Hurewicz degree identifies null-homotopic maps, \( u \) is homotopic to a constant in \( \Smooth^{0}(\Sphere^{\ell}; \manfN) \) and then in \( \VMO(\Sphere^{\ell}; \manfN) \).
    By \eqref{eqHurewiczHomotopicConstant} and the transitivity of the homotopy relation, we deduce that \( v \) is also homotopic to a constant in \( \VMO(\Sphere^{\ell}; \manfN) \).
\end{proof}

\section{Hurewicz current}

Closely related to the Hurewicz degree is the  Hurewicz current defined for Sobolev maps on \( \manfV\): 

\begin{definition}
\label{defnJacobianHurewicz}
Let \( \ell \in \{ 1, \dots, m-1 \} \).
Given a map \(u \in \Sobolev^{1, \ell} (\manfV; \manfN)\) and a closed differential form \(\varpi \in \Smooth^\infty (\manfN; \Forms^\ell)\),{}
we define the \emph{Hurewicz current} of \(u\) relative to \(\varpi\) as
\[
\dualprod{\Hur_{\varpi}(u)}{\alpha}
= \int_{\manfV}  u^\sharp \varpi \wedge \dext \alpha{}
\quad \text{for every \(\alpha \in \Smooth^\infty_c (\manfV; \Forms^{m - \ell - 1})\).}
\]
\end{definition}

We assume throughout the chapter that \(\manfV\) is orientable and we denote by \(\Smooth^\infty_c (\manfV; \Forms^{m - \ell - 1})\) the class of smooth \((m - \ell - 1)\)-differential forms that are compactly supported in \(\manfV\). 
Since \(u \in \Sobolev^{1, \ell} (\manfV; \manfN)\) and
\begin{equation}
\label{eqJacobian-estimatePullBack}
 \abs{u^\sharp \varpi}
 \le C \abs{Du}^\ell \norm{\varpi}_{\Lebesgue^\infty},
\end{equation}
the left-hand side of \eqref{eqJacobian-estimatePullBack} is summable and, as a result,  \(\Hur_{\varpi}(u)\) is well defined.
When the target manifold \(\manfN\) is the sphere \(\Sphere^{n}\) and \(\varpi\) is the standard volume form \(\omega_{\Sphere^{n}}\), the Hurewicz current is the \emph{distributional Jacobian} \cites{Ball,Brezis-Coron-Lieb} that is defined for \(m \ge n + 1\) and \(u \in \Sobolev^{1, n}(\manfV; \Sphere^{n})\) by
	\begin{equation}
 \label{eq-def-jacobian}{}
	\jac{u} = \Hur_{\omega_{\Sphere^{n}}}(u).
	\end{equation}

The Hurewicz current vanishes on smooth maps:

\begin{example}
\label{exampleHurewiczSmooth}
For every \(u \in (\Sobolev^{1, \ell} \cap \Smooth^{\infty})(\manfV; \manfN)\) and every closed differential form \(\varpi \in \Smooth^\infty (\manfN; \Forms^\ell)\), one has
\[
 \Hur_{\varpi}(u) = 0
		\quad \text{in the sense of currents in \(\manfV\).}
\]
	Given \(\alpha \in \Smooth^\infty_c (\manfV; \Forms^{m - \ell - 1})\),
	we need to show that
	\[{}
	\dualprod{\Hur_{\varpi}(u)}{\alpha}
	= 0
\quad \text{for every \(\alpha \in \Smooth^\infty_c (\manfV; \Forms^{m - \ell - 1})\).}
	\]
    Since \(u\) is smooth, the \(m\)-differential form \(u^\sharp \varpi\wedge \alpha\) satisfies
\[
\dext(u^{\sharp}\varpi \wedge \alpha)
	= \dext(u^{\sharp}\varpi) \wedge \alpha + (-1)^{\ell} u^{\sharp}\varpi \wedge \dext \alpha.
\]
Since the differential form \( u^{\sharp}\varpi \wedge \alpha \) has compact support in \( \manfV \), by Stokes's theorem we get
\[
\dualprod{\Hur_{\varpi}(u)}{\alpha} = (-1)^{\ell + 1} \int_{\manfV}\dext(u^{\sharp}\varpi)\wedge \alpha.
\]
Closedness of \(\varpi\) implies that \( \dext(u^{\sharp}\varpi) = u^{\sharp}\dext\varpi = 0 \) and the conclusion follows.
\end{example}

One deduces from Example~\ref{exampleHurewiczSmooth} that \(\Hur_{\varpi}\) is trivial in the closure of smooth maps with respect to the \( \Sobolev^{1, \ell} \)~convergence (see Definition~\ref{definitionHkp}):

\begin{proposition}
	\label{propositionHurewiczTrivial-Hilbert}
	Let \(\varpi \in \Smooth^\infty(\manfN; \Forms^\ell)\) be a closed differential form.
	For every \(u \in \Hilbert^{1, \ell}(\manfV; \manfN)\), we have
	\[{}
	\Hur_{\varpi}(u) = 0
	\quad \text{in the sense of currents in \(\manfV\).}
	\]
\end{proposition}

\begin{proof}
For every \(x\in \manfV\), \((u^\sharp \varpi) (x)\) is a continuous function of \(u (x)\) and \(D u (x)\).
By \eqref{eqJacobian-estimatePullBack}, the map 
\begin{equation}
\label{eqJacobian-26}
u \in \Sobolev^{1, \ell} (\manfV; \manfN) \longmapsto u^\sharp \varpi \in \Lebesgue^{1}(\manfV; \Forms^\ell)
\end{equation}
is then continuous.
	The conclusion then follows from Example~\ref{exampleHurewiczSmooth}  since the set \((\Sobolev^{1, \ell} \cap \Smooth^{\infty})({\manfV}; \manfN)\) is dense in \(\Hilbert^{1, \ell}(\manfV; \manfN)\).
\end{proof}

The following fundamental result establishes a relation between the condition \(\Hur_{\varpi}(u) = 0\) and the vanishing of the Hurewicz degree \(\hur_{\varpi}\) for generic compositions of \(u\) with maps from \( \Sphere^{\ell} \) to \( \manfV \).

\begin{proposition}
\label{propositionJacobianDegreeZero-Fuglede}
Let \(u \in \Sobolev^{1, \ell}(\manfV; \manfN)\) and let \(\varpi \in \Smooth^\infty(\manfN; \Forms^\ell)\) be a closed differential form. 
Then,
\[{}
\Hur_{\varpi}(u) = 0
\quad \text{in the sense of currents in \(\manfV\)}
\]
if and only if there exists an \(\ell\)-detector \(w\) such that, for every Lipschitz map \(\gamma \colon  \cBall^{\ell+1} \to \manfV\) with \(\gamma\vert_{\Sphere^{\ell}} \in \Fuglede_w(\Sphere^\ell;\manfV)\), we have 
\[{}
\hur_{\varpi}(u \compose \gamma\vert_{\Sphere^\ell}) = 0.
\] 
\end{proposition}

We postpone the proof of Proposition~\ref{propositionJacobianDegreeZero-Fuglede} to Section~\ref{sectionHurewiczGeneric}. 
The main ingredient to prove the reverse implication ``\(\Longleftarrow\)'' is a general characterization of closed forms in terms of integration of these forms on \(\ell\)-dimensional spheres in \(\R^m\).
As a by-product of this characterization, we also obtain a sufficient condition to get \(\Hur_{\varpi}(u)=0\) in terms of the restrictions of \(u\) to such spheres.
The direct implication ``\(\Longrightarrow\)'' in Proposition~\ref{propositionJacobianDegreeZero-Fuglede} relies on approximating \( \Lebesgue^{1} \)~differential forms that are closed in the sense of currents by \emph{smooth} closed differential forms and then applying a version of Fuglede's lemma for sequences of differential forms.

As a consequence of Proposition~\ref{propositionJacobianDegreeZero-Fuglede}, we derive the following property of \(\ell\)-extendable maps in connection with the Hurewicz current:

\begin{corollary}
	\label{corollaryJacobianFugledeTrivial}
	Let \(u \in \Sobolev^{1, \ell}(\manfV; \manfN)\).
	If \(u\) is \(\ell\)-extendable, then, for every closed differential form \(\varpi \in \Smooth^\infty(\manfN; \Forms^\ell)\),
	\[{}
	\Hur_{\varpi}(u)
	= 0 \quad \text{in the sense of currents in \(\manfV\).}
	\]
\end{corollary}

\begin{proof}
	Let \(w\) be an \(\ell\)-detector that verifies the \(\ell\)-extendability for \(u\).{}
	Given a Lipschitz map \(\gamma \colon  \cBall^{\ell+1} \to \manfV\) such that \(\gamma|_{\Sphere^{\ell}} \in \Fuglede_{w}(\Sphere^{\ell}; \manfV)\), the map \(u \compose \gamma|_{\Sphere^{\ell}}\) is homotopic to a constant in \(\VMO(\Sphere^{\ell}; \manfN)\).
	By invariance of \(\hur_{\varpi}\) under VMO homotopies, we then have
	\begin{equation}
		\label{eqHopf-665}
		\hur_{\varpi}(u \compose \gamma|_{\Sphere^{\ell}}) = 0.{}
	\end{equation}
	Since \(\gamma\) is arbitrary, the conclusion follows from Proposition~\ref{propositionJacobianDegreeZero-Fuglede}.
\end{proof}

Theorem~\ref{theoremhkpGreatBall} establishes the equivalence between \(\ell\)-extendability and the condition \(u \in \Hilbert^{1,\ell}(\manfV; \manfN)\) when \(\manfV = \Ball^m\). 
This fundamental result yields an alternative proof of Corollary~\ref{corollaryJacobianFugledeTrivial} based on Proposition~\ref{propositionHurewiczTrivial-Hilbert} that relies on the triviality of \(\Hur_{\varpi}\) on smooth maps and the continuity of \(\Hur_{\varpi}\) with respect to \(\Sobolev^{1, \ell}\)~convergence. 
Instead, by deducing this result from Proposition~\ref{propositionJacobianDegreeZero-Fuglede}, we present an alternative approach that avoids the approximation of Sobolev maps.

In Chapter~\ref{chapterGenericEllExtension}, given \(f \in \Smooth^{\infty}(\Sphere^{n}; \manfN)\)  we rely on the map \(u \colon  \Ball^{n + 1} \to \manfN\) defined for \(x \ne 0\) by 
\begin{equation}
    \label{eqHurewiczExample}
    u(x) = f\Bigl( \frac{x}{|x|} \Bigr)
\end{equation}
as a model example for constructing maps that are \( n \)-extendable, depending on the existence of a continuous extension of \( f \) to \( \cBall^{n + 1} \), see Lemma~\ref{propositionExtensionHomogeneity}.
As originally observed by Brezis, Coron and Lieb~\cite{Brezis-Coron-Lieb}*{Appendix~B} for the distributional Jacobian, this example also illustrates for \(m = n + 1\) and \(\ell = n\) how the Hurewicz current detects topological singularities.
Here, the presence of a singularity at \(0\) is indicated by a Dirac mass \(\delta_{0}\) and its intensity is given in terms of the degree \(\hur_{\varpi}(f)\)\,:

\begin{proposition}
	\label{lemmaJacobianDegree}
	Given \(f \in \Smooth^{\infty}(\Sphere^{n}; \manfN)\), the map \( u \) defined by \eqref{eqHurewiczExample}  belongs to \(\Sobolev^{1, n}(\Ball^{n + 1}; \manfN)\), and for every closed differential form \(\varpi \in \Smooth^\infty(\manfN; \Forms^n)\) we have
\[
\Hur_{\varpi}{(u)} = (-1)^{n + 1} \hur_{\varpi}(f) \, \delta_0
\quad \text{in the sense of currents in \(\Ball^{n + 1}\).}
\]
\end{proposition}

\resetconstant
\begin{proof}
As there exists \(\Cr{cteJacobian-450}>0\) such that, for every \( x\in \Ball^{n + 1} \setminus \{0\}\),
 \begin{equation}\label{eq113}
 \abs{Du(x)}\leq \frac{\Cl{cteJacobian-450}}{|x|},
 \end{equation}
we see that \(u\in \Sobolev^{1,n}(\Ball^{n+1};\manfN)\).
Given \(\alpha \in \Smooth_{c}^{\infty}(\Ball^{n + 1})\), we now have to prove that
 \begin{equation}
	\label{eqJacobian-66}
	\dualprod{\Hur_{\varpi}(u)}{\alpha} 
	= (-1)^{n + 1} \hur_{\varpi}(f) \, \alpha(0)
	\quad \text{for every \(\alpha \in \Smooth_{c}^{\infty}(\Ball^{n + 1})\).}
	\end{equation}
Since \(u^{\sharp}\varpi\) is smooth and closed in \(\Ball^{n + 1} \setminus \{0\}\),
	\[{}
	\dext(\alpha \, u^{\sharp}\varpi){}
	= (-1)^{n} u^{\sharp}\varpi \wedge \dext\alpha{}
	\quad \text{in \(\Ball^{n + 1} \setminus \{0\}\).}
	\]
	Thus, for every \(0 < \epsilon < 1\), by Stokes's theorem,
	\begin{equation}
    \label{eqJacobian-474}
    \int_{\Ball^{n + 1} \setminus B_{\epsilon}^{n + 1}} u^{\sharp}\varpi \wedge \dext\alpha{}
	= (-1)^{n} \int_{\Ball^{n + 1} \setminus B_{\epsilon}^{n + 1}} \dext(\alpha \, u^{\sharp}\varpi)
	= (-1)^{n + 1} \int_{\partial B_{\epsilon}^{n + 1}} \alpha \, u^{\sharp}\varpi.
	\end{equation}
	Applying Stokes's theorem again to the closed form \(u^{\sharp}\varpi\) on \(\Ball^{n + 1}\setminus B_{\epsilon}^{n + 1}\), one gets
 \[
	\int_{\partial B_{\epsilon}^{n + 1}} u^{\sharp}\varpi
	= \int_{\partial \Ball^{n + 1}} u^{\sharp}\varpi
	= \int_{\Sphere^{n}} f^{\sharp}\varpi 
	= \hur_{\varpi}(f).
	\]
We may then write
\[
	\int_{\partial B_{\epsilon}^{n + 1}} \alpha \, u^{\sharp}\varpi
	=  \hur_{\varpi}(f) \, \alpha(0)+ \int_{\partial B_{\epsilon}^{n + 1}} (\alpha - \alpha(0)) \, u^{\sharp}\varpi.
\]
By \eqref{eqJacobian-estimatePullBack} and \eqref{eq113}, for \(x \ne 0\) we have
\(
\abs{u^{\sharp}\varpi(x)}
\leq \C {\norm{\varpi}_{\Lebesgue^{\infty}}}/{\abs{x}^n}.
\)
Hence,
\[
\biggabs{\int_{\partial B_{\epsilon}^{n + 1}} (\alpha - \alpha(0)) \, u^{\sharp}\varpi}
\le \C \norm{\alpha - \alpha(0)}_{\Lebesgue^{\infty}(\partial B_{\epsilon}^{n + 1})}.
\]
The continuity of \(\alpha\) at \(0\) then implies that
\[{}
	\lim_{\epsilon \to 0}{\int_{\partial B_{\epsilon}^{n + 1}} \alpha \, u^{\sharp}\varpi} 
	=  \hur_{\varpi}(f) \, \alpha(0).
\]
By \eqref{eqJacobian-474} and Lebesgue's dominated convergence theorem, we thus have
	\[{}
	\dualprod{\Hur_{\varpi}(u)}{\alpha} 
	= \lim_{\epsilon \to 0}{\int_{\Ball^{n + 1} \setminus B_{\epsilon}^{n + 1}} u^{\sharp}\varpi \wedge \dext\alpha{}}  
	= (-1)^{n + 1} \hur_{\varpi}(f) \, \alpha(0),
	\]
	which completes the proof of \eqref{eqJacobian-66}.
\end{proof}

\section{Generic condition for closed forms}

A differential form \(\eta\in \Lebesgue^{1}(\manfV; \Forms^\ell)\) is closed whenever \( \dext\eta = 0 \) in the sense of currents, that is,
\[
\int_{\manfV} \eta \wedge \dext\alpha = 0
\quad \text{for every \(\alpha \in \Smooth^\infty_c(\manfV; \Forms^{m - \ell - 1})\).}
\]
In this section, we first show that, when \(\manfV\) is an open set \(\Omega \subset \R^m\), determining whether \(\eta\) is closed can be achieved by computing its integral along the boundaries of \((\ell+1)\)-dimensional discs parallel to the coordinate axes, in the following sense:

\begin{definition}
\label{def-standard-disc}
A set \(\Disk^{\ell+1}\) is a \emph{standard \((\ell + 1)\)-dimensional disk in \(\R^m\) parallel to the coordinate axes} whenever it can be written as
\[
\Disk^{\ell+1} = rT(\cBall^{\ell + 1}) + \xi,
\]
where \(r > 0\), \(\xi \in \R^m\) and  \(T \colon  \R^{\ell + 1} \to \R^{m}\) is a linear isometry that maps the canonical basis of \(\R^{\ell+1}\) into a subset of the canonical basis of \(\R^m\).
We denote its boundary by
\(\partial\Disk^{\ell+1} = rT(\Sphere^{\ell}) + \xi\).
\end{definition}

We may then identify closed differential forms in the following way:

\begin{proposition}
\label{proposition_char_closed_diff_form-Fuglede-Euclidean}
Let \(\Omega \subset \R^m\) be an open set, and let \(\eta \in \Lebesgue^1 (\Omega; \Forms^\ell)\).
If there exists a summable function \(w \colon \Omega \to [0, +\infty]\) such that, for every standard \((\ell + 1)\)-dimensional disk \(\Disk^{\ell+1} \subset \Omega\) parallel to the coordinate axes for which \(w|_{\partial\Disk^{\ell + 1}}\) is summable in \(\partial\Disk^{\ell + 1}\), the \(\ell\)-differential form \(\eta \) is summable in \(\partial\Disk^{\ell+1}\) and 
\[
\int_{\partial\Disk^{\ell+1}} \eta = 0,
\] 
then \(\eta\) is closed in the sense of currents in \(\Omega\).
\end{proposition}

To prove Proposition~\ref{proposition_char_closed_diff_form-Fuglede-Euclidean}, we need some preliminary results.
We begin by recalling a coarea formula for differential forms:
 
\begin{lemma}
\label{lemma-coarea-forms}
Let \(\Omega \subset \R^m\) be an open set and let \(\eta\in \Lebesgue^1(\Omega; \Forms^{\ell})\).  
If \(f\in \Smooth^{\infty}(\Omega ; \R^{m-\ell})\) has a bounded derivative, then, for almost every \(y\in \R^{m-\ell}\),   the \(\ell\)-differential form \(\eta\) is summable on the oriented smooth submanifold \(f^{-1}(y)\) and
\[
\int_{\Omega} \eta\wedge f^\sharp\dext y 
= \int_{\R^{m-\ell}} \biggl( \int_{f^{-1}(y)}\eta \biggr) \dif y .
\]
\end{lemma} 

We recall that, by Sard's lemma, for almost every \(y\in \R^{m-\ell}\), \(f^{-1}(y)\) is a smooth oriented  submanifold in \(\Omega\) of dimension \(\ell\). 
The statement above is a particular case of the slicing theory for currents, in the specific case where the current is given by \(\beta \mapsto \int_{\Omega}\eta\wedge \beta\), see e.g.\@ \cite{Giaquinta-Modica-Soucek-I}*{Section~2.2.5}.

\begin{proof}[Proof of Lemma~\ref{lemma-coarea-forms}]
Let \(U\) be the open set of points \(x\in \Omega\) such that the tangent map \(Df(x)\) has maximal rank \(m - \ell\).
We denote by \(\langle\cdot,\cdot\rangle\) the standard inner product on \(\ell\)-covectors in \(\R^m\) and by \(\abs{\cdot}\) the corresponding norm.
There exists \(\omega \in \Lebesgue^{\infty}(\R^m; \Forms^{\ell})\) such that, for every \(x\in U\), \(\omega(x)\) is the volume form associated to the oriented \(\ell\)-dimensional space \(\Ker{Df(x)}\) which satisfies \(\abs{\omega(x)}=1\). 

We claim that, for every \(\eta\in \Lebesgue^1(\Omega; \Forms^{\ell})\),
\begin{equation}\label{eq253}
\eta\wedge f^\sharp\dext y 
= \langle\eta, \omega\rangle \Jacobian{m - \ell}{f} \dif x,
\end{equation}
where \(\Jacobian{m - \ell}{f}\) is the \((m - \ell)\)-dimensional Jacobian of \(f\), \(\dext y\) is the standard volume form in \(\R^{m - \ell}\) and \(\dext x\) is the standard volume form in \(\R^{m}\).
In fact, given \(x\in U\), let \((e_1, \dots, e_\ell)\) be an oriented orthonormal basis of \(\Ker{Df(x)}\) and let \((e_{\ell+1}, \dots, e_m)\) be an oriented orthonormal basis of \((\Ker{Df(x)})^\perp\). Then,
\(\omega=e_{1}^{*}\wedge\cdots\wedge e_{\ell}^{*}\) and 
\begin{align*}
f^{\sharp}\dext y &= \det \left(Df(x)[e_{\ell+1}], \dots, Df(x)[e_m]\right)e_{\ell+1}^*\wedge\cdots\wedge e_{m}^*\\
&=(\det \phi(x))e_{\ell+1}^*\wedge\cdots\wedge e_{m}^* \,,
\end{align*}
where \(\phi(x)\vcentcolon=Df(x)|_{(\Ker{Df(x)})^\perp}\)\,.
It follows that, for every \(1\leq i_1<\ldots<i_\ell\leq m\),
\[
e_{i_1}^*\wedge \cdots \wedge e_{i_m}^* \wedge f^\sharp\dext y =
\begin{cases}
\det \phi(x) \dif x & \textrm{ if } (i_1, \ldots, i_\ell)=(1, \dots, \ell),\\
0 & \textrm{ otherwise}.
\end{cases}
\]
Moreover,
\[
\Jacobian{m - \ell}{f}(x)
=\bigl[\det{(Df(x)\compose Df(x)^*)} \bigr]^{\frac{1}{2}}
=\bigl[\det{(\phi(x)\compose \phi(x)^*)} \bigr]^{\frac{1}{2}}
=\det{\phi(x)}.
\]
Hence,  for every \(1\leq i_1<\ldots<i_\ell\leq m\),
\[
e_{i_1}^*\wedge \cdots \wedge e_{i_\ell}^* \wedge f^\sharp\dext y = \langle e_{i_1}^*\wedge \cdots \wedge e_{i_\ell}^* , \omega\rangle \Jacobian{m - \ell}{f} \dif x.
\]
By linearity, we deduce that \eqref{eq253} holds on \(U\). 
When evaluated at some \(x\not\in U\), one has \(f^{\sharp}\dext y=0\) and \(\Jacobian{m - \ell}{f} \dif x=0\), so \eqref{eq253} is actually true in the entire domain \(\Omega\).
By the classical coarea formula, this implies  that
\begin{align*}
\int_{\Omega} \eta\wedge f^\sharp \dext y 
&= \int_{\Omega} \langle\eta, \omega\rangle \Jacobian{m - \ell}{f} \dif \cH^{m}\\
&=\int_{\R^{m-\ell}} \biggl(\int_{f^{-1}(y)} \langle\eta, \omega\rangle \dif \cH^\ell \biggr) \dif y  
= \int_{\R^{m-\ell}} \biggl( \int_{f^{-1}(y)}\eta \biggr) \dif y.
\qedhere{}
\end{align*}
\end{proof} 

In the following lemma, we demonstrate how to extract information about the integral of a differential form \(\eta\) against radial functions, assuming that the integrals of \(\eta\) vanish over spheres.

\begin{lemma}
\label{lemmaIntegralFormRadial}
    Let \( a \in Q_{1}^{\ell + 1} \).
    If \( \eta \in \Lebesgue^{1}(Q_{2}^{m}; \Forms^{\ell}) \) is such that, for almost every \( r \in (0, 1) \) and almost every \( x'' \in Q_{1}^{m - \ell + 1} \),
    \[
    \int_{\partial B^{\ell + 1}_r(a) \times \{x''\}} \eta = 0 \text{,}
    \]
    then, for every \( \zeta \in \Smooth_c^\infty( \R^{m} ) \) supported in \(Q^{m}_1\) that can be written as \( \zeta(x)
= \phi(\abs{x' - a}, x'') \) with \(x = (x', x'') \in \R^{\ell + 1} \times \R^{m - \ell - 1}\) and \(\phi \in \Smooth^{\infty}_c\bigl((0, \infty) \times Q_{1}^{m-\ell-1}\bigr)\), we have
\begin{equation}
    \label{eqIntegralFormRadial}
\int_{Q_{1}^{m}} \eta \wedge \dext ( \zeta \dif x'') = 0 \text{,}
\end{equation}
where \(\dext x'' \vcentcolon= \dext x_{\ell+2}\wedge\cdots\wedge \dext x_m\).
\end{lemma}

\begin{proof}[Proof of Lemma~\ref{lemmaIntegralFormRadial}]  
Since \(\zeta\) is supported in \(Q_{1}^{m}\), we have in particular, 
\[
\phi(t, x'') = 0
\quad \text{for \(t \ge 1 - \abs{a}_{\infty}\),}
\]
where \( \abs{a}_{\infty} \vcentcolon= \max{\{ \abs{a_{1}}, \ldots, \abs{a_{\ell + 1}}\}}\).
We also take \( \delta > 0 \) such that 
\[
\phi(t, x'') = 0
\quad \text{for \(t \le \delta\).}
\]
We now write \(\zeta = \phi \compose f\), where \(f(x)=(\abs{x' - a}, x'')\).
Note that \(f\) is smooth in \( \omega \vcentcolon= \bigl( B_{1 - \abs{a}_{\infty}}^{\ell + 1}(a) \setminus B_{\delta}^{\ell + 1}(a)  \bigr) \times Q_{1}^{m - \ell - 1}\) and \(\zeta = 0\) in the complement of this set.
We then have
\[{}
\dext ( \zeta \dif x''){}
= \phi_{r} \compose f \; \dif r \wedge \dext x''
= \phi_{r} \compose f \; f^{\sharp}\dext y,
\]
where \(\phi_{r}\) denotes the derivative of \(\phi\) with respect to the first variable, \(\dext r\) is the linear form \(v \in \R^{\ell+1}\mapsto \langle \frac{x'-a}{\abs{x'-a}}, v\rangle\), and \(\dif y = \dif y_{1} \wedge \cdots \wedge \dext y_{m - \ell}\) is the standard volume form in \(\R^{m - \ell}\).
Thus, by the assumption on the support of \(\phi\) and the co-area formula (Lemma~\ref{lemma-coarea-forms}) applied to \( \phi_r \compose f \; \eta \) on \(\omega\), 
\[
\begin{split}
\int_{Q_{1}^{m}} \eta \wedge \dext ( \zeta \dif x'')  
 = \int_{\omega} \eta \wedge \dext ( \zeta \dif x'')  
& =\int_{\R^{m-\ell}} \biggl( \int_{f^{-1}(y) \cap \omega} \phi_{r} \compose f \; \eta \biggr) \dif y\\
& =\int_{(\delta, 1 - \abs{a}_{\infty}) \times Q_{1}^{m - \ell - 1}} \phi_{r}(y) \biggl( \int_{f^{-1}(y) \cap \omega}  \eta \biggr) \dif y.
\end{split}
\]
By Fubini's theorem, we then get
\[
\int_{Q_{2}^{m}} \eta \wedge \dext ( \zeta \dif x'')  
=\int_{Q_{1}^{m - \ell - 1}} \biggl( \int_{\delta}^{1 - \abs{a}_\infty} \phi_{r}(r,x'') \biggl( \int_{\partial B^{\ell + 1}_r(a) \times \{x''\}}\eta \biggr) \dif r \biggr) \dif x''.
\]
By the assumption on \(\eta\), the innermost integral vanishes for almost every \(r \in (\delta , 1 - \abs{a}_\infty)\) and almost every \(x'' \in Q_{1}^{m - \ell - 1}\).
This implies \eqref{eqIntegralFormRadial}.
\end{proof}

We next observe that smooth functions can be approximated by a linear combination of smooth radial functions.
More generally,

\begin{lemma}
\label{lemma-approx-radial-functions}
Let \(g \in \Smooth^{\infty}_c( Q_{1}^{m} )\).
Then, for every \(\epsilon > 0\), there exists a function \(\zeta \in \Smooth_{c}^{\infty}(Q_{1}^{m})\) such that
\begin{equation}
\label{eqApproximationRadialEstimate}
\norm{g - \zeta}_{\Smooth^{1}(Q_{1}^{m})}
\le \epsilon
\end{equation}
which can be written for every \(x = (x', x'') \in Q_1^{\ell + 1} \times Q_ 1^{m - \ell - 1}\) in the form
\[
\zeta(x){}
= \sum_{j\in J}{\psi(\abs{x'-a_j}) \, g(a_{j}, x'')}\text{,}
\]
where \(J\) is a finite set, \(\psi \in \Smooth_{c}^{\infty}(0, \infty)\), \((a_j)_{j\in J}\subset Q_{1}^{\ell+1}\)
and each term of the sum has compact support in \(Q_{1}^{m}\).
\end{lemma}

\begin{proof}[Proof of Lemma~\ref{lemma-approx-radial-functions}]
We extend \(g\) by zero outside \(Q_1^{m}\).
Let \(\varphi \in \Smooth^{\infty}_c(\R^{\ell + 1})\) be a radial mollifier supported in \(\Ball^{\ell + 1} \setminus \{0\}\).{}
Given \(0 < \delta < 1\), we consider the convolution of \(g\) with respect to the \(x'\) variable that, since \(\supp{g} \subset Q_{1}^m\), we may write for every \(x = (x', x'') \in Q_{1}^{\ell + 1} \times Q_{1}^{m - \ell - 1}\) as
\begin{equation}
\label{eqApproximationRadialConvolution}
g_{\delta}(x)
= \int_{\Ball^{\ell + 1}} \varphi(z') g(x' + \delta z', x'') \dif z'
= \frac{1}{\delta^{\ell+1}} \int_{Q_{1}^{\ell + 1}} \varphi \Bigl(\frac{x' - y'}{\delta}\Bigr) g(y', x'') \dif y'.
\end{equation}
By smoothness of \(g\) in both \(x'\) and \(x''\) variables and by compact support of \( g \), we have
\[{}
g_{\delta} \to g
\quad \text{and} \quad{}
Dg_{\delta} \to Dg
\quad \text{uniformly in \(Q_{1}^{m}\).}
\]
Given \(\epsilon>0\), we then take \(0 < \delta \le 1\) such that 
\begin{equation}
\label{eqJacobian-727}
\norm{g - g_{\delta}}_{\Smooth^{1}(Q_{1}^{m})}
\le \frac{\epsilon}{2}.
\end{equation}

Since the integrand of the second integral of \eqref{eqApproximationRadialConvolution} over \( Q_{1}^{\ell + 1}\) is uniformly continuous, we may approximate \(g_{\delta}\) uniformly using Riemann sums.
The same reasoning applies to \(Dg_{\delta}\) using the derivative of the same integral representation.
More precisely, given \( \tau = 1/k \) with \( k \in \N_* \)\,, let \(J_{\tau}\) be the subset of points \(h\in \Z^{\ell+1}\) such that 
\begin{equation*}
(0, \tau)^{\ell + 1} + \tau h \subset Q_{1}^{\ell + 1}.
\end{equation*}
Taking the function \( \zeta_{\tau} \colon Q_{1}^{m} \to \R \) defined by
\[
\zeta_{\tau}(x){}
= \frac{\tau^{\ell+1}}{\delta^{\ell+1}} \sum_{h \in J_{\tau}} \varphi \Bigl(\frac{x' - \tau h}{\delta}\Bigr) g(\tau h, x''),
\]  
there exists \(\tau\) sufficiently small such that, for every  \(x \in Q_{1}^{m}\), 
\[
\abs{g_{\delta}(x) - \zeta_{\tau}(x)} \le \frac{\epsilon}{4}
\quad \text{and} \quad 
\abs{Dg_{\delta}(x) - D\zeta_{\tau}(x)} \le \frac{\epsilon}{4}.
\]
Then, by \eqref{eqJacobian-727} and the triangle inequality, \(\zeta_{\tau}\) satisfies estimate \eqref{eqApproximationRadialEstimate}.

Since \(\varphi\) is radial and has compact support in \(\Ball^{\ell + 1} \setminus \{0\}\),  there exists
\(\psi \in \Smooth_{c}^{\infty}(0, \infty)\) that satisfies, for every \(x' \in \R^{\ell + 1}\),
\[
\psi(\abs{x'}) = \frac{\tau^{\ell+1}}{\delta^{\ell+1}} \varphi\Bigl(\frac{x'}{\delta} \Bigr).
\]
Hence, for every \( h \in J_\tau\), we may rewrite each term in the definition of \(\zeta_{\tau}\) as
\begin{equation}
\label{eqApproximationSumTerm}
\frac{\tau^{\ell+1}}{\delta^{\ell+1}} \varphi \Bigl(\frac{x' - \tau h}{\delta}\Bigr) g(\tau h, x'') = \psi(|x'-\tau h|)g(\tau h, x'').
\end{equation}
We are left to verify the property on the supports of these functions. To this end, take \(0 < \lambda < 1\) such that \( \supp{g} \subset Q_{1 - \lambda}^{m} \). 
Note that when \( \tau h \not\in Q_{1 - \lambda}^{\ell + 1} \), the function in \eqref{eqApproximationSumTerm} is zero.
On the other hand, if \( \tau h \in Q_{1 - \lambda}^{\ell + 1} \), the function in \eqref{eqApproximationSumTerm} is supported in \( B_{\delta}^{\ell + 1}(\tau h) \times Q_{1 - \lambda}^{m - \ell - 1} \).
By choosing \( \delta \le \lambda \), this set is contained in \( Q_{1}^{m} \).
We deduce in particular that \( \zeta_{\tau} \) is supported in \( Q_{1}^{m} \).
\end{proof}

Based on Lemmas~\ref{lemmaIntegralFormRadial} and~\ref{lemma-approx-radial-functions}, we can establish Proposition~\ref{proposition_char_closed_diff_form-Fuglede-Euclidean}:

\resetconstant{}
\begin{proof}[Proof of Proposition~\ref{proposition_char_closed_diff_form-Fuglede-Euclidean}]
Take \(\eta \in \Lebesgue^1 (\Omega; \Forms^{\ell})\)  satisfying the assumption. Since the conclusion has a local nature and is invariant by translation and dilation of the domain, we can assume without loss of generality that \( \Omega = Q_{2}^{m} \) and shall prove that \(\eta\) is a closed differential form in \(Q_{1}^{m}\).

Fix a point \(a\in Q_{1}^{\ell+1}\).
By Tonelli's theorem and a linear change of variables,
\[
\int_{Q_{1}^{m - \ell - 1}} \int_{0}^{1} \biggl( \int_{\partial B^{\ell + 1}_r(a)}  w(x', x'') \dif\cH^{\ell}(x')  \biggr) \dif r \dif x''
\le  \int_{Q_2^m} w < \infty.
\]
We then have for almost every \(r\in (0, 1)\) and almost every \(x''\in Q_{1}^{m-\ell-1}\),
\begin{equation}\label{eq1011}
\int_{\partial B^{\ell + 1}_r(a) \times \{x''\}} w 
= \int_{\partial B^{\ell + 1}_r(a)}  w(x', x'') \dif\cH^{\ell}(x') 
< \infty.
\end{equation}
By assumption on \(\eta\), we deduce that \(\eta\) is summable in
\(\partial B^{\ell + 1}_r(a) \times \{x''\}\) and satisfies
\[
\int_{\partial B^{\ell + 1}_r(a) \times \{x''\}}\eta=0.
\]

Given \(g\in \Smooth^{\infty}_c(Q_{1}^{m})\) and \(\epsilon > 0\), by Lemma~\ref{lemma-approx-radial-functions} there exists \(\zeta \in \Smooth_{c}^{\infty}(Q_{1}^{m})\) such that
\begin{equation}
\label{eqJacobian-801}
\norm{g - \zeta}_{\Smooth^{1}(Q_{1}^{m})} \le \epsilon
\end{equation}
and \(\zeta\) is the finite sum of smooth functions with compact support in \(Q_{1}^{m}\) 
of the form
\[
\psi(|x'-a|)g(a,x'')
\]
where  \(a\) is a point in \(Q_{1}^{\ell+1}\) and \(\psi \in C^{\infty}_c(0, \infty)\).
By Lemma~\ref{lemmaIntegralFormRadial} applied to each function of this form and by linearity of the integral,
\[{}
\int_{Q_{1}^{m}}\eta \wedge \dext(\zeta\dif x'') = 0.
\]
In view of \eqref{eqJacobian-801}, we get
\[
 \biggabs{\int_{Q_{1}^{m}} \eta \wedge \dext(g \dif x'')}
 \le \Cl{cteJacobian-811} \epsilon,
\]
for some constant \(\Cr{cteJacobian-811} > 0\) independent of \(\epsilon\).
Since \(\epsilon > 0\) is arbitrary, this implies
\begin{equation}\label{eq1091}{}
\int_{Q_{1}^{m}} \eta \wedge \dext(g \dif x'') = 0.
\end{equation}
Finally, given \(\alpha \in \Smooth_{c}^{\infty}(Q_{1}^{m}; \Forms^{m - \ell - 1})\), we write
\[{}
\alpha = \sum_{I}{g_{I} \dif x_{I}},
\]
where \(g_{I} \in \Smooth_{c}^{\infty}(Q_{1}^{m})\) and the sum is taken over all multi-indices \(I = (i_{1}, \ldots, i_{m - \ell - 1})\) with \(1 \le i_{1} < \ldots < i_{m - \ell - 1} \le m\).
By \eqref{eq1091} applied to every \(g_I\) (with the variable \(x''\) replaced by \(x_I\)) and by linearity of the integral, we obtain
\[{}
\int_{Q_{1}^{m}} \eta \wedge \dext\alpha = 0,
\]
which proves that \(\ell\)-differential form \(\eta\) is closed in the sense of currents in \(Q_{1}^{m}\).
\end{proof}

More generally, to identify closed differential forms on manifolds we need a version of Proposition~\ref{proposition_char_closed_diff_form-Fuglede-Euclidean} adapted to this setting:

\begin{proposition}
\label{proposition_char_closed_diff_form-Fuglede}
Let \(\eta \in \Lebesgue^1 (\manfV; \Forms^\ell)\). 
If
there exists a summable function \(w \colon  \manfV \to [0, + \infty]\) such that, for every Lipschitz map \(\gamma \colon  \cBall^{\ell+1} \to \manfV\) with \(\gamma\vert_{\Sphere^\ell}\in \Fuglede_{w}(\Sphere^\ell ; \manfV)\), the \(\ell\)-differential form \((\gamma\vert_{\Sphere^\ell})^\sharp \eta \) is summable in \(\Sphere^{\ell}\) and 
 \[
  \int_{\Sphere^\ell} (\gamma\vert_{\Sphere^\ell})^\sharp \eta = 0 \text{,}
 \]
then \(\eta\) is closed in the sense of currents in \(\manfV\). 
\end{proposition}
\begin{proof}
When \(\manfV\) is an open set in \(\R^m\), this statement is a consequence of Proposition~\ref{proposition_char_closed_diff_form-Fuglede-Euclidean}. We can thus assume that \(\manfV\) is a compact manifold without boundary. Hence, there exists a finite covering of \(\manfV\)  by open sets \((V_j)_{j \in J}\)    such that, for every \(j \in J\), there exists a smooth diffeomorphism \(h_j \colon  V_j \to \R^{m}\).
We also consider a smooth partition of unity \((\theta_j)_{j \in J}\) given by nonnegative functions in \(\Smooth^\infty (\manfV)\) such that each \(\theta_j \) is compactly supported in \(V_j\) and \(\sum\limits_{j \in J}{\theta_j} = 1\).
Given \(\eta\) satisfying the assumption and \(\alpha \in \Smooth^\infty_c (\manfV; \Forms^{m - \ell - 1})\), we write
\begin{equation}
\label{eq_closed_dist_manifold_sum}
 \int_{\manfV} \eta \wedge \dext \alpha
 = \sum_{j \in J} \int_{\manfV} \eta \wedge \dext (\theta_j \alpha)
 = \sum_{j \in J} \int_{\R^m} (h_j^{-1})^\sharp \eta \wedge \dext \bigl( (h_j^{-1})^\sharp(\theta_j \alpha) \bigr).
\end{equation}

For every \(j \in J\), we show that \((h_j^{-1})^\sharp \eta\) satisfies the assumptions of Proposition~\ref{proposition_char_closed_diff_form-Fuglede-Euclidean} with \( \Omega = \R^m \).
To this end, let \(w\) be the summable function in \(\manfV\) associated to \(\eta\) and define 
\[{}
w_j =w\compose h_j^{-1} \, \Jacobian{m}{h_j^{-1}}.
\] 
Then, by the change of variables formula, \(w_j\) is summable in \(\R^m\). 
Moreover, let  \(\gamma \colon  \cBall^{\ell+1} \to \R^m\) be a Lipschitz map such that \(\gamma\vert_{\Sphere^{\ell}} \in \Fuglede_{w_j}(\Sphere^{\ell}; \R^m)\). 
Since \(\Jacobian{m}{h_{j}^{-1}}\) is bounded from below by a positive constant on the compact set \(\gamma(\Sphere^\ell)\), we have 
\(h_j^{-1}\compose\gamma\vert_{\Sphere^\ell} \in \Fuglede_{w}(\Sphere^{\ell};\manfV)\). 
Note that
\[
(\gamma|_{\Sphere^{\ell}})^\sharp (h_j^{-1})^\sharp \eta
= (h_j^{-1} \compose \gamma|_{\Sphere^{\ell}})^\sharp \eta.
\]
Hence, by assumption on \(\eta\), the \( \ell \)-differential form \( (\gamma|_{\Sphere^{\ell}})^\sharp (h_j^{-1})^\sharp \eta \) is summable in \( \Sphere^\ell \) and
\begin{equation}
\label{eqJacobian-878}
  \int_{\Sphere^\ell} (\gamma|_{\Sphere^{\ell}})^\sharp (h_j^{-1})^\sharp \eta
  = \int_{\Sphere^\ell} (h_j^{-1} \compose \gamma|_{\Sphere^{\ell}})^\sharp \eta = 0.
\end{equation}

Given \(r>0\), \(\xi \in \R^m\) and an isometry \(T \colon  \R^{\ell + 1} \to \R^{m}\) that maps the canonical basis of \(\R^{\ell+1}\) to a subset of the canonical basis of \(\R^m\), we wish to apply \eqref{eqJacobian-878} to \(\gamma \colon \cBall^{\ell + 1} \to \R^m\) defined by \(\gamma(x) =  r T(x)+\xi\).
Observe in particular that \(\Disk^{\ell + 1} \vcentcolon= \gamma(\cBall^{\ell + 1})\) is a standard \((\ell + 1)\)-dimensional disk parallel to the coordinate axes.
Since \(T\) is an isometry, we have
\[
\int_{\partial\Disk^{\ell + 1}} w_j
= r^\ell
\int_{\Sphere^{\ell}} w_j \compose \gamma|_{\Sphere^\ell}
\quad \text{and} \quad 
\int_{\partial\Disk^{\ell + 1}}  (h_j^{-1})^\sharp \eta
= \int_{\Sphere^\ell} (\gamma|_{\Sphere^{\ell}})^\sharp (h_j^{-1})^\sharp \eta. 
\]
It then follows that if \(w_j|_{\partial\Disk^{\ell + 1}}\) is summable in \(\partial\Disk^{\ell + 1}\), then \(\gamma\vert_{\Sphere^{\ell}} \in \Fuglede_{w_j}(\Sphere^{\ell}; \R^m)\) and, by \eqref{eqJacobian-878}, we deduce that 
\[
\int_{\partial\Disk^{\ell + 1}} (h_j^{-1})^\sharp \eta = 0.
\]
The assumptions of Proposition~\ref{proposition_char_closed_diff_form-Fuglede-Euclidean}  are then satisfied by the \(\ell\)-differential form \((h_j^{-1})^\sharp \eta\) on the open set \(\Omega=\R^m\) using the summable function \(w_j\).
We deduce therefrom that \((h_{j}^{-1})^\sharp\eta\) is closed in the sense of currents in \(\R^{m}\).{}
We thus get
\[
 \int_{\R^m} (h_j^{-1})^\sharp \eta \wedge \dext\bigl( (h_j^{-1})^\sharp(\theta_j \alpha)\bigr) = 0.
\]
The conclusion now follows from \eqref{eq_closed_dist_manifold_sum}.
\end{proof}

\section{Genericity of the Hurewicz degree}
\label{sectionHurewiczGeneric}

In order to apply Proposition~\ref{proposition_char_closed_diff_form-Fuglede-Euclidean} to the \(\ell\)-differential form \(\eta \vcentcolon= u^\sharp \varpi\), where \(u\in \Sobolev^{1,\ell}(\Omega;\manfN)\),
we begin by establishing the genericity of the Hurewicz degree based on Fuglede maps:

\begin{lemma}\label{lemma-Fuglede-Hurewicz-VMO}
Let \(\varpi\in \Smooth^{\infty}(\manfN; \Forms^\ell)\) be a closed differential form.
Given \(u \in \Sobolev^{1, \ell}(\manfV; \manfN)\), there exists an \(\ell\)-detector \(w \colon  \manfV \to [0, +\infty]\) such that, for every \(\gamma \in \Fuglede_{w}(\Sphere^{\ell}; \manfV)\),  the \(\ell\)-differential form \(\gamma^{\sharp}(u^{\sharp}\varpi)\) is summable on \(\Sphere^\ell\) and the Hurewicz degree of the \(\VMO\) map \(u\compose \gamma\) satisfies
\begin{equation}\label{eq803}
\hur_{\varpi}{(u \compose \gamma)}
= \int_{\Sphere^{\ell}}\gamma^{\sharp}(u^{\sharp}\varpi).
\end{equation}
\end{lemma}

\begin{proof}
Since \(u \in \Sobolev^{1, \ell}(\manfV; \manfN)\), by Proposition~\ref{corollaryCompositionSobolevFuglede} there exists a summable function \(w \colon  \manfV \to [0, +\infty]\) such that, for every \(\gamma \in \Fuglede_{w}(\Sphere^{\ell}; \manfV)\), we have \(u \compose \gamma \in \Sobolev^{1, \ell}(\Sphere^{\ell}; \manfN)\). 
Since \(\Sobolev^{1, \ell}(\Sphere^{\ell}; \manfN)\subset \VMO(\Sphere^{\ell};\manfN)\), we have in particular that \(w\) is an \(\ell\)-detector. 
Moreover,  almost everywhere in \(\Sphere^{\ell}\), we have
\[
D(u \compose \gamma){}
= (Du \compose \gamma) [D\gamma]
\]
and thus, 
\[
(u \compose \gamma)^{\sharp}\varpi
= \gamma^{\sharp}(u^{\sharp}\varpi).
\]
This identity implies the first equality in \eqref{eq803} due to the integral formula in Proposition~\ref{propositionHurewiczVMO}. 
The second equality follows from the change of variables formula.
\end{proof}

We now have all the ingredients to prove the implication ``\(\Longleftarrow\)'' in Proposition~\ref{propositionJacobianDegreeZero-Fuglede}:

\begin{proof}[Proof of Proposition~\ref{propositionJacobianDegreeZero-Fuglede} ``\(\Longleftarrow\)'']
We denote by \(w_{1}\) the \(\ell\)-detector given by Lemma~\ref{lemma-Fuglede-Hurewicz-VMO} and by \(w_2\) the \(\ell\)-detector given by the assumption. 
Take \(\widetilde{w} = w_1 + w_2\) and let \(\gamma \colon \cBall^{\ell+1}\to \manfV\) be a Lipschitz map such that \(\gamma|_{\Sphere^\ell} \in \Fuglede_{\tilde w}(\Sphere^{\ell}; \manfV)\).
Since \(\widetilde{w} \ge w_1\), we have
\begin{equation}
\label{eqJacobian-969}
\hur_{\varpi}{(u \compose \gamma)}
= \int_{\Sphere^{\ell}}\gamma^{\sharp}(u^{\sharp}\varpi).
\end{equation}
Since \(\widetilde{w}\geq w_2\), we deduce from \eqref{eqJacobian-969} that
\[
\int_{\Sphere^{\ell}}\gamma^{\sharp}(u^{\sharp}\varpi)
= 0.
\]
The conclusion thus follows from Proposition~\ref{proposition_char_closed_diff_form-Fuglede} applied to the \(\ell\)-differential form \(\eta \vcentcolon= u^\sharp \varpi\).
\end{proof}

We now turn ourselves to the direct implication ``\(\Longrightarrow\)'' in Proposition~\ref{propositionJacobianDegreeZero-Fuglede}. 
To prove it, we begin with a counterpart of Proposition~\ref{lemmaModulusLebesguesequence} for differential forms:

\begin{proposition}
	\label{propositionFugledeForms}
	If \((\eta_{j})_{j \in \N}\) is a sequence of differential forms in \(\Lebesgue^{p}(\manfV; \Forms^{\ell})\) such that
	\[{}
	\eta_{j_{i}} \to \eta{}
	\quad \text{in \(\Lebesgue^{p}(\manfV; \Forms^{\ell})\),}
	\]
	then there exist a summable function \(w \colon  \manfV \to [0, +\infty]\) and a subsequence \((\eta_{j_{i}})_{i \in \N}\) such that, for every Riemannian manifold \(\manfA\) and every \(\gamma \in \Fuglede_{w}(\manfA; \manfV)\), we have that \(\gamma^{\sharp}\eta\) and  \((\gamma^{\sharp}\eta_{j_{i}})_{i \in \N}\) are contained in \(\Lebesgue^{p}(\manfA; \Forms^{\ell})\) and
	\begin{equation}
    \label{eqJacobian-976}
    \gamma^{\sharp}\eta_{j} \to \gamma^{\sharp}\eta{}
	\quad \text{in \(\Lebesgue^{p}(\manfA; \Forms^{\ell})\).}
	\end{equation}
\end{proposition}
\begin{proof}
We first consider the case when \(\manfV = \manfM\) is a compact manifold without boundary.
	Take a  finite family of open subsets \((V_{l})_{l \in J}\) that cover \(\manfM\) and, for each \(l \in J\), a diffeomorphism \(h_{l} \colon  V_{l} \to \R^{m}\).
    Let \((\theta_{l})_{l \in J}\) be a partition of unity subordinated to this covering.
	For each \(l \in J\),{}
	 since \(\theta_{l}\eta_j\) is supported in \(V_{l}\), we can write the differential form \((h_{l}^{-1})^{\sharp}(\theta_{l}\eta_{j})\) in \(\R^m\) as
	 \[{}
	 (h_{l}^{-1})^{\sharp}(\theta_{l}\eta_{j})
	 = \sum_{I} g_{l, j, I} \dif x_{I}\,,
	 \]
	 where \(g_{l, j, I} \in \Lebesgue^{p}(\R^m)\) vanishes outside \(h_l(\supp{\theta_{l}})\) and the sum is taken over all multi-indices \(I = (i_{1}, \ldots, i_{\ell})\) with \(1 \le i_{1} < \ldots < i_{\ell} \le m\).{}
	 As \(j \to \infty\), we have
	 \[{}
	 g_{l, j, I} \to g_{l, I}
	 \quad \text{in \(\Lebesgue^{p}(\R^m)\),}
	 \]
	 where the functions \(g_{l, I}\) are the coefficients of the \(\ell\)-differential form \((h_{l}^{-1})^{\sharp}(\theta_{l}\eta)\).{}
	 Applying recursively Proposition~\ref{lemmaModulusLebesguesequence} with respect to \(l \in J\) and the multi-indices \(I\), we find summable functions \(\widetilde w_{l} \colon  \R^m \to [0, +\infty]\) and an increasing sequence of  positive integers \((j_{i})_{i \in \N}\) such that, for every metric measure space \((X, d, \mu)\) and every \(\widetilde\gamma \in \Fuglede_{\tilde w_{l}}(X; \R^m)\), we have that \(g_{l, I} \compose \widetilde\gamma\) and \( (g_{l, j_{i}, I} \compose \widetilde\gamma)_{i \in \N} \) are contained in \(\Lebesgue^{p}(X)\) and satisfy
	 \begin{equation}
   \label{eqJacobian-998}
   g_{l, j_{i}, I} \compose \widetilde\gamma \to g_{l, I} \compose \widetilde\gamma{}
	 \quad \text{in \(\Lebesgue^{p}(X)\)}.
	 \end{equation}
	 For each \(l \in J\), the function
	 \[{}
	 w_{l}
	 \vcentcolon=
	 \begin{cases}
		 \widetilde w_{l} \compose h_{l}
		 & \text{in \(\supp{\theta_{l}}\),}\\
		 0
		 & \text{in \(\manfM \setminus \supp{\theta_{l}}\),}
	 \end{cases}
	 \]
  is summable in \( \manfM \). 
  In particular, the function \(w \vcentcolon= \sum\limits_{l \in J}w_{l}\) is also summable in \(\manfM\).{}
	 
	 Let \(\manfA\) be a Riemannian manifold and let \(\gamma \in \Fuglede_{w}(\manfA; \manfM)\).{}
	We introduce the sets \(X_{l} = \gamma^{-1}(\supp{\theta_{l}})\) and write for every \(l\in J\),
	 \begin{equation}
	 \label{eqJacobian-1042}
  \begin{split}
	 (\gamma|_{X_{l}})^{\sharp}(\theta_{l}\eta_{j_{i}})
	 & = \sum_{I} (\gamma|_{X_{l}})^{\sharp}h_{l}^{\sharp}( g_{l, j_{i}, I} \dif x_{I})\\
	 & = \sum_{I} g_{l, j_{i}, I} \compose h_{l} \compose \gamma|_{X_{l}} \, (h_{l} \compose \gamma|_{X_{l}})^{\sharp}(\dext x_{I}).
	 \end{split}
  \end{equation}
	 Since \(\gamma(X_{l}) \subset \supp{\theta_{l}} \), we have \(h_{l} \compose \gamma|_{X_{l}} \in \Fuglede_{\tilde w_{l}}(X_{l};\R^{m})\), which implies that \(g_{l, I} \compose h_{l} \compose \gamma|_{X_{l}} \) and \((g_{l, j_{i}, I} \compose h_{l} \compose \gamma|_{X_{l}})_{i \in \N}\) are contained in \(\Lebesgue^{p}(X_l)\) and, by \eqref{eqJacobian-998},
 \[
    g_{l, j_{i}, I} \compose h_{l} \compose \gamma|_{X_{l}} 
    \to g_{l, I} \compose h_{l} \compose \gamma|_{X_{l}}
    	 \quad \text{in \(\Lebesgue^{p}(X_l)\)}.
 \] 
  Hence, each term of the finite sum \eqref{eqJacobian-1042} belongs to \(\Lebesgue^{p}(X_l)\) and converges to the corresponding component of \((\gamma|_{X_l})^{\sharp}(\theta_l\eta)\) as \(i \to \infty\).   Thus, \(\gamma^{\sharp}(\theta_{l}\eta_{j_{i}})\) converges to \(\gamma^{\sharp}(\theta_{l}\eta)\) and by summing over \(l\in J\), we finally get  \eqref{eqJacobian-976}.

When \(\manfV = \Omega\) is an open set in \(\R^m\), we can repeat the above argument by considering the trivial covering of \(\Omega\), namely \(J=\{1\}\) and \(V_1=\Omega\). 
Then, \(h_1\) is replaced by the identity map from \(\Omega\) into itself and instead of \(\theta_1\), one can take the constant map equal to \(1\) on \(\Omega\). 
The rest of the proof is the same and we omit it.
\end{proof}

To complete the proof of Proposition~\ref{propositionJacobianDegreeZero-Fuglede}, we also need the following converse to Proposition~\ref{proposition_char_closed_diff_form-Fuglede}:

\begin{proposition}
\label{proposition_char_closed_diff_form-Fuglede-converse}
If \(\eta \in \Lebesgue^1 (\manfV; \Forms^\ell)\) is closed in the sense of currents in \(\manfV\), then there exists a summable function \(w \colon  \manfV \to [0, + \infty]\) such that, for every Lipschitz map \(\gamma \colon  \cBall^{\ell+1} \to \manfV\) with \(\gamma\vert_{\Sphere^\ell}\in \Fuglede_{w}(\Sphere^\ell ; \manfV)\), the \(\ell\)-differential form \((\gamma\vert_{\Sphere^\ell})^\sharp \eta \) is summable in \(\Sphere^{\ell}\) and
 \[
  \int_{\Sphere^\ell} (\gamma\vert_{\Sphere^\ell})^\sharp \eta = 0.
 \]
\end{proposition}

\begin{proof}
Since \(\eta\) is closed, there exists a sequence of closed differential forms \((\eta_j)_{j\in \N}\) in \((\Smooth^{\infty}\cap\Lebesgue^1)(\manfV; \Forms^{\ell})\) that converges to \(\eta\) in \(\Lebesgue^1(\manfV; \Forms^\ell)\), see e.g.\@ \citelist{\cite{Goldshtein-Troyanov-2006}*{Theorem~12.5} \cite{DeRham}}.
By Proposition~\ref{propositionFugledeForms}, there exist a summable function \(w\colon \manfV\to [0,+\infty]\) and a subsequence \((\eta_{j_{i}})_{i \in \N}\) such that, for every Lipschitz map \(\gamma\in \Fuglede_{w}(\Sphere^\ell; \manfV)\), we have \(\gamma^\sharp \eta \in \Lebesgue^{1}(\Sphere^{\ell}; \Forms^{\ell})\) and
\begin{equation}
\label{eqJacobian-1069}
\gamma^\sharp\eta_{j_{i}} \to \gamma^\sharp \eta 
\quad \text{in \(\Lebesgue^1 (\Sphere^\ell; \Lambda^{\ell})\).}
\end{equation}
Take a Lipschitz map \(\gamma \colon  \cBall^{\ell+1} \to \manfV\) with \(\gamma\vert_{\Sphere^\ell}\in \Fuglede_{w}(\Sphere^\ell ; \manfV)\).
Then, the \(\ell\)-differential form \((\gamma\vert_{\Sphere^\ell})^\sharp \eta \) is summable in \(\Sphere^{\ell}\). 
Moreover, since every \(\eta_{j_{i}}\) is closed, by Stokes's theorem we have
\[
\int_{\Sphere^\ell} (\gamma|_{\Sphere^{\ell}})^\sharp \eta_{j_{i}} 
= \int_{\Ball^{\ell+1}}\gamma^{\sharp}\dext \eta_{j_{i}}
=0.
\]
As \(i \to \infty\), we then deduce from \eqref{eqJacobian-1069} that 
\[
\int_{\Sphere^\ell} (\gamma|_{\Sphere^{\ell}})^\sharp \eta = 
  \lim_{i \to \infty}{\int_{\Sphere^\ell} (\gamma|_{\Sphere^{\ell}})^\sharp \eta_{j_{i}}} = 0.
  \qedhere
\]
\end{proof}

The direct implication in Proposition~\ref{propositionJacobianDegreeZero-Fuglede} now readily follows:

\begin{proof}[Proof of Proposition~\ref{propositionJacobianDegreeZero-Fuglede} ``\(\Longrightarrow\)'']
Let \(w_{1} \colon  \manfV \to [0, +\infty]\) be the \(\ell\)-detector given by Lemma~\ref{lemma-Fuglede-Hurewicz-VMO}.
Then, for every \(\gamma \in \Fuglede_{w_{1}}(\Sphere^{\ell}; \manfV)\),
\begin{equation}
\label{eq1210}
\hur_{\varpi}{(u \compose \gamma)}
= \int_{\Sphere^{\ell}}\gamma^{\sharp}(u^{\sharp}\varpi).
\end{equation}
Since we assume that \(\Hur_{\varpi}(u) = 0\) or, equivalently,  that \(u^{\sharp}\varpi\) is closed in the sense of currents in \(\manfV\), we may take a summable function \(w_2 \colon \manfV \to [0, +\infty]\) given by Proposition~\ref{proposition_char_closed_diff_form-Fuglede-converse} applied to \(\eta\vcentcolon =u^{\sharp}\varpi\). 
Let \(w = w_1 + w_2\).
Since \(w\geq w_2\), for every Lipschitz map \(\gamma \colon  \cBall^{\ell+1} \to \manfV\) with \(\gamma\vert_{\Sphere^\ell}\in \Fuglede_{w}(\Sphere^\ell ; \manfV)\), by Proposition~\ref{proposition_char_closed_diff_form-Fuglede-converse} the \(\ell\)-differential form  \((\gamma\vert_{\Sphere^\ell})^\sharp (u^{\sharp}\varpi) \) is summable in \(\Sphere^{\ell}\) and 
 \[
  \int_{\Sphere^\ell} (\gamma\vert_{\Sphere^\ell})^\sharp (u^{\sharp}\varpi) = 0.
 \]
Then, as \(w \ge w_1\), by \eqref{eq1210} this implies that
\[
\hur_{\varpi}(u \compose \gamma\vert_{\Sphere^\ell}) = 0
\] 
and gives the desired conclusion.
\end{proof}

From the direct implication in Proposition~\ref{propositionJacobianDegreeZero-Fuglede} one can derive a partial converse to Corollary~\ref{corollaryJacobianFugledeTrivial} under the topological assumption on \(\manfN\) that we introduced in Section~\ref{section_Hurewicz}:

\begin{corollary}
	\label{corollaryHurewiczNul} 
    Let \(m\geq \ell+1\) and assume that the Hurewicz degree identifies null-homotopic maps from \(\Sphere^\ell\) to \(\manfN\).
 	If \(u \in \Sobolev^{1, \ell}(\manfV; \manfN)\) and if, for every closed differential form \(\varpi \in \Smooth^{\infty}(\manfN; \Forms^{\ell})\),
	\[{}
	\Hur_{\varpi}{(u)}
	= 0
	\quad \text{in the sense of currents in \(\manfV\),}
	\]
 then \(u\) is \(\ell\)-extendable.
\end{corollary}

\begin{proof}
	Let \(w\) be an \(\ell\)-detector given by Proposition~\ref{propositionJacobianDegreeZero-Fuglede} and let \(\varpi \in \Smooth^{\infty}(\manfN; \Forms^{\ell})\) be a closed differential form.
	For every Lipschitz map \(\gamma \colon  \cBall^{\ell+1} \to \manfV\) with \(\gamma|_{\Sphere^{\ell}} \in \Fuglede_{w}(\Sphere^{\ell}; \manfV)\), we have
	\[{}
	\hur_{\varpi}{(u \compose \gamma|_{\Sphere^{\ell}})} = 0.
	\]
Applying Corollary~\ref{corollaryHurewiczHomotopicConstantConverse} to the map \(v \vcentcolon=u \compose \gamma|_{\Sphere^{\ell}}\), we can conclude that \(v\) is homotopic to a constant in \(\VMO(\Sphere^\ell; \manfN)\) and thus \(u\) is \(\ell\)-extendable.{}
\end{proof}

Combining Corollaries~\ref{corollaryJacobianFugledeTrivial} and~\ref{corollaryHurewiczNul}, we deduce Theorem~\ref{propositionJacobianSphereIntro} from the Introduction:

\begin{proof}[Proof of Theorem~\ref{propositionJacobianSphereIntro}]
Let \(u \in \Sobolev^{k, p}(\Ball^{m}; \Sphere^n)\).
Since \(\manfN\) is a compact subset of \(\R^\nu\), \(u\) is bounded and then by the Gagliardo-Nirenberg interpolation inequality we have \(u \in \Sobolev^{1, kp}(\Ball^{m}; \Sphere^n)\).
As \(n \le kp < n + 1 \le m\), we then deduce that \(\floor{kp} = n\) and, by Hölder's inequality, \(u \in \Sobolev^{1, n}(\Ball^{m}; \Sphere^n)\).
Assuming that \(u\) is \(\floor{kp}\)-extendable, then by Corollary~\ref{corollaryJacobianFugledeTrivial} with \(\ell \vcentcolon= n\), \(\varpi \vcentcolon= \omega_{\Sphere^n}\), \(\manfV \vcentcolon= \Ball^{m}\) and \(\manfN \vcentcolon= \Sphere^n\) we deduce that
\begin{equation}
\label{eqJacobian-1131}
\dext (u^{\#} \omega_{\Sphere^n}) 
=  \Hur_{\omega_{\Sphere^n}}{(u)}
	= 0
\quad \text{in the sense of currents in \(\Ball^m\),}
\end{equation}
Conversely, since by Example~\ref{exampleHopfDegree} the Hurewicz degree identifies hull-homotopic maps from \(\Sphere^n\) to \(\Sphere^n\), we know from Corollary~\ref{corollaryHurewiczNul} that, when \eqref{eqJacobian-1131} holds, \(u\) is \(n\)-extendable.
\end{proof}

It follows from Theorem~\ref{theoremhkpGreatBall} that \(\ell\)-extendability is equivalent to approximability by smooth maps in \(\Sobolev^{1, p}(\Ball^{m}; \manfN)\) for any \(\ell \leq p < \ell + 1 \leq m\). Consequently, combining Corollaries~\ref{corollaryJacobianFugledeTrivial} and~\ref{corollaryHurewiczNul} one gets

\begin{corollary}
\label{corollaryHurewiczDensityEquivalence}
    Let \(m\geq \ell+1\) and assume that the Hurewicz degree identifies null-homotopic maps from \(\Sphere^\ell\) to \(\manfN\).
 	Given \(u \in \Sobolev^{1, \ell}(\Ball^{m}; \manfN)\), one has \(u \in \Hilbert^{1, \ell}(\Ball^{m}; \manfN)\) if and only if, for every closed differential form \(\varpi \in \Smooth^{\infty}(\manfN; \Forms^{\ell})\),
	\[{}
	\Hur_{\varpi}{(u)}
	= 0
	\quad \text{in the sense of currents in \(\Ball^{m}\).}
	\]
\end{corollary}

The fact that the domain is not a general manifold is crucial here. A similar cohomological identification for \(\Hilbert^{1, \ell}(\Ball^m; \Sphere^\ell)\) has been previously established in \cites{BethuelH1, Demengel}, see also \cite{Giaquinta-Modica-Soucek-II} for an alternative approach using Cartesian currents. 
For more general manifolds \(\manfN\), the case where \(m = \ell + 1\) is treated in \cite{BethuelCoronDemengelHelein}*{Theorem~1}.
Note however that the converse of Corollary~\ref{corollaryHurewiczDensityEquivalence} is false for a general manifold \(\manfN\), see \cite{BethuelCoronDemengelHelein}*{Remark~f}:

\begin{example}
When the Hurewicz degree does not identify null-homotopic maps from \(\Sphere^\ell\) to \(\manfN\), there exists a map \(f \in \Smooth^\infty (\Sphere^\ell; \manfN)\) that is not homotopic to a constant in \(\Smooth^{0} (\Sphere^\ell; \manfN)\)  that satisfies \(\hur_{\varpi} (f) = 0\)
for every closed differential form \(\varpi \in \Smooth^\infty(\manfN; \Forms^\ell)\).
Since \(f\) cannot be continuously extended as a map from \(\cBall^{\ell + 1}\) to \(\manfN\), by Lemma~\ref{propositionExtensionHomogeneity} the map \(u\) defined by \eqref{eqLExtension-113} is not \(\ell\)-extendable, whence applying Proposition~\ref{propositionDensityExtendable} we see that \(u \not\in \Hilbert^{1, \ell}(\Ball^m; \Sphere^\ell)\).{}
	 On the other hand, by Proposition~\ref{lemmaJacobianDegree} we have
	 \[
  \Hur_{\varpi}{(u)} = (-1)^{\ell + 1} \hur_{\varpi}(f) \, \delta_0=0
  \]
  for any closed differential form \(\varpi \in \Smooth^{\infty}(\manfN; \Forms^{\ell})\).
\end{example}

The Hurewicz current allows one to substantially reduce the family of sets needed to check the extendability condition in the spirit of the criteria presented in Chapter~\ref{chapterExtendability}. 
Indeed, using Proposition~\ref{proposition_char_closed_diff_form-Fuglede-Euclidean} and Corollary~\ref{corollaryHurewiczNul}, we show that it suffices to consider restrictions of the detectors only on standard disks parallel to the coordinate axes in the sense of Definition~\ref{def-standard-disc}.

\begin{corollary}
\label{proposition-spheres-jacobian}
Let \(\Omega \subset \R^{m}\) be an open set and let \(\varpi\in \Smooth^{\infty}(\manfN;\Forms^\ell)\) be a closed differential form. 
	If \(u \in \Sobolev^{1, \ell}(\Omega; \manfN)\) and if there exists an \(\ell\)-detector \(w\)  such that, for every standard \((\ell + 1)\)-dimensional disk \(\Disk^{\ell+1} \subset \Omega\) parallel to the coordinate axes for which \(w|_{\partial\Disk^{\ell + 1}}\) is summable in \(\partial\Disk^{\ell + 1}\), the restriction \(u|_{\partial\Disk^{\ell + 1}}\) is homotopic to a constant in \(\VMO(\partial\Disk^{\ell + 1};\manfN)\), 
 	then
 \[
 \Hur_{\varpi}(u)=0
 \quad \text{in the sense of currents in \( \Omega \).}
 \]
 If in addition the Hurewicz degree identifies null-homotopic maps from \(\Sphere^\ell\) to \(\manfN\), then \( u \) is \( \ell \)-extendable.
\end{corollary}

\begin{proof}
Let \(w_1\) be the \(\ell\)-detector given by Lemma~\ref{lemma-Fuglede-Hurewicz-VMO} and take \(\widetilde{w} = w + w_1\), where \(w\) is the \(\ell\)-detector given by the assumption.
Let \(\Disk^{\ell + 1} \subset \Omega\) be a standard disk parallel to the coordinate axes for which \(\widetilde{w}|_{\partial\Disk^{\ell + 1}}\) is summable in \(\partial\Disk^{\ell + 1}\). 
Since  \(\widetilde{w} \ge w\), by assumption on \(u\) the restriction \(u|_{\partial\Disk^{\ell + 1}}\) is homotopic to a constant in \(\VMO(\partial\Disk^{\ell + 1};\manfN)\).

As in Definition~\ref{def-standard-disc}, we now write \(\Disk^{\ell + 1} = r T(\cBall^{\ell+1}) + \xi\) where \(r > 0\), \(\xi \in \R^m\) and \(T \colon \R^{\ell + 1} \to \R^m\) is a linear isometry.
We then consider the map \(\gamma \colon \cBall^{\ell+1} \to \Omega\) defined by \(\gamma(x) = r T(x)+\xi\). 
Then, \(\gamma|_{\Sphere^\ell}\) belongs to \(\Fuglede_{\tilde{w}}(\Sphere^\ell;\Omega)\).
Since \(\gamma|_{\Sphere^\ell}\) is a diffeomorphism between \(\Sphere^\ell\) and \(\partial \Disk^{\ell+1}\), we deduce that \(u\compose \gamma|_{\Sphere^\ell}\) is homotopic to a constant in \(\VMO(\Sphere^{\ell};\manfN)\). 
By Corollary~\ref{corollaryHurewiczHomotopicConstant}, it follows that 
\[
\hur_{\varpi}(u\compose \gamma|_{\Sphere^{\ell}})=0.
\]
Since \(\widetilde{w} \ge w_{1}\) and \(\gamma(\Sphere^\ell)=\partial\Disk^{\ell+1}\), we deduce from Lemma~\ref{lemma-Fuglede-Hurewicz-VMO} and the change of variable formula that 
\[ 
\int_{\partial\Disk^{\ell+1}}u^{\sharp}\varpi=0.
\]
In view of Proposition~\ref{proposition_char_closed_diff_form-Fuglede-Euclidean}, this implies that \(u^{\sharp}\varpi\) is closed or, equivalently, \(\Hur_{\varpi}(u)=0\) in the sense of currents in \(\Omega\). 
To conclude, when the Hurewicz degree identifies null-homotopic maps from \(\Sphere^\ell\) to \(\manfN\), we deduce from Corollary~\ref{corollaryHurewiczNul} that \( u \) is \( \ell \)-extendable.
\end{proof}

We do not know whether the last conclusion of Corollary~\ref{proposition-spheres-jacobian} is true without the additional topological assumption of \(\manfN\). 
Note, however, that the conclusion holds when parallel disks are replaced by generic simplices, see Proposition~\ref{propositionCharacterizationExtensionSimplex}.

\begin{remark}
\label{remark-Canevari-Orlandi}
When \(\pi_0 (\manfN)\simeq \ldots \simeq \pi_{\ell - 1} (\manfN) \simeq \{0\}\) and  \(\pi_\ell (\manfN)\) is abelian,  it is possible to construct for every \(u\in \Sobolev^{1,\ell}(\Ball^m; \manfN)\) an \((m-\ell-1)\)-dimensional flat chain \( \F_{m-\ell-1}(\Ball^m, \pi_{\ell}(\manfN)) \) that identifies the topological singularities of \(u\), without any further restriction on \(\pi_{\ell}(\manfN)\), see \cites{Pakzad-Riviere, Canevari-Orlandi}. 
More precisely, given \(\ell \le p < \ell+1\),  there exists a continuous map
\[
S^{PR}\colon \Sobolev^{1,p}(\Ball^m ; \manfN) \to \F_{m-\ell-1}(\Ball^m, \pi_{\ell}(\manfN))
\]
such that \(S^{PR}(u)=0\) if and only if \(u\in \Hilbert^{1,p}(\Ball^m;\manfN)\), see \cite{Canevari-Orlandi}*{Theorem~1}. 
When \(\manfN=\Sphere^\ell\), the operator \(S^{PR}(u)\) coincides with the distributional Jacobian \(\jac{u}\) introduced in \eqref{eq-def-jacobian}. 
In that specific situation and when \(p=\ell\), the range of \(S^{PR}(u)\) is the set  of boundaries of integer multiplicity rectifiable currents with finite mass, see \cite{Alberti-Baldo-Orlandi}. Determining the range of \(\jac \) in other settings is widely open, see however \cites{Bourgain-Brezis-Mironescu-2005, Bousquet-Mironescu, Bousquet-2007, Canevari-Orlandi-2}. 
\end{remark}

\section{Hopf invariants}
\label{section_Hopf}

We know from Example~\ref{exampleHopfFibration} that the Hurewicz degree does not identify null-homotopic maps from \( \Sphere^{3} \) to \( \Sphere^{2} \).
A more adapted topological invariant in this case is the Hopf degree~\cite{Hopf1931} which can be defined more generally for maps \( v \in \Smooth^{\infty}(\Sphere^{2n-1}; \Sphere^{n}) \) with \( n \ge 2 \) using Whitehead's integral formula~\cite{Whitehead1947}:
	 \begin{equation}
		\label{eqJacobian-1500}
	   \hopf{v} = \int_{\Sphere^{2 n - 1}} \chi \wedge v^\sharp \omega_{\Sphere^n},
	 \end{equation}
where  the differential form \(\chi \in  \Smooth^{\infty}(\Sphere^{2n-1}; \Forms^{n - 1})\) satisfies 
\begin{equation}
\label{eqJacobianChi}
\dext\chi = v^\sharp \omega_{\Sphere^n}
\quad \text{in \(\Sphere^{2 n - 1}\).}
\end{equation}
The existence of \( \chi \) is a consequence of the facts that \( v^\sharp \omega_{\Sphere^n} \) is closed and \(H^{n}_{\mathrm{dR}} (\Sphere^{2 n - 1})\) is trivial.

One shows that the value of the integral in \eqref{eqJacobian-1500} is always an integer does not depend on the choice of \(\chi\) and if \(v_{0}, v_{1}\in \Smooth^\infty (\Sphere^{2 n - 1}; \Sphere^{n})\) are homotopic in \(\Smooth^0(\Sphere^{2 n - 1}; \Sphere^{n})\), then 
\begin{equation}
\label{eqHopf-1289}
\hopf{v_{0}} = \hopf{v_{1}},
\end{equation}
see \cite{BottTu1982}*{Proposition~17.22}.
As a result, the map \(\hopf\) induces a homomorphism from \(\pi_{2n-1}(\Sphere^n)\) to \(\Z\).

When \(n = 2\) and \(v \colon  \Sphere^{3} \to \Sphere^{2}\) is the classical Hopf fibration, one has \(\hopf{v} = 1\) and, in this case, \(\hopf\) induces a group isomorphism between \(\pi_{3}(\Sphere^{2})\) and \(\Z\).
In contrast, when \(n\) is odd, one sees that
\[{}
\chi \wedge v^\sharp \omega_{\Sphere^n} 
= \chi \wedge \dext \chi = \dext (\chi \wedge \chi)/2
\]
and, as a consequence of Stokes's theorem, the Hopf degree of every smooth map is trivial.

As for the Hurewicz degree, the Hopf degree can be extended to VMO maps:

\begin{proposition}
\label{propositionSmoothHopfInvariant-VMO}
	Let \( n \ge 2 \).
    The Hopf degree extends as a continuous function \(\hopf \colon  \VMO(\Sphere^{2 n - 1}; \Sphere^n) \to \Z\) so that \eqref{eqHopf-1289} holds for all \(v_{0}, v_{1} \in \VMO(\Sphere^{2 n - 1}; \Sphere^{n})\) that are homotopic in \(\VMO(\Sphere^{2 n - 1}; \Sphere^{n})\) and this extension \(\hopf\) satisfies Whitehead's integral formula \eqref{eqJacobian-1500} for every \(v \in \Sobolev^{1, p} (\Sphere^{2 n - 1}; \Sphere^n)\) with \(p \ge 2n - 1\), where the differential form \(\chi \in \Lebesgue^\frac{p}{p - n} (\Sphere^{2 n - 1}; \Forms^{n - 1})\)  verifies \eqref{eqJacobianChi} in the sense of currents.
\end{proposition}

\resetconstant
\begin{proof}
	The continuous extension of \(\hopf\) to \(\VMO(\Sphere^{2 n - 1}; \Sphere^n)\) is performed as we did for \(\hur_{\varpi}\).
	To this end, take \(v \in \VMO(\Sphere^{2 n - 1}; \Sphere^n)\) and a sequence \((u_{j})_{j \in \N}\) in \(\Smooth^{\infty}(\Sphere^{2 n - 1}; \Sphere^n)\) that converges to \(v\) in \(\VMO(\Sphere^{2 n - 1}; \R^{n + 1})\).
    By Proposition~\ref{propositionHomotopyVMOLimit}, these maps \(u_{j}\) are homotopic in \(\VMO(\Sphere^{2 n - 1}; \Sphere^n)\) for every \(j \ge J\), where \(J \in \N\) is sufficiently large.
    Since they are smooth, by Proposition~\ref{propositionHomotopyVMOtoC} they are also homotopic in \(\Smooth^{0}(\Sphere^{2 n - 1}; \Sphere^n)\) and then, by invariance of the Hopf degree, the number \(\hopf{u_{j}}\) is independent of \(j \ge J\).
	We thus set
	\[{}
	\hopf{v} = \hopf{u_{j}}
	\quad \text{for any \(j \ge J\).}
	\]
	One shows that the Hopf degree is well defined and continuous in \(\VMO(\Sphere^{2 n - 1}; \Sphere^n)\).
	
	We now prove the validity of the integral formula when \(v \in \Sobolev^{1, p} (\Sphere^{2 n - 1}; \Sphere^n)\) and  \( p \ge 2n - 1\).
    To this end, by the counterpart of Theorem~\ref{theoremSchoen-Uhlenbeck} for maps defined on \( \Sphere^{2n - 1} \), see \cite{Schoen-Uhlenbeck}, we may take a sequence of maps \((u_j)_{j \in \N} \) in \( \Smooth^\infty (\Sphere^{2 n - 1}; \Sphere^n)\) such that \(u_{j} \to v\) in \(\Sobolev^{1, p} (\Sphere^{2 n - 1}; \R^{n + 1})\).{}
	Since \(p \ge 2n - 1\), the convergence also holds in VMO and 
\begin{equation}
	\label{eqHopfVMO}
\lim_{j \to \infty}{\hopf{u_j}} 
= \hopf{v}.
\end{equation}
For each \(j \in \N\), we consider \(\chi_j \in \Smooth^\infty (\Sphere^{2 n - 1}; \Forms^{n - 1})\) satisfying the Hodge system 
\begin{equation}
\label{problemHopfBeta}
 \left\{
 \begin{aligned}
   \dext {\chi}_j &= u_j^\sharp \omega_{\Sphere^n} & & \text{in \(\Sphere^{2 n - 1}\)},\\
   \dco {\chi}_j &= 0  & & \text{in \(\Sphere^{2 n - 1}\)},
 \end{aligned}
 \right.
\end{equation}
where \(\dco\) is the coexterior differentiation.
For the existence of \(\chi_j\)\,, we refer e.g.\@ to \cite{Morrey}*{Section~7.4}.
By classical regularity estimates \cite{Scott1995}*{Proposition 5.9}, one has
\[
 \norm{{\chi}_j}_{\Sobolev^{1, {p}/{n}}(\Sphere^{2 n - 1})}
 \le \C \norm{u_j^\sharp \omega_{\Sphere^n}}_{\Lebesgue^{{p}/{n}}(\Sphere^{2 n - 1})}
 \le \C \norm{Du_j}_{\Lebesgue^{p}(\Sphere^{2 n - 1})}^{n}.
\]
Since \(p \ge 2 n - 1\), we have \(\frac{p - n}{p} \ge \frac{n}{p} - \frac{1}{2n - 1}\), and by the Sobolev inequality we deduce that \(\Sobolev^{1, \frac{p}{n}}(\Sphere^{2n-1})\) imbeds into \(\Lebesgue^{\frac{p}{p-n}}(\Sphere^{2n-1})\).
We can thus assume, by considering a subsequence if necessary, that the sequence \((\chi_j)_{j \in \N}\) converges weakly in \(\Lebesgue^{\frac{p}{p-n}}\)  to some \({\chi} \in \Sobolev^{1, \frac{p}{n}} (\Sphere^{2 n - 1}; \Forms^{n - 1})\). 
Since \(u_j^\sharp \omega_{\Sphere^n}\) converges strongly in \(\Lebesgue^{\frac{p}{n}}\) to \(v^\sharp \omega_{\Sphere^n}\), we then have 
\[
 \left\{
 \begin{aligned}
   \dext {\chi} &= v^\sharp \omega_{\Sphere^n} & & \text{in \(\Sphere^{2 n - 1}\)},\\
   \dco {\chi} &=  0 & & \text{in \(\Sphere^{2 n - 1}\)},
 \end{aligned}
 \right. 
\]
and
\[
\label{eqHopfTildeBeta}
 \lim_{j \to \infty} \int_{\Sphere^{2 n - 1}} {\chi}_j \wedge u_j^\sharp \omega_{\Sphere^n}
=   \int_{\Sphere^{2 n - 1}} {\chi} \wedge v^\sharp \omega_{\Sphere^n}.
\]
Combining \eqref{eqHopfVMO} and the integral formula \eqref{eqJacobian-1500} for \( u_{j} \), we deduce the integral formula for \( v \).
\end{proof}

The Hopf fibration is responsible for local obstructions to the density of smooth maps in \(\Sobolev^{1, p} (\Ball^m; \Sphere^2)\) with \(3 \le p < 4\). 
In this section, our aim is to define a Hopf current not only to detect these specific obstructions but also to help identify elements in \( \Hilbert^{1, p}(\Ball^{m}; \Sphere^{2}) \) in situations where the distributional Jacobian is inadequate, see Corollary~\ref{corollaryHopfIsobe}.
We consider more generally \( \Sphere^{n} \)-valued maps and assume throughout this section that \( n \ge 2 \).
For analytical reasons, we first introduce the notion of Hopf-friendly Sobolev maps.

\begin{definition}
	Given \(p > n\), we say that a map \(u \in \Sobolev^{1, p} (\manfV; \Sphere^n)\) is \emph{Hopf-friendly} whenever there exists 
	a differential form \(\chi \in \Lebesgue^{\frac{p}{p - n}} (\manfV; \Forms^{n - 1})\) such that 
	\begin{equation}
	\label{eqJacobian-Chi}
		\dext \chi = u^\sharp \omega_{\Sphere^n}
		\quad \text{in the sense of currents in \(\manfV\).}
	\end{equation}
\end{definition}

The issue behind this concept is twofold: 
(1) the existence of \(\chi\) that requires the \( n \)-differential form \(u^\sharp \omega_{\Sphere^n}\) to be closed and (2) a suitable integrability of \(\chi\) which need not follow from the fact that \(u^\sharp \omega_{\Sphere^n} \in  \Lebesgue^{\frac{p}{n}} (\manfV; \Forms^{n})\).

\begin{proposition}
    \label{example-hopf-friendly}
    If \( m \ge 3 \) and \(p \ge \frac{2mn}{m+1}\), then every map \(u \in \Sobolev^{1, p} (\Ball^{m}; \Sphere^n)\) is Hopf-friendly.
\end{proposition}

In particular, when \( m = 4 \) and \( n = 2 \) one can take any exponent \( p \ge 16/5 \).
We first need the following property about the distributional Jacobian, which ensures that \(u^\sharp \omega_{\Sphere^n}\) is closed:

\begin{lemma}
    \label{propositionHurewiczTrivial}
    If \( p \ge n + 1 \), then every \( u \in \Sobolev^{1, p}(\manfV; \Sphere^{n}) \) satisfies
    \[
    \jac{u} = 0
    \quad \text{in the sense of currents in \(\manfV\).}
    \]
\end{lemma}

\begin{proof}[Proof of Lemma~\ref{propositionHurewiczTrivial}]
	Take an extension \(\varpi \in \Smooth^{\infty}(\R^{n + 1}; \Forms^{n})\) of \(\omega_{\Sphere^{n}}\).{}
	For almost every \(x \in \manfV\) and every \(\xi \in T_{x}\manfV\), we have \(u(x) \in \Sphere^{n}\) and \(Du(x)[\xi] \in T_{u(x)}\Sphere^{n}\).{}
    Moreover,
    \begin{equation}
    \label{eqFormsExtension}
    u^{\sharp}\varpi = u^{\sharp}\omega_{\Sphere^{n}} {}
	\quad \text{and} \quad{}
	u^{\sharp} \dext\varpi = u^{\sharp} \dext\omega_{\Sphere^{n}} = 0{}
	\quad \text{almost everywhere in \(\manfV\).}
    \end{equation}

    Let \( \alpha \in \Smooth_{c}^{\infty}(\manfV; \Forms^{m - n - 1})\). 
    We take a sequence \((f_{j})_{j \in \N}\) in \( (\Sobolev^{1, p} \cap \Smooth^{\infty})(\manfV; \R^{n + 1}) \) such that \(f_{j} \to u\) in \(\Sobolev^{1, p}(\manfV; \R^{n + 1})\).{}
    By Stokes's theorem, for every \( j \in \N \) we have 
 	\begin{equation}
	\label{eqJacobian-1151}
	  \int_{\manfV} f_{j}^{\sharp}\varpi \wedge \dext\alpha
	= (-1)^{n + 1} \int_{\manfV} f_{j}^{\sharp}\dext\varpi \wedge \alpha.
	\end{equation}   
    Since \( p \ge n + 1 \), we have
    \[
    f_{j}^{\sharp}\varpi \to u^{\sharp}\varpi
    \quad \text{in \(\Lebesgue^{1}\loc(\manfV; \Forms^{n})\)}
    \]
    and
    \[
    f_{j}^{\sharp}\dext\varpi \to u^{\sharp}\dext\varpi
    \quad \text{in \(\Lebesgue^{1}\loc(\manfV; \Forms^{n + 1})\)}.
    \]
    As \( j \to \infty \) in \eqref{eqJacobian-1151}, we get
    \[
    \int_{\manfV} u^{\sharp}\varpi \wedge \dext\alpha
	= (-1)^{n + 1} \int_{\manfV} u^{\sharp}\dext\varpi \wedge \alpha.
    \]
    Since \eqref{eqFormsExtension} holds, the integrand in the right-hand side equals zero almost everywhere.
    As the integral in the left-hand side equals \( \langle \jac{u}, \alpha \rangle  \), the conclusion follows.
\end{proof}

\begin{proof}[Proof of Proposition~\ref{example-hopf-friendly}]
 Since \( (m-1)(n-1) \ge 2 \), we have \(p \ge \frac{2mn}{m+1} \ge n + 1\).
   Lemma~\ref{propositionHurewiczTrivial} then implies that \(u^\sharp \omega_{\Sphere^n}\) is closed in the sense of currents in \(\Ball^{m}\). 
   Hence, the Hodge system
\[
 \left\{
 \begin{aligned}
   \dext {\chi} &= u^\sharp \omega_{\Sphere^n} & & \text{in \(\Ball^{m}\)},\\
   \dco {\chi} &=  0 & & \text{in \(\Ball^{m}\)},
 \end{aligned}
 \right. 
\]
    has a solution.
    Since \(u^\sharp \omega_{\Sphere^n} \in \Lebesgue^{\frac{p}{n}}(\Ball^m; \Forms^{n})\), we can require that \(\chi \in \Sobolev^{1, \frac{p}{n}} (\Ball^m; \Forms^{n - 1})\).{}
	When \(p/n \ge m\), we then deduce that \(\chi \in \Lebesgue^{q} (\Ball^m; \Forms^{n - 1})\) for every \(1 \le q < \infty\).
	Assuming that \(p/n < m\), by the Sobolev imbedding we get \(\chi \in \Lebesgue^{\frac{mp}{mn - p}} (\Ball^m; \Forms^{n - 1} )\).{}
	Since \(p \ge \frac{2mn}{m+1}\), one has \(\frac{mp}{mn - p} \ge \frac{p}{p - n}\) and the conclusion follows.{}
\end{proof}

We now introduce the Hopf current associated to a Hopf-friendly map:

\begin{definition}
    \label{definitionHopfCurrent}
    Let \(m \ge 2n\) and \(p > n\).
	Given a Hopf-friendly map \(u \in \Sobolev^{1, p} (\manfV; \Sphere^n)\) and a differential form \(\chi \in \Lebesgue^{\frac{p}{p - n}} (\manfV; \Forms^{n - 1})\) that satisfies \eqref{eqJacobian-Chi}, then \(\chi \wedge u^\sharp\omega_{\Sphere^n} \in \Lebesgue^{1}(\manfV; \Forms^{2n - 1})\) and we may define the \emph{Hopf current} of \(u\) as
\[{}
\dualprod{\Hopf{u}}{\alpha}
= \int_{\manfV} \chi \wedge u^\sharp\omega_{\Sphere^n} \wedge \dext\alpha{}
\quad \text{for every \(\alpha \in \Smooth_{c}^{\infty}(\manfV; \Forms^{m - 2n})\).}
\]
\end{definition}

When \(p \geq n+1\), this definition does not depend on the choice of \(\chi\):

\begin{proposition}
	\label{lemmaHopfIndependence}
	Let \(m \ge 2n\).
    If \(p \geq n+1\) and  \(u \in \Sobolev^{1, p} (\manfV; \Sphere^n)\), then  \( \Hopf{u} \) is independent of the differential form \( \chi \in \Lebesgue^{\frac{p}{p - n}} (\manfV; \Forms^{n - 1}) \) that satisfies \eqref{eqJacobian-Chi}.{}
\end{proposition}

\begin{proof}
	Take \(\widetilde\chi \in \Lebesgue^{\frac{p}{p - n}} (\manfV; \Forms^{n - 1})\) that also verifies \eqref{eqJacobian-Chi}.{}
	In this case, since \(\chi - \widetilde\chi\) is closed, by \cite{Goldshtein-Troyanov-2006}*{Theorem~12.5} there exists a sequence \((\lambda_{j})_{j \in \N}\) of smooth closed \( (n-1) \)-differential forms that converge in \(\Lebesgue^{\frac{p}{p - n}}\) to \(\chi - \widetilde\chi\).{}
	Given \(\alpha \in \Smooth_{c}^{\infty}(\manfV; \Forms^{m - 2n})\), by closedness of \( \lambda_{j} \) we have
 \[
 \dext(\lambda_{j} \wedge \alpha) 
 = (-1)^{n - 1} \lambda_{j} \wedge \dext\alpha, 
 \]
 where \(\lambda_{j} \wedge \dext\alpha \in \Smooth_{c}^{\infty}(\manfV; \Forms^{m - n})\).
 Thus,
 \[
 \int_{\manfV} \lambda_{j} \wedge u^\sharp \omega_{\Sphere^n} \wedge \dext\alpha{}
	 = 
	(-1)^{n - 1} \int_{\manfV} u^\sharp \omega_{\Sphere^n} \wedge \dext(\lambda_{j} \wedge \alpha){}.
 \]
 	Since \(p \ge n + 1\), it follows from Lemma~\ref{propositionHurewiczTrivial} that \( u^\sharp \omega_{\Sphere^n} \) is closed and we deduce that
	\[{}
 \int_{\manfV} \lambda_{j} \wedge u^\sharp \omega_{\Sphere^n} \wedge \dext\alpha{}
	 = 0.
	\]
 	Thus, as \(j \to \infty\), we get
	\[{}
	\int_{\manfV} (\chi - \widetilde\chi) \wedge u^\sharp \omega_{\Sphere^n} \wedge \dext\alpha{}
	= 0.
	\qedhere
	\]
\end{proof}

	A counterpart of Proposition~\ref{propositionJacobianDegreeZero-Fuglede} for the Hopf current involves the generic composition of Hopf-friendly maps with Lipschitz maps from \(\Sphere^{2n - 1}\) to \(\Sphere^{n}\):

\begin{proposition}
\label{propositionHopfCondition}
Let \( m \ge 2n \) and \(p \ge 2n - 1\). 
If a map \(u \in \Sobolev^{1, p} (\manfV; \Sphere^{n})\) is Hopf-friendly, then  
	\begin{equation}
	\label{eqJacobian-1157}
	\Hopf{u} = 0
	\quad \text{in the sense of currents in \(\manfV\)}
	\end{equation}
if and only if there exists a \((2n - 1)\)-detector \(w\) such that, for every Lipschitz map \(\gamma \colon  \cBall^{2n} \to \manfV\) with \(\gamma\vert_{\Sphere^{2n-1}}\in \Fuglede_{w}(\Sphere^{2n-1}; \manfV)\), we have
\[{}
\hopf{(u \compose \gamma\vert_{\Sphere^{2n-1}})} = 0.
\]
\end{proposition}

The proof of Proposition~\ref{propositionHopfCondition} relies on the genericity of the Hopf degree in the spirit of Lemma~\ref{lemma-Fuglede-Hurewicz-VMO}.
To this end, we first show that the exactness of a differential form is preserved by composition with Fuglede maps.

\begin{lemma}
	\label{lemmaFugledeFormsPullBack}
Given  \(p, q\geq 1\), let \(\chi \in \Lebesgue^{p}(\manfV; \Lambda^{\ell})\) and \(g \in \Lebesgue^{q}(\manfV; \Lambda^{\ell + 1})\) be such that
	\[{}
	\dext\chi = g
	\quad \text{in the sense of currents in \(\manfV\).}
	\]
	Then, there exists a summable function \(w \colon  \manfV \to [0, +\infty]\) such that, for every \(r \in \N_*\) and every \(\gamma \in \Fuglede_{w}(\Sphere^{r}; \manfV)\), we have \(\gamma^{\sharp}\chi \in \Lebesgue^{p}(\Sphere^{r}; \Lambda^{\ell})\), \(\gamma^{\sharp}g \in \Lebesgue^{q}(\Sphere^{r}; \Lambda^{\ell + 1})\) and
	\[{}
	\dext(\gamma^{\sharp}\chi) = \gamma^{\sharp}g
	\quad \text{in the sense of currents in \(\Sphere^{r}\).}
	\]
\end{lemma}

\begin{proof}[Proof of Lemma~\ref{lemmaFugledeFormsPullBack}]
	As in \cite{IwaniecScottStroffolini}*{Corollary~3.6}, there exists a sequence \((\chi_{j})_{j \in \N}\) in \(\Smooth^{\infty}(\manfV; \Forms^{\ell})\) such that
	\[{}
	\chi_{j} \to \chi{}
	\quad \text{in \(\Lebesgue^{p}(\manfV; \Lambda^{\ell})\),} \quad{}
	\dext\chi_{j} \to g{}
	\quad \text{in \(\Lebesgue^{q}(\manfV; \Lambda^{\ell + 1})\).}
	\]
	We then apply Proposition~\ref{propositionFugledeForms} to both sequences \((\chi_{j})_{j \in \N}\) and \((\dext\chi_{j})_{j \in \N}\) and denote by \(w \colon  \manfV \to [0, +\infty]\) the summable function and \((j_{i})_{i \in \N}\) the indices of the subsequence thus obtained.
	Take \(\gamma \in \Fuglede_{w}(\Sphere^{r}; \manfV)\).{}
	Since 
    \[
    \dext(\gamma^{\sharp}\chi_{j_{i}}) = \gamma^{\sharp}\dext{\chi_{j_{i}}}
    \quad \text{almost everywhere in \(\Sphere^{r}\),}
    \]
    we have by Stokes's theorem that, for every \(\alpha \in \Smooth^{\infty}(\Sphere^{r}; \Forms^{r - \ell - 1})\),
	\[{}
	\int_{\Sphere^{r}} \gamma^{\sharp}\chi_{j_{i}} \wedge \dext\alpha 
	= (-1)^{\ell + 1} \int_{\Sphere^{r}} \gamma^{\sharp}\dext\chi_{j_{i}} \wedge \alpha.
	\]
	Then, as \(i \to \infty\),{}
	\[{}
	\int_{\Sphere^{r}} \gamma^{\sharp}\chi \wedge \dext\alpha 
	= (-1)^{\ell + 1} \int_{\Sphere^{r}} \gamma^{\sharp}g \wedge \alpha.{}
	\]
	Since \(\alpha\) is arbitrary, the conclusion follows.	
\end{proof}

\begin{proof}[Proof of Proposition~\ref{propositionHopfCondition}]
Since \( u \) is Hopf-friendly, there exists a differential form \(\chi \in \Lebesgue^{\frac{p}{p - n}}(\manfV; \Forms^{n - 1})\) that satisfies \eqref{eqJacobian-Chi}.
Since \( p \ge 2n-1 \), by the inclusion \eqref{eqDetector-775} we have that  \( u \) is a \( \VMO^{2n-1} \)~map.
We prove that there exists a \((2n - 1)\)-detector \({w}\) such that, for every \(\gamma \in \Fuglede_{ w}(\Sphere^{2n-1}; \manfV)\), we have \(\gamma^{\sharp}(\chi \wedge u^\sharp \omega_{\Sphere^n}) \in \Lebesgue^{1}(\Sphere^{2n - 1}; \Forms^{2n-1})\) and
\begin{equation}
	\label{eqJacobian-1645}
\hopf{(u \compose \gamma)} 
= \int_{\Sphere^{2 n - 1}} \gamma^{\sharp}(\chi \wedge u^\sharp \omega_{\Sphere^n}).
\end{equation}

By \eqref{eqJacobian-Chi} and Lemma~\ref{lemmaFugledeFormsPullBack}, there exists a summable function \(w_{1} \colon  \manfV \to [0, +\infty]\) such that, for every \(\gamma \in \Fuglede_{w_{1}}(\Sphere^{2n - 1}; \manfV)\), we have \(\gamma^{\sharp}\chi \in \Lebesgue^{\frac{p}{p - n}}(\Sphere^{2n - 1}; \Lambda^{n - 1})\), \(\gamma^{\sharp}(u^{\sharp}\omega_{\Sphere^{n}}) \in \Lebesgue^{\frac{p}{n}}(\Sphere^{2n - 1}; \Lambda^{n})\) and
	\begin{equation}
\label{eqJacobian-1651}
	\dext(\gamma^{\sharp}\chi) = \gamma^{\sharp}(u^{\sharp}\omega_{\Sphere^{n}})
	\quad \text{in the sense of currents in \(\Sphere^{2n - 1}\).}
	\end{equation}
Since \(u \in \Sobolev^{1, p}(\manfV; \Sphere^{n})\), by Proposition~\ref{corollaryCompositionSobolevFuglede} there exists a summable function \(w_{2} \colon  \manfV\to [0,+ \infty]\) such that, for every \(\gamma \in \Fuglede_{w_{2}}(\Sphere^{2 n - 1}; \manfV)\),  we have \(u \compose \gamma \in \Sobolev^{1, p} (\Sphere^{2 n - 1}; \Sphere^n)\) and 
\[{}
D (u \compose \gamma) = (D u \compose \gamma) [D \gamma]
\quad \text{almost everywhere in \(\Sphere^{2n - 1}\).}
\]
Thus,
\begin{equation}
\label{eqJacobian-1662}
(u \compose \gamma)^\sharp \omega_{\Sphere^n} 
= \gamma^{\sharp}(u^\sharp \omega_{\Sphere^n})
\quad \text{almost everywhere in \(\Sphere^{2n - 1}\).}
\end{equation}
Let \(w_{3}\) be a \((2n - 1)\)-detector of \( u \) and take \(w = w_{1} + w_{2} + w_{3}\).{}
For \(\gamma \in \Fuglede_{{w}}(\Sphere^{2 n - 1}; \manfV)\), we may combine \eqref{eqJacobian-1651} and \eqref{eqJacobian-1662}
to get
\[{}
\dext(\gamma^{\sharp}\chi) 
= \gamma^{\sharp}(u^\sharp \omega_{\Sphere^n})
= (u \compose \gamma)^{\sharp}\omega_{\Sphere^n}
\quad \text{in the sense of currents in \(\Sphere^{2 n - 1}\).}
\]
Since \(u \compose \gamma \in \Sobolev^{1, p}(\Sphere^{2n - 1}; \Sphere^{n})\) with \(p \ge 2n - 1\), by Proposition~\ref{propositionSmoothHopfInvariant-VMO} we then get
\[{}
\hopf{(u \compose \gamma)} 
= \int_{\Sphere^{2 n - 1}} \gamma^{\sharp}\chi \wedge \gamma^{\sharp}(u^\sharp \omega_{\Sphere^n}).
\]
As we have
\[{}
\gamma^{\sharp}\chi \wedge \gamma^{\sharp}(u^\sharp \omega_{\Sphere^n})
= \gamma^{\sharp}(\chi \wedge u^\sharp \omega_{\Sphere^n})
\quad \text{almost everywhere in \(\Sphere^{2n - 1}\),}
\]
identity \eqref{eqJacobian-1645} follows.

To conclude the proof of the proposition, note that \(\Hopf{u} = 0\) if and only if \(\chi \wedge u^{\sharp}\omega_{\Sphere^n}\) is closed. From here, we apply Proposition~\ref{proposition_char_closed_diff_form-Fuglede-converse} to this \((2n-1)\)-differential form to establish the implication ``\(\Longrightarrow\)'', and use Proposition~\ref{proposition_char_closed_diff_form-Fuglede} for the reverse implication ``\(\Longleftarrow\)''.
\end{proof}

As a consequence of Proposition~\ref{propositionHopfCondition}, we obtain a counterpart of Corollary~\ref{corollaryJacobianFugledeTrivial} for the Hopf current:
\begin{corollary}
	\label{corollaryHopfFugledeTrivial}
	Let \(p \ge 2n - 1\) and let \(u \in \Sobolev^{1, p}(\manfV; \Sphere^{n})\) be a Hopf-friendly map. 
 If \(u\)  is \((2n-1)\)-extendable, then
	\[{}
	\Hopf{u}
	= 0 \quad \text{in the sense of currents in \(\manfV\).}
	\]
 The converse is true whenever \(\hopf\) induces an injective homomorphism from \(\pi_{2n-1}(\Sphere^{n})\) into \(\Z\).
\end{corollary}

\begin{proof}
For the direct implication, let \(w\) be a \((2n - 1)\)-detector that verifies the \((2n - 1)\)-extendability for \(u\).
	Given a Lipschitz map \(\gamma \colon  \cBall^{2n} \to \manfV\) such that \(\gamma|_{\Sphere^{2n - 1}} \in \Fuglede_{w}(\Sphere^{2n - 1}; \manfV)\), the map \(u \compose \gamma|_{\Sphere^{2n - 1}}\) is homotopic to a constant in \(\VMO(\Sphere^{2n - 1}; \Sphere^{n})\).
	By invariance of \(\hopf\) with respect to VMO homotopy, we then have \(\hopf{(u \compose \gamma|_{\Sphere^{2n - 1}})} = 0\).
	Since \(\gamma\) is arbitrary, the conclusion follows from Proposition~\ref{propositionHopfCondition}.

 To prove the converse, let \(w\) be a \((2n-1)\)-detector given by Proposition~\ref{propositionHopfCondition}.
	Thus, for every Lipschitz map \(\gamma \colon  \cBall^{2n} \to \manfV\) with \(\gamma|_{\Sphere^{2n-1}} \in \Fuglede_{w}(\Sphere^{2n-1}; \manfV)\), we have \(\hopf{(u \compose \gamma|_{\Sphere^{2n-1}})} = 0\).
	By injectivity of \(\hopf\), this implies that \(u \compose \gamma|_{\Sphere^{2n-1}}\) is homotopic to a constant in \(\VMO(\Sphere^{2n-1}; \Sphere^{n})\).
    In fact, by Proposition~\ref{propositionHomotopyVMOContinuousMap} there exists \(v \in \Smooth^0(\Sphere^{2n - 1}; \Sphere^{n})\) that is homotopic to \(u \compose \gamma|_{\Sphere^{2n-1}}\) in \(\VMO(\Sphere^{2n-1}; \Sphere^{n})\).
    By invariance of \(\hopf\) under \(\VMO\) homotopy, we get 
    \[
    \hopf{(v)} = \hopf{(u \compose \gamma|_{\Sphere^{2n-1}})} = 0.
    \]
    Therefore, by injectivity of \(\hopf\), we find that \(v\) is homotopic to a constant in \(\Smooth^{0}(\Sphere^{2n-1}; \Sphere^n)\) and also in \(\VMO\).
    Hence, by transitivity of the homotopy relation, \(u \compose \gamma|_{\Sphere^{2n-1}}\) is homotopic to a constant in \(\VMO(\Sphere^{2n-1}; \Sphere^n)\).
    Therefore, \(u\) is \((2n-1)\)-extendable.	
\end{proof}

\begin{remark}
The injectivity assumption of \(\hopf\) holds whenever \(n\) is even and \(\pi_{2n-1}(\Sphere^n) \simeq \Z\). Indeed, the quotient group \(\pi_{2n-1}(\Sphere^n)/\ker{\hopf}\) is isomorphic to \(\hopf(\pi_{2n-1}(\Sphere^n))\) which is a nontrivial group without torsion, since \(n\) is even. 
Given that \(\pi_{2n-1}(\Sphere^n) \simeq \Z\), this implies that \(\ker{\hopf}\) is trivial, since otherwise \(\pi_{2n-1}(\Sphere^n)/\ker{\hopf}\) would be a finite group. 
We recall that the isomorphism \(\pi_{2n-1}(\Sphere^n) \simeq \Z\) holds when \(n=2\) and \(n=6\).
\end{remark}

Combining Theorem~\ref{theoremhkpGreatBall} and Corollary~\ref{corollaryHopfFugledeTrivial}, one deduces in particular the following identification of \( \Hilbert^{1, p} (\Ball^4; \Sphere^{2})\) in terms of the Hopf current:

\begin{corollary}
    \label{corollaryHopfIsobe}
    Let \(3 \le p < 4\) and let \(u \in \Sobolev^{1,p}(\Ball^4;\Sphere^2)\) be a Hopf-friendly map.
    Then, \(u \in \Hilbert^{1, p} (\Ball^4; \Sphere^{2})\) if and only if
    \[
    \Hopf{u} = 0
    \quad \text{in the sense of currents in \(\Ball^4\).}
    \]
\end{corollary}

We thus recover the characterization of \(\Hilbert^{1, p} (\Ball^4; \Sphere^{2})\) by  Isobe~\cite{Isobe1995} when \(16/5 \le p < 4\). 
Recall that, by Proposition~\ref{example-hopf-friendly}, any map in \(\Sobolev^{1,16/5}(\Ball^4;\Sphere^2)\) is Hopf-friendly since the lower bound \(16/5\) ensures the summability of \(\chi \wedge u^{\sharp} \omega_{\Sphere^n}\) and we may thus consider it as a valid current.

To conclude, we observe that by Serre's theorem \cite{Serre1951}, the elements in \(\pi_k (\Sphere^n) \otimes \R\) are detected by the Hopf invariant. 
This implies that when the manifold \(\manfN\) is a sphere, the constructions of this chapter give a complete list of the local topological obstructions that can be detected by integral formulas. 
We refer the reader to \cites{Hardt-Riviere-2008} for other target manifolds \(\manfN\) that can be covered by similar methods.

\cleardoublepage
\chapter{\texorpdfstring{$(\ell, \MakeLowercase{e})$}{(l, e)}-extendability}
\label{chapterExtensionGeneral}

In Chapter~\ref{chapterGenericEllExtension}, we introduced the concept of generic \(\ell\)-extension using Fuglede maps on spheres and the boundaries of simplices, focusing on properties specific to this setting. The foundations established there provide the groundwork for the broader notion of \((\ell, e)\)-extendability, which we develop in this chapter, where Fuglede maps are defined on simplicial complexes. This progression allows us to highlight the common properties of both notions and explore their differences without redundancy.

As before, we formulate the extension property in the context of \(\VMO^\ell\) maps by relying on generic compositions. Specifically, a map \(u \in \VMO^\ell(\manfV; \manfN)\) is \emph{\((\ell, e)\)-extendable} whenever, for a generic Lipschitz map \(\gamma\) defined on an \(e\)-dimensional polytope \(K^e\), the composition \(u \circ \gamma|_{K^\ell}\) is \(\VMO\)-homotopic to a map in \(\Smooth^{0}(K^e; \manfN)\). When \(e = \ell + 1\), this recovers the \(\ell\)-extension property from Chapter~\ref{chapterGenericEllExtension}. However, when \(e \geq \ell + 2\), \((\ell, e)\)-extendability captures global properties of maps between \(\manfV\) and \(\manfN\). In fact, higher-dimensional simplicial complexes reveal global topological obstructions that are invisible in codimension \(1\) and are more adapted to capture the transition from local to global phenomena.

The set \(\VMO^{e}(\manfV; \manfN)\) provides a large class of \((\ell, e)\)-extendable maps. For example, when \(e \in \{1, \ldots, m-1\}\), a map \(u \colon \R^{m} \to \Sphere^e\) defined for \(x = (x', x'') \in \R^{e+1} \times \R^{m-e-1}\), with \(x' \neq 0\), by \(u(x) = x' / \abs{x'}\) is \(\VMO^{e}\) and thus \((\ell, e)\)-extendable. This property follows from a general consistency result that reveals a hierarchical structure among different levels of extension: If a map is \((\ell, e)\)-extendable, it is also \((i, j)\)-extendable for all \(i \le \ell\) and \(j \le e\) with \(i \leq j\). The proof of this fact involves restrictions of \(\VMO^\ell\) maps defined on a polytope to lower-dimensional polytopes.

To establish the consistency property, we address a key limitation of \(\VMO\) functions: Their lack of a well defined notion of restriction comparable to the trace for Sobolev functions. To this end, we introduce the space \(\VMO^\#(K^\ell | K^r)\), a subset of \(\VMO(K^\ell)\), whose elements exhibit vanishing crossed mean oscillations with respect to the subpolytope \(K^r\). This framework ensures that functions, when confined to \(K^r\), remain in \(\VMO\) and enables the analysis of homotopies between \(\VMO\) maps on \(K^r\). This flexible approach provides the robust tools required to handle restrictions and prove the consistency property.

\section{A general extension property}
\label{sectionExtensionHomotopyGroups}

We generalize the notion of \(\ell\)-extension using simplicial complexes in the spirit of Proposition~\ref{lemma_l_l1_property}:

\begin{definition}
\label{definitionExtensionVMO}
Let \(\ell, e\in \N\) with \(\ell \leq e\).  
A map \(u \in \VMO^{\ell}(\manfV ; \manfN)\) is \emph{\((\ell, e)\)-extendable} whenever \(u\) has an \(\ell\)-detector \(w \colon  \manfV \to [0,+\infty]\) with the following property:
For every simplicial complex \(\cK^e\) and every Lipschitz map \(\gamma \colon  K^e \to  \manfV\) with 
\[{}
\gamma|_{K^{\ell}} \in \Fuglede_{w}(K^\ell; \manfV) \text{,}
\]
there exists \(F \in \Smooth^{0}(K^{e}; \manfN)\) such that
\[{}
u \compose \gamma |_{K^{\ell}} \sim F|_{K^{\ell}}
\quad \text{in \(\VMO(K^{\ell}; \manfN)\).}
\]
\end{definition}

It follows from Proposition~\ref{lemma_l_l1_property} that \((\ell, \ell + 1)\)-extendability is equivalent to \(\ell\)-extendability in the sense of Definition~\ref{definitionExtensionVMO-1}.
Note also that every map \(u \in \VMO^{\ell}(\manfV ; \manfN)\) is \((\ell, \ell)\)-extendable as it suffices to take any \(\ell\)-detector for \(u\) and then apply Proposition~\ref{propositionHomotopyVMOContinuousMap}.
While \( \ell \)-extendability is an intrinsically local property, in the sense that one can merely rely on Fuglede maps on simplices (Proposition~\ref{propositionExtendabilitySimplex}), the same cannot be said about \( (\ell, e) \)-extendability in general. 
In Chapter~\ref{chapterDecouple}, we explain a missing global property that, combined with \( \ell \)-extendability, allows one to deduce \( (\ell, e) \)-extendability.

The above notion of extendability is stable with respect to composition with Fuglede maps:

\begin{proposition}
	\label{corollary_ouverture}
	If \(u \in \VMO^{\ell}(\manfV; \manfN)\) is \((\ell, e)\)-extendable and if \(w\) is any \(\ell\)-detector for \(u\) given by Definition~\ref{definitionExtensionVMO}, then \(u \compose \Phi\) is also \((\ell, e)\)-extendable for every  \(\Phi \in \Fuglede_{w}(\widetilde\manfV; \manfV)\), where \( \widetilde\manfV \) is a Lipschitz open subset of \( \R^{m} \) or a compact Riemannian manifold without boundary.
\end{proposition}

\begin{proof}
	We first observe that the summable function \(w \compose \Phi\) is an \(\ell\)-detector for \(u \compose \Phi\).{}
	Indeed, for any simplicial complex \(\cK^{\ell}\) and \(\gamma \in \Fuglede_{w \compose \Phi}(K^{\ell}; \widetilde\manfV)\), we have \(\Phi \compose \gamma \in \Fuglede_{w}(K^{\ell}; \manfV)\).{}
	Since \(w\) is an \(\ell\)-detector for \(u\), we deduce that \(u \compose \Phi \compose \gamma \in \VMO(K^{\ell}; \manfN)\).{}
	
	We now show that \(u \compose \Phi\) is \((\ell, e)\)-extendable using the \(\ell\)-detector \(w \compose \Phi\).
	To this end, take any simplicial complex \(\cE^{e}\) and a Lipschitz map \(\gamma \colon  E^{e} \to \widetilde\manfV\) such that \(\gamma|_{E^{\ell}} \in \Fuglede_{w \compose \Phi}(E^{\ell}; \widetilde\manfV)\).{}
	As \(\Phi \compose \gamma|_{E^{\ell}} \in \Fuglede_{w}(E^{\ell}; \manfV)\), we can rely on the \((\ell, e)\)-extendability of \(u\) to deduce  the existence of \(F \in \Smooth^{0}(E^{e}; \manfN)\) such that
\[{}
u \compose \Phi \compose \gamma |_{E^{\ell}} \sim F|_{E^{\ell}}
\quad \text{in \(\VMO(E^{\ell}; \manfN)\).}
\qedhere
\]
\end{proof}

Extendability is also closed under \(\VMO^\ell\) convergence of maps, introduced in Definition~\ref{def_conv_VMOl}:

\begin{proposition}
    \label{propositionExtensionVMOClosed}
    Let \((u_j)_{j \in \N}\) be a sequence of \((\ell, e)\)-extendable maps in \(\VMO^\ell(\manfV; \manfN)\).
    If 
    \[
    u_j \to u
    \quad \text{in \(\VMO^\ell(\manfV; \R^\nu)\),}
    \]
    then \(u\) is \((\ell, e)\)-extendable.
\end{proposition}
\begin{proof}
    Let \(w\) be an \(\ell\)-detector given by \(\VMO^\ell\)~convergence of \((u_j)_{j \in \N}\).
    For each \(j \in \N\), let \(w_j\) be an \(\ell\)-detector of \(u_j\) given by \((\ell, e)\)-extendability.
    Take a sequence of positive numbers \((\alpha_j)_{j \in \N}\) such that \(\widetilde{w} \vcentcolon= w + \sum\limits_{j = 0}^\infty{\alpha_j w_j}\) is summable.
    Given a simplicial complex \(\cK^e\) and a Lipschitz map \(\gamma \colon  K^e \to \manfV\) with \(\gamma|_{K^\ell} \in \Fuglede_{\tilde{w}}(K^\ell; \manfV)\), since \(\widetilde w \ge w\) we have
    \[
    u_j \compose \gamma|_{K^\ell} \to u \compose \gamma|_{K^\ell}
    \quad \text{in \(\VMO(K^\ell; \R^\nu)\).}
    \]
    By Proposition~\ref{propositionHomotopyVMOLimit}, for \(j\) large we then get
	\[{}
	u_{j} \compose \gamma|_{K^\ell} \sim u \compose \gamma|_{K^\ell}
	\quad \text{in \(\VMO(K^{\ell}; \manfN)\).}
	\] 
    Since \(\widetilde{w} \ge \alpha_{j} w_{j}\), there exists \(F \in \Smooth^0(K^e; \manfN)\) such that
	\[{}
	u_{j} \compose \gamma|_{K^\ell} \sim F|_{K^\ell}
	\quad \text{in \(\VMO(K^{\ell}; \manfN)\).}
	\] 
	We thus have the conclusion by the transitivity of the homotopy relation.
\end{proof}

\begin{corollary}
	  \label{corollary-strong-convergence-extendability}
    Let \(k \in \N_*\) and \(1 \le p < \infty\).{}	
    If \(u \in \Hilbert^{k, p}(\manfV; \manfN)\), then \(u\) is \((\ell, e)\)-extendable for every \(\ell \le \min{\{kp, e\}}\).
\end{corollary}
\begin{proof}
	Take a sequence \((u_{j})_{j \in \N}\) in \((\Smooth^{\infty} \cap \Sobolev^{k, p})(\manfV; \manfN)\) that converges to \(u\) in \(\Sobolev^{k, p}(\manfV;\R^\nu)\).{}
	By compactness of \(\manfN\), it follows from the Gagliardo-Nirenberg interpolation inequality that \(u \in \Sobolev^{1, kp}(\manfV; \manfN)\) and
	\[{}
	u_{j} \to u
	\quad \text{in \(\Sobolev^{1, kp}(\manfV; \R^\nu)\).}
	\]
	As \(kp \ge \ell\), by Proposition~\ref{lemmaFugledeSobolevDetector},  there exists a subsequence \((u_{j_i})_{i \in \N}\) such that
	\[{}
	u_{j_i} \to u
	\quad \text{in \(\VMO^{\ell}(\manfV; \R^\nu)\).}
	\]
	Since each \(u_{j_i}\) is \((\ell, e)\)-extendable, by Proposition~\ref{propositionExtensionVMOClosed} so is \(u\).
\end{proof}

Proposition~\ref{propositionExtensionVMOClosed} also implies a necessary condition for the weak density problem involving \(\Hilbert^{k, p}\weak (\manfV; \manfN)\), see also \cite{Hang-Lin}*{Theorem~7.1}:

\begin{definition}
	\label{definitionHkpWeak}
	Let \(k \in \N_*\) and \(1 \le p < \infty\).{}
	We denote by \(\Hilbert^{k, p}\weak (\manfV; \manfN)\)  the set of maps \(u \in \Sobolev^{k, p}(\manfV; \manfN)\) for which there exists a sequence \((u_{j})_{j \in \N}\) in \((\Smooth^{\infty}\cap \Sobolev^{k, p})(\manfV; \manfN)\) that is bounded in \(\Sobolev^{k, p}(\manfV; \R^\nu)\) and 
	\[{}
	u_{j} \to u \quad \text{in \(\Lebesgue^{p}(\manfV; \R^\nu)\).}
	\]
\end{definition}

\begin{corollary}
    \label{corollary-weak-convergence-extendability}
	Let \(k \in \N_*\) and \(1 \le p < \infty\).{}	
	If \(u \in \Hilbert^{k, p}\weak (\manfV; \manfN)\), then \(u\) is \((\ell, e)\)-extendable for every \(\ell < \min{\{kp, e\}}\).
\end{corollary}

\begin{proof}
	Take a sequence \((u_{j})_{j \in \N}\) in \((\Smooth^{\infty}\cap \Sobolev^{k, p})(\manfV; \manfN)\) as in Definition~\ref{definitionHkpWeak}. 
    Since \(\manfN\) is compact, \((u_{j})_{j \in \N}\) is bounded in \(\Lebesgue^{\infty}(\manfV;\R^{\nu})\). 
 This sequence is then uniformly bounded in \((\Sobolev^{k,p}\cap \Lebesgue^{\infty})(\manfV;\manfN)\) and then, by the Gagliardo-Nirenberg interpolation inequality, it is also uniformly bounded in \(\Sobolev^{1, kp}(\manfV;\manfN)\). 
	By interpolation in fractional Sobolev spaces (here, we use the fact that \(\manfV\) is a Lipschitz open set), 
 \((u_{j})_{j \in \N}\) converges to \(u\) in \(\Sobolev^{s, kp}(\manfV; \R^{\nu})\) for every \(0 < s < 1\).{}
	Since \(\ell < kp\), we may take \(s < 1\) such that \(\ell \le skp\).{}
	Thus, by Proposition~\ref{lemmaFugledeSobolevDetector} there exists a subsequence \((u_{j_i})_{i \in \N}\) such that
	\[{}
	u_{j_i} \to u
	\quad \text{in \(\VMO^{\ell}(\manfV; \R^\nu)\).}
	\]
	Each \(u_{j_i}\) is \((\ell, e)\)-extendable on \(\manfV\) for every \(e \ge \ell\), and thus, by Proposition~\ref{propositionExtensionVMOClosed}, so is \(u\). 
\end{proof}

Compared to Corollary~\ref{corollary-strong-convergence-extendability}, the necessary condition of weak approximation involves a strict inequality, which gives a weaker property when \(kp\) is an integer:

\begin{example}
    Take \(\manfV = \Ball^{n + 1}\) and \(\manfN = \Sphere^n\).
    When \(k \le n\), from \citelist{\cite{Hajlasz} \cite{BPVS:2013}*{Theorem~1.7}} we have 
    \[
    \Hilbert\weak^{k, n/k}(\Ball^{n + 1}; \Sphere^{n})
    = \Sobolev^{k, n/k}(\Ball^{n + 1}; \Sphere^{n}).
    \]
    However, the map \(u \colon \Ball^{n + 1} \to \Sphere^{n}\) defined for \(x \ne 0\) by \(u(x) = x/\abs{x}\) belongs to \(\Sobolev^{k, n/k}(\Ball^{n + 1}; \Sphere^{n})\) but  \(u\) is not \(n\)-extendable in view of Lemma~\ref{propositionExtensionHomogeneity}.
    It then follows from Corollary~\ref{corollary-strong-convergence-extendability} that \(u \not\in \Hilbert^{k, n/k}(\Ball^{n + 1}; \Sphere^{n})\).
\end{example}

\section{Comparison between two types of extendability}

Extendability of \( \VMO^{\ell} \)~maps enjoys the following consistency property related to the extension for lower-dimensional skeletons:

\begin{proposition}
\label{proposition_extension_montone_small}
If \(u \in \VMO^{\ell}(\manfV ; \manfN)\) is \((\ell, e)\)-extendable, then \(u\) is \((i, j)\)-extendable for every \(i \in \{0, \dots, \ell\}\) and every \(j \in \{0, \dots, e\}\) with \(i \le j\).
\end{proposition}

Since \(\ell\)-extendability is equivalent to \((\ell, \ell + 1)\)-extendability, we deduce more generally that

\begin{corollary}
    If \(u \in \VMO^{\ell}(\manfV ; \manfN)\) is \((\ell, e)\)-extendable for some \(e \ge \ell + 1\), then \(u\) is \(\ell\)-extendable.
\end{corollary}

Proposition~\ref{proposition_extension_montone_small} applied with \(e = \ell + 1\) gives the following counterpart of the consistency property that refers solely to the \(\ell\)-extendability:

\begin{corollary}
    If \(u \in \VMO^{\ell}(\manfV ; \manfN)\) is \(\ell\)-extendable, then \(u\) is \(j\)-extendable for every \(j \le \ell\).
\end{corollary}

Since every map in \(\VMO^{e}(\manfV; \manfN)\) is \((e, e)\)-extendable,  Proposition~\ref{proposition_extension_montone_small} also implies a purely analytical sufficient condition for \((\ell, e)\)-extendability:

\begin{corollary}
	\label{propositionVMOlExtension}
	If \(u \in \VMO^{e}(\manfV; \manfN)\), then \(u\) is \((\ell, e)\)-extendable for every \(\ell \le e\).
\end{corollary}

\begin{example}
	Let \(e \le m-1\) and let \(T^{e^{*}} \subset \manfV\) be a structured singular set of rank \(e^{*} \vcentcolon= m - e - 1\).
	If \(u \colon  \manfV \to \manfN\) is a measurable function that is continuous in \(\manfV \setminus T^{e^{*}}\), then by Proposition~\ref{proposition_Extension_Property_R_0-New} the map \(u\) is \(\VMO^{e}\).
    It then follows from Corollary~\ref{propositionVMOlExtension} that \(u\) is \((\ell, e)\)-extendable for every \(\ell \le e\).{}
\end{example}

Both concepts of extendability are equivalent under some additional topological assumption on \(\manfV\) or \(\manfN\).{}
Indeed, if \(\manfV\) satisfies
\begin{equation}
	\label{eqExtension-58}
	\pi_0 (\manfV) \simeq \ldots \simeq \pi_{\ell} (\manfV) \simeq \{0\},
\end{equation}
then every \(f \in \Smooth^{0}(K^{\ell}; \manfV)\) is homotopic to a constant map in \(\Smooth^{0}(K^{\ell}; \manfV)\).{}
As a result,

\begin{proposition}
\label{proposition_condition_domain}
If \(u \in \VMO^{\ell}(\manfV; \manfN)\) is  \(\ell\)-extendable
and \(\manfV\) satisfies \eqref{eqExtension-58}, then \(u\) is \((\ell, e)\)-extendable for every \(e \ge \ell\).
\end{proposition}

\begin{proof}[Proof of Proposition~\ref{proposition_condition_domain}]
We recall that every \(\VMO^{\ell}\)~map \(u\) is \((\ell, \ell)\)-extendable and if, in addition, \(u\) is \(\ell\)-extendable, then \(u\) is \((\ell, \ell+1)\)-extendable.
We may thus assume that \(e \ge \ell + 2\).
Let \(w\) be an \(\ell\)-detector for \(u\) given by Theorem~\ref{proposition_Reference_Homotopy} and let \(\cK^e\) be a simplicial complex. 
Given a Lipschitz map \(\gamma \colon  K^e \to \manfV\),
by the topological assumption on \(\manfV\) there exists a constant map \(\mathrm{c} \colon  K^e \to \manfV\) such that
\[{}
\gamma\vert_{K^{\ell}} \sim \mathrm{c}\vert_{K^{\ell}}
\quad \text{in \(\Smooth^{0}(K^{\ell}; \manfV)\).}
\]
We now suppose that \(\gamma\vert_{K^{\ell}} \in \Fuglede_{w}(K^{\ell} ; \manfV)\).
We may also assume that \(w \compose \mathrm{c} < + \infty\), which implies that \(\mathrm{c} \in \Fuglede_{w}(K^\ell ; \manfV)\). 
Since \(u\) is \(\ell\)-extendable, by Theorem~\ref{proposition_Reference_Homotopy} we then have
\[{}
u \compose \gamma\vert_{K^{\ell}}
\sim u \compose \mathrm{c}\vert_{K^{\ell}}
\quad \text{in \(\VMO (K^\ell; \manfN)\)}.
\]
As \(u \compose \mathrm{c}\) is constant and in particular continuous in \(K^{e}\), the conclusion follows.
\end{proof}

Another situation where \(\ell\)-extendability is equivalent with \((\ell, e)\)-extendability arises when \(\manfN\) satisfies
\begin{equation}
	\label{eqExtension-94}
	\pi_{\ell + 1} (\manfN) \simeq \ldots \simeq \pi_{e  - 1} (\manfN) \simeq \{0\}.
\end{equation}
Indeed, this assumption on the homotopy groups of \(\manfN\) implies that every \(F \in \Smooth^{0}(K^{\ell + 1}; \manfN)\) has an extension \(\widetilde{F} \in \Smooth^{0}(K^e ; \manfN)\).{}
We then have

\begin{proposition}
\label{proposition_Extension_homotop_group_increasing}
If \(u \in \VMO^{\ell}(\manfV; \manfN)\) is \(\ell\)-extendable and \(\manfN\) satisfies \eqref{eqExtension-94} for some \(e \ge \ell + 2\), then \(u\) is \((\ell, e)\)-extendable.
\end{proposition}

\begin{proof}[Proof of Proposition~\ref{proposition_Extension_homotop_group_increasing}]
Let \(w\) be an \(\ell\)-detector for \(u\) given by Proposition~\ref{lemma_l_l1_property}, let \(\cK^{e}\) be a simplicial complex and let \(\gamma \colon  K^e \to \manfV\) be a Lipschitz map with \(\gamma\vert_{K^{\ell}} \in \Fuglede_{w}(K^\ell ; \manfV)\). 
Since \(u\) is \(\ell\)-extendable, by Proposition~\ref{lemma_l_l1_property} there exists \(F \in \Smooth^{0}(K^{\ell + 1}; \manfN)\) such that 
\[{}
u \compose \gamma\vert_{K^{\ell}} \sim F\vert_{K^{\ell}}
\quad \text{in \(\VMO (K^\ell; \manfN)\).}
\]
By the topological assumption on \(\manfN\), there exists \(\widetilde{F} \in \Smooth^{0}(K^e ; \manfN)\) such that \(\widetilde{F}\vert_{K^{\ell + 1}} = F\).{}
Hence, \(u\) is \((\ell, e)\)-extendable.
\end{proof}

The topological information on \(\manfN\) can also be used to obtain the equivalence between \((i, e)\)- and \((\ell, e)\)-extendability for \(i + 1 \le \ell\).{}
More precisely, if 
\begin{equation}
\label{eqExtension-142}
\pi_{i+1} (\manfN) \simeq \ldots \simeq \pi_{\ell} (\manfN) \simeq \{0\},
\end{equation}
then every \(f_{1}, f_{2} \in \Smooth^{0}(K^{\ell}; \manfN)\) such that
\[{}
f_{1}\vert_{K^{i}} \sim f_{2}\vert_{K^{i}}
\quad \text{in \(\Smooth^{0}(K^{i}; \manfN)\)}
\]
are homotopic in \(\Smooth^{0}(K^{\ell}; \manfN)\).
Using this property, we prove that

\begin{proposition}
\label{proposition_Extension_homotop_group_decreasing}
If \(u \in \VMO^{\ell}(\manfV; \manfN)\) is \((i, e)\)-extendable for some  \(i \in \{0, \dots, \ell - 1\}\) and if \(\manfN\) satisfies \eqref{eqExtension-142}, then \(u\) is also \((\ell, e)\)-extendable.
\end{proposition}

We deduce the following assertion that is reminiscent of analogous results in the literature concerning the approximation of Sobolev maps, see~\citelist{\cite{Hang-Lin}*{Corollary~5.3}\cite{Hajlasz}}:

\begin{corollary}
	\label{corollaryHangLin}
	If \(u \in \VMO^{\ell} (\manfV; \manfN)\) and if there exists \(i \in \{0, \dots, \ell - 1\}\) such that
	\[{}
	\pi_0 (\manfV) \simeq \ldots \simeq \pi_i (\manfV)
	 \simeq 
	\pi_{i + 1} (\manfN) \simeq \ldots \simeq \pi_{\ell} (\manfN) \simeq \{0\} \text{,}
	\] 
	then \(u\) is \((\ell, e)\)-extendable for every \(e \ge \ell\).
\end{corollary}

\begin{proof}[Proof of Corollary~\ref{corollaryHangLin}]
   We recall that, since \(i + 1 \le \ell\), it follows from Proposition~\ref{propositionVMOellInclusion} that every \(\VMO^{\ell}\) map is \(\VMO^{i + 1}\) and then is \((i + 1, i + 1)\)-extendable.
   Hence, by Proposition~\ref{proposition_extension_montone_small}, \(u\) is \(i\)-extendable.
   Then, for every \(e \ge \ell\), by Proposition~\ref{proposition_condition_domain} and the topological assumption on \(\manfV\), \(u\) is \((i, e)\)-extendable. 
   Finally, by Proposition~\ref{proposition_Extension_homotop_group_decreasing} and the topological assumption on \(\manfN\), \(u\) is \((\ell, e)\)-extendable.	
\end{proof}

We postpone the proofs of Propositions~\ref{proposition_extension_montone_small} and~\ref{proposition_Extension_homotop_group_decreasing} to Section~\ref{sectionExtensionProofs}.
We rely on a notion of restriction to lower dimensional sets for a subclass of \(\VMO\) functions that we denote by \(\VMO^{\#}\).{}
For example, given a simplicial complex \(\cK^{\ell}\) and \(r \in \{0, \dots, \ell - 1\}\), we consider the polytope \(K^r\) associated to the \(r\)-dimensional subskeleton \(\cK^r\) and then introduce the set \(\VMO^{\#}(K^{\ell}| K^{r})\). 
We refer the reader to the next section for its definition and main properties. 
Let us simply mention here that functions in \(\VMO^{\#}(K^{\ell}| K^{r})\) have a well defined trace in \(\VMO(K^{r})\) obtained by restriction to \(K^{r}\). 
We define accordingly the vector valued version \(\VMO^{\#}(K^{\ell}| K^{r}; \manfN)\) of this space which has the following fundamental property: 
If two maps \(v_{1}, v_{2} \in \VMO^{\#}(K^{\ell}| K^{r}; \manfN)\) are homotopic in \(\VMO(K^{\ell}; \manfN)\), then
\[{}
v_{1}|_{K^{r}} \sim v_{2}|_{K^{r}}
\quad \text{in \(\VMO(K^{r}; \manfN)\).}
\]
This assertion  is proved in the next section, see Proposition~\ref{propositionVMOSharpHomotopy}.

\medskip
A setting where \(\ell\)-extendability is weaker than \((\ell, e)\)-extendability arises for maps involving real and complex projective spaces \(\RP^{m}\) and \(\CP^{m}\), respectively.
These spaces can be introduced in two ways, see \cite{Hatcher_2002}*{Examples~0.4 and~0.6}.
Firstly, \(\RP^{m}\) is defined as the quotient of the unit sphere \(\Sphere^{m}\) in \(\R^{m + 1}\) by identifying the antipodal points \(x\) and \(-x\).
Another construction of \(\RP^{m}\) consists of starting with the unit closed ball \(\cBall^{m} \subset \R^{m}\) and then identifying antipodal points \(x\) and \(-x\), but only on the sphere \(\Sphere^{m - 1}\).
A combination of both approaches gives the inclusion \(\RP^{m - 1} \subset \RP^{m}\).{}

Similarly, \(\CP^{m}\) is defined as the quotient of the unit sphere \(\Sphere^{2m + 1}\) in \(\CC^{m + 1}\) using the equivalence relation \(x \sim y\) given by \(x = \lambda y\) for some \(\lambda \in \CC\).{}
Another way of constructing \(\CP^{m}\) consists of taking the closed unit ball \(\cBall^{2m} \subset \CC^{m}\) and then identifying any two points \(x, y \in \partial\Ball^{2m}\) whenever \(x=\lambda y\)  for some \(\lambda \in \CC\).{}
Here again one gets an inclusion \(\CP^{m - 1} \subset \CP^{m}\).{}

To provide a counterexample to extendability, we shall rely on the topological property that the identity maps on \(\RP^3\) and \(\CP^1\) do not have a continuous extension as maps \(\RP^4 \to \RP^3\) and \(\CP^2 \to \CP^1\), respectively.
In other words, there are no retractions of \(\RP^{4}\) to \(\RP^{3}\) nor of \(\CP^{2}\) to \(\CP^{1}\).{}
This latter fact is explained in the proof of \cite{Hang-Lin}*{Corollary~5.5}.

\begin{example}
\label{exampleExtensionCP}
	There exist \(\VMO^{3}\) maps \(\RP^{4} \to \RP^{3}\) and \(\CP^{2} \to \CP^{1}\) that are \(2\)-extendable but not \((2, 4)\)-extendable.
We start from the map
 \[
 x\in \cBall^{4} \setminus \{0\} \longmapsto \frac{x}{\abs{x}}\in \Sphere^{3}
 \]	
that, in view of the equivalence relations involved in the definitions of real and complex projective spaces, yields
continuous maps \(\RP^{4}\setminus \{0\}\) into \(\RP^3\) and \(\CP^{2}\setminus \{0\}\) into \(\CP^1\). 
Let us denote by \(u_r \colon  \RP^{4} \to \RP^{3}\) and \(u_c \colon  \CP^{2} \to \CP^{1}\) the resulting maps.
Their restrictions \(u_r|_{\RP^{3}}\) and \(u_c|_{\CP^{1}}\) are the identity maps and 
we have \(u_r \in \Sobolev^{1, p}(\RP^{4}; \RP^{3})\) and \(u_c \in \Sobolev^{1, p}(\CP^{2}; \CP^{1})\) for every \(1 \le p < 4\).{}
	Hence, by Proposition~\ref{propositionSobolevApproximationFuglede}, they are \(\VMO^{3}\) maps and then, applying Corollary~\ref{propositionVMOlExtension} with \(e = 3\) and \(\ell = 2\), we deduce that \(u_r\) and \(u_c\) are both \(2\)-extendable.

	However, \(u_r\) and \(u_c\) are not \((2, 4)\)-extendable. 
 We present a detailed proof for the map \(u_c\)\,, which can be readily adapted for \(u_r\)\,. 
 We assume by contradiction that there exists a \(2\)-detector \(w \colon  \CP^{2} \to [0, +\infty]\) that verifies the \((2, 4)\)-extendability condition for \(u_c\).{}
	We prove that, as a consequence, the identity map on \(\CP^{1}\) has a continuous extension as a map \(\CP^{2} \to \CP^{1}\), which is not true.
	To this end, we take a simplicial complex \(\cK^{4}\), a subcomplex \(\cL^{2}\) and a biLipschitz homeomorphism \(\gamma \colon  K^{4} \to \CP^{2}\) such that \(\gamma(L^{2}) = \CP^{1}\).{}
	Observe that \(\gamma|_{K^{2}}\) need not be a Fuglede map with respect to \(w\).{}
	We thus take a transversal perturbation of the identity \(\tau \colon  \CP^{2} \times B_{\delta}^{q}(0) \to \CP^{2}\) with \(q \ge 2\).{}
	We deduce from Proposition~\ref{lemmaTransversalFamily} that \(\tau_{\xi} \compose \gamma|_{K^{2}} \in \Fuglede_{w}(K^{2}; \CP^{2})\)
	for almost every \(\xi \in B_{\delta}^{q}(0)\).
	By the \((2, 4)\)-extendability of \(u_{c}\), for any such \(\xi\) there exists \(F \in \Smooth^{0}(K^{4}; \CP^{1})\) such that
	\[{}
	u_c \compose \tau_{\xi} \compose \gamma|_{K^{2}}{}
	\sim F|_{K^{2}}
	\quad \text{in \(\VMO(K^{2}; \CP^{1})\).}
	\]
    In view of Lemma~\ref{lemmaVMORestriction}, we can restrict the homotopy relation above to \(L^{2}\).  
	Since for \(\xi\) small, the map \(u_c \compose \tau_{\xi} \compose \gamma|_{L^{2}}\) is continuous, it follows from Proposition~\ref{propositionHomotopyVMOtoC} that
	\begin{equation}
	\label{eqExtension-924}
	u_c \compose \tau_{\xi} \compose \gamma|_{L^{2}}{}
	\sim F|_{L^{2}}
	\quad \text{in \(\Smooth^{0}(L^{2}; \CP^{1})\).}
	\end{equation}
	On the other hand, also for \(\xi\) small enough, the map 
	\(
	t \in [0, 1] 
	\mapsto \tau_{t\xi} \compose \gamma|_{L^{2}}
	\)
	is a continuous path between \(\gamma|_{L^{2}}\) and \(\tau_{\xi} \compose \gamma|_{L^{2}}\) that avoids \(0 \in \CP^{2}\).{}
	Hence, by continuity of \(u_c\) on \(\CP^{2} \setminus \{0\}\),{}
	\begin{equation}
	\label{eqExtension-938}
	u_c \compose \tau_{\xi} \compose \gamma|_{L^{2}}{}
	\sim u_c \compose \gamma|_{L^{2}}
	\quad \text{in \(\Smooth^{0}(L^{2}; \CP^{1})\).}
	\end{equation}
	Combining \eqref{eqExtension-924} and \eqref{eqExtension-938}, we thus have
	\[{}
	u_c \compose \gamma|_{L^{2}} \sim F
	\quad \text{in \(\Smooth^{0}(L^{2}; \CP^{1})\).}
	\]
	Hence,
	\[{}
	u_c|_{\CP^{1}} \sim F \compose \gamma^{-1}|_{\CP^{1}}
	\quad \text{in \(\Smooth^{0}(\CP^{1}; \CP^{1})\).}
	\]
	Since \(F \compose \gamma^{-1} \in \Smooth^{0}(\CP^{2}; \CP^{1})\), we deduce that the identity \(u_c|_{\CP^{1}}\) has a continuous extension to \(\CP^{2}\), which is a contradiction.
\end{example}

\section{\texorpdfstring{VMO$^{\#}$}{VMOsharp} functions}\label{section:VMOsharp}

A straightforward property of \(\VMO\) maps on simplicial complexes concerns the restriction to a subskeleton of the same dimension, see Lemma~\ref{lemmaVMORestriction}.
In this section, we investigate the restriction of \(\VMO\) maps to lower dimensional subskeletons.
As there is no trace operator available in the standard \(\VMO\) setting, we consider a subspace of \(\VMO\) that is more adapted to restrictions.
To this end, let \(\cL^{r_{1}} \supset \cL^{r_{2}}\) be two subskeletons of a simplicial complex  \(\cK^{\ell}\).{}
We introduce the \emph{crossed mean oscillation} of a measurable function \(v \colon  L^{r_{1}} \to \R\) as the quantity
\begin{equation}
\label{eqExtension-452}
[v]_{L^{r_{1}}, L^{\smash{r_{2}}}, \rho} 
= \sup_{x\in K^\ell}{\frac{1}{\rho^{r_{1} + r_{2}}}\int_{B^{\ell}_\rho(x)\cap L^{r_{1}}}\int_{B^{\ell}_\rho(x)\cap L^{\smash{r_{2}}}}\abs{v(y) - v(z)}\dif \cH^{\smash{r_{2}}}(z) \dif \cH^{r_{1}}(y)}.
\end{equation}
We observe that for every \(\rho > \Diam{K^{\ell}}\),
\[{}
\seminorm{v}_{L^{r_1}, L^{r_2}, \rho}
\le \seminorm{v}_{L^{r_1}, L^{r_2}, \Diam{K^{\ell}}}\,.
\]

We are mainly interested in the case where \(\cL^{r_{1}}\) is the entire skeleton \(\cK^{\ell}\) and \(\cL^{r} = \cL^{r_{2}}\) is an arbitrary subskeleton of \(\cK^{\ell}\).
In fact, 

\begin{proposition}
	\label{lemmaVMOSharpMixed}
	Let \(r_{1}, r_{2}, r \in \{0, \dots, \ell\}\) with \(r_{1} \ge r_{2}\).
    For every \(0 < \rho \le \Diam{K^{\ell}}\),{} we have
	\begin{gather*}
	\seminorm{v}_{L^{r_{1}}, L^{r_{2}}, \rho}
	   \le C' \bigl(\seminorm{v}_{K^{\ell}, L^{r_{1}}, \rho} + \seminorm{v}_{K^{\ell}, L^{r_{2}}, \rho} \bigr) \text{,}\\
    \seminorm{v}_{K^{\ell}, L^{r}, \Diam{K^{\ell}}} 
     \le C'' \bigl(\seminorm{v}_{\VMO(K^\ell)}+ \seminorm{v}_{K^\ell, L^{r}, \rho}\bigr) \text{,}
	\end{gather*}
    where \(C'\) does not depend on \(\rho\), while \(C''\) depends on \(\rho\).
\end{proposition}

\resetconstant
\begin{proof}
Let \(y,z,t\in K^\ell\). By the triangle inequality, we have 
\[
|v(y) - v(z)| 
\le |v(t) - v(y)| + |v(t)- v(z)|. 
\]
Given \(\rho_1, \rho_2, \rho_3>0\) and two subskeletons \(\cL^{r_{1}}, \cL^{r_{2}}\) of \(\cK^\ell\) with \(r_1,r_2\in \{0, \dots, \ell\}\),  for every \(x\in K^{\ell}\), 
we integrate with respect to the variables \(y\) over \(B_{\rho_1}^{\ell}(x) \cap L^{r_1}\), \(z\) over \(B_{\rho_2}^{\ell}(x) \cap L^{r_2}\) and \(t\) over \(B_{\rho_3}^{\ell}(x)\) to get
\begin{multline*}
\int_{B_{\rho_1}^{\ell}(x)\cap L^{r_1}} \int_{B_{\rho_2}^{\ell}(x)\cap L^{r_2}}|v(y)- v(z)| \dif\cH^{r_2}(z) \dif\cH^{r_1}(y)\\
\leq 
\frac{\cH^{r_2}\bigl(B_{\rho_2}^{\ell}(x)\cap L^{r_2}\bigr)}{\cH^{\ell}(B_{\rho_3}^{\ell}(x))} \int_{B_{\rho_3}^{\ell}(x)} \int_{B_{\rho_1}^{\ell}(x)\cap L^{r_1}}|v(t)- v(y)| \dif\cH^{r_1}(y) \dif\cH^{\ell}(t) \\
+\frac{\cH^{r_1}\bigl(B_{\rho_1}^{\ell}(x)\cap L^{r_1}\bigr)}{\cH^{\ell}(B_{\rho_3}^{\ell}(x))} \int_{B_{\rho_3}^{\ell}(x)} \int_{B_{\rho_2}^{\ell}(x)\cap L^{r_2}}|v(t) - v(z)| \dif\cH^{r_2}(z) \dif\cH^{\ell}(t). 
\end{multline*}
Since \(\cH^{\ell}(B_{\rho_3}^{\ell}(x))
\ge c \rho_{3}^{\ell}\) and  for \(i\in \{1,2\}\),
\[{}
\cH^{r_i}\bigl(B_{\rho_i}^{\ell}(x)\cap L^{r_i}\bigr){}
\le \C \rho_{i}^{r_i},
\]
we get 
\begin{multline*}
\int_{B_{\rho_1}^{\ell}(x)\cap L^{r_1}} \int_{B_{\rho_2}^{\ell}(x)\cap L^{r_2}}|v(y)- v(z)| \dif\cH^{r_2}(z) \dif\cH^{r_1}(y)\\
\leq \C\biggl(
\frac{\rho_{2}^{r_2}}{\rho_{3}^{r_3}} \int_{B_{\rho_3}^{\ell}(x)} \int_{B_{\rho_1}^{\ell}(x)\cap L^{r_1}}|v(t)- v(y)| \dif\cH^{r_1}(y) \dif\cH^{\ell}(t) \\
+\frac{\rho_{1}^{r_1}}{\rho_{3}^{\ell}} \int_{B_{\rho_3}^{\ell}(x)} \int_{B_{\rho_2}^{\ell}(x)\cap L^{r_2}}|v(t) - v(z)| \dif\cH^{r_2}(z) \dif\cH^{\ell}(t)\biggr). 
\end{multline*}
Taking \(\rho_1=\rho_2=\rho_3=\rho\) and \(r_1\geq r_2\),  we get the first conclusion. If instead, we write the above inequality with \(\rho_1=\Diam K^{\ell}\), \(\rho_2=\rho_3=\rho\) and  \(L^{r_1}=K^\ell, L^{r_2}=L^{r}\),
we obtain
\begin{multline*}
\int_{K^\ell}\int_{B_{\rho}^{\ell}(x)\cap L^{r}}|v(y)- v(z)| \dif\cH^{r}(z) \dif\cH^{\ell}(y)\\
\leq \C\biggl(\frac{1}{\rho^{\ell-r}}
 \int_{K^{\ell}} \int_{B_{\rho}^{\ell}(x)\cap K^{\ell}}|v(t)- v(y)| \dif\cH^{\ell}(y) \dif\cH^{\ell}(t) \\
+\frac{1}{\rho^{\ell}}\int_{B_{\rho}^{\ell}(x)} \int_{B_{\rho}^{\ell}(x)\cap L^{r}}|v(t) - v(z)| \dif\cH^{r}(z) \dif\cH^{\ell}(t)\biggr). 
\end{multline*}
The last term is not larger than \(\rho^{r}\seminorm{v}_{K^{\ell}, L^{r}, \rho}\) while the first term can be estimated, up to a multiplicative constant depending only on \(\Diam{K^\ell}\),  by \(\rho^{r-\ell}\seminorm{v}_{\VMO(K^{\ell})}\). Covering \(L^{r}\) by a finite number of balls of radius \(\rho\), we get the second conclusion.
\end{proof}

We denote by \(\VMO^\#(K^{\ell}| L^{r})\) the set of measurable functions \(v\colon  K^{\ell} \to \R\) such that\label{eqExtensionVMOSharp} 
\[{}
v \in \VMO(K^{\ell})
\quad \text{and} \quad
\lim_{\rho \to 0 }{\seminorm{v}_{K^{\ell}, L^{r}, \rho}} = 0.
\]
We observe that \(\VMO^\#(K^{\ell}| K^{\ell}) = \VMO(K^{\ell})\) and, by compactness of \(K^{\ell}\), we have \(\Smooth^{0} (K^\ell)\subset \VMO^\#(K^\ell| L^r)\).

\begin{proposition}
    \label{propositionVMOSharpSeminormFinite}
    If \(v \in \VMO^\#(K^{\ell}| L^{r})\), then \(\sup\limits_{\rho > 0}{\seminorm{v}_{K^{\ell}, L^{r}, \rho}} < \infty\).
\end{proposition}

\begin{proof}
It follows from the definition of  \(\VMO^\#(K^{\ell}| L^{r})\) that there exists \(\rho_0>0\) such that \(\seminorm{v}_{K^{\ell},L^{r}, \rho}\leq 1\) for every \(0<\rho \leq \rho_0\). By the second inequality in Proposition~\ref{lemmaVMOSharpMixed} applied with \(\rho=\rho_0\), we also have that \(\seminorm{v}_{K^{\ell},L^{r},\Diam{K^{\ell}}}<\infty\). Finally, for every \(\rho_0\leq \rho \leq \Diam{K^{\ell}}\), 
\[
\seminorm{v}_{K^{\ell},L^{r}, \rho} \leq \rho_{0}^{-\ell-r}\seminorm{v}_{K^{\ell},L^{r}, \Diam{K^{\ell}}}<\infty.
\]
This proves that \(\sup\limits_{0<\rho\leq \Diam{K^{\ell}}}{\seminorm{v}_{K^{\ell}, L^{r}, \rho}}\) is finite.     
\end{proof}

By Proposition~\ref{propositionVMOSharpSeminormFinite}, we may introduce the following norm in  \(\VMO^\#(K^{\ell}| L^{r})\),
\begin{equation}
\label{eqExtension-586}
\norm{v}_{\VMO^\#(K^{\ell}| L^{r})} 
= {\|v\|_{\VMO(K^{\ell})}} + \seminorm{v}_{\VMO^\#(K^{\ell}| L^{r})},
\end{equation}
where
\begin{equation}
\label{eqExtension-592}
\seminorm{v}_{\VMO^\#(K^{\ell}|L^{r})}
\vcentcolon=
\sup_{0<\rho\leq \Diam{K^{\ell}} }{\seminorm{v}_{K^{\ell}, L^{r}, \rho}}=\sup_{\rho>0 }{\seminorm{v}_{K^{\ell}, L^{r}, \rho}}\,.
\end{equation}

An important property of functions in \(\VMO^\#(K^{\ell}|L^{r})\) is that their restriction to \(L^{r}\) preserves the \(\VMO\)~property:

\begin{proposition}
	\label{propositionVMOSharpRestrictionSkeleton}
	For every \(v \in \VMO^\#(K^{\ell}| L^{r})\), we have \(v|_{L^{r}} \in \VMO(L^{r})\) and
	\[{}
	\norm{v}_{\VMO(L^{r})}
	\le C \norm{v}_{\VMO^\#(K^{\ell}|L^{r})}.
	\]
\end{proposition}

\resetconstant
\begin{proof}[Proof of Proposition~\ref{propositionVMOSharpRestrictionSkeleton}]
	We have
	\[{}
	\begin{split}
	\int_{L^{r}}{\abs{v} \dif\cH^{r}}
	& \le \int_{L^{r}}{\biggabs{v - \fint_{K^{\ell}}{v\dif\cH^{\ell}}} \dif\cH^{r}} + \cH^{r}(L^{r}) \biggabs{\fint_{K^{\ell}}{v\dif\cH^{\ell}}}\\
	& \le \int_{L^{r}}\fint_{K^{\ell}}{\abs{v(y) - v(z)} \dif\cH^{\ell}(z)\dif\cH^{r}(y)} + \C \int_{K^{\ell}}{\abs{v}\dif\cH^{\ell}}.
	\end{split}
	\]
	Hence,
	\[{}
	\norm{v}_{\Lebesgue^{1}(L^{r})}
	\le \C\bigr(\seminorm{v}_{\VMO^\#(K^{\ell}| L^{r})} + \norm{v}_{\Lebesgue^{1}(K^{\ell})}\bigr).
	\]
	Next, for every \(x \in L^{r}\) and every \(0 < \rho \le \Diam{K^{\ell}}\),{}
	\[
	\cH^{\ell}(B_{\rho}^{\ell}(x) \cap L^{r})
	\ge c\rho^{r}.
	\]
	It then follows from Proposition~\ref{lemmaVMOSharpMixed} that the mean oscillation of \(v|_{L^{r}}\), see \eqref{eqDetector-33} with \(X = L^{r}\), satisfies
	\[{}
	\seminorm{v|_{L^{r}}}_{\rho}
	\le \frac{1}{c^{2}} \seminorm{v}_{L^{r}, L^{r}, \rho} 
	\le \Cl{cteVMO-92} \seminorm{v}_{K^{\ell}, L^{r}, \rho}\,.
	\]
	Thus, \(\lim\limits_{\rho \to 0}{\seminorm{v|_{L^{r}}}_{\rho}} = 0\) and, taking the supremum with respect to \(\rho\),
	\[{}
	\seminorm{v}_{\VMO(L^{r})}
	\le \Cr{cteVMO-92} \seminorm{v}_{\VMO^{\#}(K^{\ell}| L^{r})}.
	\qedhere
	\]
\end{proof}

We now rephrase the definition of \(\VMO^{\#}(K^{\ell}|L^{r})\) in terms of a pure \(\VMO\) extension to a suitable \(\ell\)-dimensional set:

\begin{proposition}
	\label{propositionVMOExtension}
	Given a measurable function \(v \colon  K^{\ell} \to \R\), we have \(v \in \VMO^{\#}(K^{\ell}| L^{r})\) if and only if \(\widetilde{v} \in \VMO(E^{\ell})\), where \(E^\ell\vcentcolon= (K^{\ell} \times \{0'\}) \cup (L^{r} \times [0, 1]^{\ell - r})\) is a subset of \(\R^{\ell}\times \R^{\ell-r}\) and \(\widetilde{v}(y, s) \vcentcolon= v(y)\) for every \((y, s) \in E^{\ell}\).
\end{proposition}

We denote by \(\mu\) the measure \(\cH^{\ell}\lfloor_{E^\ell}\), which agrees with \(\cH^{r} \otimes \cH^{\ell - r}\) on \( L^{r} \times [0, 1]^{\ell - r}\).
It is more convenient to work on \(E^{\ell}\) with the equivalent product distance
\begin{equation*}
\widetilde d((y, s), (z, t))
\vcentcolon= \max{\bigl\{ d(y, z), \abs{s - t} \bigr\}}
\quad \text{for every \((y, s), (z, t) \in E^{\ell}\)}.
\end{equation*}
Given \(\xi \in E^{\ell}\) and \(\rho>0\), we denote by \(\widetilde{B}_\rho(\xi)\) the open ball of center \(\xi\) and radius \(\rho\) associated with \(\widetilde{d}\) in \(E^{\ell}\).
	Then, for every \(\xi = (x, t) \in E^{\ell}\),
	\begin{equation}
		\label{eqBallProduct}
	\widetilde{B}_{\rho}(\xi){}
	\subset (B_{\rho}^{\ell}(x) \times \{0'\}) \cup \bigl[ (B_{\rho}^{\ell}(x) \cap L^{r}) \times (B_{\rho}^{\ell - r}(t) \cap [0, 1]^{\ell - r}) \bigr]
	\end{equation}
	with equality when \(\abs{t} < \rho\).
	From \eqref{eqBallProduct}, for a measurable function \(\varphi \colon  E^{\ell} \to [0, +\infty]\) we thus have
	\begin{equation}
		\label{eqEstimateVMOSharpExtension}
	\int_{\widetilde{B}_{\rho}(\xi)}{\varphi \dif\mu}
	\le{}
	\int_{B_{\rho}^{\ell}(x) \times \{0'\}}{\varphi \dif\cH^{\ell}}
	+ 
	\int_{B_{\rho}^{\ell}(x) \cap L^{r}}\biggl( \int_{B_{\rho}^{\ell - r}(t) \cap [0, 1]^{\ell - r}} \varphi \dif\cH^{\ell - r}  \biggr) \dif\cH^{r},
	\end{equation}
	where equality holds when \(\abs{t} < \rho\).

\resetconstant
\begin{proof}[Proof of Proposition~\ref{propositionVMOExtension}]
	``\(\Longleftarrow\)''.
	By restriction of \(\widetilde{v}\) to \(K^{\ell} \times \{0'\}\), we have \(v\in \Lebesgue^1(K^\ell)\).
	For every \(x \in K^{\ell}\) and \(\zeta \in E^{\ell}\), we apply the case of equality in \eqref{eqEstimateVMOSharpExtension} with \(\xi = (x, 0')\) and \(\varphi = \abs{\widetilde{v} - \widetilde{v}(\zeta)}\), for any \(\zeta \in E^{\ell}\), to get
	\begin{multline*}
	\int_{\widetilde{B}_{\rho}(x, 0')}{\abs{\widetilde{v}(\eta) - \widetilde{v}(\zeta)} \dif\mu(\eta)}
	=
	\int_{B_{\rho}^{\ell}(x)}{\abs{v(y) - \widetilde{v}(\zeta)} \dif\cH^{\ell}(y)}\\
	+ 
	C_{\rho} \int_{B_{\rho}^{\ell}(x) \cap L^{r}}\abs{v(y) - \widetilde{v}(\zeta)} \dif\cH^{r}(y),
	\end{multline*}
	with \(C_{\rho} \vcentcolon= \cH^{\ell - r}(B_{\rho}^{\ell - r}(0') \cap [0, 1]^{\ell - r})\).
	By a scaling argument, 
	\begin{equation}
	\label{eqVMOSets-271}
	C_{\rho} = C_{1}\rho^{\ell - r} 
	\quad \text{for every \(0 < \rho \le 1\).}
	\end{equation}
	We next integrate with respect to \(\zeta\) over \(\widetilde{B}_{\rho}(x, 0')\) and use the case of equality in \eqref{eqEstimateVMOSharpExtension} now with \(\varphi = \abs{v(y) - \widetilde{v}}\) for each \(y \in B_{\rho}^{\ell}(x)\).	
	This gives
	\begin{multline}
	\label{eqExtension-617}
    \int_{\widetilde{B}_{\rho}(x, 0')}\int_{\widetilde{B}_{\rho}(x, 0')}{\abs{\widetilde{v}(\eta) - \widetilde{v}(\zeta)} \dif\mu(\eta)}\dif\mu(\zeta)\\
	=
	\int_{B_{\rho}^{\ell}(x)}\int_{B_{\rho}^{\ell}(x)}{\abs{v(y) - v(z)} \dif\cH^{\ell}(y)}\dif\cH^{\ell}(z)\\
	\begin{aligned}
	&+ 
	2 C_{\rho} \int_{B_{\rho}^{\ell}(x)}\int_{B_{\rho}^{\ell}(x) \cap L^{r}}\abs{v(y) - v(z)} \dif\cH^{r}(y)\dif\cH^{\ell}(z)\\
	&+
	C_{\rho}^{2} \int_{B_{\rho}^{\ell}(x) \cap L^{r}}\int_{B_{\rho}^{\ell}(x) \cap L^{r}}\abs{v(y) - v(z)} \dif\cH^{r}(y)\dif\cH^{r}(z).
	\end{aligned}
	\end{multline}	
	We divide both sides by \(\rho^{2\ell}\).{}
	Since, for every \(0 < \rho \le 1\),
	\[
    \mu(\widetilde B_{\rho}(x, 0')) \le c_{1} \rho^{\ell}
	\quad \text{and} \quad
	\cH^{\ell}(B_{\rho}^{\ell}(x)) \ge c_{2} \rho^{\ell},
	\]
    we get
	\[{}
	\seminorm{\widetilde{v}}_{\rho}
	\ge c_{3} \bigl( \seminorm{v}_{\rho} + \seminorm{v}_{K^{\ell}, L^{r}, \rho} \bigr), 
	\]
    where the mean oscillation \(\seminorm{\widetilde{v}}_{\rho}\) is defined in terms of the double integral in the left-hand side of \eqref{eqExtension-617}.
	As \(\rho \to 0\), we deduce that if \(\widetilde{v} \in \VMO(E^{\ell})\), then \(v \in \VMO^{\#}(K^{\ell}| L^{r})\).

	``\(\Longrightarrow\)''.
	In this case, a similar computation with \(\xi = (x, t) \in E^{\ell}\) yields
	\begin{multline*}
	\int_{\widetilde{B}_{\rho}(\xi)}\int_{\widetilde{B}_{\rho}(\xi)}{\abs{\widetilde{v}(\eta) - \widetilde{v}(\zeta)} \dif\mu(\eta)}\dif\mu(\zeta)\\
	\le
	\int_{B_{\rho}^{\ell}(x)}\int_{B_{\rho}^{\ell}(x)}{\abs{v(y) - v(z)} \dif\cH^{\ell}(y)}\dif\cH^{\ell}(z)\\
	\begin{aligned}
	&+ 
	2 C_{\rho} \int_{B_{\rho}^{\ell}(x)}\int_{B_{\rho}^{\ell}(x) \cap L^{r}}\abs{v(y) - v(z)} \dif\cH^{r}(y)\dif\cH^{\ell}(z)\\
	&+
	C_{\rho}^{2} \int_{B_{\rho}^{\ell}(x) \cap L^{r}}\int_{B_{\rho}^{\ell}(x) \cap L^{r}}\abs{v(y) - v(z)} \dif\cH^{r}(y)\dif\cH^{r}(z).
	\end{aligned}
	\end{multline*}	
    Divide both sides by \(\rho^{2\ell}\).
    Recalling that, for  \(0 < \rho \le 1\),
    \[
     \mu(\widetilde B_{\rho}(\xi)) \ge c_{3} \rho^{\ell}
	\quad \text{and} \quad
	\cH^{\ell}(B_{\rho}^{\ell}(x)) \le c_{4} \rho^{\ell},
    \]
    and using \eqref{eqVMOSets-271} and Proposition~\ref{lemmaVMOSharpMixed},  we get
	\[{}
	\seminorm{\widetilde{v}}_{\rho}
	\le C \bigl( \seminorm{v}_{\rho} + \seminorm{v}_{K^{\ell}, L^{r}, \rho} +  \seminorm{v}_{L^{r}, L^{r}, \rho} \bigr)
	\le C' \bigl( \seminorm{v}_{\rho} + \seminorm{v}_{K^{\ell}, L^{r}, \rho} \bigr).
	\]
	When \(v \in \VMO^{\#}(K^{\ell}| L^{r})\),  the quantity in the right-hand side converges to zero as \(\rho \to 0\). Since \(\widetilde{v}\in L^{1}(E^{\ell})\),  we deduce that \(\widetilde{v} \in \VMO(E^{\ell})\).
\end{proof}

We next consider the analogue to \(\VMO^\#(K^\ell| L^{r})\) of the density of uniformly continuous functions in the classical \(\VMO\) setting:

\begin{proposition}
	\label{propositionDensityVMOSharpSets}
	For every \(v\in \VMO^\#(K^\ell| L^{r})\) and every \(s > 0\), the averaged function \(v_{s} \colon  K^{\ell} \to \R\) defined by \eqref{eqVMOConvolution} with \(\mu=\cH^{\ell}\lfloor_{K^{\ell}}\)
satisfies
\[{}
\lim_{s \to 0}{\norm{v_{s} - v}_{\VMO^\#(K^\ell|L^{r})}}
= 0.
\]
\end{proposition}

The first tool in the proof of Proposition~\ref{propositionDensityVMOSharpSets} is a characterization of convergence in \(\VMO^{\#}\) in the spirit of Lemma~\ref{lemmaEquiVMO}:

\begin{lemma}
	\label{lemmaEquiVMOSharp}
	Let \((v_{j})_{j \in \N}\) be a sequence in \(\VMO^{\#}(K^{\ell}| L^{r})\).{}
	Given \(v \in \VMO^{\#}(K^{\ell}| L^{r})\), we have
	\[{}
	v_{j} \to v
	\quad \text{in \(\VMO^{\#}(K^{\ell}| L^{r})\)}
	\]
	if and only if the following properties hold:
	\begin{enumerate}[\((i)\)]
		\item 
		\label{item-Sharp211}
		\(v_{j} \to v\) in \(\VMO(K^{\ell})\),
		\item 
		\label{item-Sharp214}
		\(v_{j}|_{L^{r}} \to v|_{L^{r}}\) in \(\Lebesgue^{1}(L^{r})\),
		\item 
		\label{item-Sharp217}
		for each \(\epsilon > 0\), there exists \(\delta > 0\) such that
		\[{}
		\seminorm{v_{j}}_{K^{\ell}, L^{r}, \rho}
		\le \epsilon{}
		\quad \text{for every \(0 < \rho \le \delta\) and \(j \in \N\).}
		\]
	\end{enumerate}
\end{lemma}

\begin{proof}[Proof of Lemma~\ref{lemmaEquiVMOSharp}]
	``\(\Longrightarrow\)''{}
	Property \((\ref{item-Sharp211})\) follows from the definition of the \(\VMO^{\#}\) norm, while \((\ref{item-Sharp214})\) is a consequence of Proposition~\ref{propositionVMOSharpRestrictionSkeleton}.
	It thus suffices to prove \((\ref{item-Sharp217})\).{}
	Note that, for every \(j \in \N\),{}
	\[{}
	\seminorm{v_{j}}_{K^{\ell}, L^{r}, \rho}
	\le \seminorm{v_{j} - v}_{\VMO^{\#}(K^{\ell}| L^{r})} + \seminorm{v}_{K^{\ell}, L^{r}, \rho}\,.
	\]
	Hence, for every \(\delta > 0\),{}
	\[{}
	\limsup_{j \to \infty}{\Bigl( \sup_{\rho \le \delta}{\seminorm{v_{j}}_{K^{\ell}, L^{r}, \rho}}  \Bigr)}
	\le \sup_{\rho \le \delta}{\seminorm{v}_{K^{\ell}, L^{r}, \rho}}.
	\]
	One then proceeds as in the first part of the proof of Lemma~\ref{lemmaEquiVMO} using here the fact that \(v \in \VMO^{\#}(K^{\ell}| L^{r})\).
	
	``\(\Longleftarrow\)''{}
	By \((\ref{item-Sharp211})\), we have
	\[{}
	\limsup_{j \to \infty}{\norm{v_{j} - v}_{\VMO^{\#}(K^{\ell}| L^{r})}}
	\le \limsup_{j \to \infty}{\seminorm{v_{j} - v}_{\VMO^{\#}(K^{\ell}| L^{r})}}.
	\]
	As in the proof of \eqref{eqBMOLipschitz}, for every \(\delta > 0\) one considers separately the cases \(\rho \le \delta\) and \(\rho > \delta\) to get
	\begin{multline*}
	\seminorm{v_{j} - v}_{\VMO^{\#}(K^{\ell}| L^{r})}
	\le C_{\delta} \bigl(\norm{v_{j} - v}_{\Lebesgue^{1}(K^{\ell})} + \norm{v_{j} - v}_{\Lebesgue^{1}(L^{r})} \bigr)\\
	 + \sup_{\rho \le \delta}{\seminorm{v_{j}}_{K^{\ell}, L^{r}, \rho}} + \sup_{\rho \le \delta}{\seminorm{v}_{K^{\ell}, L^{r}, \rho}}\,.
	\end{multline*}
	Hence, by \((\ref{item-Sharp211})\) and \((\ref{item-Sharp214})\),
	\[
	\limsup_{j \to \infty}{\seminorm{v_{j} - v}_{\VMO^{\#}(K^{\ell}| L^{r})}}
	\le \limsup_{j \to \infty}{\Bigl( \sup_{\rho \le \delta}{\seminorm{v_{j}}_{K^{\ell}, L^{r}, \rho}} \Bigr)} + \sup_{\rho \le \delta}{\seminorm{v}_{K^{\ell}, L^{r}, \rho}}\,.
	\]
	The conclusion then follows, since the right-hand side converges to zero as \(\delta \to 0\) by  \((\ref{item-Sharp217})\) and the fact that  \(v \in \VMO^{\#}(K^{\ell}| L^{r})\).{}
\end{proof}

The proof of Proposition~\ref{propositionDensityVMOSharpSets} also relies on the following estimates:

\begin{lemma}
	\label{lemmaVMOSharpAverage}
	Let \(\cL^{r} \subset \cL^{\ell}\) be subcomplexes of  \(\cK^{\ell}\) and,  for \(j \in \{r, \ell\}\) and \(s > 0\), denote the neighborhood of \(L^{j}\) of radius \(s/2\) by 
    \[
    A_{s, j} = \bigl\{\xi \in K^{\ell} : d(\xi, L^{j}) < s/2 \bigr\}.
    \]
	If \(v \colon  L^{\ell} \to \R\) is summable and \(v|_{L^{r}}\) is summable, then the averaged function \(v_{s, j} \colon A_{s, j} \to \R\) defined for \(j \in \{r, \ell\}\) by
	\[{}
	v_{s, j}(x)
	= \fint_{B_{s}^{\ell}(x) \cap L^{j}} v \dif\cH^{j}
	\]
	is such that, for every \(x \in K^{\ell}\) and every \(0 < \rho \le \Diam{K^{\ell}}\),
	\begin{multline*}
	\frac{1}{\rho^{\ell + r}}
	\int_{B_{\rho}^{\ell}(x) \cap A_{s, j}}{\int_{B_{\rho}^{\ell}(x) \cap L^{r}}{\abs{v_{s, j}(y) -  v_{s, r}(z)} \dif\cH^{r}(z) \dif\cH^{\ell}(y)} }\\
	\le	C \max{\bigl\{ [v]_{L^{j}, L^{r}, 2\rho}\,, [v]_{L^{j}, L^{r}, 2s} \bigr\}}\,.
	\end{multline*}
\end{lemma}

\resetconstant
\begin{proof}[Proof of Lemma~\ref{lemmaVMOSharpAverage}]
When \(s \ge \Diam{K^\ell}\), \(v_{s, j}\) is a constant, and the estimate is straightforward.
We may thus assume in the sequel that \(0 < s < \Diam{K^\ell}\).
Take \(y \in A_{s, j}\) and \(z\in L^r\).
	By monotonicity of the integral,
	\[{}
	\abs{v_{s, j}(y) - v_{s, r}(z)}
	\le \fint_{B_{s}^{\ell}(y) \cap L^{j}}\fint_{B_{s}^{\ell}(z) \cap L^{r}}{\abs{v(\eta) - v(\zeta)} \dif\cH^{r}(\zeta)\dif\cH^{j}(\eta)}.
	\]
	Since
	\[{}
	\cH^{j}(B_{s}^{\ell}(y) \cap L^{j}) \ge c_{1}s^{j}
	\quad \text{and} \quad{}
	\cH^{r}(B_{s}^{\ell}(z) \cap L^{r}) \ge c_{1}s^{r},
	\]
	we get
	\[{}
	\abs{v_{s, j}(y) - v_{s, r}(z)}
	\le \frac{\C}{s^{j + r}} \int_{B_{s}^{\ell}(y) \cap L^{j}}\int_{B_{s}^{\ell}(z) \cap L^{r}}{\abs{v(\eta) - v(\zeta)} \dif\cH^{r}(\zeta)\dif\cH^{j}(\eta)}.
	\]
	Given \(x\in K^\ell\) and \(0 < \rho \le \Diam{K^\ell}\), we integrate with respect to \(y\) and \(z\) and apply Tonelli's theorem to get
	\begin{multline*}
	\int_{B_{\rho}^{\ell}(x) \cap A_{s, j}} \int_{B_{\rho}^{\ell}(x) \cap L^{r}} \abs{v_{s, j}(y) - v_{s, r}(z)} \dif\cH^{r}(z) \dif\cH^{\ell}(y)\\
	\le \frac{\C}{s^{j + r}} \int_{B_{\rho + s}^{\ell}(x) \cap L^{j}}{\int_{B_{\rho + s}^{\ell}(x) \cap L^{r}}{\abs{v(\eta) - v(\zeta)}\alpha_{\ell}(\eta)  \alpha_{r}(\zeta) \dif\cH^{r}(\zeta) \dif\cH^{j}(\eta)} },
	\end{multline*}
	where 
	\[{}
	\alpha_{i}(\xi) \vcentcolon= \cH^{i}\bigl(B_{\rho}^{\ell}(x) \cap B_{s}^{\ell}(\xi) \cap L^{i}\bigr)
	\le \C \min{\{\rho^{i}, s^{i}\}}.
	\]
	For \(\rho < s\), we have
	\begin{multline*}
	\int_{B_{\rho}^{\ell}(x) \cap A_{s, j}} \int_{B_{\rho}^{\ell}(x) \cap L^{r}} \abs{v_{s, j}(y) - v_{s, j}(z)} \dif\cH^{r}(z) \dif\cH^{\ell}(y)\\
	\le \Cl{cte-288} \frac{\rho^{\ell + r}}{s^{j + r}}  \int_{B_{2s}^{\ell}(x) \cap L^{j}}{\int_{B_{2 s}^{\ell}(x) \cap L^{r}}{\abs{v(\eta) - v(\zeta)} \dif\cH^{r}(\zeta) \dif\cH^{j}(\eta)} }.
	\end{multline*}
	For \(s \le \rho\), we use that \(j \le \ell\) to estimate \(s^{\ell - j} \le \rho^{\ell - j}\)
	and then
	\begin{multline*}
	\int_{B_{\rho}^{\ell}(x) \cap A_{s, j}} \int_{B_{\rho}^{\ell}(x) \cap L^{r}} \abs{v_{s, j}(y) - v_{s, j}(z)} \dif\cH^{r}(z) \dif\cH^{\ell}(y)\\
	\le \Cr{cte-288} \rho^{\ell - j} \int_{B_{2\rho}^{\ell}(x) \cap L^{j}}{\int_{B_{2\rho}^{\ell}(x) \cap L^{r}}{\abs{v(\eta) - v(\zeta)} \dif\cH^{r}(\zeta) \dif\cH^{j}(\eta)} }.
	\end{multline*}
	The conclusion then follows from both integral estimates.
\end{proof}

\resetconstant
\begin{proof}[Proof of Proposition~\ref{propositionDensityVMOSharpSets}]
Let  \((s_{j})_{j \in \N}\) be any sequence of positive numbers that converges to zero. Since each \(v_{s_j}\) is uniformly continuous in \(K^\ell\), it belongs to \(\VMO^\#(K^\ell| L^r)\).
	We then verify that the assumptions of Lemma~\ref{lemmaEquiVMOSharp} are satisfied by \((v_{s_{j}})_{j \in \N}\).
	By Proposition~\ref{propositionDensityVMO} we have
	\begin{equation*}
	v_{s_{j}} \to v
	\quad \text{in \(\VMO(K^{\ell})\).}
	\end{equation*}
	We now prove that
	\begin{equation}
	\label{eqDensityVMOSharp2}
	v_{s_{j}}|_{L^{r}} \to v|_{L^{r}}
	\quad \text{in \(\Lebesgue^{1}(L^{r})\).}
	\end{equation}
	To this end, we use the function \(v_{s, r}\) introduced in Lemma~\ref{lemmaVMOSharpAverage}.
	We recall that, by Proposition~\ref{propositionVMOSharpRestrictionSkeleton}, \(v|_{L^{r}} \in \Lebesgue^{1}(L^{r})\), whence \(v_{s, r}\) is well defined.
	By the triangle inequality,
	\[{}
	\norm{v_{s} - v}_{\Lebesgue^{1}(L^{r})}
	\le \norm{v_{s} - v_{s, r}}_{\Lebesgue^{1}(L^{r})} + \norm{v_{s, r} - v}_{\Lebesgue^{1}(L^{r})}.
	\]
	The last term converges to zero as \(s \to 0\).
    In fact, since \(v_{s, r}\) is the averaged function associated to \(v\) in \(L^{r}\), this convergence follows from a standard argument based on the density of \(\Smooth^0(L^r)\) in \(\Lebesgue^1(L^r)\) and the estimate \(\norm{f_{s,r}}_{\Lebesgue^1(L^r)}\leq C \norm{f}_{\Lebesgue^1(L^r)}\).
	We next observe that
	\begin{equation}
	\label{eqDensityVMOSharp2a}
	\norm{v_{s} - v_{s, r}}_{L^{\infty}(A_{s, r})}
	\le \C \seminorm{v}_{K^{\ell}, L^{r}, s}\,,
	\end{equation}
	where \(A_{s, r}\) is the neighborhood of \(L^{r}\) defined in Lemma~\ref{lemmaVMOSharpAverage} with \(L^{\ell}=K^\ell\). 
	Indeed, for every \(x \in A_{s, r}\)\,,
	\[{}
	\abs{v_{s}(x) - v_{s, r}(x)}
	\le \fint_{B_{s}^{\ell}(x)}\fint_{B_{s}^{\ell}(x) \cap L^{r}}{\abs{v(y) - v(z)} \dif\cH^{r}(z) \dif\cH^{\ell}(y)}.
	\]
	Observe that, for every \(x \in K^{\ell}\) and every
	 \(0 < s \le \Diam{K^{\ell}}\), we have
	\[{}
	\cH^{\ell}(B_{s}^{\ell}(x)) \ge c_{1} s^{\ell},
	\]
	while for every \(0 < s \le \Diam{K^{\ell}}\) and every \(x \in A_{s, r}\),
	\[{}
	\cH^{r}(B_{s}^{\ell}(x) \cap L^{r}) \ge c_{2} s^{r}.
	\]
	Hence,
	\[{}
	\abs{v_{s}(x) - v_{s, r}(x)}
	\le \frac{\C}{s^{\ell + r}} \int_{B_{s}^{\ell}(x)}\int_{B_{s}^{\ell}(x) \cap L^{r}}{\abs{v(y) - v(z)} \dif\cH^{r}(z) \dif\cH^{\ell}(y)},
	\]
	which implies \eqref{eqDensityVMOSharp2a}.
	From \eqref{eqDensityVMOSharp2a}, we thus have
	\[{}
	\norm{v_{s} - v_{s, r}}_{\Lebesgue^{1}(L^{r})}
	\le \cH^{r}(L^{r}) \norm{v_{s} - v_{s, r}}_{\Lebesgue^{\infty}(L^{r})}
	\le \C \seminorm{v}_{K^{\ell}, L^{r}, s}\,,
	\]
	and the quantity in the right-hand side converges to zero as \(s \to 0\) since \(v \in \VMO^{\#}(K^{\ell}| L^{r})\).
	Convergence \eqref{eqDensityVMOSharp2} thus follows.
	
	As \(v \in \VMO^{\#}(K^{\ell}| L^{r})\), to verify \((\ref{item-Sharp217})\) in Lemma~\ref{lemmaEquiVMOSharp} it suffices to prove that
	\begin{equation}
	\label{eqDensityVMOSharp3}
	\seminorm{v_{s}}_{K^{\ell}, L^{r}, \rho}
	\le \C \max{\bigl\{\seminorm{v}_{K^{\ell}, L^{r}, 2\rho}\,, \seminorm{v}_{K^{\ell}, L^{r}, 2s} \bigr\}},
	\end{equation}
	for every \(0 < \rho, s \le \Diam{K^{\ell}}\).
	To this end, we first write \(v_{s}\) as
	\[{}
	v_{s} = \chi_{A_{s, r}}(v_{s} - v_{s, r}) + V_{s}\,,
	\]
	where
	\[{}
	V_{s} \vcentcolon={}
	 \chi_{A_{s, r}} v_{s, r} + \chi_{K^{\ell} \setminus A_{s, r}} v_{s}\,.
	\]
	Then,
	\[
	\seminorm{v_{s}}_{K^{\ell}, L^{r}, \rho} 
	\le \C \norm{v_{s} - v_{s, r}}_{L^{\infty}(A_{s, r})} + \seminorm{V_{s}}_{K^{\ell}, L^{r}, \rho}.
	\]
	In view of \eqref{eqDensityVMOSharp2a}, it suffices to estimate the second term in the right-hand side.
	We take \(y \in K^{\ell}\) and \(z \in L^{r}\).{}
	In particular, \(z \in A_{s, r}\) and then \(V_{s}(z) = v_{s, r}(z)\).{}
	Thus,
	\[{}
	\begin{split}
	\abs{V_{s}(y) - V_{s}(z)}
	& \le \chi_{A_{s, r}}(y) \abs{v_{s, r}(y) - v_{s, r}(z)} + \chi_{K^{\ell} \setminus A_{s, r}}(y) \abs{v_{s}(y) - v_{s, r}(z)}\\
	& \le \chi_{A_{s, r}}(y) \abs{v_{s, r}(y) - v_{s, r}(z)} + \abs{v_{s}(y) - v_{s, r}(z)}. 
	\end{split}
	\]
	Integrating with respect to \(y\) and \(z\), for every \(x \in K^{\ell}\) we get
	\begin{multline*}
	\int_{B_{\rho}^{\ell}(x)}{\int_{B_{\rho}^{\ell}(x) \cap L^{r}}{\abs{V_{s}(y) -  V_{s}(z)} \dif\cH^{r}(z) \dif\cH^{\ell}(y)} }\\
	\le \int_{B_{\rho}^{\ell}(x) \cap A_{s, r}}{\int_{B_{\rho}^{\ell}(x) \cap L^{r}}{\abs{v_{s, r}(y) -  v_{s, r}(z)} \dif\cH^{r}(z) \dif\cH^{\ell}(y)} } \\
	+ \int_{B_{\rho}^{\ell}(x)}{\int_{B_{\rho}^{\ell}(x) \cap L^{r}}{\abs{v_{s}(y) -  v_{s, r}(z)} \dif\cH^{r}(z) \dif\cH^{\ell}(y)} }.
	\end{multline*}
	Note that \(v_s = v_{s, \ell} \).
    We may then apply Lemma~\ref{lemmaVMOSharpAverage} with \(\cL^{\ell} = \cK^{\ell}\) to both terms in the right-hand side to get
	\[{}
	\seminorm{V_{s}}_{K^{\ell}, L^{r}, \rho} 
	\le \C \Bigl(\max{\bigl\{\seminorm{v}_{L^{r}, L^{r}, 2\rho}\,, \seminorm{v}_{L^{r}, L^{r}, 2s} \bigr\}} + \max{\bigl\{\seminorm{v}_{K^{\ell}, L^{r}, 2\rho}\,, \seminorm{v}_{K^{\ell}, L^{r}, 2s} \bigr\}}\Bigr).
	\]
	Applying Proposition~\ref{lemmaVMOSharpMixed} to the first term in parentheses, we thus get
	\[{}
	\seminorm{V_{s}}_{K^{\ell}, L^{r}, \rho} 
	\le \C \max{\bigl\{\seminorm{v}_{K^{\ell}, L^{r}, 2\rho}\,, \seminorm{v}_{K^{\ell}, L^{r}, 2s} \bigr\}}\,,weak-convergence-extendability
	\]
	which implies \eqref{eqDensityVMOSharp3}.
	The proposition now follows from Lemma~\ref{lemmaEquiVMOSharp}.
\end{proof}

In analogy with the standard \(\VMO\) case, one defines the space \(\VMO^{\#}(K^{\ell}| L^{r}; \manfN)\) of functions with values in a compact manifold \(\manfN \subset \R^{\nu}\).

\begin{proposition}
	\label{propositionVMOSharpDensityManifold}
	For every \(v \in \VMO^{\#}(K^{\ell}| L^{r}; \manfN)\), we have
	\[{}
	\lim_{s \to 0}{\norm{\Pi \compose v_{s} - v}_{\VMO^{\#}(K^{\ell}| L^{r})}} = 0.
	\]
	Hence, \(\Smooth^{0}(K^{\ell}; \manfN)\) is dense in \(\VMO^{\#}(K^{\ell}| L^{r}; \manfN)\).{}
\end{proposition}

\resetconstant
\begin{proof}
	By Proposition~\ref{propositionDensityVMO}, 
	\[{}
	\lim_{s \to 0}{\norm{v_{s} - v}_{\VMO(K^{\ell})}}
	= 0.
	\]
	Since the projected map \(\Pi \compose v_{s}\) is well defined for every \(s > 0\) sufficiently small (Lemma~\ref{lemmaVMOUniformConvergence}) and since \(\Pi \compose v = v\), by Lipschitz continuity of \(\Pi\) in a tubular neighborhood of \(\manfN\) we deduce from 	Lemma~\ref{lemma_VMO_average_homotopy} that
	\[{}
	\lim_{s \to 0}{\norm{\Pi \compose v_{s} - v}_{\VMO(K^{\ell})}} = 0.
	\]
	The Lipschitz continuity of \(\Pi\) also implies that
	\[{}
		\norm{\Pi \compose v_{s} - v}_{\Lebesgue^{1}(L^{r})}
		 \le C \norm{v_{s} - v}_{\Lebesgue^{1}(L^{r})}
		  \quad \text{and} \quad
	\seminorm{\Pi \compose v_{s}}_{K^{\ell}, L^{r}, \rho}
	\le \C \seminorm{v_{s}}_{K^{\ell}, L^{r}, \rho}\,.
	\]
	The conclusion is then a consequence of Proposition~\ref{propositionDensityVMOSharpSets} and Lemma~\ref{lemmaEquiVMOSharp} applied to each component of \(v\).
\end{proof}

\(\VMO^{\#}\) functions are well-suited to restrictions of the homotopy equivalence in \(\VMO\):

\begin{proposition}
	\label{propositionVMOSharpHomotopy}
	If \(u, v \in \VMO^{\#}(K^{\ell}| L^{r}; \manfN)\) are such that
	\[{}
	u \sim v 
	\quad \text{in \(\VMO(K^{\ell}; \manfN)\),}
	\]
	then 
	\[{}
	u|_{L^{r}} \sim v|_{L^{r}}
	\quad \text{in \(\VMO(L^{r}; \manfN)\).}
	\]
\end{proposition}

\begin{proof}
	By Proposition~\ref{propositionVMOSharpDensityManifold}, we can take sequences \((u_{j})_{j \in \N}\) and \((v_{j})_{j \in \N}\) in \(\Smooth^{0}(K^{\ell}; \manfN)\) that converge to \(u\) and \(v\), respectively, in terms of the \(\VMO^{\#}(K^{\ell}| L^{r})\) norm.
	In particular, they converge in \(\VMO(K^{\ell}; \manfN)\) and, by Proposition~\ref{propositionVMOSharpRestrictionSkeleton}, their restrictions to \(L^{r}\) converge in \(\VMO(L^{r}; \manfN)\).{}
	Hence, by Proposition~\ref{propositionHomotopyVMOLimit} there exists \(j \in \N\) sufficiently large such that
	\begin{gather}
	\label{eqVMO-617}
	u_{j} \sim u
	\quad \text{and} \quad{}
	v_{j} \sim v
	\quad \text{in \(\VMO(K^{\ell}; \manfN)\),}\\
	u_{j}|_{L^{r}} \sim u|_{L^{r}}
	\quad \text{and} \quad{}
	v_{j}|_{L^{r}} \sim v|_{L^{r}}
	\quad \text{in \(\VMO(L^{r}; \manfN)\).}
	\label{eqHomotopyRestriction}
	\end{gather}
	Since we also have \(u \sim v\) in \(\VMO(K^{\ell}; \manfN)\), it thus follows from \eqref{eqVMO-617} and transitivity of the homotopy relation that
	\[{}
	u_{j} \sim v_{j}
	\quad \text{in \(\VMO(K^{\ell}; \manfN)\).}
	\]
	As \(u_{j}\) and \(v_{j}\) are both continuous, we can apply Proposition~\ref{propositionHomotopyVMOtoC} to deduce that
	\[{}
	u_{j} \sim v_{j}
	\quad \text{in \(\Smooth^{0}(K^{\ell}; \manfN)\).}
	\]
	Hence, by restriction of the homotopy to \(L^{r}\),
	\[{}
	u_{j}|_{L^{r}} \sim v_{j}|_{L^{r}}
	\quad \text{in \(\Smooth^{0}(L^{r}; \manfN)\).}
	\]
	In particular, the homotopy relation also holds in \(\VMO(L^{r}; \manfN)\).
	The conclusion now follows from \eqref{eqHomotopyRestriction} and the transitivity of the homotopy relation.	
\end{proof}

\section{Extension from subskeletons}
\label{sectionExtensionProofs}

In this section, we prove Propositions~\ref{proposition_extension_montone_small} and ~\ref{proposition_Extension_homotop_group_decreasing}.
To do this, we need the following \(\VMO^{\#}\)~genericity of \(\VMO^{\ell}\)~maps.

\begin{lemma}
\label{lemma-generic-vmo-sharp-Small}
Let \(u \colon  \manfV \to \R\) be a measurable function.{}
If \(w\) is an \(\ell\)-detector for \(u\), then, for every simplicial complex \(\cK^{\ell}\), every subcomplex \(\cL^{r}\) and every \(\gamma \colon  K^{\ell} \to \manfV\) such that
\[{}
\gamma \in \Fuglede_{w}(K^{\ell}; \manfV)
\quad \text{and} \quad
\gamma\vert_{L^{r}} \in \Fuglede_{w}(L^{r}; \manfV)
\]
we have
\[{}
u \compose \gamma \in \VMO^{\#}(K^{\ell}| L^{r}).
\]
\end{lemma}
\begin{proof}[Proof of Lemma~\ref{lemma-generic-vmo-sharp-Small}]
Let \(\gamma\) be as in the statement and take a simplicial complex \(\cE^{\ell}\) such that \(E^{\ell} = (K^{\ell} \times \{0'\}) \cup (L^{r} \times [0, 1]^{\ell - r})\).{}
The function \(\widetilde{\gamma} \colon  E^{\ell} \to \manfV\) defined by \(\widetilde{\gamma}(z, s) = \gamma(z)\) for every \((z, s) \in E^{\ell}\) belongs to \(\Fuglede_{w}(E^{\ell}; \manfV)\).{}
Since \(w\) is an \(\ell\)-detector, we have \(u \compose \widetilde{\gamma} \in \VMO(E^{\ell})\).
It follows from Proposition~\ref{propositionVMOExtension} that \(u \compose \gamma \in \VMO^{\#}(K^{\ell}| L^{r})\).
\end{proof}

We now have all the ingredients to prove that \((\ell,e)\)-extendability implies \((i,j)\)-extendability when \(i\leq \ell\) and \(i\leq j\leq e\):

\begin{proof}[Proof of Proposition~\ref{proposition_extension_montone_small}]
Let \(w\) be an \(\ell\)-detector given by Definition~\ref{definitionExtensionVMO}. 
Since \(i \le \ell\), by Proposition~\ref{propositionFugledeDetector} \(w\) is also an \(i\)-detector for \(u\).{}
Let \( \tau \colon U \to \manfV \) be a transversal perturbation of the identity.
We then take an \(i\)-detector \(\widetilde{w} \ge w\) given by Lemma~\ref{propositionGenericStability}.

Let \(\cK^{j}\) be a simplicial complex with \(i\leq j \leq e\) and \(\gamma \colon  K^j \to \manfV\) be a Lipschitz map such that 
\[{}
\gamma\vert_{K^{i}} \in \Fuglede_{\tilde{w}}(K^{i}; \manfV).
\]
We now take any simplicial complex \(\cE^{e}\) that contains \(\cK^{j}\) as a subcomplex together with a Lipschitz extension \(\widehat{\gamma} \colon  E^{e} \to \manfV\) of \(\gamma\).
For example, one can choose \(\cE^{e}\) so that \(E^{e} = K^{j} \times [0, 1]^{e - j}\), where \(K^{j}\) is identified with \(K^{j} \times \{0'\}\), and \(\widehat{\gamma}(x, t) = \gamma(x)\) for every \((x, t) \in E^{e}\).{}

Since \(\widehat{\gamma}|_{K^i}\in \Fuglede_{\tilde{w}}(K^i;\manfV)\), we deduce from  Lemma~\ref{propositionGenericStability} that there exists \( \xi \in B_{\delta}^{q} \) with \(\widehat{\gamma}(E^{e}) \times \{ \xi \} \subset U\) such that the Lipschitz map
\(\widetilde{\gamma} \vcentcolon= \tau_{\xi} \compose \widehat{\gamma} \) verifies 
\[{}
	\widetilde\gamma\vert_{E^{\ell}} \in \Fuglede_{w}(E^{\ell}; \manfV),
	\quad
	\widetilde\gamma\vert_{K^{i}} \in \Fuglede_{w}(K^{i}; \manfV)
	\]
and
\begin{equation}
\label{eqExtension-1045}
u \compose \gamma\vert_{K^{i}}
= u \compose \widehat\gamma\vert_{K^{i}} 
\sim  
u \compose \widetilde\gamma\vert_{K^{i}} 
\quad \text{in \(\VMO(K^{i}; \manfN)\).}
\end{equation}
It follows from Lemma~\ref{lemma-generic-vmo-sharp-Small} applied to the components of \(u\) that 
\begin{equation}
	\label{eqExtension-248}
	u \compose \widetilde{\gamma}\vert_{E^{\ell}} \in \VMO^{\#}(E^{\ell}| K^{i}; \manfN).
\end{equation}
Since \(\widetilde\gamma\vert_{E^{\ell}} \in \Fuglede_{w}(E^{\ell}; \manfV)\), by \((\ell, e)\)-extendability of \(u\), there exists \(F \in \Smooth^{0}(E^{e}; \manfN)\) such that
\[{}
u \compose \widetilde{\gamma}\vert_{E^{\ell}} \sim F\vert_{E^{\ell}}
\quad \text{in \(\VMO(E^{\ell}; \manfN)\).}
\]
By \eqref{eqExtension-248} and Proposition~\ref{propositionVMOSharpHomotopy}, we are allowed to restrict this homotopy relation to \(K^{i}\).{}
Using \eqref{eqExtension-1045} we then get
\[{}
u \compose \gamma\vert_{K^{i}} \sim u \compose \widetilde{\gamma}\vert_{K^{i}} \sim F\vert_{K^{i}}
\quad \text{in \(\VMO(K^{i}; \manfN)\).}
\]
We thus deduce that \(u\) is \((i, j)\)-extendable by taking the continuous map \(F\vert_{K^{j}}\).{}
\end{proof}

We finally prove that \((i,e)\)-extendability implies \((\ell, e)\)-extendability when \(i\leq \ell\) and the homotopy groups \(\pi_j(\manfN)\) are trivial for \(i+1\leq j \leq \ell\)\,:

\begin{proof}[Proof of Proposition~\ref{proposition_Extension_homotop_group_decreasing}]
	Let \(w\) be an \(\ell\)-detector given by the \((i, e)\)-extendability of \(u\) and let \( \tau \colon U \to \manfV \) be a transversal perturbation of the identity. 	
	We take an \(\ell\)-detector \(\widetilde w \ge w\) given by Lemma~\ref{propositionGenericStability} and a Lipschitz map \(\gamma \colon  K^{e} \to \manfV\) such that \(\gamma\vert_{K^{\ell}} \in \Fuglede_{\tilde{w}}(K^{\ell}; \manfV)\).{}
	By Lemma~\ref{propositionGenericStability}, there exists \( \xi \in B_{\delta}^{q} \) with \(\gamma(E^{e}) \times \{ \xi \} \subset U\) such that the Lipschitz map \(\widetilde{\gamma} \vcentcolon= \tau_{\xi} \compose \gamma \) verifies
	\[{}
	\widetilde\gamma\vert_{K^{\ell}} \in \Fuglede_{w}(K^{\ell}; \manfV),
	\quad
	\widetilde\gamma\vert_{K^{i}} \in \Fuglede_{w}(K^{i}; \manfV)
	\]
	and
	\begin{equation}
	\label{eq-ExtensionHomotopy-1}
			u \compose \widetilde{\gamma}\vert_{K^{\ell}} \sim u \compose {\gamma}\vert_{K^{\ell}}
	\quad \text{in \(\VMO(K^{\ell}; \manfN)\).}
	\end{equation}
	It follows from Lemma~\ref{lemma-generic-vmo-sharp-Small} applied to the components of \(u\) that 
	\begin{equation}
		\label{eqExtension-1093}
		u \compose \widetilde{\gamma}|_{K^{\ell}} \in \VMO^{\#}(K^{\ell}| K^{i}; \manfN).
	\end{equation}	
	Since \(\widetilde{\gamma}|_{K^i}\in \Fuglede_w(K^i;\manfV)\), by \((i, e)\)-extendability of \(u\) there exists \(F \in \Smooth^{0}(K^{e}; \manfN)\) such that
	\begin{equation}
	\label{eqExtensionHomotopyN-extension1}
	u \compose \widetilde{\gamma}\vert_{K^{i}} \sim F\vert_{K^{i}}
	\quad \text{in \(\VMO(K^{i}; \manfN)\).}
	\end{equation}
	Since \(u \compose \widetilde{\gamma}|_{K^{\ell}} \in \VMO(K^{\ell}; \manfN)\), by Proposition~\ref{propositionHomotopyVMOContinuousMap} there exists \(G \in \Smooth^{0}(K^{\ell}; \manfN)\) such that
	\begin{equation}
	\label{eq-ExtensionHomotopy-2}
	u \compose \widetilde{\gamma}\vert_{K^{\ell}} \sim G
	\quad  \text{in \(\VMO(K^{\ell}; \manfN)\).}
	\end{equation}
	By \eqref{eqExtension-1093} and Proposition~\ref{propositionVMOSharpHomotopy}, we can restrict \eqref{eq-ExtensionHomotopy-2} to \(K^{i}\) and conclude that
	\[{}
	u \compose \widetilde{\gamma}\vert_{K^{i}} \sim G\vert_{K^{i}}
	\quad  \text{in \(\VMO(K^{i}; \manfN)\).}
	\]
	By \eqref{eqExtensionHomotopyN-extension1} and the transitivity of the homotopy relation, we then get
	\[
	G\vert_{K^{i}} \sim F\vert_{K^{i}}
	\quad \text{in \(\VMO(K^{i}; \manfN)\).}
	\] 
	As both functions are continuous, by Proposition~\ref{propositionHomotopyVMOtoC} they are also homotopic in \(\Smooth^{0}(K^{i}; \manfN)\).{}
	By the topological assumption on \(\manfN\), we thus have
	\[
	G \sim F\vert_{K^{\ell}}
	\quad \text{in \(\Smooth^{0}(K^{\ell}; \manfN)\).}
	\] 
	In particular, they are also homotopic in \(\VMO(K^{\ell}; \manfN)\).{}
	Then, by \eqref{eq-ExtensionHomotopy-1}, \eqref{eq-ExtensionHomotopy-2} and the transitivity of the homotopy relation, we conclude that
	\[{}
	u \compose \gamma\vert_{K^{\ell}} \sim F\vert_{K^{\ell}}
	\quad  \text{in \(\VMO(K^{\ell}; \manfN)\).}
	\]
    Since \(F \in \Smooth^{0}(K^{e}; \manfN)\), we deduce that \(u\) is \((\ell, e)\)-extendable.
	\end{proof}

\section{\texorpdfstring{$ \VMO^{\ell} $}{VMOl}-homotopy extension property}
\label{sectionHEP}

Given a simplicial complex \(\cK^{e}\) and a map \(f \in \Smooth^{0}(K^{\ell}; \manfN)\), the classical homotopy extension property states that if there exists \(F \in \Smooth^{0}(K^{e}; \manfN)\) such that
\[
f \sim F|_{K^{\ell}} \quad \text{in} \ \Smooth^{0}(K^{\ell}; \manfN),
\]
then \(f\) has a continuous extension in \(\Smooth^{0}(K^{e}; \manfN)\). 
The notion of \((\ell, e)\)-extendability was modeled on this property to handle the fact that we are dealing with maps that are merely \( \VMO \). 
Our goal in this section is to present an analogue of the homotopy extension property at the level of \( \VMO^{\ell} \)~maps defined on an open subset of \(\manfV\). 
We begin by introducing a notion of homotopy that is better suited to this setting.

\begin{definition}
    \label{definitionVMOEllHomotopy}
    Let \( A \subset \manfV \) be an open set.
    Two maps \( u, v \in \VMO^{\ell}(A; \manfN) \) are \emph{\( \VMO^{\ell} \)-homotopic}, which we denote by
    \[
    u \sim v
    \quad \text{in \( \VMO^{\ell}(A; \manfN) \),}
    \]     
    whenever there exists an \( \ell \)-detector \( w \) for \( u \) and \( v \) such that, for every simplicial complex \( \cK^{\ell} \) and every \( \gamma \in \Fuglede_{w}(K^{\ell}; A) \),
    \[
    u \compose \gamma \sim v \compose \gamma
    \quad \text{in \( \VMO(K^{\ell}; \manfN) \).}
    \]
\end{definition}
    
Observe that if two maps \( u, v \in \VMO^{\ell}(A; \manfN) \) are \( \VMO^{\ell} \)-homotopic,  and if \(u\) is \((\ell,e)\)-extendable,  then  \(v\) is  \((\ell,e)\)-extendable.
Moreover, this definition is stable under \( \VMO^{\ell} \) convergence in the following sense:

\begin{proposition}
    \label{propositionHomotopyVMOEll}
    Let \( v \in \VMO^{\ell}(\manfV; \manfN) \) and let \( (v_{j})_{j \in \N} \) be a sequence in \( \VMO^{\ell}(\manfV; \manfN) \) such that \( v_{j} \to v \) in \( \VMO^{\ell}(\manfV; \R^{\nu}) \).
    If a map \( u \in \VMO^{\ell}(\manfV; \manfN) \) is \(\VMO^\ell\)-homotopic to each \(v_j\), then \(u\) is also \(\VMO^\ell\)-homotopic to \(v\).
\end{proposition}

\begin{proof}
    For each \( j \in \N \), let \( w_{j} \) be an \( \ell \)-detector given by Definition~\ref{definitionVMOEllHomotopy} for the homotopy between \( u \) and \( v_{j} \).
    Let \( (\alpha_{j})_{j \in \N} \) be a sequence of positive numbers such that \( w \vcentcolon= \sum_{j}{\alpha_{j}w_{j}} \) is summable.
    Note that \( w \) is a common \( \ell \)-detector for all these homotopies. 
    By convergence of the sequence \( (v_{j})_{j \in \N} \) in \( \VMO^{\ell}(\manfV; \R^{\nu}) \) we then take a summable function \( \widetilde w \) given by Definition~\ref{def_conv_VMOl} of \(\VMO^\ell\) convergence, which we may suppose to satisfy \( \widetilde w \ge w \). 
    Given a simplicial complex \( \cK^{\ell} \) and  \( \gamma \in \Fuglede_{\tilde w}(K^{\ell}; \manfV) \), then, for every \( j \in \N \), 
    \begin{equation}
        \label{eqStrong-687}
    u \compose \gamma|_{K^\ell} \sim v_{j} \compose \gamma|_{K^\ell}
    \quad \text{in \( \VMO(K^{\ell}; \manfN) \).}
    \end{equation}
    Moreover, by Definition~\ref{def_conv_VMOl},
    \[
    v_j \compose \gamma|_{K^\ell} \to v \compose \gamma|_{K^\ell}
    \quad \text{in \( \VMO(K^{\ell}; \R^{\nu}) \).}
    \]
    By Proposition~\ref{propositionHomotopyVMOLimit}, there exists \(J \in \N\) such that, for every \(j \ge J\),{}
   \begin{equation}
    \label{eqStrong-698}
    v_{j} \compose \gamma|_{K^\ell} \sim v \compose \gamma|_{K^\ell}
    \quad \text{in \(\VMO(K^{\ell}; \manfN)\).}
    \end{equation}
    Take such an integer \(j\).
    By transitivity of the homotopy relation, we may combine \eqref{eqStrong-687} and \eqref{eqStrong-698} to get
    \[
    u \compose \gamma|_{K^\ell} \sim v \compose \gamma|_{K^\ell}
    \quad \text{in \( \VMO(K^{\ell}; \manfN) \),}
    \]
    which gives the conclusion.
\end{proof}

A homotopy extension property involving \( \VMO^{\ell} \)~maps, adapted for our purposes in Chapter~\ref{chapter-approximation-Sobolev-manifolds}, is given by the following:

\begin{proposition}
    \label{propositionEllEExtensionContinuous}
    Let \( A \subset \manfV \) be an open set.
    If \( u \in \Smooth^{0}(A; \manfN) \) is such that there exists an \((\ell, e)\)-extendable map \( v \in \VMO^{\ell}(\manfV; \manfN)\) with
    \[
    u \sim v
    \quad \text{in \( \VMO^{\ell}(A; \manfN) \),}
    \]
    then, for every simplicial complex \( \cK^{e} \) and every \( \gamma \in \Smooth^{0}(K^{e}; \manfV) \) with \( \gamma(K^{\ell}) \subset A \), the map \( u \compose \gamma|_{K^{\ell}} \) has a continuous extension in \(\Smooth^0(K^{e} ; \manfN)\).
\end{proposition}

\begin{proof}
    Let \( \gamma \in \Smooth^{0}(K^{e}; \manfV) \) with \( \gamma(K^{\ell}) \subset A \).
    Let \( w \colon \manfV \to [0, +\infty] \) be an \( \ell \)-detector for \( v \) given by Definition~\ref{definitionExtensionVMO} and let \( \widetilde{w} \colon A \to [0, +\infty] \) be an \( \ell \)-detector for \( u \) and \( v|_{A} \) given by Definition~\ref{definitionVMOEllHomotopy}.
    We may assume that \( \widetilde{w} \ge w \) in \( A \).
    Take a sequence of Lipschitz maps \( \gamma_{j} \colon K^{e} \to \manfV \) such that \( \gamma_{j} \to \gamma \) uniformly.
    Let \( \tau \colon U \to \manfV \) be a transversal perturbation of the identity, where \(U\) is an open subset of \(\manfV\times \R^q\) with \(\manfV\times \{0\}\subset U\). 
    Since \(\tau(z,0)=z\) for every \(z\in \manfV\) and  \( \gamma(K^{\ell})\subset A\), we deduce that the compact set \(\gamma(K^{\ell})\times \{0\} \) is contained in the open set \( \tau^{-1}(A)\). 
    Hence,  there exists \(\delta>0\) such that \(B_{\delta}^m(\gamma(K^\ell))\times B^{q}_{\delta}\subset \tau^{-1}(A)\). 
    By uniform convergence of \((\gamma_j)_{j \in \N}\)\,, we may also assume that  \( \gamma_j (K^{\ell})\subset  B_{\delta}^m(\gamma(K^\ell))\)  for every \( j \in \N \).

    Using Proposition~\ref{lemmaTransversalFamily}, one finds a subsequence \((\gamma_{j_i})_{i \in \N}\) and a sequence \( (\xi_{i})_{i \in \N} \) in \( B_{\delta}^q \) that converges to  \( 0 \) such that \( \tau_{\xi_{i}} \compose \gamma_{j_i}|_{K^{\ell}} \in \Fuglede_{\tilde w}(K^{\ell}; A) \) for every \(i \in \N\).
    Since \( u \) and \(v|_{A} \) are \(\VMO^\ell\)-homotopic in \(A\), for every \( i \in \N \) we have
    \[
    u \compose \tau_{\xi_i} \compose \gamma_{j_i}|_{K^{\ell}} 
    \sim v \compose \tau_{\xi_i} \compose \gamma_{j_i}|_{K^{\ell}}
    \quad \text{in \(\VMO(K^{\ell}; \manfN)\).}
    \]
    Since we also have \( \tau_{\xi_{i}} \compose \gamma_{j_i}|_{K^{\ell}} \in \Fuglede_{w}(K^{\ell}; \manfV) \), there exists \( F_i \in \Smooth^0(K^{e}; \manfN) \) such that
    \[
    v \compose \tau_{\xi_i} \compose \gamma_{j_i}|_{K^{\ell}}
    \sim F_{i}|_{K^{\ell}}
    \quad \text{in \(\VMO(K^{\ell}; \manfN)\).}
    \]
    Thus, by transitivity of the homotopy relation,
    \[
    u \compose \tau_{\xi_i} \compose \gamma_{j_i}|_{K^{\ell}} 
    \sim F_i|_{K^{\ell}}
    \quad \text{in \(\VMO(K^{\ell}; \manfN)\).}
    \]
    As both functions are continuous, by Proposition~\ref{propositionHomotopyVMOtoC}  the homotopy relation is also true in \(\Smooth^{0}(K^{\ell}; \manfN)\).
    By the homotopy extension property, it follows that \(u \compose \tau_{\xi_i} \compose \gamma_{j_i}|_{K^{\ell}}\) has a continuous extension in \(\Smooth^0(K^{e} ; \manfN)\).
    Since the sequence \( (u \compose \tau_{\xi_i} \compose \gamma_{j_i}|_{K^{\ell}})_{i \in \N} \) converges uniformly to \( u \compose \gamma|_{K^{\ell}} \), the map \( u \compose \gamma|_{K^{\ell}} \) also has a continuous extension in \(\Smooth^0(K^{e} ; \manfN)\).
    \end{proof}

We can formulate the notion of \(\VMO^\ell\)-homotopy in the spirit of the characterization of \((\ell, e)\)-extendability by almost every translation of Lipschitz maps presented in Chapter~\ref{chapterExtendability}:

\begin{proposition}
    \label{propositionWhiteHomotopyMaps}
    Let \(A \subset \manfV\) be an open set, and let \(\tau \colon U \to \manfV\) be a transversal perturbation of the identity.
    Two maps \(u, v \in \VMO^\ell(A; \manfN)\) are \(\VMO^\ell\)-homotopic if and only if for every simplicial complex \(\cK^\ell\), every Lipschitz map \(\gamma \colon K^\ell \to A\) and for almost every \(\xi \in \R^q\) in a neighborhood of \(0\), we have
    \begin{equation}
    \label{eqExtensionProperty-1284}
    u \compose \tau_{\xi} \compose \gamma \sim v \compose \tau_{\xi} \compose \gamma
    \quad \text{in \(\VMO(K^\ell; \manfN)\).}
    \end{equation}
\end{proposition}

\begin{proof}
    ``\(\Longrightarrow\)''. Let \(w\) be an \(\ell\)-detector given by Definition~\ref{definitionVMOEllHomotopy}.
    Given a Lipschitz map \(\gamma \colon K^e \to A\), take \(\delta > 0\) such that \(\gamma(K^e) \times B_\delta^q \subset \tau^{-1}(A)\).
    By \eqref{eqExtension-568}, for almost every \(\xi \in B_\delta^q\) we have \(\tau_\xi \compose \gamma|_{K^\ell} \in \Fuglede_{w}(K^\ell; A)\).
    For any \(\xi \in B_\delta^q\) satisfying this property we then have 
\[
u \compose \tau_{\xi} \compose \gamma|_{K^\ell} \sim v \compose \tau_{\xi} \compose \gamma|_{K^\ell}
\quad \text{in \(\VMO(K^\ell; \manfN)\).}
\]

    ``\(\Longleftarrow\)''. 
    Let \(w\) be an \(\ell\)-detector for \(u\) and \(v\) and take a summable function \(\widetilde{w} \ge w\) given by Proposition~\ref{propositionFugledeApproximation}.
    Let \(\cK^\ell\) be a simplicial complex and \(\gamma \in \Fuglede_{\tilde w}(K^\ell; A)\).
    We apply Proposition~\ref{propositionFugledeApproximation} to the constant sequence \(\gamma_j \vcentcolon= \gamma\) to get a sequence \((\xi_i)_{i \in \N}\) converging to \(0\) in \(\R^q\) such that \(\tau_{\xi_i} \compose \gamma \to \gamma\) in \(\Fuglede_{w}(K^\ell; \manfV)\).
    By genericity of the homotopy relation \eqref{eqExtensionProperty-1284}, we may choose \(\xi_i\) such that, for every \(i \in \N\), 
    \begin{equation}
    \label{eqExtension-1386}
    u \compose \tau_{\xi_i} \compose \gamma \sim v \compose \tau_{\xi_i} \compose \gamma
    \quad \text{in \(\VMO(K^\ell; \manfN)\).}
    \end{equation}
    Then, by stability of \(\VMO^\ell\)~functions,
    \[
    u \compose \tau_{\xi_i} \compose \gamma \to u \compose \gamma
    \quad \text{and} \quad 
    v \compose \tau_{\xi_i} \compose \gamma \to v \compose \gamma
    \quad \text{in \(\VMO(K^\ell; \manfN)\).}
    \]
    By Proposition~\ref{propositionHomotopyVMOLimit}, for \(i\) sufficiently large we thus have
    \[
    u \compose \tau_{\xi_i} \compose \gamma \sim u \compose \gamma
    \quad \text{and} \quad 
    v \compose \tau_{\xi_i} \compose \gamma \sim v \compose \gamma
    \quad \text{in \(\VMO(K^\ell; \manfN)\).}
    \]
    Therefore, by \eqref{eqExtension-1386} and the transitivity of the homotopy relation,
    \[
    u \compose \gamma \sim v \compose \gamma
    \quad \text{in \(\VMO(K^\ell; \manfN)\).}
    \]
    We conclude that \(u\) and \(v\) are \(\VMO^\ell\)-homotopic.
\end{proof}

Assuming that \(u, v \in \VMO^\ell(\manfV; \manfN)\) are \(\VMO^\ell\)-homotopic and also \(\ell\)-extendable, a combination of Corollary~\ref{corollatyWhiteHomotopyType} and Proposition~\ref{propositionWhiteHomotopyMaps}, implies that for every simplicial complex \(\cK^\ell\), every Lipschitz maps \(\gamma_0, \gamma_1 \colon K^\ell \to \manfV\) that are homotopic in \(\Smooth^0(K^\ell; \manfV)\) and for almost every \(\xi_0, \xi_1 \in \R^q\) in a neighborhood of \(0\), one has
    \[
    u \compose \tau_{\xi_0} \compose \gamma_0 \sim v \compose \tau_{\xi_1} \compose \gamma_1
    \quad \text{in \(\VMO(K^\ell; \manfN)\).}
    \]
This property is the \emph{\(\ell\)-homotopy between \(u\) and \(v\)},  formulated by Hang and Lin~\cite{Hang-Lin}.
This establishes an equivalence between their definition of homotopy between two Sobolev maps and ours based on the formalism of Fuglede's maps in the context of \(\VMO^\ell\) maps.

Hang and Lin have established in \cite{Hang-Lin} the following fundamental result: Two maps \(u, v \in \Sobolev^{1,p}(\manfV; \manfN)\) can be connected by a continuous path in \(\Sobolev^{1,p}(\manfV; \manfN)\) if and only if \(u\) and \(v\) are \((\floor{p}-1)\)-homotopic in the sense of \cite{Hang-Lin}. 
It is plausible that the case \(k \geq 2\) could be addressed using the Fuglede maps used in this work. Indeed, recall that, by Proposition~\ref{propositionLExtensionSobolev}, every map in \(\Sobolev^{k,p}(\manfV; \manfN)\) is \((\floor{kp} - 1)\)-extendable. 
One may conjecture that two maps \(u, v \in \Sobolev^{k,p}(\manfV; \manfN)\) can be continuously connected in \(\Sobolev^{k,p}(\manfV; \manfN)\) whenever they are  \(\VMO^{\floor{kp}-1}\)-homotopic.

The consideration of homotopic classes of maps is useful in the study of minimization problems on compact manifolds \(\manfM\) without boundary.
Given \(1 < p <\infty \), let us consider the functional
\begin{equation}\label{min-harm-map}
E_{1, p} \colon u \in \Sobolev^{1, p}(\manfM;\manfN) \longmapsto \int_{\manfM} |Du|^p\dif x. 
\end{equation}
When \(\manfM\) has no boundary, one may look for nontrivial critical points of \(E_{1, p}\).
To prove their existence, one may minimize \(E_{1, p}\) on subclasses of \(\Sobolev^{1,p}(\manfM;\manfN)\) that are sequentially weakly closed, so that the direct method in the calculus of variations may be applied.
For example, as observed by White~\cite{White}, the \((\floor{p}-1)\)-homotopy type is stable under weak \(\Sobolev^{1,\floor{p}}\)~convergence. 
Due to the compact imbedding \(\Sobolev^{1,p}(\manfM;\manfN)\subset \Sobolev^{s,q}(\manfM;\manfN)\) for every \(0 < s < 1\) and \(1 < q < \infty\) such that \(sq< p\), this fact is also a consequence of Proposition~\ref{lemmaFugledeSobolevDetector}.
Hence, given any \(g\in \Smooth^{\infty}(\manfM;\manfN)\), the infimum of \(E_{1,p}\) is achieved on the set of maps \(u\in \Sobolev^{1,p}(\manfM;\manfN)\) that have the same \((\floor{p}-1)\)-homotopy type as \(g\). 

A related and classical question in this field involves instead the minimization of \(E_{1, p}\) on (standard) homotopy classes of \(\Smooth^{\infty}(\manfM;\manfN)\). 
This problem is highly nontrivial, even in the case \(p\geq m\) for which we know that minimizing maps on \(\Sobolev^{1,p}(\manfM;\manfN)\) are smooth. 
One difficulty is that the homotopy classes in \(\Smooth^{\infty}(\manfM;\manfN)\) are not sequentially weakly closed. 
For example, the identity map \(\Id  \colon  \Sphere^3 \to \Sphere^3\) is always a critical point of \(E_{1,2}\) but is homotopic to smooth maps in \(\Smooth^{\infty}(\Sphere^3;\Sphere^3)\) with arbitrarily small energy, so that
\[
\inf{\bigl\{E_{1,2}(u) : u\in \Smooth(\Sphere^3;\Sphere^3),\ u \textrm{ homotopic to } \Id \bigr\}} =0.
\]
Actually, White~\cite{White-1986} showed that, for every \(g\in \Smooth^{\infty}(\manfM;\manfN)\), the infimum of \(E_{1,p}\) in the homotopic class of \(g\) is \(0\) if and only if \(g\) has the same \(\floor{p}\)-homotopy type as a constant map.

\cleardoublepage
\chapter{Decoupling of local and global obstructions}
\label{chapterDecouple}

The property of \(\ell\)-extendability captures a fundamentally \emph{local} aspect of Sobolev maps, as illustrated by its characterization through generic restrictions and cohomological criteria. However, \((\ell, e)\)-extendability introduces a genuinely \emph{global} condition, which in the absence of topological assumptions on \(\manfV\) and \(\manfN\) requires additional considerations beyond local properties. This chapter addresses the relationship between these two notions by identifying a missing global condition that, when combined with \(\ell\)-extendability, becomes equivalent to \((\ell, e)\)-extendability.

We introduce the concept of \emph{coarse \((\ell, e)\)-extendability}, which bridges the gap between local and global properties. Generally speaking, a map \(u\) is coarsely \((\ell, e)\)-extendable whenever there exists a fixed triangulation that captures the global structure of the domain \(\manfV\) such that the restriction \(u|_{T^\ell}\) to its \(\ell\)-dimensional skeleton is homotopic to the restriction of some map in \(\Smooth^{0}(T^{e}; \manfN)\), defined on the \(e\)-dimensional skeleton. This notion formalizes a way to detect global obstructions to \((\ell, e)\)-extendability without requiring the full extension property across generic simplicial complexes.

The main result of this chapter establishes that a map \(u\) is \((\ell, e)\)-extendable if and only if it is both \(\ell\)-extendable and coarsely \((\ell, e)\)-extendable. This characterization rigorously decouples the local and global obstructions to \((\ell, e)\)-extendability and provides a practical framework for analyzing this property.

\section{From \texorpdfstring{$\ell$}{l}-extension to \texorpdfstring{$(\ell, \MakeLowercase{e})$}{(l, e)}-extension}

We rely on the following notion of triangulation of a compact smooth manifold \(\manfM\) without boundary.

\begin{definition}
\label{defnTriangulation}
    Let \(\cE^{m}\) be a simplicial complex and \(\Phi \colon  E^{m} \to \manfM\) be a biLipschitz homeomorphism.
	For each \(r \in \{0, \ldots, m\}\), 
	we say that 
	\(
	T^{r} \vcentcolon= \Phi(E^{r})
	\)
	equipped with the Hausdorff measure \(\cH^{r}\) is an \(r\)-dimensional \emph{skeleton} of \(\manfM\) and
	\[{}
	\cT \vcentcolon= \{T^0, \ldots, T^{m} \}
	\] 
	is a \emph{triangulation} of \(\manfM\).
\end{definition}

Such a biLipschitz homeomorphism \(\Phi\) always exists, see e.g.\@~\cite{Munkres}.
The \((\ell,m)\)-extendability of a map \(u\) gives an information about the extension of \(u\) on generic \(\ell\)-dimensional skeletons of \(\manfM\)\,:

\begin{proposition}
	\label{propositionDecoupleRestriction}
	If \(u \in \VMO^{\ell}(\manfM; \manfN)\) is \((\ell , e)\)-extendable with \(e \le m\), then there exists an \(\ell\)-detector \(w\) with the following property:
	For every triangulation \(\cT\) of \(\manfM\) with \(w|_{T^{\ell}}\) summable in \(T^{\ell}\), there exists \(F \in \Smooth^{0}(T^{e}; \manfN)\) such that
	\[{}
	u|_{T^{\ell}} \sim F|_{T^{\ell}}
	\quad \text{in \(\VMO(T^{\ell}; \manfN)\).}
	\]
\end{proposition}

\begin{proof}
	Given a triangulation \(\cT\) of \(\manfM\), let \(\cE^{m}\) be a simplicial complex and \(\Phi \colon  E^{m} \to \manfM\) be a biLipschitz homeomorphism associated to \(\cT\).{}
	Let \(w\) be an \(\ell\)-detector given by Definition~\ref{definitionExtensionVMO}.
	If \(w|_{T^{\ell}}\) is summable in \(T^{\ell}\), then by a change of variables we have that \(w \compose \Phi|_{E^{\ell}}\) is summable in \(E^{\ell}\).{}
	In other words, \(\Phi|_{E^{\ell}} \in \Fuglede_{w}(E^{\ell}; \manfM)\).{}
	By assumption on \(w\), there exists \(\widetilde F \in \Smooth^{0}(E^{e}; \manfN)\) such that
	\[{}
	u \compose \Phi|_{E^{\ell}} \sim \widetilde F|_{E^{\ell}}
	\quad \text{in \(\VMO(E^{\ell}; \manfN)\).}
	\]
	Since \(\Phi\) is a biLipschitz homeomorphism, 
	\[{}
	u|_{T^{\ell}} \sim \widetilde F \compose \Phi^{-1}|_{T^{\ell}}
	\quad \text{in \(\VMO(T^{\ell}; \manfN)\).}
	\]
	We have the conclusion with \(F \vcentcolon= \widetilde{F} \compose \Phi^{-1}|_{T^{e}}\).
\end{proof}

From Proposition~\ref{propositionDecoupleRestriction}, we see that \((\ell , e)\)-extendability of \(u\) implicitly involves testing \(u\) over generic \(\ell\)-dimensional skeletons.
We now consider a weaker notion of extendability where only one skeleton is required for each \(\ell\)-detector.

\begin{definition}
	Let \(\ell, e \in \N\) with \(\ell < e \le m\).{}
	We say that a map \(u \in \VMO^{\ell}(\manfM; \manfN)\) is \emph{coarsely \((\ell , e)\)-extendable} whenever, for each \(\ell\)-detector \(w\), there exist a triangulation \(\cT\) of \(\manfM\) with \(w|_{T^{\ell}}\) summable in \(T^{\ell}\) and \(F \in \Smooth^{0}(T^{e}; \manfN)\) such that
	\[{}
	u|_{T^{\ell}} \sim F|_{T^{\ell}}
	\quad \text{in \(\VMO(T^{\ell}; \manfN)\).}
	\]
\end{definition}

\begin{proposition}
	\label{propositionCoarseExtension}
	If \(u \in \VMO^{\ell}(\manfM; \manfN)\) is \((\ell, e)\)-extendable with \(e \le m\), then \(u\) is coarsely \((\ell, e)\)-extendable.
\end{proposition}

\resetconstant
\begin{proof}
	We fix a simplicial complex \(\cE^{m}\) and a biLipschitz homeomorphism \(\Psi \colon  E^{m} \to \manfM\). Let \(w\) be an \(\ell\)-detector and take a transversal perturbation of the identity \(\tau \colon  \manfM \times B_{\delta}^{q}(0) \to \manfM\). Using that \(\tau\) is smooth and that \(\tau_0\) agrees with the identity, we can assume, by reducing \(\delta\) if necessary, that \(\tau_\xi\) is a smooth diffeomorphism from \(\manfM\) onto itself.
	Let \(\overline{w}\geq w\) be an \(\ell\)-detector given by Definition~\ref{definitionExtensionVMO}.
	By Proposition~\ref{lemmaTransversalFamily}, we have
	\[{}
	\int_{B_{\delta}^{q}(0)} \biggl( \int_{E^{\ell}} \overline{w} \compose \tau_{\xi} \compose \Psi \dif\cH^{\ell} \biggr) \dif\xi{}
	\le \C\int_{\manfM} w.
	\] 
    Therefore, we can take \(\xi \in B_{\delta}^{q}(0)\) such that the inner integral on the left-hand side is finite.
	Hence, \(\tau_{\xi} \compose \Psi|_{E^{\ell}} \in \Fuglede_{\Bar{w}}(E^{\ell}; \manfM)\).{}
	By the choice of \(\overline{w}\), there exists \(F \in \Smooth^{0}(E^{e}; \manfN)\) such that 
	\begin{equation}
		\label{eqDecouple-93}
	u \compose \tau_{\xi} \compose \Psi|_{E^{\ell}}
	\sim F|_{E^{\ell}}
	\quad \text{in \(\VMO(E^{\ell}; \manfN)\).}
	\end{equation}
	We introduce \(\Phi = \tau_{\xi} \compose \Psi\) and \(\cT  = \bigl\{ T^0, \ldots, T^{m} \bigr\}\) with \(T^i\vcentcolon = \Phi(E^i)\) for each \(i \in \{0, \ldots, m\}\).{}
	Since \(\Phi\) is a biLipschitz homeomorphism and \(w \compose \Phi|_{E^{\ell}}\) is summable in \(E^{\ell}\), by a change of variables we have \(w|_{T^{\ell}}\) summable in \(T^{\ell}\).
	Moreover, from \eqref{eqDecouple-93},
	\[{}
	u|_{T^{\ell}}
	\sim F \compose \Phi^{-1}|_{T^{\ell}}
	\quad \text{in \(\VMO(T^{\ell}; \manfN)\).}
	\]
	Since \(F \compose \Phi^{-1}|_{T^{e}} \in \Smooth^{0}(T^{e}; \manfN)\), we deduce that \(u\) is coarsely \((\ell, e)\)-extendable.
\end{proof}

\begin{example}
	\label{exampleCharacterization-114}
	Let \(\cT= \bigl\{ T^0, \ldots, T^{m} \bigr\}\) be a triangulation of \(\manfM\) and let \(u \in \VMO^{\ell}(\manfM; \manfN)\) be such that \(u\) is continuous in a neighborhood of \(T^{\ell}\).
	We claim that if \(u|_{T^{\ell}}\) has a continuous extension to \(T^{e}\), then \(u\) is coarsely \((\ell, e)\)-extendable.
    Indeed, take an \(\ell\)-detector \(w\). 
    It need not be true that \(w|_{T^{\ell}}\) is summable in \(T^{\ell}\).
	
	To verify that \(u\) is coarsely \((\ell, e)\)-extendable, we take a transversal perturbation of the identity \(\tau \colon  \manfM \times B_{\delta}^{q}(0) \to \manfM\) and we can assume as in the proof of Proposition~\ref{propositionCoarseExtension} that \(\tau_\xi\) is a smooth diffeomorphism from \(\manfM\) onto itself for every \(\xi\in B_{\delta}^{q}(0)\).{}
	For \(0 < \epsilon \le \delta\) sufficiently small, the map \((x,\xi)\mapsto u \compose \tau_{\xi}(x)\) is continuous in \(T^\ell\times B_{\epsilon}^{q}(0)\).{}
	Let \(\cE^{m}\) be a simplicial complex and \(\Psi \colon  E^{m} \to \manfM\) be  a biLipschitz homeormorphism associated to \(\cT\).{}
	As in the proof of Proposition~\ref{propositionCoarseExtension}, one finds \(\xi \in B_{\epsilon}^{q}(0)\) such that \(w \compose \tau_{\xi} \compose \Psi|_{E^{\ell}}\) is summable in \(E^{\ell}\).{}
	Then, taking \(\Phi  = \tau_{\xi} \compose \Psi\) and the triangulation \(\widetilde{\cT}  = \bigl\{ \widetilde{T}^0, \ldots, \widetilde{T}^{m } \bigr\}\) with \(\widetilde{T}^i\vcentcolon=\Phi(E^i)\) for \(i \in \{0, \ldots, m\}\), we have by a change of variables that \(w|_{\widetilde T^{\ell}}\) is summable in \(\widetilde T^{\ell}\).{}
	Moreover, 
	\[{}
	u \compose \Phi|_{E^{\ell}}
	\sim u \compose \Psi|_{E^{\ell}}
	\quad \text{in \(\Smooth^{0}(E^{\ell}; \manfN)\).}
	\]
	From the assumption on \(u\), the map \(u \compose \Psi|_{E^{\ell}}\) can be continuously extended to \(E^{e}\).
    Hence, by the homotopy extension property, the same is true for \(u \compose \Phi|_{E^{\ell}}\).
	We deduce that the continuous map \(u|_{\widetilde T^{\ell}}\) has a continuous extension to \(\widetilde T^{e}\), which means that \(u\) is coarsely \((\ell, e)\)-extendable.
\end{example}

A \(\VMO^{\ell}\) map may be coarsely \((\ell, e)\)-extendable without being \((\ell, e)\)-extendable:

\begin{example}
	Let \(\manfM = \Sphere^{n + 1}\), \(\manfN = \Sphere^{n}\), \(\ell = n\) and \(e = n + 1\).
    There exists a map \(u \in \Sobolev^{1, n}(\Sphere^{n + 1}; \Sphere^{n})\) that is coarsely \((n, n+1)\)-extendable but not \((n, n+1)\)-extendable.
    To construct such a map, let \(a, b \in \Ball^{n + 1}\) with \(a \ne b\) and let \(v \in \Sobolev^{1, n}(\Ball^{n + 1}; \Sphere^{n})\) be such that 
	\begin{enumerate}[(a)]
		\item \(v\) is smooth except at \(a\) and \(b\),{}
		\item for any \(\epsilon > 0\) small, \(v|_{\partial B_{\epsilon}^{n + 1}(a)}\) has degree \(1\) and \(v|_{\partial B_{\epsilon}^{n + 1}(b)}\) has degree \(-1\),
		\item \(v = c\) in \(\Ball^{n + 1} \setminus \overline{B}{}_r^{n+1}\) for some constant \(c \in \Sphere^{n + 1}\) and \(0 < r < 1\).
	\end{enumerate} 
	
    Take a simplicial complex \(\cE^{n + 1}\) and a biLipschitz homeomorphism \(\Phi \colon  E^{n + 1} \to \Sphere^{n + 1}\). 
	Given \(\Sigma^{n + 1} \in \cE^{n + 1}\), we then consider an imbedding \(\Psi \colon  \Ball^{n + 1} \to \Sphere^{n + 1}\) whose image is contained in \(\Phi(\Sigma^{n + 1})\).{}
	Then, the map
	\[{}
	u \vcentcolon = 
	\begin{cases}
		v \compose \Psi^{-1}	& \text{in \(\Psi(\Ball^{n + 1})\),}\\
		c & \text{in \(\Sphere^{n + 1} \setminus \Psi(\Ball^{n + 1})\),}
	\end{cases}
	\]
	belongs to \(\Sobolev^{1, n}(\Sphere^{n + 1}; \Sphere^{n}) \subset \VMO^{n}(\Sphere^{n + 1}; \Sphere^{n})\) and is constant in a neighborhood of the \(n\)-dimensional skeleton \(\Phi(E^{n})\).{}
	Hence, by Example~\ref{exampleCharacterization-114}, \(u\) is coarsely \((n, n+1)\)-extendable.{}
	On the other hand, we have that 
	\[{}
	\jac{u} \ne 0
	\quad \text{in the sense of currents in \(\Sphere^{n + 1}\)}
	\] 
	and then, by Corollary~\ref{corollaryJacobianFugledeTrivial} with \( \varpi = \omega_{\Sphere^{n}}\) and by Proposition~\ref{lemma_l_l1_property}, \(u\) is not \((n,n+1)\)-extendability.
    We conclude that \(u\) is not \((n, n+1)\) extendable while being coarsely \((n, n+1)\)-extendable. 
\end{example}

In Section~\ref{sectionDecoupleProof}, we prove that coarse \((\ell, e)\)-extendability is the global ingredient that lacks to \(\ell\)-extendability:

\begin{theorem}
	\label{propositionDecouple}
	Let \(u \in \VMO^{\ell}(\manfM; \manfN)\) and \(e \in \{\ell + 1, \ldots, m\}\).{}
	Then, \(u\) is \((\ell, e)\)-extendable if and only if \(u\) is \(\ell\)-extendable and coarsely \((\ell, e)\)-extendable.
\end{theorem}

In order to prove Theorem~\ref{propositionDecouple}, we rely on a convenient intersection of the sets \(\VMO^{\#}(K^\ell | K^r)\) that are introduced in the previous chapter. 
This choice preserves the \(\VMO\)~property after composition with simplicial maps. 
We devote the next section to a detailed description of this fact.

\section{Composition with simplicial maps}\label{Section:composition-VMO}

Given a simplicial complex \(\cK^{\ell}\) and \(r\in \{0, \dots,\ell\}\), we have introduced the set \(\VMO^\#(K^\ell|K^r)\) in Section~\ref{section:VMOsharp}. 
We then define the set \(\VMO^\#(\cK^\ell)\) by
\begin{equation}
\label{eqDecoupleVMOSharp}
\VMO^\#(\cK^\ell){}
= \bigcap_{r = 0}^{\ell - 1}{\VMO^\#(K^\ell| K^{r})},
\end{equation}
equipped with the norm
\begin{equation}
\label{eqDecoupleVMOSharpNorm}
\norm{v}_{\VMO^{\#}(\cK^{\ell})} 
\vcentcolon= \sum_{r = 0}^{\ell - 1}{\norm{v}_{\VMO^\#(K^{\ell}| K^{r})}}.
\end{equation}

The goal of this section is to prove the following ingredient in the proof of Theorem~\ref{propositionDecouple}:

\begin{proposition}
\label{lemma-composition-simplicial-VMOsharp}
For every simplicial map \(\sigma \colon  L^r\to K^{\ell}\) between two finite simplicial complexes \(\cL^{r}\) and \(\cK^{\ell}\), the linear transformation
\[{}
v \in \VMO^\#(\cK^\ell) \longmapsto v\compose \sigma \in \VMO(L^r)
\]
is well defined and continuous.
\end{proposition}

We recall that a map \(\sigma \colon  L^{r} \to K^{\ell}\) is \emph{simplicial} whenever \(\sigma\) maps the vertices of each simplex \(\Lambda^{r} \in \cL^{r}\) into vertices of a simplex of \(\cK^{\ell}\) and \(\sigma|_{\Lambda^{r}}\) is affine.
As we do not require \(\sigma|_{\Lambda^{r}}\) to be injective, \(\sigma(\Lambda^{r})\) may be a simplex of dimension less than \(r\).
The proof of Proposition~\ref{lemma-composition-simplicial-VMOsharp} is based on the following estimate:
\begin{lemma}
	\label{lemmaVMOSimplicial}
	Let \(\sigma \colon  L^r\to K^{\ell}\) be a simplicial map. 
	Given a measurable function \(\varphi\colon K^\ell \to [0,+\infty]\) and  \(\alpha > \max{\{\abs{\sigma}_{\Lip}, 0\}}\), then for every \(x\in L^{r}\) and every \(\rho > 0\) we have 
\begin{equation*}
\int_{B_\rho^{r}(x)} \varphi\compose \sigma \dif\cH^{r}
\leq C\sum_{i=0}^{\ell}\rho^{r-i} \int_{B_{\alpha \rho}^{\ell}(\sigma(x))\cap K^i} \varphi \dif\cH^{i}.
\end{equation*}
\end{lemma}

\resetconstant
\begin{proof}[Proof of Lemma~\ref{lemmaVMOSimplicial}]
Let \(\Lambda^r \in \cL^r\). 
Since \(\sigma\) is simplicial, there exist \(i \in \{0, \dots, r\}\) and a simplex \(\Sigma^{i}\in \cK^i\) such that \(\sigma\) maps \(\Lambda^r\) onto \(\Sigma^{i}\). 
By the coarea formula, when \(i \in \{1, \dots, r\}\) we have
\begin{multline*}
\int_{B_\rho^{r}(x) \cap \Lambda^r} \varphi(\sigma(y)) \Jacobian{i}{\sigma}(y) \dif\cH^{r}(y)\\[-1em]
= \int_{\sigma(B_\rho^{r}(x)\cap \Lambda^r)} \varphi(\xi)\cH^{r-i}\bigl(\sigma^{-1}(\xi)\cap B_\rho^{r}(x)\cap \Lambda^r\bigr) \dif\cH^i(\xi),
\end{multline*}
where \(\Jacobian{i}{\sigma}\) is the \(i\)-dimensional Jacobian of \(\sigma\).
Since \(\sigma\) is affine from \(\Lambda^r\) onto \(\Sigma^{i}\), it follows that  \(\Jacobian{i}{\sigma}\) is a positive constant on \(\Lambda^{r}\) and then
\[
\int_{B_\rho^{r}(x) \cap \Lambda^r} \varphi \compose \sigma \dif\cH^{r}
\leq \Cl{eq-1243} \rho^{r-i} \int_{\sigma(B_\rho^{r}(x)\cap \Lambda^r)} \varphi \dif\cH^i.
\]
Since \(\alpha \ge \abs{\sigma}_{\Lip}\), we have \(\sigma(B_\rho^{r}(x) \cap \Lambda^r)\subset B_{\alpha\rho}^{\ell}(\sigma(x))\cap \Sigma^{i}\). 
Hence,
\[
\int_{B_\rho^{r}(x)\cap \Lambda^r} \varphi \compose \sigma \dif\cH^{r} 
\leq \Cr{eq-1243} \rho^{r-i} \int_{B_{\alpha \rho}^{\ell}(\sigma(x))\cap \Sigma^i} \varphi \dif\cH^i.
\] 
Since \(\alpha > 0\), this estimate also holds when \(i = 0\) by direct inspection.
Summing over the simplices \(\Lambda^r\), one gets the conclusion.
\end{proof}

\begin{proof}[Proof of Proposition~\ref{lemma-composition-simplicial-VMOsharp}]
By Lemma~\ref{lemmaVMOSimplicial} applied to \(\varphi \vcentcolon= |v|\) and \(\rho \vcentcolon= \Diam{L^{r}}\), one obtains 
\[{}
\norm{v \compose \sigma}_{\Lebesgue^{1}(L^{r})} 
\leq \C \sum_{i=0}^\ell{\norm{v}_{\Lebesgue^{1}(K^{i})}}.
\]
By Proposition~\ref{propositionVMOSharpRestrictionSkeleton}, it then follows that
\begin{equation}
	\label{eqVMOSuperL1}
\norm{v \compose \sigma}_{\Lebesgue^{1}(L^{r})} 
\leq \C \sum_{i=0}^\ell{\norm{v}_{\VMO^{\#}(K^{\ell}| K^{i})}}.
\end{equation}
To estimate the \(\VMO\) seminorm of \(v \compose \sigma\), fix \(x\in L^r\) and \(0 < \rho \le \Diam{L^{r}}\).  
For every \(z\in B_\rho^{r}(x) \cap L^{r}\), we apply Lemma~\ref{lemmaVMOSimplicial}
to the function \(\varphi \vcentcolon=|v  - v \compose \sigma(z) |\). 
This gives
\begin{multline*}
\int_{B_\rho^{r}(x)} |v\compose \sigma(y)-v\compose \sigma(z)|\dif\cH^{r}(y) \\
\leq \Cl{eq-1332} \sum_{i=0}^{\ell}\rho^{r-i} \int_{B_{\alpha \rho}^{\ell}(\sigma(x))\cap K^i} |v(\xi) - v\compose \sigma(z)|\dif\cH^{i}(\xi),
\end{multline*}
where \(\alpha > \max{\{\abs{\sigma}_{\Lip}, 0\}}\).
Then, by Tonelli's theorem,
\begin{multline*}
\int_{B_\rho^{r}(x)}\int_{B_\rho^{r}(x)} |v\compose \sigma(y)-v\compose \sigma(z)|\dif\cH^{r}(y) \dif\cH^{r}(z)
\\ \leq \Cr{eq-1332}\sum_{i=0}^{\ell}\rho^{r-i} \int_{B_{ \alpha \rho}^{\ell}(\sigma(x))\cap K^i} \int_{B_\rho^{r}(x)} |v(\xi) - v\compose \sigma(z)|\dif \cH^{r}(z) \dif \cH^{i}(\xi).
\end{multline*}
We apply again Lemma~\ref{lemmaVMOSimplicial} to the function \(\varphi \vcentcolon=|v(\xi) - v\compose \sigma|\) in order to estimate the inner integral in the right-hand side. 
One gets
\begin{multline*}
	\int_{B_\rho^{r}(x)}\int_{B_\rho^{r}(x)} |v\compose \sigma(y)-v\compose \sigma(z)|\dif \cH^{r}(y) \dif\cH^{r}(z) 
\\ \leq \Cl{cte-779}  
\sum_{i_{1}=0}^{\ell} \sum_{i_{2}=0}^{\ell} \rho^{2r-i_{1}-i_{2}} \!
\int_{B_{\alpha \rho}^{\ell}(\sigma(x))\cap K^{i_{1}}} \! \int_{B_{\alpha \rho}^{\ell}(\sigma(x))\cap K^{i_{2}}} \! |v(\xi)-v(\eta)| \dif\cH^{i_{2}}(\eta) \dif\cH^{i_{1}}(\xi).
\end{multline*}
We recall that \(\cH^{r}(B^{r}_{\rho}(x)) \ge c\rho^{r}\).
Since \(x\) is arbitrary, by Lemma~\ref{lemmaVMOSharpMixed} for every \(0 < \rho \le \Diam{L^{r}}\) we get
\begin{equation}
\label{eqVMOSuperBMO}
\seminorm{v\circ \sigma}_{\rho} \leq \Cr{cte-779} \sum_{j = 0}^{\ell}{\seminorm{v}_{K^{\ell}, K^{j}, \alpha\rho}}.
\end{equation}
We thus have \(v \compose \sigma \in \VMO(L^{r})\).
Observe that, for \(\rho > \Diam{L^{r}}\),{}
\[
[v \circ \sigma]_{\rho} 
\leq \C \norm{v \circ \sigma}_{\Lebesgue^{1}(L^{r})}.
\]
The continuity of the map \(v \mapsto v \circ \sigma\) thus follows from \eqref{eqVMOSuperL1} and \eqref{eqVMOSuperBMO}.
\end{proof}

Given a compact Riemannian manifold \(\manfN\) isometrically imbedded in the Euclidean space \(\R^\nu\), we analogously define the set \(\VMO^{\#}(\cK^{\ell}; \manfN)\) of maps \(v \colon K^\ell \to \R^\nu\) whose components belong to \(\VMO^{\#}(\cK^\ell)\) and \(v(x)\in \manfN\) for every \(x\in K^\ell\).

\begin{corollary}
	\label{corollaryHomotopySimplicialMap}
	If \(v \in \VMO^{\#}(\cK^{\ell}; \manfN)\) and \(\sigma \colon  L^{r} \to K^{\ell}\) is a simplicial map, then, for every \(f \in \Smooth^{0}(K^{\ell}; \manfN)\) such that
	\begin{equation}
    \label{eqDecouple-311}
    v \sim f 
	\quad \text{in \(\VMO(K^{\ell}; \manfN)\),}
	\end{equation}
	we have
	\[{}
	v \compose \sigma \sim f \compose \sigma 
	\quad \text{in \(\VMO(L^{r}; \manfN)\).}
	\]
\end{corollary}

\begin{proof}
	By Proposition~\ref{propositionVMOSharpDensityManifold} applied to \(\VMO^{\#}(K^{\ell}| K^{i}; \manfN)\) for each \(i \in \{0, \dots, \ell - 1\}\), there exists a sequence \((f_{j})_{j \in \N}\) in \(\Smooth^{0}(K^{\ell}; \manfN)\) such that
	\[{}
	f_{j} \to v
	\quad \text{in \(\VMO^{\#}(\cK^{\ell}; \manfN)\).}
	\]
	Observe that \((f_{j})_{j \in \N}\) can be chosen independently of \(i\) as we can take \(f_{j} = \Pi \circ v_{s_{j}}\) for a sequence of positive numbers \((s_{j})_{j \in \N}\) that converges to \(0\).
	By Proposition~\ref{propositionHomotopyVMOLimit}, there exists \(J_{1} \in \N\) such that, for every \(j \ge J_{1}\)\,,
	\[{}
	f_{j} \sim v 
	\quad \text{in \(\VMO(K^{\ell}; \manfN)\).}
	\]	
	Thus, by \eqref{eqDecouple-311} and transitivity of the homotopy relation,
	\[{}
	f_{j} \sim f 
	\quad \text{in \(\VMO(K^{\ell}; \manfN)\).}
	\]	
	Since both maps are continuous, by Proposition~\ref{propositionHomotopyVMOtoC} they are also homotopic in \(\Smooth^{0}(K^{\ell}; \manfN)\).{}
	Hence, by composition with \(\sigma\),{}
	\begin{equation}
	\label{eqVMOSets-859}
	f_{j} \compose \sigma \sim f \compose \sigma 
	\quad \text{in \(\Smooth^{0}(L^{r}; \manfN)\)}
	\end{equation}
	and then also in \(\VMO(L^{r}; \manfN)\).
	
	On the other hand, since \(\sigma\) is a simplicial map, it follows from Proposition~\ref{lemma-composition-simplicial-VMOsharp} that
	\[{}
	f_{j} \compose \sigma \to v \compose \sigma{}
	\quad \text{in \(\VMO(L^{r}; \manfN)\).}
	\]
	A further application of Proposition~\ref{propositionHomotopyVMOLimit} gives \(J_{2} \in \N\) such that, for every \(j \ge J_{2}\),
	\begin{equation}
	\label{eqVMOSets-871}
	f_{j} \compose \sigma \sim v \compose \sigma{}
	\quad \text{in \(\VMO(L^{r}; \manfN)\).}
	\end{equation}
	The conclusion follows by combining \eqref{eqVMOSets-859} and \eqref{eqVMOSets-871} for any \(j \ge \max{\{J_{1}, J_{2}\}}\).
\end{proof}

\section{\texorpdfstring{$(\ell, \MakeLowercase{e})$}{(l, e)}-extension based on homotopy equivalences}
\label{sectionDecoupleProof}

The next result provides a decoupling of \((\ell, e)\)-extendability based on the \(\ell\)-extendability, which is a \emph{local} property, and a \emph{global} assumption provided by simplicial complexes that are homotopically equivalent to \(\manfM\). Proposition~\ref{theoremGlobalModuloLocal} plays a key role in the proof of Theorem~\ref{propositionDecouple}.

\begin{proposition}
\label{theoremGlobalModuloLocal}
Let \(u \in \VMO^{\ell}(\manfM ; \manfN)\) and \(e \in \{\ell + 1, \ldots, m\}\).
If \(u\) is \(\ell\)-extendable and if, for every \(\ell\)-detector \(w\), there exist a simplicial complex \(\cE^{m}\), a Lipschitz homotopy equivalence \(\varphi \colon  E^{m} \to \manfM\)  such that
\(
\varphi\vert_{E^{\ell}} \in \Fuglede_{w}(E^{\ell}; \manfM)
\)
and  a map \(F \in \Smooth^{0}(E^{e}; \manfN)\) with
\[{}
u \compose \varphi\vert_{E^{\ell}} \sim F\vert_{E^{\ell}}
\quad \text{in \(\VMO(E^{\ell}; \manfN)\)},
\]
then \(u\) is \((\ell, e)\)-extendable.
\end{proposition}

We recall that \(\varphi \in \Smooth^{0}(E^{m}; \manfM)\) is a homotopy equivalence whenever there exists \(\psi \in \Smooth^{0}(\manfM; E^{m})\) such that \(\varphi \compose \psi\) and \(\psi \compose \varphi\) are homotopic to the identity maps \(\Id\) on \(\manfM\) and \(E^m\), respectively.

\resetconstant
\begin{proof}[Proof of Proposition~\ref{theoremGlobalModuloLocal}]
Let \(w\) be an \(\ell\)-detector given by Theorem~\ref{proposition_Reference_Homotopy}.
Take a  Lipschitz homotopy equivalence \(\varphi \colon  E^{m} \to \manfM\) as in the statement and a Lipschitz map \(\psi \colon  \manfM \to E^{m}\) such that
\begin{equation}
\label{eqExtension-649}
\varphi \compose \psi \sim \Id{}
\quad \text{in \(\Smooth^{0}(\manfM; \manfM)\).}
\end{equation}
Given a simplicial complex \(\cK^{e}\) and a Lipschitz map \(\gamma \colon  K^{e} \to \manfM\) with \(\gamma\vert_{K^{\ell}} \in \Fuglede_{w}(K^{\ell}; \manfM)\),{}
by the simplicial approximation theorem \cite{Hatcher_2002}*{Theorem~2C.1}, there exists \(\sigma \colon  K^e \to E^m\) that satisfies
\begin{equation}
	\label{eqExtension-822}
\sigma \sim \psi \compose \gamma{}
\quad \text{in \(\Smooth^{0}(K^{e}; E^{m})\),}
\end{equation}
which is a simplicial map with respect to a subdivision of \(\cK^{e}\).{}
In particular, \(\sigma(K^{i}) \subset E^{i}\) for every \(i \in \{0, \ldots, e\}\).{}
From \eqref{eqExtension-649} and \eqref{eqExtension-822}, we also have
\begin{equation}
	\label{eqDecouple-252}
\varphi \compose \sigma 
\sim \varphi \compose \psi \compose \gamma 
\sim \gamma{}
\quad \text{in \(\Smooth^{0}(K^{e}; \manfM)\).}
\end{equation}
We observe that \(\varphi \compose \sigma\vert_{K^{\ell}}\) need not belong to \(\Fuglede_{w}(K^{\ell}; \manfM)\).{}
For this reason, we need to replace \(\varphi\)\,:

\begin{Claim}
	There exists a Lipschitz homotopy equivalence \(\widetilde{\varphi} \colon  E^{m} \to \manfM\) such that \(\widetilde{\varphi} \sim \varphi\) in \(\Smooth^{0}(E^{m}; \manfM)\), with
	\begin{equation}
	\label{eq-971-Decouple}
	\widetilde{\varphi}\vert_{E^{i}} \in \Fuglede_{w}(E^{i}; \manfM)
	\quad\text{for every \(i \in \{0, \dots, m\}\)}
	\end{equation}
	and
	\[
	\widetilde{\varphi} \compose \sigma\vert_{K^{\ell}} \in \Fuglede_{w}(K^{\ell}; \manfM).
	\]
\end{Claim}

\begin{proof}[Proof of the Claim]
	We take  a transversal perturbation of the identity \(\tau \colon  \manfM\times B_{\delta}^{q}(0) \to \manfM\).
	By Proposition~\ref{lemmaTransversalFamily}, we have
	\[{}
	\int_{{B}^q_\delta(0)} \biggl( \sum_{i = 0}^{m} \int_{E^{i}} w \compose \tau_\xi \compose \varphi \dif\cH^{i} \biggr) \dif \xi
	\le \C \int_{\manfM} w
	\]
	and
	\[{}
	\int_{{B}^q_\delta(0)} \biggl( \int_{K^{\ell}} w \compose \tau_\xi \compose \varphi \compose \sigma \dif\cH^{\ell} \biggr) \dif \xi
	\le \C \int_{\manfM} w.
	\]
	Thus, there exists \(\xi \in B_{\delta}^{q}(0)\) such that
	\[
	\tau_{\xi} \compose \varphi\vert_{E^{i}} \in \Fuglede_{w}(E^{i}; \manfM)
	\quad \text{for every \(i \in \{0, \dots, m\}\)}
	\]
	and
	\[{}
	\tau_{\xi} \compose \varphi \compose \sigma\vert_{K^{\ell}} \in \Fuglede_{w}(K^{\ell}; \manfM).
	\]
	It then suffices to take \(\widetilde\varphi = \tau_{\xi} \compose \varphi\).
\end{proof}

Since  \(u \in \VMO^{\ell}(\manfM ; \manfN)\) is \(\ell\)-extendable and since the restrictions \(\varphi|_{E^{\ell}}\) and \(\widetilde{\varphi}|_{E^{\ell}}\) belong to \(\Fuglede_{w}(E^{\ell}; \manfM)\) and 
\[{}
	 \varphi|_{E^{\ell}} \sim \widetilde{\varphi}|_{E^{\ell}}
	\quad \text{in \(\Smooth^{0}(E^{\ell}; \manfM)\),}
\]
by Theorem~\ref{proposition_Reference_Homotopy} we have
\[{}
u \compose \varphi\vert_{E^{\ell}}
\sim{} 
u \compose \widetilde{\varphi}\vert_{E^{\ell}}{}
\quad \text{in \(\VMO(E^{\ell}; \manfN)\).}
\]
Hence, by assumption on \(\varphi\) and transitivity of the homotopy relation, there exists a map \(F \in \Smooth^{0}(E^{e}; \manfN)\) such that
	\begin{equation}
	\label{eqGlobal2}
u \compose \widetilde{\varphi}\vert_{E^{\ell}}
\sim{} 
F\vert_{E^{\ell}}{}
\quad \text{in \(\VMO(E^{\ell}; \manfN)\).}
	\end{equation}
By the choice of \(\widetilde{\varphi}\), we have \(\widetilde{\varphi} \compose \sigma\vert_{K^{\ell}} \in \Fuglede_{w}(K^{\ell}; \manfM)\).{}
Since \(\widetilde\varphi\) is homotopic to \(\varphi\), by \eqref{eqDecouple-252} we also have
\[{}
\widetilde\varphi \compose \sigma 
\sim \gamma{}
\quad \text{in \(\Smooth^{0}(K^{e}; \manfM)\).}
\]
A further application of Theorem~\ref{proposition_Reference_Homotopy} then gives
\begin{equation}
	\label{eqExtension-856}
u \compose \widetilde{\varphi} \compose \sigma\vert_{K^{\ell}}
\sim{} 
u \compose \gamma\vert_{K^{\ell}}
\quad \text{in \(\VMO(K^{\ell}; \manfN)\).}
\end{equation}

We now wish to combine \eqref{eqGlobal2} and \eqref{eqExtension-856}.
For this purpose, we apply \eqref{eq-971-Decouple} and Lemma~\ref{lemma-generic-vmo-sharp-Small} to deduce that \(u \compose \widetilde{\varphi} \in \VMO^{\#}(\cE^{\ell}; \manfN)\).
Since \(\sigma\) is simplicial and \(\sigma(K^{\ell}) \subset E^{\ell}\), it follows from Corollary~\ref{corollaryHomotopySimplicialMap} and \eqref{eqGlobal2} that
\begin{equation}
	\label{eqExtension-867}
	u \compose \widetilde{\varphi} \compose \sigma\vert_{K^{\ell}} \sim F \compose \sigma\vert_{K^{\ell}}
	\quad \text{in \(\VMO(K^{\ell}; \manfN)\).}
\end{equation}
Combining \eqref{eqExtension-856} and \eqref{eqExtension-867} we thus have
\[{}
u \compose \gamma\vert_{K^{\ell}}
\sim F \compose \sigma|_{K^{\ell}}
\quad \text{in \(\VMO(K^{\ell}; \manfN)\).}
\] 
Since \(F \compose \sigma \in \Smooth^{0}(K^{e}; \manfN)\), we deduce that \(u\) is \((\ell, e)\)-extendable.
\end{proof}

In the particular case where \(\varphi\) is a biLipschitz homeomorphism associated with a triangulation of \(\manfM\),  Proposition~\ref{theoremGlobalModuloLocal} implies the converse implication in Theorem~\ref{propositionDecouple}:
 
\resetconstant
\begin{proof}[Proof of Theorem~\ref{propositionDecouple}]
	If \(u\) is \((\ell, e)\)-extendable, then \(u\) is \(\ell\)-extendable by Propositions~\ref{lemma_l_l1_property} and~\ref{proposition_extension_montone_small}.
	By Proposition~\ref{propositionCoarseExtension}, \(u\) is also coarsely \((\ell, e)\)-extendable.

	We now prove the converse.
	To this end, assume that \(u\) is \(\ell\)-extendable and coarsely \((\ell, e)\)-extendable.{}
	Given an \(\ell\)-detector \(w\), 
	let \(\cT\) be a triangulation of \(\manfM\) such that \(w|_{T^{\ell}}\) is summable in \(T^\ell\) and
	\[{}
	u|_{T^{\ell}} \sim \widetilde{F}|_{T^{\ell}}
	\quad \text{in \(\VMO(T^{\ell}; \manfN)\)}
	\]
	for some \(\widetilde{F} \in \Smooth^{0}(T^{e}; \manfN)\).{}
	If \(\Phi \colon  E^{m} \to \manfM\) is a biLipschitz homeomorphism associated to \(\cT\), then
	\[{}
	u \compose \Phi|_{E^{\ell}} \sim \widetilde{F} \compose \Phi|_{E^{\ell}}
	\quad \text{in \(\VMO(E^{\ell}; \manfN)\)}.
	\]
	The assumptions of Proposition~\ref{theoremGlobalModuloLocal} are then satisfied by taking \(\varphi = \Phi\) and \(F= \widetilde{F}\circ \Phi|_{E^e}\). 
    We conclude that \(u\) is \((\ell,e)\)-extendable.
\end{proof}

\cleardoublepage
\chapter{Strong approximation of Sobolev maps}
\label{chapter-approximation-Sobolev-manifolds}

Building on the ideas and tools developed in this monograph, notably in Chapters~\ref{chapterFuglede}, \ref{chapter-3-Detectors}, \ref{section_VMO}, \ref{chapterGenericEllExtension}, and \ref{chapterExtensionGeneral}, we now establish the main theorems introduced in Chapter~\ref{chapterIntroduction} concerning the approximation of \((\floor{kp}, e)\)-extendable maps in \(\Sobolev^{k, p}(\manfM; \manfN)\). Specifically, we prove Theorem~\ref{theoremhkpGreat}, which states that maps satisfying this property with \(e = m\) can be approximated by smooth maps, and Theorem~\ref{thm_density_manifold_open_introduction}, which asserts that when the extension property is valid only for some \(e < m\), Sobolev maps can still be approximated by maps that are smooth except on a structured singular set of rank \(m - e - 1\).
We also establish their counterparts for the maps defined in a bounded Lipschitz open subset \(\Omega\) of \(\R^m\).

We revisit the opening construction presented in Chapter~\ref{chapterFuglede}, where maps were opened around points. Here, the concept is generalized to larger sets, such as segments, planes, and higher-dimensional structures, using cubications of radius \(\eta\) in \(\R^m\). The use of Fuglede maps provides a clean presentation: Compositions are genuine compositions, unlike in our previous work \cite{BPVS_MO}, where \(u \compose \Phi\) was used as a shorthand for the Sobolev limit of smooth maps \(u_j \compose \Phi\). While the presentation here is more streamlined, the properties and results themselves are essentially unchanged, adapted to the new setting.

An important advantage of the new framework is the simplicity with which the extendability property is preserved through opening. Proposition~\ref{corollary_ouverture} illustrates that the process is straightforward to implement, and the resulting map \(u_\eta^\mathrm{op}\), obtained by opening \(u\) near the \(\floor{kp}\)-dimensional skeleton of the bad cubes in a cubication, retains the same extendability as \(u\).

We also revisit and adapt several constructions from \cite{BPVS_MO}, including thickening, adaptive smoothing, and shrinking, each serving a distinct purpose in transforming Sobolev maps. Thickening generalizes zero-degree homogenization, extending Sobolev maps while preserving their target manifold \(\manfN\) but introducing structured singular sets of dimension \(m - \floor{kp} - 1\). Adaptive smoothing uses convolution with variable parameters to approximate maps obtained through opening and thickening, ensuring respect for the structured singular sets. Shrinking refines thickening by confining modifications to small neighborhoods around the dual skeleton, providing precise control over the \(\Sobolev^{k, p}\) norm. 

Theorems~\ref{theoremhkpGreat} and~\ref{thm_density_manifold_open_introduction} are first established for maps defined on bounded Lipschitz open subsets \(\manfV = \Omega\) of \(\R^m\). As \(u_\eta^\mathrm{op}\) undergoes thickening and smoothing, we obtain a map \(u_\eta^\mathrm{sm}\) that is smooth except for a structured singular set of rank \(m - \floor{kp} - 1\). A new key step, absent from \cite{BPVS_MO}, is the construction of the map \(u_\eta^\mathrm{fop}\), which is \((\floor{kp}, e)\)-homotopic to \(\Pi \compose u_\eta^\mathrm{sm}\). 
The existence of a \((\floor{kp}, e)\)-homotopy between these maps is essential for removing or reducing the size of singularities through a homotopy extension property that confines discontinuities to a structured set of rank \(m - e - 1\).

However, topological considerations are qualitative by nature, whereas the approximation of Sobolev maps requires precise quantitative control. Shrinking provides the necessary bridge by confining modifications necessary for \(\Pi \compose u_\eta^\mathrm{sm}\) to small neighborhoods of the dual skeleton of dimension \(m - \floor{kp} - 1\) where the singularities are located. The map \(u_\eta^\mathrm{sh}\) then obtained converges to \(u\) in \(\Sobolev^{k, p}\) as \(\eta\) tends to zero.

The final part of this chapter extends the analysis to maps defined on compact manifolds without boundary \(\manfV = \manfM\), a case not considered in \cite{BPVS_MO}. The maps in \(\manfM\) are extended to a neighborhood in \(\R^\varkappa\), preserving their \((\floor{kp}, e)\)-extension property. This allows us to conclude the proofs of Theorems~\ref{theoremhkpGreat} and~\ref{thm_density_manifold_open_introduction}, showing that maps satisfying the \((\floor{kp}, e)\)-extension property can be approximated by sequences in \(\ClassR_{m - e - 1}(\manfM; \manfN)\).

This chapter thus provides a comprehensive framework for constructing strong approximations of Sobolev maps.
By tying together the local and global properties explored throughout this monograph, and by emphasizing both the qualitative topological properties and the quantitative analytical precision, it sets the stage for further developments.

\section{\texorpdfstring{\(\ClassR_{\MakeLowercase{i}}\)}{Ri}~maps}

In this chapter, we approximate Sobolev maps in \(\manfV\) with values in \(\manfN\) by more regular maps which are smooth outside a structured singular set in the sense of Definition~\ref{defnStructuredSingularSet}. 
We assume henceforth that \(\manfV\) is either a bounded Lipschitz open subset \(\Omega\) of \(\R^m\) or a compact manifold without boundary \(\manfM\) imbedded in \(\R^\varkappa\).

\begin{definition}
\label{definitionExtension}
Let \(i\in \{0, \dots, m-1\}\).{}
We say that \(u  \colon  \manfV \setminus T^i \to \manfN\) is an \( \ClassR_{i} \)~map in \(\manfV\) with singular set \(T^{i}\) whenever \( u \) is smooth in \(\manfV \setminus T^{i}\) and \(T^{i}\) is a structured singular set of rank \(i\) such that, for every \(j\in \N_*\) and every \(x\in \manfV \setminus T^{i}\),{}
\begin{equation}
\label{eqRClassDerivativeEstimatePointwise}
\abs{D^ju(x)}\leq \frac{C}{(d(x, T^{i}))^j}
\end{equation}
for some constant \(C\geq 0\) depending on \(u\) and \(j\).
We denote the class of all \(\ClassR_i\)~maps in \(\manfV\) with singular sets of rank \(i\) by \(\ClassR_i(\manfV; \manfN)\).
\end{definition}

\begin{example}
The map \( u \colon \Ball^{m} \setminus \{0\} \to \Sphere^{m - 1} \) defined by \( u(x) = x/\abs{x} \) is an \( \ClassR_{0} \)~map in \(\Ball^{m}\) with singular set \(T^0 = \{0\}\).
More generally, given \( i \in \{0, \ldots, m - 2\} \), the map \( u \colon (\Ball^{m - i} \times \Ball^{i}) \setminus (\{0'\} \times \Ball^{i}) \to \Sphere^{m - i - 1} \) defined for \( x = (x', x'')\) with \( x' \ne 0 \) by \( u(x) = x'/\abs{x'} \) is an \( \ClassR_{i} \)~map in \(\Ball^{m - i} \times \Ball^{i}\) with singular set \( T^{i} = \{0'\} \times \Ball^{i} \).
\end{example}

Denote the dual dimension associated to \( e \in \{0, \ldots, m\} \) with respect to \( m \) by 
\[
e^* = m - e - 1.
\]
We have a natural inclusion of \(\ClassR_{e^*}\)~maps in Sobolev spaces.
Although an \(\ClassR_{e^*}\)~map in \(\manfV\) is not defined in its singular set, we implicitly take an arbitrary extension to have a measurable function defined pointwise in \(\manfV\).

\begin{proposition}
\label{propositionClassRInclusionSobolev} 
If \( e \in \{0, \ldots, m\} \) and \(kp < e + 1\), then
\[
\ClassR_{e^*}(\manfV; \manfN) \subset \Sobolev^{k,p}(\manfV; \manfN).
\]
\end{proposition}

\begin{proof}
    By assumption, we have
    \begin{equation}
    \label{eqStrong-45}
    e^{*} = m - e - 1 < m - kp.
    \end{equation}
    Let \(u\) be an \( \ClassR_{e^{*}} \)~map in \(\manfV\) with a structured singular set \( T^{e^{*}} \) of rank \(e^*\).
    If \( \manfV = \manfM\) is a compact manifold, then \(T^{e^{*}}\) is contained in a finite union of \(e^{*}\)-dimensional compact submanifolds \((S_\alpha)_{\alpha \in J}\) of \(\manfM\) without boundary that intersect cleanly.
    If \( \manfV = \Omega\) is a bounded open subset of \(\R^m\), then \(T^{e^{*}}\) is contained in a finite union of sets \((S_\alpha)_{\alpha\in J}\) that are \(e^{*}\)-dimensional affine spaces of \(\R^{m}\) orthogonal to \(m - e^{*} = e + 1\) coordinate axes.
    In both cases, for every \( x \in \manfV \) one has
    \[
    \frac{1}{d(x, T^{e^{*}})^{k}}
    \le \sum_{\alpha}{\frac{1}{d(x, S_{\alpha})^{k}}}.
    \]
    Since \( S_{\alpha} \) has dimension \( e^{*} \) and \eqref{eqStrong-45} holds, we have \( 1/d(\cdot, S_{\alpha})^{k}  \in \Lebesgue^p(\manfV) \).
    Here, when \(\manfV = \Omega\) is an open set we rely on the boundedness of \(\Omega\), and in the case where \(\manfV = \manfM\) is a compact manifold, we use the fact that each \(S_{\alpha}\) is compact.
Hence, for every \( j \in \{ 1, \dots, k\} \),
    \[
    \frac{1}{d(\cdot, T^{e^{*}})^{j}} \in \Lebesgue^p(\manfV).
    \]
    We deduce from \eqref{eqRClassDerivativeEstimatePointwise} that \( \abs{D^{j}u} \in \Lebesgue^p(\manfV) \).
    Since \( T^{e^{*}} \) has dimension smaller than \( m - 1 \), the conclusion then follows from a standard property of Sobolev spaces for functions that are smooth except on a closed set of \( (m-1) \)-dimensional Hausdorff measure zero, see \cite{Brezis-Mironescu}*{Lemma~1.10}.
\end{proof}

In the statement of Proposition~\ref{propositionClassRInclusionSobolev}, to accommodate the case \( e = m \), which means that \( e^{*} = -1 \),  we denote
\[
\ClassR_{-1}(\manfV; \manfN)
= \Smooth^{\infty}(\overline{\manfV}; \manfN)
\]
and that \( T^{-1} \vcentcolon= \emptyset \) for all maps in this family.

Our goal in this chapter is to establish the following result that contains Theorems~\ref{theoremhkpGreatBall}, \ref{theoremhkpGreat}, \ref{theoremIntroductionClassR} and~\ref{thm_density_manifold_open_introduction} in the Introduction:

\begin{theorem}\label{thm_density_manifold_open}
Let \(kp < m\) and \(e \in \{\floor{kp}, \dots, m\}\). 
Then, \(u \in \Sobolev^{k, p}(\manfV; \manfN)\) is \((\floor{kp}, e)\)-extendable if and only if there exists a sequence \((u_j)_{j\in \N}\) in \(\ClassR_{e^{*}}(\manfV; \manfN)\) such that
\[{}
u_{j} \to u 
\quad \text{in \(\Sobolev^{k, p}(\manfV; \manfN)\).}
\]
\end{theorem} 

We recall that every map \( u \in \Sobolev^{k, p}(\manfV; \manfN)\) is \((\floor{kp}, \floor{kp})\)-extendable, and so by Theorem~\ref{thm_density_manifold_open} one can always approximate \( u \) by a sequence of \(\ClassR_{m - \floor{kp} - 1}\)~maps, which is Theorem~\ref{theoremIntroductionClassR} with \(\manfV = \manfM\).
On the other hand, when \( u \) is \((\floor{kp}, m)\)-extendable, then in the previous statement with \(e = m\) one gets a sequence \((u_j)_{j\in \N}\) in \(\Smooth^\infty(\overline{\manfV}; \manfN)\) as stated in Theorem~\ref{theoremhkpGreat} also with \(\manfV = \manfM\).

When \(\manfV = Q_1^m\) is a cube in \(\R^m\), Detaille~\cite{DetailleClassR} recently showed that all maps in \(\ClassR_{e^*}(Q_1^m; \manfN)\) can be approximated in the \(\Sobolev^{k, p}\)~distance by smooth maps \(u \colon Q_1^m \setminus T^{e^*} \to \manfN\) that satisfy the derivative estimate \eqref{eqRClassDerivativeEstimatePointwise}, where \(T^{e^*}\) is a closed imbedded \(e^*\)-dimensional submanifold in \(\R^m\). 
Detaille's method builds on the shrinking procedure discussed in Section~\ref{section_shrinking} below and incorporates an additional technique to further uncross singularities.
Combined with Theorem~\ref{thm_density_manifold_open}, one thus gets the approximation of  \((\floor{kp}, e)\)-extendable maps in \(\Sobolev^{k, p}(Q_1^m; \manfN)\) by maps that are smooth except on imbedded \(e^*\)-dimensional submanifolds in \(\R^m\).

\medskip

We begin with the proof of the reverse implication of Theorem~\ref{thm_density_manifold_open} that is based on the relation between \( \VMO^{\ell} \)~convergence and the extension property.

\begin{proof}[Proof of Theorem~\ref{thm_density_manifold_open}. ``\(\Longleftarrow\)'']
	Since \(\manfN\) is bounded in \(\R^{\nu}\), by the Gagliardo-Nirenberg interpolation inequality the sequence \( (u_j)_{j \in \N} \) also converges in \(\Sobolev^{1, kp}(\manfV; \manfN)\). 
    By Proposition~\ref{lemmaFugledeSobolevDetector} with \(\ell = \floor{kp}\), we may take a subsequence \( (u_{j_{i}})_{i \in \N}\) such that
    \begin{equation}
        \label{eqJEMS-117}
    u_{j_i} \to u 
    \quad \text{in \(\VMO^{\floor{kp}}(\manfV; \R^\nu)\).}
    \end{equation}
    By Proposition~\ref{proposition_Extension_Property_R_0-New}, each map \(u_{j_i}\) belongs to \(\VMO^{e}(\manfV; \manfN)\) and we deduce from Corollary~\ref{propositionVMOlExtension} that \(u_{j_i}\) is \((\floor{kp}, e)\)-extendable.
    By \eqref{eqJEMS-117} and Proposition~\ref{propositionExtensionVMOClosed}, their limit \(u\) is also \((\floor{kp}, e)\)-extendable.
\end{proof}

To prove the direct implication of Theorem~\ref{thm_density_manifold_open},
we begin with the case where \( \manfV \) is an open subset \(\Omega\) of \(\R^m\), which we then use to deduce the case where \( \manfV \) is a compact manifold \(\manfM\).
We begin by explaining in the next sections some of the tools that we shall use for open sets.

\section{Opening revisited}

Opening has been introduced in Section~\ref{section_opening} as a tool to transform a function \(u \in \Sobolev^{k, p}(\R^{m})\)  into a new Sobolev function \( \widetilde{u} \) that is constant around a given point \( b \in \R^{m} \).
This is done by composition with a smooth map \(\Phi^{\mathrm{op}}  \colon  \R^{m} \to \R^{m}\) that is constant in a neighborhood of \( b \).
In particular, \( \Phi^{\mathrm{op}} \) is never a diffeomorphism and one cannot rely on the usual composition properties of Sobolev functions to show that \( u \compose \Phi^{\mathrm{op}} \in \Sobolev^{k, p}(\R^{m}) \).
In fact, the choice of \( \Phi^{\mathrm{op}} \) is highly dependent on \( u \), and different functions usually require different maps.

However, there is some flexibility regarding opening a map around larger sets. 
For example, assume that \( \omega \subset \R^{j} \) is an open set for some \( j \in \{0, \dots, m - 1 \} \).
For each point \( x' \in \omega \), one can open \( u \) around the set \(\omega \times \{0''\}\), where \(0'' \in \R^{m - j} \), with respect to the remaining \( m - j \) variables.
This procedure gives a smooth map \( \Phi^{\mathrm{op}} \colon  \omega \times \R^{m - j} \to \omega \times \R^{m - j} \) such that
\[
u \compose \Phi^{\mathrm{op}}(x', x'')
= u \compose \Phi^{\mathrm{op}}(x', 0'')
\]
for every \( x''  \in \R^{m - j} \) in a neighborhood of \( 0'' \).
It is important here that the choice of \( \Phi^{\mathrm{op}} \) be made independently of \( x' \).
We have thus opened \( u \) in the cylindrical region \( \omega \times \R^{m - j} \).

\begin{figure}
\centering
\hfill\includegraphics{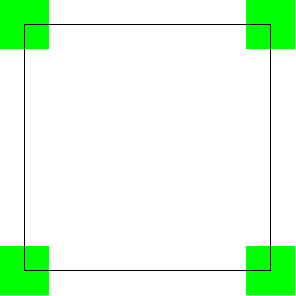}
\hfill\includegraphics{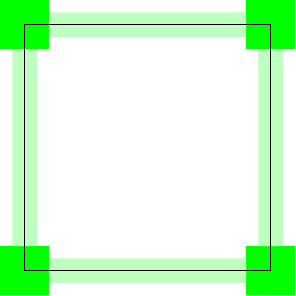}
\hfill{}
\caption{Opening around vertices (left) and then edges (right)}
\label{figureJEMSOpeningSkeletonCube}
\end{figure}

A further step consists in combining this construction at several levels.
For example, assume that we decompose \(\R^{2}\) as a grid of squares and we are given a function \(u \in \Sobolev^{k, p}(\R^{2})\).{}
We can first apply the opening technique around each vertex and then orthogonally around each side, see Figure~\ref{figureJEMSOpeningSkeletonCube}. 
Since we begin by constructing a new map that is constant near the vertices, the cylindrical constructions around the sides can be smoothly combined, and we obtain at the end a smooth map \(\Phi^{\mathrm{op}}  \colon  \R^{2} \to \R^{2}\)  that is constant around the vertices and depends on one variable near the sides of each square. 

We may pursue this construction using a cubication in \( \R^{m} \), which is the framework with which we shall deal in the proof of Theorem~\ref{thm_density_manifold_open}.
We first recall what we mean by a cubication:

\begin{definition}
\label{defnCubication}
A finite family of closed cubes \(\cS^m\) is a \emph{cubication} of \(A \subset \R^m\) if all cubes have the same radius, if \(\bigcup\limits_{\sigma^m \in \cS^m} \sigma^m = A\), and if for every \(\sigma^m_1, \sigma^m_2 \in \cS^m\) which are not disjoint, \(\sigma^m_1 \cap \sigma^m_2\) is a common face of dimension \(i \in \{0, \dotsc, m\}\).
\end{definition}

The radius of a cubication is the radius of any of its cubes.

\begin{definition}
\label{defnCubicationUnion}
Given a cubication \(\cS^m\) of \(A \subset \R^m\) and \(j \in \{0, \dotsc, m\}\), the skeleton of dimension \(j\) is the set \(\cS^j\) of all \(j\) dimensional faces of all cubes in \(\cS^m\). 
A subskeleton of dimension \(j\) of \(\cS^m\) is a subset of \(\cS^j\).
Given a skeleton \(\cS^j\), we denote by \(S^j\) the union of all elements of \(\cS^j\), that is,
\[
S^j = \bigcup_{\sigma^j \in \cS^j} \sigma^j.
\]
\end{definition}

Assuming that \( \cS^{m} \) is a cubication in \( \R^{m} \) and \( \cU^{\ell} \) is a subskeleton of dimension \( \ell < m \), we may successively open a map \( u \) around vertices in \( \cU^{0} \), around edges in \( \cU^{1} \), around faces in \( \cU^{2} \) and so on, up to \( \ell \)-dimensional cells in \( \cU^{\ell} \).
Although opening follows a cylindrical geometry, by pursuing opening from lower to higher dimensions, the construction agrees on the intersection of two adjacent cells, which is itself a cell of lower dimension where the function has already been opened.
We then obtain by combination of these constructions a smooth map \(\Phi^{\mathrm{op}}  \colon  \R^{m} \to \R^{m}\) such that \( u \compose  \Phi^{\mathrm{op}} \) belongs to \( \Sobolev^{k, p}(\R^{m}) \) and, for each cell \( \sigma^{j} \in \cU^{j} \) with \( j \le \ell \), the function \( u \compose  \Phi^{\mathrm{op}} \) is locally constant around \( \sigma^{j} \) with respect to \( m - j \)  coordinate directions that are orthogonal to \( \sigma^{j} \).
The rigorous implementation of this strategy is presented in \cite{BPVS_MO}*{Section~2}.

As we shall be dealing with Fuglede maps, it is convenient to gather properties of an opening map \( \Phi^{\mathrm{op}} \) constructed to preserve the summable nature of a given function \(w\).
As we explain below, for the sake of application \(w\) can be taken as a Sobolev detector for a Sobolev function \( u \).

\begin{proposition}\label{proposition_ouverture-new}
Let \(\cU^{\ell}\) be a subskeleton of a cubication in \( \R^{m} \) of radius \(\eta>0\) and let \(0<\rho<1/2\). 
For every summable function \(w  \colon  \R^{m} \to [0,+\infty]\), there exists a smooth map \(\Phi^{\mathrm{op}} \in \Fuglede_{w} (\R^{m}; \R^{m})\) such that \(\Phi^{\mathrm{op}} = \Id\) in a neighborhood of \(\R^{m} \setminus (U^{\ell} + Q^{m}_{2\rho\eta})\) and
\begin{enumerate}[\((i)\)]
\item 
\label{itemOpeningi}
for every \(i\in \{0, \dots, \ell\}\) and every \(\sigma^{i}\in \cU^{i}\), \(\Phi^{\mathrm{op}}\) is constant on each \(m-i\) dimensional cube of radius \(\rho \eta\) which is orthogonal to \(\sigma^{i}\),{}
\item 
\label{itemOpeningii}
for every \(i \in \{1, \dots, k\}\), \(\norm{D^{i}\Phi^{\mathrm{op}}}_{\Lebesgue^{\infty}(U^{\ell} + Q^{m}_{2\rho\eta})} \le C/\eta^{i-1}\),
\item
\label{itemOpeningiii}
for every \(\sigma^\ell \in \cU^\ell\), we have
\begin{equation}
\label{eq-density-Ushe1Ka4}
\Phi^{\mathrm{op}}(\sigma^\ell + Q^m_{2\rho\eta}) \subset \sigma^\ell + Q^m_{2\rho\eta} 
\end{equation}
 and
\begin{equation}\label{eq-density-752}
\norm{w \compose \Phi^{\mathrm{op}}}_{\Lebesgue^1(\sigma^{\ell} + Q^{m}_{2\rho\eta})}
\le C' \norm{w}_{\Lebesgue^1({\sigma^{\ell} + Q^{m}_{2\rho\eta}})}\text{,}
\end{equation}
\end{enumerate}
for some constants \(C\), \(C' > 0\) depending on \(m\).
\end{proposition}

In this statement, Property~\((\ref{itemOpeningi})\) encodes the opening nature of \(\Phi^{\mathrm{op}}\) near the set \(U^\ell\).
Property~\((\ref{itemOpeningii})\) ensures a compatible scaling of \(\Phi^{\mathrm{op}}\) with respect to the width of the neighborhood of \(U^\ell+ Q^m_{2\rho\eta}\) where \(\Phi^{\mathrm{op}}\) is deformed from the identity map.
The inclusion \eqref{eq-density-Ushe1Ka4} in Property~\((\ref{itemOpeningiii})\) encodes the localized character of the opening map \(\Phi^{\mathrm{op}}\), while \eqref{eq-density-752} is the fundamental property that relates \(\Phi^{\mathrm{op}}\) to \(w\) by ensuring the summability of \(w \compose \Phi^{\mathrm{op}}\) and the local control of its \(\Lebesgue^1\)~norm.

We refer the reader to \cite{BPVS_MO}*{Proposition~2.1} for the proof of Proposition~\ref{proposition_ouverture-new}.
There are three differences between the statement above and \cite{BPVS_MO}*{Proposition~2.1}.
The first one is that the estimate $(\ref{itemOpeningii})$ is not explicitly stated  in \cite{BPVS_MO}*{Proposition~2.1}. 
However, it can be easily deduced from a scaling argument.
The second one is that the estimate corresponding to \eqref{eq-density-752} in \cite{BPVS_MO} involves the \(\Lebesgue^{p}\) norm of the derivatives of \(u\) instead of the \(\Lebesgue^{1}\) norm of \(w\). 
We thus rely on the zeroth order counterpart of \cite{BPVS_MO}*{Proposition~2.1}, which makes unnecessary the approximation argument provided by \cite{BPVS_MO}*{Lemma~2.4}.
The idea of the proof is explained in Section~\ref{section_opening}, based on the estimate \eqref{eqJEMS-20}.
The third difference between the statement above and \cite{BPVS_MO}*{Proposition~2.1} is related to the inclusion \eqref{eq-density-Ushe1Ka4} above. 
Whereas it is now formulated in terms of each cell \(\sigma^\ell\), it was stated  in \cite{BPVS_MO} in the weaker form 
\begin{equation}
    \label{eqStrong1-207}
    \Phi^{\mathrm{op}}(U^\ell + Q^m_{2\rho\eta}) \subset U^\ell + Q^m_{2\rho\eta}.
\end{equation}
Still, we observe that the construction there actually gives the stronger inclusion \eqref{eq-density-Ushe1Ka4} above.
This is explained in detail in the proof of \cite{BPVS:2017}*{Proposition~5.2}.

\medskip

We now explain how to combine Propositions~\ref{propositionFugledeWkp} and~\ref{proposition_ouverture-new} to rigorously justify the general opening construction in the framework of Sobolev functions.
To this end, take a subskeleton \(\cU^{\ell}\) of a cubication in \( \R^{m} \) of radius \(\eta>0\) and let \(0<\rho<1/2\).
Given \( u \in \Sobolev^{k, p}(U^{\ell} + Q_{2\rho\eta}^{m}) \), denote by \( \widehat{w} \colon U^{\ell} + Q_{2\rho\eta}^{m} \to [0, + \infty] \) a Sobolev detector for \(u\) provided by Proposition~\ref{propositionFugledeWkp}.
We extend \( \widehat{w} \) by zero to \( \R^{m} \) and apply Proposition~\ref{proposition_ouverture-new} with \(\widehat{w}\) to obtain a map \( \Phi^{\mathrm{op}} \in \Fuglede_{\hat{w}}(\R^{m}; \R^{m}) \) obtained by opening.
From Property~$(\ref{itemOpeningiii})$ above, the inclusion \eqref{eqStrong1-207} is satisfied and then
\[
\Phi^{\mathrm{op}}|_{U^{\ell} + Q_{2\rho\eta}^{m}} \in \Fuglede_{\hat{w}}(U^{\ell} + Q_{2\rho\eta}^{m}; U^{\ell} + Q_{2\rho\eta}^{m}).
\]
It then follows from Proposition~\ref{propositionFugledeWkp} that
\( u \compose \Phi^{\mathrm{op}} \in \Sobolev^{k, p}(U^{\ell} + Q_{2\rho\eta}^{m}) \) and, for every \( j \in \{1, \dots, k\} \),
\resetconstant
  \begin{equation}
    \label{eq-density-281}
  \norm{D^{j}(u \compose \Phi^{\mathrm{op}})}_{\Lebesgue^{p}(U^{\ell} + Q_{2\rho\eta}^{m})}
  \le \C \, \norm{\widehat{w} \compose \Phi^{\mathrm{op}}}_{\Lebesgue^1(U^{\ell} + Q_{2\rho\eta}^{m})} ^{1/p} \sum_{i=1}^{j}{B_{i} \norm{D^{i}u}_{\Lebesgue^{p}(U^{\ell} + Q_{2\rho\eta}^{m})}}.
  \end{equation}
From the definition of the constant \( B_{i} \) in Proposition~\ref{propositionFugledeWkp} and Property~$(\ref{itemOpeningii})$, we have
\begin{equation}
    \label{eq-density-232}
    B_i 
     \le \C \sum\limits_{r_{1} + \dots + r_{i} = j}{\eta^{1-r_{1}} \cdots \eta^{1-r_{i}}}
    = \C \eta^{i - j}.
\end{equation}
Since \( \int_{\R^{m}} \widehat{w} \le 1 \), combining \eqref{eq-density-752}, \eqref{eq-density-281} and \eqref{eq-density-232}, one thus gets
  \begin{equation}
    \label{eq-density-288}
  \norm{D^{j}(u \compose \Phi^{\mathrm{op}})}_{\Lebesgue^{p}(U^{\ell} + Q_{2\rho\eta}^{m})}
  \le \C \, \sum_{i=1}^{j}{\eta^{i-j} \norm{D^{i}u}_{\Lebesgue^{p}(U^{\ell} + Q_{2\rho\eta}^{m})}}.
  \end{equation}
  
For the sake of the proof of Theorem~\ref{thm_density_manifold_open}, such a global estimate is insufficient for our purposes.
A minor modification of this argument yields a counterpart of \eqref{eq-density-288} in a neighborhood of each cell \( \sigma^{\ell} \in \cU^{\ell} \).
More precisely,

\begin{corollary}
    \label{corollaryOpeningCubicationSobolev}
    Let \(\cU^{\ell}\) be a subskeleton of a cubication in \( \R^{m} \) of radius \(\eta>0\), let \(0<\rho<1/2\), and let \(O \supset U^{\ell} + Q_{2\rho\eta}^{m}\) be an open set.
    If \( u \in \Sobolev^{k, p}(O) \), then there exist a summable function \( w \colon \R^{m} \to [0, +\infty] \) and a smooth map \(\Phi^{\mathrm{op}} \in \Fuglede_{w} (\R^{m}; \R^{m})\) as in  Proposition~\ref{proposition_ouverture-new} such that
    \[
    u \compose \Phi^{\mathrm{op}} \in \Sobolev^{k, p}(O)
    \]
    with \(u \compose \Phi^{\mathrm{op}} = u\) in \(O \setminus (U^{\ell} + Q_{2\rho\eta}^{m})\) and, for every \( \sigma^{\ell} \in \cU^{\ell} \) and every \( j \in \{1, \dots, k\} \), 
  \begin{equation}
    \label{eq-density-301}
  \norm{D^{j}(u \compose \Phi^{\mathrm{op}})}_{\Lebesgue^{p}(\sigma^{\ell} + Q_{2\rho\eta}^{m})}
  \le C'' \, \sum_{i=1}^{j}{\eta^{i - j} \norm{D^{i}u}_{\Lebesgue^{p}(\sigma^{\ell} + Q_{2\rho\eta}^{m})}}.
  \end{equation}
\end{corollary}

\begin{proof}
For each \( \sigma^{\ell} \in \cU^{\ell} \), we take the Sobolev detector \( w_{\sigma^{\ell}} \colon  \sigma^{\ell} + Q^{m}_{2\rho\eta} \to [0, +\infty] \) given by Proposition~\ref{propositionFugledeWkp} applied to \( \manfV = \sigma^{\ell} + Q^{m}_{2\rho\eta} \) and the restriction \( u|_{\sigma^{\ell} + Q^{m}_{2\rho\eta}} \).
Let \( w \colon \R^{m} \to [0, \infty] \) be a summable function such that \( w \ge w_{\sigma^{\ell}} \) in \( \sigma^{\ell} + Q^{m}_{2\rho\eta} \) for every \( \sigma^{\ell} \in \cU^{\ell} \).
Since \( \cU^{\ell} \) is a subskeleton of a cubication, there exists such a function whose integrals \( \int_{\sigma^{\ell} + Q^{m}_{2\rho\eta}} w \) are bounded from above by a constant depending only on \( m \).
Proceeding as above with \( U^{\ell} + Q_{2\rho\eta}^{m} \) replaced by \(  \sigma^{\ell} + Q^{m}_{2\rho\eta} \)\,, from \eqref{eq-density-752} we now get \eqref{eq-density-301}.
Moreover, since \(\Phi^{\mathrm{op}} = \Id\) in a neighborhood of \(O \setminus (U^{\ell} + Q_{2\rho\eta}^{m})\) and \( u \compose \Phi^{\mathrm{op}} \) is a \( \Sobolev^{k, p} \)~function on each open set \(\sigma^{\ell} + Q_{2\rho\eta}^{m} \), we conclude that \( u \compose \Phi^{\mathrm{op}} \in \Sobolev^{k, p}(O) \).
\end{proof}

Note that, around each cell \( \sigma^{\ell} \in \cU^{\ell} \), the function \( u \compose \Phi^{\mathrm{op}} \) depends on at most \( \ell \) coordinate variables.
Assuming that \emph{\( u \) is bounded}, it follows from the Gagliardo-Nirenberg interpolation inequality that the restriction \( u|_{U^{\ell} + Q_{2\rho\eta}^{m}} \) is a \( \Sobolev^{1, kp} \)~function.
For \( \ell < kp \), we then deduce from Morrey's imbedding that \( u \compose \Phi^{\mathrm{op}}|_{U^{\ell} + Q_{2\rho\eta}^{m}} \) equals almost everywhere a continuous function.
It is possible to handle the case where the equality \( \ell = kp \) holds using the \( \VMO \)~setting:

\begin{proposition}
    \label{propositionOpeningVMO}
    Let \( \cS^{m} \) be a cubication in \( \R^{m} \) of radius \( \eta > 0 \) containing \( \cU^{m} \) and let \( u \in (\Sobolev^{k, p} \cap L^{\infty})(S^{m} + Q_{2\rho\eta}^{m}) \).
    If \( \ell \le kp \), then the summable function \(w \colon \R^m \to [0,+\infty]\) and the smooth map \(\Phi^{\mathrm{op}} \in \Fuglede_{w} (\R^{m}; \R^{m})\) given by Corollary~\ref{corollaryOpeningCubicationSobolev} can be chosen so that the restriction \( u \compose \Phi^{\mathrm{op}}|_{U^{\ell} + Q_{\rho\eta}^{m}} \) is a \(\VMO\)~function such that, for every \( a \in U^{\ell} + Q^m_{\rho\eta/2} \) and every \(0 < r \le \rho\eta/2\), 
\[
\fint\limits_{Q_r^m (a)}\fint\limits_{Q_r^m (a)} \abs{u \compose \Phi^{\mathrm{op}}(x) - u \compose \Phi^{\mathrm{op}} (y)} \dif x \dif y 
\le \frac{C'''}{\eta^{\frac{m}{kp} - 1}} \norm{Du}_{\Lebesgue^{kp}(\sigma^{m} + Q_{2\rho\eta}^m)}\text{,}
\]
where \(\sigma^{m} \in \cS^{m} \) is such that \(a \in \sigma^m + Q^m_{\rho\eta/2}\)\,, and \(C''' > 0\) is a constant that depends on \(m\), \(k\), \(p\) and \(\rho\).
\end{proposition}

As \(\ell \le kp\), the proof can be found in \cite{BPVS_MO}*{Addendum~2 to Proposition~2.1}, where the statement deals with \( \Sobolev^{1, kp} \)~functions.
The estimate there is provided for points \(a \in \sigma^m\) with \(\sigma^m \in \cU^m\) such that \(Q_r^m (a) \subset U^{\ell} + Q^m_{\rho\eta}\)\,.
However, the proof only requires that \(Q_r^m (a) \subset U^{\ell} + Q^m_{\rho\eta}\)\,, which is always the case when \( a \in U^{\ell} + Q^m_{\rho\eta/2} \) and \(0 < r \le \rho\eta/2\) as we assume in Proposition~\ref{propositionOpeningVMO} above.
The estimate in \cite{BPVS_MO}*{Addendum~2 to Proposition~2.1} for the double average is given in terms of
\resetconstant
\[
\C \, \frac{r^{1-\frac{\ell}{kp}}}{\eta^{\frac{m-\ell}{kp}}} \norm{Du}_{\Lebesgue^{kp}(\sigma^{m} + Q_{2\rho\eta}^m)}
\]
in the right-hand side.
Since in our case we assume that \(\ell \le kp\) and \(r \le \rho\eta/2\), we have
\[
\frac{r^{1-\frac{\ell}{kp}}}{\eta^{\frac{m-\ell}{kp}}}
\le \frac{(\rho\eta/2)^{1-\frac{\ell}{kp}}}{\eta^{\frac{m-\ell}{kp}}}
\le \frac{\C}{\eta^{\frac{m}{kp} - 1}},
\]
which thus implies the estimate above.

\medskip



\medskip

Before concluding this section, we investigate whether it is possible to approximate a function by opened Sobolev functions, depending on the size of the set where the opening is performed.
More precisely, assume that we start with a \( \Sobolev^{k, p} \) function \( u \) and, for each \( \eta > 0 \) we are given a cubication \( \cU^{m}_{\eta} \) in \( \R^{m} \) of radius \( \eta \).
We can then consider the map \( \Phi^{\mathrm{op}}_{\eta} \) given by Corollary~\ref{corollaryOpeningCubicationSobolev} and ask whether the functions \( u \compose \Phi^{\mathrm{op}}_{\eta}  \) converge to \( u \) in \(  \Sobolev^{k, p} \) as \( \eta \to 0 \).
An answer to this problem is given in the following statement.

\begin{proposition}
    \label{corollaryOpeningConvergence}
    Let \(O \subset \R^m\) be a Lipschitz open set, let \( u \in (\Sobolev^{k, p} \cap \Lebesgue^{\infty})(O) \), and let  \(0<\rho<1/2\).
    If \( (\cU_{\eta}^m )_{0 < \eta < 1} \) is a family of cubications such that, for every \( 0 < \eta < 1\), \(U_{\eta}^m+Q^{m}_{2\rho\eta}\subset O\) and
    \[
    \abs{U_{\eta}^m} \le C \eta^{kp}\text{,}
    \] 
    then the maps \( \Phi^{\mathrm{op}}_{\eta} \) given by Corollary~\ref{corollaryOpeningCubicationSobolev} associated to \( \cU_{\eta}^{\ell} \) and \( u \) satisfy
    \[
    \lim_{\eta \to 0}{\norm{u \compose \Phi^{\mathrm{op}}_{\eta} - u}_{\Sobolev^{k, p}(O)}}
    = 0.
    \]
\end{proposition}

To answer this question, we rely on the following

\begin{lemma}
    \label{lemmaOpeningApproximation}
    Let \( (A_{\eta})_{0 < \eta < 1} \) be a family of Borel sets in \( O \subset \R^{m} \).
    If \( v \in \Lebesgue^{kp/i}(O) \) for some \( 0 < i \le k\) and if, for every \(0 < \eta < 1\), we have \( \abs{A_{\eta}} \le C \eta^{kp} \), then
    \[
    \lim_{\eta \to 0}{\eta^{i - k} \norm{v}_{\Lebesgue^{p}(A_{\eta})}}
    = 0.
    \]
\end{lemma}

\resetconstant
\begin{proof}[Proof of Lemma~\ref{lemmaOpeningApproximation}]
    By Hölder's inequality, we have
    \[
    \norm{v}_{\Lebesgue^{p}(A_{\eta})}
    \le \abs{A_{\eta}}^{\frac{k - i}{kp}}
    \norm{v}_{\Lebesgue^{{kp}/{i}}(A_{\eta})}.
    \]
    Thus, by the assumption on the measure of \( A_{\eta} \),
    \[
    0 \le \eta^{i - k}\norm{v}_{\Lebesgue^{p}(A_{\eta})}
    \le C^{\frac{k-i}{kp}}\norm{v}_{\Lebesgue^{{kp}/{i}} (A_{\eta})}.
    \]
    Since \( \abs{A_{\eta}} \to 0 \) as \( \eta \to 0 \), the quantity in the right-hand side converges to zero and the conclusion follows.
\end{proof}

\begin{proof}[Proof of Proposition~\ref{corollaryOpeningConvergence}]
    Since \( \Phi^{\mathrm{op}}_{\eta} = \Id \) in the complement of \( U^{\ell}_{\eta} + Q_{2\rho\eta}^{m} \), for every \( j \in \{1, \dots, k\} \)  we have
    \begin{multline}
    \label{eqStrong1-342}
    \norm{D^{j}(u \compose \Phi^{\mathrm{op}}_{\eta}) - D^{j}u}_{\Lebesgue^{p}(O)}\\
    \le \norm{D^{j}(u \compose \Phi^{\mathrm{op}}_{\eta})}_{\Lebesgue^{p}(U^{\ell}_{\eta} + Q_{2\rho\eta}^{m})} + \norm{D^{j}u}_{\Lebesgue^{p}(U^{\ell}_{\eta} + Q_{2\rho\eta}^{m})}.
    \end{multline}
    The radius of the cubication \(\cU^{m}_\eta\) being \(\eta\),  the measures of \( U^{m}_{\eta} \) and \( U^{\ell}_{\eta} + Q_{2\rho\eta}^{m} \) are comparable and thus,
\begin{equation}\label{eq-meas-comparable}
\abs{ U^{\ell}_{\eta} + Q_{2\rho\eta}^{m}} \le \C \eta^{kp}.
\end{equation}      
In particular,  \( \abs{U^{\ell}_{\eta} + Q_{2\rho\eta}^{m}} \to 0 \) as \( \eta \to 0 \), which implies that the last term in the right-hand side of ~\eqref{eqStrong1-342} converges to zero.
    
In order to estimate the first term, we rewrite \eqref{eq-density-288} using the map \(\Phi^{\mathrm{op}}_{\eta}\)\,,
    \begin{equation}
    \label{eqStrong1-347}
    \norm{D^{j}(u \compose \Phi^{\mathrm{op}}_{\eta})}_{\Lebesgue^{p}(U^{\ell}_{\eta} + Q_{2\rho\eta}^{m})}
  \le \C \sum_{i=1}^{j}{\eta^{i - j} \norm{D^{i}u}_{\Lebesgue^{p}(U^{\ell}_{\eta} + Q_{2\rho\eta}^{m})}}.
    \end{equation}
    Since \( u \) is bounded and \(O\) is a Lipschitz open set, by the Gagliardo-Nirenberg interpolation inequality we have \( D^{i}u \in \Lebesgue^{kp/i}(O) \).
    Using \eqref{eq-meas-comparable} and the fact that \( j \le k \), it  follows from Lemma~\ref{lemmaOpeningApproximation} that the right-hand side of \eqref{eqStrong1-347} converges to zero.
    Thus, as both terms in \eqref{eqStrong1-342} converge to zero, we deduce that
    \[
    \lim_{\eta \to 0}{\norm{D^{j}(u \compose \Phi^{\mathrm{op}}_{\eta}) - D^{j}u}_{\Lebesgue^{p}(O)}}
    = 0.
    \]
    To conclude the proof it suffices to recall that \( u = u \compose  \Phi^{\mathrm{op}}_{\eta} \) in the complement of \(  U^{\ell}_{\eta} + Q_{2\rho\eta}^{m} \), which implies that \( (u \compose \Phi^{\mathrm{op}}_{\eta})_{0 < \eta < 1} \) converges in measure to \( u \).
\end{proof}

\section{Thickening}
\label{subsection_thickening}

The thickening tool has been used in problems involving compact target manifolds~\cites{Bethuel-Zheng, Bethuel, Hang-Lin}, and allows one to extend a Sobolev map \(u\) defined on the boundary of a star-shaped domain to the whole domain, preserving the range of \(u\).
	We first explain the idea using the strategy of zero-degree homogenization in a typical model situation that we illustrate in Figure~\ref{figureJEMSThickeningDisc}.

\begin{figure}
\centering{}
\hfill\includegraphics[align=c]{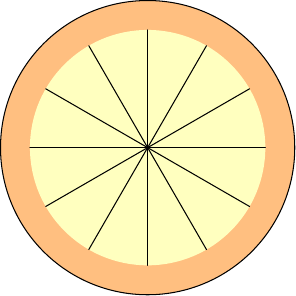}
\hfill\includegraphics[align=c,scale=0.9]{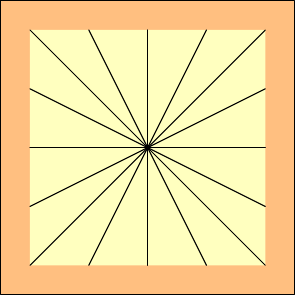}
\hfill{}
\caption{Thickening in a disc and in a square}
\label{figureJEMSThickeningDisc}
\end{figure}

Assume that we have a smooth function \(u \colon \cBall^{m} \setminus B_r^m \to \R \), with \(0 < r < 1\).
The goal is to deform \(u\) to obtain a new function defined in the thickened domain \(\cBall^m \setminus \{0\}\).
More precisely, given \(r < \rho < 1 \) we wish to construct a smooth map \(\Phi^{\mathrm{th}}  \colon  \cBall^{m} \setminus\{0\} \to \cBall^{m}\) such that
\begin{enumerate}[\((\mathrm{th}_1)\)]
	\item{}
	\label{itemJEMS-599} 
	\(u \compose \Phi^{\mathrm{th}}\) is smooth in \(\cBall^{m} \setminus \{0\}\),{}
	\item{}
	\label{itemJEMS-600}
	\(u \compose \Phi^{\mathrm{th}} = u\) in \(\cBall^{m} \setminus B_{\rho}^m\),{}
	\item{}
	\label{itemJEMS-604}
	\(u \compose \Phi^{\mathrm{th}} \in \Sobolev^{k, p}(\Ball^{m})\) and
	\[{}
	\norm{u \compose \Phi^{\mathrm{th}}}_{\Sobolev^{k, p}(\Ball^{m})}
	\le	C\norm{u}_{\Sobolev^{k, p}(\Ball^{m} \setminus \overline{B}{}_r^m)},
	\]
\end{enumerate}
for a constant \(C > 0\) depending on \(r\) and \(\rho\).

If one is willing to focus on the case of order \(k = 1\) and obtain a map \(u \compose \Phi^{\mathrm{th}}\) that is merely locally Lipschitz continuous in \(\cBall^{m} \setminus \{0\}\) instead of being smooth, then a good choice for \(\Phi^{\mathrm{th}}\) is
\[{}
\Phi^{\mathrm{th}}(x)
=
\begin{cases}
	x	& \text{if \(x \in \cBall^{m} \setminus B_{t}^{m}\),}\\
	t {x}/{\abs{x}} & \text{if \(x \in B_{t}^{m} \setminus \{0\}\),}
\end{cases}
\]
where \(r < t < \rho\) is chosen such that
	\begin{equation}
    \label{eqJEMS-510}
    \int_{\partial B_{t}^{m}} (\abs{u}^{p} + \abs{Du}^{p})  \dif\sigma
	\le \frac{1}{\rho - r} \int_{B_\rho^{m} \setminus B_r^m}{(\abs{u}^{p} + \abs{Du}^{p})}.
	\end{equation}
The existence of such a \(t\) follows from the integration formula in polar coordinates,
	\[{}
	\int_{r}^{\rho} \biggl( \int_{\partial B_{s}^{m}} (\abs{u}^{p} + \abs{Du}^{p})  \dif\sigma \biggr) \dif s
	= \int_{B_\rho^{m} \setminus B_r^m}{(\abs{u}^{p} + \abs{Du}^{p})}.
	\]

Although this standard approach nicely handles the first order case, it is not suitable for \(k \ge 2\) as it does not take into account the compatibility of the derivatives at the interface \(\partial B_{t}^{m}\).
To remedy this obstruction, we rely on an alternative approach that has been used in \cite{BPVS_MO}*{Section~4} where \(\Phi^{\mathrm{th}}  \colon  \cBall^{m} \setminus \{0\} \to \cBall^{m} \setminus \overline{B}{}^{m}_{r}\) is a diffeomorphism.
The map constructed in this case is also independent of the initial function \(u\).

\begin{proposition}
	\label{propositionJEMSThickeningModel}
	Given \(0 < r < \rho < 1\), there exists a diffeomorphism \(\Phi^{\mathrm{th}}  \colon  \cBall^{m} \setminus \{0\} \to \cBall^{m} \setminus \overline{B}{}_{r}^{m}\) such that \(\Phi^{\mathrm{th}} = \Id\) in \(\cBall^{m} \setminus B_\rho^m\) and
	\begin{enumerate}[\((i)\)]
		\item for every \(j \in \N_*\), there exists \(C' > 0\) depending on \(j\), \(m\), \(r\) and \(\rho\) such that
	\[
	\abs{D^j \Phi^{\mathrm{th}}(x)} \le  \frac{C'}{\abs{x}^j}\text{,}
	\]
	\item for every \(\beta < m\), there exists \(C'' > 0\) depending on \(\beta\), \(m\), \(r\) and \(\rho\) such that
	\[
	\Jacobian{m}{\Phi^{\mathrm{th}}}(x) \ge \frac{C''}{\abs{x}^\beta}.
	\]
	\end{enumerate}
\end{proposition}

\resetconstant
\begin{proof}
	We take \(\Phi^{\mathrm{th}}  \colon  \cBall^{m} \setminus \{0\} \to \cBall^{m} \setminus \overline{B}{}^{m}_{r}\) of the form
\(\Phi^{\mathrm{th}}(x) = \varphi(\abs{x}) x\) where \(\varphi  \colon  (0, 1]{} \to (0, \infty)\) is a smooth function such that
	\begin{enumerate}[(a)]
		\item \label{voisinage1} for every \(\rho \le s \le 1\), \(\varphi(s) = 1\),		
		\item \label{deriveevarphi} for every \(j \in \N \), there exists \(\Cl{eqJEMS-654-bis} > 0\) such that
		\(
		\abs{\varphi^{(j)}(s)} 
		\le {\Cr{eqJEMS-654-bis}}/{s^{j+1}},
		\) 
		\item \label{deriveeminoreevarphi}
		for every \(\alpha < 1\), there exists \(\Cl{eqJEMS-655} > 0\) such that
		\(
		(\varphi(s)s)' \ge {\Cr{eqJEMS-655}}/{s^{\alpha}},
		\)		
		\item \label{increasingvarphi}
		\(
		\lim\limits_{s \to 0}{\varphi(s)s} = r.
		\)		
	\end{enumerate}
	In a neighborhood of \(s = 0\), one can take for example
	\[{}
	\varphi(s){}
	= \frac{r}{s} \Bigl( 1 + \frac{1}{\ln{(1/s)}} \Bigr).
	\]
	
	We observe that \(\Phi^{\mathrm{th}}\) as above is then a diffeomorphism between \(\cBall^{m} \setminus \{0\}\) and \(\cBall^{m} \setminus \overline{B}{}_{r}^{m}\).{}
	Using the Leibniz rule, one deduces the estimate of \(\abs{D^{j}\Phi^{\mathrm{th}}}\) for every \(j \in \N_*\).{}
	To compute \(\Jacobian{m}{\Phi^{\mathrm{th}}}\), we first note that, for every \(v \in \R^{m}\),{}
	\[{}
	D\Phi^{\mathrm{th}}(x)[v]
	= \varphi(\abs{x}) v + \varphi'(\abs{x}) \frac{x \cdot v}{\abs{x}} x.
	\]
	Thus, \(D\Phi^{\mathrm{th}}(x)\) is a rank-one perturbation of \(\varphi(\abs{x})\Id\). 
	Hence, it has the eigenvalues \(\varphi(\abs{x})\), with multiplicity \(m - 1\), and \(\varphi(\abs{x}) + \varphi'(\abs{x}) \abs{x}\).{}
	We get
	\[{}
	\Jacobian{m}{\Phi^{\mathrm{th}}}(x)
	= (\varphi(\abs{x}))^{m - 1} \bigl( \varphi(\abs{x}) + \varphi'(\abs{x}) \abs{x} \bigr).
	\]
	Since \(\varphi(s) \ge r/s\) and, for every \(\alpha < 1\), we have \((\varphi(s)s)' \ge \Cr{eqJEMS-655}/s^{\alpha}\),  we deduce that
	\[{}
	\abs{\Jacobian{m}{\Phi^{\mathrm{th}}}(x)}
	\ge \frac{r^{m - 1} \Cr{eqJEMS-655}}{\abs{x}^{m - 1 + \alpha}}.
	\qedhere
	\]
\end{proof}

Using the proposition above, one can handle the model problem posed in the beginning of the section:

\begin{corollary}
    \label{corollaryThickeningBall}
 If \(kp < m\) and \(\Phi^{\mathrm{th}}\) is the diffeomorphism given by Proposition~\ref{propositionJEMSThickeningModel}, then \(u \compose \Phi^{\mathrm{th}}\) satisfies properties \((\mathrm{th}_{\ref{itemJEMS-599}})\)--\((\mathrm{th}_{\ref{itemJEMS-604}})\).
\end{corollary}
\resetconstant
\begin{proof}
	We take \(\Phi^{\mathrm{th}}\) given by Proposition~\ref{propositionJEMSThickeningModel}.
    Then, by construction \(u \compose \Phi^{\mathrm{th}}\) verifies \((\mathrm{th}_{\ref{itemJEMS-599}})\) and~\((\mathrm{th}_{\ref{itemJEMS-600}})\).
    We now prove that, for every \(j \in \{1, \dots, m\}\),{}
	\begin{equation}
		\label{eqJEMSThickeningCorollary-Model}
	\norm{D^{j}(u \compose \Phi^{\mathrm{th}})}_{\Lebesgue^{p}(\Ball^{m})}
	\le \C \sum_{i=1}^{j}\norm{D^{i}u}_{\Lebesgue^{p}(\Ball^{m} \setminus \overline{B}{}_{r}^{m})}.
	\end{equation}
	By the Faà di Bruno composition formula and the estimates of \(\abs{D^{i}\Phi^{\mathrm{th}}}\), for \( x \ne 0 \) we find
	\[{}
	\abs{D^{j}(u \compose \Phi^{\mathrm{th}})(x)}
	\le \C \sum_{i = 1}^{j} \abs{(D^{i} u \compose \Phi^{\mathrm{th}})(x)} \cdot \frac{1}{\abs{x}^{j}}. 
	\]
	Since \(jp \le kp < m\), we have \(1/\abs{x}^{jp} \le \C \, \Jacobian{m}{\Phi^{\mathrm{th}}}(x)\).{}
	Thus,
	\[
	\int_{\Ball^{m}}\abs{D^{j}(u \compose \Phi^{\mathrm{th}})}^{p}
	\le \C \sum_{i = 1}^{j} \int_{\Ball^{m}} \abs{D^{i} u \compose \Phi^{\mathrm{th}}}^{p} \Jacobian{m}{\Phi^{\mathrm{th}}}.
	\]
	The change of variables \(y = \Phi^{\mathrm{th}}(x)\) in the right-hand side then implies \eqref{eqJEMSThickeningCorollary-Model}.
	The \(\Lebesgue^{p}\) estimate for \(u \compose \Phi^{\mathrm{th}} \) follows along the same lines.
    Since \( u \compose \Phi^{\mathrm{th}} \) is smooth except at \( 0 \) and has its derivatives in \( \Lebesgue^{p} \), we conclude that \( u \compose \Phi^{\mathrm{th}} \in \Sobolev^{k, p}(\Ball^{m}) \), see \cite{Brezis-Mironescu}*{Lemma~1.10}.
\end{proof}

We have presented thickening with respect to a ball, but a similar strategy can be performed in a cube and more generally using a cubication \( \cS^{m} \), with \(\Phi^{\mathrm{th}}\) being the identity in a neighborhood of \(S^{m - 1}\) and smooth except at the centers of each cube in \( \cS^{m} \).
Such a construction can then be generalized to maps that are smooth in a neighborhood of the \(\ell\)-dimensional skeleton of the cubication, based on an iteration argument. 
The map obtained is smooth except on the dual skeleton of dimension \(\ell^{*} = m - \ell - 1\).
The precise meaning of dual skeleton we use is the following:

\begin{definition}
\label{definitionDualSkeleton}
Let \( \cS^{m} \) be a cubication in \( \R^{m} \) and let \( \ell \in \{0, \dots, m - 1 \}\).
The dual skeleton \(\cT^{\ell^*}\) of\/ \(\cS^\ell\) is the 
family of all \( \ell^{*} \)-dimensional cubes of the form \(\sigma^{\ell^*} + x - a\), where \(\sigma^{\ell^*} \in \cS^{\ell^*}\), \(a\) is the center and \(x\) a vertex of a cube of \(\cS^m\).
\end{definition}

For the proof of Theorem~\ref{thm_density_manifold_open}, we need a variant of Proposition~\ref{propositionJEMSThickeningModel} adapted to a cubication.
For the convenience of the reader, we write the statement using the notation that is employed in the proof of the theorem.

\begin{proposition}\label{propositionthickeningfromaskeleton}
Let \( \cU^{m} \subset \cS^m\) be cubications in \( \R^{m} \) of radius \( \eta > 0 \) and, given \( \ell \in \{0, \dots, m - 1 \}\), let \(\cZ^{\ell^*}\) be the dual skeleton of\/ \(\cU^\ell\).
Then, for every \( 0 < \rho < 1/2 \), there exists an injective smooth map \(\Phi^{\mathrm{th}}  \colon  \R^m \setminus Z^{\ell^*} \to \R^m\) such that
\begin{enumerate}[$(i)$]
\item for every \(\sigma^m \in \cS^m\), \(\Phi^{\mathrm{th}}(\sigma^m \setminus Z^{\ell^*}) \subset \sigma^m \setminus Z^{\ell^*}\),
\label{itempropositionthickeningfromaskeleton1}
\item 
\(\Phi^{\mathrm{th}}(U^m\setminus Z^{\ell^*}) \subset U^\ell + Q^{m}_{\rho\eta/2}\)\,,
\label{itempropositionthickeningfromaskeleton2}
\item \(\Phi^{\mathrm{th}} = \Id\) in \((U^\ell +Q^{m}_{\rho\eta/4}) \cup \bigl(\R^{m} \setminus(U^m + Q^{m}_{\rho\eta/2})\bigr)\),
\label{itempropositionthickeningfromaskeleton3}
\item for every \(j \in \N_*\) and every \(x \in  (U^m + Q^{m}_{\rho\eta/2})\setminus Z^{\ell^*}\),
\[
\abs{D^j \Phi^{\mathrm{th}}(x)} \le  
\frac{C'\eta}{d(x, Z^{\ell^*})^j}\text{,}
\]
for some constant \(C' > 0\) depending on \(j\), \(m\) and \(\rho\),
\label{itempropositionthickeningfromaskeleton4}
\item for every \(0 < \beta < \ell + 1\), every \(j \in \N_*\) and every \(x \in \R^m \setminus Z^{\ell^*}\),
\[
\eta^{j-1} \abs{D^j \Phi^{\mathrm{th}}(x)} 
    \le  C'' \bigl(\Jacobian{m}{\Phi^{\mathrm{th}}}(x)\bigr)^{{j}/{\beta}}\text{,}
\]
for some constant \(C'' > 0\) depending on \(\beta\), \(j\), \(m\) and \(\rho\).
\label{itempropositionthickeningfromaskeleton5}
\end{enumerate}
\end{proposition}

This statement is essentially the same as \cite{BPVS_MO}*{Proposition~4.1}, where the reader can find the proof.
The only difference here is that Property~\((\ref{itempropositionthickeningfromaskeleton3})\)  appears in  \cite{BPVS_MO} in the weaker form \(\Phi^{\mathrm{th}} = \Id\) in \(\R^{m} \setminus (U^m + Q^{m}_{\rho\eta/2})\). 
However, the additional requirement that \(\Phi^{\mathrm{th}} = \Id\) on \(U^\ell + Q^{m}_{\rho\eta/{4}}\) is an easy consequence of the construction detailed in the proof. 
More specifically, \(\Phi^{\mathrm{th}}\) is defined as the compositions of maps \(\Psi_i\) with \(i = \ell+1, \dots, m\)\,:
\[
\Phi^{\mathrm{th}} = \Psi_{\ell+1}\compose \Psi_{\ell+2}\compose \dots \compose \Psi_m\,,
\]
where each map \(\Psi_i\) satisfies \( \Psi_i = \Id\) on \(U^{i-1}+Q^{m}_{\epsilon_{i-1}\eta}\) for some \(\epsilon_{i-1}\)  that can be chosen larger than \({\rho}/{4}\). 
By composition one deduces the desired property that \(\Phi^{\mathrm{th}} = \Id\) on \(U^\ell + Q^{m}_{\rho\eta/4}\).{}

We deduce the following consequence of Proposition~\ref{propositionthickeningfromaskeleton} for the composition of \(\Phi^{\mathrm{th}}\) with smooth functions:

\begin{corollary}
    \label{corollaryThickeningCubication}
    Let \(\Phi^{\mathrm{th}}  \colon  \R^m \setminus Z^{\ell^*} \to \R^m\) be the  injective smooth map given by Proposition~\ref{propositionthickeningfromaskeleton} and let \(O \supset U^{m} + Q^{m}_{\rho\eta/2} \) be an open set.
    Then, for every smooth function \( u \colon O \to \R \) with bounded derivatives, \( u \compose \Phi^{\mathrm{th}} \) is an \( \ClassR_{\ell^{*}} \)~function in \(O\) with singular set \( Z^{\ell^{*}} \) such that \( u \compose \Phi^{\mathrm{th}} = u \) in \( O \setminus (U^{m} + Q^{m}_{\rho\eta/2}) \).
    Moreover, if \( kp < \ell + 1 \) and \(O\) is bounded, then 
    \( u \compose \Phi^{\mathrm{th}} \in \Sobolev^{k, p}(O) \) and, 
    for every \( j \in \{1, \dots, k\} \), 
  \begin{equation*}
  \norm{D^{j}(u \compose \Phi^{\mathrm{th}})}_{\Lebesgue^{p}(U^{m} + Q_{\rho\eta/2}^{m})}
  \le C''' \, \sum_{i=1}^{j}{\eta^{i-j} \norm{D^{i}u}_{\Lebesgue^{p}(U^{m} + Q_{\rho\eta/2}^{m})}}.
  \end{equation*}
\end{corollary}

\resetconstant
\begin{proof}
We first observe that, by Property~\((\ref{itempropositionthickeningfromaskeleton3})\) and by injectivity of \(\Phi^{\mathrm{th}}\), we have
\[
 \Phi^{\mathrm{th}}\bigl((U^m + Q^{m}_{\rho\eta/2})\setminus Z^{\ell^*}\bigr) \subset U^m + Q^{m}_{\rho\eta/2}.
\]
In particular, \( u \compose \Phi^{\mathrm{th}} \) is well-defined in \(O \setminus Z^{\ell^*}\) and \( u \compose \Phi^{\mathrm{th}} = u \) in \( O \setminus (U^{m} + Q^{m}_{\rho\eta/2}) \).
By the Fa\'a di Bruno composition formula, for every \( j \in \N_{*} \) and \( x \not\in Z^{\ell^{*}} \),
\begin{multline}
    \label{eqStrong1-611}
\abs{D^j (u \compose \Phi^{\mathrm{th}})(x)} \\
\le \C \sum_{i = 1}^j \sum_{\substack{1 \le \theta_1 \le \dotsc \le \theta_i\\\theta_1 + \dots + \theta_i = j}} \abs{D^i u(\Phi^{\mathrm{th}}(x))} \abs{D^{\theta_1}\Phi^\mathrm{th}(x)} \dotsm \abs{D^{\theta_i}\Phi^\mathrm{th}(x)}.
\end{multline}
Hence, by the pointwise estimates in Property~\((\ref{itempropositionthickeningfromaskeleton4})\) satisfied by the derivatives of \(\Phi^{\mathrm{th}}\),
\begin{multline*}
    \abs{D^j (u \compose \Phi^{\mathrm{th}})(x)}\\ 
 \le \C \sum_{i = 1}^j \sum_{\substack{1 \le \theta_1 \le \dotsc \le \theta_i\\\theta_1 + \dots + \theta_i = j}}  \frac{\eta}{(d(x, Z^{\ell^*}))^{\theta_1}} \dotsm \frac{\eta}{(d(x, Z^{\ell^*}))^{\theta_i}}
 \le  \frac{\C}{(d(x, Z^{\ell^*}))^{j}}.
\end{multline*}
We thus have that \(u \compose \Phi^{\mathrm{th}}\) is an \( \ClassR_{\ell^{*}}\)~function in \(O\) with singular set \(Z^{\ell^*}\).
For \( kp < \ell + 1 \) and \(O\) bounded, we deduce from Proposition~\ref{propositionClassRInclusionSobolev} that \( u \compose \Phi^{\mathrm{th}} \in \Sobolev^{k, p}(O) \).
The \(\Lebesgue^{p}\)~estimates for \( D^{j}(u \compose \Phi^{\mathrm{th}}) \) in \(U^m + Q^{m}_{\rho\eta/2}\) can then be proved along the line of the proof of Corollary~\ref{corollaryThickeningBall} by a change of variable using \eqref{eqStrong1-611} and the pointwise estimate of the Jacobian from Property~\((\ref{itempropositionthickeningfromaskeleton5})\), see \cite{BPVS_MO}*{Corollary~4.2}.
\end{proof}

\section{Adaptive smoothing}
\label{sectionJEMSSmoothing}

Convolution is a powerful tool for building smooth approximations of a locally summable function \(u\).
In \cite{BPVS_MO}, we relied on the adaptive smoothing, based on a convolution with variable parameter, to take into account the mean oscillations of \(u\).
We apply again this strategy in the proof of Theorem~\ref{thm_density_manifold_open}.

To illustrate a typical problem where the adaptive smoothing can be used, take a bounded function \(u \colon \R^m \setminus \{0\} \to \R\) that is smooth in \(\R^{m} \setminus \overline{B}{}_{1/2}^m\) and satisfies, for every \(y, z \in \R^m \setminus \{0\}\), the locally Lipschitz estimate
\begin{equation}
\label{eqJEMS-SmoothingEstimate}
\abs{u(y) - u(z)}
\le C \frac{\abs{y - z}}{\min{\{\abs{y}, \abs{z}\}}}.
\end{equation}
Given \(\epsilon > 0\), we wish to find a bounded smooth function \(v \colon \R^m \setminus \{0\} \to \R\) such that
\begin{enumerate}[\((\mathrm{sm}_1)\)]
	\item{}
	\label{itemJEMS-676}
	\(v= u\) in \(\R^m \setminus \Ball^{m}\),{}
	\item{}
	\label{itemJEMS-679}
	for every \(x \in \Ball^m \setminus \{0\}\), we have \(\abs{u(x) - v(x)} \le \epsilon\),
    \item
    \label{itemJEMS-682}
    for every \(j \in \N_*\) and every \(x \in \Ball^m \setminus \{0\}\),
	\[{}
	\abs{D^j v(x)}
	\le	\frac{C_j'}{\epsilon^j \abs{x}^j}.
	\]
\end{enumerate}
A natural strategy is to take \(v\) in the form
\[
v = u\zeta + (\varphi_t * u) (1 - \zeta),
\]
where \(\zeta \colon \R^m \to \R\) is a smooth function such that \(\zeta = 1\) in \(\Ball^m\) and \(\zeta = 0\) in \(B_{3/4}^m\)\,, and \(\varphi_t * u\) is the convolution defined by
\[
\varphi_{t} * u (x)
= \int_{\Ball^m}\varphi(z)u(x+ tz)\dif z
\]
using a \emph{mollifier} \(\varphi \colon \R^m \to \R\), that is, a nonnegative smooth function supported in \(\Ball^m\) which also satisfies the normalization condition
\begin{equation}
\label{eqJEMS-ConvolutionNormalization}
\int_{\Ball^m}\varphi
= 1.
\end{equation}

Although properties \((\mathrm{sm}_{\ref{itemJEMS-676}})\) and \((\mathrm{sm}_{\ref{itemJEMS-676}})\) are satisfied for any \(t > 0\), it is less clear that \((\mathrm{sm}_{\ref{itemJEMS-679}})\) holds for a suitable choice of \(t\).
To fully exploit the local Lipschitz estimate one could take, for example, a partition of the unit \((\zeta_j)_{j \in \N}\) subordinated to a covering of the set \(\Ball^m \setminus \{0\}\) by dyadic-type annuli and then choose a suitable convolution parameter \(t\) for each convolution product \(\varphi_t * (\zeta_j u)\).

To avoid this double averaging process --- partition of unit and convolution --- one may rely directly on the adaptive smoothing.
To explain the general setting on which we shall work, take a locally summable function \(u \colon A \to \R\) defined in an open set \(A \subset \R^m\).
Given a continuous nonnegative function \(\psi \colon A \to \R\), we define the open set
\begin{equation}
\label{eqJEMS-AdaptiveConvolutionDomain}
A_{\psi}
= \bigl\{ x \in A : \psi(x) < d(x, \partial A )\bigr\}.
\end{equation}
The \emph{adaptive convolution} is defined for every \(x \in A_{\psi}\) as
\begin{equation}
\label{eqJEMS-AdaptiveConvolution}
\varphi_{\psi} * u (x)
= \int_{\Ball^m} \varphi(z) u(x+\psi(x)z) \dif z.
\end{equation}
By the normalization condition \eqref{eqJEMS-ConvolutionNormalization}, this function coincides with \(u\) on the level set \(\{\psi=0\}\).

Using a suitable choice of \(\psi\), one can prove the existence of a smooth bounded function \(v\) that satisfies properties \((\mathrm{sm}_{\ref{itemJEMS-676}})\)--\((\mathrm{sm}_{\ref{itemJEMS-682}})\).
More generally, the following holds:

\begin{proposition}
	\label{propositionAdaptiveApplication-New}
Let \(A \subset \R^m\) be a bounded open set, let \( F_0 \subset F_1 \subset \R^{m} \) be closed sets, and let \(\delta>0\).
If \(u \colon A \setminus F_0 \to \R\) is a bounded  function that is smooth in \(A \setminus \overline{F_1 + B_\delta^m}\) and satisfies, for every \(y, z \in  A \setminus F_0\)\,, the locally Lipschitz estimate
	\begin{equation}\label{eqJEMS-709}
	\abs{u(y) - u(z)}
	\le C \frac{\abs{y - z}}{\min{\bigl\{d(y, F_0), d(z, F_0)\bigr\}}}\text{,}
	\end{equation}
then, for every \(0 < \epsilon \le 1\) and every open subset \(O \Subset A\), there exists a bounded smooth function \(v \colon O \setminus F_0 \to \R\) such that 
\begin{enumerate}[\((i)\)]
    \item
    \label{itemJEMS-745a}
    \(v = u\) on \(O \setminus (F_1 + B^{m}_{2\delta})\),
    \item 
    \label{itemJEMS-746}
    for every \(x \in ((F_1 + B^{m}_{2\delta}) \setminus F_0) \cap O\), we have \(\abs{u(x) - v(x)} \le \epsilon\),
    \item 
    \label{itemJEMS-propositionAdaptiveApplication-last} 
    for every \(j \in \N_{*}\) and every \(x \in ((F_1 + B^{m}_{2\delta}) \setminus F_0) \cap O\),
\[
\abs{D^{j}v(x)}
\le \frac{C_j'}{\epsilon^j d(x, F_0)^j} \text{,}
\]
for some constant \(C_j' > 0\) independent of \(x\) and \(\epsilon\).
\end{enumerate}
\end{proposition}

To obtain a function that satisfies \((\mathrm{sm}_{\ref{itemJEMS-676}})\)--\((\mathrm{sm}_{\ref{itemJEMS-682}})\), it suffices to apply this proposition with \(F_0 = F_1 = \{0\} \), \(A = B_2^m\), \(O = B_{3/2}^m\) and \(\delta = 1/2\).
In the next section, we give an application of Proposition~\ref{propositionAdaptiveApplication-New} to the problem of smooth extension of a map where \(F_0\) and \(F_1\) are taken as dual skeletons of different dimensions for a given cubication.

We prove Proposition~\ref{propositionAdaptiveApplication-New} based on a couple of properties satisfied by \(\varphi_\psi * u\).

\begin{lemma}
\label{lemma-adaptative-smooth}
The adaptive convolution \(\varphi_{\psi} * u\) is smooth on any open subset of \(A_\psi\) where \(\psi\) is smooth and positive, and also on any open subset of \(A_\psi\) where \(u\) and \(\psi\) are smooth.
\end{lemma}

\begin{proof}[Proof of Lemma~\ref{lemma-adaptative-smooth}]
Let \(\omega_1 \subset A_\psi\) be an open set where \(\psi\) is smooth and positive.
For any \(x \in \omega_1\), since \(\varphi\) is supported in \(\Ball^m\), by a linear change of variables in the integral in \eqref{eqJEMS-AdaptiveConvolution} we have
\[
\varphi_{\psi} * u (x)
= \frac{1}{\psi(x)^m} \int_{\Ball^m} \varphi\Bigl( \frac{y - x}{\psi(x)} \Bigr) u(y) \dif y.
\]
Since \(\psi\) is smooth in \(\omega_1\), by differentiation under the integral sign we deduce that \(\varphi_{\psi} * u\) is also smooth in \(\omega_1\).

Now let \(\omega_2 \subset A_\psi\) be an open set where \(u\) and \(\psi\) are smooth.
By the previous analysis, \(\varphi_{\psi} * u\) is smooth in \(\omega_2 \cap \{\psi > 0\}\).
Since \(u\) is smooth in \(\omega_2\) and \(\varphi_{\psi} * u = u\) in \(\omega_2 \cap \{\psi = 0\}\), we also have that \(\varphi_{\psi} * u\) is smooth in \(\omega_2 \cap \Int{\{\psi = 0\}}\).
We now take any point \(y\) in the remaining set \(\omega_2 \cap \partial{\{\psi = 0\}}\).
Since \(\psi(y) = 0 < d(y, \partial\omega)\), there exists a neighborhood \(\widetilde{\omega} \Subset \omega_2\) of \(y\) where \(\psi < d(\cdot, \partial\omega)\).
Then, the map \((x, z) \in \widetilde\omega \times \Ball^{m}  \mapsto u(x + \psi(x)z)\) is smooth.
By differentiation under the integral sign in \eqref{eqJEMS-AdaptiveConvolution}, we deduce that \(\varphi_{\psi}*u\) is smooth in \(\widetilde\omega\) and the conclusion follows.
\end{proof}

Using the Faà di Bruno formula, one derives pointwise estimates for \(\varphi_\psi * u\)\,:

\begin{lemma}
	\label{lemmaDensityNewConvolution}
	Suppose that \(u \in \Lebesgue^{\infty}(A)\).
	If \(\psi\) is smooth and positive in an open subset \(\omega\subset  A_\psi \) and if, for every \(j \in \N_*\), we have
	\begin{equation}
    \label{eqPointwiseConvolutionParameter}
    \abs{D^{j}\psi}
	\le C_j'' \psi^{1 - j} \quad \text{in \(\omega\),}
	\end{equation}
	then, for every \(j \in \N_*\),{}
	\[{}
	\abs{D^{j}(\varphi_{\psi} * u)}
	\le C_j''' \psi^{- j} \quad \text{in \(\omega\).}
	\]
\end{lemma}

\resetconstant
\begin{proof}[Proof of Lemma~\ref{lemmaDensityNewConvolution}]
	Denoting \(G(x, t) = \varphi_{t} * u (x)\), we have \((\varphi_{\psi} * u)(x) = G(x, \psi(x))\).{}
	Then, by the Faà di Bruno formula, for every \(x\in \omega\),
\begin{multline}
\label{eqDensityNew-494}
\abs{D^{j}(\varphi_{\psi} * u)(x)}\\
\leq \C \sum_{l = 1}^{j} \sum_{a + b = l} \abs{D^{a}_x \partial_{t}^{b} G(x, \psi(x))} \sum_{s_1+\dots +s_{b}=j-a}\abs{D^{s_1}\psi(x)}\cdots\abs{D^{s_{b}}\psi(x)}.
\end{multline}
Note that
\[{}
D_{x}G(x, t)
= - \frac{1}{t} (D\varphi)_{t} * u(x)
\quad \text{and} \quad 
\partial_{t}G(x, t)
= - \frac{1}{t} \widetilde\varphi_{t} * u(x),
\]
where \(\widetilde\varphi(z) \vcentcolon= D\varphi(z)[z] + m \varphi(z)\).{}
Proceeding by induction, one deduces from the boundedness of \( u \) that, for every \( a, b \in \N \),
\begin{equation}
\label{eqDensityNew-511}
\abs{D^{a}_x \partial_{t}^{b} G(x, t)}
\le \frac{\C}{t^{a + b}}.
\end{equation}
The conclusion then follows from the combination of \eqref{eqPointwiseConvolutionParameter}, \eqref{eqDensityNew-494} and \eqref{eqDensityNew-511}. 
\end{proof}

\resetconstant
\begin{proof}[Proof of Proposition~\ref{propositionAdaptiveApplication-New}]
Let \(O\Subset A\). 
Our goal is to rely on the adaptive smoothing and take \(v\) of the form
\begin{equation*}
    v = \varphi_{\epsilon\psi} * u,
    \end{equation*}
	where \( \varphi \) is a mollifier, \(0<\epsilon<1\) and \( \psi \colon \R^m \to \R \) is a suitably chosen nonnegative Lipschitz function such that \(O \subset A_{\epsilon\psi}\)\,.
    More precisely, we introduce a parameter \(\delta<\overline{\delta}<2\delta\) and take \(\psi\) smooth on \(\R^m \setminus F_0\) and supported in \(F_1 + B_{2\delta}^{m}\) such that
	\begin{alignat}{2}
    \label{eqStrong-898}
    \psi & \leq d(\cdot,F_0) 
     &&\quad \text{in \( \R^m\),}\\
    \label{eqStrong-897}
    \psi & \geq c\,d(\cdot,F_0) 
     &&\quad \text{in \( F_1 + B_{\bar{\delta}}^{m} \)}	
    \end{alignat}
    for some constant \( 0 < c < 1 \). 
    We also require that, for every \( j \in \N_{*} \),
\begin{equation}\label{eqJEMS-846}
    \abs{D^{j}\psi}
	\le \Cl{cteStrong-731} \psi^{1 - j}
    \quad \text{in \(((F_1 + B_{\bar{\delta}}^{m})\setminus F_0)\cap A \),}
\end{equation}
for some constant \( \Cr{cteStrong-731} > 0 \) depending on \( j \).

To construct such a function \(\psi\), we start from a regularized distance \(\psi_0\) to \(F_0\), see \cite{Stein1970}*{Theorem~2, p.~171}, that is, \(\psi_0 \colon \R^m\to \R\) is a nonnegative smooth function on \(\R^m\setminus F_0\) which satisfies, for every \(x\in \R^m\setminus F_0\) and every \(j\in \N_*\), 
\[
c_1 d(x,F_0)\leq \psi_0(x)\leq c_2 d(x,F_0)
\quad \text{and} \quad 
\abs{D^j\psi_0(x)}\leq \widetilde C_j d(x,F_0)^{1-j}.
\]
Here, the constants \(c_1, c_2 > 0\) and \(\widetilde C_j> 0\) are independent of \(F_0\). Up to dividing \(\psi_0\) by \(c_2\), we can assume that \(c_2=1\). To get the function \(\psi\), we  multiply \(\psi_0\)  by a smooth cut-off function equal to \(1\) on \(F_1+ B_{\bar{\delta}}^{m}\) that vanishes outside \(F_1+B_{2\delta}^{m}\). 
The estimate \eqref{eqJEMS-846} then follows from the Leibniz formula.

Let \(M>0\) be the maximum of the continuous function \(d(\cdot,F_0)\) on the compact set \(\overline{A}\) and take \(0 < \epsilon_0 \le 1/2\) such that 
\begin{equation*}
    O+B^{m}_{M\epsilon_0} \Subset A.
\end{equation*}
Since \(\psi\leq M\) on \(A\), this implies that \(O\subset A_{\epsilon_0\psi}\)\,.

We proceed to prove that the conclusion of Proposition~\ref{propositionAdaptiveApplication-New} holds true for any \(0<\epsilon\leq \epsilon_0\) sufficiently small. 
Using that \(\psi\) is positive on \(\overline{F_1 + B_{\delta}^{m} }\setminus F_0\), one has the inclusion
\begin{equation}
    \label{eqJEMS-851}
    A_{\epsilon\psi} \setminus F_0
    \subset \{\psi>0\} \cup ( A \setminus \overline{F_1 + B_{\delta}^m} ).
    \end{equation}
Since \(\psi\) is smooth outside \(F_0\) and \( u \) is smooth in \( A \setminus \overline{F_1 + B_{\delta}^m} \), by \eqref{eqJEMS-851} and Lemma~\ref{lemma-adaptative-smooth} we deduce that \( \varphi_{\epsilon\psi} * u \) is smooth on \(A_{\epsilon\psi} \setminus F_0 \).
In addition, since \(\psi\) is supported in \(F_1 + B_{2\delta}^{m}\)\,, 
    \[
    \varphi_{\epsilon\psi} * u = u
    \quad \text{in \(A_{\epsilon\psi} \setminus (F_1 + B_{2\delta}^{m})\),}
\]
which implies assertion~\((\ref{itemJEMS-745a})\) as \(O\subset A_{\epsilon_0\psi}\)\,.

    Next, to prove the uniform estimate in \((\ref{itemJEMS-746})\), we first observe that, for every \(x \in A_{\epsilon\psi} \setminus F_0\) and every \(z \in \Ball^{m} \), by \eqref{eqStrong-898} we have
\[
d(x+\epsilon \psi(x)z,F_0)\geq d(x,F_0)-\abs{\epsilon \psi(x)z}\geq d(x,F_0)-\epsilon\abs{ \psi(x)} \geq (1-\varepsilon)d(x,F_0).
\]
Hence, by \eqref{eqJEMS-709} and \eqref{eqStrong-898},
	\[{}
	\bigabs{u(x + \epsilon \psi(x) z) - u(x)}
\le C \frac{\epsilon\psi(x) \abs{z}}{(1-\epsilon)d(x, F_0)}
	\le C \frac{\epsilon}{1-\epsilon}.
	\]
This implies that
    \[
    \abs{\varphi_{\epsilon\psi} * u(x) - u(x)}
	\le C \frac{\epsilon}{1-\epsilon}.
    \]
   As \(0 < \epsilon \le \epsilon_0 \le 1/2\), the right-hand side is bounded by \(2C\epsilon\) and \((\ref{itemJEMS-746})\) follows in this case.

To prove assertion~\((\ref{itemJEMS-propositionAdaptiveApplication-last})\), first note that the counterpart of \eqref{eqJEMS-846} with \(\psi\) replaced by \(\epsilon\psi\) is also satisfied since \(\epsilon \le 1\).
By Lemma~\ref{lemmaDensityNewConvolution} applied with \(\epsilon\psi\) and by \eqref{eqStrong-897}, one then gets
\begin{equation}\label{eq894}
	\abs{D^{j}(\varphi_{\epsilon\psi} * u)}
	\le \C (\epsilon\psi)^{-j}\le \C(\epsilon d(\cdot, F_0))^{-j}
    \quad 
    \text{in \(((F_1 + B_{\bar{\delta}}^{m}) \setminus F_0)\cap A_{\epsilon\psi} \)\,.}
\end{equation}
Before proceeding, we introduce a second parameter \(\delta<\underline{\delta}<\overline{\delta}\).
By reducing \(\epsilon_0\) if necessary, we can assume that \(\epsilon_0<(\overline{\delta}-\underline{\delta})/M\).
Since \( u \) and \(\psi\) are smooth and have bounded derivatives on the compact set \(\overline{O+B^{m}_{M\epsilon_0}}\setminus (F_1+B^{m}_{\underline{\delta}})\), we get from the definition~\eqref{eqJEMS-AdaptiveConvolution} of the adaptative convolution that
\begin{equation}\label{eq8944}
	\abs{D^{j}(\varphi_{\epsilon\psi} * u)}
	\le \Cl{cteJEMS-952} 
    \quad 
    \text{in \(\bigl(\overline{O+B^{m}_{M\epsilon_0}}\setminus (F_1+B^{m}_{\underline{\delta}})\bigr)_{\epsilon\psi}\)\,,}
\end{equation}
for some constant \(\Cr{cteJEMS-952} > 0\) depending on \(j\).
Using that \(\epsilon<(\overline{\delta}-\underline{\delta})/M\), we have  
\[
\bigl((F_1+B^{m}_{2\delta})\setminus (F_1+B^{m}_{\overline{\delta}})\bigr) \cap O 
\subset \bigl(\overline{O+B^{m}_{M\varepsilon_0}}\setminus (F_1+B^{m}_{\underline{\delta}})\bigr)_{\epsilon\psi}\,.
\]
The estimate in \((\ref{itemJEMS-propositionAdaptiveApplication-last})\) then follows from \eqref{eq894} and \eqref{eq8944}.

We thus have the conclusion for \(0 < \epsilon \le \epsilon_0\).
To establish Proposition~\ref{propositionAdaptiveApplication-New} when \(\epsilon_0 \le \epsilon \le 1\), it suffices to apply the previous case with \(\epsilon_0\).
\end{proof}

The main properties of adaptive convolution in the setting of \(\Sobolev^{k,p}\) spaces are summarized in the following statement, see \cite{BPVS_MO}*{Proposition~3.2}:

\begin{proposition}
\label{propositionConvolutionEstimates}
 Let \(\varphi\) be a mollifier, and let \(\psi\colon A \to \R\) be a nonnegative smooth function such that \(\norm{D\psi}_{\Lebesgue^{\infty}(A)}<1\). 
 For every \(u\in \Sobolev^{k,p}(A)\), we have \(\varphi_\psi*u\in \Sobolev^{k,p}(A_\psi)\) and, for every \(j\in \{1, \dots, k\}\),
 \[
\norm{D^j(\varphi_{\psi} \ast u) }_{\Lebesgue^p(A_\psi)}
\leq  \frac{C}{(1 - \norm{D\psi}_{\Lebesgue^\infty(A_\psi)})^{1/p}} \sum_{i=1}^j  \eta^{i - j} \norm{D^i u}_{\Lebesgue^p(A)}
\]
and
\begin{multline*}
\norm{D^j(\varphi_{\psi} \ast u) - D^j u}_{\Lebesgue^p(A_\psi)}\\
\leq \sup_{v \in B_1^m}{\norm{\tau_{\psi v}(D^j u) - D^j u}_{\Lebesgue^p(A_\psi)}} + \frac{C'}{(1 - \norm{D\psi}_{\Lebesgue^\infty(A_\psi)})^{1/p}} \sum_{i=1}^j  \eta^{i - j} \norm{D^i u}_{\Lebesgue^p(B)}\text{,}
\end{multline*}
where \(B \vcentcolon= \bigcup\limits_{x\in A_\psi\cap \supp D\psi}{B^{m}_{\psi(x)}(x)}\), \((\tau_{\psi v} D^j u) (x) \vcentcolon= D^ju(x + \psi(x)v)\) and \(\eta>0\) is such that, for every \(j \in \{2, \dotsc, k\}\), 
\[
  \norm{D^j \psi}_{\Lebesgue^\infty(A_\psi)} \le \eta^{1-j}.  
\]
\end{proposition}

\section{Smooth extensions}

Our next goal is to remove the singularity that is created by thickening, or when that is not possible due to a topological obstruction, to decrease the dimension of the singular set. 
We wish to construct an \( \ClassR_i \)~map with \(i\) as small as possible.
To illustrate the strategy that we pursue in this section, we begin with a model example that involves a continuous map with values in a manifold \( \manfN \), where we can fully remove the singularity.

\begin{example}
    \label{exampleSingularityBallRemovable}
    Suppose that we are given \( u \in \Smooth^{0}(\cBall^{m} \setminus \{0\}; \manfN) \) and we know that \( u|_{\partial \Ball^{m}} \) has a continuous extension from \( \cBall^{m} \) to \( \manfN \).
    Then, for any \( 0 < r < 1 \), one can find \( v \in \Smooth^{0}(\cBall^{m}; \manfN) \) such that
    \begin{equation*}
    v = u
    \quad \text{in \(\cBall^{m} \setminus B_{r}^{m}\).}
    \end{equation*}
    To this end, denote by \( \overline{u} \) a continuous extension of \( u|_{\partial \Ball^{m}} \) from \( \cBall^{m} \) to \( \manfN \).
    Then, it suffices to take
    \[
    v(x)
    = 
    \begin{cases}
    u(x)  & \text{if \(x \in \cBall^{m} \setminus B_{r}^{m}\),}\\
    \overline{u}(2x/r)       & \text{if \(x \in B_{r/2}^{m}\)\,,}
    \end{cases}
    \]
    and to extend it continuously on \(\overline{B}{}_{r}^{m} \setminus B_{r/2}^{m}\) using the values of \( u \) on \( \cBall^{m} \setminus B_{r}^{m}\).
\end{example}

We pursue this example in greater generality in the context of cubications.
Given a cubication \( \cS^{m} \) in \( \R^{m} \) and the dual skeleton \( \cT^{\ell^{*}} \) of \( \cS^{\ell} \) with \( \ell \in \{0, \dots, m\} \), we rely on the existence of a continuous retraction from \(S^m \setminus T^{\ell^{*}}\) to \(S^{\ell}\) which is homotopic to the identity map in \(S^{m} \setminus T^{\ell^{*}}\).{}
More precisely, we work with a locally Lipschitz homotopy \(H_\ell  \colon  (S^m \setminus T^{\ell^*}) \times [0, 1]  \to S^m \setminus T^{\ell^*}\) such that 
\begin{enumerate}[$({a}_1)$]
\item for every \(y \in S^m\setminus T^{\ell^*}\), \(H_\ell(y, 0) = y\),
\label{item15431}
\item for every \(y \in S^m \setminus T^{\ell^*}\), \(H_\ell(y, 1) \in S^\ell\),
\label{item15432}
\item for every \(y \in S^\ell\), \(H_\ell(y, 1) = y\).
\label{item15433}
\end{enumerate}
This map \( H_\ell(\cdot, t) \) can be constructed taking the map \( F_{1, 1-t} \) defined on \( S^{m} \) by White~\cite{White-1986}*{p.\,130} for \( 0 \leq  t < 1 \) and, in his notation, \(k = \ell\). 
When one removes the dual skeleton \( T^{\ell^{*}} \) from White's construction, one obtains by continuous extension a locally Lipschitz map on \( S^m\setminus T^{\ell^*} \) that includes \( t = 1 \).
In this case, \( F_{1, 0} \) is the usual retraction \(P_\ell \colon S^{m}\setminus T^{\ell^*}\to S^{\ell}\).

To focus on the topological aspect of the construction, we begin by dealing with  the following generalization of Example~\ref{exampleSingularityBallRemovable} for continuous maps on cubications:

\begin{proposition}
	\label{propositionDensityNewExtension}
    Let \( e \in \{ \ell, \ldots, m \} \),
    let \( \eta > 0 \) be the radius of the cubication \( \cS^{m} \), and let \( 0 < \mu < 1 \).  
    If \(u\in \Smooth^{0}(S^m\setminus T^{\ell^*};\manfN)\) is a map whose restriction \( u|_{S^{\ell}} \) has a continuous extension from \( S^{e} \) to \( \manfN \), then there exists \(v \in \Smooth^{0}(S^{m} \setminus T^{e^{*}}; \manfN)\) such that 
\begin{equation}
\label{eqStrong2-634}
    v= u
    \quad \text{in \(S^{m} \setminus (T^{\ell^{*}} + Q_{\mu\eta}^{m})\).}
\end{equation}
\end{proposition}	

\begin{proof}
    We denote by \( \overline{u}\) a continuous extension of \( u|_{S^{\ell}} \) from \( S^{e}\) to \(\manfN\).
    Given \(0 < \underline\delta < \delta < \mu\), we take \(v  \colon  S^m \setminus T^{e^{*}} \to \manfN\) as
\begin{equation}
\label{eqStrong2-614}
v(x) = 
\begin{cases}
(u \circ H_\ell)(x, \theta(x))		& \text{if \(x \in S^m \setminus (T^{\ell^*} + Q_{\delta\eta}^m)\),}\\
(\overline{u} \circ P_{e} \circ H_\ell)(x, \theta(x))		& \text{if \(x \in (T^{\ell^*} + Q_{\delta\eta}^m) \setminus (T^{\ell^*} + Q_{\underline{\delta}\eta}^m)\),}\\
(\overline{u} \circ P_{e})(x)		& \text{if \(x \in (T^{\ell^*} + Q_{\underline{\delta}\eta}^m) \setminus T^{e^{*}} \),}
\end{cases}
\end{equation}
where \(P_e \colon S^{m}\setminus T^{e^*}\to S^{e}\) is the usual retraction and \(\theta  \colon  S^m \to [0, 1]\) is a Lipschitz function such that 	
\begin{enumerate}[$({b}_1)$]
	\item
    \label{itemExtension610}
    \(\supp{\theta} \subset (T^{\ell^*} + Q_{\mu\eta}^m) \setminus \overline{T^{\ell^*} + Q_{\underline\delta \eta}^m}\),
    \item 
    \label{itemExtension607}
    \(\theta = 1\) in \(\bigl(\partial(T^{\ell^*} + Q_{\delta\eta}^m) \bigr) \cap S^m\).
\end{enumerate}

By \((a_{\ref{item15431}})\) and \((b_{\ref{itemExtension610}})\), we have
\[
v = u
\quad \text{in \(S^m \setminus (T^{\ell^*} + Q_{\mu\eta}^m)\).}
\]
To verify that \(v\) is continuous, observe that, by \((a_{\ref{item15431}})\) and \((b_{\ref{itemExtension610}})\), 
\[
v = \overline{u} \compose P_{e} 
\quad \text{in a neighborhood of \(\partial(T^{\ell^*} + Q_{\underline{\delta}\eta}^m)\) in \(S^m\).}
\]
Since \( u = \overline{u} \) on \( S^{\ell} \) and \( P_{e} \) coincides with the identity map on \( S^{e} \), we have \( u = \overline{u} \compose P_{e} \) on \( S^{\ell} \).
Then, by \((a_{\ref{item15432}})\),
\[
(u \compose H_{\ell})(\cdot, 1) 
= (\overline{u} \compose P_{e} \compose H_{\ell})(\cdot, 1) 
\quad \text{in \(S^{m} \setminus T^{\ell^{*}}\).}
\]
It then follows from \((b_{\ref{itemExtension607}})\) that \(v\) is continuous in a neighborhood of \(\partial(T^{\ell^*} + Q_{\delta\eta}^m)\) in \(S^m\).
This implies that \( v \) is continuous on \( S^m \setminus T^{e^{*}} \).
\end{proof}

We now investigate the smooth analogue of Proposition~\ref{propositionDensityNewExtension} that we need in the proof of Theorem~\ref{thm_density_manifold_open}.

\begin{proposition}
	\label{propositionDensityNewExtensionSmooth}
    Let \( e \in \{ \ell, \ldots, m \} \),
  let \( \eta > 0 \) be the radius of the cubication \( \cS^{m} \), and let \( 0 < \mu < 1 \).
  If \(u\in \Smooth^{0}(S^m\setminus T^{\ell^*};\manfN)\) is smooth and if \( u|_{S^{\ell}} \) has a continuous extension from \( S^{e} \) to \( \manfN \), then there exists \(v \in \ClassR_{e^{*}}(\Int{S^{m}}; \manfN)\) with singular set \( T^{e^{*}} \) such that  
  \[
    v= u
    \quad \text{in \((\Int{S^{m}}) \setminus (T^{\ell^{*}} + Q_{\mu\eta}^{m})\).}
  \]
\end{proposition}

Since the domain \(S^m\setminus T^{\ell^*}\) is not open, by saying that \(u\) is smooth we mean that there exist an open set \(A \supset S^m\setminus T^{\ell^*}\) and a smooth map \(\widetilde{u} \colon A \to \R^\nu\) such that \(\widetilde{u}|_{S^m\setminus T^{\ell^*}} = u\).

We begin with a known extension lemma in the context of Lipschitz maps with values in \(\manfN\). 
We rely here on the existence of some parameter \(\iota>0\) such that the nearest point projection \(\Pi\) to \(\manfN\) is well defined and smooth in a neighborhood of the tubular neighborhood \(\overline{\manfN+B^{\nu}_{\iota}}\) of \(\manfN\).

\begin{lemma}\label{lemma-Lipschitz-extension}
     Let \( e \in \{ \ell, \ldots, m \} \).
 If \(u \in \Smooth^0(S^\ell; \manfN)\) has a continuous extension from \(S^e\) to \(\manfN\) and if \(u\) is Lipschitz continuous, then \(u\) has a Lipschitz continuous extension from \(S^e\) to \(\manfN\).  
\end{lemma}

\begin{proof}[Proof of Lemma~\ref{lemma-Lipschitz-extension}]
Let us denote by \(u_{c}\colon S^{e}\to \manfN\) a continuous extension  of  \(u\) and by  \(v\colon S^{e}\to \R^{\nu}\) a Lipschitz approximation of \(u_{c}\) such that, for every \(x\in S^e\),
        \begin{equation}\label{eq850}
            \abs{v(x)-u_{c}(x)}<\frac{\iota}{2}.
        \end{equation}
    Let \(u_{l} \colon \R^m \to \R^{\nu}\) be a Lipschitz continuous extension of \(u\).
    Denoting by \(P_\ell \colon S^m\setminus T^{\ell^*} \to S^\ell\)  the usual retraction, there exists a bounded open neighborhood \(U\subset \R^m \setminus T^{\ell^*}\) of \(S^\ell\) such that, for every \(x\in U \cap S^e\),
        \begin{equation}\label{eq855}
        \abs{u_{c}(x)-u_{c}(P_\ell(x))}<\frac{\iota}{2}
        \end{equation}
        and, for every \(x\in U\),
        \begin{equation}\label{eq859}
        \abs{u_{l}(x)-u_{l}(P_\ell(x))}<\frac{\iota}{2}.
        \end{equation}
        
        Let \(\eta\in \Smooth^{\infty}_c(\R^m; [0,1])\) be such that \(\eta = 1\) in \(S^\ell\) and \(\eta = 0\) in \(S^{e}\setminus U\). 
        We claim that, for every \(x\in S^e\), 
        \begin{equation}\label{eq860}
           d\bigl((1-\eta(x))v(x)+\eta(x)u_{l}(x), \manfN\bigr)<\iota. 
        \end{equation}
        Indeed, when \(x\in S^{e}\setminus U\), the left-hand side equals \(d(v(x),\manfN)\), which in turn is not larger than \(\abs{v(x)-u_{c}(x)}\).  
        We can thus conclude using \eqref{eq850} in this case.
        When instead \(x\in S^{e}\cap U\), the left-hand side of \eqref{eq860} is not larger than
\begin{multline*}
\bigl|(1-\eta(x))v(x)+\eta(x)u_{l}(x)- u(P_\ell(x))\bigr| \\
\leq (1-\eta(x))\abs{v(x)-u(P_\ell(x))} + \eta(x)\abs{u_{l}(x)-u(P_\ell(x))}.
\end{multline*}
Using \(u_{c}(P_{\ell}(x))=u(P_{\ell}(x))\) and the estimates \eqref{eq850} and \eqref{eq855}, we get
\[
\abs{v(x)-u(P_\ell(x))}
\leq \abs{v(x)-u_{c}(x)}+\abs{u_{c}(x)-u_{c}(P_{\ell}(x))}
<\iota.
\]
Based on identity \(u_{l}(P_{\ell}(x))=u(P_\ell(x))\) and the estimate \eqref{eq859},
\[
\abs{u_{l}(x)-u(P_\ell(x))}=\abs{u_{l}(x)-u_{l}(P_\ell(x))}<\iota.
\]
It follows that
\[
\bigl|(1-\eta(x))v(x)+\eta(x)u_{l}(x)- u(P_\ell(x))\bigr| <\iota
\]
and then \eqref{eq860} also holds in the case \(x\in S^{e}\cap U\). 
We can thus define
\[
\overline{u}\colon x\in S^e \longmapsto \Pi\bigl((1-\eta(x))v(x)+\eta(x)u_{l}(x)\bigr).
\]
Then, \(\overline{u}\) is a Lipschitz extension of \(u\) from \(S^e\) to \(\manfN\). 
\end{proof}

We also need the following variant of the McShane-Whitney extension lemma for Lipschitz functions:

\begin{lemma}
\label{lemmaMcShane-Whitneysingular}
Let \( e \in \{0, \ldots, m \} \).
If \(v \in \Smooth^0(S^m\setminus T^{e^*})\) is a bounded function such that, 
for every \(y,z \in S^m\setminus T^{e^*}\),
\begin{equation}\label{eq59}
|v(y)-v(z)|
\leq C\frac{|y-z|}{\min{\bigl\{ d(y,T^{e^*}),d (z,T^{e^*}) \bigr\}}} \text{,}
\end{equation}
then \(v\) has an extension to \(\R^m\setminus T^{e^*}\) which satisfies this estimate for every \(y, z \in \R^m\setminus T^{e^*}\), involving some constant that depends on \(C\) and on the sup norm \(\norm{v}_{\infty}\).
\end{lemma}

\begin{proof}[Proof of Lemma~\ref{lemmaMcShane-Whitneysingular}]
We begin by showing that \(f \colon S^m\setminus T^{e^*} \to \R\) defined by
\[
f(x) = v(x)d(x, T^{e^*})
\]
is Lipschitz continuous.
To this end, let \(y, z \in S^m\setminus T^{e^*}\). 
Assume, for example, that \(d(z,T^{e^*})\leq d(y,T^{e^*})\). 
By the triangle inequality and the Lipschitz continuity of \(d(\cdot, T^{e^*})\), we have 
\begin{align*}
|f(y)-f(z)|&\leq |v(y)||d(y,T^{e^*})-d(z,T^{e^*})|+d(z,T^{e^*})|v(y)-v(z)|\\
&\leq \|v\|_{\infty}|y-z|+d(z,T^{e^*})|v(y)-v(z)|.
\end{align*}
By \eqref{eq59}, this gives
\[
|f(y)-f(z)|\leq (\|v\|_{\infty}+C)|y-z|.
\]
We can then extend \(f\) to \(\R^m\) as a Lipschitz continuous function of rank \(\|v\|_{\infty}+C\) that we denote by \(f_l\).
Since \(|f|\leq \|v\|_{\infty}d(\cdot, T^{e^*})\) in \(S^m\setminus T^{e^*}\), the function \(f_l\) vanishes on \(T^{e^*}\) and thus satisfies, for every \(y\in \R^m\),
\begin{equation}\label{eq84}
|f_l(y)|\leq (\|v\|_{\infty}+C)d(y, T^{e^*}).
\end{equation}
We define \(\overline{v} \colon \R^m \setminus T^{e^*} \to \R\) by
\[
\overline{v}(x) = \frac{f_l(x)}{d(x, T^{e^*})}.
\]
Then, \(\overline{v}\) coincides with \(v\) in \(S^m\setminus T^{e^*}\).
Let us show that \(\overline{v}\) is locally Lipschitz outside \(T^{e^*}\) and satisfies \eqref{eq59} for every \(y, z \in \R^m\setminus T^{e^*}\), with a different constant. 
Indeed, assume for instance that \(d(z,T^{e^*}) \leq d(y,T^{e^*})\). 
Then, we write
\begin{align*}
|\overline{v}(y)-\overline{v}(z)|&= \frac{|f_l(y)d(z,T^{e^*})- f_l(z)d(y,T^{e^*})|}{d(z,T^{e^*})d(y,T^{e^*})}\\
&\leq |f_l(y)|\frac{|d(z,T^{e^*})-d(y, T^{e^*})|}{d(z,T^{e^*})d(y,T^{e^*})}+ d(y,T^{e^*})\frac{|f_l(y)-f_l(z)|}{d(z,T^{e^*})d(y,T^{e^*})}.
\end{align*}
Using the Lipschitz continuity of \(d(\cdot,T^{e^*})\) and \(f_l\)\,, one gets
\[
|\overline{v}(y)-\overline{v}(z)|
\leq |f_l(y)|\frac{|y-z|}{d(z,T^{e^*})d(y,T^{e^*})}+(\|v\|_{\infty} + C)\frac{|y-z|}{d(z,T^{e^*})}.
\] 
In view of \eqref{eq84}, this yields
\[
|\overline{v}(y)-\overline{v}(z)|\leq 2(\|v\|_{\infty}+C)\frac{|y-z|}{d(z,T^{e^*})},
\]
which is the desired estimate.
\end{proof}

The above Lipschitz extension is used in the following application of adaptive smoothing:

\begin{lemma}
	\label{propositionAdaptiveApplication}
     Let \( e \in \{ \ell, \ldots, m \} \),
     let \( \cS^{m} \) be a cubication in \( \R^{m} \) of radius \( \eta > 0 \), let \(\cT^{\ell^*}\) and \(\cT^{e^*}\) be the dual skeletons of \(\cS^\ell\) and \(\cS^e\) respectively, and let \(0 < \mu < 1/2\).
 If  \(v  \in \Smooth^0(S^{m} \setminus T^{e^{*}} ; \manfN)\) is smooth in \(S^m \setminus (T^{\ell^*} + Q^{m}_{\mu\eta/4})\) and satisfies, for every \(y, z \in  S^{m} \setminus T^{e^{*}}\), the locally Lipschitz estimate
	\begin{equation}\label{eq709}{}
	\abs{v(y) - v(z)}
	\le C' \frac{\abs{y - z}}{\min{\bigl\{d(y, T^{e^{*}}), d(z, T^{e^{*}})\bigr\}}} \text{,}
	\end{equation}
then there exists \(\widetilde v  \in \ClassR_{e^*}(\Int{S^{m}} ; \manfN)\) with singular set \( T^{e^{*}}\) such that 
\[
\widetilde v = v
\quad \text{in \((\Int{S^m}) \setminus (T^{\ell^*} + Q^{m}_{\mu\eta})\).}
\]
\end{lemma}

\resetconstant
\begin{proof}[Proof of Lemma~\ref{propositionAdaptiveApplication}]
    We first show that \(v\) has a locally Lipschitz extension \(\overline{v} \colon \R^m \setminus T^{e^{*}} \to \R^\nu\) that is smooth in \(\R^m\setminus (T^{\ell^*} + Q^{m}_{\mu\eta/2})\) and satisfies estimate \eqref{eq709} for every \(y, z \in  \R^m \setminus T^{e^{*}}\).
    Since \(v\) is smooth in the closed set \(S^m \setminus (T^{\ell^*} + Q^{m}_{\mu\eta/4})\), its restriction to this set has a smooth extension \(v_s \colon \R^m \to \R^\nu\).
    To handle \(v\) in a neighborhood of \(((T^{\ell^*} + Q^{m}_{\mu\eta/4}) \setminus T^{e^*}) \cap S^m\), we rely on Lemma~\ref{lemmaMcShane-Whitneysingular} applied to each component of \(v\) to get a locally Lipschitz map \(v_c \colon \R^m\setminus T^{e^*} \to \R^\nu\) which agrees with \(v\) on \(S^m\setminus T^{e^*}\) and satisfies \eqref{eq709} for every \(y, z \in \R^m\setminus T^{e^*}\), possibly with a different constant.
    
Let \(\zeta\in \Smooth^{\infty}_c(\R^m)\) supported in \(T^{\ell^*}+Q^{m}_{\mu\eta/2}\) and such that \(\zeta = 1\) in \(T^{\ell^*}+Q^{m}_{\mu\eta/4}\). 
We then set
\[
\overline{v} = (1-\zeta) v_s+ \zeta v_c\, .
\]
    Since  \(v_s = v\) in \(S^m \setminus (T^{\ell^*}+Q^{m}_{\mu\eta/4})\) and  \(v_c=v\) in \(S^m\setminus T^{e^*}\), it follows that 
    \[
    \overline{v} = v
    \quad \text{in \(S^m \setminus T^{e^*}\).}
    \]
    Moreover, \(\overline{v}\) is smooth in 
    \(\R^m \setminus (T^{\ell^*} + Q^{m}_{\mu\eta/2}) \).
    Since \(v_s\) is locally Lipschitz continuous in \(\R^m\) and \(v_c\) satisfies \eqref{eq709} in \(\R^m \setminus T^{e^*}\), we deduce that \(\overline{v}\) is locally Lipschitz continuous and also verifies estimate \eqref{eq709}.
    The map \(\overline{v}\) thus has the required properties.

    We now apply an easy variant of Proposition~\ref{propositionAdaptiveApplication-New} where neighborhoods are defined using cubes instead of balls.
    Taking \(F_0 = T^{e^*}\), \(F_1 = T^{\ell^*}\), \(\delta = \mu\eta/2\), and open sets \(O \Subset A\) containing \(S^m\), we obtain a smooth function \(\widehat{v} \colon O \setminus T^{e^*} \to \R^\nu\) such that 
\begin{enumerate}[\((i)\)]
    \item \(\widehat{v} = \overline{v}\) on \(O \setminus (T^{\ell^*} + Q^{m}_{\mu\eta})\),
    \item 
    \label{itemJEMS-1512}
    for every \(x \in O \setminus T^{e^*}\), we have \(\abs{\overline{v}(x) - \widehat{v}(x)} \le \epsilon\),
    \item 
    for every \(j \in \N_{*}\) and every \(x \in ((T^{\ell^*} + Q^{m}_{\mu\eta}) \setminus T^{e^*}) \cap O\),
\[
\abs{D^{j}\widehat{v}(x)}
\le \frac{C_j''}{\epsilon^j d(x, T^{e^*})^j}.
\]
\end{enumerate}
    Since \(\overline{v} = v\) on \(S^m \setminus T^{e^*}\) and the image of \(v\) is contained in \(\manfN\), from \((\ref{itemJEMS-1512})\) we have
    \[
    \widehat{v}(S^m \setminus T^{e^*})
    \subset \manfN + B_{\epsilon}^\nu\,.
    \]
    We choose \(\epsilon > 0\) such that the nearest point projection \(\Pi \colon \overline{\manfN + B_\epsilon^\nu} \to \manfN\) is well defined and smooth.
    It then suffices to take \(\widetilde v = \Pi \compose \widehat{v}|_{(\Int{S^m}) \setminus T^{e^*}}\).
\end{proof}

\resetconstant
\begin{proof}[Proof of Proposition~\ref{propositionDensityNewExtensionSmooth}]
We begin with a local Lipschitz counterpart of Proposition~\ref{propositionDensityNewExtension} applied with \(\mu/4\) instead of \(\mu\).
    The standard retraction \(P_e \colon S^{m}\setminus T^{e^*}\to S^{e}\) satisfies, for every \( y, z \in  S^{m} \setminus T^{e^{*}} \),
	\[{}
	\abs{P_{e}(y) - P_{e}(z)}
	\le \C \frac{\abs{y - z}}{\min{\bigl\{d(y, T^{e^{*}}), d(z, T^{e^{*}})\bigr\}}},
	\]
    which can be proved using the estimate \( \abs{DP_{e}(x)} \le \C/d(x, T^{e^{*}})\) for \( x \in S^{m} \setminus T^{e^{*}} \), see \cite{White-1986}*{p.\,130}.
    Let \(\overline{u}:S^e\to \manfN\) be a Lipschitz continuous extension of \(u|_{S^\ell}\), whose existence is provided by Lemma~\ref{lemma-Lipschitz-extension}. 
    For every \(y, z \in  S^{m} \setminus T^{e^{*}}\), we then have
\[{}
	\bigabs{(\overline{u} \compose P_{e})(y) - (\overline{u} \compose P_{e})(z)}
	\le \C \frac{\abs{y - z}}{\min{\bigl\{d(y, T^{e^{*}}), d(z, T^{e^{*}})\bigr\}}}.
\]
 
    We recall that the map \( v \) defined by \eqref{eqStrong2-614} is continuous and satisfies \eqref{eqStrong2-634} with \(\mu/4\) instead of \(\mu\). 
    Observe that \(H_{\ell}\) is Lipschitz on \(K \vcentcolon= \bigl( S^{m} \setminus (T^{\ell^*} + Q_{\underline\delta\eta}^m) \bigr) \times [0, 1]\) and \(H_{\ell}(K)\) is a compact subset of \(S^{m} \setminus T^{\ell^{*}} \subset S^{m} \setminus T^{e^{*}}\).{}
Thus, \(v\) is Lipschitz continuous on \(S^{m} \setminus (T^{\ell^*} + Q_{\underline\delta\eta}^m)\).{}
Since \(v = \overline{u} \compose P_{e}\) on \((T^{\ell^*} + Q_{\underline\delta\eta}^m) \setminus T^{e^{*}}\), we deduce that \(v\) satisfies 
	\[{}
	\abs{v(y) - v(z)}
	\le \C \frac{\abs{y - z}}{\min{\bigl\{d(y, T^{e^{*}}), d(z, T^{e^{*}})\bigr\}}}.
	\]
The conclusion now follows from Lemma~\ref{propositionAdaptiveApplication}.
\end{proof}

\section{Shrinking}
\label{section_shrinking}

Proposition~\ref{propositionDensityNewExtensionSmooth} does not address the question of how to get a convenient control of the \(\Sobolev^{k, p}\) norm of \(v\), as the topological assumption does not contain quantitative information on the derivatives of \(v\).
This goal can be achieved by using a follow-up of the thickening tool, which we call shrinking.
The idea is to confine the modifications on the map \( u \) to a small neighborhood of the dual skeleton.

To illustrate the idea, we take a smooth function \(u \colon \cBall^{m} \to \R\).
We wish to construct a new one of the form \(u \compose \Phi^{\mathrm{sh}}\) that dramatically reduces the contribution of the derivatives of \(u\) in a neighborhood of \(0\) when we compute the \(\Sobolev^{k, p}\)~norm of \(u \compose \Phi^{\mathrm{sh}}\).{}
More precisely, given \(0 < r < \rho < 1\), for every \(\epsilon > 0\) we look for a smooth map \(\Phi^{\mathrm{sh}}  \colon  \cBall^{m} \to \cBall^{m}\) such that
\begin{enumerate}[\((\mathrm{sh}_1)\)]
	\item{} 
	\label{itemJEMS-601}\(\Phi^{\mathrm{sh}} = \Id\) in \(\cBall^{m} \setminus B_\rho^m\),{}
	\item{}
	\label{itemJEMS-604-bis}
    \(\norm{u \compose \Phi^{\mathrm{sh}}}_{\Sobolev^{k, p}(\Ball^{m})}
	\le	C\norm{u}_{\Sobolev^{k, p}(\Ball^{m} \setminus \overline B{}_{r}^{m})} + \epsilon \norm{u}_{\Sobolev^{k, p}(B_{r}^{m})}\),
\end{enumerate}
for some constant \(C > 0\) independent of \(u\) and \(\epsilon\).

Shrinking can be thought of as a desingularized version of thickening that requires careful estimates near the origin.
The argument is based on a variant of the construction of Proposition~\ref{propositionJEMSThickeningModel}, which is recovered as \(\epsilon \to 0\). {}

As we did for thickening when \(k = 1\), when one is satisfied with having a Lipschitz function instead of a smooth one it is possible to take \(\Phi^{\mathrm{sh}}\) of the form
\[{}
\Phi^{\mathrm{sh}}(x)
=
\begin{cases}
	x	& \text{if \(x \in \cBall^{m} \setminus B_{t}^{m}\),}\\
	t {x}/{\abs{x}} & \text{if \(x \in B_{t}^{m} \setminus B_{\tau}^{m}\),}\\
	t {x}/{\tau} & \text{if \(x \in B_{\tau}^{m}\),}\\
\end{cases}
\]
where \(\tau\) and \(t\) are such that \(0 < \tau < r < t < \rho < 1\).
Picking \(t\) that satisfies \eqref{eqJEMS-510}, one gets
\resetconstant
	\begin{multline*}
    \int_{\Ball^{m}}{(\abs{u \compose \Phi^{\mathrm{sh}}}^{p} + \abs{D (u\compose \Phi^{\mathrm{sh}}) }^{p})}\\
	\le \C \int_{\Ball^{m} \setminus \overline B{}_r^m}{(\abs{u}^{p} + \abs{Du}^{p})}
    + \Bigl( \frac{\tau}{t} \Bigr)^{m - p} \int_{B_t^{m}}{(\abs{u}^{p} + \abs{Du}^{p})}.
	\end{multline*}
Since \(\tau /t \le 1\), when \(p < m\) one can choose \(\tau\) sufficiently small that \((\mathrm{sh}_{\ref{itemJEMS-604-bis}})\) holds.

\begin{figure}
\centering{}
\includegraphics{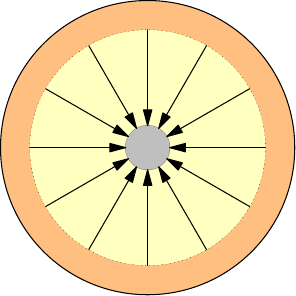}
\caption{Values of \(u \compose \Phi^{\mathrm{sh}}\) on the smaller disk (gray) are computed using \(u\) mostly from the larger yellow region}
\label{figureJEMSShrinkingBall}
\end{figure}

A smooth counterpart of this construction that is valid for every \( k \in \N_{*} \) can be achieved using the following proposition that allows one to take \(\Phi^{\mathrm{sh}} \) as a diffeomorphism independent of \(u\).

\begin{proposition}
	\label{propositionJEMSShrinkingModel}
	Given \(0 < \tau < r < \rho < 1\), there exists a diffeomorphism \(\Phi^{\mathrm{sh}}  \colon  \cBall^{m} \to \cBall^{m}\) with \(\Phi^{\mathrm{sh}}(B_{\tau}^{m}) = B_{r}^{m}\)  such that \(\Phi^{\mathrm{sh}} = \Id\) in \(\cBall^{m} \setminus B_\rho^m\) and
	\begin{enumerate}[\((i)\)]
	\item \label{propositionJEMSShrinkingModel-i} for every \(j \in \N_*\), there exists \(C' > 0\) depending on \(j\), \(m\), \(r\) and \(\rho\) such that
	\[
	\abs{D^j \Phi^{\mathrm{sh}}(x)} \le  \frac{C'}{(\abs{x}^{2} + \tau^{2})^{{j}/{2}}}
	\quad \text{for every \(x \in \Ball^{m}\),}	
	\]
	\item \label{propositionJEMSShrinkingModel-ii} for every \(\beta < m\), there exists \(C'' > 0\) depending on \(\beta\), \(m\), \(r\) and \(\rho\)  such that
	\[
	\Jacobian{m}{\Phi^{\mathrm{sh}}}(x) \ge \frac{C''}{\abs{x}^\beta} 
	\quad \text{for every \(x \in \Ball^{m} \setminus B_{\tau}^{m}\),}
	\]
	\item \label{propositionJEMSShrinkingModel-iii} there exists \(C''' > 0\) depending on \(m\), \(r\) and \(\rho\)  such that
	\[
	\Jacobian{m}{\Phi^{\mathrm{sh}}}(x) \ge \frac{C'''}{\tau^{m}}
	\quad \text{for every \(x \in B_{\tau}^{m}\).}
	\]
	\end{enumerate}
\end{proposition}

\resetconstant
\begin{proof}
	Given \(\gamma > 0\), we take \(\Phi^{\mathrm{sh}}  \colon  \cBall^{m} \to \cBall^{m}\) of the form 
	\[{}
	\Phi^{\mathrm{sh}}(x) = \varphi\bigl(\sqrt{\abs{x}^{2} + \gamma \tau^{2}}\bigr) \, x
	\]
	where \(\varphi \colon (0,\infty)\to (0,\infty)\) is non-increasing and  such that \(\varphi|_{(0, 1]}\) satisfies properties \eqref{voisinage1}--\eqref{deriveeminoreevarphi} in the proof of Proposition~\ref{propositionJEMSThickeningModel}. 
    We also assume that, for every \(x \ge 1\), \(\varphi(x) = 1\) and, in replacement of \eqref{increasingvarphi}, we require that
	\begin{equation}
		\label{eqJEMS-938}
	\lim_{s \to 0}{\varphi(s)s} = \theta r,
	\end{equation}
	for some fixed \(1 < \theta < 1/r\).

    To prove that \(\Phi^{\mathrm{sh}}\) is injective, it suffices to show that
\begin{equation}
\label{eqJEMSb-113}
t \in [0, \infty) \longmapsto \varphi \bigl(\sqrt{t^2+\gamma \tau^2} \bigr) \, t \quad \text{is increasing.}
\end{equation}
Let \(0<t_1<t_2\). Since the function \(s\mapsto \varphi(s)s\) is increasing, we have
\[
 \varphi\bigl(\sqrt{\smash[b]{t_{1}^2+\gamma \tau^2}} \bigr) < \varphi\bigl(\sqrt{\smash[b]{t_{2}^2+\gamma \tau^2}} \bigr)\frac{\sqrt{t_{2}^2+\gamma \tau^2}}{\sqrt{t_{1}^2+\gamma \tau^2}}.
\]
 By an elementary computation, the last factor on the right-hand side is not larger than \(t_2/t_1\), which implies that
\[
\varphi\bigl(\sqrt{\smash[b]{t_{1}^2+\gamma \tau^2}}\bigr) \, t_1
< \varphi\bigl(\sqrt{\smash[b]{t_{2}^2+\gamma \tau^2}} \bigr) \, t_2.
\]
This proves \eqref{eqJEMSb-113}.

    We then take \(\gamma > 0\) such that
	\begin{equation}\label{eq1306}
	\varphi\bigl(\tau \sqrt{1 + \gamma}\bigr) \, \tau = r
	\end{equation}
	which, in view of \eqref{eqJEMSb-113}, ensures that \(\Phi^{\mathrm{sh}}(B_{\tau}^{m}) = B_{r}^{m}\).{}
	The existence of such a \(\gamma\) follows from the intermediate value theorem.
    In fact, we first observe that
    \[
    \varphi(\tau/\tau)= 1 < r/\tau.
    \]
    Next, since by \eqref{deriveeminoreevarphi} the function \(s \mapsto \varphi(s)s\) is increasing, it follows from \eqref{eqJEMS-938} that
	\(\varphi(\tau\theta)\tau\theta > \theta r\), and then
    \[
    \varphi(\tau\theta)\tau >r.
    \]
    By continuity of \(\varphi\), there exists \(\gamma>0 \) such that \eqref{eq1306} holds. 
    Moreover, since \(\varphi\) is non-increasing, we must have 
    \begin{equation}
    \label{eqJEMSb-145}
    \theta<\sqrt{1 + \gamma}< {1}/{\tau} .
    \end{equation}

    One can estimate \(D^{j}\Phi^{\mathrm{sh}}\) as in the proof of Proposition~\ref{propositionJEMSThickeningModel}.
    Before proceeding with the lower bound of \(\Jacobian{m}{\Phi^{\mathrm{sh}}}\), note that, from the first inequality in \eqref{eqJEMSb-145} and the fact that \(\theta > 1\), there exists \(\delta > 0\) independent of \(\tau\) such that \(\gamma \ge \delta\). 
    From the second inequality in \eqref{eqJEMSb-145} and the fact that \(\varphi|_{(0, 1]}\) satisfies \eqref{deriveevarphi} for \(j=0\), we deduce that 
    \[
    \frac{r}{\tau}
    = \varphi(\tau\sqrt{1+\gamma})
    \leq \frac{\Cr{eqJEMS-654-bis}}{\tau \sqrt{1+\gamma}}
    \]
    and thus \(\sqrt{1+\gamma}\leq {\Cr{eqJEMS-654-bis}}/r\).    
	Now, denoting \(z =\sqrt{\abs{x}^{2} + \gamma \tau^{2}}\), one has
	\[{}
	\begin{split}
	\Jacobian{m}{\Phi^{\mathrm{sh}}}(x)
	&= (\varphi(z))^{m - 1}\Bigl(\varphi(z) + \varphi'(z) \frac{\abs{x}^{2}}{z}\Bigr){}\\
	&= (\varphi(z))^{m} \frac{\gamma\tau^{2}}{z^{2}} + (\varphi(z))^{m - 1} (\varphi(z) + \varphi'(z)z) \frac{\abs{x}^{2}}{z^2}.
	\end{split}
	\]
	Observe that each term is nonnegative.
	When \(\abs{x} \ge \tau\), we discard the first one and obtain the estimate for \(\Jacobian{m}{\Phi^{\mathrm{sh}}}(x)\) as in the proof of Proposition~\ref{propositionJEMSThickeningModel}.
	Here, we also use the fact that \(\sqrt{1+\gamma}\) is bounded from above by \({\Cr{eqJEMS-654-bis}}/r\).
    When \(\abs{x} \le \tau\), we discard the second one. 
    Since \(\gamma \ge \delta\), \(\sqrt{1+\gamma}\leq {\Cr{eqJEMS-654-bis}}/r\) and \(\varphi(z)\geq \theta r/z\), we have
	\[{}
	\Jacobian{m}{\Phi^{\mathrm{sh}}}(x)
	\ge (\varphi(z))^{m} \frac{\gamma\tau^{2}}{z^{2}}
	\ge \frac{\C}{\tau^{m}}.
	\qedhere
	\]
\end{proof}

\begin{corollary}
    \label{corollaryShrinkingBall}
    Let \(kp < m\) and \( 0 < r < \rho < 1 \).
    Given \( \epsilon > 0 \), there exists \( 0 < \tau < r \) such that the diffeomorphism \(\Phi^{\mathrm{sh}}\) provided by Proposition~\ref{propositionJEMSShrinkingModel} verifies properties \((\mathrm{sh}_{\ref{itemJEMS-601}})\) and \((\mathrm{sh}_{\ref{itemJEMS-604-bis}})\).
\end{corollary}

\resetconstant
\begin{proof}
The fact that \((\mathrm{sh}_{\ref{itemJEMS-601}})\) is satisfied is already contained in the statement of Proposition~\ref{propositionJEMSShrinkingModel}. For \((\mathrm{sh}_{\ref{itemJEMS-604-bis}})\), the Faá di Bruno composition formula and the pointwise estimates of \(D^{i}\Phi^{\mathrm{sh}}\) imply that, for every \(j \in \{1, \dots, k\}\), we have
\[{}
\abs{D^{j}(u \compose \Phi^{\mathrm{sh}})(x)}^{p}
\le \C \sum_{i = 1}^{j} \abs{(D^{i}u \compose \Phi^{\mathrm{sh}})(x)}^{p} \cdot \frac{1}{(\abs{x}^{2} + \tau^{2})^{{{jp}/{2}}}}.
\]
Since \(jp \le kp < m\), using the pointwise estimates of the Jacobian we have
\[{}
\frac{1}{(\abs{x}^{2} + \tau^{2})^{{jp}/{2}}}
\le \C \Jacobian{m}{\Phi^{\mathrm{sh}}}(x)
\quad \text{for every \(x \in \Ball^{m} \setminus B_{\tau}^{m}\)}
\]
and
\[{}
\frac{1}{(\abs{x}^{2} + \tau^{2})^{{jp}/{2}}}
\le \C \tau^{m - jp} \Jacobian{m}{\Phi^{\mathrm{sh}}}(x)
\quad \text{for every \(x \in B_{\tau}^{m}\)}.
\]
Thus, by additivity of the integral and the change of variables \(y = \Phi^{\mathrm{sh}}(x)\), we get
	\begin{multline*}
	\int_{\Ball^{m}}\abs{D^{j}(u \compose \Phi^{\mathrm{sh}})}^{p}\\
    \begin{aligned}
	& \le \Cl{eqJEMS-1370} \sum_{i = 1}^{j}\biggl(  \int_{\Ball^{m} \setminus B_{\tau}^{m}} \abs{D^{i}u \compose \Phi^{\mathrm{sh}}}^{p} \Jacobian{m}{\Phi^{\mathrm{sh}}}
	+ \tau^{m - jp} \int_{B_{\tau}^{m}} \abs{D^{i}u \compose \Phi^{\mathrm{sh}}}^{p} \Jacobian{m}{\Phi^{\mathrm{sh}}} \biggr)\\
	& = \Cr{eqJEMS-1370} \sum_{i = 1}^{j} \biggl(  \int_{\Ball^{m} \setminus B_{r}^{m}} \abs{D^{i}u}^{p} + \tau^{m - jp} \int_{B_{r}^{m}} \abs{D^{i}u}^{p} \biggr).
	\end{aligned}
	\end{multline*}
To conclude, it suffices to observe that \(\tau^{m - jp} \to 0\) as \(\tau \to 0\).
\end{proof}

Proposition~\ref{propositionJEMSShrinkingModel} has the following counterpart that involves shrinking around the dual skeleton of a cubication, see \cite{BPVS_MO}*{Proposition~8.1}:

\begin{proposition}
\label{lemmaShrinkingFaceNearDualSkeletonGlobal}
Let \( \cS^{m} \) be a cubication in \( \R^{m} \) of radius \( \eta > 0 \) and let  \(\cT^{\ell^*}\) be  the dual skeleton of\/ \(\cS^\ell\).
Then, for every \(0 < \tau < \mu < 1/2\), there exists a diffeomorphism \(\Phi^{\mathrm{sh}}  \colon  \R^m  \to \R^m\) such that 
\begin{enumerate}[$(i)$]
\item 
\label{itempropositionshrinkingfromaskeleton2}
for every \(\sigma^m \in \cS^m\),
\(\Phi^{\mathrm{sh}}(\sigma^m)\subset \sigma^m\),
\item{}
\label{itempropositionshrinkingfromaskeleton1}
 \(T^{\ell^*} + Q^m_{\mu\eta} \subset \Phi^{\mathrm{sh}}(T^{\ell^*} + Q^m_{\tau\mu\eta})\),
\item 
\label{itempropositionshrinkingfromaskeletonIdentity}
\(\Phi^{\mathrm{sh}} = \Id\) in \(\R^{m} \setminus (T^{\ell^*} + Q^m_{2\mu\eta})\), 
\item 
\label{itempropositionshrinkingDistance}
for every \(x \in S^{m}\) and every \( i \in \{\ell, \dots, m\} \),
\[{}
d(\Phi^{\mathrm{sh}}(x), T^{i^{*}}) \geq C d(x, T^{i^{*}}) \text{,}
\]
\item{} 
\label{itempropositionshrinkingfromaskeleton5}
for every \(0 < \beta < \ell + 1\), for every \(j \in \N_*\) and every \(x\in \R^m\), 
\[
(\mu\eta)^{j-1}\abs{D^j \Phi^{\mathrm{sh}}(x)} \le  C' \bigl(\Jacobian{m}{\Phi^{\mathrm{sh}}}(x)\bigr)^{{j}/{\beta}} \text{,}
\]
\item{}
\label{itempropositionshrinkingfromaskeleton4}
for every \(0 < \beta < \ell + 1\), every \(j \in \N_*\) and every \(x\in \Phi^{-1}(T^{\ell^*} + Q^m_{\mu\eta})\), 
\[
(\mu\eta)^{j-1}\abs{D^j \Phi^{\mathrm{sh}}(x)} \le  C'' \tau^{j(\frac{\ell + 1}{\beta} - 1)} \bigl(\Jacobian{m}{\Phi^{\mathrm{sh}}}(x)\bigr)^{{j}/{\beta}} \text{,}
\]
\end{enumerate}
for some constants \(C, C' > 0\) depending on \(\beta\), \(j\) and \(m\).
\end{proposition}

Let us explain why assertion~\( (\ref{itempropositionshrinkingDistance}) \), that is not present in the statement of \cite{BPVS_MO}*{Proposition~8.1}, is true.
We begin by observing that the proof of \cite{BPVS_MO}*{Proposition~8.1} gives the following additional properties of \( \Phi^{\mathrm{sh}} \)\,: 
\begin{enumerate}[\((a)\)]
    \item \(\Phi^{\mathrm{sh}}(\sigma^m) = \sigma^m\) for every cube \(\sigma^{m}\in \cS^{m}\),
    \item for every cubes \(\sigma^{m}_1\) and \(\sigma^{m}_2\) with centers \(a_1\) and \(a_2\) respectively, 
\[
\Phi^{\mathrm{sh}}\vert_{\sigma_{1}^m} 
= \Phi^{\mathrm{sh}}\vert_{\sigma_{2}^m}(\cdot+a_2-a_1).  
\]
    \item 
    for every \( i \in \{ \ell, \dots, m \} \),
\begin{equation}
\label{eqStrong2-1188}
T^{i^*} \subset \Phi^{\mathrm{sh}} (T^{i^*}).
\end{equation}
\end{enumerate}
From \eqref{eqStrong2-1188}, there exists \(y \in T^{i^{*}}\) such that 
\[
d(\Phi^{\mathrm{sh}}(x), T^{i^{*}}) = |\Phi^{\mathrm{sh}}(x) - \Phi^{\mathrm{sh}}(y)|.
\]
Using that \(\Phi^{\mathrm{sh}}\) is a diffeomorphism, there exists \(C>0\) that does not depend on \(x\) or \(y\) such that
\[
|\Phi^{\mathrm{sh}}(x) - \Phi^{\mathrm{sh}}(y)| 
\geq C|x - y|\geq C d(x, T^{i^{*}}).
\]
Property~\( (\ref{itempropositionshrinkingDistance}) \) then follows by combining both estimates.

\medskip

In what concerns the proof of Theorem~\ref{thm_density_manifold_open}, we need the following property of \(\ClassR_i\)~functions obtained from shrinking around the singular set, where we use a notation adapted to the context of that proof:

\begin{corollary}
    \label{corollaryShrinkingCubication}
    Let \(\Phi^{\mathrm{sh}}  \colon  \R^m \to \R^m\) be the diffeomorphism given by Proposition~\ref{lemmaShrinkingFaceNearDualSkeletonGlobal} and let \( e \in \{ \ell, \ldots, m \} \).
    Then, for every \(\ClassR_{e^{*}}\)~function \(u\) in \(\Int{S^{m}}\) with singular set \( T^{e^{*}} \), \( u \compose \Phi^{\mathrm{sh}} \) is also an \(\ClassR_{e^{*}}\)~function in \( \Int{S^{m}} \) with singular set \( T^{e^{*}} \) that satisfies
    \[
    u \compose \Phi^{\mathrm{sh}} = u \quad \text{in \((\Int{S^{m}}) \setminus (T^{\ell^{*}} + Q^{m}_{2\mu\eta}) \).}
    \]
    Moreover, if \( kp < \ell + 1 \), then 
    \( u \compose \Phi^{\mathrm{sh}} \in \Sobolev^{k, p}(\Int{S^{m}}) \) and, for every \( j \in \{1, \dots, k\} \), 
  \begin{multline*}
  \norm{D^{j}(u \compose \Phi^{\mathrm{sh}})}_{\Lebesgue^{p}((\Int{S^{m}}) \cap (T^{\ell^{*}} + Q^{m}_{2\mu\eta}))}\\
\le C'' \sum_{i=1}^j  (\mu\eta)^{i - j} \norm{D^{i} u}_{\Lebesgue^p((\Int{S^{m}}) \cap (T^{\ell^*} + Q^m_{2\mu\eta}) \setminus (T^{\ell^*} + Q^m_{\mu\eta}))}\\
+ C''' \tau^{\frac{\ell+1-kp}{p}} \sum_{i=1}^j  (\mu\eta)^{i - j} \norm{D^{i} u}_{\Lebesgue^p((\Int{S^{m}}) \cap (T^{\ell^*} + Q^m_{\mu\eta}))}.
  \end{multline*}
\end{corollary}

\resetconstant
\begin{proof}
Since \( \Phi^{\mathrm{sh}} \) is a diffeomorphism, by 
\( (\ref{itempropositionshrinkingfromaskeleton2}) \) and
\( (\ref{itempropositionshrinkingfromaskeletonIdentity}) \), we have 
\(\Phi^{\mathrm{sh}}(\Int{S^{m}}) \subset \Int{S^{m}}\). 
By injectivity of \(\Phi^\mathrm{sh}\) and \eqref{eqStrong2-1188} with \(i = e\), we thus get
\[
 \Phi^{\mathrm{sh}}((\Int{S^{m}}) \setminus T^{e^{*}} ) 
 \subset (\Int{S^{m}}) \setminus T^{e^{*}} .
\]
Since \( u \) is smooth in the complement of \( T^{e^{*}} \), we deduce that \(u \compose \Phi^{\mathrm{sh}}\) is also smooth in the complement of \( T^{e^{*}} \).
To prove that \( u \compose \Phi^{\mathrm{sh}}\) is an \(\ClassR_{e^{*}}\)~function in \(\Int{S^{m}}\) with singular set \( T^{e^{*}} \), we rely on the Faá di~Bruno composition formula and the fact that \(\Phi^{\mathrm{sh}}\) is smooth. 
In fact, for every \(j \in \N_*\),
\[
|D^{j} (u \compose \Phi^{\mathrm{sh}})|
\leq \C \sum_{i=1}^{j} |D^i u \compose \Phi^{\mathrm{sh}}|.
\]
Since \( u \) is an \(\ClassR_{e^{*}}\)~function with singular set \(T^{e^{*}}\), for every \(  x \in (\Int{S^{m}}) \setminus T^{e^{*}} \), we then have
\[
|D^{j}(u \compose \Phi^{\mathrm{sh}})(x)|
\leq \frac{\C}{d(\Phi^{\mathrm{sh}}(x) ,T^{e^{*}})^{j}}.
\]
Finally, from property~\( (\ref{itempropositionshrinkingDistance}) \) in Proposition~\ref{lemmaShrinkingFaceNearDualSkeletonGlobal} we obtain
\[
|D^{j}(u \compose \Phi^{\mathrm{sh}})(x)|
\leq \frac{\C}{d(x, T^{e^{*}})^{j}}.
\]
This proves that \(u \compose \Phi^{\mathrm{sh}}\) is also an \(\ClassR_{e^{*}}\)~function with singular set \( T^{e^{*}} \).
For \( kp < \ell + 1 \), by Proposition~\ref{propositionClassRInclusionSobolev}, we have \( u \compose \Phi^{\mathrm{sh}}  \in \Sobolev^{k, p}(\Int{S^{m}}) \).
The \(\Lebesgue^{p}\)~estimates for \( D^{j}(u \compose \Phi^{\mathrm{th}}) \) in \((\Int{S^{m}}) \cap (T^{\ell^{*}} + Q^{m}_{2\mu\eta})\) can be proved by a change of variable along the lines of the proof of Corollary~\ref{corollaryShrinkingBall}, see \cite{BPVS_MO}*{Corollary~8.2}.
\end{proof}

\section{Density on open sets}

We now focus on the proof of the direct implication of Theorem~\ref{thm_density_manifold_open} when \( \manfV = \Omega \) is a bounded Lipschitz open subset of \( \R^m \).
To this end, in Parts~\ref{partThmOpen-1}--\ref{partThmOpen-6} we consider the case where there exists some bounded Lipschitz open neighborhood \(O\) of \(\overline{\Omega}\) such that \(u \in \Sobolev^{k, p}(O; \manfN)\)  and \(u\) is \((\ell, e)\)-extendable in \(O\) with
\[
\ell \vcentcolon= \floor{kp}.
\]
Later on in Part~\ref{partThmOpen-7} we then explain how the analytical assumption in \(\Omega\) allows us to reduce the case \(u \in \Sobolev^{k, p}(\Omega; \manfN)\) to this one by approximation.

\resetconstant

\begin{PartThm}
\label{partThmOpen-1}
We proceed to define good and bad cubes, and then estimate the total volume of bad cubes in terms of the \(\Lebesgue^{kp}\) norm of \(Du\). 
\end{PartThm}

\begin{proof}[\nopunct]
We take a cubication  \(\cS^m\) of radius \( \eta > 0\) so that the polytope \(S^m = \bigcup\limits_{\sigma^m \in \cS^m}\sigma^m\) satisfies
\begin{equation*}
\label{eqJEMS-80}
\overline{\Omega} \subset S^m \subset S^{m} + Q_{2\eta}^{m} \subset O.
\end{equation*}
This condition gives us some room around \(S^{m}\) to deal with the outermost cubes in the cubication that intersect the boundary of \(S^{m}\).
We rely on the fundamental idea of Bethuel~\cite{Bethuel} based on the classification of the elements in \(\cS^m\) as good or bad cubes.
Figure~\ref{figureJEMSCubication} illustrates a cubication of a square in \(\R^{2}\).
We rely on this model case to exemplify our tools throughout the proof.

\begin{figure}
\centering
\hfill{}
\includegraphics{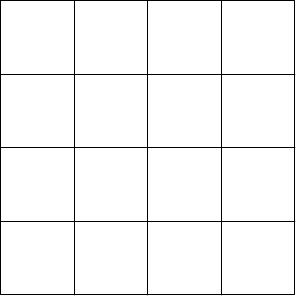}\hfill{}
\includegraphics{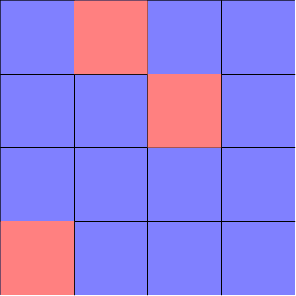}\hfill{}
\caption{Cubication of a square and its good and bad cubes (colored blue and red, respectively)}
\label{figureJEMSCubication}
\label{figureJEMSCubication-GB}
\end{figure}

Since \(\manfN\) is compact, the map \(u\) is bounded and then, by the Gagliardo-Nirenberg interpolation inequality, we have \(Du \in \Lebesgue^{kp}(O)\).
Let \(\iota>0\) be such that the nearest point projection \(\Pi\) is well defined and smooth on   \(\overline{\manfN +  B^{\nu}_\iota}\).
Having fixed \(0 < \rho < 1/2\), we say that \(\sigma^{m} \in \cS^m\) is a \emph{good cube} provided that
\begin{equation}
\label{eqJEMSGoodBall}
\frac{\alpha}{\eta^{\frac{m}{kp} - 1}} \norm{Du}_{L^{kp}(\sigma^m + Q_{2\rho\eta}^{m})} \le \iota,
\end{equation}
for some constant \(\alpha > 0\) to be chosen later on.
Otherwise, we call \(\sigma^{m}\) a \emph{bad cube}.
Let us denote by \(\cE^{m}\) the subskeleton of \(\cS^{m}\) that contains all bad cubes and, accordingly, \(E^{m}\) is the union of bad cubes.\label{eqJEMSBadBall}
In a good cube, most of the values of the map \(u\) lie in a geodesic ball centered at some fixed point \(\xi \in \manfN\), and \(u\) does not oscillate too much at the scale \(2\rho\eta\).{}
The choice of \(\alpha\) is made later to ensure the control of this oscillation and guarantee that the images of the maps that we construct lie in a tubular neighborhood of \(\manfN\).

It is straightforward to estimate the number of bad cubes to find that their total volume tends to zero as \(\eta \to 0\).
More precisely, the Lebesgue measure of the set \(E^m + Q^m_{2\rho\eta}\) satisfies
\begin{equation}
\label{eqJEMS-103}
 \bigabs{E^m + Q^m_{2\rho\eta}}
	\le {\Cl{cteStrong2-59} } \eta^{kp} \norm{Du}_{\Lebesgue^{kp}(O)}^{kp},
\end{equation}
for some \(\Cr{cteStrong2-59}>0\) depending on \(m\), \(\alpha\) and \(\iota\).
Indeed, for each \(\sigma^m \in \cE^{m}\) we have
\[
1  < \frac{\alpha^{kp}}{\iota^{kp} \eta^{m-kp}} \int_{\sigma^m + Q_{2\rho\eta}^m} \abs{Du}^{kp}.
\]
Since the cubes \(\sigma^m + Q_{2 \rho \eta}^m\) have finite overlapping depending solely on the dimension \(m\), we estimate the number \(\#\cE^m \) of bad cubes as follows:
\[
\#\cE^m
 <  \sum_{\sigma^{m} \in \cE^{m}} \frac{\alpha^{kp}}{\iota^{kp} \eta^{m-kp}} \int_{\sigma^m + Q_{2\rho\eta}^m} \abs{Du}^{kp}
 \le  \frac{\C } {\eta^{m-kp}} \int_{E^m + Q_{2\rho\eta}^m}  \abs{Du}^{kp}.
\]
By finite subadditivity of the Lebesgue measure, we also have 
\[
\bigabs{E^m + Q^m_{2\rho\eta}} 
\le \sum_{\sigma^{m} \in \cE^{m}} \bigabs{\sigma^m + Q^m_{2\rho\eta}}
 \le (4\eta)^{m} \, (\#\cE^m). 
\]
Combining both estimates, we get \eqref{eqJEMS-103}.

It is more convenient to work in the sequel with a cubication larger than \(\cE^{m}\).{}
To this end, we define \(\cU^{m}\) as the subskeleton of \(\cS^{m}\) that contains \(\cE^{m}\) and is formed by all closed cubes \(\sigma^{m} \in \cS^{m}\) that intersect some element of \(\cE^{m}\).\label{eqJEMSUglyBall} 
A simple argument based on dilations of cubes shows that the number of elements in \(\cU^{m}\) is comparable to that of \(\cE^{m}\).{}
In particular, see \cite{BPVS_MO}*{p.~799}, the set \(U^{m} + Q_{2 \rho\eta}^{m}\) satisfies an estimate analogous to \eqref{eqJEMS-103}, namely
\begin{equation}
    \label{eq-claim_measure_bad}
     \bigabs{U^m + Q^m_{2\rho\eta}}
	\le {\C } \eta^{kp} \norm{Du}_{\Lebesgue^{kp}(O)}^{kp}\,.
\end{equation}
This concludes Part~\ref{partThmOpen-1} of the proof of Theorem~\ref{thm_density_manifold_open} for open sets.
\end{proof}

\begin{PartThm}
\label{partThmOpen-2}
We apply opening to modify the map \( u \) in a neighborhood of \(U^\ell\) and \(S^\ell\) to obtain maps \(u_{\eta}^{\mathrm{op}}\) and \(u_{\eta}^{\mathrm{fop}} \colon O \to \manfN\), respectively, that are also \((\ell, e)\)-extendable.
\end{PartThm}

\begin{proof}[\nopunct]
We begin by showing the existence of a summable function \(\widetilde{w}  \colon  \R^{m} \to [0,+\infty]\) for which, using the smooth map \(\Phi^{\mathrm{op}} \in \Fuglede_{\tilde{w}} (\R^{m}; \R^{m})\) given by Proposition~\ref{proposition_ouverture-new} with subskeleton \( \cU^{\ell} \), we have that \(u \compose \Phi^{\mathrm{op}}\) is an \( (\ell, e) \)-extendable map in \(O\) that belongs to \((\Sobolev^{k, p} \cap \Sobolev^{1, kp})(O)\) and is such that, for every \(\sigma^{\ell} \in \cU^{\ell} \) and every \(j\in \{1, \dots, k\}\),
\begin{align}
	\label{eqJEMS-478}
    \norm{D^{j}(u \compose \Phi^{\mathrm{op}})}_{\Lebesgue^{p}(\sigma^\ell + Q^{m}_{2\rho \eta})}
	&\le {\C } \, \sum_{i=1}^{j}{\eta^{i - j} \norm{D^{i}u}_{\Lebesgue^{p}(\sigma^\ell + Q^{m}_{2\rho \eta})}},\\
    \label{eqJEMS-482}
    \norm{D(u \compose \Phi^{\mathrm{op}})}_{\Lebesgue^{kp}(\sigma^\ell + Q^{m}_{2\rho \eta})}
	&\le {\C } \norm{Du}_{\Lebesgue^{kp}(\sigma^\ell + Q^{m}_{2\rho \eta})},
\end{align}
where the constants are independent of \(u\), \(\eta\) and \(\cU^{m}\).

Indeed, by Corollary~\ref{corollaryOpeningCubicationSobolev} there exists a summable function \( w_{1} \) for which \eqref{eqJEMS-478} holds when \(\Phi^{\mathrm{op}} \in \Fuglede_{w_{1}} (\R^{m}; \R^{m})\).
Since we also have that \( u \) is a \( \Sobolev^{1, kp} \)~map, the same corollary gives a summable function \( w_{2} \) for which \eqref{eqJEMS-482} holds for \(\Phi^{\mathrm{op}} \in \Fuglede_{w_{2}} (\R^{m}; \R^{m})\).
Finally, since \( u \) is \( (\ell, e) \)-extendable in \(O\), there exists an \( \ell \)-detector \( w_{3} \) for which Definition~\ref{definitionExtensionVMO} is satisfied. 
Extending \(w_3\) by zero outside \(O\), we may assume that \(w_3\) is a summable function defined in \(\R^m\).
It then suffices to choose
\begin{equation}
    \label{eqJEMSb-438}
\widetilde{w} = w_{1} + w_{2} + w_{3} .
\end{equation}
To see that \(u \compose \Phi^{\mathrm{op}}\) is also \((\ell, e)\)-extendable in \( O \),  observe that \(\Phi^{\mathrm{op}}(O) \subset O\).
Since \(\Phi^{\mathrm{op}} \in \Fuglede_{\tilde w}(\R^m; \R^m)\) and \( w_{3} \le \widetilde w\), we then have \(\Phi^{\mathrm{op}}|_O \in \Fuglede_{\tilde w_3}(O; O)\) and the \((\ell, e)\)-extendability of \( u \compose \Phi^{\mathrm{op}} \) follows from Proposition~\ref{corollary_ouverture}.

\begin{figure}
\centering{}
\hfill\includegraphics{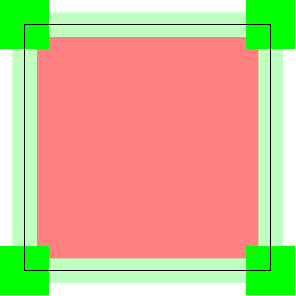}
\hfill\includegraphics{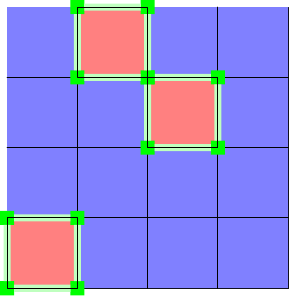}
\hfill{}
\caption{Construction of \(u^{\mathrm{op}}\) on a bad cube (left) and on the entire domain (right)}
\label{figureJEMSOpeningSkeleton}
\end{figure}

We now let
\begin{equation}
\label{eqJEMS-uop}
u_{\eta}^{\mathrm{op}} = u \compose \Phi^{\mathrm{op}}
\quad \text{in \(O\)}.
\end{equation}
Figure~\ref{figureJEMSOpeningSkeleton} illustrates opening applied to the \(1\)-dimensional skeleton of the bad cubes in our model case \(\R^{2}\).
As we mentioned earlier, we also need to work with neighboring cubes, but for the sake of clarity, we systematically illustrate the construction using \(\cE^{m}\) instead of \(\cU^{m}\).

Since \(u_{\eta}^{\mathrm{op}} = u\) in the complement of \(U^{\ell} + Q_{2\rho\eta}^{m}\)\,, by \eqref{eqJEMS-478} and the additivity of the Lebesgue integral, for every \(j \in \{1, \dots, k\}\) we deduce that
\begin{equation}
\label{inequalityMainOpening}
\begin{split}
\norm{D^j u_{\eta}^\mathrm{op} - D^j u}_{\Lebesgue^p(O)} 
& = \norm{D^j u_{\eta}^\mathrm{op} - D^j u}_{\Lebesgue^p(U^\ell + Q^m_{2\rho\eta})}\\
& \le \norm{D^j u_{\eta}^\mathrm{op}}_{\Lebesgue^p(U^\ell + Q^m_{2\rho\eta})} + \norm{D^j u}_{\Lebesgue^p(U^\ell + Q^m_{2\rho\eta})}\\
& \le {\C } \sum_{i = 1}^j \eta^{i - j} \norm{D^i u}_{\Lebesgue^p(U^\ell + Q^m_{2\rho\eta})}.
\end{split}
\end{equation}

Since \(\ell \le kp\) and \(u_{\eta}^\mathrm{op}\) is bounded and depends on at most \(\ell\) variables in a neighborhood of each element of \(\cU^{\ell}\), by Proposition~\ref{propositionOpeningVMO} the choices of \(\widetilde{w}\) and \(\Phi^{\mathrm{op}}\) we be made so that the restriction \( u_{\eta}^{\mathrm{op}}|_{U^\ell + Q^m_{\rho\eta}}\) is a \(\VMO\)~map and,
for every \( a \in U^{\ell} + Q^m_{\rho\eta/2} \) and every \(0 < r \le \rho\eta/2\), 
\[
\fint\limits_{Q_r^m (a)}\fint\limits_{Q_r^m (a)} \abs{u_\eta^{\mathrm{op}}(y) - u_\eta^{\mathrm{op}} (z)} \dif z \dif y 
\le \frac{\C}{\eta^{\frac{m}{kp} - 1}} \norm{Du}_{\Lebesgue^{kp}(\sigma^{m} + Q_{2\rho\eta}^m)},
\]
where \(\sigma^{m} \in \cS^{m} \) is such that \(a \in \sigma^m + Q^m_{\rho\eta/2}\)\,.

To fully explore the \((\ell, e)\)-extendability of \(u\), later on in Part~\ref{partThmOpen-4} we rely on a variant of the map \(u^{\mathrm{op}}\) that we denote by \(u_{\eta}^{\mathrm{fop}}\).{}
It is obtained from the opening technique using the entire skeleton \(\cS^{\ell}\) instead of \(\cU^{\ell}\), see Figure~\ref{figureCubicationDensityNew}.
To construct this additional map, we need to work from the beginning with a larger \(\ell\)-detector \(\widetilde{w}\), where in \eqref{eqJEMSb-438} we take as \(w_1\) a summable function provided by Corollary~\ref{corollaryOpeningCubicationSobolev} and Proposition~\ref{propositionOpeningVMO} using the larger \(\ell\)-dimension skeleton \(S^\ell\) instead of \(U^\ell\).
We then take the map \(\Phi^{\mathrm{fop}} \in \Fuglede_{\tilde{w}} (\R^{m}; \R^{m})\) given by these statements using \(S^\ell\) and define
\begin{equation}
\label{eqJEMS-ufop}
u_{\eta}^{\mathrm{fop}} = u \compose \Phi^{\mathrm{fop}}
\quad \text{in \( O \)\,.} 
\end{equation}

\begin{figure}
\centering
\hfill{}
\includegraphics{opening_skeleton-light}\hfill{}
\includegraphics{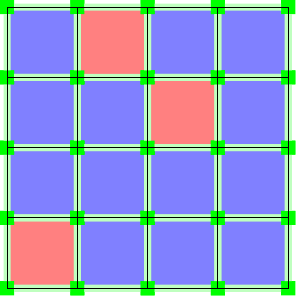}\hfill{}
\caption{Opening used to construct \(u_{\eta}^{\mathrm{op}}\) (left) and \(u_{\eta}^{\mathrm{fop}}\) (right)}
\label{figureCubicationDensityNew}
\end{figure}

Since the maps \(\Phi^{\mathrm{op}}\) and \(\Phi^{\mathrm{fop}}\) are constructed by iterating opening on each cell \(\sigma^i\) successively, from \(i=0\) up to \(i=\ell\), we can assume that in property~\((\ref{itemOpeningi})\) of Proposition~\ref{proposition_ouverture-new} the two maps \(\Phi^{\mathrm{op}}\)  and \(\Phi^{\mathrm{fop}}\) are equal to the same constants on the \(m-\ell\) dimensional cubes of radius \(\rho \eta\) which are orthogonal to each \(\sigma^{\ell} \in \cU^{\ell} \), that is,
 \begin{equation}\label{eqDensityNew-142}
\Phi^{\mathrm{op}} = \Phi^{\mathrm{fop}}
\quad \text{in \(U^{m} + Q_{\rho\eta}^{m}\)\,.}     
\end{equation}

As for the map \( u_{\eta}^{\mathrm{op}} \), by Proposition~\ref{corollary_ouverture} we have that \( u_{\eta}^{\mathrm{fop}} \) is \( (\ell, e) \)-extendable in \(O\)\,.
By Corollary~\ref{corollaryOpeningCubicationSobolev}, the family of maps \(u_{\eta}^{\mathrm{fop}}\) is bounded in \(\Sobolev^{k, p}(O; \manfN)\) but, unlike \(u_{\eta}^{\mathrm{op}}\), it does not need to converge strongly to \(u\).
The reason is that the measure of the set \(S^{\ell} + Q_{2\rho\eta}^{m}\) where opening takes place is bounded from below by some positive constant independent of \(\eta\), cf.\@~Proposition~\ref{corollaryOpeningConvergence}.
This concludes Part~\ref{partThmOpen-2} of the proof of Theorem~\ref{thm_density_manifold_open} for open sets.
\end{proof}

The next step of the proof relies on the adaptive smoothing introduced in Section~\ref{sectionJEMSSmoothing}. 

\begin{PartThm}
    \label{partThmOpen-3}
    We apply the adaptive smoothing to \(u_\eta^\mathrm{op}\) to obtain a map \(u_\eta^\mathrm{sm} \colon S^{m}+Q^{m}_{\rho\eta} \to \R^\nu\) whose images of \(S^m \setminus U^m\) and \(U^\ell + Q_{\rho\eta/2}^m\) are contained in a small tubular neighborhood of \(\manfN\).
\end{PartThm}

\begin{proof}[\nopunct]
We exploit the properties of \(u\) in various regions with the help of adaptive smoothing at different scales.  
To this end, we use a positive smooth function \(\psi_\eta \colon S^{m}+Q^{m}_{\rho\eta} \to \R\) with \({\psi_\eta} \le \rho\eta\) of the form
\begin{equation}
\label{eqJEMSb-551}
{\psi_\eta} = t\zeta + s (1 - \zeta),
\end{equation}
where \(s\) and \(t\) are positive parameters such that \(s \ll t\) and \(t \sim \eta\), and \(\zeta \colon S^{m}+Q^{m}_{\rho\eta} \to \R\) is a smooth function that satisfies
\begin{enumerate}[$(i)$]
\item \(0 \le \zeta \le 1\) in \(S^m+Q^{m}_{\rho\eta}\)\,,
\item \(\zeta = 1\) in \((S^m+Q^{m}_{\rho\eta}) \setminus (E^m+ Q^{m}_{3\rho\eta/4})\),
\item \(\zeta = 0\) in \(E^m+Q^{m}_{\rho\eta/2}\)\,,
\item  for every \( j\in \{1, \ldots, k\} \), \( \eta^j\norm{D^{j} \zeta}_{\Lebesgue^{\infty}}\leq \Cl{cteStrong2-185} \),
\end{enumerate}
where the constant \( \Cr{cteStrong2-185} > 0 \) depends only on \( m \). 
Actually, we take
\begin{equation}
\label{eqJEMSb-563}
t =\min{ \Bigl\{\rho, \frac{1}{\Cr{cteStrong2-185}} \Bigr\} } \frac{\eta}{4},
\end{equation}
while the choice of \(s\) will be specified subsequently.

Given a mollifier \( \varphi\in \Smooth^{\infty}_{c}(\R^m)\), we then set
\begin{equation}
\label{eqJEMS-usm}
u_{\eta}^\mathrm{sm} = \varphi_{{\psi_\eta}} * u_{\eta}^{\mathrm{op}} 
\quad \text{in \(S^{m}+Q^{m}_{\rho\eta}\)\,.}
\end{equation}
We can rely on  \cite{BPVS_MO}*{Estimate~(5.2)} to deduce from the estimates for the derivatives in the adaptive smoothing and opening that, for every \(j \in \{1, \dots, k\}\), 
\begin{multline}
\label{inequalityMainSmoothening}
\norm{D^j u_{\eta}^\mathrm{sm} - D^j u_{\eta}^\mathrm{op}}_{\Lebesgue^p(S^{m})}\\
\le \sup_{v \in B_1^m}{\norm{\tau_{{\psi_\eta} v}(D^j u) - D^j u}_{\Lebesgue^p(S^{m})}}
+ \C  \sum_{i=1}^j  \eta^{i - j} \norm{D^i u}_{\Lebesgue^p(U^m + Q^m_{2\rho\eta})}.
\end{multline}

\begin{figure}
\centering
\includegraphics{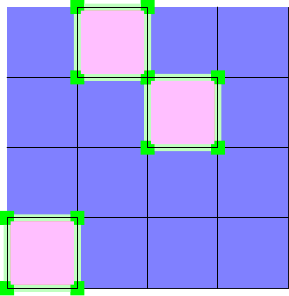}
\caption{Inside the pink squares, \(u^{\mathrm{sm}}\) can be far from \(\manfN\)}
\label{figureJEMSCubication-Far}
\end{figure}

We claim that a suitable choice of \(s\) in \eqref{eqJEMSb-551} yields the inclusions
\begin{equation}
\label{eqJEMS-Inclusion}
u_{\eta}^{\mathrm{sm}}( S^{m} \setminus U^{m} )
\subset \manfN + B_{\iota}^{\nu}
\end{equation}
and
\begin{equation}
\label{eqJEMS-Inclusion-bis}
(\varphi_{r{\psi_\eta}} * u_{\eta}^{\mathrm{op}})\bigl(U^{\ell} + Q^{m}_{\rho\eta/2} \bigr){}
\subset \manfN + B_{\iota}^{\nu}
\quad \text{for every \( 0 < r \le 1\),}
\end{equation}
see Figure~\ref{figureJEMSCubication-Far}.
To this end, for every \( 0 < r \le 1\), we rely on the inequality
\begin{equation}
\label{eqStrong2-227}
d(\varphi_{r{\psi_\eta}} * u_{\eta}^{\mathrm{op}}(x), \manfN)
\le {\C }  \fint_{B_{r{\psi_\eta}(x)}^m(x)}\fint_{B_{r{\psi_\eta}(x)}^m(x)} \abs{u^\mathrm{op}_\eta(y) - u^\mathrm{op}_\eta(z)} \dif z\dif y.
\end{equation}
that follows from a straightforward variant of Lemma~\ref{lemmaVMOUniformConvergence}.

We first establish \eqref{eqJEMS-Inclusion}.
For every \(x\) in the set \(S^{m} \setminus U^{m}\subset S^{m}\setminus (E^m+Q^{m}_{3\rho\eta/4})\), we may write \({\psi_\eta}(x)=t=c \eta\) for some \(0<c<\rho/2\).
Since \({\psi_\eta} \le \rho\eta/2\),
then, by assumption on \(x\), there exists \(\sigma^m \in \cS^m \setminus \cU^m\) such that 
\[
B^{m}_{{\psi_\eta}(x)}(x) \subset \sigma^m + Q^{m}_{\rho\eta}\,.
\]
Combining \eqref{eqStrong2-227} with \( r = 1 \) and the Poincaré-Wirtinger inequality, we get
\[
d(u_{\eta}^{\mathrm{sm}}(x), \manfN)
\le \max_{\sigma^m \in \cS^m \setminus \cU^m} \frac{\C}{\eta^{\frac{m}{kp} - 1}} \norm{Du_{\eta}^{\mathrm{op}}}_{\Lebesgue^{kp}(\sigma^m + Q_{2\rho\eta}^m)}.
\]
By~\eqref{eqJEMS-482} and using also the fact that \(u_{\eta}^{\mathrm{op}}=u\) in \((\sigma^m + Q_{2\rho\eta}^m)\setminus (U^\ell+Q_{2\rho\eta}^m)\), we have
\begin{equation}\label{eq-Addendum1}
\norm{Du_{\eta}^{\mathrm{op}}}_{\Lebesgue^{kp}(\sigma^m + Q_{2\rho\eta}^m)}\leq \C
\norm{Du}_{\Lebesgue^{kp}(\sigma^m + Q_{2\rho\eta}^m)}.
\end{equation}
It follows that
\[{}
d(u_{\eta}^{\mathrm{sm}}(x), \manfN)
\le \max_{\sigma^m \in \cS^m \setminus \cU^m} \frac{\Cl{cteStrong2-239} }{\eta^{\frac{m}{kp} - 1}} \norm{Du}_{\Lebesgue^{kp}(\sigma^m + Q_{2\rho\eta}^m)}.
\]
This implies the inclusion \eqref{eqJEMS-Inclusion} by taking the constant \(\alpha > 0\) in \eqref{eqJEMSGoodBall} larger than \(\Cr{cteStrong2-239}\).

To prove \eqref{eqJEMS-Inclusion-bis}, we consider two different cases.
For every \(x \in (U^{\ell} + Q^{m}_{\rho\eta/2}) \setminus (E^{m}+Q^{m}_{\rho\eta/2})\), there exists \(\sigma^m \in \cS^m \setminus \cE^m\) such that \(x\in \sigma^m+Q^{m}_{\rho\eta/2}\)\,.
We recall that \(u_{\eta}^{\mathrm{op}}\) in \(U^{\ell} + Q^{m}_{\rho\eta}\) depends locally on \(\ell\) coordinate variables.
Since \(\ell\leq kp\), by \eqref{eqStrong2-227} and Proposition~\ref{propositionOpeningVMO}, for every \(0 < r \le 1\) we then have
\[{}
d(\varphi_{r{\psi_\eta}}*u_{\eta}^{\mathrm{op}}(x), \manfN)
\le \max_{\sigma^m \in \cS^m \setminus \cE^m} \frac{\C }{\eta^{\frac{m}{kp} - 1}} \norm{Du}_{\Lebesgue^{kp}(\sigma^m + Q_{2\rho\eta}^m)}.
\]
By increasing the constant \(\alpha > 0\) in \eqref{eqJEMSGoodBall} if necessary, this implies that 
\[
(\varphi_{r{\psi_\eta}}*u_{\eta}^{\mathrm{op}})\bigl( (U^{\ell} + Q^{m}_{\rho\eta/2}) \setminus (E^{m}+Q^{m}_{\rho\eta/2}) \bigr)
\subset \manfN + B_{\iota}^{\nu}\,.
\]

By Proposition~\ref{propositionOpeningVMO}, the restriction of \(u_{\eta}^{\mathrm{op}}\) in \(U^{\ell} + Q^{m}_{\rho\eta}\) is \(\VMO\).
As in Lemma~\ref{lemmaVMOUniformConvergence}, when \(s \to 0\) we then deduce that
\[{}
d(\varphi_{s} * u_{\eta}^{\mathrm{op}}, \manfN) \to 0
\quad \text{uniformly in \((U^{\ell} + Q^{m}_{\rho\eta/2})\cap (E^m+Q^{m}_{\rho\eta/2})\)\,.}
\]
Since, for every \(x \in (U^{\ell} + Q^{m}_{\rho\eta/2}) \cap (E^{m}+Q^{m}_{\rho\eta/2})\), we have \({\psi_\eta}(x) = s\), by choosing \(s \le \rho\eta/2\) sufficiently small in \eqref{eqJEMSb-551} we get
\[
(\varphi_{r{\psi_\eta}} * u_{\eta}^{\mathrm{op}})\bigl((U^{\ell} + Q^{m}_{\rho\eta/2})\cap (E^{m}+Q^{m}_{\rho\eta/2}) \bigr){}
\subset \manfN + B_{\iota}^{\nu}\,,
\]
thus completing the proof of \eqref{eqJEMS-Inclusion-bis}.
This concludes Part~\ref{partThmOpen-3} of the proof of Theorem~\ref{thm_density_manifold_open} for open sets.
\end{proof}

The next step of the proof relies on the thickening introduced in Section~\ref{subsection_thickening}.

\begin{PartThm}
 \label{partThmOpen-5}
Denoting by \( \cZ^{\ell^{*}} \) the dual skeleton of \(\cU^{\ell}\), we apply thickening to obtain an \(\ClassR_{\ell^{*}}\)~map \(u_\eta^\mathrm{th} \colon (S^{m} + Q^{m}_{\rho\eta/2}) \setminus Z^{\ell^{*}} \to \R^\nu\) with singular set \( Z^{\ell^{*}}\) such that \(u_\eta^\mathrm{th}(S^m \setminus Z^{\ell^{*}}) \subset \manfN + B_\iota^\nu\). \label{eqJEMSUglyCubesDualSkeleton} 
\end{PartThm}

\begin{proof}[\nopunct]
Let \(\Phi^{\mathrm{th}} \colon \R^{m} \setminus Z^{\ell^{*}} \to \R^{m} \) be the map given by Proposition~\ref{propositionthickeningfromaskeleton} and take
\begin{equation}
\label{eqJEMS-uth}
u_{\eta}^{\mathrm{th}} = u_{\eta}^{\mathrm{sm}} \compose \Phi^{\mathrm{th}}
\quad \text{in \((S^m + Q^m_{\rho\eta/2}) \setminus Z^{\ell^{*}} \).}{}
\end{equation}
Figure~\ref{figureJEMSThickeningCube} illustrates thickening applied to the 1-dimensional skeleton of the
bad cubes in our model case \(\R^{2}\), where the dual skeleton \( \cZ^{\ell^{*}} \) is given by the centers of the bad cubes.

\begin{figure}
\centering{}
\hfill\includegraphics{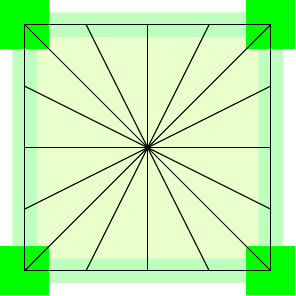}
\hfill\includegraphics{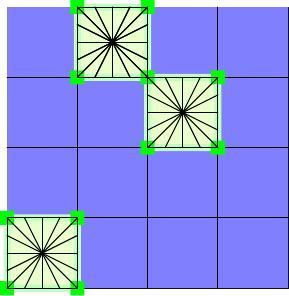}
\hfill{}
\caption{\(u_{\eta}^\mathrm{th}\) is singular at the center of each cube}
\label{figureJEMSThickeningCube}
\end{figure}

It follows from Corollary~\ref{corollaryThickeningCubication} applied to \(u^\mathrm{sm}_\eta\) that \( u_{\eta}^{\mathrm{th}} \) is an \(\ClassR_{\ell^{*}}\)~map with singular set \(Z^{\ell^{*}}\)   and, for every \(j \in \{1, \dots, k\}\),
\begin{equation}
\label{inequalityMainThickening}
\norm{D^j u^\mathrm{th}_\eta - D^j u^\mathrm{sm}_\eta}_{\Lebesgue^p(\Int{S^{m}})} 
\le {\C } \sum_{i = 1}^j \eta^{i-j} \norm{D^i u^\mathrm{sm}_\eta}_{\Lebesgue^p(U^m + Q^m_{\rho\eta/2})}.
\end{equation}
As we show in \cite{BPVS_MO}*{p.\,794}, by a combination of \eqref{inequalityMainOpening}, \eqref{inequalityMainSmoothening} and \eqref{inequalityMainThickening}, we then get
\begin{multline}
\label{eqStrong2-289}
\norm{D^j u^\mathrm{th}_\eta - D^j u}_{\Lebesgue^p(\Int{S^{m}})}
\le \C  \sup_{v \in \Ball^{m}}{\norm{\tau_{\psi_\eta v}(D^j u) - D^j u}_{\Lebesgue^p(\Int{S^{m}})}}\\
+ \C  \sum_{i = 1}^j \eta^{i-j} \norm{D^i u}_{\Lebesgue^p(U^m_\eta + Q^m_{2\rho\eta})},
\end{multline}
where we explicit the dependence of \(U^m\) with respect to \(\eta\).
Since \eqref{eq-claim_measure_bad} holds, we may apply Lemma~\ref{lemmaOpeningApproximation} to deduce that
\begin{equation}
	\label{eqJEMS-964}
\lim_{\eta \to 0} \sum_{i=1}^j  \eta^{i-j} \norm{D^i u}_{\Lebesgue^p(U^m_\eta + Q^m_{2\rho\eta})} = 0.
\end{equation}
We now combine \eqref{eqStrong2-289} and \eqref{eqJEMS-964}.
Since \(\psi_\eta \le c\eta\),  we then have that, for every \(j \in \{1, \dots, k\}\),
\begin{equation}
\label{eqStrong2-385}
\lim_{\eta \to 0}{\norm{D^j u^\mathrm{th}_\eta - D^j u}}_{\Lebesgue^{p}(\Int{S^{m}})}
= 0.
\end{equation}
Finally, from properties~\((\ref{itempropositionthickeningfromaskeleton1})\) and~\((\ref{itempropositionthickeningfromaskeleton2})\) in Proposition~\ref{propositionthickeningfromaskeleton}, we have
\[
\Phi^{\mathrm{th}}((S^m\setminus U^m)\setminus Z^{\ell^{*}}) \subset S^m\setminus U^m
\quad \text{and} \quad
\Phi^{\mathrm{th}}(U^m\setminus Z^{\ell^*}) \subset U^\ell + Q^{m}_{\rho\eta/2}\,.
\]
It thus follows from inclusions \eqref{eqJEMS-Inclusion} and \eqref{eqJEMS-Inclusion-bis} that
\begin{equation}\label{eq-target-uth}
u^\mathrm{th}_\eta(S^m\setminus Z^{\ell^*})\subset \manfN+B^{\nu}_{\iota}.
\end{equation}
In particular, we can define the composition \(\Pi\compose u^\mathrm{th}_\eta\) in \(S^m\setminus Z^{\ell^*}\).
This concludes Part~\ref{partThmOpen-5} of the proof of Theorem~\ref{thm_density_manifold_open} for open sets.
\end{proof}

So far, the proof of Theorem~\ref{thm_density_manifold_open} follows along the lines of the strong density of \(\ClassR_{\ell^*}\)~maps in \(\Sobolev^{k, p}(\manfV; \manfN)\) with \(\ell = \floor{kp}\) from \cite{BPVS_MO}.
There, to prove the density of smooth maps, we next rely on the assumption \( \pi_{\floor{kp}}(\manfN) \simeq \{ 0 \}  \), which need not be satisfied here.
To pursue our proof, we now focus on the connection between \( (\ell, e) \)-extendability and the existence of a continuous extension.

\begin{PartThm}
    \label{partThmOpen-4}
    Relying on the map \(u_\eta^\mathrm{fop}\) introduced in Part~\ref{partThmOpen-2}, we show that \(\Pi \compose u^{\mathrm{sm}}_{\eta}|_{S^{\ell}}\) can be continuously extended to \(S^{e}\).
\end{PartThm}

\begin{proof}[\nopunct]
We begin by recalling that \(u^{\mathrm{sm}}_{\eta}=\varphi_{\psi_\eta} * u^{\mathrm{op}}_\eta\) for some suitable functions \(\varphi\) and \({\psi_\eta}\) introduced in Part~\ref{partThmOpen-3}. 
Using the same functions \(\varphi\) and \({\psi_\eta}\), we then define accordingly the map 
\begin{equation}
\label{eqJEMS-ufsm}
u^{\mathrm{fsm}}_{\eta} =\varphi_{\psi_\eta} * u^{\mathrm{fop}}_\eta
\quad \text{in \( S^{m} + Q_{\rho\eta}^{m} \)\,,}
\end{equation}
where \(u^{\mathrm{fop}}_\eta\) is constructed by applying opening in a neighborhood of the full skeleton \(S^\ell\), see Figure~\ref{figureCubicationDensityNew}.
Although \(u^{\mathrm{sm}}_{\eta}\) and \( u^{\mathrm{fsm}}_{\eta} \) need not be close with respect to the \( \Sobolev^{k, p} \)~distance, we claim that they satisfy the pointwise estimate
\begin{equation}\label{eq296}
\abs{u^{\mathrm{sm}}_\eta - u^{\mathrm{fsm}}_\eta}
\leq \iota{}
\quad \text{in \(S^{m}+ Q_{\rho\eta/2}^{m}\)\,.}
\end{equation}
We also have the following counterpart of \eqref{eqJEMS-Inclusion-bis} for \(u^{\mathrm{fop}}_{\eta}\), 
\begin{equation}
    \label{eqStrong2-470}
(\varphi_{r{\psi_\eta}} * u^{\mathrm{fop}}_\eta) (S^{\ell} + Q_{\rho\eta/2}^{m})
\subset \manfN + B_{\iota}^{\nu}
\quad \text{for every \( 0 < r \le 1 \)}.
\end{equation}
In particular, 
\begin{equation}
\label{eqStrong2-367}
u^{\mathrm{fsm}}_{\eta} (S^{\ell} + Q_{\rho\eta/2}^{m})
\subset \manfN + B_{\iota}^{\nu}.
\end{equation}
We later rely on \eqref{eq296} and \eqref{eqStrong2-470} to establish the existence of a continuous extension for \(\Pi \compose u^{\mathrm{sm}}_{\eta}|_{S^{\ell}}\) to \( S^{e} \), based on the homotopy extension property.
It is unclear whether the analogue of \eqref{eqStrong2-470} holds for the map \(u^{\mathrm{op}}_\eta\) in the set \((S^{\ell} \setminus U^{\ell}) + Q_{\rho\eta/2}^{m}\) and justifies the utility of \(u^{\mathrm{fop}}_\eta\).{}

Let us first justify \eqref{eq296}.
Since \eqref{eqDensityNew-142} holds, we have
\begin{equation}
\label{eqJEMSb-777}
u_{\eta}^{\mathrm{op}} = u_{\eta}^{\mathrm{fop}}
\quad \text{in \(U^{m} + Q^{m}_{\rho\eta}\)\,.}  
\end{equation}
As \(\psi_\eta \le \rho\eta/4\), we get  
\[{}
u_{\eta}^{\mathrm{sm}} = u_{\eta}^{\mathrm{fsm}}
\quad \text{in \(U^{m}+ Q^{m}_{3\rho\eta/4}\)\,.}
\] 
Recall that \(\psi_\eta = c\eta\) in \((S^{m}+ Q^{m}_{\rho\eta})\setminus (U^m+Q^{m}_{3\rho\eta/4})\) for some constant \(0<c<\rho\) independent of \(\eta\). 
Given \(\sigma^m\in \cS^{m} \setminus \cU^{m}\), by a linear change of variables for every \(x\in (\sigma^m+ Q^{m}_{\rho\eta/2})\setminus (U^m+Q^{m}_{3\rho\eta/4})\) we have
\[
\begin{split}
\abs{u_{\eta}^{\mathrm{sm}}(x) -  u_{\eta}^{\mathrm{fsm}}(x)} 
& \leq \int_{Q_{c\eta}^{m}} \varphi(z) \abs{u^{\mathrm{op}}_\eta(x+c\eta z) - u^{\mathrm{fop}}_\eta(x+c\eta z)}\dif z\\
& \leq \C\fint_{\sigma^m+Q^{m}_{(c+\rho/2)\eta}} \abs{u^{\mathrm{op}}_\eta(y)-u^{\mathrm{fop}}_\eta(y)}\dif y.
\end{split}
\]
Since \(u^{\mathrm{op}}_\eta=u=u^{\mathrm{fop}}_\eta\) in \(\sigma^m\setminus (\partial \sigma^m+Q^{m}_{2\rho\eta})\), the Poincar\'e inequality implies that
\[
\abs{u_{\eta}^{\mathrm{sm}}(x) -  u_{\eta}^{\mathrm{fsm}}(x)} 
\le \frac{\Cl{cteStrong2-415} }{\eta^{\frac{m}{kp} - 1}} \norm{Du^{\mathrm{op}}_\eta - Du_{\eta}^{\mathrm{fop}}}_{\Lebesgue^{kp}(\sigma^m+Q^{m}_{2\rho\eta})}.
\]
Since \(u_{\eta}^{\mathrm{fop}}=u_{\eta}^{\mathrm{op}}=u\) outside \(S^\ell+Q^{m}_{2\rho\eta}\), it then follows from the triangle inequality that
\[
\abs{u_{\eta}^{\mathrm{sm}}(x) -  u_{\eta}^{\mathrm{fsm}}(x)} 
\le \frac{\Cr{cteStrong2-415}}{\eta^{\frac{m}{kp}-1}}\left( \norm{Du_{\eta}^{\mathrm{op}}}_{\Lebesgue^{kp}(\sigma^\ell+Q^{m}_{2\rho\eta})} + \norm{Du_{\eta}^{\mathrm{fop}}}_{\Lebesgue^{kp}(\sigma^\ell+Q^{m}_{2\rho\eta})} \right),
\]
where \(\sigma^\ell\) denotes the \(\ell\)-dimensional skeleton of \(\sigma^m\).
Applying \eqref{eq-Addendum1} and its counterpart for \(u_{\eta}^{\mathrm{fop}}\), this gives
\[
\abs{u_{\eta}^{\mathrm{sm}}(x) -  u_{\eta}^{\mathrm{fsm}}(x)} 
\leq\frac{\Cl{cteStrong2-422} }{\eta^{\frac{m}{kp}-1}}\norm{Du}_{\Lebesgue^{kp}(\sigma^m+Q^{m}_{2\rho\eta})}.
\]
Since \(\sigma^{m}\) is a good cube, estimate \eqref{eqJEMSGoodBall} holds and we deduce that
\[
\abs{u_{\eta}^{\mathrm{sm}}(x) -  u_{\eta}^{\mathrm{fsm}}(x)} 
\leq \frac{\Cr{cteStrong2-422}}{\alpha} \iota.
\]
By increasing \(\alpha\) if necessary, one can assume that \(\Cr{cteStrong2-422}/\alpha \le 1\).  
Estimate \eqref{eq296} is thus satisfied.

By \eqref{eqJEMS-Inclusion-bis} and \eqref{eqJEMSb-777}, to prove \eqref{eqStrong2-470} it suffices to show that
\begin{equation}
\label{eqStrong2-436}
(\varphi_{r{\psi_\eta}} * u^{\mathrm{fop}}_\eta) ((S^{\ell}\setminus U^{\ell}) + Q_{\rho\eta/2}^{m})
\subset \manfN + B_{\iota}^{\nu}
\quad \text{for every \( 0 < r \le 1 \)}.
\end{equation}
By the counterpart of \eqref{eqStrong2-227} for \(u_{\eta}^{\mathrm{fop}}\),
\begin{equation}
\label{eqStrong2-443}
d(\varphi_{r{\psi_\eta}} * u^{\mathrm{fop}}_\eta(x), \manfN)
\le \C   \fint_{B_{r{\psi_\eta}(x)}^m(x)}\fint_{B_{r{\psi_\eta}(x)}^m(x)} \abs{u^\mathrm{fop}_\eta(y) - u^\mathrm{fop}_\eta(z)} \dif z\dif y.
\end{equation}
Given a cube \( \sigma^{m} \in \cS^{m}\setminus\cU^{m} \), denote as before by \(\sigma^{\ell}\) its \(\ell\)-dimensional skeleton.
For every \( x \in \sigma^{\ell} + Q_{\rho\eta/2}^{m} \) and every \( 0 < r\leq  1 \), we have \( B_{r{\psi_\eta}(x)}^m(x) \subset Q^{m}_{r{\psi_\eta}(x)}(x) \).
Since \(\ell \le kp\), it then follows from \eqref{eqStrong2-443} and Proposition~\ref{propositionOpeningVMO} applied with \(S^\ell\) instead of \(U^\ell\) that
\[
d(\varphi_{r{\psi_\eta}} * u^{\mathrm{fop}}_\eta(x), \manfN)
\le \frac{\C }{\eta^{\frac{m}{kp}-1}}\norm{Du}_{\Lebesgue^{kp}(\sigma^m+Q^{m}_{2\rho\eta})}.
\]
Since \( \sigma^{m} \) is a good cube, we obtain \eqref{eqStrong2-436} by choosing \( \alpha \) larger if necessary.
We conclude that the assertions~\eqref{eq296} and~\eqref{eqStrong2-470} are satisfied.

In order to prove that \(\Pi \compose u^{\mathrm{sm}}_{\eta}|_{S^{\ell}} \) has a continuous extension to \( S^{e} \), we first establish that
\begin{equation}
\label{eqStrong2-480}
\Pi \compose u^{\mathrm{fsm}}_{\eta}
\sim u_{\eta}^{\mathrm{fop}}
\quad \text{in \(\VMO^{\ell}(S^{\ell} + Q_{\rho\eta/2}^m; \manfN) \),}
\end{equation}
where the \(\VMO^{\ell}\)~homotopy is introduced in Definition~\ref{definitionVMOEllHomotopy}.
Recall that the nearest point projection \(\Pi\) is defined and smooth in the tubular neighborhood \(\overline{\manfN+B^{\nu}_{\iota}}\). By taking \(\iota\) smaller, we may assume that this is the case even in the larger set \(\overline{\manfN+B^{\nu}_{2\iota}}\)\,.
As a consequence of \eqref{eqStrong2-367}, the map \(\Pi \compose u^{\mathrm{fsm}}_{\eta}\) is well defined and smooth in \(S^{\ell} + Q_{\rho\eta/2}^{m}\)\,.{}

We now show \eqref{eqStrong2-480}.
Take a sequence \((r_{j})_{j \in \N}\) in \((0, 1]\) that converges to \(0\) and let 
\[
v_{j} = \Pi \compose (\varphi_{r_{j}{\psi_\eta}} * u^{\mathrm{fop}}_{\eta})
\quad  \text{in \(S^{\ell} + Q_{\rho\eta/2}^{m}\)\,,}
\]
which is well defined and smooth by \eqref{eqStrong2-470}.
By Proposition~\ref{propositionConvolutionEstimates}, we have
	\[
	v_j \to \Pi \compose u^{\mathrm{fop}}_{\eta} 
    = u^{\mathrm{fop}}_{\eta}
    \quad \text{in  \(\Sobolev^{1, kp}(S^\ell+Q^{m}_{\rho\eta/2}; \manfN)\)}.
	\]
	Since \(kp \ge \ell\), it follows from Proposition~\ref{lemmaFugledeSobolevDetector} that there exists a subsequence \((v_{j_{i}})_{i \in \N}\) such that
    \begin{equation}
    \label{eqStrong2-491}
    v_{j_{i}} \to u^{\mathrm{fop}}_{\eta}
    \quad \text{in  \(\VMO^{\ell}(S^\ell+Q^{m}_{\rho\eta/2}; \manfN)\)}.
    \end{equation}
    By \eqref{eqStrong2-470}, the map
    \[
    r \in (0, 1] \longmapsto \Pi \compose (\varphi_{r {\psi_\eta}} * u^{\mathrm{fop}}_{\eta}) \in \Smooth^{0}(S^\ell+Q^{m}_{\rho\eta/2}; \manfN)
    \]
    is well defined and continuous.
    Thus, for every \( j \in \N \), we have
    \begin{equation}
    \label{eqStrong2-502}
    v_{j} \sim \Pi \compose  (\varphi_{{\psi_\eta}} * u^{\mathrm{fop}}_\eta) = \Pi \compose u^{\mathrm{fsm}}_\eta
    \quad \text{in  \(\Smooth^{0}(S^\ell+Q^{m}_{\rho\eta/2}; \manfN)\)}
    \end{equation}
    and in particular also in \( \VMO^{\ell}(S^\ell+Q^{m}_{\rho\eta/2}; \manfN) \).
    The homotopy relation \eqref{eqStrong2-480} then follows from \eqref{eqStrong2-491}, \eqref{eqStrong2-502} and Proposition~\ref{propositionHomotopyVMOEll}.
    
    Using \eqref{eqStrong2-480}, we now conclude the proof that \(\Pi \compose u^{\mathrm{sm}}_{\eta}|_{S^{\ell}} \) has a continuous extension to \( S^{e} \).
    We first observe that
	\begin{equation}
    \label{eqStrong2-515}
    \Pi \compose u^{\mathrm{sm}}_{\eta}
    \sim \Pi \compose u^{\mathrm{fsm}}_{\eta}
	\quad \text{in \(\Smooth^{0}(S^{\ell} + Q_{\rho\eta/2}^{m}; \manfN)\),}
	\end{equation}
	that follows from the homotopy
	\[{}
	t \in [0, 1] \longmapsto \Pi \compose \bigl( (1 - t) u^{\mathrm{sm}}_{\eta} + t u^{\mathrm{fsm}}_{\eta} \bigr)|_{S^{\ell}}\,.
	\]
	That this homotopy is well defined is a consequence of \eqref{eq296}.
    Combining \eqref{eqStrong2-480} and \eqref{eqStrong2-515}, we then obtain 
\begin{equation}
\label{eqStrongPart5Goal}
\Pi \compose u^{\mathrm{sm}}_{\eta}
\sim u_{\eta}^{\mathrm{fop}}
\quad \text{in \(\VMO^{\ell}(S^{\ell} + Q_{\rho\eta/2}^m; \manfN) \). }
\end{equation}

We now take a simplicial complex \( \cK^{e} \) such that \( K^{e} = S^{e} \) and \( K^{\ell} = S^{\ell} \), which can be constructed using a baricentric subdivision of \(S^e\).
Since \( \Pi \compose u^{\mathrm{sm}}_{\eta}\) is continuous in \( S^{\ell} + Q_{\rho\eta/2}^m \)\,, the map \(u_{\eta}^{\mathrm{fop}}\) is \( (\ell, e) \)-extendable in \( S^{m}+ Q_{\rho\eta/2}^m \) and \eqref{eqStrongPart5Goal} is satisfied, it then follows from Proposition~\ref{propositionEllEExtensionContinuous}, using as \( \gamma \) the identity map in \( K^{e} = S^{e} \), that \(\Pi \compose u^{\mathrm{sm}}_{\eta}|_{S^{\ell}} \) has a continuous extension to \( S^{e} \).
Since \(  u_{\eta}^{\mathrm{sm}} = u_{\eta}^{\mathrm{th}} \) in \( S^{\ell} \), the same property is true for \(\Pi \compose u_{\eta}^{\mathrm{th}}|_{S^{\ell}} \)\,.
Denoting by \( \cT^{e^{*}} \) and \( \cT^{\ell^{*}} \) the dual skeletons of \( \cS^{e} \) and \( \cS^{\ell} \), respectively, Proposition~\ref{propositionDensityNewExtensionSmooth} applied to \(\Pi \compose u_{\eta}^{\mathrm{th}}\) with \( 0 < \mu < 1/2 \) then gives us a smooth map
\begin{equation}
    \label{eqJEMS-uex}
    u_{\eta, \mu}^{\mathrm{ex}} \colon S^{m} \setminus T^{e^{*}} \to \manfN
\end{equation} 
such that \(u_{\eta, \mu}^{\mathrm{ex}} = \Pi \compose u_{\eta}^{\mathrm{th}}\) in \((\Int{S^m}) \setminus (T^{\ell^*} + Q^m_{\mu\eta})\).
This concludes Part~\ref{partThmOpen-4} of the proof of Theorem~\ref{thm_density_manifold_open} for open sets.
\end{proof}

Although \(u_{\eta, \mu}^{\mathrm{ex}}\) is an \(\ClassR_{e^*}\)~map in \(\Int{S^m}\), it is not suitable for the approximation problem in the \(\Sobolev^{k, p}\)~scale as we have no \(\Lebesgue^p\)~estimate of its derivatives in a neighborhood of \(T^{\ell^{*}}\).
To remedy this issue, the next step of the proof relies on shrinking that is explained in Section~\ref{section_shrinking}.

\begin{figure}
\centering{}
\hfill\includegraphics{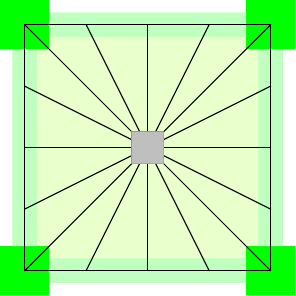}
\hfill\includegraphics{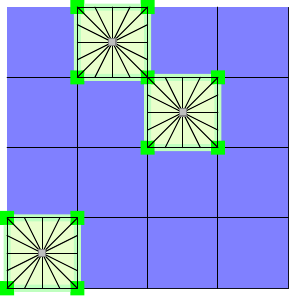}
\hfill{}
\caption{\(u_{\eta, \tau, \mu}^{\mathrm{sh}}\) is smooth in each cube}
\label{figureJEMSShrinkingCube}
\end{figure}

\begin{PartThm}
\label{partThmOpen-6}
We apply shrinking to replace \(u_{\eta, \mu}^{\mathrm{ex}}\) by a smooth map \(u_{\eta, \tau, \mu}^{\mathrm{sh}} \colon (\Int{S^m}) \setminus T^{e^*} \to \manfN\) whose derivatives on a neighborhood of  \(T^{\ell^*}\) provide only a small contribution in the \(\Lebesgue^p\)~scale.
\end{PartThm}

\begin{proof}[\nopunct]
Given \( 0 < \tau < \mu \), let \(\Phi^{\mathrm{sh}} \colon \R^{m} \to \R^{m} \) be the diffeomorphism given by Proposition~\ref{lemmaShrinkingFaceNearDualSkeletonGlobal}.
We then denote
\begin{equation}
\label{eqJEMS-ush}
u_{\eta, \tau, \mu}^{\mathrm{sh}}
= u_{\eta, \mu}^{\mathrm{ex}} \compose \Phi^{\mathrm{sh}}
\quad \text{in \((\Int{S^m}) \setminus T^{e^*}\).}
\end{equation}
By Corollary~\ref{corollaryShrinkingCubication}, we have that \(u_{\eta, \tau, \mu}^{\mathrm{sh}}\) belongs to \( \ClassR_{e^{*}}(\Int{S^m}; \manfN)\) with singular set \(T^{e^{*}}\).
As illustrated in Figure~\ref{figureJEMSShrinkingCube}, when \( T^{e^{*}} = \emptyset \) the map \( u_{\eta, \tau, \mu}^{\mathrm{sh}} \) is smooth in the entire domain \(\Int{S^m}\).

Since 
\[
u_{\eta, \tau, \mu}^{\mathrm{sh}} 
= u_{\eta, \mu}^{\mathrm{ex}} 
= \Pi \compose u_{\eta}^{\mathrm{th}}
\quad \text{in \((\Int{S^m}) \setminus (T^{\ell^*} + Q^m_{2\mu\eta})\),}
\]
by the triangle inequality for every \( j \in \{1, \dots, k \} \) we have
\begin{multline*}
\norm{D^{j} u_{\eta, \tau, \mu}^{\mathrm{sh}} - D^{j}(\Pi \compose u_{\eta}^{\mathrm{th}}) }_{\Lebesgue^p(\Int{S^m})}\\
\leq  \norm{D^{j} u_{\eta, \tau, \mu}^{\mathrm{sh}}}_{\Lebesgue^p((\Int{S^m}) \cap (T^{\ell^*} + Q^m_{2\mu\eta}))}
+ \norm{D^{j}(\Pi \compose u_{\eta}^{\mathrm{th}})}_{\Lebesgue^p((\Int{S^m}) \cap (T^{\ell^*} + Q^m_{2\mu\eta}))}.
\end{multline*}
By Corollary~\ref{corollaryShrinkingCubication},
\begin{multline}
\label{eqStrong2-594}
\norm{D^j u_{\eta, \tau, \mu}^{\mathrm{sh}}}_{\Lebesgue^p((\Int{S^m}) \cap (T_\eta^{\ell^*} + Q^m_{2\mu\eta}))}\\
 \leq 
 \C  \sum_{i=1}^j{  (\mu\eta)^{i - j} \norm{D^i(\Pi \compose u_{\eta}^{\mathrm{th}})}_{\Lebesgue^p((\Int{S^m}) \cap (T_\eta^{\ell^*} + Q^m_{2\mu\eta}) \setminus (T_\eta^{\ell^*} + Q^m_{\mu\eta}))}}\\
+ \C  \tau^{\frac{\ell+1-kp}{p}} \sum_{i=1}^j{(\mu\eta)^{i - j} \norm{D^i u_{\eta, \mu}^{\mathrm{ex}}}_{\Lebesgue^p((\Int{S^m}) \cap (T_\eta^{\ell^*} + Q^m_{\mu\eta}))}},
\end{multline}
where we now make explicit the dependence of \(T^{\ell^*}\) on \(\eta\).
By the Gagliardo-Nirenberg interpolation inequality, we have \( D^{i}(\Pi \compose u_{\eta}^{\mathrm{th}}) \in \Lebesgue^{kp/i}(\Int{S^m}) \).
Since \( \ell+1 > kp \) and there exists a constant \(\Cr{cteJEMS-605} > 0\) depending on \(\eta\) such that
\[
\abs{T_\eta^{\ell^*} + Q^m_{\mu\eta}} \le \Cl{cteJEMS-605}\mu^{\ell + 1},
\]
by Lemma~\ref{lemmaOpeningApproximation} we have
\[
\lim_{\mu \to 0}{(\mu\eta)^{i - j} \norm{D^i(\Pi \compose u_{\eta}^{\mathrm{th}})}_{\Lebesgue^p((\Int{S^m}) \cap (T_\eta^{\ell^*} + Q^m_{2\mu\eta}) \setminus (T_\eta^{\ell^*} + Q^m_{\mu\eta}))}} = 0.
\]
Using again \( \ell + 1 - kp > 0 \), the second term on the right-hand side of \eqref{eqStrong2-594} converges to zero as \( \tau \to 0 \), with \(\eta\) and \(\mu\) fixed.
Hence, by a diagonalization argument, there exists a family  \((\tau_\mu)_{\mu\searrow 0}\) converging to zero such that, as \(\mu \to 0\),
\[
u_{\eta, \tau_\mu, \mu}^{\mathrm{sh}} \to  \Pi \compose u_{\eta}^{\mathrm{th}}
\quad \text{in \(\Sobolev^{k, p}(\Int{S^m}; \manfN)\).}
\] 
Combining this with \eqref{eqStrong2-385}, one gets a further family \((\mu_\eta)_{\eta\searrow 0}\) converging to zero such that, as \(\eta \to 0\),
\[
u_{\eta, \tau_{\mu_\eta}, \mu_\eta}^{\mathrm{sh}} \to u
\quad \text{in \(\Sobolev^{k, p}(\Int{S^m}; \manfN)\).}
\]
This concludes Part~\ref{partThmOpen-6} of the proof of Theorem~\ref{thm_density_manifold_open} for open sets.
\end{proof}

\begin{PartThm}
\label{partThmOpen-7}
 We explain why every \((\ell,e)\)-extendable map \(u\in \Sobolev^{k,p}(\Omega; \manfN)\) can be approximated by \((\ell,e)\)-extendable maps defined in larger open sets \(O \Supset \Omega\).
\end{PartThm}

\begin{proof}[\nopunct]
By Lipschitz regularity of \( \Omega \), there exists a sequence of diffeomorphisms \(\Phi_s  \colon  \R^m \to \R^m\) with \(\Phi_{s}(\overline{\Omega}) \subset \Omega\) such that,  for every \(j \in \N \),
\begin{equation}
\label{item-enlarging-convergence}
D^{j}\Phi_{s} \to D^{j} \Id
\quad \text{uniformly in \(\R^{m}\),}
\end{equation}
see e.g.~\cite{Detaille}*{Lemma~6.2}.
Since \(u\in \Sobolev^{k, p}(\Omega; \manfN)\), the Faá di Bruno composition formula and \eqref{item-enlarging-convergence} imply that \((u\compose \Phi_s\vert_{\Omega})_{s\in \N}\) converges to \(u\) in \(\Sobolev^{k,p}(\Omega; \manfN)\). 
Moreover,  \(u\compose \Phi_s\) is defined in \(O_{s} \vcentcolon= \Phi_{s}^{-1}(\Omega)\) which is an open set containing \(\overline{\Omega}\).{}
Let \( w \) be an \( \ell \)-detector for \( u \) given by Definition~\ref{definitionExtensionVMO}.
Since \( \Phi_{s} \) is a diffeormorphism with \( \Phi_s(O_s) = \Omega  \), we have \( \Phi_{s} \in \Fuglede_{w}(O_{s}; \Omega) \).
By Proposition~\ref{corollary_ouverture}, the map \(u\compose \Phi_s\) is \( (\ell, e) \)-extendable in \( O_{s} \)\,.

Since every map \(u\compose \Phi_s\vert_{\Omega}\) can be approximated in \(\Sobolev^{k,p}(\Omega; \manfN)\) by a sequence in \(\ClassR_{e^{*}}(\Omega; \manfN)\), we can conclude by a diagonalization argument that \(u\) itself can be approximated by such maps. 
This completes the proof of the converse part in Theorem~\ref{thm_density_manifold_open} for open sets.
\end{proof}

Instead of approximating \(u\) by a sequence of functions defined in a larger set, one could try to follow an alternative approach by extending \(u\) outside \(\Omega\). 
More specifically, given \(k\geq 2\) and \(u\in \Sobolev^{k, p}(\Omega; \manfN)\), one can wonder whether there exist an open set \(O \Supset \Omega\) and \(v\in \Sobolev^{k, p}(O; \manfN)\) such that \(v\vert_\Omega=u\) in \(\Omega\). 
Even when \(\Omega= \Ball^{m}\), we do not know the answer to this problem. 
Observe that the solution is straightforward when \(k = 1\) by using a reflection near the boundary.

\begin{remark}\label{rk-equivalence-of-definition}
In Definition~\ref{definitionHkp}, we define the space \(\Hilbert^{k,p}(\Omega;\manfN)\) as the set of all \(\Sobolev^{k,p}\) limits of sequences \((u_{j})_{j \in \N}\) in \((\Smooth^{\infty}\cap \Sobolev^{k, p})(\Omega; \manfN)\). 
Since \(\Hilbert^{k,p}(\Omega;\manfN)\) is closed in \(\Sobolev^{k,p}(\Omega;\manfN)\), it contains the \(\Sobolev^{k,p}\) closure of \(\Smooth^{\infty}(\overline{\Omega};\manfN)\). 
The reverse inclusion holds when \(\Omega\) is Lipschitz since the sequence \((\Phi_s)_{s \in \N}\) of diffeomorphisms from \(\R^m\) into itself that satisfies \(\Phi_s(\overline{\Omega})\subset \Omega\) and \eqref{item-enlarging-convergence}
is such that any \(u\in (\Smooth^{\infty}\cap \Sobolev^{k, p})(\Omega; \manfN)\) can be approximated in \(\Sobolev^{k,p}(\Omega;\manfN)\) by the maps  \(u\compose \Phi_s\in \Smooth^{\infty}(\overline{\Omega};\manfN)\).
\end{remark}

\section{Genericity in a neighborhood of \texorpdfstring{\( \manfM \)}{M}}

In this section, we consider a smooth compact manifold \(\manfM\) of dimension \( m \) without boundary and isometrically imbedded in the Euclidean space \(\R^\varkappa\), for some \(\varkappa\geq m\geq 1\).
    We may take in this case the nearest point projection  \(\widetilde{\Pi} \colon  \overline{\Omega} \to \manfM\) using a bounded Lipschitz open neighborhood \(\Omega\) of \( \manfM \) so that \(\widetilde{\Pi}\) is a smooth submersion.
    This projection allows one to transform a Sobolev function \(u\) in \(\manfM\) into a Sobolev function \( u \compose \widetilde{\Pi} \) in \(\Omega\).
    We first show that generic extendability is preserved in this process:

\begin{proposition}
    \label{propositionExtensionProjection}
    If \(u \in \Sobolev^{k, p}(\manfM; \manfN)\) is \((\ell, e)\)-extendable for some \(\ell \le \min{\{kp, e\}}\), then \( u \compose \widetilde{\Pi} \) is also \((\ell, e)\)-extendable.
\end{proposition}

\resetconstant
\begin{proof}
Since \(\widetilde{\Pi}\) is a smooth submersion, we have \( u \compose \widetilde{\Pi} \in \Sobolev^{k, p}(\Omega; \manfN)\).
Since \( \manfN \) is compact, \( u \compose \widetilde{\Pi} \) is bounded, and then by the Gagliardo-Nirenberg interpolation inequality, \( u \compose \widetilde{\Pi} \in \Sobolev^{1, kp}(\Omega; \manfN)\).
By Proposition~\ref{propositionSobolevApproximationFuglede}, we deduce that \( u \compose \widetilde{\Pi} \) is a \( \VMO^{\ell} \)~map for every \( \ell \le kp \).

Let \(w  \colon  \manfM\to [0,+\infty]\) be an \(\ell\)-detector for \(u\) given by Definition~\ref{definitionExtensionVMO}. 
Then, by the coarea formula,
\[
\begin{split}
\int_{\Omega} w \compose \widetilde{\Pi} 
\leq \Cl{cteDensityNew-861} \int_{\Omega} ( w\compose \widetilde{\Pi}) \, \Jacobian{\varkappa}{\widetilde{\Pi}}
&= \Cr{cteDensityNew-861} \int_{\manfM} w(x) \, \cH^{\varkappa - m}(\widetilde{\Pi}^{-1}(\{x\}))\dif\cH^{m}(x)\\
& \leq \C \int_{\manfM} w(x) \dif\cH^{m}(x)< +\infty.
\end{split}
\]
Hence, \(w\compose \widetilde{\Pi}\) is summable in \( \Omega \). 
To see that \(w\compose \widetilde{\Pi}\) is an \(\ell\)-detector for \(u \compose \widetilde{\Pi}\), take a simplicial complex  \(\cK^{\ell}\) and \(\gamma \in \Fuglede_{w\compose \widetilde{\Pi}}(K^{\ell}; \Omega)\). 
Then, \(\widetilde{\Pi} \compose \gamma \in \Fuglede_{w}(K^{\ell}; \manfM)\). 
Since \(w\) is an \(\ell\)-detector for \(u\), 
\[
(u\compose \widetilde{\Pi})\compose \gamma = u\compose (\widetilde{\Pi}\compose \gamma) \in \VMO(K^\ell;\manfN).
\] 
Therefore, \(w\compose \widetilde{\Pi}\) is an \(\ell\)-detector for \(u \compose \widetilde{\Pi}\).

We finally observe that \(u \compose \widetilde{\Pi}\) satisfies Definition~\ref{definitionExtensionVMO} using the \(\ell\)-detector \(w\compose \widetilde{\Pi}\). 
Let \(\cK^{e}\) be a simplicial complex and let \(\gamma  \colon  K^{e} \to \Omega\) be a Lipschitz map such that \(\gamma|_{K^{\ell}} \in \Fuglede_{w\compose \widetilde{\Pi}}(K^{\ell}; \Omega)\). Then, as above, \(\widetilde{\Pi} \compose \gamma|_{K^{\ell}} \in \Fuglede_{w}(K^{\ell}; \manfM)\).
Since \(u\) is \((\ell, e)\)-extendable, there exists \(F \in \Smooth^{0}(K^{e};  \manfN)\) such that 
\[{}
u \compose \widetilde{\Pi} \compose \gamma|_{K^{\ell}}
= u \compose (\widetilde{\Pi} \compose \gamma|_{K^{\ell}})
\sim F|_{K^{\ell}}
\quad \text{in \(\VMO(K^{\ell}; \manfN)\).}
\] 
Thus, \(u \compose \widetilde{\Pi}\) is \((\ell, e)\)-extendable.
\end{proof}

We now prove the following generic perturbation property of \( \ClassR \)~maps. 

\begin{proposition}
\label{propositionClassRGeneric}
    Let \( \Omega \subset \R^{\varkappa} \) be an open set containing \( \manfM \), let \( \delta > 0 \) be such that \( \manfM + B_{\delta}^{\varkappa} \subset \Omega \) and, for \( \xi \in \R^{\varkappa} \), let \( \tau_{\xi} \colon \R^{\varkappa} \to \R^{\varkappa} \) be the translation \( \tau_{\xi}(x) = x + \xi \).
    If \( u \) is an \(\ClassR_{\varkappa - j} \)~function in \(\Omega\) for some \( j \in \{ 0, \dots, m+1 \} \), then \( \tau_{\xi}u|_{\manfM}\) is an \( \ClassR_{m - j} \)~function in \( \manfM \) for almost every \( \xi \in B_{\delta}^{\varkappa} \).
\end{proposition}
Here, the notation \(\tau_{\xi}u\) refers to the function \(x \in \manfM \mapsto u(x + \xi)\).
When \(j=m+1\), the conclusion of the proposition means that  \( \tau_{\xi}u|_{\manfM} \in \Smooth^{\infty}(\manfM)\).

We need a couple of elementary known facts, which we present here for the sake of completeness:

\begin{lemma}\label{fact_transversality_planes}
Let \( j \in \{ 0, \dots, m +1\} \) with \(j\leq \varkappa\) and \(P\) be an affine subspace of\/ \(\R^\varkappa\) of dimension \(\varkappa-j\).
Then, for almost every \(\xi\in\R^\varkappa\), the set \((P - \xi)\) intersects \(\manfM\) cleanly. 
More precisely, \(\manfM\cap (P - \xi)\) is either empty or a smooth submanifold of \(\manfM\) of dimension \(m-j\), and in this case we have, for every \(x\in \manfM\cap (P - \xi)\), 
\[
\Tangent{x}\bigl( \manfM\cap(P - \xi)\bigr)=\Tangent{x} \manfM\cap \Tangent{x}(P - \xi).
\]  
\end{lemma} 

When \(j=m+1\), the conclusion of the lemma means that \(\manfM\cap (P - \xi)\) is empty for almost every \(\xi \in \R^{\varkappa}\). 

\begin{proof}[Proof of Lemma~\ref{fact_transversality_planes}]
The function \(f  \colon  \manfM \times \R^\varkappa \to \R^\varkappa\) defined by \( f (x, \xi) = x + \xi \) is a surjective submersion. 
Since \( P \) has dimension \( \varkappa - j \), we deduce that \(f^{-1}(P)\) is  a submanifold of \(\manfM\times \R^\varkappa\) of dimension \(m + \varkappa - j\). 
Let \(p_2  \colon  f^{-1}(P) \to \R^\varkappa\) be given by \( p_{2}(x, \xi) =  \xi \). 
Observe that
\begin{equation}\label{eq784}
p_{2}^{-1}(\xi){}
= \bigl\{ (x, \xi)\in \manfM\times \R^\varkappa : x + \xi \in P \bigr\} 
= (\manfM\cap (P-\xi))\times \{\xi\}.
\end{equation}
By Sard's lemma, almost every \(\xi\in \R^\varkappa\) is a regular value of \(p_2\). This implies that \(p_{2}^{-1}(\xi)\)  is a submanifold of \(f^{-1}(P)\) and for every \(x\in \manfM\cap (P-\xi)\), we have
\[
\Ker{Dp_2(x,\xi)}=\Tangent{(x,\xi)}p_2^{-1}(\xi).
\]
By \eqref{eq784}, the right-hand side is \(\Tangent{x}(\manfM\cap (P-\xi))\times \{0\}\), while by the definition of \(p_2\) the left-hand side is equal to 
\[
(\Tangent{x}\manfM\times \{0\}) \cap \Tangent{(x,\xi)}f^{-1}(P)=\bigl(\Tangent{x}\manfM \cap \Tangent{x}(P-\xi)\bigr)\times \{0\}.
\]
The conclusion thus holds for every such \(\xi\). 
\end{proof}

\begin{lemma}\label{fact_transversality_estimates}
Let \(P\) be an affine subspace of \(\R^\varkappa\) that intersects \(\manfM\) cleanly.
Then, there exists \(C>0\) such that, for every \(x\in \manfM\), 
\[
d(x,\manfM\cap P)
\leq C d(x, P).
\]
\end{lemma}
\begin{proof}[Proof of Lemma~\ref{fact_transversality_estimates}]
Let us assume by contradiction that there exist sequences \((x_{j})_{j \in \N}\subset \manfM\), \((y_{j})_{j \in \N} \subset P\) and 
\((z_{j})_{j \in \N} \subset \manfM \cap P\) such that 
\begin{equation}
    \label{eqStrong2-759}
    \abs{x_{j}-y_{j}}= d(x_{j}, P),
    \quad \abs{x_{j}-z_{j}}= d(x_{j}, \manfM\cap P)
\end{equation}
and
\begin{equation}\label{eq1748}
|x_{j}-z_{j}|
> j|x_{j}-y_{j}|.
\end{equation}
Since the above inequality is strict, one has \(x_{j}\neq z_{j}\), which by \eqref{eqStrong2-759} implies that \(x_{j}\not\in P\).
In particular, \(x_{j}\neq y_{j}\). 

Consider the conical region with the central axis orthogonal to \(P\) passing through \( x_{j} \) and \( y_{j} \), where \( x_j \) is the vertex of the cone and the surface of the cone passes through \( z_{j} \).
Since \( z_{j} \) is a closest point from \( \manfM \cap P \) to \( x_{j} \), there are no points of \( \manfM \cap P \) in the interior of the cone.
Hence, \( z_j \) also minimizes the distance between \( y_{j} \) and \( \manfM \cap P \),
\begin{equation*}
  |y_{j}-z_{j}|
  = d(y_{j}, \manfM\cap P).
\end{equation*}
Let \( T \) be the right triangle with vertices \(x_{j}\), \(y_{j}\) and  \(z_{j}\), where the right angle is located at \( y_{j} \).
Denoting by \( \theta \) the angle of \( T \) at \( z_{j} \), we have
\[
\sin{\theta}
= \frac{\abs{x_{j}-y_{j}}}{\abs{x_{j}-z_{j}}}
\quad \text{and} \quad
\cos{\theta}
= a_{j} \cdot b_{j}\,,
\]
where \(a_{j}\) and \(b_{j}\) are the unit vectors in \((\Tangent{z_{j}}(\manfM\cap P))^\perp\) defined by
\[
a_{j}=\frac{x_{j}-z_{j}}{\abs{x_{j}-z_{j}}}
\quad \text{and} \quad b_{j}=\frac{y_{j}-z_{j}}{\abs{y_{j}-z_{j}}}.
\]
In particular, we have
\begin{equation}\label{eq1756}
\frac{|x_{j}-y_{j}|}{|x_{j}-z_{j}|}
= \sqrt{\smash[b]{1- (a_{j} \cdot b_{j})^2}}\, .
\end{equation}
Since \(\manfM\) is compact, up to extraction of subsequences we may assume that  \((x_{j})_{j \in \N}\), \((y_{j})_{j \in \N}\) and 
\((z_{j})_{j \in \N}\) converge to some \(x\in \manfM\), \(y\in P\) and \(z\in \manfM\cap P\), and that \((a_{j})_{j \in \N}\) and \((b_{j})_{j \in \N}\) converge to unit vectors \(a, b\in (\Tangent{z}(\manfM\cap P))^\perp\), respectively. 
By \eqref{eq1748}, one has \(x=y\), which implies that \(x\in \manfM\cap P\), and thus by \eqref{eqStrong2-759} we get \(x=z\). 
Since  \(x_{j}, z_{j}\in \manfM\),  the vector \(a\) also belongs to \(\Tangent{z}\manfM\). 
Since \(P\) intersects \(\manfM\) cleanly, 
\[
\Tangent{z}(\manfM\cap P)
= \Tangent{z} \manfM \cap \Tangent{z} P.
\]
We conclude that \(a\in (\Tangent{z}\manfM \cap \Tangent{z}P)^\perp\cap \Tangent{z}\manfM\). 
Similarly, \(b\in (\Tangent{z}\manfM\cap \Tangent{z}P)^\perp\cap  \Tangent{z} P\).
It follows that \(a\) and \(b\) are not colinear.   
By \eqref{eq1748}, the left-hand side of \eqref{eq1756} converges to \(0\) while the right-hand side converges to \(\sqrt{1-(a \cdot b)^2}\), which is strictly positive since the two unit vectors \(a\) and \(b\) are not colinear. 
This yields the desired contradiction.
\end{proof}

\begin{proof}[Proof of Proposition~\ref{propositionClassRGeneric}]
    By assumption, \(u \) is smooth in \(\Omega\) outside a finite union \(T^{\varkappa - j}\) of \((\varkappa - j)\)-dimensional planes, whence \(\tau_\xi u\vert_{\manfM}\) is smooth outside \(\manfM \cap (T^{\varkappa - j} - \xi)\) for every \(\xi \in B_\delta^{\varkappa}\).
    When \(j = m + 1\), it follows from Lemma~\ref{fact_transversality_planes}  that \(\manfM \cap (T^{\varkappa - j} - \xi) = \emptyset\) for almost every \(\xi\in B^{\varkappa}_{\delta}\) and then \( \tau_{\xi}u\vert_{\manfM}\) is smooth.
    When \(j \le m\), Lemma~\ref{fact_transversality_planes} applied to each intersection of planes in \(T^{\varkappa - j}\) implies that, for almost every \(\xi \in B^{\varkappa}_{\delta}\), the set \(\manfM \cap (T^{\varkappa - j} - \xi)\) is a finite union of compact manifolds of dimension \(m - j\) that intersect cleanly. 
    By Lemma~\ref{fact_transversality_estimates}, the derivatives of \( \tau_{\xi}u\vert_{\manfM} \) satisfy estimate \eqref{eqRClassDerivativeEstimatePointwise} and then \(\tau_\xi u\vert_{\manfM}\) is an \(\ClassR_{m - j}\)~function in \(\manfM\) with a singular set \(\manfM \cap (T^{\varkappa - j} - \xi)\).
\end{proof}

\section{Density on manifolds}

We now conclude the proof of the direct implication of Theorem~\ref{thm_density_manifold_open} in the remaining case where \( \manfV = \manfM \) is a compact manifold of dimension \( m \) without boundary.

\begin{proof}[Proof of Theorem~\ref{thm_density_manifold_open} when \( \manfV = \manfM \) is a compact manifold]
Let \(\ell = \floor{kp}\) and let \(u \in \Sobolev^{k, p}(\manfM; \manfN)\) be an \((\ell, e)\)-extendable map for some \(e \in \{\ell, \dots, m\}\).
We take an open neighborhood \(\Omega \subset \R^\varkappa\) of \(\manfM\) with smooth boundary such that the nearest point projection \(\widetilde{\Pi} \colon  \overline{\Omega} \to \manfM\) is a smooth submersion. 
The map \(\widetilde{u} \vcentcolon= u\compose \widetilde{\Pi}\) belongs to \(\Sobolev^{k, p}(\Omega; \manfN)\) and, by Proposition~\ref{propositionExtensionProjection}, is \( (\ell, e) \)-extendable.

Let \(\delta > 0\) be such that \(\manfM+B_{\delta}^\varkappa\subset \Omega\).
By Theorem~\ref{thm_density_manifold_open} for open sets,  there exists a sequence \((u_j)_{j\in \N}\) in \( \ClassR_{\varkappa-e-1}(\Omega; \manfN)\) with
\[{}
u_{j} \to \widetilde{u} \quad \text{in \(\Sobolev^{k, p}(\Omega; \manfN)\).}
\]
It follows from Morrey's property for \(\Sobolev^{k, p}\)~maps, see Proposition~\ref{lemmaModulusSobolevsequenceHighkp}, that there exists a subsequence \((u_{j_{i}})_{i \in \N}\) such that, for almost every \(\xi \in B^{\varkappa}_{\delta}\), 
\begin{equation}
\label{eqDensityNew-922}
\tau_{\xi}u_{j_{i}}\vert_{\manfM} \to \tau_{\xi}\widetilde{u}\vert_{\manfM}
\quad \text{in \(\Sobolev^{k,p}(\manfM; \manfN)\).}
\end{equation}

Given \(i\in \N\), by Proposition~\ref{propositionClassRGeneric} for almost every \(\xi \in B^{\varkappa}_{\delta}\) we have 
\begin{equation}
\label{eqDensityNew-930}
\tau_{\xi}u_{j_{i}}\vert_{\manfM} \in \ClassR_{m-e-1}(\manfM; \manfN).
\end{equation}
Since \(\manfM\) is compact and \(\widetilde{\Pi}\) is the identity on \(\manfM\), the map \(\tau_{\xi}\widetilde{\Pi}\) is a smooth diffeomorphism from \(\manfM\) into \(\manfM\), for every \(\xi\) sufficiently small. 
We now take \(\xi\) where the latter holds, including \eqref{eqDensityNew-922} and \eqref{eqDensityNew-930} for every \(i \in \N\).
Then, the sequence \((\tau_{\xi}u_{j_{i}} \compose (\tau_{\xi}\widetilde{\Pi})^{-1})_{i\in \N}\) is contained in \(\ClassR_{m-e-1}(\manfM;\manfN)\) and converges with respect to the \( \Sobolev^{k, p} \)~distance to 
\[
\tau_{\xi}\widetilde{u} \compose (\tau_{\xi}\widetilde{\Pi})^{-1}
= u\compose \tau_{\xi} \widetilde{\Pi}\compose (\tau_{\xi}\widetilde{\Pi})^{-1}=u.
\]
This completes the proof of the theorem.
\end{proof}

Whether \(\manfV\) is a Lipschitz open set \(\Omega\) or a compact manifold \(\manfM\), we have the following identification of \(\Hilbert^{k, p} (\manfV; \manfN)\) that exploits the fact that only the product \(kp\) intervenes in the statement of Theorem~\ref{thm_density_manifold_open} and reformulates Corollary~\ref{corollaryHkpH1kp} when \(\manfV = \manfM\):

\begin{corollary}\label{corollary_h1kp_manifold}
If \( kp < m \), then
\[
(\Sobolev^{k, p} \cap \Hilbert^{1, \floor{kp}}) (\manfV; \manfN) 
= \Hilbert^{k, p} (\manfV; \manfN).
\]
\end{corollary}
\begin{proof}
It suffices to prove the inclusion ``\( \subset \)''.
To this end, let \(u\in (\Sobolev^{k, p} \cap \Hilbert^{1,\floor{kp}}) (\manfV; \manfN)\). 
By Theorem~\ref{thm_density_manifold_open}, since \( u \in \Hilbert^{1,\floor{kp}} (\manfV; \manfN)\),  we have that \(u\) is \((\floor{kp}, m)\)-extendable. 
The direct implication of the same theorem then implies that \(u\in \Hilbert^{k,p}(\manfV; \manfN)\).
\end{proof}

Note that \( kp < \floor{kp} + 1 \) and then \( p < \frac{\floor{kp} + 1}{k} \). 
Hence,
\[
\Sobolev^{k, \frac{\floor{kp} + 1}{k}}(\manfV; \manfN)
\subset \Sobolev^{k, p}(\manfV; \manfN).
\]
We deduce from Theorem~\ref{thm_density_manifold_open} a criterion of local approximation by smooth maps, which has been already observed in \cite{Isobe2006}*{Theorem~1.1} in the study of obstructions to approximations problems of order \(1\):

\begin{corollary}
\label{corollary_more_regular}
If \(kp < m\), then
\[
\Sobolev^{k, \frac{\floor{kp} + 1}{k}}(\Ball^m; \manfN)
\subset \Hilbert^{k, p}(\Ball^m; \manfN).
\]
\end{corollary}

\begin{proof}
Take \( u \in \Sobolev^{k, \frac{\floor{kp} + 1}{k}}(\Ball^m; \manfN) \).
By Proposition~\ref{propositionLExtensionSobolev}, \( u \) is  \(\floor{kp}\)-extendable and then, by Proposition~\ref{proposition_condition_domain}, also \((\floor{kp}, m)\)-extendable since \(\Ball^m\) is topologically trivial. 
The conclusion then follows from Theorem~\ref{thm_density_manifold_open} for open sets.
\end{proof}

We can derive some consequences from Theorem~\ref{thm_density_manifold_open} for the weak density problem.
We refer to Definition~\ref{definitionHkpWeak} for the notation \(\Hilbert^{k,p}\weak(\manfV;\manfN)\)  and the precise meaning of the weak convergence of a sequence in \(\Sobolev^{k,p}(\manfV;\manfN)\).
The following proposition contains \cite{Bethuel}*{Theorem~3} concerning the weak density problem of order \( 1 \).

\begin{corollary}
\label{corollaryWeak1}
If \( kp < m\) and  \( kp\not\in \N\), then 
\[
\Hilbert^{k,p}\weak(\manfV;\manfN)
=\Hilbert^{k,p}(\manfV;\manfN).
\]
\end{corollary}

\begin{proof}
    We observe that \(\Hilbert^{k,p}(\manfV;\manfN)\subset\Hilbert^{k,p}\weak(\manfV;\manfN)\).
Moreover, in view of Corollary~\ref{corollary-weak-convergence-extendability}, any map \(u\in \Hilbert\weak^{k,p}(\manfV;\manfN)\) is \((\ell, m)\)-extendable for every \(\ell< kp\). 
In particular, if \(kp\) is not an integer, the choice \(\ell=\floor{kp}\) is admissible. 
Together with Theorem~\ref{thm_density_manifold_open}, the conclusion follows.
\end{proof}

When \(kp\) is an integer, the situation is more subtle and additional assumptions are needed.

\begin{corollary}
\label{corollaryWeak2}
If \( kp < m\) is such that \( kp \in \N\) and \(\pi_{kp}(\manfN)\simeq \{0\}\), then 
\[
\Hilbert\weak^{k,p}(\manfV;\manfN) = \Hilbert^{k,p}(\manfV;\manfN).
\]
\end{corollary}
\begin{proof}
We deduce from Corollary~\ref{corollary-weak-convergence-extendability} that every \(u\in \Hilbert\weak^{k,p}(\manfV;\manfN)\) is \((kp-1, m)\)-extendable. 
By the assumption \(\pi_{kp}(\manfN)\simeq \{0\}\), Proposition~\ref{proposition_Extension_homotop_group_decreasing} applies with \(i=kp-1\) and \(e=m\).
It then follows that \(u\) is \((kp, m)\)-extendable. 
We can conclude using Theorem~\ref{thm_density_manifold_open}.
\end{proof}

As mentioned in the Introduction, equality \(\Hilbert\weak^{k,p}(\manfV;\manfN) = \Hilbert^{k,p}(\manfV;\manfN)\) is false in general when \(kp\) is an integer. 
A natural question concerns how to detect elements of \(\Hilbert\weak^{k,p}(\manfV;\manfN)\).
Given \(\ell\in \{0, \dots, m-1\}\) and \(u\in \ClassR_{\ell^{*}}(\manfM;\manfN)\), by \cite{Hang-Lin-III}*{Theorem~5.5} when \(\manfV = \manfM\) is a compact manifold we know that \(u\in \Hilbert^{1,\ell}\weak(\manfM;\manfN)\) if and only if \(u\) is \((\ell-1,m)\)-extendable.
It is plausible that a suitable adaptation of shrinking could provide a counterpart of this property in the \( \Hilbert^{k, p}\weak \)~setting when \( k \ge 2 \) and \(kp\) is an integer.

Theorem~\ref{thm_density_manifold_open} is closely related to the regularity of minimizers for variational problems that involve constraints.
The infimum of the energy functional 
\begin{equation}\label{min-harm-map-bis}
E_{1, p} \colon u \in \Sobolev^{1, p}(\manfV;\manfN) \longmapsto \int_{\manfV} |Du|^p\dif x
\end{equation}
is zero and attained by the constant maps. 
When \(\manfV = \manfM\) is a compact manifold with non-empty boundary, it is natural to impose boundary conditions: For every \(g\in \Sobolev^{1,p}(\manfM;\manfN)\), one can minimize \(E_{1,p}\) in the subset \(\Sobolev^{1,p}_g(\manfM;\manfN)\) of maps \(u\in \Sobolev^{1,p}(\manfM;\manfN)\) so that \(u = g\) in \(\partial\manfM\) in the sense of traces. 
The minimum in this case is a critical point of \(E_{1, p}\) that is not constant, provided that \(g\) itself is not a constant map.

The minimizers of the functional \(E_{1, p}\) in \eqref{min-harm-map-bis} with \(p = 2\) are called \emph{minimizing harmonic maps}.
They are known to be smooth outside a set of Hausdorff dimension at most \(m-3\). 
This regularity result has been extended to the case \(1 < p <\infty\) in \cite{HardtLin1987} in the sense that minimizers of \(E_{1, p}\) have a singular set of Hausdorff dimension at most \(m-\floor{p}-1\). 
Observe that this is the dimension of the singular sets of maps in \(\ClassR_{m-\floor{p}-1}(\manfM;\manfN)\), which is dense in \(\Sobolev^{1, p}(\manfM;\manfN)\) by Theorem~\ref{theoremIntroductionClassR}. 

At this point, another connection arises between the minimization of \(E_{1,p}\) and the density results in \(\Sobolev^{1,p}(\manfM;\manfN)\), the \emph{gap phenomenon}.  
Analogously to \(\Sobolev^{1,p}_g(\manfM;\manfN)\), we define the set \(\Hilbert^{1,p}_g(\manfM;\manfN)\) as the strong limit in \(\Sobolev^{1,p}(\manfM;\R^{\nu})\) of sequences \((u_j)_{j\in \N}\) in \(\Smooth^{\infty}(\manfM; \manfN)\)  such that \(u_j=g\) in \(\partial M\).
It is an open problem whether it is possible to identify the elements in \(\Hilbert^{1,p}_g(\manfM;\manfN)\) in terms of \((\floor{p}, m)\)-extendability when \(p < m\).

Even in the case \(\manfM=\cBall^3\), \(\manfN=\Sphere^2\) and \(p=2\), there exist maps \(g\in \Smooth^{\infty}(\Sphere^2;\Sphere^2)\) which are homotopic to a constant and the following gap holds, see \cite{Hardt-Lin-1986},
\begin{equation}\label{eq-gap}
\min{\bigl\{E_{1,2}(u) : u\in \Sobolev^{1,2}_g(\Ball^3;\Sphere^2)\bigr\}} 
< \inf{\bigl\{E_{1,2}(u) : u\in \Hilbert^{1,2}_g(\Ball^3;\Sphere^2) \bigr\}}.
\end{equation}
This strict inequality dramatically emphasizes the lack of density of smooth maps in the Sobolev space \(\Sobolev^{1,2}_g(\Ball^3;\Sphere^2)\). 
More generally, if \(p \le \min{\{m-1, n\}}\) is an integer and if there exists \(G\in \Smooth^{\infty}(\Sphere^{\floor{p}}; \manfN)\) such that, apart from the zero multiple, no integer multiple of the image of \(G\) is homologous to zero, then one can find \(g\in \Smooth^{\infty}(\cBall^m;\manfN)\) such that, see \cite{Giaquinta-Modica-Soucek-gap},
\begin{equation}\label{eq-gap2}
\min{\bigl\{E_{1,p}(u) : u\in \Sobolev^{1,p}_g(\Ball^m;\manfN)\bigr\}} < \inf{\bigl\{E_{1,p}(u) : u\in \Hilbert^{1,p}_g(\Ball^m;\manfN)\bigr\}}.
\end{equation}
Of course, this gap can only occur when \(\Hilbert^{1,p}_g(\Ball^m;\manfN) \neq \Sobolev^{1,p}_g(\Ball^m;\manfN)\), and we refer to \cites{Bethuel, Hang-Lin-IV} for specific conditions on \(\manfN\) and \(g\) implying that they coincide.
One of the most fascinating open problems in this field is to determine whether the infima on the right-hand sides of \eqref{eq-gap} or \eqref{eq-gap2} are indeed minima and whether minimizers are continuous, see \citelist{\cite{Brezis-2023}*{Section~4} \cite{Brezis-Mironescu}*{Chapter~13}}.

Gap phenomena have also been studied in the setting of compact manifolds \(\manfM\) without boundary, see Section~\ref{sectionHEP}. 
The substitute for the boundary condition is a constraint on the \(\ell\)-homotopy type of functions with \(\ell = \floor{p}-1\) or \(\ell = \floor{p}\) depending on whether the minimization is performed on \(\Sobolev^{1,p}(\manfM;\manfN)\) or \(\Hilbert^{1,p}(\manfM;\manfN)\), respectively. 
According to \cite{Lin-1999}, for every integer \(p\geq 2\), a minimizing sequence of \(E_{1,p}\) in the space \(\Smooth^{\infty}(\manfM;\manfN)\) with the same homotopy class as \(g\) contains a weakly converging subsequence whose weak limit \(u\) is \(\Smooth^{1, \alpha}\) away from a closed \((m-p)\)-dimensional set \(T\) and \(u\) has the same \((p-1)\)-homotopy type as \(g\). 
If, in addition, smooth maps are dense in \(\Sobolev^{1,p}(\manfM;\manfN)\), then \(u\) is a minimizer of \(E_{1,p}\) which has the same \(p\)-homotopy type as \(g\) and is \(\Smooth^{1,\alpha}\) away from an \((m-p-1)\)-dimensional set. 
This result has been partially extended to the setting of higher-order Sobolev spaces \(\Sobolev^{k,p}(\manfM;\manfN)\) when \(p=2\) and \(m=2k\), see \cite{He-Jiang-Lin}.


\backmatter

\cleardoublepage
\chapter*{Notation}

\subsection*{Basic Notation and General Conventions}
\begin{notationitemize}
    \item \(\N = \{0, 1, 2, \ldots\}\), \(\N_* = \{1, 2, \ldots\}\),
    \item \(\floor{p}\): largest integer less than or equal to \(p\),
    \item \(\ell^* = m - e - 1\): dual dimension in \(\R^m\) with respect to \(\ell\),
    \item \(\Ball^m\): unit open ball centered at \(0\) in \(\R^m\), and \(\cBall^m\): its closure,
    \item \(B_r^m\), \(B_r^m(a)\): open balls of radius \(r\) centered at \(0\) or \(a\) in \(\R^m\),
    \item \(B_r^m(A)\): open neighborhood of radius \(r > 0\) of a set \(A\),
    \item \(Q_r^m\), \(Q_r^m(a)\): open cubes of radius \(r\) (side length \(2r\)) centered at \(0\) or \(a\) in \(\R^m\),
    \item \(|a|\): Euclidean norm of \(a \in \R^m\); \(\abs{a}_{\infty} = \max{\{ \abs{a_{1}}, \ldots, \abs{a_{m}}\}}\): its sup norm,
    \item \((X, d, \mu)\) or \(X\): metric measure space with distance \(d\) and Borel measure \(\mu\) (in Chapters~\ref{chapter-3-Detectors} and~\ref{section_VMO}, \(X\) satisfies the doubling property, metric continuity and uniform nondegeneracy),
    \item \(\cH^\ell\): Hausdorff measure of dimension \(\ell\).
\end{notationitemize}

\subsection*{Manifolds and Domains}
\begin{notationitemize}
    \item \(\manfM\): compact \(m\)-dimensional manifold without boundary, imbedded in \(\R^\varkappa\),
    \item \(\manfN\): compact \(n\)-dimensional manifold without boundary, imbedded in \(\R^\nu\),
    \item \(\Omega\): Lipschitz open subset of \(\R^m\) (also bounded in Chapters~\ref{chapterIntroduction} and~\ref{chapter-approximation-Sobolev-manifolds}),
    \item \(\manfV\): denotes \(\Omega\) or \(\manfM\), as described above,
    \item \(\manfA\): Riemannian manifold \(=\) a Lipschitz open set or a compact manifold without boundary,
    \item \(\Tangent{x}\manfA\): tangent plane at \(x \in A\),
    \item \(\Int{F}\): set of interior points of \(F\).
\end{notationitemize}

\subsection*{Simplices, Complexes, and Cubications}
\begin{notationitemize}
    \item \(\Simplex^\ell\): simplex of dimension \(\ell\) (Definition~\ref{definitionSimplex}),
    \item \(\cK^\ell, \cL^\ell, \cE^\ell\): simplicial complexes of dimension \(\ell\) (Definition~\ref{definitionComplex}), and \(K^\ell\), \(L^\ell\), \(E^\ell\): their polytopes,
    \item \(\cS^m\): cubication in \(\R^m\) (Definition~\ref{defnCubication}),
    \item \(\cT^m\): triangulation of \(\manfM\) (Definition~\ref{defnTriangulation}) or dual skeleton of a cubication (Definition~\ref{definitionDualSkeleton}).
\end{notationitemize}

\subsection*{Functions, Maps and Perturbations}
\begin{notationitemize}
    \item \(\norm{u}_\infty = \sup{\abs{u}}\): sup norm of a bounded function \(u\),
    \item \(\abs{\gamma}_{\Lip}\): smallest Lipschitz constant of a Lipschitz function \(\gamma\),
    \item \(\Fuglede_w(X; \manfV)\): Fuglede maps associated to a summable function \(w\) (Definition~\ref{defnFuglede}),
    \item \(\tau \colon U \to \manfV\): transversal perturbation of the identity (Definition~\ref{definitionTransversalFamily}),
    \item \(\Pi\), \(\widetilde{\Pi}\): nearest point projections to \(\manfN\) and \(\manfM\), respectively.
\end{notationitemize}

\subsection*{Functions and Differentiation}
\begin{notationitemize}
  \item \(\Maximal f\): maximal function of \(f\) (Eq.~\eqref{eqMaximal}),
  \item \(\Jacobian{k}{\gamma}\): Jacobian of dimension \(k\) of a Lipschitz map \(\gamma\) (Eq.~\eqref{eqFugledeJacobian} or \eqref{eqLExtensionJacobian}, depending on the dimensions),
  \item \(\varphi_\psi * u\): adaptive convolution, defined on \(A_\psi\) (Eq.~\eqref{eqJEMS-AdaptiveConvolution} and~\eqref{eqJEMS-AdaptiveConvolutionDomain}).
\end{notationitemize}

\subsection*{Function Spaces}
\begin{notationitemize}
    \item \(\Sobolev^{s, p}(\manfV)\): fractional Sobolev functions of order \(0 < s < 1\) (Definition~\ref{definitionSobolevFractional}),
    \item \(\Sobolev^{k, p}(\manfV)\): Sobolev real functions of order \(k \in \N_*\) (Eq.~\eqref{eq-Introduction-distance}),
    \item \(\Sobolev^{k, p}(\manfV; \manfN)\): Sobolev maps of order \(k \in \N_*\) with values in \(\manfN\),
    \item \(\Hilbert^{k, p}(\manfV; \manfN)\): strong closure of \((\Smooth^\infty \cap \Sobolev^{k, p})(\manfV; \manfN)\) with respect to the \(\Sobolev^{k, p}\)~distance (Definition~\ref{definitionHkp}), and \(\Hilbert\weak^{k, p}(\manfV; \manfN)\): its weak closure (Definition~\ref{definitionHkpWeak}),
    \item \(\ClassR_{i}(\manfM; \manfN)\): maps  with values in \(\manfN\) that are smooth except on a structured singular set of rank \(i\) (Definition~\ref{definitionExtension}),
    \item \(\VMO(X)\): functions with vanishing mean oscillation (p.~\pageref{eqDetector-VMO}),
    \item \(\VMO^\ell(\manfV)\): functions that are generically \(\VMO\) and generically stable under \(\VMO\) convergence on \(\ell\)-dimensional polytopes (Definition~\ref{definitionVMOell}),
    \item \(\VMO^\#(K^\ell | L^r)\) and \(\VMO^\#(\cK^\ell)\): functions with vanishing crossed oscillations (p.~\pageref{eqExtensionVMOSharp} and Eq.~\eqref{eqDecoupleVMOSharp}).
\end{notationitemize}

\subsection*{Oscillations and \(\VMO\)-norms}
\begin{notationitemize}
    \item \(\seminorm{v}_{\rho}\): mean oscillation at scale \(\rho\) (Eq.~\eqref{eqDetector-33}),
    \item \(\seminorm{v}_{\VMO(X)}\) and \(\norm{v}_{\VMO(X)}\): \(\VMO\)~seminorm and norm (Eq.~\eqref{eqDetectorSeminorm} and \eqref{eqDetector-VMO}),
    \item \([v]_{L^{r_{1}}, L^{r_{2}}, \rho}\): crossed mean oscillation  at scale \(\rho\) (Eq.~\eqref{eqExtension-452}),
    \item \(\seminorm{v}_{\VMO^\#(K^{\ell}|L^{r})}\) and \(\norm{v}_{\VMO^\#(K^{\ell}|L^{r})}\): seminorm and norm in \(\VMO^\#(K^{\ell}|L^{r})\) (Eq.~\eqref{eqExtension-592} and \eqref{eqExtension-586}),
    \item \(\norm{v}_{\VMO^{\#}(\cK^{\ell})} \): norm in \(\VMO^{\#}(\cK^{\ell})\) (Eq.~\eqref{eqDecoupleVMOSharpNorm}).
\end{notationitemize}

\subsection*{Convergences and Homotopies}
\begin{notationitemize}
    \item convergence in \(\Fuglede_w(X; \manfV)\) (Definition~\ref{definitionFugledeConvergence}),
    \item convergence in \(\VMO^\ell\) (Definition~\ref{def_conv_VMOl}),
    \item homotopy in \(\VMO\) (Definition~\ref{defnVMOHomotopy}),
    \item homotopy in \(\VMO^\ell\) (Definition~\ref{definitionVMOEllHomotopy}).
\end{notationitemize}

\subsection*{Jacobians, Currents, and Topological Invariants}
\begin{notationitemize}
    \item \(\varpi\): closed differential form in \(\Smooth^{\infty}(\manfN; \Forms^\ell)\),
    \item \(\hur_\varpi(v)\): Hurewicz degree of a smooth or \(\VMO\) map \(v \colon \Sphere^{\ell} \to \manfN\) (Definition~\ref{definitionHurewiczDegree} and Proposition~\ref{propositionHurewiczVMO}),
    \item \(\Hur_{\varpi}(u)\): Hurewicz current of \(u \in \Sobolev^{1, \ell}(\manfV; \manfN)\) (Definition~\ref{defnJacobianHurewicz}),
    \item \(\omega_{\Sphere^\ell}\): volume form on \(\Sphere^\ell\) such that \(\int_{\Sphere^\ell} \omega_{\Sphere^\ell} = 1\),    
    \item \(\jac{u} = \Hur_{\omega_{\Sphere^n}}(u)\): distributional Jacobian of  \(u \in \Sobolev^{1, n}(\manfV; \Sphere^n)\) (Eq.~\eqref{eq-def-jacobian}),
    \item \(\hopf{v}\): Hopf degree of a smooth or \(\VMO\) map \(v \colon \Sphere^{2n-1} \to \Sphere^{n} \) (Eq.~\eqref{eqJacobian-1500} and Proposition~\ref{propositionSmoothHopfInvariant-VMO}),
    \item \(\Hopf{u}\): Hopf current of \(u \in \Sobolev^{1, p} (\manfV; \Sphere^n)\) (Definition~\ref{definitionHopfCurrent}).
\end{notationitemize}

\subsection*{Sobolev approximation}
\begin{notationitemize}
    \item \(\cS^m\): cubication in \(\R^m\) (Definition~\ref{defnCubication}),
    \item \(S^j\): union of all elements in \(\cS^j\) (Definition~\ref{defnCubicationUnion}),
    \item \(\cT^{\ell^*}\): dual skeleton of \(\cS^{\ell}\) (Definition~\ref{definitionDualSkeleton}),
    \item \(\cE^m\): subskeleton of bad cubes in \(\cS^m\) (p.~\pageref{eqJEMSBadBall}),
    \item \(\cU^m \supset \cE^m\): subskeleton of cubes in \(\cS^m\) that intersects a bad cube
    (p.~\pageref{eqJEMSUglyBall}),
    \item \(\cZ^{\ell^*}\): dual skeleton of \(\cU^{\ell}\) (p.~\pageref{eqJEMSUglyCubesDualSkeleton}),
    \item \(u_{\eta}^{\mathrm{op}} = u \compose \Phi^{\mathrm{op}}\): opening of \(u\) is a neighborhood of \(U^\ell\) (Eq.~\eqref{eqJEMS-uop}),
    \item \(u_{\eta}^{\mathrm{fop}} = u \compose \Phi^{\mathrm{fop}}\): opening of \(u\) is a neighborhood of \(S^\ell\) (Eq.~\eqref{eqJEMS-ufop}),
    \item \(u_{\eta}^\mathrm{sm} = \varphi_{\psi} * u_{\eta}^{\mathrm{op}}\): adaptive smoothing of \(u_{\eta}^{\mathrm{op}}\)  (Eq.~\eqref{eqJEMS-usm}),
    \item \(u_{\eta}^{\mathrm{th}} = u_{\eta}^{\mathrm{sm}} \compose \Phi^{\mathrm{th}}\): thickening of \(u_{\eta}^{\mathrm{sm}}\) with respect to \(T^{\ell^*} \cap U^m\) (Eq.~\eqref{eqJEMS-uth}),
    \item \(u_{\eta}^\mathrm{fsm} = \varphi_{\psi} * u_{\eta}^{\mathrm{fop}}\): adaptive smoothing of \(u_{\eta}^{\mathrm{fop}}\)  (Eq.~\eqref{eqJEMS-ufsm}),
    \item \(u_{\eta, \mu}^{\mathrm{ex}}\): smooth extension of \(u_{\eta}^{\mathrm{th}}\) to \(S^{m} \setminus T^{e^{*}}\) (Eq.~\eqref{eqJEMS-uex}),
    \item \(u_{\eta, \tau, \mu}^{\mathrm{sh}} = u_{\eta, \mu}^{\mathrm{ex}} \compose \Phi^{\mathrm{sh}}\): shrinking of \(u_{\eta, \mu}^{\mathrm{ex}}\) in a neighborhood of \(T^{\ell^*}\) (Eq.~\eqref{eqJEMS-ush}).   
\end{notationitemize}


\cleardoublepage


\end{document}